\documentclass{amsart}

\usepackage[mathscr]{euscript}
\DeclareFontFamily{OT1}{rsfs}{}
\DeclareFontShape{OT1}{rsfs}{n}{it}{<-> rsfs10}{}
\DeclareMathAlphabet{\curly}{OT1}{rsfs}{n}{it}

\makeatletter
\newcommand{\eqnum}{\refstepcounter{equation}\textup{\tagform@{\theequation}}}
\makeatother

\newcommand\beq[1]{\begin{equation}\label{#1}}
\newcommand\eeq{\end{equation}}
\newcommand\beqa{\begin{eqnarray*}}
\newcommand\eeqa{\end{eqnarray*}}

\title[Categorical DT theory for local surfaces]{Categorical 
Donaldson-Thomas theory for local surfaces}
\date{}
\author{Yukinobu Toda}

\usepackage{amscd}
\usepackage{amsmath}
\usepackage{amssymb}
\usepackage{amsthm}
\usepackage{float}
\usepackage[dvips]{graphicx}
\usepackage{xypic}

\usepackage{array}
\usepackage{amscd}
\usepackage[all]{xy}

\usepackage{tabulary}
\usepackage{booktabs}

\usepackage{amssymb, paralist, xspace, url, amscd, euscript, mathrsfs, stmaryrd}

\DeclareFontFamily{U}{rsfs}{%
\skewchar\font127}
\DeclareFontShape{U}{rsfs}{m}{n}{%
<-6>rsfs5<6-8.5>rsfs7<8.5->rsfs10}{}
\DeclareSymbolFont{rsfs}{U}{rsfs}{m}{n}
\DeclareSymbolFontAlphabet
{\mathrsfs}{rsfs}
\DeclareRobustCommand*\rsfs{%
\@fontswitch\relax\mathrsfs}

\theoremstyle{plain}
\newtheorem{thm}{Theorem}[subsection]
\newtheorem{prop}[thm]{Proposition}
\newtheorem{lem}[thm]{Lemma}

\newtheorem{defi}[thm]{Definition}
\newtheorem{rmk}[thm]{Remark}
\newtheorem{cor}[thm]{Corollary}

\newtheorem{prop-defi}[thm]{Proposition-Definition}
\newtheorem{thm-defi}[thm]{Theorem-Definition}
\newtheorem{lem-defi}[thm]{Lemma-Definition}

\newtheorem{assum}[thm]{Assumption}
\newtheorem{conj}[thm]{Conjecture}
\newtheorem{exam}[thm]{Example}

\newcommand{\ssslash}{/\!\!/}

\newcommand{\aA}{\mathcal{A}}
\newcommand{\bB}{\mathcal{B}}
\newcommand{\cC}{\mathcal{C}}
\newcommand{\dD}{\mathcal{D}}
\newcommand{\eE}{\mathcal{E}}
\newcommand{\fF}{\mathcal{F}}

\newcommand{\hH}{\mathcal{H}}
\newcommand{\iI}{\mathcal{I}}
\newcommand{\jJ}{\mathcal{J}}
\newcommand{\kK}{\mathcal{K}}
\newcommand{\lL}{\mathcal{L}}
\newcommand{\mM}{\mathcal{M}}
\newcommand{\nN}{\mathcal{N}}
\newcommand{\oO}{\mathcal{O}}
\newcommand{\pP}{\mathcal{P}}

\newcommand{\rR}{\mathcal{R}}
\newcommand{\sS}{\mathcal{S}}
\newcommand{\tT}{\mathcal{T}}
\newcommand{\uU}{\mathcal{U}}
\newcommand{\vV}{\mathcal{V}}
\newcommand{\wW}{\mathcal{W}}
\newcommand{\xX}{\mathcal{X}}
\newcommand{\yY}{\mathcal{Y}}
\newcommand{\zZ}{\mathcal{Z}}

\newcommand{\fE}{\mathfrak{E}}
\newcommand{\ffF}{\mathfrak{F}}

\newcommand{\fI}{\mathfrak{I}}

\newcommand{\fM}{\mathfrak{M}}
\newcommand{\fN}{\mathfrak{N}}

\newcommand{\fP}{\mathfrak{P}}

\newcommand{\fT}{\mathfrak{T}}
\newcommand{\fU}{\mathfrak{U}}

\newcommand{\Supp}{\mathop{\rm Supp}\nolimits}
\newcommand{\Hom}{\mathop{\rm Hom}\nolimits}

\newcommand{\dotimes}{\stackrel{\textbf{L}}{\otimes}}

\newcommand{\dR}{\mathbf{R}}
\newcommand{\dL}{\mathbf{L}}
\newcommand{\NS}{\mathop{\rm NS}\nolimits}

\newcommand{\Pic}{\mathop{\rm Pic}\nolimits}

\newcommand{\id}{\textrm{id}}

\newcommand{\ch}{\mathop{\rm ch}\nolimits}

\newcommand{\Ext}{\mathop{\rm Ext}\nolimits}
\newcommand{\Spec}{\mathop{\rm Spec}\nolimits}
\newcommand{\rank}{\mathop{\rm rank}\nolimits}
\newcommand{\Coh}{\mathop{\rm Coh}\nolimits}

\newcommand{\Mod}{\mathop{\rm Mod}\nolimits}

\newcommand{\ev}{\mathop{\rm ev}\nolimits}

\newcommand{\us}{\mathchar`-\rm{us}}
\newcommand{\sss}{\mathchar`-\rm{ss}}

\newcommand{\cneq}{\mathrel{\raise.095ex\hbox{:}\mkern-4.2mu=}}
\newcommand{\eqcn}{\mathrel{=\mkern-4.5mu\raise.095ex\hbox{:}}}

\newcommand{\Cok}{\mathop{\rm Cok}\nolimits}
\newcommand{\ext}{\mathop{\rm ext}\nolimits}

\newcommand{\Aut}{\mathop{\rm Aut}\nolimits}

\newcommand{\Cone}{\mathop{\rm Cone}\nolimits}

\newcommand{\Perv}{\mathop{\rm Perv}\nolimits}

\newcommand{\DT}{\mathop{\rm DT}\nolimits}

\newcommand{\modu}{\mathop{\rm mod}\nolimits}
\newcommand{\End}{\mathop{\rm End}\nolimits}

\newcommand{\Imm}{\mathop{\rm Im}\nolimits}

\newcommand{\RHom}{\mathop{\dR\mathrm{Hom}}\nolimits}
\newcommand{\Ker}{\mathop{\rm Ker}\nolimits}
\newcommand{\Ree}{\mathop{\rm Re}\nolimits}

\newcommand{\GL}{\mathop{\rm GL}\nolimits}

\newcommand{\vdim}{\mathop{\rm vdim}\nolimits}

\newcommand{\Spf}{\mathop{\rm Spf}\nolimits}
\newcommand{\cl}{\mathop{\rm cl}\nolimits}

\newcommand{\Ind}{\mathop{\rm Ind}\nolimits}
\newcommand{\ind}{\mathop{\rm ind}\nolimits}

\newcommand{\MF}{\mathop{\rm MF}\nolimits}

\newcommand{\Crit}{\mathop{\rm Crit}\nolimits}
\newcommand{\bS}{\mathbb{S}}
\newcommand{\bU}{\mathbb{U}}

\newcommand{\wt}{\mathrm{wt}}
\newcommand{\intt}{\mathrm{int}}
\newcommand{\coh}{\mathrm{coh}}
\newcommand{\qcoh}{\mathrm{qcoh}}
\newcommand{\Dbc}{D^b_{\rm{coh}}}
\newcommand{\CC}{\mathbb{C}^{\ast}}
\newcommand{\inclusion}{\ar@<-0.3ex>@{^{(}->}[r]}
\newcommand{\upinclusion}{\ar@<-0.3ex>@{^{(}->}[u]}
\newcommand{\leinclusion}{\ar@<-0.3ex>@{^{(}->}[l]}
\newcommand{\doinclusion}{\ar@<-0.3ex>@{^{(}->}[d]}
\newcommand{\diasquare}{\ar@{}[rd]|\square}

\newcommand{\lkakko}{[\![}
\newcommand{\rkakko}{]\!]}

\newcommand{\lgakko}{(\!(}
\newcommand{\rgakko}{)\!)}

\newcommand{\C}{\mathbb{C}^{\ast}}
\newcommand{\rig}{\mathchar`-\rm{rig}}

\newcommand{\st}{\mathchar`-\rm{st}}

\newcommand{\dDT}{\mathcal{DT}}
\newcommand{\wdDT}{\widehat{\mathcal{DT}}}
\newcommand{\bff}{\mathbf{f}}

\def\acts{\curvearrowright}
\def\lacts{\curvearrowleft}

\makeatletter
\renewcommand{\theequation}{%
	\thesubsection.\arabic{equation}}
\@addtoreset{equation}{subsection}
\makeatother

\begin{document}

\begin{abstract}
We develope 
$\mathbb{C}^{\ast}$-equivariant
categorical 
Donaldson-Thomas theory 
for local surfaces, i.e. 
the total spaces of canonical line bundles on 
smooth projective surfaces. 
We introduce $\C$-equivariant DT categories for local 
surfaces as Verdier quotients of 
derived categories of coherent sheaves on 
derived moduli stacks of coherent sheaves on surfaces, by 
subcategories of objects whose singular supports 
are contained in unstable loci. 
Via Koszul duality, 
our construction may be regarded as certain gluing of 
$\mathbb{C}^{\ast}$-equivariant
derived categories of factorizations. 

We also develope $\C$-equivariant DT theory for 
stable D0-D2-D6 bound states on local surfaces, 
including categorical Pandharipande-Thomas theory. 
The key result toward the construction is 
the description of the 
stack of D0-D2-D6 bound states on the local surface
as the dual obstruction cone over the moduli 
stack of pairs on the surface. 
 We propose 
 several 
 conjectures on wall-crossing of 
PT categories, 
motivated by categorifications of wall-crossing 
formula of PT invariants and 
d-critical analogue of D/K equivalence conjecture in birational 
geometry. 

We establish three ways toward the categorical wall-crossing 
conjecture: semiorthogonal decomposition via linear Koszul duality, 
window theorem for DT categories, and categorified Hall products. 
These techniques indicate several implications, e.g. 
rationality of generating series of PT categories, wall-crossing equivalence
of DT categories for one dimensional stable sheaves, 
and categorical MNOP/PT correspondence for reduced curve classes.

\end{abstract}

\maketitle

\setcounter{tocdepth}{2}
\tableofcontents

\section{Introduction}
\subsection{Motivation}
\subsubsection{Background}
On a smooth projective Calabi-Yau 3-fold,
Thomas~\cite{HCasson}
introduced the integer valued invariants
which virtually count stable coherent sheaves on it, 
and give complexifications of Casson 
invariants on real 3-folds. 
 His invariants
are now called \textit{Donaldson-Thomas (DT)} invariants,
 and have been developed in 
several contexts, e.g. relation with Gromov-Witten 
invariants~\cite{MR2264664, MR3600040}, wall-crossing 
under change of stability conditions~\cite{K-S, JS, Tsurvey}, 
relation with cluster algebras~\cite{MR3079250, MR3739230}, or so on.  

The original DT invariants in~\cite{HCasson}
were defined via integrations of zero-dimensional 
virtual fundamental classes 
on the moduli spaces of stable sheaves on CY 3-folds. 
Later, Behrend~\cite{MR2600874} showed that 
the DT invariants are also defined via the 
weighted Euler characteristics of 
certain constructible functions (called Behrend functions) on the above moduli spaces. 
Now by the works~\cite{JS, BBJ}, 
these moduli spaces are known to be locally written as critical 
loci of some functions, and Behrend functions coincide
with point-wise Euler characteristics of the vanishing cycles.  
Based on this fact, 
Kontsevich-Soibelman~\cite{K-S, MR2851153} 
proposed refinements of DT invariants, 
called motivic (cohomological) DT invariants.
They are defined 
by gluing 
the locally defined motivic (perverse sheaves of) 
vanishing cycles 
through a choice of an orientation data, that 
is a square root of the virtual canonical line bundle 
on the moduli space. 
The foundations of such refined DT invariants 
were established in~\cite{BDM, MR3353002}, 
and for example used in~\cite{MR3842061} to give 
a mathematical definition of Gopakumar-Vafa
invariants. 

As we mentioned above, the moduli space $M$ of stable sheaves
on a
CY 3-fold is locally written as a critical locus
$\mathrm{Crit}(w)$ of some function $w \colon Y \to \mathbb{C}$ on 
a smooth scheme $Y$. 
More strongly, it is known that 
$M$ is a truncation of a derived scheme with a $(-1)$-shifted 
symplectic structure~\cite{MR3090262}, which is locally 
equivalent to the derived critical locus of $w$~\cite{BBJ, MR3352237}, 
and induces d-critical structure on $M$ introduced by Joyce~\cite{MR3399099}. 
Then locally on $M$, we can attach the triangulated 
category $\mathrm{MF}(Y, w)$
 of matrix factorizations of $w$. 
Its objects consist of locally free sheaves $\pP_0$, $\pP_1$ on $Y$
together with maps
\begin{align}\notag
\xymatrix{
\pP_0 \ar@/^8pt/[r]^{\alpha_0} &  \ar@/^8pt/[l]^{\alpha_1} \pP_1,  
} \
\alpha_0 \circ \alpha_1=\cdot w, \ 
\alpha_1 \circ \alpha_0=\cdot w. 
\end{align}
The triangulated category 
$\mathrm{MF}(Y, w)$ 
is regarded as a categorification of 
the vanishing cycle, as its periodic cyclic homology
recovers the $\mathbb{Z}/2$-graded
hypercohomology of the 
perverse sheaf of vanishing cycles of $w$
(see~\cite{MR3815168, MR3877165}). 
Therefore 
as mentioned in~\cite[(J)]{Jslide}, \cite[Section~6.1]{MR3728637}, 
it is natural to try to 
glue the
locally defined categories $\mathrm{MF}(Y, w)$, 
and regard it as a categorification of the DT invariant.  
If such a gluing exists, say $\mathcal{DT}(M)$, 
then we would like to call $\mathcal{DT}(M)$ 
as a \textit{DT category}, and study its properties. 

However the construction of such a gluing 
does not seem to be obvious in general. 
It may be constructed by taking the homotopy limit 
of dg-enhancements $\mathrm{MF}_{\rm{dg}}(Y, w)$
of $\mathrm{MF}(Y, w)$, 
twisted by another dg-categories determined by a choice of an orientation 
data. In order to take such a limit, we need an $\infty$-functor from 
the category of open subsets $U \subset M$ written as $U=\mathrm{Crit}(w)$
to the $\infty$-category of dg-categories, 
of the form
\begin{align*}
U \mapsto \mathrm{MF}_{\rm{dg}}(Y, w) \otimes (\mathrm{twisting}).
\end{align*}
A construction of such an $\infty$-functor in general seems to require 
a substantial work on $\infty$-categories, which is beyond 
the scope of this paper. 

Instead of constructing such a gluing in general, 
we focus on the case that 
a CY 3-fold is  
the total space of the canonical line bundle 
on a smooth projective surface, called a
\textit{local surface}. 
In this case $\mathbb{C}^{\ast}$ acts on 
fibers of the canonical bundle, 
and we 
construct $\mathbb{C}^{\ast}$-equivariant 
version of DT categories on local surfaces.  
Although this is a particular class of CY 3-folds, 
the moduli space $M$ is not necessary written as a 
global critical locus, so the gluing problem is still not obvious. 
Indeed our construction is not a direct gluing of matrix factorizations,
but rather indirect
through the
derived categories of coherent sheaves on derived moduli 
stacks of 
coherent sheaves on the surface, together with
Koszul duality. 
There are two simplifications in this case: 
the first one is that
a relevant $\infty$-functor is 
a priori given from the existence of the 
derived moduli stacks, so 
we can avoid technical subtleties on $\infty$-categories. 
The second one is that
there exists a canonical orientation data in this case, so 
the
construction is canonical without taking twisting. 

Then we have a well-defined categorification of
DT invariants on the local surface, and can
formulate several 
conjectures on DT categories. 
Our main motivation of studying $\mathcal{DT}(M)$ is 
to categorify wall-crossing formulas of DT invariants and 
give their relation with D/K equivalence conjecture in birational geometry, 
as we discuss below. 

\subsubsection{Motivation from d-critical birational geometry}\label{intro:motivation}
The development of wall-crossing formulas
of DT invariants~\cite{JS, K-S}
is one of the recent important progress 
on the study of DT invariants. 
The moduli space of stable sheaves (or Bridgeland stable objects~\cite{MR2373143})
depends on a choice of a stability condition $\sigma$, 
so should be written as $M^{\sigma}$. 
In general there is a wall-chamber structure on 
the space of stability conditions such that 
$M^{\sigma}$ is constant if $\sigma$ lies on a 
chamber, but may change when $\sigma$ crosses a wall. 
The change of the associated DT invariants are 
explicitly expressed using motivic Hall algebras. 

Suppose that $\sigma$ lies on a wall and 
$\sigma_{\pm}$ lie on its adjacent chambers. 
Then 
we have the following diagram 
\begin{align}\label{intro:diagram}
\xymatrix{
M^{\sigma_{+}} \ar[rd] & & M^{\sigma_-} \ar[ld] \\
& \overline{M}^{\sigma} &. 
}
\end{align}
where $\overline{M}^{\sigma}$ is the good moduli space of 
the moduli stack of 
$\sigma$-semistable objects. 
In~\cite{Toddbir}, the author formulated 
the notion of \textit{d-critical flips}, \textit{d-critical flops} 
for diagrams of schemes with d-critical structures
such as the wall-crossing diagram (\ref{intro:diagram}). 
These are d-critical analogue of usual flips and flops in birational geometry, 
but not honest birational maps so should 
be interpreted as virtual birational maps. 
Roughly speaking, they are to do with 
the sizes of virtual canonical line bundles: 
for a d-critical flip (flop), we have the 
inequality (equality) in terms of virtual canonical line bundles
(see~\cite[Definition~3.23]{Toddbir})
\begin{align*}
M^{\sigma_+} >_{K^{\rm{vir}}} (=_{K^{\rm{vir}}}) M^{\sigma_-}.
\end{align*}
This is an analogue of the similar 
inequality for usual
flip (flop) via canonical line bundles in birational geometry\footnote{For a birational 
map of smooth projective varieties $Y_+ \dashrightarrow Y_-$, 
we have the inequality (equality) $Y_+ >_K(=_K) Y_-$ if there is a common 
resolution 
$p_{\pm} \colon W \to Y_{\pm}$ 
such that $p_+^{\ast}K_{Y_+} -p_{-}^{\ast}K_{Y_-}$ is effective (trivial).}. 

On the other hand if two smooth projective varieties are related 
by a usual flip (flop), 
it is conjectured by Bondal-Orlov~\cite{B-O2}
and Kawamata~\cite{MR1949787} that there exists a fully-faithful functor
(equivalence) of their derived categories of coherent sheaves. 
The above conjecture relates sizes of 
derived categories and canonical line bundles, so is called a \textit{D/K equivalence conjecture}. 
Recently Halpern-Leistner~\cite{HalpK3, HalpK32} proves that the 
derived categories of coherent sheaves on 
Bridgeland stable objects on K3 surfaces are equivalent 
under the wall-crossing diagram (\ref{intro:diagram}). 
In the case of K3 surfaces, both sides of (\ref{intro:diagram}) 
are birational holomorphic symplectic manifolds, so connected 
by flops. 
Therefore his result proves D/K equivalence conjecture in this case. 

Our motivation on studying DT categories
is to formulate a
CY 3-fold 
 version of the above Halpern-Leistner's work for K3 surfaces. 
Namely under the wall-crossing diagram (\ref{intro:diagram})
for moduli spaces of stable objects on CY 3-folds, 
if it is a d-critical flip (flop), we expect the 
existence of a
fully-faithful functor (equivalence)
\begin{align}\notag
\mathcal{DT}(M^{\sigma_-})
\hookrightarrow (\stackrel{\sim}{\to}) \mathcal{DT}(M^{\sigma_+}). 
\end{align}
This expectation 
relates DT categories
of $M^{\sigma_{\pm}}$ with their virtual 
canonical line bundles, 
 so is 
an analogue of the D/K equivalence conjecture for d-critical loci. 

Furthermore
we expect that the semiorthogonal complement of 
the above fully-faithful embedding recovers the wall-crossing 
formula of numerical DT invariants. 
In this way, the 
study of DT categories will give a 
link between D/K equivalence conjecture in birational geometry 
and wall-crossing formula of DT invariants.

\subsection{Categorical DT theory for local surfaces}
\subsubsection{Idea from Koszul duality}
Let $S$ be a smooth projective surface and
\begin{align*}
\pi \colon X=\mathrm{Tot}_S(\omega_S) \to S
\end{align*}
the total space of the canonical line bundle on $S$, 
which is a non-compact CY 3-fold. 
Let $\mM_S$, $\mM_X$ be the (classical) moduli stacks of 
compactly supported coherent sheaves on $S$, $X$ respectively, 
with fixed Chern characters. 
There is a natural map 
$\pi_{\ast} \colon \mM_X \to \mM_S$, and it is well-known that 
$\mM_X$ is the dual obstruction 
cone over $\mM_S$ associated with a natural 
perfect obstruction theory of $\mM_S$ (see~\cite{MR3607000} and 
Lemma~\ref{lem:moduliX}). 

The above fact is explained from derived algebraic geometry 
in the following way. Let
 $\mM_S \subset \mathfrak{M}_S$ be the derived enhancement of 
$\mM_S$ which is a quasi-smooth derived stack, and consider its 
$(-1)$-shifted cotangent derived stack 
$\Omega_{\mathfrak{M}_S}[-1] \to \fM_S$
which admits a $(-1)$-shifted symplectic structure~\cite{Calaque}. 
 Then we have 
\begin{align}\label{intro:isom:M}
\mM_X \cong t_0(\Omega_{\mathfrak{M}_S}[-1])
\end{align}
where $t_0(-)$ means the classical truncation of the 
derived stack. 
Locally on $\mathfrak{M}_S$, this means as follows. 
Let $\mathfrak{U}$ be an affine derived scheme 
given by the derived zero locus of a section $s$ of a vector bundle 
$V \to Y$ on a smooth affine scheme $Y$. 
Then for a smooth morphism $\mathfrak{U} \to \mathfrak{M}_S$, 
the isomorphism (\ref{intro:isom:M}) over $t_0(\mathfrak{U}) \to \mM_S$
implies the isomorphism 
\begin{align}\label{intro:isom}
t_0(\mathfrak{U}) \times_{\mM_S} \mM_X
\stackrel{\cong}{\to}
\mathrm{Crit}(w) \subset
V^{\vee}. 
\end{align}
Here $w \colon V^{\vee} \to \mathbb{C}$
is 
defined by $w(x, v)=\langle s(x), v \rangle$
for $x \in Y$ and $v \in V|_{x}^{\vee}$. 

On the other hand, under the above local setting 
we have the Koszul duality equivalence
(see Theorem~\ref{thm:knoer} and~\cite{MR3071664, MR2982435, MR3631231, ObRoz})
\begin{align}\label{intro:koszul}
\Phi \colon D^b_{\rm{coh}}(\mathfrak{U}) \stackrel{\sim}{\to}
\mathrm{MF}^{\mathbb{C}^{\ast}}(V^{\vee}, w). 
\end{align} 
Here the left hand side is the derived category of dg-modules 
over $\oO_{\fU}$ with bounded coherent cohomologies, 
and 
the right hand side is the category of $\mathbb{C}^{\ast}$-equivariant 
matrix factorizations of $w$, 
 where $\mathbb{C}^{\ast}$ acts on 
the fibers of $V^{\vee} \to Y$ by weight two. 
Our point of view is that, 
via the equivalence (\ref{intro:koszul}), 
the derived category of the global 
moduli stack $\mathfrak{M}_S$ is regarded as a
gluing of  categories of $\mathbb{C}^{\ast}$-equivariant 
matrix factorizations. 
Thus it may be reasonable to 
set
\begin{align}\label{intro:catDT}
\mathcal{DT}^{\mathbb{C}^{\ast}}(\mM_X) \cneq 
D^b_{\rm{coh}}(\mathfrak{M}_S)
\end{align}
and call it the $\mathbb{C}^{\ast}$-equivariant 
DT category for $\mM_X$. 

However $\mM_X$ is usually 
not of finite type, and in order to obtain 
a nice moduli space we need to take its open 
substack of stable sheaves. 
As we will discuss below, we 
will use the notion of singular supports to 
extract information of (un)stable locus. 

\subsubsection{DT category for stable sheaves}
Let us take a stability condition $\sigma$ on 
the abelian category of compactly supported coherent sheaves on $X$. 
For example we can take 
\begin{align*}
\sigma=B+iH \in \mathrm{NS}(S)_{\mathbb{C}}
\end{align*}
such that $H$ is an ample class, and 
consider ($B$-twisted) $H$-Gieseker stability condition associated with 
$\sigma$ (see~Subsection~\ref{subsec:catDT:sigma}). 
Then we have the open substack
$\mM_X^{\sigma} \subset \mM_X$
corresponding to $\sigma$-stable sheaves on $X$, 
which is a $\mathbb{C}^{\ast}$-gerbe over a quasi-projective scheme 
$M_X^{\sigma}$. 
Our purpose is to define the $\mathbb{C}^{\ast}$-equivariant
DT category 
for $M_X^{\sigma}$. 
In general we have 
an open immersion 
\begin{align}\label{intro:shift}
t_0(\Omega_{\mathfrak{M}_{S}^{\sigma}}[-1])
\subset \mM_{X}^{\sigma}
\end{align}
where 
$\mathfrak{M}_{S}^{\sigma} \subset \mathfrak{M}_S$
is the derived open substack 
of $\sigma$-stable sheaves on $S$. 
If (\ref{intro:shift}) is an isomorphism, 
then   
we may define 
\begin{align*}
\mathcal{DT}^{\mathbb{C}^{\ast}}(M_{X}^{\sigma})
\cneq D^b_{\rm{coh}}(\mathfrak{M}_{S}^{\sigma, \C\rig})
\end{align*}
similarly to (\ref{intro:catDT}).
Here $\mathbb{C}^{\ast}\rig$ means 
the $\mathbb{C}^{\ast}$-rigidification, 
i.e. 
getting rid of the trivial $\mathbb{C}^{\ast}$-autmorphisms. 
However (\ref{intro:shift})
is not an isomorphism in many cases, 
as $\pi_{\ast}E$ for a stable sheaf $E$ on $X$ is not 
necessary stable on $S$.
Therefore we need to find another way to define 
$\mathcal{DT}^{\mathbb{C}^{\ast}}(M_X^{\sigma})$. 

Let $\zZ^{\sigma \us} \subset \mM_{X}$
be the closed substack
defined to be the complement of $\mM_X^{\sigma}$. 
Then its 
pull-back under the smooth map  
$t_0(\mathfrak{U}) \to \mM_S$
is a $\mathbb{C}^{\ast}$-invariant 
closed subscheme $Z \subset \mathrm{Crit}(w)$
via the isomorphism (\ref{intro:isom}). 
We would like to 
construct the DT category for $M_X^{\sigma}$
as a gluing of  
triangulated categories 
$\mathrm{MF}^{\mathbb{C}^{\ast}}(V^{\vee} \setminus Z, w)$. 
Our key observation is that the latter 
category is equivalent to the quotient category
\begin{align}\label{intro:MF:restrict}
\mathrm{MF}^{\mathbb{C}^{\ast}}(V^{\vee}, w)/\mathrm{MF}^{\mathbb{C}^{\ast}}(V^{\vee}, w)_Z \stackrel{\sim}{\to}
\mathrm{MF}^{\mathbb{C}^{\ast}}(V^{\vee} \setminus Z, w)
\end{align}
by the restriction functor. 
Here $\mathrm{MF}^{\mathbb{C}^{\ast}}(V^{\vee}, w)_Z$
is the subcategory of 
$\mathrm{MF}^{\mathbb{C}^{\ast}}(V^{\vee}, w)$
consisting of objects isomorphic to zero on 
$V^{\vee} \setminus Z$. 
We will see that, under the equivalence (\ref{intro:koszul}),
the subcategory $\mathrm{MF}^{\mathbb{C}^{\ast}}(V^{\vee}, w)_Z$
corresponds to the subcategory 
$\cC_Z \subset D^b_{\rm{coh}}(\mathfrak{U})$
consisting of objects in $D^b_{\rm{coh}}(\mathfrak{U})$
whose \textit{singular 
supports} are contained in $Z$
(see~\cite[Section~H]{MR3300415} and Proposition~\ref{prop:koszul:Z}). 

The notion of singular supports for (ind)coherent sheaves on 
quasi-smooth derived stacks was introduced and developed by 
Arinkin-Gaitsgory~\cite{MR3300415} in their formulation of 
geometric Langlands conjecture, following an 
earlier work by Benson-Iyengar-Krause~\cite{MR2489634}. 
For an object $\fF \in D^b_{\rm{coh}}(\mathfrak{U})$, its singular 
support is defined to be the support of the 
action of the Hochschild cohomology of $\mathfrak{U}$ 
to the graded vector space $\Hom^{2\ast}(\fF, \fF)$, which 
can be shown to lie on $t_0(\Omega_{\mathfrak{U}}[-1])$. 
As proved in~\cite{MR3300415}, the singular support is intrinsic to 
the derived stack, e.g. independent of presentation of $\mathfrak{U}$, 
preserved by 
pull-backs of smooth morphisms of quasi-smooth affine 
derived schemes, etc. 
Therefore we have the subcategory 
$\cC_{\zZ^{\sigma \us}}\subset D^b_{\rm{coh}}(\mathfrak{M}_S)$
consisting of objects whose singular supports are contained in $\zZ^{\sigma \us}$. 

From the equivalences (\ref{intro:koszul}), (\ref{intro:MF:restrict}),
the
Verdier quotient of $D^b_{\rm{coh}}(\mathfrak{M}_S)$
by the subcategory 
$\cC_{\zZ^{\sigma \us}}$
may be regarded as a gluing of the triangulated categories 
(\ref{intro:MF:restrict}). 
We need two slight modifications of this construction. 
The first one is that, since $\mathfrak{M}_S$ is not quasi-compact in general, 
we replace $\mathfrak{M}_S$ by its quasi-compact 
derived open substack $\mathfrak{M}_{S, \rm{qc}}$
whose truncation contains $\pi_{\ast}(\mM_X^{\sigma})$. 
The second one is that, 
since $\mM_X^{\sigma}$ is a 
$\mathbb{C}^{\ast}$-gerbe over $M_X^{\sigma}$, 
we replace 
$\mathfrak{M}_{S, \rm{qc}}$ by its 
$\mathbb{C}^{\ast}$-rigidification 
$\mathfrak{M}_{S, \rm{qc}}^{\mathbb{C}^{\ast}\rig}$
to 
get rid of the contributions of $\mathbb{C}^{\ast}$-autmorphisms. 
Then 
we propose the following definition: 
\begin{defi}\label{intro:def:catDT}\emph{(Definition~\ref{defi:catDT})}
We define
the $\mathbb{C}^{\ast}$-equivariant DT category 
for $M_X^{\sigma}$
to be the Verdier quotient
\begin{align*}
\mathcal{DT}^{\mathbb{C}^{\ast}}(M_X^{\sigma}) \cneq 
D^b_{\rm{coh}}(\mathfrak{M}_{S, \rm{qc}}^{\mathbb{C}^{\ast}\rig})/
\cC_{\zZ^{\sigma \us, \C\rig}}. 
\end{align*}
\end{defi}
We will see that the above definition does not depend (up to equivalence)
on 
a choice of a derived open substack
$\mathfrak{M}_{S, \rm{qc}} \subset \mathfrak{M}_S$. 
We will also define the $\lambda$-twisted version
for $\lambda \in \mathbb{Z}$, 
the dg-enhancements of
DT categories, and also 
$\wdDT$-version of DT categories. 
The relevant notations are summarized in Table~1 and 
Table~2
in the beginning of Section~\ref{sec:catDT:shift}. 
As a globalization of the result in~\cite{MR3815168}, 
we conjecture that 
the periodic cyclic homology of the 
dg-enhancement of $\mathcal{DT}^{\mathbb{C}^{\ast}}(M_X^{\sigma})$
recovers the $\mathbb{Z}/2$-graded 
cohomological DT invariant of $M_X^{\sigma}$, 
constructed in~\cite{MR3353002} using
 the canonical orientation data of $M_X^{\sigma}$
(see Conjecture~\ref{conj:periodic}). 
We only prove a numerical version of it
in a very special case: 
\begin{prop}\emph{(Proposition~\ref{prop:num})}\label{intro:prop:comparison}
Under some technical condition on 
$\mathfrak{M}_{S, \rm{qc}}^{\mathbb{C}^{\ast}\rig}$
(which is automatically satisfied if 
$t_0(\mathfrak{M}_{S, \rm{qc}}^{\mathbb{C}^{\ast}\rig})$
is a quasi-projective scheme), 
 we have the identity
\begin{align*}
\chi(\mathrm{HP}_{\ast}(\mathcal{DT}_{\rm{dg}}
^{\mathbb{C}^{\ast}}(M_X^{\sigma})))
=(-1)^{\vdim \mM_S+1}\DT(M_X^{\sigma}). 
\end{align*}
Here dg indicates the dg-enhancement, 
 $\mathrm{HP}_{\ast}(-)$
is the periodic cyclic homology, $\chi(-)$ is the Euler 
characteristic, $\vdim$ is the virtual 
dimension, and $\DT(-)$ is the numerical DT invariant.  
\end{prop}

Once we have a definition of DT categories, 
we can formulate categorical wall-crossing 
based on d-critical analogue of D/K equivalence conjecture. 
Let $N_{\le 1}(S)$ be the group of numerical classes of 
one or zero dimensional coherent sheaves on $S$. 
For a primitive element $v \in N_{\le 1}(S)$
and a stability condition $\sigma$ which lies on a wall, 
let $\sigma_{\pm}$ lie on its adjacent chambers. 
Then it is observed in~\cite{Toddbir} that
the
diagram (\ref{intro:diagram}) is a d-critical flop. 
Therefore we propose the following conjecture: 
\begin{conj}\label{intro:conjGV}\emph{(Conjecture~\ref{conj1})}
	There is an equivalence of DT categories
	\begin{align*}
		\dDT^{\C}(M_{X}^{\sigma_+}(v)) \stackrel{\sim}{\to}
		\dDT^{\C}(M_{X}^{\sigma_-}(v)). 
		\end{align*}
	\end{conj}
The DT invariants associated with $M_{X}^{\sigma}(v)$ 
are known as genus zero Gopakumar-Vafa invariants, which 
are well-known to be wall-crossing invariant (see~\cite{Tsurvey}). 
Therefore the above conjecture also categorifies
wall-crossing invariance of genus zero GV invariants. 

\subsection{Categorical DT theory for D0-D2-D6 bound states}
\subsubsection{Categorical MNOP/PT theories}\label{intro:subsec:dtpt}
The MNOP conjecture~\cite{MR2264664}
(which is proved in many cases~\cite{MR3600040}) is a conjectural 
relationship between  
generating series of 
Gromov-Witten invariants on a 3-fold 
and those of DT invariants 
counting one or zero dimensional subschemes on it. 
In our situation of the local surface, the relevant moduli 
space in the DT side 
is 
\begin{align*}
I_n(X, \beta), \quad (\beta, n) \in \NS(S) \oplus \mathbb{Z}
\end{align*}
which parametrizes ideal sheaves
$I_C \subset \oO_X$ for a compactly supported one or zero
dimensional subschemes $C \subset X$
satisfying $\pi_{\ast}[C]=\beta$ and $\chi(\oO_C)=n$. 

The notion of stable pairs was introduced by 
Pandharipande-Thomas~\cite{MR2545686} 
in order to give 
a better formulation of MNOP conjecture.
By definition a PT stable pair is a pair
$(F, s)$, 
where $F$ is a compactly supported pure one dimensional 
sheaf on $X$, and $s \colon \oO_X \to F$ is surjective in dimension one. 
The moduli space of stable pairs 
is denoted by 
\begin{align*}
P_n(X, \beta), \quad 
(\beta, n) \in \mathrm{NS}(S) \oplus \mathbb{Z}
\end{align*}
and it parametrizes stable pairs $(F, s)$
 such that 
$\pi_{\ast}[F]=\beta$ and $\chi(F)=n$. 
The moduli space $P_n(X, \beta)$ is regarded as 
a moduli space of two term complexes in the derived category
\begin{align}\label{intro:twoterm}
I^{\bullet}=(\oO_X \stackrel{s}{\to} F) \in D^b_{\rm{coh}}(X).
\end{align}
The relationship between the corresponding DT type 
invariants 
is conjectured in~\cite{MR2545686} and proved in~\cite{MR2813335, Tcurve1, MR2869309}
for CY 3-folds using wall-crossing arguments. 
If we denote by $I_{n, \beta}$, $P_{n, \beta}$ the corresponding 
DT invariants, the formula is 
\begin{align}\notag
	\sum_{n\in \mathbb{Z}} I_{n, \beta}q^n
	=M(-q)^{e(X)} \sum_{n\in \mathbb{Z}}P_{n, \beta}q^n
		\end{align}
where $M(q) \cneq \prod_{k\ge 1}(1-q^k)^{-k}$ is the MacMahon 
function. 

We would like to categorify the DT type invariants 
defined from $I_n(X, \beta)$ and $P_n(X, \beta)$, 
and try to formulate a categorical relationship of 
these moduli spaces. 
Here we note that, 
since an ideal sheaf $I_C$ 
nor the two term complex (\ref{intro:twoterm}) 
 do not have compact supports, 
 the construction in Definition~\ref{intro:def:catDT} does not 
 apply for this purpose. 
Our strategy is to 
realize the moduli spaces $I_n(X, \beta)$, $P_n(X, \beta)$
as open substacks of the 
dual obstruction cone over 
the moduli stack $\mM_S^{\dag}$ of 
pairs $(\oO_S \to F)$, where 
$F$ is a one or zero dimensional sheaf on $S$. 
Indeed we will show that
 the dual obstruction cone over 
$\mM_S^{\dag}$ is isomorphic 
to the moduli stack 
$\mM_X^{\dag}$ which parametrizes   
rank one objects 
in the extension closure
\begin{align*}
\aA_X=\langle \oO_{\overline{X}}, \Coh_{\le 1}(X)[-1] \rangle_{\rm{ex}}
\subset D^b_{\rm{coh}}(\overline{X}). 
\end{align*}
Here $X \subset \overline{X}$ is a compactification of $X$, and 
$\Coh_{\le 1}(X)$ is the category of compactly supported 
coherent sheaves on $X$ whose supports have dimensions 
less than or equal to one. 
The category $\aA_X$ (called the category of \textit{D0-D2-D6 bound states})
is an abelian category, which 
 was introduced 
 in~\cite{Tcurve1}
 and studied in~\cite{Tsurvey, MR3021446, Thall}
to show several properties on PT invariants. 
 As we mentioned above, we show the following: 
\begin{thm}\emph{(Theorem~\ref{thm:D026})}\label{intro:thm:MX}
The stack $\mM_X^{\dag}$ of rank one objects in $\aA_X$ 
admits a natural morphism $\pi_{\ast}^{\dag} \colon \mM_X^{\dag} \to 
\mM_S^{\dag}$, and is isomorphic 
to the dual obstruction cone over $\mM_S^{\dag}$. 
\end{thm}

As we will mention in Remark~\ref{rmk:dia:BS}, 
an object in $\aA_X$ may not be 
written as a two term complex $(\oO_X \to F)$ 
for a one or zero dimensional sheaf $F$, 
so the existence of the morphism $\pi_{\ast}^{\dag}$
is a priori not obvious. 
Our idea for Theorem~\ref{intro:thm:MX} is to 
describe rank one objects in $\aA_X$ in terms of 
certain diagrams of coherent sheaves on $S$, 
and construct the map $\pi_{\ast}^{\dag}$ in terms of the latter 
diagram (see Corollary~\ref{lem:moduli:isom}). 
From the definition of $\aA_X$, the moduli spaces
$I_n(X, \beta)$, $P_n(X, \beta)$
are open substacks of $\mM_X^{\dag}$. 
Similarly to Definition~\ref{intro:def:catDT}, 
we will define (see~Definition~\ref{cat:DTPT})
\begin{align*}
\mathcal{DT}^{\mathbb{C}^{\ast}}(I_n(X, \beta)), \quad 
\mathcal{DT}^{\mathbb{C}^{\ast}}(P_n(X, \beta))
\end{align*}
as Verdier-quotients of derived category of coherent sheaves on 
some quasi-compact open substack of 
$\mathfrak{M}_S^{\dag}$, 
where $\mathfrak{M}_S^{\dag}$ is a derived enhancement of $\mM_S^{\dag}$, 
by the subcategory of objects whose singular supports are 
contained in unstable loci. 
We call 
them as $\mathbb{C}^{\ast}$-equivariant
\textit{MNOP category}, \textit{PT category} respectively. 

On the other hand, it is observed in~\cite{Toddbir} that 
there is a wall-crossing diagram
\begin{align*}
	\xymatrix{
I_n(X, \beta) \ar[rd] & & P_n(X, \beta) \ar[ld] \\
& T_n(X, \beta) &	
}
	\end{align*}
which is a d-critical flip. Therefore 
following d-critical analogue of D/K equivalence conjecture, 
we propose the following: 
\begin{conj}\label{intro:conj:dtpt}
	There exists a fully-faithful functor 
	\begin{align*}
		\dDT^{\C}(P_n(X, \beta)) \hookrightarrow 
		\dDT^{\C}(I_n(X, \beta)). 
		\end{align*}
	\end{conj}

\subsubsection{DT category for stable D0-D2-D6 bound states}
From the author's 
previous works~\cite{Tolim2, Tsurvey}, it is known that 
the moduli space $P_n(X, \beta)$ is identified with
the moduli space of certain 
stable objects in the abelian category $\aA_X$. 
A stability parameter is given by $t \in \mathbb{R}$ and 
the stable pair moduli space corresponds to the $t\to \infty$ limit. 
By changing the stability parameter $t \in \mathbb{R}$, 
we have 
another moduli space of stable objects
and the corresponding DT category
(see Subsection~\ref{subsec:catDT:stableD026})
\begin{align*}
P_n^t(X, \beta) \subset \mM_X^{\dag}, \ 
\mathcal{DT}^{\mathbb{C}^{\ast}}(P_n^t(X, \beta)). 
\end{align*}
As we will discuss below, 
we study the relationships 
of the above DT categories 
along with the 
arguments in Subsection~\ref{intro:motivation}.

In~\cite{Tolim2, Tsurvey} it is proved that, 
for a fixed $(\beta, n)$, 
there exists a finite set of walls 
in the space of stability parameters 
$t \in \mathbb{R}$.
By taking $t_1>t_2>\cdots>t_N>0$ which 
do not lie on walls, we have 
the sequence of moduli spaces
\begin{align}\label{intro:dmmp}
P_n(X, \beta) \dashrightarrow P_n^{t_1}(X, \beta) \dashrightarrow
P_n^{t_2}(X, \beta) \dashrightarrow \cdots \dashrightarrow
P_n^{t_N}(X, \beta). 
\end{align}
If we denote by $P_{n, \beta}^t \in \mathbb{Z}$ the corresponding 
DT invariant for a generic $t$, the following 
wall-crossing formula for $t_i>t_j$ is 
proved in~\cite{Tolim2, Tsurvey}
\begin{align}\label{intro:PT:WCF}
\sum_{n, \beta}P_{n, \beta}^{t_i}q^n y^{\beta}
=\prod_{t_j<n/H \cdot \beta<t_i}
\exp((-1)^{n-1} n N_{n, \beta}q^n y^{\beta}) 
\sum_{n, \beta}P_{n, \beta}^{t_j}q^n y^{\beta},
\end{align}
where $N_{n, \beta} \in \mathbb{Q}$ is the 
generalized DT invariant~\cite{JS} counting 
one dimensional semistable sheaves on $X$ 
with numerical class 
$(\beta, n)$. 
The above wall-crossing formula is relevant in showing 
the rationality of the generating series of PT 
invariants~\cite{Tolim2, Tsurvey}. 

Furthermore in~\cite{Toddbir}, it is proved that the
above sequence is a d-critical minimal model program, 
that is 
a d-critical analogue of usual minimal model program in birational 
geometry~\cite{MR1658959}. 
In particular we have the inequalities in terms 
of virtual canonical line 
bundles
\begin{align*}
P_n(X, \beta) >_{K^{\rm{vir}}}
P_n^{t_1}(X, \beta)>_{K^{\rm{vir}}} \cdots >_{K^{\rm{vir}}}
 P_n^{t_N}(X, \beta). 
\end{align*}
Then as we discussed in Subsection~\ref{intro:motivation}, we formulate 
the following conjecture: 
\begin{conj}\emph{(Conjecture~\ref{conj:DT/PT}, Conjecture~\ref{conj2})}\label{intro:conj:PTwcf}
There exist fully-faithful functors
\begin{align*}
\mathcal{DT}^{\mathbb{C}^{\ast}}(P_n^{t_N}(X, \beta))
\hookrightarrow \cdots \hookrightarrow
\mathcal{DT}^{\mathbb{C}^{\ast}}(P_n^{t_1}(X, \beta)) 
\hookrightarrow 
\mathcal{DT}^{\mathbb{C}^{\ast}}(P_n(X, \beta)). 
\end{align*}
\end{conj}

The above conjecture also makes sense when $S$ is non-compact
(see Remark~\ref{rmk:non-compact}). 
Using window theorem for GIT quotients 
developed in~\cite{MR3327537, MR3895631}, we 
show the above conjecture in the simplest CY 3-fold
\begin{align*}
X=\mathrm{Tot}_{\mathbb{P}^1}(\oO_{\mathbb{P}^1}(-1)
\oplus \oO_{\mathbb{P}^1}(-1))=\mathrm{Tot}_S(\omega_S)
\end{align*}
where $S \to \mathbb{C}^2$ is the blow-up at the origin. 

\begin{thm}\emph{(Theorem~\ref{thm:catDTPT:loc}, 
Theorem~\ref{thm:PTWCF:loc})}\label{intro:thm:simplest}
Conjecture~\ref{intro:conj:dtpt}, 
Conjecture~\ref{intro:conj:PTwcf} are true for 
the blow-up at the origin $S \to \mathbb{C}^2$. 
\end{thm}

\subsection{Three approaches}
In this paper, we pursue and develope
three approaches toward the conjectures stated in 
the previous section: 
linear Koszul duality, window theorem for DT categories and 
categorified Hall products. 
\subsubsection{Wall-crossing via linear Koszul duality}
When the curve class $\beta$ is irreducible, 
there is only one wall $t_0 \in \mathbb{R}$
for the moduli space $P_n^t(X, \beta)$. 
In this case, the Serre duality gives an 
isomorphism 
$P_n^{t_{0-}}(X, \beta) \cong P_{-n}(X, \beta)$. 
	The wall-crossing diagram at $t=t_0$ is 
	\begin{align}\label{intro:irred:wall}
	\xymatrix{
P_n(X, \beta) \ar[rd] & & P_{-n}(X, \beta)\ar[ld] \\
& M_n(X, \beta) &	
}
\end{align}
where $M_n(X, \beta)$ is the moduli space of 
one dimensional stable sheaves on $X$ 
with numerical class $(\beta, n)$. 
The above diagram is a d-critical flip (flop) for $n>0$ ($n=0$). 

Our observation is that, locally on $M_n(X, \beta)$, the above 
diagram is described by linear Koszul dual pairs studied by 
Mirkovi\'{c}-Riche~\cite{MrRi, MrRi2}. 
Let $Y$ be a smooth affine scheme 
and $\eE=(\eE^{-1} \to \eE^0)$ be a 
two term complex of vector bundles on $Y$
with $\rank(\eE) \ge 0$.  
A Koszul dual pair in~\cite{MrRi, MrRi2}
is given by 
\begin{align*}
	(Y^{\dag}, Y^{\sharp}), \ Y^{\dag}=\Spec S(\eE), \ 
	Y^{\sharp}=\Spec S(\eE^{\vee}[1]). 
	\end{align*}
The Koszul duality equivalence in~\cite{MrRi, MrRi2}
is a derived equivalence
\begin{align}\label{intro:linear:eq}
	\Dbc([Y^{\dag}/\C])^{\rm{op}} \stackrel{\sim}{\to}
	\Dbc([Y^{\sharp}/\C]). 
	\end{align}
It will turn out that the diagram (\ref{intro:irred:wall}) is locally modeled 
by the diagram 
\begin{align}\label{intro:kpair}
	\xymatrix{
\mathbb{P}(Y^{\dag}) \ar[rd]  & & \ar[ld] \mathbb{P}(Y^{\sharp}) \\
& Y. &	
}
	\end{align}

We will globalize the above equivalence (\ref{intro:linear:eq})
by allowing $Y$ to be a derived stack. 
Then we show some semiorthogonal decompositions of 
derived categories of both sides in (\ref{intro:linear:eq}), 
where derived categories of projectivizations in (\ref{intro:kpair}) 
appear as semiorthogonal summands. 
By comparing these semi-orthogonal decompositions 
under the equivalence (\ref{intro:linear:eq}), 
we will show the semiorthogonal decomposition
of the derived category of $\mathbb{P}(Y^{\dag})$
where that of $\mathbb{P}(Y^{\sharp})$ appears as a semiorthogonal summand. 
By applying this result to the diagram (\ref{intro:irred:wall}), 
we
show that Conjecture~\ref{intro:conj:PTwcf} is true 
for the above wall-crossing, which holds for \textit{any}
surface $S$. 
\begin{thm}\emph{(Theorem~\ref{thm:duality}, Theorem~\ref{thm:PT:irredu})}\label{intro:thm:PT:irredu}
Suppose that $\beta$ is irreducible and $n\ge 0$. 
Then Conjecture~\ref{intro:conj:PTwcf} holds. 
Moreover there is an equivalence 
\begin{align*}
\mathcal{DT}^{\mathbb{C}^{\ast}}(P_n^{t_{0-}}(X, \beta))
\stackrel{\sim}{\to} \mathcal{DT}^{\mathbb{C}^{\ast}}(P_{-n}(X, \beta))
\end{align*} 
together with a semiorthogonal decomposition
\begin{align}\label{intro:sod}
\mathcal{DT}^{\mathbb{C}^{\ast}}(P_n(X, \beta))
=\langle \Upsilon_{-n+1}, \ldots, \Upsilon_0, \mathcal{DT}^{\mathbb{C}^{\ast}}(P_{-n}(X, \beta))\rangle. 
\end{align}
Here each $\Upsilon_{\lambda}$ is equivalent to a 
$\lambda$-twisted DT category 
$\mathcal{DT}^{\mathbb{C}^{\ast}}(\mM_n(X, \beta))_{\lambda}$, where 
$\mM_n(X, \beta)$ is the moduli stack 
of one dimensional stable
sheaves on $X$ with numerical class $(\beta, n)$. 
\end{thm}

The above result together with Proposition~\ref{intro:prop:comparison}
recover the wall-crossing formula of numerical 
PT invariants for irreducible curve classes proved in~\cite{MR2552254}
(which is a special case of~(\ref{intro:PT:WCF})), 
by taking the 
Euler characteristics of the periodic cyclic homologies of both sides of 
(\ref{intro:sod}) (see Lemma~\ref{rmk:wcf}). 
In~\cite{TodDK}, a semiorthogonal decomposition 
similar to (\ref{intro:sod}) is proved 
for derived categories of coherent sheaves on $P_n(X, \beta)$
assuming that it is non-singular. 
In our situation $P_n(X, \beta)$ is not necessary non-singular, 
and even if it is non-singular the triangulated category
$\mathcal{DT}^{\mathbb{C}^{\ast}}(P_n(X, \beta))$ may not 
be equivalent to $D^b_{\rm{coh}}(P_n(X, \beta))$. 
Therefore the 
argument similar to~\cite{TodDK} does not apply for 
Theorem~\ref{intro:thm:PT:irredu}. 

As we mentioned above, the 
wall-crossing phenomena (\ref{intro:dmmp})
was essential in showing the rationality of the PT 
generating series. Using 
Theorem~\ref{intro:thm:PT:irredu}, 
we have the following categorical analogue 
of the rationality for irreducible curve classes: 

\begin{cor}\label{intro:cor:rational}\emph{(Corollary~\ref{cor:PT:rational})}
Suppose that $\beta$ is irreducible. 
Then the generating series 
\begin{align*}
P_{\beta}^{\rm{cat}}(q) \cneq 
\sum_{n \in \mathbb{Z}}
[\mathcal{DT}^{\mathbb{C}^{\ast}}(P_n(X, \beta))]
q^n \in K(\Delta \mathchar`-\mathop{\mathrm{Cat}})\lgakko q \rgakko
\end{align*}
lies in 
$K(\Delta \mathchar`-\mathop{\mathrm{Cat}}) \otimes_{\mathbb{Z}}
\mathbb{Q}(q)^{\rm{inv}}$. 
Here $K(\Delta \mathchar`-\mathop{\mathrm{Cat}})$ is the 
Grothendieck group of triangulated categories 
and $\mathbb{Q}(q)^{\rm{inv}} \subset \mathbb{Q}\lgakko q \rgakko$
is the subspace of 
rational functions invariant under $q \leftrightarrow 1/q$. 
\end{cor}

\subsubsection{Window theorem for DT categories}
Our next approach toward 
Conjectures~\ref{intro:conjGV}, \ref{intro:conj:dtpt}, 
\ref{intro:conj:PTwcf} 
is to establish \textit{window theorem}
for DT categories and apply it. 
The original window theorem is developed 
for 
derived categories of GIT quotients~\cite{MR3327537, MR3895631}, 
and stated as follows. 
Let $Y$ be a smooth affine variety with an action of a
reductive algebraic group $G$. 
For a given $G$-equivariant line bundle $l$ on $Y$, 
we have the $G$-invariant 
open subset $Y^{l\sss} \subset Y$ of 
$l$-semistable points. 
Then the window theorem 
implies the existence of 
a triangulated subcategory (called \textit{window subcategory})
$\wW([Y/G]) \subset \Dbc([Y/G])$ such that the 
composition
\begin{align*}
	\wW([Y/G]) \hookrightarrow \Dbc([Y/G]) \twoheadrightarrow 
	\Dbc([Y^{l\sss}/G])
\end{align*}
is an equivalence. 
The window theorem has been used to show 
several derived equivalences of birational varieties 
given as variations of GIT quotients, showing many 
evidence for original D/K equivalence conjecture. 

We expect that its analogy holds for DT categories. 
Namely we expect 
the existence of
a subcategory 
$\wW(\fM^{\C\rig}_{S, \rm{qc}}) \subset \Dbc(\fM^{\C\rig}_{S, \rm{qc}})$
such that the composition 
\begin{align*}
\wW(\fM^{\C\rig}_{S, \rm{qc}}) \hookrightarrow \Dbc(\fM^{\C\rig}_{S, \rm{qc}}) \to 
\dDT^{\C}(M_X^{\sigma})
\end{align*}
is an equivalence. 
Once such a subcategory exists, 
we then try to compare 
them for different $\sigma$ inside 
$\Dbc(\fM^{\C\rig}_{S, \rm{qc}})$, which may 
yield the above conjectures.

We establish window theorem for DT categories, assuming 
the existence of good moduli space of $\fM^{\C\rig}_{S, \rm{qc}}$. 
More generally we prove such a statement for 
DT categories associated with $(-1)$-shifted cotangent 
over quasi-smooth derived stacks
as follows. 
\begin{thm}\label{intro:thm:equivalence}\emph{(Theorem~\ref{thm:equivalence})}
	For a quasi-smooth and QCA derived stack $\fM$, suppose 
	that 
	$\mM=t_0(\fM)$ admits a good moduli space, 
	a symmetric structure $\bS$ of $\fM$ is given, 
	and $\fM$  
	satisfies formal neighborhood theorem. 
	Let us take $l, \delta \in \Pic(\mM)_{\mathbb{R}}$
	such that $\delta$ is $l$-generic. 
	Then there exists a triangulated subcategory 
	$\wW_{\delta}^{\intt/\bS}(\fM) \subset \Dbc(\fM)$
	such that, for the $l$-semistable locus 
	$\nN^{l\sss} \subset \nN \cneq t_0(\Omega_{\fM}[-1])$, 
	the composition
	\begin{align*}
		\wW_{\delta}^{\intt/\bS}(\fM) \hookrightarrow 
		\Dbc(\fM) \to 
		\dDT^{\C}(\nN^{l\sss})
	\end{align*}
	is fully-faithful, 
	which is an equivalence if $l$ is $\bS$-generic and 
	compatible with 
	$\bS$. 
\end{thm}
Here several 
notions `symmetric structure', `formal neighborhood theorem', 
`$\bS$-generic', `$l$-generic', `compatible with $\bS$' appear 
in the statement of the above theorem. 
We do not explain these technical notions here 
and refer to 
Subsection~\ref{subsec:sym} for details. 
We will apply Theorem~\ref{intro:thm:equivalence}
to prove Conjecture~\ref{intro:conjGV} under some assumption. 

\begin{thm}\emph{(Theorem~\ref{thm:surface})}\label{intro:thm:surface}
	In the setting of 
	Conjecture~\ref{intro:conjGV}, suppose that the following condition holds
	\begin{align}\notag
		\mM_{X}^{\sigma}(v) \subset 
		\pi_{\ast}^{-1}(\mM_{S}^{\sigma}(v)).
	\end{align}
Then 
	Conjecture~\ref{intro:conjGV} holds. 
\end{thm}
The above result 
also implies some derived equivalences of 
moduli spaces of one dimensional stable sheaves on $S$ (not on $X$). 
For example
as a corollary of Theorem~\ref{intro:thm:surface},
we show that derived moduli spaces of one dimensional 
stable sheaves with reduced supports are derived equivalent
(see Corollary~\ref{cor:reduced}). 
The result of Theorem~\ref{intro:thm:surface} is also 
used to show 
a derived equivalence of 
Sacca's Calabi-Yau manifolds~\cite{Sacca} 
obtained as moduli spaces of one dimensional stable sheaves on 
Enriques surfaces
for different polarizations (see~Corollary~\ref{cor:Enriques}), 
which gives an evidence of the usual D/K 
equivalence conjecture. 
Another application of Theorem~\ref{intro:thm:surface} is the proof of 
Conjecture~\ref{intro:conj:dtpt} for reduced curve classes. 
\begin{thm}\label{intro:thm:dtpt}
	Conjecture~\ref{intro:conj:dtpt} holds if $\beta$ is a reduced class. 
	\end{thm}

The results and proofs of
Theorem~\ref{intro:thm:equivalence} and 
its applications 
are much inspired by 
ideas and techniques of Halpern-Leistner~\cite{HalpK3, HalpK32}, 
where he proves that Bridgeland moduli spaces of 
stable objects on K3 surfaces are derived equivalent 
under wall-crossing. 
Indeed Theorem~\ref{intro:thm:surface} applied
for (some rigidified version of)
derived moduli stacks of semistable objects on K3 surfaces 
should recover 
the above Halpern-Leistner's result (see Remark~\ref{rmk:K3}), 
though we will not discuss this case in this paper.  
On the other hand, our result 
applies other than moduli stacks of semistable objects on K3 surfaces. 
The latter moduli stacks 
are 0-shifted symplectic, 
and this fact 
is essential in the proof of derived equivalence in~\cite{HalpK3, HalpK32}. 
In our situation the derived stack $\fM$ is not necessary 
0-shifted symplectic, but 
is equipped with a 
\textit{symmetric structure} $\bS$ and 
$l \in \Pic(\mM)_{\mathbb{R}}$ \textit{compatible with} 
$\bS$. 
A symmetric structure $\bS$ is a choice of 
direct sum decompositions 
of $\Aut(x)$-representations
\begin{align}\label{intro:sym}
	\hH^0(\mathbb{T}_{\fM}|_{x}) \oplus \hH^1(\mathbb{T}_{\fM}|_{x})^{\vee}
	=\bS_x \oplus \bU_x
\end{align}
at each closed point $x \in \mM$
such that $\bS_x$ is a symmetric $\Aut(x)$-representation. 
An element $l \in \Pic(\mM)_{\mathbb{R}}$ is compatible with $\bS$ if 
the $l_x$-semistable locus in the LHS of (\ref{intro:sym}) 
is the pull-back of the $l_x$-semistable locus in $\bS_x$. 
We will use the symmetric structure $\bS$ together with an
auxiliary data 
$\delta \in \Pic(\mM)_{\mathbb{R}}$ to define 
the \textit{intrinsic window subcategory}
(see Definition~\ref{defi:intwin})
\begin{align}\label{intro:int}
	\wW_{\delta}^{\intt/\bS}(\fM) \subset \Dbc(\fM)
\end{align}
which gives a desired window subcategory. 

Locally on the good moduli space, the intrinsic
window subcategory is constructed so that 
it coincides with the magic window 
subcategory on the derived category of factorizations 
via Koszul duality. 
Here the 
magic window 
subcategory is defined by Halpern-Leistner and Sam~\cite{HLKSAM}
which gives stability independent descriptions of 
window subcategories in~\cite{MR3327537, MR3895631}, 
and is itself   
based on combinatorial arguments by
\v{S}penko and Van den Bergh~\cite{MR3698338}.
So the subcategory (\ref{intro:int}) is interpreted as a gluing of 
magic window subcategories in the Koszul dual side, rather
than those on $\fM$ itself.

\subsubsection{Categorified Hall products}
We also give an approach toward the previous conjectures 
using Porta-Sala's two dimensional categorified Hall products~\cite{PoSa}. 
The Porta-Sala categorified Hall product
is the functor
\begin{align*}
	\Dbc(\fM_S) \times \Dbc(\fM_S) \to \Dbc(\fM_S)
	\end{align*}
defined through the derived moduli stack of 
short exact sequences of coherent sheaves
on $S$. 
It gives a categorification of 
cohomological Hall algebras constructed by 
Kapranov-Vasserot~\cite{KaVa2}. 
There are also recent works giving 
categorifications of geometric representations
through the correspondences similar to the 
categorified Hall products~\cite{Negut, Negut2, Zhao, Zhao2, Tudor}. 

In~\cite{THtype}, 
we proved that Porta-Sala categorified Hall products descend to 
products on DT categories
\begin{align*}
	\dDT^{\C}(\mM_X^{\sigma}) \times \dDT^{\C}(\mM_X^{\sigma}),
	\to \dDT^{\C}(\mM_X^{\sigma})
	\end{align*}
which also acts on MNOP/PT categories, and 
should  
categorify critical cohomological Hall products for local surfaces.
We apply the above categorified Hall products to 
describe conjectural semiorthogonal decompositions
associated with wall-crossing of 
MNOP/PT moduli spaces, which also give ansatz for d-critical analogue of D/K equivalence conjecture. 

Let $\tT_n(X, \beta)$ be the moduli stack of 
pairs $(F, s)$, where $F$ is a compactly supported one dimensional 
coherent 
sheaf on $X$ and $s \colon \oO_X \to F$ is surjective in dimension one. 
The moduli stack $\tT_n(X, \beta)$ is nothing but the 
moduli stack of semistable objects at MNOP/PT wall, 
and its good moduli space coincides with $T_n(X, \beta)$ in the 
diagram (\ref{intro:irred:wall}). 
Also let $\mM_i(X)$ be the moduli stack of zero dimensional 
sheaves on $X$ with length $i$. 
The following is the main result by this approach
specialized to the case of MNOP/PT wall, 
which gives another method to construct window 
subcategories.

\begin{thm}\label{conj:intro:dtpt}\emph{(Theorem~\ref{thm:pfconj}, Theorem~\ref{thm:conj2})}
	Let $k_I, k_P \colon \mathbb{Z}_{\ge 1} \to \mathbb{R}$ be
	maps. 
	If $\beta$ is a reduced class, there 
	exist semiorthogonal decompositions
	\begin{align*}
		\dDT^{\C}(\tT_n(X, \beta)) &=
		\langle \ldots, \Upsilon_{>k^I(2)}^{I, 2}, \Upsilon_{>k^I(1)}^{I, 1}, 
		\wW^I_{k^I}, \Upsilon_{\le k^I(1)}^{I, 1}, \Upsilon_{\le k^I(2)}^{I, 2}, \ldots \rangle \\
		&=	\langle \ldots, \Upsilon_{<k^P(2)}^{P, 2}, \Upsilon_{<k^P(1)}^{P, 1}, 
		\wW^P_{k^P}, \Upsilon_{\ge k^P(1)}^{P, 1}, \Upsilon_{\ge k^P(2)}^{P, 2}, \ldots \rangle
	\end{align*}
	satisfying the following: 
	
	(i)
	There exist semiorthogonal decompositions
	\begin{align*}
		&\Upsilon_{>k}^{I, i}
		=\langle \ldots,  \Upsilon_{\lfloor k \rfloor +2}^{I, i}, \Upsilon_{\lfloor k \rfloor +1}^{I, i} \rangle, \ 
		\Upsilon_{\le k}^{I, i}
		=\langle \Upsilon_{\lfloor k \rfloor}^{I, i}, \Upsilon_{\lfloor k \rfloor-1}^{I, i}, \ldots \rangle, \\
		&\Upsilon_{<k}^{P, i}
		=\langle \ldots,  \Upsilon_{\lceil k \rceil -2}^{P, i}, \Upsilon_{\lceil k \rceil -1}^{P, i} \rangle, \ 
		\Upsilon_{\ge k}^{P, i}
		=\langle \Upsilon_{\lceil k \rceil}^{P, i}, \Upsilon_{\lceil k \rceil+1}^{P, i}, \ldots \rangle
	\end{align*}
	with equivalences induced by Porta-Sala 
	categorified Hall products
	\begin{align*}
		&\ast_j \colon \dDT^{\C}(I_{n-i}(X, \beta)) \otimes 
		\dDT^{\C}(\mM_i(X))_{j} \stackrel{\sim}{\to} \Upsilon_{j}^{I, i}, \\
		&\ast_j \colon \dDT^{\C}(\mM_i(X))_{j} \otimes 
		\dDT^{\C}(P_{n-i}(X, \beta))
		\stackrel{\sim}{\to} \Upsilon_{j}^{P, i}.
	\end{align*}
	
	(ii) The following composition functors are equivalences
	\begin{align*}
		&\wW^I_{k^I} \hookrightarrow \dDT^{\C}(\tT_n(X, \beta)) \twoheadrightarrow 
		\dDT^{\C}(I_n(X, \beta)), \\
		&\wW^P_{k^P} \hookrightarrow \dDT^{\C}(\tT_n(X, \beta)) \twoheadrightarrow 
		\dDT^{\C}(P_n(X, \beta)).
	\end{align*}
	
	(iii) By setting $k^I(i)=i/2$ and $k^P(i)=0$, we have 
	$\wW^P_{k^P} \subset \wW^I_{k^I}$. 
	In particular we have the fully-faithful functor
	\begin{align}\label{intro:ff:dtpt}
		\dDT^{\C}(P_n(X, \beta)) \hookrightarrow 
		\dDT^{\C}(I_n(X, \beta)).
	\end{align}
\end{thm}

Note that the fully-faithful functor (\ref{intro:ff:dtpt}) implies 
Conjecture~\ref{intro:conj:dtpt} for reduced curve 
classes, so the above result 
gives another proof of Theorem~\ref{intro:thm:dtpt}. 
On the other hand 
the result itself is conjectured to hold for any curve class, 
so the categorified Hall product approach 
may give an ansatz applied for any curve class. 
The above approach can be also extended
to wall-crossing of D0-D2-D6 bound states. 
We will formulate a similar semiorthogonal 
decomposition for wall-crossing of D0-D2-D6 bound 
states, and prove it for reduced curve class $\beta$
and $t>m(\beta)$ where $m(\beta)$ depends only on $\beta$. 

\begin{thm}\label{thm:intro:pm}\emph{(Theorem~\ref{thm:pfconj}, Theorem~\ref{thm:conj2})}
	Let $\beta$ be a reduced class. 
	Then there is $m(\beta)>0$ which only depends on $\beta$ such that 
	for $t>m(\beta)$ 
	there exists a fully-faithful functor 
	\begin{align}\label{ff:intro:pm}
		\dDT^{\C}(P_n^{t_-}(X, \beta)) \hookrightarrow 
		\dDT^{\C}(P_n^{t_+}(X, \beta)), 
	\end{align}
i.e. Conjecture~\ref{intro:conj:PTwcf} holds for reduced curve classes 
and $t_i>m(\beta)$. 
\end{thm}

We note that $m(\beta)$ only depends on $\beta$, so for $n\gg 0$ there 
exist many walls in the region $t>m(\beta)$ where we can apply Theorem~\ref{thm:intro:pm}. 
If furthermore $t \in \mathbb{R}_{>0}$ is a simple wall, 
then we can describe the semiorthogonal complement of 
the fully-faithful functor (\ref{ff:intro:pm}), which categorifies wall-crossing formula of numerical DT
invariants (\ref{intro:PT:WCF}) in~\cite{Tolim2, Thall}.  
\begin{thm}\label{intro:thm:simple}\emph{(Theorem~\ref{thm:simple})}
	Under the assumption of Theorem~\ref{thm:intro:pm}, suppose furthermore
	that $t>m(\beta)$ is a simple wall. 
	Then there exists a semiorthogonal decomposition 
	\begin{align}\notag
		&\mathcal{DT}^{\C}(P_n^{t_+}(X, \beta)) \\
		\notag	&=\left \langle  \bigoplus_{\mathbf{d}}
		\Upsilon^{-, \mathbf{d}}_{[-\frac{1}{2}\beta_1 \beta_2-\frac{1}{2}n_2, 
			-\frac{1}{2}\beta_1 \beta_2 )}, \mathcal{DT}^{\C}(P_n^{t_-}(X, \beta)), 
		\bigoplus_{\mathbf{d}}
		\Upsilon^{-, \mathbf{d}}_{[-\frac{1}{2}\beta_1 \beta_2, 
			-\frac{1}{2}\beta_1 \beta_2+\frac{1}{2}n_2 )}
		\right \rangle. 	
	\end{align}
	Here $\mathbf{d}$ is a decomposition $(\beta, n)=(\beta_1, n_1)+(\beta_2, n_2)$ 
	with $n_2/(\beta_2 \cdot H)=t$, 
	and 
	$\Upsilon_{[a, b)}^{-, \mathbf{d}}$
	admits semiorthogonal decomposition 
	\begin{align*}
		\Upsilon_{[a, b)}^{-, \mathbf{d}}=\left\langle \Upsilon_{\lceil a \rceil}^{-, \mathbf{d}}, \Upsilon_{\lceil a \rceil +1}^{-, \mathbf{d}}, 
		\ldots, \Upsilon_{\lceil b \rceil -1}^{-, \mathbf{d}} \right\rangle
	\end{align*} 
	with equivalences induced by Porta-Sala categorified Hall products
	\begin{align*}
		\ast_j \colon \dDT^{\C}(\mM_{n_2}^H(X, \beta_2))_j
		\otimes \dDT^{\C}(P_{n_2}^{t_-}(X, \beta_1)) \stackrel{\sim}{\to}
		\Upsilon_j^{-, \mathbf{d}}. 
	\end{align*}
	Here $\mM_{n_2}^H(X, \beta_2)$ is the moduli stack of one dimensional 
	$H$-semistable sheaves on $X$ with numerical class $(\beta_2, n_2)$. 
\end{thm}

\subsection{Organization of this paper}
\subsubsection{Plan of the paper}
In Section~\ref{sec:Koszul},
we recall the notion of singular supports
for coherent sheaves on quasi-smooth affine derived schemes, 
the derived categories factorizations, Koszul duality equivalence, 
relation of singular supports with usual supports of factorizations 
under Koszul duality, 
and prove several functoriality under Koszul duality.  

In Section~\ref{sec:catDT:shift}, we define DT categories 
for $(-1)$-shifted cotangent stacks over quasi-smooth derived stacks 
as Verdier quotients, and prove their fundamental properties. 
We then
define DT categories for (semi)stable sheaves on local surfaces, 
and formulate wall-crossing equivalence for DT categories of 
one dimensional stable sheaves.  

In Section~\ref{sec:catMNOP}, we focus on D0-D2-D6 bound states
and introduce MNOP category, PT category and DT categories for 
stable D0-D2-D6 bound states. 
An alternative description of D0-D2-D6 bound states plays 
a key role here. 
We then formulate categorical wall-crossing conjecture 
for these DT categories and prove it for local $(-1, -1)$-curve. 
We also concentrate on wall-crossing for irreducible curve class, 
and study the conjecture in this case via linear Koszul duality. 

In Section~\ref{sec:window:DT}, we prove window theorem for 
DT categories, and use it to prove conjecture on categorical 
wall-crossing equivalence 
of DT categories for one dimensional stable sheaves, 
under the some assumption of preservation of
semistability by push-forward. 

In Section~\ref{sec:cat:hall}, we use categorified Hall products 
to describe conjectural semiorthogonal decomposition of DT categories 
at a wall, and prove it under some assumption, including 
MNOP/PT wall for reduced classes. 

In Section~\ref{sec:technical}, we prove several technical results 
postponed in the previous sections. 

\subsubsection{Notation and convention}\label{subsec:notation}
In this paper, all the schemes or (derived) stacks
are locally of finite presentation over 
$\mathbb{C}$, except formal fibers along with 
good moduli space morphisms (e.g. see Subsection~\ref{subsec:fib}).  
For a scheme or derived stack $Y$
and a quasi-coherent sheaf $\fF$ on it, we denote by 
$S_{\oO_Y}(\fF)$ its symmetric product
$\oplus_{i\ge 0} \mathrm{Sym}^i_{\oO_Y}(\fF)$. 
We omit the subscript $\oO_Y$ if it is clear from the context. 
We denote by $\mathbb{D}_Y$ the derived dual functor
$\mathbb{D}_Y(-) \cneq \dR \hH om_{\oO_Y}(-, \oO_Y)$. 
For a derived stack $\mathfrak{M}$, we always 
denote by $t_0(\mathfrak{M})$ the underived
stack given by the truncation. 
For a triangulated category $\dD$ and 
a set of objects $\sS \subset \dD$, 
we denote by $\langle \sS \rangle_{\rm{ex}}$ the 
extension closure, i.e. the smallest extension closed 
subcategory which contains $\sS$. 
For an algebraic group $G$ which acts on $Y$, 
we denote by $[Y/G]$ the associated 
quotient stack. 

For
a triangulated category $\dD$ and its 
triangulated subcategory $\dD' \subset \dD$, 
we denote by 
$\dD/\dD'$
its Verdier quotient. 
This
is the triangulated category 
whose objects coincide with 
objects in $\dD$, and 
for $E, F \in \dD$ 
the set of morphisms $\Hom_{\dD/\dD'}(E, F)$
is the 
equivalence classes of 
roof diagrams in $\dD$
\begin{align*}
\xymatrix{
E & \ar[l]_-{u} G \ar[r] & F,
} 
\quad
\Cone(u) \in \dD'. 
\end{align*}
Then the sequence of 
exact functors of triangulated categories 
$\dD' \to \dD \to \dD/\dD'$
is called a localization sequence. 
We denote by $\dD^{\rm{cp}} \subset \dD$ the 
subcategory of compact objects. 
A subcategory $\dD' \subset \dD$ is called dense 
if any object in $\dD$ is a direct summand of an object in $\dD'$. 

For a smooth projective surface $S$, we denote by 
$\mathrm{NS}(S)$ the Neron-Severi group of $S$. 
An effective class $\beta \in \mathrm{NS}(S)$ is 
called reduced (resp.~irreducible) if any 
effective divisor in $S$ with class $\beta$ is a 
reduced divisor (resp.~irreducible divisor).  
For a coherent sheaf $E$ on $S$ whose supports 
had dimension less than or equal to one, we denote 
by $l(E)$ the fundamental one cycle of $E$. 

We often take limits or ind-completions for dg-categories, 
and then take its homotopy categories. 
Here we note that a limit of triangulated categories cannot be 
defined in general, and in order to take it we have to replace 
triangulated categories with their 
dg-enhancements and define the limit 
in the $\infty$-category of dg-categories as in (\ref{L:limit}) 
(see~\cite{MR3701353, MR2762557}). 
By abuse of notation, we use the following notation for the limit. 
For an $\infty$-category $\iI$, 
let $\{\dD_i^{\rm{dg}}\}_{i\in \iI}$
be the $\iI$-diagram of dg-categories
$\dD_i^{\rm{dg}}$ for $i\in \iI$, 
and denote by $\dD_i=\mathrm{Ho}(\dD_i^{\rm{dg}})$ 
the homotopy category of $\dD_i^{\rm{dg}}$. 
Then we sometimes write
\begin{align*}
	\lim_{i\in \iI}\dD_i \cneq \
	\mathrm{Ho}\left( \lim_{i\in \iI} \dD_i^{\rm{dg}} \right). 
\end{align*}
Similarly for a dg-category $\dD^{\rm{dg}}$
with $\dD=\mathrm{Ho}(\dD^{\rm{dg}})$, 
we write 
$\Ind \dD \cneq \mathrm{Ho}(\Ind \dD^{\rm{dg}})$
where $\Ind \dD^{\rm{dg}}$ is a dg-categorical 
ind-completion of $\dD^{\rm{dg}}$ (denoted as $\widehat{\dD}^{\rm{dg}}$
in~\cite[Section~7]{Todg}). 
When we discuss induced functors on limits or ind-completions, 
we implicitly take these functors on dg-enhancements
and finally take homotopy categories. 

\subsection{Acknowledgements}
The author is grateful to Tomoyuki Abe, Yuki Hirano
and 
Daniel Halpern-Leistner
for valuable discussions, 
and Yalong Cao, Flancesco Sala, Mauro Porta, Aurelio Carlucci
for comments on the first version of this paper. 
The author is supported by World Premier International Research Center
Initiative (WPI initiative), MEXT, Japan, and Grant-in Aid for Scientific
Research grant (No.~19H01779) from MEXT, Japan.

\section{Koszul duality equivalence}\label{sec:Koszul}
The notion of singular supports was 
introduced by Arinkin-Gaitgsory~\cite{MR3300415}
(which is itself based on 
the earlier work by Benson-Iyengar-Krause~\cite{MR2489634})
in order to give a categorical formulation of geometric 
Langlands conjecture. 
In this section, we first
review the singular support theory for 
(ind)coherent sheaves on quasi-smooth affine derived 
schemes. 
For a quasi-smooth affine 
derived scheme $\fU$ and $F \in \Dbc(\fU)$, 
its singular support is a conical closed subspace
\begin{align*}
	\mathrm{Supp}^{\rm{sg}}(F) \subset t_0(\Omega_{\fU}[-1])
	\end{align*}
where $\Omega_{\fU}[-1] \to \fU$ is the $(-1)$-shifted cotangent 
over $\fU$. 
If $F$ is a perfect complex, then 
its singular support is contained in the zero section of 
$t_0(\Omega_{\fU}[-1]) \to \fU$, but 
if $F$ is not perfect then its singular support lies beyond the 
zero section. 
So the singular support may be interpreted as a measure how 
a given coherent sheaf is far from a perfect complex. 

Our viewpoint for singular supports is to interpret 
them as usual supports for factorizations 
via Koszul duality. The derived category of factorizations 
$\MF_{\coh}(Y, w)$ 
is determined by a pair $(Y, w)$ for 
a smooth scheme $Y$ together with a function $w \colon Y \to \mathbb{C}$, 
which is 
introduced and studied in~\cite{MR3270588, MR3366002}
and generalizes 
Orlov's categories of matrix factorizations~\cite{Orsin}. 
Our key tool is a Koszul duality equivalence 
proved by several people~\cite{MR3071664, MR2982435, MR3631231, ObRoz}
which claims that, if $\fU$ is represented by a 
derived zero locus of a section of a vector bundle $V \to Y$
on a smooth scheme $Y$, 
then there is a function $w \colon V^{\vee} \to \mathbb{C}$
and an equivalence 
\begin{align*}
	\Phi \colon \Dbc(\fU) \stackrel{\sim}{\to} \MF_{\coh}^{\C}(V^{\vee}, w), 
	\end{align*}
where the right hand side is the derived category 
of $\C$-equivariant factorizations 
for $(V^{\vee}, w)$. 
We will see that, under the above equivalence $\Phi$, 
the singular support for $F \in \Dbc(\fU)$ is the same 
as the usual support for $\Phi(F)$. 
This fact is briefly claimed in~\cite[Section~H]{MR3300415},  
and we will give its more details. 

For a conical closed subset
$Z \subset t_0(\Omega_{\fU}[-1])$, 
we define $\cC_{Z} \subset \Dbc(\fU)$ to be the 
subcategory of objects whose singular supports are
contained in $Z$. 
The above interpretation of 
singular supports yields the equivalence
\begin{align*}
	\Phi \colon 
	\Dbc(\fU)/\cC_{Z} \stackrel{\sim}{\to}
	\MF_{\coh}^{\C}(V^{\vee} \setminus Z, w). 
	\end{align*}
The quotient category in the left hand side will be
our local model for DT categories for local surfaces, by taking 
$Z$ to be unstable locus with respect to some stability condition.
The above construction also works under equivariant setting, and we will 
discuss with the presence of $G$-action for an affine algebraic group $G$. 

In the later sections, we will use the above Koszul duality equivalence 
to prove several statements on singular supports using 
the theory of derived categories of factorizations. 
In particular, it will be useful to 
compare several functorial properties of 
 $\Dbc(\fU)$ for affine derived schemes $\fU$, 
 or more generally $\Dbc([\fU/G])$ for a $G$-action of $\fU$, 
  under Koszul duality equivalence. 
 For example we will see that the push-forward functor 
 associated with a morphism $\mathbf{f} \colon \fU_1\to \fU_2$
 is described by the composition of pull-back/push-forward in the 
 Koszul dual side. 
 
 The organization of this section is as follows. 
 In Section~\ref{subsec:ssupport}, we recall the 
 notion of singular supports for quasi-smooth affine 
 derived schemes. In Section~\ref{subsec:def:fact}, we review the theory of 
 derived categories of factorizations and their functorial properties. 
 In Section~\ref{subsec:Koszul}, we recall and reprove Koszul duality equivalence 
 in equivariant setting, and compare singular 
 supports with usual supports in the factorization category side. 
 In Section~\ref{subsec:funct}, we compare several 
 functorial properties under Koszul duality.

\subsection{Singular supports of coherent sheaves}\label{subsec:ssupport}
\subsubsection{Local model}
As a local model of our quasi-smooth derived stack, 
we consider the following situation. 
Let $Y$ be a smooth affine $\mathbb{C}$-scheme of finite presentation
and $V$ a vector bundle on it with 
a section $s$
\begin{align}\notag
	\xymatrix{
		V \ar[r] &  \ar@/_10pt/[l]_s Y. 
	}
\end{align}
By abuse of notation, we also denote by $V$ the locally free 
sheaf of local sections of $V \to Y$. 
Then the
contraction by the section 
$s \colon V^{\vee} \to \oO_Y$
determines the Koszul complex
\begin{align}\notag
	\rR(V \to Y, s) \cneq 
	\left(  
	\cdots \to \bigwedge^2 V^{\vee} \stackrel{s}{\to} V^{\vee} \stackrel{s}{\to}
	\oO_Y \to 0 \right)
\end{align}
which is a dg-algebra over $\oO_Y$.
We have the quasi-smooth affine derived scheme
\begin{align}\label{frak:U}
	\mathfrak{U} \cneq \Spec \rR(V\to Y, s).
\end{align}
Note that we have the following derived 
Cartesian square: 
\begin{align}\notag
	\xymatrix{
		\mathfrak{U} \ar[r] \ar[d] \ar@{}[dr]|\square & Y \ar[d]^-{0} \\
		Y \ar[r]^-{s} & V,
	}
\end{align}
i.e. $\mathfrak{U}$ is the derived zero locus of $s$. 
The classical truncation of $\mathfrak{U}$
is the closed subscheme of $Y$
given by  
\begin{align*}
	\uU \cneq t_0(\mathfrak{U})=(s=0) \subset Y.
\end{align*}
Also the cotangent complex of $\mathfrak{U}$ is given by 
\begin{align*}
	\mathbb{L}_{\mathfrak{U}} =(V^{\vee} \stackrel{ds}{\to}\Omega_Y)
	\otimes_{\oO_Y} 
	\oO_{\mathfrak{U}}. 
\end{align*}
Then its  
$(-1)$-shifted cotangent derived scheme of $\mathfrak{U}$
is written as 
\begin{align*}
	\Omega_{\mathfrak{U}}[-1] & \cneq 
	\Spec S_{\oO_{\mathfrak{U}}}(\mathbb{T}_{\mathfrak{U}}[1]) \\
	&=\Spec \rR((V \oplus \Omega_Y)\times_Y V^{\vee} \to V^{\vee}, s \oplus ds).
\end{align*}
Here
$\mathbb{T}_{\mathfrak{U}}$ is the
tangent complex of $\mathfrak{U}$, which is dual 
to its cotangent complex $\mathbb{L}_{\mathfrak{U}}$. 
The dg-algebra 
$\rR((V \oplus \Omega_Y)\times_Y V^{\vee} \to V^{\vee}, s \oplus ds)$
is the Koszul complex over $\oO_{V^{\vee}}$ determined by 
\begin{align}\label{diff:Omega}
	s \oplus ds : 
	V^{\vee} \oplus T_Y \to \oO_Y \oplus V
	\subset \oO_{V^{\vee}}. 
\end{align}

On the other hand, let $w \colon V^{\vee} \to \mathbb{C}$
be the function defined by 
\begin{align}\label{def:w}
	w=s \in \Gamma(Y, V) \subset
	\Gamma(Y, S(V))=
	\Gamma(Y, p_{\ast}\oO_{V^{\vee}}).
\end{align}
Here $p \colon V^{\vee} \to Y$ is the projection. 
Explicitly, the function $w$ is written as
\begin{align*}
	w(x, v)=\langle s(x), v\rangle, \ 
	x \in Y, \ v \in V^{\vee}|_{x}.
\end{align*}
Note that $0_Y^{\ast}w=0$ where $0_Y \colon Y \to V^{\vee}$ is the zero 
section of $p$. 
From the description of the differential (\ref{diff:Omega}),
the classical truncation
of $\Omega_{\mathfrak{U}}[-1]$ 
is the critical locus of $w$,
\begin{align}\label{Crit(w)}
	t_0(\Omega_{\mathfrak{U}}[-1])=
	\Spec S(\hH^1(\mathbb{T}_{\mathfrak{U}}))=
	\mathrm{Crit}(w)
	\subset V^{\vee}.
\end{align}
Moreover the restriction 
of the projection 
$p \colon V^{\vee} \to Y$
to $\mathrm{Crit}(w)$ maps to $\uU$. 
As a summary, we obtain the following diagram
\begin{align}\label{diagram:KZ}
	\xymatrix{
		\mathrm{Crit}(w) \ar@<-0.3ex>@{^{(}->}[r] \ar[d]_-{p_0} & \Omega_{\mathfrak{U}}[-1]
		\ar[d]\ar[r] & V^{\vee} \ar[r]^-{w} \ar[d]_-{p} & \mathbb{C} \\
		\uU \ar@<-0.3ex>@{^{(}->}[r] & \mathfrak{U} \ar[r] & 
		Y\ar@/_10pt/[u]_-{0_Y}. 
	}
\end{align}
In what follows, we will consider weight two 
$\mathbb{C}^{\ast}$-action on the fibers of $p\colon V^{\vee} \to Y$, 
so that $w$ is of weight two as well. 

\subsubsection{Definition of singular supports}\label{subsec:ssuport}
Here we review the theory of singular supports of 
coherent sheaves on $\mathfrak{U}$, developed in~\cite{MR3300415}
following the earlier work~\cite{MR2489634}. 

Let $\mathrm{HH}^{\ast}(\mathfrak{U})$ be the 
Hochschild cohomology
\begin{align*}
	\mathrm{HH}^{\ast}(\mathfrak{U})
	\cneq \Hom_{\mathfrak{U} \times \mathfrak{U}}^{\ast}
	(\Delta_{\ast}\oO_{\mathfrak{U}}, \Delta_{\ast}\oO_{\mathfrak{U}}).
\end{align*}
Here $\Delta \colon \mathfrak{U} \to \mathfrak{U} \times \mathfrak{U}$ is the diagonal. 
Then it is shown in~\cite[Section~4]{MR3300415}
that there 
exists a
canonical map 
$\hH^1(\mathbb{T}_{\mathfrak{U}}) \to \mathrm{HH}^2(\mathfrak{U})$, 
so the map of 
graded rings
\begin{align}\label{ssuport:map}
	S(\hH^1(\mathbb{T}_{\mathfrak{U}})) \to \mathrm{HH}^{2\ast}(\mathfrak{U}) 
	\to \mathrm{Nat}_{D^b_{\rm{coh}}(\mathfrak{U})}(\id, \id[2\ast]). 
\end{align}
Here $\mathrm{Nat}_{D^b_{\rm{coh}}(\mathfrak{U})}(\id, \id[2\ast])$
is the group of natural transformations from $\id$ to $\id[2\ast]$
on $D^b_{\rm{coh}}(\mathfrak{U})$, and the right arrow is defined 
by taking Fourier-Mukai transforms associated with 
morphisms $\Delta_{\ast}\oO_{\mathfrak{U}} \to \Delta_{\ast}\oO_{\mathfrak{U}}[2\ast]$. 
Then by (\ref{Crit(w)}), 
the above maps induce the map for 
each $\fF \in D^b_{\rm{coh}}(\mathfrak{U})$, 
\begin{align}\label{map:V}
	\oO_{\mathrm{Crit}(w)}
	\to \Hom^{2\ast}(\fF, \fF).
\end{align}
The above map 
defines the 
$\mathbb{C}^{\ast}$-equivariant 
$\oO_{\mathrm{Crit}(w)}$-module 
structure on $\Hom^{2\ast}(\fF, \fF)$, 
which is finitely generated  by~\cite[Theorem~4.1.8]{MR3300415}. 
Below a closed subset $Z \subset \mathrm{Crit}(w)$ is called
\textit{conical} if it is closed under the fiberwise 
$\mathbb{C}^{\ast}$-action on $\mathrm{Crit}(w)$. 
\begin{defi}\emph{(\cite{MR3300415, MR2489634})}\label{defi:ssuport}
	For $\fF \in D^b_{\rm{coh}}(\mathfrak{U})$, its singular
	support is the conical closed subset
	\begin{align*}
		\mathrm{Supp}^{\rm{sg}}(\fF) \subset \mathrm{Crit}(w)
	\end{align*}
	defined to be the support of $\Hom^{2\ast}(\fF, \fF)$
	as $\oO_{\mathrm{Crit}(w)}$-module. 
\end{defi}
We define the  
subcategory of $D^b_{\rm{coh}}(\mathfrak{U})$
with fixed singular support, 
which is a triangulated subcategory of 
$D^b_{\rm{coh}}(\mathfrak{U})$
(for example, this is a consequence of 
Proposition~\ref{prop:koszul:Z}). 
\begin{defi}\label{defi:ssuport2}
	For a conical closed subset $Z \subset \mathrm{Crit}(w)$, 
	we define the subcategory
	\begin{align*}
		\cC_Z \subset D^b_{\rm{coh}}(\mathfrak{U})
	\end{align*}
	to be consisting of $\fF \in D^b_{\rm{coh}}(\mathfrak{U})$
	whose singular support is contained in $Z$. 
\end{defi}

Here we give some examples on the descriptions of $\cC_Z$, using 
the notation of the diagram (\ref{diagram:KZ}). 
\begin{exam}\label{exam:ssupport}
	\begin{enumerate}
		\item Let $W \subset \uU$ be a closed 
		subscheme and take $Z=p_0^{-1}(W)$. 
		Then $\cC_Z$ is the subcategory of 
		objects in $D^b_{\rm{coh}}(\mathfrak{U})$
		whose usual supports are contained in $W$
		(see~\cite[Corollary~4.1.2]{MR3300415}).
		
		\item Let $Z=i(\uU)$. 
		Then $\cC_Z$ is the category of 
		perfect complexes on $\mathfrak{U}$ 
		(see~\cite[Theorem~4.2.6]{MR3300415}). 
		
		\item Let $Z=0_Y(W)$
		for a closed subscheme $W \subset \uU$. 
		Since 
		$\cC_Z=\cC_{p_0^{-1}(W)} \cap \cC_{0_Y(\uU)}$, 
		from (i) and (ii)
		we see that 
		$\cC_Z$ is the 
		category of perfect complexes on $\mathfrak{U}$
		whose usual supports are contained in $W$. 
	\end{enumerate}
\end{exam}
Let $\Ind \Dbc(\fU)$, $\Ind \cC_{Z}$ be the ind-completions (see~\cite{MR3136100})
\begin{align*}
	\Ind \Dbc(\fU) \cneq \Ind(\Dbc(\fU)), \ \Ind \cC_{Z} \cneq \Ind(\cC_{Z}). 
	\end{align*}
	The notion of singular support is extended to 
	objects in $\Ind \Dbc(\fU)$ (see~\cite[Definition~4.1.4]{MR3300415}), 
	and the subcategory 
	\begin{align*}
		\Ind \cC_{\zZ} \subset \Ind \Dbc(\fU)
	\end{align*}
	coincides with the subcategory of 
	objects    
	with singular supports contained in $Z$ (see~\cite[Corollary~4.3.3]{MR3300415}). 
	
The quotient category
\begin{align*}
	\Dbc(\fU)/\cC_Z
\end{align*}
is a
local model of our $\mathbb{C}^{\ast}$-equivariant 
DT category. 
Here we describe a one simple example. 

\begin{exam}\label{exam:simple}
	Let $\mathfrak{U}$ be the affine derived scheme given by
	\begin{align*}
		\mathfrak{U}=\Spec 
		\rR(\mathbb{C}^3 \to \mathbb{C}^2, s=xy). 
	\end{align*}
	Then $\mathfrak{U}$ is equivalent to 
	its classical truncation $\uU=(xy=0) \subset \mathbb{C}^2$. 
	The $(-1)$-shifted cotangent scheme
	$\Omega_{\mathfrak{U}}[-1]$ is the derived 
	critical locus of 
	$w \colon \mathbb{C}^3 \to \mathbb{C}$
	given by $w(x, y, z)=xyz$, so 
	\begin{align*}
		t_0(\Omega_{\mathfrak{U}}[-1])=
		\mathrm{Crit}(w)=\{xy=yz=zx=0\} \subset \mathbb{C}^3.
	\end{align*}
	Let us take $Z=\{(0, 0, 0)\} \subset \mathrm{Crit}(w)$. 
	Then from Example~\ref{exam:ssupport} (iii), 
	we see that $\cC_Z$
	is the subcategory of perfect complexes on $\mathfrak{U}$
	supported on $(0, 0)$. 
	One can check that the following 
	natural functor is an equivalence
	\begin{align*}
		D^b_{\rm{coh}}(\mathfrak{U})/\cC_Z
		\stackrel{\sim}{\to}
		D_{\rm{sg}}(\uU)
		\oplus D^b_{\rm{coh}}(\uU_{1}) \oplus D^b_{\rm{coh}}(\uU_2).
	\end{align*}
	Here $D_{\rm{sg}}(\uU)=D^b_{\rm{coh}}(\uU)/\mathrm{Perf}(\uU)$
	is the triangulated category of singularities~\cite{Orsin}, 
	and $\uU_1$, $\uU_2$ are the connected components of 
	$\uU\setminus \{(0, 0)\}$, which are isomorphic to 
	$\mathbb{C}^{\ast}$. 
\end{exam}

For the later use, we also give another description of singular supports via 
relative Hochschild cohomology 
\begin{align*}
	\mathrm{HH}^{\ast}(\mathfrak{U}/Y)
	\cneq \Hom_{\mathfrak{U} \times_Y \mathfrak{U}}^{\ast}
	(\Delta_{\ast}\oO_{\mathfrak{U}}, \Delta_{\ast}\oO_{\mathfrak{U}}).
\end{align*}
Let us  
calculate $\mathrm{HH}^{\ast}(\mathfrak{U}/Y)$. 
The automorphism of the vector bundle 
\begin{align}\label{aut:V}
	V \oplus V \stackrel{\cong}{\to} V \oplus V, \ 
	(x, y) \mapsto (x+y, x-y)/2
\end{align}
induces the equivalence
\begin{align}\label{isom:fproduct}
	\mathfrak{U} \times_{Y} \mathfrak{U} \stackrel{\sim}{\to}
	\mathfrak{U}^{\flat} \times_Y \mathfrak{U}
\end{align}
where $\mathfrak{U}^{\flat}=\Spec S(V^{\vee}[1])$, 
defined as in (\ref{frak:U})
with $s=0$.  
Note that $\oO_Y$ is a dg-$\oO_{\mathfrak{U}^{\flat}}$-module 
by the projection $\oO_{\mathfrak{U}^{\flat}} \to \oO_Y$. 
Under the equivalence (\ref{isom:fproduct}), 
the object $\Delta_{\ast}\oO_{\mathfrak{U}}$ 
corresponds to $\oO_Y \boxtimes \oO_{\mathfrak{U}}$, so we have
\begin{align}\notag
	\RHom_{\mathfrak{U} \times_Y \mathfrak{U}}(\Delta_{\ast}\oO_{\mathfrak{U}}, \Delta_{\ast}\oO_{\mathfrak{U}})
	&\cong \RHom_{\mathfrak{U}^{\flat} \times_{Y} \mathfrak{U}}(\oO_Y \boxtimes \oO_{\mathfrak{U}}, \oO_Y \boxtimes \oO_{\mathfrak{U}}) \\
	&\label{isom:relHH}
	\cong \RHom_{\mathfrak{U}^{\flat}}(\oO_Y, \oO_Y) \dotimes_{\oO_Y}
	\RHom_{\mathfrak{U}}(\oO_{\mathfrak{U}}, \oO_{\mathfrak{U}}) \\
	&\notag
	\cong S(V[-2]) \otimes_{\oO_Y} \oO_{\mathfrak{U}}.
\end{align}
Here the isomorphism
$\RHom_{\mathfrak{U}^{\flat}}(\oO_Y, \oO_Y)
\stackrel{\cong}{\to} S(V[-2])$
is given by 
taking the Koszul resolution of $\oO_Y$
as $S(V^{\vee}[1])$-module
\begin{align*}
	&S(V^{\vee}[2]) \otimes_{\oO_Y} S(V^{\vee}[1]) \\
	&=\left( \cdots \to \bigwedge^2 (V^{\vee}[1]) \otimes S(V^{\vee}[1]) \to V^{\vee}[1] \otimes S(V^{\vee}[1]) \to S(V^{\vee}[1]) \right)
	\stackrel{\sim}{\to} \oO_Y. 
\end{align*}
In particular, 
the map $x\mapsto x\otimes 1$ gives a map 
of graded algebras
\begin{align}\label{ssuport:map2}
	S(V) \to \mathrm{HH}^{2\ast}(\mathfrak{U}/Y) \to
	\mathrm{Nat}_{D^b_{\rm{coh}}(\mathfrak{U})}(\id, \id[2\ast]).
\end{align}
Similarly to (\ref{ssuport:map}), the right arrow is given by taking FM transforms. 
Then similarly to (\ref{map:V}), the 
maps (\ref{ssuport:map2})
determine the 
$\mathbb{C}^{\ast}$-equivariant 
$S(V)$-module structure on 
$\Hom^{2\ast}(\fF, \fF)$ for $\fF \in D^b_{\rm{coh}}(\mathfrak{U})$. 
\begin{lem}\label{lem:ssuport}\emph{(\cite[Lemma~5.3.4]{MR3300415})}
	The image of $\mathrm{Sing}^{\rm{sg}}(\fF)$
	under the closed immersion 
	$\mathrm{Crit}(w) \hookrightarrow V^{\vee}$ 
	coincides with the support of 
	$\Hom^{2\ast}(\fF, \fF)$ as a
	$S(V)$-module determined by the map (\ref{ssuport:map2}). 
\end{lem}

\subsection{The derived categories of factorizations}\label{subsec:def:fact}
\label{subsec:dfac}
\subsubsection{Definition of factorizations}\label{subsub:fact}
Let $Y$ be a smooth scheme which admits an 
action of an algebraic group $G$, 
and also a $\mathbb{C}^{\ast}$-action which commutes with the 
above $G$-action. 
Let $\tau \colon G \times \mathbb{C}^{\ast} 
\to \mathbb{C}^{\ast}$ be a character 
given by the second projection, 
and let 
$w \in \Gamma(Y, \oO_Y)$ be $\tau$-semi invariant
of weight two, i.e. 
$g^{\ast}w=\tau(g)^2 w$ for any $g \in G \times \mathbb{C}^{\ast}$. 
Given data as above, the \textit{derived category of factorizations} 
\begin{align}\label{der:fact}
	\mathrm{MF}_{\rm{qcoh}}^{\mathbb{C}^{\ast}}([Y/G], w)
\end{align}
is defined to be the triangulated category
whose objects consist of quasi-coherent
$(G\times \mathbb{C}^{\ast})$\textit{-equivariant factorizations of }$w$, i.e. 
a pair
\begin{align}\label{factorization}
	(\pP, d_{\pP}), \ d_{\pP} \colon \pP \to \pP\langle 1 \rangle, \ 
	d_{\pP} \circ d_{\pP}=w
\end{align}
where $\pP$ is a $(G \times \mathbb{C}^{\ast})$-equivariant 
quasi-coherent sheaf on $Y$, 
$d_{\pP}$ is a $(G \times \C)$-equivariant morphism. 
Here $\langle n \rangle$ 
means the twist by the $(G\times \mathbb{C}^{\ast})$-character
$\tau^n$. 
The category (\ref{der:fact}) is defined to be
the localization 
of the homotopy category of 
the factorizations (\ref{factorization})
by its subcategory of acyclic factorizations
(see~\cite[Section~3]{MR3270588}, \cite[Section~2.2]{MR3366002} for 
details). 
We have the triangulated subcategory
\begin{align*}
	\mathrm{MF}_{\rm{coh}}^{\mathbb{C}^{\ast}}([Y/G], w)
	\subset \mathrm{MF}_{\rm{qcoh}}^{\mathbb{C}^{\ast}}([Y/G], w)
\end{align*}
consisting of factorizations (\ref{factorization}) such that
$\pP$ is coherent. 
If $Y$ is affine, the above subcategory is equivalent to 
the triangulated category of 
$(G \times \C)$-equivariant matrix factorizations of $w$
defined by Orlov~\cite{Orsin}. 

\begin{rmk}\label{rmk:curved}
	The derived categories of factorizations are also defined in terms of 
	curved dg-modules over curved dg-algebras
	introduced in~\cite{MR2931331}. 
	By definition, a \textit{commutative curved dg-algebra}
	is a triple
	\begin{align*}
		(R^{\bullet}, d_R, c), \ d_R \colon R^{\bullet} \to R^{\bullet+1}, \ 
		c \in R^2. 
	\end{align*}
	Here $(R^{\bullet}, d_R)$ is a commutative dg-algebra, and 
	the element 
	$c \in R^2$ (called \textit{curvature}) satisfies $d_R c=0$. 
	A \textit{curved dg-module} over a curved 
	dg-algebra $(R^{\bullet}, d_R, c)$ is a 
	pair
	\begin{align*}
		(M^{\bullet}, d_M), \ 
		d_M \colon M^{\bullet} \to M^{\bullet+1}, \ d_M^2 =\cdot c
	\end{align*}
	where $M^{\bullet}$ is a graded $R^{\bullet}$-module, and 
	the map 
	$R^{\bullet} \to \End(M^{\bullet})$
	defining the $R^{\bullet}$-module structure
	commutes with $(d_R, [d_M, -])$. 
	By its definition, giving a 
	$(G \times \mathbb{C}^{\ast})$-equivariant 
	factorization of $w \colon Y \to \mathbb{C}$ is 
	equivalent to giving a $G$-equivariant 
	curved dg-module over a $G$-equivariant curved 
	dg-algebra
	\begin{align}\label{curve:dga}
		(\oO_Y, 0, w).
	\end{align}
	Then 
	$\mathrm{MF}_{\rm{coh}}^{\mathbb{C}^{\ast}}([Y/G], w)$
	is naturally identified with the derived category of $G$-equivariant 
	finitely generated curved dg-modules over 
	(\ref{curve:dga}). 
\end{rmk}
\begin{rmk}\label{rmk:def:Z2}
	The $\mathbb{Z}/2$-periodic derived category of (coherent, quasi-coherent) factorizations 
	is also defined in a similar way. 
For a smooth scheme $Y$ with $G$-action and 
	a $G$-invariant function $w \colon Y \to \mathbb{C}$, 
	they are denoted by 
	\begin{align*}
		\MF_{\coh}^{\mathbb{Z}/2\mathbb{Z}}([Y/G], w), \ 
			\MF_{\qcoh}^{\mathbb{Z}/2\mathbb{Z}}([Y/G], w)
			\end{align*}
		respectively. They are $\mathbb{Z}/2\mathbb{Z}$-periodic triangulated 
		categories, and consists of 
	$\mathbb{Z}/2 \times G$-equivariant factorizations 
		\begin{align*}
			(\pP, d_{\pP}), \ d_{\pP} \colon \pP \to \pP\langle 1 \rangle, \ d_{\pP} \circ d_{\pP}=w
			\end{align*}
		where $\pP$ is coherent, quasi-coherent, respectively. 
		Here $\mathbb{Z}/2\mathbb{Z}$ acts on $Y$ trivially, 
		and $\langle 1 \rangle$ means twist with respect to 
		a non-trivial $\mathbb{Z}/2\mathbb{Z}$-character. 
	\end{rmk}
For a $(G \times \C)$-invariant closed subscheme $Z \subset Y$
with $\zZ=[Z/G]$, 
the subcategory
\begin{align}\label{emb:support}
	\mathrm{MF}_{\star}^{\C}([Y/G], w)_{\zZ} \subset 
	\mathrm{MF}_{\star}^{\C}([Y/G], w)
\end{align}
for $\star \in \{\rm{qcoh}, \rm{coh}\}$
is defined to be the kernel of the restriction functor
\begin{align*}
	\mathrm{MF}_{\star}^{\C}([Y/G], w) \to 
	\mathrm{MF}_{\star}^{\C}([(Y \setminus Z)/G], w). 
\end{align*}
By taking the Verdier quotients, we have an equivalence
(cf.~\cite[Theorem~1.3]{MR2734686})
\begin{align}\label{quot:MF}
	\mathrm{MF}_{\star}^{\C}([Y/G], w)/
	\mathrm{MF}_{\star}^{\C}([Y/G], w)_{\zZ}
	\stackrel{\sim}{\to}
	\mathrm{MF}_{\star}^{\C}([(Y \setminus Z)/G], w).
\end{align}
Since 
any object in $\MF^{\C}_{\star}([Y/G], w)$
is supported on $\Crit(w)$ 
(see~\cite[Corollary~3.18]{MR3112502}), 
for $Z'=Z \cap \Crit(w)$
we have equivalences 
\begin{align}\notag
	\MF^{\C}_{\star}([Y/G], w)_{\zZ'}
	\stackrel{\sim}{\to}
	\MF^{\C}_{\star}([Y/G], w)_{\zZ}.
\end{align}
Combined with (\ref{quot:MF}), we also have the equivalence
\begin{align}\label{equiv:crit}
	\mathrm{MF}_{\star}^{\C}([(Y \setminus Z')/G], w)
	\stackrel{\sim}{\to}
	\mathrm{MF}_{\star}^{\C}([(Y \setminus Z)/G], w).
\end{align}
In particular, 
let $Y' \subset Y$ be a $(G \times \C)$-invariant 
open subset which contains $\Crit(w)$. 
Then by taking $Z=Y\setminus Y'$
so that $Z'=\emptyset$, 
the restriction functor gives an equivalence
\begin{align}\label{rest:equiv}
	\MF_{\star}^{\C}([Y/G], w) \stackrel{\sim}{\to}
	\MF_{\star}^{\C}([Y'/G], w). 
	\end{align}

For $\star=\qcoh$, the embedding (\ref{emb:support}) admits a 
right adjoint (see~\cite[Proposition~2.3.9]{MR3895631})
\begin{align*}
	\Gamma_{\zZ} \colon \MF_{\qcoh}^{\C}([Y/G], w)
	\to \MF_{\qcoh}^{\C}([Y/G], w)_{\zZ}
\end{align*}
so that 
for an object $\pP \in \MF_{\qcoh}^{\C}([Y/G], w)$,
there is a distinguished triangle
\begin{align}\label{dist:support}
	\Gamma_{\zZ}(\pP) \to \pP \to \pP|_{[Y/G]\setminus \zZ}
\end{align}
where $\pP|_{[Y/G]\setminus Z}\cneq j_{\ast}j^{\ast}\pP$
for the open immersion $j \colon [Y/G]\setminus \zZ \hookrightarrow [Y/G]$. 

\subsubsection{Functoriality of derived categories of factorizations}
Let $Y'$ be another smooth scheme with 
$(G' \times \C)$-action 
and $f \colon [Y'/G'] \to [Y/G]$ 
a $\C$-equivariant morphism. 
We have the push-forward/pull-back functors (cf.~\cite[Section~3]{MR3270588})
\begin{align}\label{funct:MF}
	\xymatrix{
		\MF_{\qcoh}^{\C}([Y'/G'], f^{\ast}w) 
		\ar@<0.5ex>[r]^-{f_{\ast}} &
		\ar@<0.5ex>[l]^-{f^{\ast}}
		\MF_{\qcoh}^{\C}([Y/G], w) }
\end{align}
such that $f^{\ast} \dashv f_{\ast}$. 
If $f$ is proper, 
we also have functors 
\begin{align*}
	\xymatrix{
		\MF_{\qcoh}^{\C}([Y'/G'], f^{\ast}w) 
		\ar@<0.5ex>[r]^-{f_{!}} &
		\ar@<0.5ex>[l]^-{f^{!}}
		\MF_{\qcoh}^{\C}([Y/G], w) }
\end{align*}
where $f_{\ast}\dashv f^!$, $f_{!} \dashv f^{\ast}$. 
They are related to the functors (\ref{funct:MF}) 
as 
\begin{align*}
	f^!(-)=f^{\ast}(-) \otimes \omega_f, \ 
	f_{!}(-)=f_{\ast}(- \otimes \omega_f). 
\end{align*}
Here $\omega_f \cneq \det \mathbb{L}_f[\dim f]$. 
The functor $f^{\ast}$ preserves 
$\MF^{\C}_{\coh}(-)$, and if $f$ is proper 
then $f^!$, $f_{\ast}$ and $f_{!}$ also preserve
$\MF^{\C}_{\coh}(-)$. 

Let 
$\lambda \colon \C \to G$ be a one parameter subgroup 
contained in the center of $G$ and acts on $Y$ trivial. 
Then we have the decomposition for $\star \in \{\coh, \qcoh\}$
\begin{align*}
	\MF_{\star}^{\C}([Y/G], w)=
	\bigoplus_{j \in \mathbb{Z}}
	\MF_{\star}^{\C}([Y/G], w)_{\lambda \mathchar`- \wt= j}
\end{align*}
where $\MF_{\star}^{\C}([Y/G], w)_j$ is the weight 
$j$-part with respect to $\lambda$. 
We define the subcategory 
\begin{align*}
	\MF_{\coh}^{\C}([Y/G], w)_{\lambda \mathchar`- \rm{below}}
	\subset \MF_{\qcoh}^{\C}([Y/G], w)
\end{align*}
to be consisting of objects whose $\lambda$-weights are bounded 
below, and each $\lambda$-weight $j$ part is an object in 
$\MF_{\coh}^{\C}([Y/G], w)_{\lambda \mathchar`- \wt= j}$. 

Let $W \to Y$ be a $(G\times \C)$-equivariant vector bundle, 
where $\lambda$ acts on fibers of $W \to Y$ by negative weights. 
Then 
we have the $\C$-equivariant morphism of stacks 
$f \colon [W/G] \to [Y/G]$. 
\begin{lem}\label{lem:stack:push}
	The functor $f_{\ast}$ in (\ref{funct:MF}) restricts to the functor
	\begin{align*}
		f_{\ast} \colon \MF_{\coh}^{\C}([W/G], f^{\ast}w) \to 
		\MF_{\coh}^{\C}([Y/G], w)_{\lambda \mathchar`- \rm{below}}. 
	\end{align*}
\end{lem}
\begin{proof}
	By push-forward $f_{\ast}$, 
	a coherent sheaf on $[W/G]$ is regarded as a coherent 
	module over the sheaf of algebras
	$\mathrm{Sym}(W^{\vee})$ on $[Y/G]$. 
	It has only non-negative $\lambda$-weights, and
	each $\lambda$-weight $j$ part is a coherent sheaf on $[Y/G]$. 
	Therefore the lemma holds. 
\end{proof}

\subsection{Koszul duality equivalence}\label{subsec:Koszul}
\subsubsection{$G$-equivariant tuple}
\begin{defi}\label{def:tuple}
		Let $G$ be an affine algebraic group. 
	A $G$-equivariant tuple is a tuple 
	\begin{align}\notag
		(Y, V, s)
	\end{align}
	where 
	$Y$ is a smooth affine $\mathbb{C}$-scheme of finite presentation
	with $G$-action, 
	$V \to Y$ is a $G$-equivariant 
	vector bundle and $s$ is a $G$-invariant section of $V$. 
		Given $(Y, V, s)$ as above,
	the associated derived stack is denoted by 
	\begin{align*}
		[\fU/G]\cneq [\Spec \rR(V \to Y, s)/G].
	\end{align*}
Its classical truncation is denoted by $[\uU/G]$, which is a closed substack of $\yY\cneq [Y/G]$. 
	\end{defi}

For a $G$-equivariant tuple $(Y, V, s)$, we 
take the weight two $\C$-action on the fibers of 
$V^{\vee} \to Y$. 
Then we have the $(G \times \C)$-action on $V^{\vee}$
and 
the function $w$ defined in (\ref{def:w}) is $\tau$-semi invariant 
of weight two, where 
$\tau$ is the projection
$G \times \mathbb{C}^{\ast} \to \mathbb{C}^{\ast}$. 
For a $(G \times \mathbb{C}^{\ast})$-invariant closed subset 
$Z \subset \mathrm{Crit}(w)$, we have the subcategories
\begin{align*}
	\cC_{\zZ} \subset \Dbc([\fU/G]), \ 
	\Ind \cC_{\zZ} \subset \Ind \Dbc([\fU/G])
\end{align*}
consisting of objects 
whose singular supports are contained in 
the closed substack
\begin{align*}
	\zZ=[Z/G] \subset [t_0(\Omega_{\fU}[-1])/G].
\end{align*}
More precisely these are 
objects whose pull-back via $\fU \to [\fU/G]$ 
are contained in $\cC_{Z}$, $\Ind \cC_{Z}$ respectively
(see Definition~\ref{def:CZ2} for a general case). 
\begin{lem}\label{lem:gcomplete}
	The category $\Ind \cC_{\zZ}$
	is compactly generated 
	such that $(\Ind \cC_{\zZ})^{\rm{cp}}=\cC_{\zZ}$. 
	\end{lem}
\begin{proof}
	Since the stack $[\fU/G]$ is a global complete intersection 
	stack in the sense of~\cite[Section~9.1]{MR3300415}, 
	the lemma follows from~\cite[Corollary~9.2.7, Corollary~9.2.8]{MR3300415}.
	\end{proof}
Note that we have the following commutative diagram, which we will 
often use below. 
\begin{align}\label{diagram:UA}
	\xymatrix{
		& & [V/G] \ar[d] & [V^{\vee}/G] \ar[r]^-{w} \ar[d] & \mathbb{C} \\
		[\uU/G] \ar@<-0.3ex>@{^{(}->}[r] \ar[d]_-{\pi_{\uU}}
		& [\fU/G] \ar@<-0.3ex>@{^{(}->}[r]^-{j} \ar[rd]_-{\pi_{\fU}} & [Y/G] 
		\ar[d]_-{\pi_{Y}} \ar@/_10pt/[u]_-{s} \ar@{=}[r]& [Y/G]  \ar@/_10pt/[u]_-{0_Y}
		\ar[d]_-{\pi_Y} & \\
		\uU \ssslash G \ar@<-0.3ex>@{^{(}->}[rr]
		& & Y \ssslash G \ar@{=}[r] & Y\ssslash G. &
	}
\end{align}
Here $0_Y$ is the zero section
of $V^{\vee} \to Y$, and $\pi_{\uU}$, $\pi_Y$ are 
good moduli space morphisms
(see Subsection~\ref{subsubsec:gmoduli} for good moduli spaces). 

\subsubsection{Koszul duality equivalence}
The construction in the previous subsection 
yields the derived category of factorizations and 
its subcategory with fixed support
\begin{align*}
	\MF^{\C}_{\coh}([V^{\vee}/G], w)_{\zZ} \subset 
	\MF^{\C}_{\coh}([V^{\vee}/G], w). 
\end{align*}
Let $\kK_s$ be the following $(G\times \C)$-equivariant
factorization of $w$
(called \textit{Koszul factorization})
\begin{align}\label{Phi(U):KF}
	\kK_s \cneq 
	\left(
	\oO_{V^{\vee}} \otimes_{\oO_Y} \oO_{\mathfrak{U}}, 
	d_{\kK_s} \right). 
\end{align}
Here $G$ acts on $\oO_{V^{\vee}} \otimes_{\oO_Y} \oO_{\fU}$ diagonally, 
the $\C$-action is given by the grading  
\begin{align*}
	\oO_{V^{\vee}} \otimes_{\oO_Y}\oO_{\mathfrak{U}}=
	S(V[-2]) \otimes_{\oO_Y} S(V^{\vee}[1]),
\end{align*}
and the weight one map $d_{\kK_s}$ is given by 
\begin{align*}
	d_{\kK_s}=
	1 \otimes d_{\oO_{\mathfrak{U}}}+\eta \colon \oO_{V^{\vee}} \otimes_{\oO_Y}\oO_{\mathfrak{U}} \to
	\oO_{V^{\vee}} \otimes_{\oO_Y}\oO_{\mathfrak{U}} \langle 1 \rangle,
\end{align*}
where 
$\eta \in V\otimes_{\oO_Y} V^{\vee} \subset \oO_{V^{\vee}} \otimes_{\oO_Y}
\oO_{\fU}$
corresponds to 
$\id \in \Hom(V, V)$. 

The object $\kK_s$ 
also admits a $G$-equivariant 
dg $\oO_{\mathfrak{U}}$-module structure
by the right factor of the tensor product. 
In particular for a dg-$\oO_{\fU}$-module $\fF^{\bullet}$, 
we can form its tensor product $\kK_s \otimes_{\oO_{\fU}}\fF^{\bullet}$
which is a factorization of $w$. 
It is explicitly described as follows. 
The dg-$\oO_{\fU}$-module $\fF^{\bullet}$ consists of 
a complex of $\oO_{Y}$-modules $(\fF^{\bullet}, d_{\fF^{\bullet}})$
together with 
$\oO_Y$-module homomorphisms 
$\eta^i \colon \fF^i \to V \otimes \fF^{i-1}$
satisfying the axiom of dg-modules. 
Then 
$\kK_s \otimes_{\oO_{\fU}}\fF^{\bullet}$
is given by 
\begin{align*}
	\left(
	\bigoplus_{i\in \mathbb{Z}}\fF^i \langle i \rangle
	\otimes_{\oO_Y}\oO_{V^{\vee}} \to 
	\bigoplus_{i\in \mathbb{Z}}\fF^i \langle i+1 \rangle
	\otimes_{\oO_Y}\oO_{V^{\vee}}\right)
	\end{align*}
where the weight one morphism is an $\oO_{V^{\vee}}$-module homomorphism 
induced by
\begin{align*}
d_{\fF^{\bullet}}^i +\eta^i \colon 
\fF^i \to \fF^{i+1} \oplus (V \otimes \fF^{i-1}). 	
	\end{align*}
The correspondence $(-) \to \kK_s \otimes_{\oO_{\fU}}(-)$ is understood 
as a Fourier-Mukai functor (see Remark~\ref{rmk:precise}) so that 
it induces a functor 
on derived categories. 
Then 
the following result is proved in 
several references
(see~\cite[Proposition~4.8]{MR3631231}
and also~\cite{MR3071664, MR2982435, ObRoz}). 
\begin{thm}\emph{(cf.~\cite{MR3071664, MR2982435, MR3631231, ObRoz})}
	\label{thm:knoer}
	The functor $\Phi$ defined by 
	\begin{align}\label{funct:koszul}
		\Phi \colon D^b_{\rm{coh}}([\mathfrak{U}/G]) \to
		\mathrm{MF}_{\rm{coh}}^{\mathbb{C}^{\ast}}([V^{\vee}/G], w), \ 
		(-) \mapsto \kK_s \otimes_{\oO_{\mathfrak{U}}} (-)
	\end{align}
is an equivalence of triangulated categories. 
	Moreover a quasi-inverse of $\Phi$ is given by 
	\begin{align}\label{Psi:qinverse}
		\Psi \colon \mathrm{MF}_{\rm{coh}}^{\mathbb{C}^{\ast}}([V^{\vee}/G], w) \to 
		D^b_{\rm{coh}}([\mathfrak{U}/G]), \ 
		(-) \mapsto 
		\RHom_{\mathrm{MF}_{\rm{coh}}^{\mathbb{C}^{\ast}}([V^{\vee}/G], w)}(\kK_s, -). 
	\end{align}
\end{thm}
\begin{proof}
	The result is essentially proved in~\cite{MR3071664, MR2982435, MR3631231, ObRoz}, 
	but the currently stated version is not stated in the references so we give its proof. 
	
	We first prove the case of $G=\{1\}$. 	
	Since $\kK_s$ consists of free $\oO_{\mathfrak{U}}$-modules, 
	we do not need to take $\dL$ 
	in the computation of 
	$\Phi$.  
	In particular it takes coherent $\oO_{\mathfrak{U}}$-modules to 
	coherent factorizations of $w \colon V^{\vee} \to \mathbb{C}$, so the functor 
	(\ref{funct:koszul}) is well-defined. 
	Also 
	for $\pP \in \mathrm{MF}^{\mathbb{C}^{\ast}}_{\rm{coh}}(V^{\vee}, w)$, 
	we have 
	\begin{align}\label{funct:Psi:0}
		\Psi(\pP)=\RHom_{\mathrm{MF}_{\rm{coh}}^{\mathbb{C}^{\ast}}(V^{\vee}, w)}(\kK_s, \pP)
		\cong \kK_s^{\vee} \otimes_{\oO_{V^{\vee}}} \pP
	\end{align}
	where 
	$\kK_s^{\vee}$ is the $\oO_{V^{\vee}}$-dual of $\kK_s$, which is a 
	factorization of $-w$. 
	Below we will use the following functors of derived categories 
	of factorizations 
	\begin{align*}
		\mathrm{MF}_{\rm{coh}}^{\mathbb{C}^{\ast}}(V^{\vee}, w)
		\stackrel{0_Y^{\ast}}{\to} \mathrm{MF}_{\rm{coh}}^{\mathbb{C}^{\ast}}(Y, 0)
		\stackrel{0^{-w}_{Y\ast}}{\to}
		\mathrm{MF}_{\rm{coh}}^{\mathbb{C}^{\ast}}(V^{\vee}, -w). 
	\end{align*}
	Here 
	$0^{-w}_{Y\ast}$ is the 
	push-forward along with the zero section 
	$0_Y \colon Y \to V^{\vee}$ which makes sense as $0_Y^{\ast}w=0$. 
	Also for a $\mathbb{C}^{\ast}$-equivariant coherent sheaf 
	$\fF$ on $Y$ (resp.~on $V^{\vee}$), 
	we regard it as an object in $\mathrm{MF}_{\rm{coh}}^{\mathbb{C}^{\ast}}(Y, 0)$
	(resp.~$\mathrm{MF}_{\rm{coh}}^{\mathbb{C}^{\ast}}(V^{\vee}, 0)$)
	by taking the factorization of $0$, 
	$(\fF \stackrel{0}{\to} \fF\langle 1 \rangle)$. 
	Using the above notation, by~\cite[Proposition~3.20, Lemma~3.21]{MR3270588}, we have 
	the isomorphisms as factorizations of $-w$: 
	\begin{align}\label{isom:K:dual}
		\kK_s^{\vee}\cong 
		\kK_{-s} \otimes_{\oO_Y} \det V[-\rank(V)]
		\cong 0^{-w}_{Y\ast} \oO_Y.
	\end{align}
	Therefore we have the isomorphism of dg-algebras
	\begin{align}\label{qis:K}
		\RHom_{\MF_{\coh}^{\C}(V^{\vee}, w)}(\kK_s, \kK_s) 
			=\kK_s^{\vee} \otimes_{\oO_{V^{\vee}}}\kK_s,
			\cong \oO_{\fU},
			\end{align}
		therefore 
	(\ref{funct:Psi:0}) is a dg-$\oO_{\mathfrak{U}}$-module.
	Moreover each cohomology of (\ref{funct:Psi:0}) 
	is a coherent $\oO_Y$-module, 
	so it determines an object in $D^b_{\rm{coh}}(\mathfrak{U})$, 
	and the functor (\ref{Psi:qinverse}) is well-defined. 
	
	The composition $\Psi \circ \Phi$ is given by 
	\begin{align*}
		\Psi \circ \Phi(-)=\kK_s^{\vee} \otimes_{\oO_{V^{\vee}}}\kK_s  \otimes_{\oO_{\mathfrak{U}}} (-). 
	\end{align*}
	By (\ref{qis:K}), we conclude that
	$\Psi \circ \Phi$ is the identity functor.  
	Therefore $\Phi$ is fully-faithful, 
	and it remains to show that 
	$\Phi$ is essentially surjective. 
	As $\Psi$ is a right adjoint of $\Phi$, 
	it is enough to show that 
	if an object $\pP \in \mathrm{MF}^{\mathbb{C}^{\ast}}_{\rm{coh}}(V^{\vee}, w)$
	satisfies $\Psi(\pP) \cong 0$, then $\pP \cong 0$. 
	Suppose that $\pP$ is such an object. 
	By the above description of $\kK_s^{\vee}$, 
	the push-forward of $\Psi(\pP) \cong 0$ along the closed immersion 
	 $\mathfrak{U} \hookrightarrow Y$
	is $i^{\ast}\pP \cong 0$. 
	Let $I \subset \oO_{V^{\vee}}$ be the ideal sheaf of the zero section 
	$Y \subset V^{\vee}$, 
	and set $Y^{(2)}=\Spec_{Y} \oO_{V^{\vee}}/I^2$. 
	By applying $\pP\dotimes_{\oO_{V^{\vee}}} (-)$ to the exact sequence
	\begin{align*}
		0 \to I/I^2 \to \oO_{Y^{(2)}} \to \oO_Y \to 0
	\end{align*}
	we also see that $\pP|_{Y^{(2)}}\cong 0$
	in $\mathrm{MF}^{\mathbb{C}^{\ast}}_{\rm{coh}}(Y^{(2)}, w|_{Y^{(2)}})$. 
	By repeating this argument, we see that 
	$\pP|_{\widehat{V}^{\vee}}\cong 0$ 
	in $\mathrm{MF}^{\mathbb{C}^{\ast}}_{\rm{coh}}(\widehat{V}^{\vee}, w|_{\widehat{V}^{\vee}})$, 
	where $\widehat{V}^{\vee}$ is the formal 
	completion of $V^{\vee}$ at $Y$. 
	On the other hand 
	$\Hom_{\mathrm{MF}_{\rm{coh}}^{\mathbb{C}^{\ast}}(V^{\vee}, w)}(\pP, \pP)$
	is a finitely generated 
	$\oO_{V^{\vee}}$-module 
	whose support is a conical
	closed subset $Z \subset V^{\vee}$. 
	As $\pP|_{\widehat{V}^{\vee}} \cong 0$, 
	we have $Z \cap \widehat{V}^{\vee}=\emptyset$. 
	But $Z$ is conical, so we must have $Z=\emptyset$,
	hence $\pP \cong 0$ holds.  
	
	For a general $G$, the natural transformations
	\begin{align*}
		(-) \to \Psi \circ \Phi(-), \ 
		\Psi \circ \Phi(-) \to (-)
		\end{align*}
	are isomorphisms after forgetting $G$-equivariant structures
	by the above argument. 
	Therefore they are isomorphisms since the forgetting functors
	\begin{align*}
		\Dbc([\fU/G]) \to \Dbc(\fU), \ 
		\MF_{\coh}^{\C}([V^{\vee}/G], w) \to \MF_{\coh}^{\C}(V^{\vee}, w)
		\end{align*}
	are conservative. 
\end{proof}

\begin{rmk}\label{rmk:precise}
	The functor $\Phi$ is explained in terms of curved dg-modules~\cite{MR2931331}
	as follows. 
	The graded algebra
	$\oO_{V^{\vee}} \otimes_{\oO_Y}\oO_{\mathfrak{U}}$ 
	together with 
	the degree one map $1 \otimes d_{\oO_{\mathfrak{U}}}$
	and 
	the degree two element $w \otimes 1 \in 
	\oO_{V^{\vee}} \otimes_{\oO_Y}\oO_{\mathfrak{U}}$
	define
	the commutative curved dg-algebra (see Remark~\ref{rmk:curved})
	\begin{align*}
		(\oO_{V^{\vee}} \otimes_{\oO_Y}\oO_{\mathfrak{U}}, 1 \otimes 
		d_{\oO_{\mathfrak{U}}}, w \otimes 1)
	\end{align*}
	with 
	differential $1 \otimes d_{\oO_{\mathfrak{U}}}$ and 
	curvature $w \otimes 1$. 
	Then 
	the object $\kK_s$
	is a curved dg-module over it, 
	and the functor $\Phi$ is regarded as 
	a Fourier-Mukai functor from curved dg-modules over a
	curved dg-algebra
	$(\oO_{\mathfrak{U}}, d_{\oO_{\mathfrak{U}}}, 0)$
	to those over  
	$(\oO_{V^{\vee}}, 0, w)$ whose kernel object 
	is $\kK_s$.
\end{rmk}

\begin{rmk}\label{rmk:regular}
	If $s$ is a regular section, then 
	the closed immersion $\uU \hookrightarrow \mathfrak{U}$ is an equivalence, 
	and the object (\ref{Phi(U):KF}) is isomorphic to 
	the factorization $p^{\ast}\oO_{\uU} \stackrel{0}{\to} p^{\ast}\oO_{\uU}\langle 1 \rangle$. 
	So it coincides with the equivalence constructed 
	in~\cite{MR2982435, MR3631231}. 
\end{rmk}

\begin{rmk}\label{rmk:szero}
	When $s=0$ then $w=0$, 
	and we have objects
	$\oO_Y \in D^b_{\rm{coh}}(\mathfrak{U})$ and 
	$\oO_{V^{\vee}} \in \mathrm{MF}^{\mathbb{C}^{\ast}}_{\rm{coh}}(V^{\vee}, 0)$, 
	where $\oO_{V^{\vee}}=(\oO_{V^{\vee}} \stackrel{0}{\to} \oO_{V^{\vee}}\langle 1 \rangle)$. 
	Since $\Phi(\oO_Y)=\oO_{V^{\vee}}$ for
	the equivalence in Theorem~\ref{thm:knoer},
	we have the quasi-isomorphism
	\begin{align*}
		\Phi^{\flat} \colon \RHom_{\mathfrak{U}}(\oO_Y, \oO_Y) \stackrel{\sim}{\to}
		\RHom_{\mathrm{MF}^{\mathbb{C}^{\ast}}_{\rm{coh}}(V^{\vee}, 0)}(\oO_{V^{\vee}}, \oO_{V^{\vee}})
		=S(V[-2]).
	\end{align*}
	The above quasi-isomorphism 
	coincides with the one given in (\ref{isom:relHH}) 
	used in defining the singular support. 
\end{rmk}

\begin{rmk}\label{rmk:tautological}
	By applying Theorem~\ref{thm:knoer}
	for $V=Y$, we obtain the equivalence
	\begin{align}\notag
		\Dbc([Y/G]) \stackrel{\sim}{\to} \MF_{\coh}^{\C}([Y/G], 0). 
		\end{align}
	The above equivalence is 
	a tautological equivalence sending 
	$(\fF^{\bullet}, d_{\fF^{\bullet}})$ to 
	$(\fF^{\bullet}, d_{\fF^{\bullet}})$, where the 
	cohomological gradings of the LHS are
	$\C$-weights of the RHS. 
	\end{rmk}

\begin{rmk}\label{rmk:periodic}
	The result of Theorem~\ref{thm:knoer} is not true for 
	$\mathbb{Z}/2$-periodic version unless 
	$\Crit(w)=\uU$, or equivalently $\fU=\uU$ and it is smooth. 
	Namely the similarly defined functor 
	\begin{align*}
		\Psi 
		 \colon \MF_{\coh}^{\mathbb{Z}/2\mathbb{Z}}(V^{\vee}, w)
	\to D^{\mathbb{Z}/2\mathbb{Z}}_{\coh}(\fU)
			\end{align*}
		is not an equivalence. 
		Here the RHS is the $\mathbb{Z}/2$-periodic 
		derived category of dg-modules over $\fU$ with coherent 
		cohomologies. 
		Indeed if $\uU \subsetneq \Crit(w)$, there is an object 
	in $\MF_{\coh}^{\mathbb{Z}/2\mathbb{Z}}(V^{\vee}, w)$ supported 
	away from $\uU$, so the above functor has a non-trivial kernel. 
	This is the reason we construct the DT category for $\C$-equivariant case. 	
	\end{rmk}

\subsubsection{Singular supports under Koszul duality}
Under the Koszul duality equivalence in Theorem~\ref{thm:knoer}, 
we can compare the singular supports and the usual supports 
in the derived categories of factorizations.  
The following proposition was claimed in~\cite[Section~H]{MR3300415}, 
and we give its details. 
\begin{prop}\label{prop:koszul:Z}
	Let $Z \subset \Crit(w)$ be a $(G\times \C)$-invariant 
	closed subset and set $\zZ=[Z/G]$. 
Then the equivalence $\Phi$ in Theorem~\ref{thm:knoer}
	restricts to the equivalence
		\begin{align}\label{koszul:C}
			\Phi \colon 
			\cC_{\zZ} \stackrel{\sim}{\to}
		\mathrm{MF}_{\coh}^{\mathbb{C}^{\ast}}([V^{\vee}/G], w)_{\zZ}.
	\end{align}
	In particular by taking Verdier quotients and 
	using (\ref{quot:MF}), we also have an equivalence
	\begin{align}\label{equiv:MF}
		\Phi \colon 
		\Dbc([\fU/G])/\cC_{\zZ} \stackrel{\sim}{\to}
		\mathrm{MF}_{\coh}^{\mathbb{C}^{\ast}}([(V^{\vee} \setminus Z)/G], w). 
	\end{align}
\end{prop}
\begin{proof}
	We may assume that $G=\{1\}$. 
	The equivalence $\Phi$ in (\ref{funct:koszul})
	induces the isomorphism of 
	natural transforms
	\begin{align*}
		\Phi^{\mathrm{Nat}} \colon 
		\mathrm{Nat}_{D^b_{\rm{coh}}(\mathfrak{U})}(\id, \id[m])
		\stackrel{\cong}{\to}
		\mathrm{Nat}_{\mathrm{MF}^{\mathbb{C}^{\ast}}_{\rm{coh}}(V^{\vee}, w)}(\id, \id[m])
	\end{align*}
	given by $(-) \mapsto \Phi \circ (-) \circ \Phi^{-1}$. 
	By Lemma~\ref{lem:ssuport}, 
	it is enough to show that the following diagram commutes:
	\begin{align}\label{commute:nat}
		\xymatrix{
			V \ar[d] \ar@{=}[r] & V\ar[d] \\ 
			\mathrm{Nat}_{D^b_{\rm{coh}}(\mathfrak{U})}(\id, \id[2])
			\ar[r]^-{\Phi^{\mathrm{Nat}}} 
			& \mathrm{Nat}_{\mathrm{MF}^{\mathbb{C}^{\ast}}_{\rm{coh}}(V^{\vee}, w)}
			(\id, \id[2]).
		}
	\end{align}
	Here the left vertical arrow is 
	given in (\ref{ssuport:map2}), and the 
	right vertical arrow is given by the 
	multiplication via 
	$\oO_{V^{\vee}}$-module structure on 
	each factorization. 
	
	Let $\Phi^{\mathrm{Fun}}$ be the functor  
	\begin{align*}
		\Phi^{\mathrm{Fun}} \colon  
		D^b_{\rm{coh}}(\mathfrak{U}\times_Y \mathfrak{U})
		\to
		\mathrm{MF}_{\rm{coh}}^{\mathbb{C}^{\ast}}(V^{\vee} \times_Y V^{\vee}, 
		-w \boxplus w)
	\end{align*}
	given by the Fourier-Mukai kernel 
	$p_{13}^{\ast}\kK_s^{\vee} \otimes p_{24}^{\ast}\kK_s$, 
	where $p_{ij}$ are the projections from 
	$\mathfrak{U} \times_Y \mathfrak{U} \times_Y V^{\vee} \times_Y V^{\vee}$
	onto the corresponding factors. 
	Here $\mathbb{C}^{\ast}$ acts on fibers of 
	$V^{\vee} \times_Y V^{\vee} \to Y$ by weight two. 
	Note that $\mathfrak{U} \times_Y \mathfrak{U}$ is 
	equivalent to the derived zero locus of $(-s, s) \colon Y \to V \oplus V$. 
	Moreover we have the isomorphism by the first isomorphism in (\ref{isom:K:dual})
	\begin{align*}
		p_{13}^{\ast}\kK_s^{\vee} \otimes p_{24}^{\ast}\kK_s
		\cong \kK_{(-s, s)} \otimes_{\oO_Y}\det V[-\rank(V)]. 
	\end{align*}
	Therefore an argument similar to Proposition~\ref{thm:knoer} shows that 
	the functor $\Phi^{\mathrm{Fun}}$ is an equivalence. 
	From a categorical point of view, the equivalence 
	$\Phi^{\mathrm{Fun}}$ 
	identifies FM transforms 
	on $D^b_{\rm{coh}}(\mathfrak{U})$ over $Y$ with those on 
	$\mathrm{MF}_{\rm{coh}}^{\mathbb{C}^{\ast}}(V^{\vee}, w)$
	via $(-) \mapsto \Phi \circ (-) \circ \Phi^{-1}$, 
	which can be verified along with the similar 
	argument of~\cite[Proposition~8.2]{Cal}.
	Since $\Phi \circ \id \circ \Phi^{-1}$ is the identity functor, 
	we should have 
	\begin{align}\label{isom:Fun:Delta}
		\Phi^{\rm{Fun}}(\Delta_{\ast}\oO_{\mathfrak{U}})
		\cong \Delta_{\ast}\oO_{V^{\vee}}\cneq (\Delta_{\ast}\oO_{V^{\vee}} 
		\stackrel{0}{\to} \Delta_{\ast}\oO_{V^{\vee}}\langle 1 \rangle)
	\end{align}
	in $\mathrm{MF}_{\rm{coh}}^{\mathbb{C}^{\ast}}(V^{\vee} \times_Y V^{\vee}, 
	-w \boxplus w)$.
	We can check the above isomorphism directly as follows:  
	\begin{align*}
		(\Phi^{\rm{Fun}})^{-1}(\Delta_{\ast}\oO_{V^{\vee}})
		&= \RHom_{\mathrm{MF}_{\rm{coh}}^{\mathbb{C}^{\ast}}(V^{\vee} \times_Y V^{\vee}, 
			-w \boxplus w)}(p_{13}^{\ast}\kK_s^{\vee} \otimes p_{24}^{\ast}\kK_s, \Delta_{\ast}\oO_{V^{\vee}}) \\
		&\cong \RHom_{\mathrm{MF}_{\rm{coh}}^{\mathbb{C}^{\ast}}(V^{\vee}, 0)}(\kK_s^{\vee}\otimes_{\oO_{\mathfrak{U}}} \kK_s, \oO_{V^{\vee}}) \\
		& \cong
		\RHom_{\mathrm{MF}_{\rm{coh}}^{\mathbb{C}^{\ast}}(V^{\vee}, w)}(\kK_s, \kK_s) \\&\cong \Delta_{\ast}\oO_{\mathfrak{U}}.
	\end{align*}
	Here the last isomorphism follows using (\ref{isom:K:dual}). 
	
	By the equivalence
	of $\Phi^{\rm{Fun}}$ together with the isomorphism (\ref{isom:Fun:Delta}), 
	the functor 
	$\Phi^{\rm{Fun}}$ induces the 
	isomorphism on relative Hochschild cohomologies
	\begin{align}\label{isom:HH:MF}
		\Phi^{\rm{HH}} \colon 
		\mathrm{HH}^{\ast}(\mathfrak{U}/Y) \stackrel{\cong}{\to}
		\mathrm{HH}^{\ast}((V^{\vee}, w)/Y)).
	\end{align}
	Here the right hand side is defined by
	\begin{align*} 
		\mathrm{HH}^{\ast}((V^{\vee}, w)/Y))\cneq 
		\Hom^{\ast}_{\mathrm{MF}_{\rm{coh}}^{\mathbb{C}^{\ast}}(V^{\vee} \times_Y V^{\vee}, 
			-w \boxplus w)}(\Delta_{\ast}\oO_{V^{\vee}}, \Delta_{\ast}\oO_{V^{\vee}}). 
	\end{align*}
	The isomorphism (\ref{isom:HH:MF})
	is compatible with 
	$\Phi^{\rm{Nat}}$ under the natural maps from relative
	Hochschild cohomologies to natural transformations
	$\id \to \id[\ast]$. 
	
	Using the automorphism of $V^{\oplus 2}$
	given in (\ref{aut:V}), 
	the equivalence $\Phi^{\rm{Fun}}$ is identified with 
	the functor
	\begin{align*}
		D^b_{\rm{coh}}(\mathfrak{U}^{\flat} \times_Y \mathfrak{U})
		\stackrel{\sim}{\to} \mathrm{MF}_{\rm{coh}}^{\mathbb{C}^{\ast}}
		(V^{\vee} \times_Y V^{\vee}, 
		0 \boxplus w)
	\end{align*} 
	given by the FM transform with kernel 
	$p_{13}^{\ast}\kK_0^{\vee} \otimes p_{24}^{\ast}\kK_s$
	on $\mathfrak{U}^{\flat} \times_Y \mathfrak{U} \times_Y V^{\vee} \times_Y V^{\vee}$. 
	Moreover it takes $\oO_Y \boxtimes \oO_{\mathfrak{U}}$ to 
	$\oO_{V^{\vee}} \boxtimes \oO_Y$, as they correspond 
	to $\Delta_{\ast}\oO_{\mathfrak{U}}$, $\Delta_{\ast}\oO_{V^{\vee}}$ under 
	the automorphism (\ref{aut:V}). 
	Also we have 
	\begin{align*}
		&\RHom^{\ast}_{\mathrm{MF}_{\rm{coh}}^{\mathbb{C}^{\ast}}(V^{\vee} \times_Y V^{\vee}, 
			-w \boxplus w)}(\Delta_{\ast}\oO_{V^{\vee}}, \Delta_{\ast}\oO_{V^{\vee}}) \\
		&\cong \RHom_{\mathrm{MF}_{\rm{coh}}^{\mathbb{C}^{\ast}}(V^{\vee} \times_Y V^{\vee}, 
			0 \boxplus w)}(\oO_{V^{\vee}} \boxtimes \oO_Y, \oO_{V^{\vee}} \boxtimes \oO_Y) \\
		& \cong \RHom_{\mathrm{MF}^{\mathbb{C}^{\ast}}_{\rm{coh}}(V^{\vee}, 0)}(\oO_{V^{\vee}}, 
		\oO_{V^{\vee}}) \dotimes_{\oO_Y} \RHom_{\mathrm{MF}^{\mathbb{C}^{\ast}}_{\rm{coh}}(V^{\vee}, w)}(\oO_Y, \oO_Y)
	\end{align*}
	and the map $\Phi^{\rm{HH}}$ is obtained by taking 
	the cohomologies of the following map:
	\begin{align*}
		\Phi^{\flat} \dotimes \Phi \colon 
		&\RHom_{\mathfrak{U}^{\flat}}(\oO_Y, \oO_Y) \dotimes_{\oO_Y}
		\RHom_{\mathfrak{U}}(\oO_{\mathfrak{U}}, \oO_{\mathfrak{U}}) \\
		& \to
		\RHom_{\mathrm{MF}^{\mathbb{C}^{\ast}}_{\rm{coh}}(V^{\vee}, 0)}(\oO_{V^{\vee}}, 
		\oO_{V^{\vee}}) \dotimes_{\oO_Y} \RHom_{\mathrm{MF}^{\mathbb{C}^{\ast}}_{\rm{coh}}(V^{\vee}, w)}(\oO_Y, \oO_Y).
	\end{align*}
	Here $\Phi^{\flat}$ is given in Remark~\ref{rmk:szero}. 
	Note that 
	the first factor of the right hand side is 
	\begin{align*}
		\RHom_{\mathrm{MF}^{\mathbb{C}^{\ast}}_{\rm{coh}}(V^{\vee}, 0)}(\oO_{V^{\vee}}, 
		\oO_{V^{\vee}})=S(V[-2])
	\end{align*}
	and 
	it maps to the natural transforms $\id \to \id[\ast]$
	on $\mathrm{MF}^{\mathbb{C}^{\ast}}_{\rm{coh}}(V^{\vee}, w)$ by 
	the 
	multiplication via the $\oO_{V^{\vee}}$-module structure 
	on factorizations of $w$. 
	Therefore 
	together with Remark~\ref{rmk:szero}, 
	the above description of $\Phi^{\rm{HH}}$ immediately implies
	the commutativity of the diagram (\ref{commute:nat}). 
\end{proof}

The following lemma gives an ind-version of
Theorem~\ref{thm:knoer} and Proposition~\ref{prop:koszul:Z}:

\begin{lem}\label{lem:ind}
	The equivalences (\ref{funct:koszul}), (\ref{koszul:C}) extend to equivalences
	\begin{align*}
		\Phi^{\rm{ind}} \colon \Ind \Dbc([\fU/G]) \stackrel{\sim}{\to}
		\mathrm{MF}_{\qcoh}^{\C}([V^{\vee}/G], w), \
		\Phi^{\rm{ind}} \colon  \Ind \cC_{\zZ} \stackrel{\sim}{\to} 
		\MF_{\mathrm{qcoh}}^{\mathbb{C}^{\ast}}([V^{\vee}/G], w)_{\zZ}. 
	\end{align*}
\end{lem}
\begin{proof}
	It is enough to show the second equivalence. 
	By taking ind-completion of the equivalence (\ref{koszul:C}), 
	we have the equivalence
	\begin{align*}
		\Phi^{\ind} \colon \Ind(\cC_{\zZ}) \stackrel{\sim}{\to}
		\Ind (\MF^{\CC}_{\coh}([V^{\vee}/G], w)_{\zZ}).
	\end{align*}
The left hand side is $\Ind \cC_{\zZ}$ by Lemma~\ref{lem:gcomplete}. 
As for the right hand side, it is proved in~\cite[Proposition~3.15]{MR3270588}
that $\MF_{\qcoh}^{\CC}([V^{\vee}/G], w)$ is 
	generated by the 
	subcategory of compact objects 
	$\MF^{\CC}_{\coh}([V^{\vee}/G], w)$. 
	The same argument applies to the fixed support case, 
	i.e. $\MF_{\qcoh}^{\CC}([V^{\vee}/G], w)_{\zZ}$
	is generated by the subcategory of compact objects 
	$\MF^{\CC}_{\coh}([V^{\vee}/G], w)_{\zZ}$. 
	Indeed the essential point in the proof of \textit{loc.~cit.~}
	is the fact that any $G$-equivariant quasi-coherent sheaf 
	is a union of $G$-equivariant 
	coherent sheaves. 
	This obviously holds for fixed supported case: 
	any $G$-equivariant quasi-coherent sheaf 
	supported on $Z$ is a union of $G$-equivariant 
	coherent sheaves supported on $Z$.
	Therefore  
	the naturally defined functor
	\begin{align*}
		\Ind (\MF^{\CC}_{\coh}([V^{\vee}/G], w)_{\zZ})
		\to \MF_{\qcoh}^{\CC}([V^{\vee}/G], w)_{\zZ}
	\end{align*}
	is an equivalence by~\cite[Corollary~6.3.5]{Kabook}, so 
	we have the desired equivalence. 
\end{proof}
\begin{rmk}\label{rmk:qcoh+}
	By~\cite[Proposition~1.2.4]{MR3136100}, we have the equivalence
	\begin{align*}
		\Ind \Dbc([\fU/G])^+ \stackrel{\sim}{\to}
		D_{\qcoh}([\fU/G])^+
		\end{align*}
	where the subscript $+$ indicates 
	bounded below cohomologies.  
	Under the above equivalence, the equivalence $\Phi^{\ind}$ restricts to the functor
	\begin{align*}
		\Phi \colon 	D_{\qcoh}([\fU/G])^+
		\stackrel{\sim}{\to} \MF_{\qcoh}^{\C}([V^{\vee}/G], w)^+
		\end{align*}
	where the RHS is the subcategory of $\MF_{\qcoh}^{\C}([V^{\vee}/G], w)^+$
	with bounded below $\C$-weights. The above equivalence 
	is given by $\kK_s \otimes_{\oO_{\fU}}(-)$ as in 
	Theorem~\ref{thm:knoer}. 
	\end{rmk}

\subsection{Some functorial properties of Koszul duality equivalence}\label{subsec:funct}
Here we show some functorial properties of 
Koszul duality equivalence in Theorem~\ref{thm:knoer}. 
Below we will use some terminology of 
functors in Subsection~\ref{subsub:QCA}.

\subsubsection{Functoriality under push-forward}
We define the morphism of equivariant tuples as follows: 
\begin{defi}\label{def:eq:morphism}
Let $G_i$ for $i=1, 2$ be affine algebraic groups
with an algebraic group homomorphism $\phi \colon G_1 \to G_2$, 
and $(Y_i, V_i, s_i)$ be
$G_i$-equivariant tuples. 
A \textit{morphism of equivariant tuples}
is a commutative diagram 
\begin{align}\label{dia:Y}
	\xymatrix{
		V_1 \ar[r]^-{g} \ar[d] & V_2 \ar[d] \\
		Y_1 \ar[r]^-{f} \ar@/^10pt/[u]^-{s_1}
		& Y_2, \ar@/_10pt/[u]_-{s_2}&
	} 
\end{align}
satisfying the followings: 
\begin{enumerate}
	\item The bottom morphism $f \colon Y_1 \to Y_2$ is equivariant with respect to 
	$\phi \colon G_1 \to G_2$. 
	\item The top morphism $g$ is a composition 
	$V_1 \stackrel{g'}{\to} f^{\ast}V_2 \stackrel{f}{\to}V_2$, 
	where $g'$ is a morphism of $G_1$-equivariant vector bundles on 
	$Y_1$. 
	Here the $G_1$-equivariant structure of $f^{\ast}V_2$ is induced 
	by $\phi \colon G_1 \to G_2$. 
\end{enumerate}
If $G=G_1=G_2$ and $\phi=\id$, we call the 
above diagram as a morphism of $G$-equivariant tuples. 
\end{defi}
The diagram (\ref{dia:Y}) induces the diagram 
for smooth stacks $\yY_i=[Y_i/G_i], \ \vV_i=[V_i/G_i]$, 
\begin{align}\label{dia:Ystack}
	\xymatrix{
	\vV_1 \ar[r]^-{g} \ar[d] & \vV_2 \ar[d] \\
	\yY_1 \ar[r]^-{f} \ar@/^10pt/[u]^-{s_1}
	& \yY_2,  \ar@/_10pt/[u]_-{s_2}&
}
\end{align}
which also induces the 
morphism of derived stacks 
\begin{align*}
	\mathbf{f} \colon 
	[\fU_1/G_1] \to [\fU_2/G_2], \ 
	\fU_i \cneq \Spec \rR(V_i \to Y_i, s_i). 
\end{align*}
Then we have the push-forward functor
(see~\cite[Section~3.6]{MR3037900})
\begin{align*}
	\mathbf{f}_{\ast}^{\rm{ind}} \colon 
	\Ind \Dbc([\fU_1/G_1]) \to \Ind \Dbc([\fU_2/G_2]). 
	\end{align*}
Let $w_i \colon V_i^{\vee} \to \mathbb{C}$ be given as in (\ref{def:w})
defined from $(Y_i, V_i, s_i)$. 
The diagram (\ref{dia:Y}) induces the 
following diagram
\begin{align}\label{diagram:dual}
	\xymatrix{
		&  \mathbb{C} & \\
		V_1^{\vee} \ar[ur]^-{w_1} \ar[d]_-{p_1} & f^{\ast}V^{\vee}_2 \ar[l]^-{g} 
		\ar[r]_-{f} \ar@{}[dr]|\square
		\ar[d]_-{\overline{p}}
		\ar[u]^{\overline{w}}
		& V_2^{\vee} \ar[d]^-{p_2} \ar[ul]_-{w_2} \\
		Y_1 & Y_1 \ar@{=}[l] \ar[r]_-{f} & Y_2. 
	}
\end{align}
Here $\overline{w}$ is determined by 
\begin{align*}
	\overline{w}=f^{\ast}s_2 \in \Gamma(Y_1, f^{\ast}V_2) \subset
	\Gamma(Y_1, S(f^{\ast}V_2)). 
\end{align*}
The commutativity of (\ref{dia:Y}) implies that 
the diagram (\ref{diagram:dual}) is also commutative. 
The following lemma will be used later. 
\begin{lem}\label{lem:critical}
	In the notation of the diagram (\ref{diagram:dual}),
	we have 
	\begin{align}\label{identity:crit}
		g^{-1}(\mathrm{Crit}(w_1)) \cap 
		f^{-1}(\mathrm{Crit}(w_2))= t_0(\Omega_{\mathfrak{U}_2}[-1]\times_{\fU_2}\fU_1). 
	\end{align}
\end{lem}
\begin{proof}
	The left hand side of (\ref{identity:crit}) is 
	defined by the ideal of $S_{\oO_{Y_1}}(f^{\ast}V_2)$
	generated by the image of the following map
	\begin{align}\label{sds0}
		(s_1\oplus f^{\ast}s_2) \oplus (g \circ ds_1\oplus f^{\ast}ds_2) \colon 
		(V_1^{\vee} \oplus f^{\ast}V_2^{\vee}) \oplus (T_{Y_1} \oplus f^{\ast}T_{Y_2})
		\to \oO_{Y_1} \oplus f^{\ast}V_2. 
	\end{align}
	Since there exist commutative diagrams
	\begin{align*}
		\xymatrix{
		f^{\ast}V_2^{\vee} \ar[r] \ar@/^15pt/[rr]^-{f^{\ast}s_2}& V_1^{\vee} \ar[r]_-{s_1} & \oO_{Y_1}}, \ 
		\xymatrix{
		T_{Y_1} \ar[r] \ar@/^15pt/[rr]^-{g \circ ds_1}& f^{\ast}T_{Y_2} \ar[r]_-{f^{\ast}ds_2} & f^{\ast}V_2}
		\end{align*}
		the image of the
	map (\ref{sds0}) coincides with the image of the map
	\begin{align}\label{sds}
		s_1 \oplus f^{\ast}ds_2 \colon 
		V_1^{\vee} \oplus f^{\ast}T_{Y_2} \to \oO_{Y_1} \oplus f^{\ast}V_2. 
	\end{align}
	On the other hand, we have
	\begin{align*}
		\Omega_{\mathfrak{U}_2}[-1]\times_{\fU_2}\fU_1
		=\Spec S_{\oO_{Y_1}}(V_1^{\vee}[1] \oplus f^{\ast}T_{Y_2}[1] \oplus f^{\ast}V_2)
	\end{align*}
	and the differential on the RHS is the Koszul differential determined by 
	the map (\ref{sds}). 
	Therefore (\ref{identity:crit}) holds. 
\end{proof}
In the notation of the diagram (\ref{diagram:dual}), we 
define the following functor
\begin{align}\label{funct:fg}
	f_{\ast} \circ g^{\ast} \colon 
	\mathrm{MF}_{\qcoh}^{\C}([V_1^{\vee}/G_1], w_1)
	\stackrel{g^{\ast}}{\to}
	\mathrm{MF}_{\qcoh}^{\C}([f^{\ast}V_2^{\vee}/G_1], \overline{w})
	\stackrel{f_{\ast}}{\to} 	\mathrm{MF}_{\qcoh}^{\C}([V_2^{\vee}/G_2], w_2). 
\end{align}
The above functor is compatible with 
compositions of morphisms of derived stacks as follows. 
Let $(Y_3, V_3, s_3)$ be another $G_3$-equivariant tuple, 
and suppose that we have 
morphisms of equivariant tuples
with respect to algebraic group homomorphisms 
$G_1 \to G_2 \to G_3$
\begin{align}\notag
	\xymatrix{
		V_1 \ar[d] \ar[r]_-{g} \ar@/^10pt/[rr]^-{g''} & V_2 \ar[r]_-{g'}  \ar[d] & V_3 \ar[d] \\
		Y_1 \ar@/^10pt/[u]^-{s_1}\ar[r]^-{f} \ar@/_10pt/[rr]_-{f''} & Y_2 \ar[r]^-{f'} 
		\ar@/^10pt/[u]^-{s_2} 
		& Y_3 \ar@/_10pt/[u]_-{s_3}
	}
\end{align}
which induces the diagram of derived stacks
\begin{align}\notag
	\mathbf{f}'' \colon 
	[\fU_1/G_1] \stackrel{\mathbf{f}}{\to}
	[\fU_2/G_2] \stackrel{\mathbf{f}'}{\to} [\fU_3/G_3]. 
\end{align}
\begin{lem}\label{lem:funct:3}
	There is a natural isomorphism of functors
	\begin{align*}
		(f'_{\ast}{g'}^{\ast})(f_{\ast}g^{\ast}) \cong 
		f''_{\ast}{g''}^{\ast} \colon 
		\mathrm{MF}_{\qcoh}^{\C}([V_1^{\vee}/G_1], w_1)
		\to
		\mathrm{MF}_{\qcoh}^{\C}([V_3^{\vee}/G_3], w_3).
	\end{align*}
\end{lem}
\begin{proof}
	We have the following diagram
	\begin{align}\notag
		\xymatrix{
			& & {f''}^{\ast}\vV_3^{\vee}
			\ar@{}[dd]|\square
			\ar[ld]_-{g'} \ar[rd]^-{f} \ar@/_30pt/[lldd]_-{g''} 
			\ar@/^30pt/[rrdd]^-{f''} & & \\
			& f^{\ast}\vV_2^{\vee} \ar[ld]_-{g} 
			\ar[rd]^-{f} & & {f'}^{\ast}\vV_3^{\vee} \ar[ld]_-{g'}
			\ar[rd]^-{f'} & \\
			\vV_1^{\vee} & & \vV_2^{\vee} & & \vV_3^{\vee}.
		}
	\end{align}	
	Here the middle square is a derived Cartesian. 
	Then the lemma follows from the derived base change. 	
\end{proof}

We have the following lemma which 
describes $\mathbf{f}_{\ast}^{\ind}$ 
in the Koszul dual size.

\begin{lem}\label{lem:commute1}
	The following diagram commutes:
	\begin{align}\label{dia:com:ind}
		\xymatrix{
			\Ind \Dbc([\fU_1/G_1]) \ar[r]^-{\Phi_1^{\rm{ind}}} \ar[d]_-{\mathbf{f}^{\rm{ind}}_{\ast}} & 
			\mathrm{MF}_{\qcoh}^{\C}([V_1^{\vee}/G_1], w_1) \ar[d]^-{f_{\ast}\circ g^{\ast}}  \\
			\Ind \Dbc([\fU_2/G_2]) \ar[r]^-{\Phi_2^{\rm{ind}}}  & 
			\mathrm{MF}_{\qcoh}^{\C}([V_2^{\vee}/G_2], w_2). 	
		}
	\end{align}
	Here the horizontal arrows are equivalences in Lemma~\ref{lem:ind}. 
	In particular if the morphism $f \colon \yY_1 \to \yY_2$ 
	in the diagram (\ref{dia:Ystack}) is proper, the
	above diagram restricts to the commutative diagram 
	\begin{align}\notag
		\xymatrix{
			\Dbc([\fU_1/G_1]) \ar[r]^-{\Phi_1} 
			\ar[d]_-{\mathbf{f}_{\ast}} & 
			\mathrm{MF}_{\coh}^{\C}([V_1^{\vee}/G_1], w_1) \ar[d]^-{f_{\ast}\circ g^{\ast}}  \\
			\Dbc([\fU_2/G_2]) \ar[r]^-{\Phi_2}  & 
			\mathrm{MF}_{\coh}^{\C}([V_2^{\vee}/G_2], w_2). 	
		}
	\end{align}
	\end{lem}
\begin{proof}
		It is enough to show the commutativity of (\ref{dia:com:ind})
	for $\fF^{\bullet} \in \Dbc([\fU_1/G_1])$. 
	Then $\mathbf{f}_{\ast}^{\ind}\fF^{\bullet}$ 
	has bounded below cohomologies, so by Remark~\ref{rmk:qcoh+}
	we can compute 
	$\Phi_2^{\ind}\mathbf{f}_{\ast}^{\ind}(\fF^{\bullet})$
	as $\kK_{s_2} \otimes_{\oO_{\fU_2}}\mathbf{f}_{\ast}(\fF^{\bullet})$. 
		The morphism $\mathbf{f}$ factors as 
	\begin{align*}
		\mathbf{f} \colon [\fU_1/G_1] \to [\fU_2/G_1]
		\to [\fU_2/G_2].  
	\end{align*}
	Therefore by Lemma~\ref{lem:funct:3}, we may assume either 
	$G_1=G_2$ or $(Y_1, V_1, s_1)=(Y_2, V_2, s_2)$. 
	
	In the case of $(Y_1, V_1, s_1)=(Y_2, V_2, s_2)$,  
	the vertical arrows in (\ref{dia:com:ind}) are push-forwards
	along 
	$[\fU_1/G_1] \to [\fU_1/G_2]$ and 
	$[V_1^{\vee}/G_1] \to [V_1^{\vee}/G_2]$. 
	Therefore we can compute $\mathbf{f}_{\ast}\fF^{\bullet}$ by replacing 
	$\fF^{\bullet}$ by a bounded below complex such that each term 
	$\fF^i$ is acyclic with respect to 
	the push-forward along $[Y_1/G_1] \to [Y_1/G_2]$, 
	and regard them as a complex of $G_2$-equivariant sheaves. 
	Then $\fF^i \otimes_{\oO_{Y_1}}\oO_{V_1^{\vee}}$ is acyclic 
	with respect to the push-forward along 
	$[V_1^{\vee}/G_1] \to [V_1^{\vee}/G_2]$, so the 
	diagram 
(\ref{dia:com:ind}) tautologically commutes. 

In the case of $G_1=G_2$, 
the morphism $f \colon [Y_1/G_1] \to [Y_2/G_1  ]$ is affine. 
So if $\fF^{\bullet}$ is represented by
a complex of $\oO_{Y_1}$-modules $(\fF^{\bullet}, d_{\fF^{\bullet}})$
together with $\eta^i \colon \fF^i \to V_1 \otimes \fF^{i-1}$, 
then $\mathbf{f}_{\ast}\fF^{\bullet}$ is represented by 
\begin{align*}
	(f_{\ast}\fF^{\bullet}, f_{\ast}d_{\fF^{\bullet}}), \quad
	f_{\ast}(g \circ \eta^i) \colon 
	f_{\ast}\fF^i \to V_2 \otimes f_{\ast}\fF^{i-1}. 
	\end{align*}
Here $g \colon V_1 \to f^{\ast}V_2$ is a morphism 
in the diagram (\ref{dia:Y}). 
Therefore 
$\kK_{s_2} \otimes_{\oO_{\fU_2}}\mathbf{f}_{\ast}\fF^{\bullet}$ is 
\begin{align}\label{fact:push1}
	\left(
	\bigoplus_{i\in \mathbb{Z}}f_{\ast}\fF^i \langle i \rangle
	\otimes_{\oO_{Y_2}}\oO_{V_2^{\vee}} \to 
	\bigoplus_{i\in \mathbb{Z}}f_{\ast}\fF^i \langle i+1 \rangle
	\otimes_{\oO_{Y_2}}\oO_{V_2^{\vee}}\right)
\end{align}
where the morphism is a $\oO_{V_2^{\vee}}$-module homomorphism 
determined by 
\begin{align*}
	f_{\ast}d_{\fF}^i+f_{\ast}(g \circ \eta^i) \colon 
	f_{\ast}\fF^i \to f_{\ast}\fF^{i+1} \oplus (V_2 \otimes f_{\ast}\fF^{i-1}). 
	\end{align*}
On the other hand, 
$g^{\ast}(\Phi_1(\fF^{\bullet}))$
is given by 
\begin{align}\label{fact:push2}
	\left(
	\bigoplus_{i\in \mathbb{Z}}\fF^i \langle i \rangle
	\otimes_{\oO_{Y_1}}\oO_{f^{\ast}V_2^{\vee}} \to 
	\bigoplus_{i\in \mathbb{Z}}\fF^i \langle i+1 \rangle
	\otimes_{\oO_{Y_1}}\oO_{f^{\ast}V_2^{\vee}}\right)
\end{align}
where the morphism is a $\oO_{f^{\ast}V_2^{\vee}}$-module homomorphism 
determined by 
\begin{align*}
	d_{\fF}^i+g \circ \eta^i \colon 
	\fF^i \to \fF^{i+1} \oplus (f^{\ast}V_2 \otimes_{\oO_{Y_1}} \fF^{i-1}). 
\end{align*}
Therefore by pushing forward (\ref{fact:push2}) along 
$f \colon f^{\ast}V_2^{\vee} \to V_2^{\vee}$, 
we obtain the object (\ref{fact:push1}). 
\end{proof}

Let $(Y, V, s)$ be a $G$-equivariant tuple, 
and consider the diagram (\ref{diagram:UA}). 
For the zero section $0_Y \colon Y \to V^{\vee}$, 
we have $0_Y^{\ast}w=0$. 
Therefore we have the functor
\begin{align*}
	0_Y^{\ast} \colon \mathrm{MF}^{\C}_{\coh}([V^{\vee}/G], w) \to 
	\mathrm{MF}_{\coh}^{\C}([Y/G], 0). 
\end{align*}
The following lemma is obtained from Lemma~\ref{lem:commute1}, which 
is also easy to check directly. 
\begin{lem}\label{lem:useful}
	The following diagram commutes
	\begin{align}\label{dia:useful}
		\xymatrix{
			\Dbc([\fU/G]) \ar[r]^-{\Phi}_-{\sim}  \ar[d]_-{j_{\ast}} & 
			\mathrm{MF}_{\coh}^{\C}([V^{\vee}/G], w) \ar[d]^-{0_Y^{\ast}} \\
			\Dbc([Y/G]) \ar[r]_-{\sim} & \mathrm{MF}_{\coh}^{\C}([Y/G], 0). 
		}
	\end{align}
	Here the 
	bottom arrow is a tautological 
	equivalence (see~Remark~\ref{rmk:tautological}). 
\end{lem}
\begin{proof}
	The lemma follows by applying Lemma~\ref{lem:commute1} to 
	the following morphism of $G$-equivariant tuples
	\begin{align*}
			\xymatrix{
			V \ar[r]^-{p} \ar[d]^-{p} & Y \ar[d]_-{\id} \\
			Y \ar[r]^-{\id} \ar@/^10pt/[u]^-{s}
			& Y. \ar@/_10pt/[u]_-{\id}&
		} 
		\end{align*}
	Alternatively from 
	the construction of $\Phi$
	and noting that $\oO_Y \otimes_{\oO_{V^{\vee}}} \kK_s=\oO_{\fU}$, we have
	\begin{align*}
		0_Y^{\ast}\Phi(\fF)=\oO_{Y}\otimes_{\oO_{V^{\vee}}} \kK_s \otimes_{\oO_{\fU}}\fF
		\cong \fF. 
	\end{align*}
\end{proof}

\subsubsection{Functoriality under pull-back}
Let us consider a morphism of equivariant tuples in Definition~\ref{def:tuple}. 
Suppose that the morphism $f \colon \yY_1 \to \yY_2$ in the diagram (\ref{dia:Ystack}) is a proper 
morphism of smooth stacks, so 
in particular $\bff \colon [\fU_1/G_1] \to [\fU_2/G_2]$ is proper. 
Then we have the continuous right adjoint 
of $\bff_{\ast}^{\rm{ind}}$ (see~\cite[Section~10.1]{MR3136100})
\begin{align*}
	\bff^{!} \colon \Ind \Dbc([\fU_2/G_2]) \to \Ind \Dbc([\fU_1/G_1]). 
	\end{align*}
On the other hand, we also have the functor 
\begin{align}\notag
	g_{\ast} \circ f^{!} \colon 
	\mathrm{MF}_{\qcoh}^{\C}([V_2^{\vee}/G_2], w_2)
	\stackrel{f^{!}}{\to}
	\mathrm{MF}_{\qcoh}^{\C}([f^{\ast}V_2^{\vee}/G_1], \overline{w})
	\stackrel{g_{\ast}}{\to} 	\mathrm{MF}_{\qcoh}^{\C}([V_1^{\vee}/G_1], w_1)
\end{align}
which is a right adjoint of 
the functor (\ref{funct:fg}). 
\begin{lem}\label{lem:commute2}
	Suppose that $f \colon \yY_1 \to \yY_2$ is proper. 
	Then the following diagram commutes:
	\begin{align*}
		\xymatrix{
			\Ind \Dbc([\fU_2/G_2]) \ar[r]^-{\Phi_2^{\rm{ind}}} \ar[d]_-{\mathbf{f}^!} & 
			\mathrm{MF}_{\qcoh}^{\C}([V_2^{\vee}/G_2], w_2) \ar[d]^-{g_{\ast} \circ f^!}  \\
			\Ind \Dbc([\fU_1/G_1]) \ar[r]^-{\Phi_1^{\rm{ind}}}  & 
			\mathrm{MF}_{\qcoh}^{\C}([V_1^{\vee}/G_1], w_1). 	
		}
	\end{align*}
\end{lem}	
\begin{proof}
	The lemma follows from Lemma~\ref{lem:commute1} by taking the right 
	adjoints of the vertical arrows in the diagram (\ref{dia:com:ind}). 
\end{proof}

Suppose that $\mathbf{f} \colon [\fU_1/G_1] \to [\fU_2/G_2]$ is 
quasi-smooth. 
Then we have the functor (see~\cite[Section~3.1]{MR3701352})
\begin{align*}
	\bff^{\ast} \colon \Dbc([\fU_2/G_2]) \to \Dbc([\fU_1/G_1]). 
	\end{align*}
If furthermore $g \colon f^{\ast}\vV_2^{\vee} \to \vV_1^{\vee}$ in 
the diagram (\ref{diagram:dual}) is proper, 
then we have the functor 
\begin{align*}
	g_{!} \circ f^{\ast} \colon 
	\mathrm{MF}_{\coh}^{\C}([V_2^{\vee}/G_2], w_2)
	\stackrel{f^{\ast}}{\to}
	\mathrm{MF}_{\qcoh}^{\C}([f^{\ast}V_2^{\vee}/G_1], \overline{w})
	\stackrel{g_{!}}{\to} 	\mathrm{MF}_{\qcoh}^{\C}([V_1^{\vee}/G_1], w_1)
\end{align*}
which is a left adjoint of the functor (\ref{funct:fg}). 
\begin{lem}\label{lem:commute3}
	Suppose that $\mathbf{f} \colon [\fU_1/G_1] \to [\fU_2/G_2]$ is 
	quasi-smooth and the morphism $g \colon f^{\ast}\vV_2^{\vee} \to \vV_1^{\vee}$ is proper. 
	Then the following diagram commutes:
	\begin{align*}
		\xymatrix{
			\Dbc([\fU_2/G_2]) \ar[r]^-{\Phi_2} \ar[d]_-{\mathbf{f}^{\ast}} &
			\mathrm{MF}_{\coh}^{\C}([V_2^{\vee}/G_2], w_2) \ar[d]^-{g_{!} \circ f^{\ast}}  \\
			\Dbc([\fU_1/G_1]) \ar[r]^-{\Phi_1}  & 
			\mathrm{MF}_{\coh}^{\C}([V_1^{\vee}/G_1], w_1). 	
		}
	\end{align*}
\end{lem}
\begin{proof}
	The ind-completion of $\mathbf{f}^{\ast}$
	\begin{align*}
		\mathbf{f}^{\ind \ast} \colon \Ind \Dbc([\fU_2/G_2])
		\to \Ind \Dbc([\fU_1/G_1])
	\end{align*}
	is a left adjoint of $\mathbf{f}_{\ast}^{\ind}$ by~\cite[Proposition~3.1.6]{MR3701352}. 
	Therefore the lemma follows by taking left adjoints of the vertical arrows in the diagram (\ref{dia:com:ind}), and restrict it
	to $\Dbc([\fU_2/G_2])$. 
\end{proof}

\section{Categorical Donaldson-Thomas theory for local surfaces}\label{sec:catDT:shift}
In this section, 
for a quasi-smooth and QCA derived stack $\mathfrak{M}$ and 
a conical closed substack $\zZ \subset \nN \cneq t_0(\Omega_{\fM}[-1])$, 
we introduce 
$\C$-equivariant DT category for the complement 
$\nN^{\rm{ss}} \cneq \nN \setminus \zZ$. 
A basic construction is to take the 
Verdier quotient 
\begin{align}\label{subintro:DTcat}
	\dDT^{\C}(\nN^{\rm{ss}}) \cneq \Dbc(\fM)/\cC_{\zZ}
	\end{align}
where $\cC_{\zZ} \subset \Dbc(\fM)$ is the subcategory
of objects whose singular supports are contained in $\zZ$. 
We call the above quotient category as \textit{DT category}. 
The above construction is a global analogue of the 
quotient category considered in Proposition~\ref{prop:koszul:Z}. 
So in particular if $\fM$ is a global 
derived zero locus then $\dDT^{\C}(\nN^{\rm{ss}})$ is 
equivalent to the derived category of $\C$-equivariant factorizations.  
We will study the properties of the above quotient category in detail. 

In the presence of $\mathbb{C}^{\ast}$-automorphisms 
in the stabilizer groups of $\mathfrak{M}$, 
we further introduce the
variants of 
DT categories either by taking the 
$\mathbb{C}^{\ast}$-rigidifications or fix the $\mathbb{C}^{\ast}$-weight $\lambda$. 
The $\lambda$-twist version is an analogue of the category of 
twisted sheaves, twisted by a Brauer class, studied in~\cite{MR2309155}. 
We will also define their dg-enhancements. 
The relevant notation is summarized in the following table: 

\begin{table}[htb]\label{table1}
\caption{Notation of $\mathbb{C}^{\ast}$-equivariant
DT categories}
\begin{tabular}{|l||c|c|c|} \hline
  & $\mathbb{C}^{\ast}$-equivariant & $\mathbb{C}^{\ast}$-rigidified &
  $\lambda$-twist  \\ \hline \hline
$\triangle$ category & $\mathcal{DT}^{\mathbb{C}^{\ast}}(\nN^{\rm{ss}})$ &
  $\mathcal{DT}^{\mathbb{C}^{\ast}}((\nN^{\rm{ss}})^{\mathbb{C}^{\ast}\rig})$ &
$\mathcal{DT}^{\mathbb{C}^{\ast}}(\nN^{\rm{ss}})_{\lambda}$ \\ \hline
dg category & $\mathcal{DT}_{\rm{dg}}^{\mathbb{C}^{\ast}}(\nN^{\rm{ss}})$ &
  $\mathcal{DT}_{\rm{dg}}^{\mathbb{C}^{\ast}}((\nN^{\rm{ss}})^{\mathbb{C}^{\ast}\rig})$ &
$\mathcal{DT}_{\rm{dg}}^{\mathbb{C}^{\ast}}(\nN^{\rm{ss}})_{\lambda}$ \\ \hline
\end{tabular}
\end{table}

We will also give a variant of the above construction, 
defined as limits of quotient categories
\begin{align}\label{subintro:DTcat2}
	\wdDT^{\C}(\nN^{\rm{ss}}) \cneq 
	\lim_{\fU \stackrel{\alpha}{\to}\fM} \left(\Dbc(\fU)/\cC_{\alpha^{\ast}\zZ} \right). 
	\end{align}
The above construction is a more direct gluing of derived 
factorization categories. 
The $\C$-rigidified version, 
the $\lambda$-twisted version and dg-enhancements
are similarly defined, and 
the relevant notation is summarized in the following table: 

\begin{table}[htb]\label{table2}
	\caption{Notation of $\mathbb{C}^{\ast}$-equivariant
		categorical $\widehat{\rm{DT}}$ theories}
	\begin{tabular}{|l||c|c|c|} \hline
		& $\mathbb{C}^{\ast}$-equivariant & $\mathbb{C}^{\ast}$-rigidified &
		$\lambda$-twist  \\ \hline \hline
		$\triangle$ category & $\wdDT^{\C}(\nN^{\rm{ss}})$ &
		$\wdDT^{\C}((\nN^{\rm{ss}})^{\mathbb{C}^{\ast}\rig})$ &
		$\wdDT^{\C}(\nN^{\rm{ss}})_{\lambda}$ \\ \hline
		dg category & $\widehat{\mathcal{DT}}_{\rm{dg}}^{\mathbb{C}^{\ast}}(\nN^{\rm{ss}})$ &
		$\widehat{\mathcal{DT}}_{\rm{dg}}^{\mathbb{C}^{\ast}}((\nN^{\rm{ss}})^{\mathbb{C}^{\ast}\rig})$ &
		$\widehat{\mathcal{DT}}_{\rm{dg}}^{\mathbb{C}^{\ast}}(\nN^{\rm{ss}})_{\lambda}$ \\ \hline
	\end{tabular}
\end{table}

There is a natural functor 
from the former category (\ref{subintro:DTcat}) 
to the latter category (\ref{subintro:DTcat2}), but in general 
it is not clear whether this is an equivalence or not. 
It will turn out that, in many cases we are interested in, 
both constructions give equivalent categories. 
Both of constructions will play important roles in this paper. 
The former category is useful to discuss 
categorified Hall products for DT categories
in Section~\ref{sec:cat:hall}, while the latter category is useful 
to glue window subcategories in Section~\ref{sec:window:DT}. 
The comparisons of both constructions will be discussed in Section~\ref{sec:compare}. 
In this section, we mainly focus on the former construction and 
develope its general theory. 

It is a natural question to ask whether  
our DT category recovers the original DT invariant. 
If $\nN^{\rm{ss}}$ is a scheme, then the associated DT invariant 
is defined by taking the integration of the Behrend function which is a point-wise 
Euler characteristic of the vanishing cycles. 
As Efimov proved in~\cite{MR3815168}, 
there is a close relationship between vanishing cycles and 
derived categories of factorizations:
the periodic cyclic homology of 
the derived category of factorizations is isomorphic to 
the hyper-cohomology of the perverse sheaf of vanishing cycles. 
Based on this result, we will give  
some decategorification result for DT categories, proving
that the Euler characteristic of periodic cyclic homology
of the DT category recovers the DT invariant.  

Based on the above construction, we will introduce 
DT category for local surfaces. 
For a smooth projective surface $S$, the local surface 
is the total space of its canonical line bundle 
$X=\mathrm{Tot}(\omega_S)$, which is a non-compact CY 3-fold. 
A key fact is that the moduli stack
of compactly supported coherent sheaves on $X$ is isomorphic to the 
$(-1)$-shifted cotangent over the derived moduli stack $\fM_S$
of coherent sheaves on $S$. 
If we impose stability condition, 
the unstable locus is a conical closed substack 
in $t_0(\Omega_{\fM_S}[-1])$. 
For a stability condition $\sigma$, 
we define 
the DT category for the moduli space $M_X^{\sigma}(v)$ of $\sigma$-stable sheaves on $X$ with 
numerical class $v$ based on the construction (\ref{subintro:DTcat}), and denoted as 
\begin{align*}
	\dDT^{\C}(M_X^{\sigma}(v)). 
	\end{align*}
We will then formulate conjectures on categorical wall-crossing of 
DT categories following d-critical analogue of D/K equivalence 
conjecture, e.g. 
wall-crossing equivalence of DT categories for one dimensional 
stable sheaves. It also categorifies 
wall-crossing invariance of genus zero Gopakumar-Vafa invariants. 

The organization of this section is as follows. 
In Section~\ref{subsec:back}, we recall several 
notions of derived stacks, quasi-smoothness, (quasi, ind)coherent sheaves on them, 
$(-1)$-shifted cotangent, and their singular supports.
In Section~\ref{subsec:catVer}, we introduce 
DT categories for $(-1)$-shifted cotangents
 and prove their fundamental properties. 
In Section~\ref{subsec:cohDT}, we prove some decategorification 
result of our DT category. In Section~\ref{subsec:catDTsurf}, we introduce 
DT category for local surfaces and 
propose the categorical wall-crossing conjecture.

\subsection{Some background on derived stacks}\label{subsec:back}
\label{subsec:qsmooth}
\subsubsection{Quasi-smooth derived stack}
Below, we denote by $\mathfrak{M}$ a
derived Artin stack over $\mathbb{C}$.
This means that 
$\mathfrak{M}$ is a contravariant
$\infty$-functor from 
the $\infty$-category of 
affine derived schemes over $\mathbb{C}$ to 
the $\infty$-category of 
simplicial sets 
\begin{align*}
\mathfrak{M} \colon 
dAff^{op} \to SSets
\end{align*}
satisfying some conditions (see~\cite[Section~3.2]{MR3285853} for details). 
Here $dAff^{\rm{op}}$ is defined to be the
$\infty$-category of 
commutative simplicial $\mathbb{C}$-algebras, 
which is equivalent to the $\infty$-category of 
commutative differential graded 
$\mathbb{C}$-algebras with non-positive degrees. 
The classical truncation of $\mathfrak{M}$ is denoted by 
\begin{align*}
\mM \cneq t_0(\mathfrak{M}) \colon 
Aff^{op} \hookrightarrow 
dAff^{op} \to SSets
\end{align*}
where the first arrow is a natural functor 
from the category of affine schemes
to affine derived schemes. 

Following~\cite{MR3285853}, we define the 
dg-category of quasi-coherent sheaves on $\mathfrak{M}$
as follows. 
For an affine derived scheme $\mathfrak{U}=\Spec R$
for a cdga $R$, 
let $\Mod(R)$ be the dg-category 
of differential graded 
$R$-modules which is equipped 
with a 
projective model category structure. 
Then 
$L_{\rm{qcoh}}(\mathfrak{U})$ 
is defined to be the dg-categorical 
localization of 
$\Mod(R)$
by quasi-isomorphisms
in the sense of~\cite[Section~2.4]{MR2762557}. 
The category $L_{\rm{qcoh}}(\mathfrak{U})$ is a 
dg-category, which is weakly equivalent to 
the dg-subcategory of $\Mod(R)$
consisting of cofibrant $R$-modules, 
and its homotopy category is equivalent to 
the derived category $D_{\rm{qcoh}}(\mathfrak{U})$
of dg-modules over $R$. 
Then the dg-category
$L_{\rm{qcoh}}(\mathfrak{M})$ is defined 
to be the limit in the $\infty$-category of dg-categories 
\begin{align}\label{L:limit}
L_{\rm{qcoh}}(\mathfrak{M}) \cneq
\lim_{\mathfrak{U} \to \mathfrak{M}} L_{\rm{qcoh}}(\mathfrak{U}).
\end{align}
The limit is taken for the 
$\infty$-category of smooth morphisms 
$\alpha \colon \fU \to \fM$ for affine 
derived schemes $\fU$ with 
1-morphisms given by commutative 
diagrams
\begin{align}\label{dia:smooth}
\xymatrix{
\mathfrak{U} \ar[rd]_{\alpha} \ar[rr]^-{\rho} & & \mathfrak{U}' 
\ar[ld]^-{\alpha'} \\
& \mathfrak{M}. & 
}
\end{align}
Here $\rho$ is a 0-representable smooth 
morphism, 
$\alpha' \circ \rho$ is equivalent to $\alpha$, and 
we assign the pull-back $\rho^{\ast} \colon L_{\qcoh}(\fU') \to L_{\qcoh}(\fU)$
with the above diagram. 
Roughly speaking an object of $L_{\qcoh}(\fM)$ is a collection of 
objects $\eE_{\fU} \in L_{\qcoh}(\fU)$ 
with equivalences $\rho^{\ast}\eE_{\fU} \sim \eE_{\fU'}$ 
for each diagram (\ref{dia:smooth}) satisfying 
homotopy coherence conditions. 
The homotopy category of $L_{\rm{qcoh}}(\mathfrak{M})$
is denoted by 
$D_{\rm{qcoh}}(\mathfrak{M})$, which is a 
triangulated category. We have the dg and
triangulated subcategories
\begin{align*}
L_{\rm{coh}}(\mathfrak{M}) \subset 
L_{\rm{qcoh}}(\mathfrak{M}), \ 
D^b_{\rm{coh}}(\mathfrak{M}) \subset D_{\rm{qcoh}}(\mathfrak{M})
\end{align*}
consisting of objects which have bounded coherent 
cohomologies. 
We note that there 
is a bounded t-structure on $D_{\rm{coh}}^b(\mathfrak{M})$
whose heart coincides with 
$\Coh(\mM)$. 
In the convention of Subsection~\ref{subsec:notation}, 
we may write $D_{\qcoh}(\fM)$, 
$\Dbc(\fM)$ as
\begin{align*}
	D_{\qcoh}(\fM)=\lim_{\fU \to \fM} D_{\qcoh}(\fU), \ 
	\Dbc(\fM)=\lim_{\fU \to \fM} \Dbc(\fU). 
\end{align*}

A morphism of derived stacks $f \colon \fM \to \fN$ is 
called \textit{quasi-smooth} if 
$\mathbb{L}_f$ is perfect 
such that for any point $x \to \mM$
the restriction $\mathbb{L}_f|_{x}$ is 
of cohomological 
amplitude $[-1, 1]$.
Here 
$\mathbb{L}_f$ is the $f$-relative cotangent complex. 
A derived stack 
$\mathfrak{M}$ over $\mathbb{C}$
is called \textit{quasi-smooth}
if $\fM \to \Spec \mathbb{C}$ is quasi-smooth. 
By~\cite[Theorem~2.8]{MR3352237}, 
the quasi-smoothness of $\mathfrak{M}$ is equivalent to 
that $\fM$ 
is a 1-stack, 
and 
any point of $\mathfrak{M}$ lies 
in the image of a $0$-representable 
smooth morphism 
\begin{align}\label{map:alpha}
	\alpha \colon \mathfrak{U} \to \mathfrak{M}
\end{align}
where $\mathfrak{U}$ is an affine derived scheme 
obtained as a derived zero locus as 
in (\ref{frak:U}). 
In this case, we have 
\begin{align*}
	\Dbc(\fM)=\lim_{\mathfrak{U} \stackrel{\alpha}{\to} \mathfrak{M}} 
	\Dbc(\fU)
\end{align*}
where the limit is taken for 
the $\infty$-category $\iI$ of smooth morphisms (\ref{map:alpha})
from $\fU$ of the form (\ref{frak:U}) with 1-morphisms
given by (\ref{dia:smooth}).
In this paper when we write 
$\lim_{\mathfrak{U} \stackrel{\alpha}{\to} \mathfrak{M}}(-)$
for a quasi-smooth $\fM$, 
the limit is always taken for the $\infty$-category $\iI$
as above. 

\subsubsection{Ind-coherent sheaves on QCA stacks}
\label{subsub:QCA}
Following~\cite[Definition~1.1.8]{MR3037900}, 
a derived stack $\mathfrak{M}$ is called 
\textit{QCA (quasi-compact and with affine automorphism groups)}
if the following conditions hold:
\begin{enumerate}
	\item $\mathfrak{M}$ is quasi-compact;
	\item The automorphism groups of its geometric points are affine;
	\item The classical inertia stack $I_{\mM} \cneq \Delta
	\times_{\mM \times \mM} \Delta$
	is of finite presentation over $\mM$. 
\end{enumerate}
Let $\fM$ be a quasi-smooth
derived stack. 
We denote by $\Ind \Dbc(\fM)$ the category of its ind-coherent 
sheaves (see~\cite{MR3136100})
\begin{align}\label{lim:ind}
	\Ind \Dbc(\fM) \cneq \lim_{\fU \stackrel{\alpha}{\to}\fM}
	\Ind \Dbc(\fU),
\end{align}
where $\alpha$ is a smooth morphism from an affine 
derived scheme $\fU$ of the form (\ref{frak:U}). 
The QCA condition will be useful 
since we have the following theorem:  
\begin{thm}\emph{(\cite[Theorem~3.3.5]{MR3037900})}\label{thm:QCA}
If $\fM$ is QCA, 
then 
$\Ind \Dbc(\fM)$ is compactly generated 
with $(\Ind \Dbc(\fM))^{\rm{cp}}=\Dbc(\fM)$. 
In particular, we have 
\begin{align}\label{ind:M}
	\Ind \Dbc(\fM)=\Ind(\Dbc(\fM)).
\end{align} 
\end{thm}
Let $f \colon \fM_1 \to \fM_2$ be a
morphism between quasi-smooth and QCA derived stacks $\fM_i$. 
Then we have the following pair
of functors (see~\cite[Section~3.6]{MR3037900}, \cite[Section~10.1]{MR3136100})
\begin{align*}
	\xymatrix{
	\Ind \Dbc(\fM_1)
\ar@<0.5ex>[r]^-{f_{\ast}^{\rm{ind}}} &
\ar@<0.5ex>[l]^-{f^{!}}
\Ind \Dbc(\fM_2)
}.
\end{align*}
If $f$ is proper, 
then $f_{\ast}^{\rm{ind}}\dashv f^{!}$
and $f_{\ast}^{\rm{ind}}$
restricts to the functor (see~\cite[Section~4.2]{PoSa})
\begin{align}\label{f:push}
	f_{\ast} \colon \Dbc(\fM_1) \to \Dbc(\fM_2). 
\end{align}
If $f$ is quasi-smooth, 
we also have the pull-back functor
(see~\cite[Section~4.2]{PoSa})
\begin{align}\label{f:back}
	f^{\ast} \colon \Dbc(\fM_2) \to \Dbc(\fM_1). 
\end{align} 
Its ind-completion $f^{\rm{ind}\ast}$
fits into the adjoint pair (see~\cite[Proposition~3.1.6]{MR3701352}, 
\cite[Section~3.7.7]{MR3037900})
\begin{align*}
	\xymatrix{
	\Ind \Dbc(\fM_1)
	\ar@<0.5ex>[r]^-{f_{\ast}^{\rm{ind}}} &
	\ar@<0.5ex>[l]^-{f^{\rm{ind}\ast}}
	\Ind \Dbc(\fM_2)
},
\end{align*}
such that $f^{\rm{ind}\ast}\dashv f_{\ast}^{\rm{ind}}$.

\subsubsection{$(-1)$-shifted cotangent derived stack}\label{subsec:shifted:cotangent}
Let $\mathfrak{M}$ be a quasi-smooth derived stack. 
We denote by 
$\Omega_{\mathfrak{M}}[-1]$ the $(-1)$-shifted
cotangent derived stack of $\mathfrak{M}$
\begin{align*}
p \colon 
\Omega_{\mathfrak{M}}[-1] \cneq 
\Spec_{\mathfrak{M}}
S(\mathbb{T}_{\mathfrak{M}}[1]) \to \mathfrak{M}. 
\end{align*}
Here $\mathbb{T}_{\mathfrak{M}} \in D^b_{\rm{coh}}(\mathfrak{M})$
is the tangent complex of $\mathfrak{M}$, which 
is dual to the cotangent complex 
$\mathbb{L}_{\mathfrak{M}}$ of $\mathfrak{M}$. 
The derived stack $\Omega_{\mathfrak{M}}[-1]$ 
admits a natural $(-1)$-shifted symplectic 
structure~\cite{MR3090262, Calaque}, 
which induces the 
d-critical structure~\cite{MR3399099}
on its 
classical truncation $\nN$ 
\begin{align}\label{map:p0}
p_0 \colon \nN \cneq t_0(\Omega_{\mathfrak{M}}[-1])
\to \mM.
\end{align}
The stack $\nN$ is 
also  
described in terms of the dual obstruction cone
introduced in~\cite{MR3607000}.
Note that the quasi-smoothness of $\mathfrak{M}$
implies that the
natural map of cotangent complexes
\begin{align}\label{perf:obs}
\eE^{\bullet} \cneq \mathbb{L}_{\mathfrak{M}}|_{\mM} \to 
\tau_{\ge -1} \mathbb{L}_{\mM}
\end{align} 
is a perfect obstruction theory on $\mM$~\cite{BF}. 
Then we have 
\begin{align*}
\nN=
\mathrm{Obs}^{\ast}(\eE^{\bullet}) \cneq 
\Spec_{\mM}S(\hH^1(\eE^{\bullet \vee})). 
\end{align*}
The cone $\mathrm{Obs}^{\ast}(\eE^{\bullet})$
over $\mM$ is called 
the dual obstruction cone associated with
the perfect obstruction theory (\ref{perf:obs}). 

Let $\fM_1$, $\fM_2$ be
quasi-smooth derived stacks 
with truncations $\mM_i=t_0(\fM_i)$. 
Let $f \colon \fM_1 \to \fM_2$ be a morphism. 
Then 
the morphism 
$f^{\ast}\mathbb{L}_{\fM_2} \to \mathbb{L}_{\fM_1}$
induces the diagram
\begin{align}\label{diagram:induced}
	\xymatrix{
		t_0(\Omega_{\fM_1}[-1]) \ar[d] & 
		t_0(\Omega_{\fM_2}[-1]\times_{\fM_2}\fM_1)
		 \ar[d] \ar[l]_-{f^{\diamondsuit}} 
		\ar[r]^-{f^{\spadesuit}} 
		\ar@{}[rd]|\square
		& 
		t_0(\Omega_{\fM_2}[-1]) \ar[d] \\
		\mM_1 & \mM_1 \ar@{=}[l] \ar[r]_-{f} & \mM_2. 
	}
\end{align}
\begin{lem}\label{lem:beta0}
The morphism $f$ is quasi-smooth if and only
if $f^{\diamondsuit}$ is a closed immersion, 
$f$ is smooth if and only if $f^{\diamondsuit}$ is an isomorphism. 
\end{lem}
\begin{proof}
	The lemma follows from the exact sequence
	\begin{align*}
		\cdots \to \hH^1(\mathbb{T}_f|_{\mM_1}) \to 
		\hH^1(\mathbb{T}_{\fM_1}|_{\mM_1}) \to 
		\hH^1(f^{\ast}\mathbb{T}_{\fM_2}|_{\mM_1}) \to 
		\hH^2(\mathbb{T}_f|_{\mM_1}) \to \cdots. 
		\end{align*}
	\end{proof}

\subsubsection{Good moduli spaces of Artin stacks}\label{subsubsec:gmoduli}
In general for a classical Artin stack $\mM$, 
its \textit{good moduli space}
is an algebraic space $M$ 
together with a quasi-compact morphism, 
\begin{align*}
	\pi_{\mM} \colon \mM \to M
\end{align*} 
satisfying the following conditions
(cf.~\cite[Section~1.2]{MR3237451}):
\begin{enumerate}
	\item The push-forward
	$\pi_{\mM\ast} \colon \mathrm{QCoh}(\mM) \to \mathrm{QCoh}(M)$
	is exact. 
	\item The induced morphism $\oO_M \to \pi_{\mM\ast}\oO_{\mM}$ is an 
	isomorphism. 
\end{enumerate}
The good moduli space morphism $\pi_{\mM}$ is universally closed.
Moreover 
for each closed point $y \in M$, 
there exists a unique closed point $x \in \pi_{\mM}^{-1}(y)$, 
and its automorphism group $\Aut(x)$ is reductive 
(see~\cite[Theorem~4.16, Proposition~12.14]{MR3237451}). 

Let $\fM$ be a quasi-smooth derived stack and take 
its classical truncation 
$\mM=t_0(\fM)$. 
Suppose that it admits a good moduli 
space $\pi_{\mM} \colon \mM \to M$. 
We denote by $\dD_{{\rm{\acute{e}t}}/M}$ the category of 
\'etale morphisms $\iota \colon U \to M$
for an algebraic space $U$, with morphisms 
given by 
commutative diagrams
\begin{align}\label{mor:Det}
	\xymatrix{
U' \ar[rr]^-{\rho} \ar[rd]_-{\iota'}  & & U \ar[ld]^-{\iota} \\ 	
& M. &
}
	\end{align}
Here $\rho$ is an \'{e}tale morphism.
For each object $(\iota \colon U \to M) \in \dD_{{\rm{\acute{e}t}}/M}$
there is a unique (up to equivalence) derived stack 
$\fM_U$ together with Cartesian diagrams 
		\begin{align}\label{Cartesian:U}
		\xymatrix{
			\fM_U \ar[d]_-{\iota_{\fM}} \diasquare
			 &\mM_U \ar@<0.3ex>@{_{(}->}[l] \ar[r] \ar[d]_-{\iota_{\mM}}\ar@{}[rd]|\square
			& U \ar[d]^-{\iota} \\
			\fM &\ar@<0.3ex>@{_{(}->}[l] \mM \ar[r]^-{\pi_{\mM}} & M
		}
	\end{align}
such that $\mM_U=t_0(\fM_U)$ and $\iota_{\fM}$ is \'{e}tale. 
The above diagram exists since
the category of \'{e}tale morphisms 
with target $\mM$ is equivalent to that with target $\fM$.
Moreover given a diagram (\ref{mor:Det}),
there exists a unique morphism 
$\rho_{\fM} \colon \fM_{U'} \to \fM_U$ 
and Cartesian squares
\begin{align}\label{dia:rho}
		\xymatrix{
		\fM_{U'} \ar[d]_-{\rho_{\fM}} \diasquare
		&\mM_{U'} \ar@<0.3ex>@{_{(}->}[l] \ar[r] \ar[d]_-{\rho_{\mM}}\ar@{}[rd]|\square
		& U' \ar[d]^-{\rho} \\
		\fM_U &\ar@<0.3ex>@{_{(}->}[l] \mM_U \ar[r] & U. 
	}
\end{align}
We will use the following \'{e}tale local structure result for 
good moduli spaces proved in~\cite{AHR}. 
\begin{thm}\label{thm:AHR}\emph{(\cite[Theorem~2.9]{AHR})}
	For any closed point $y\in M$,
	there exists an \'{e}tale neighborhood
	$\iota \colon (U, 0) \to (M, y)$
	and a commutative isomorphisms
	\begin{align*}
		\xymatrix{
			[\uU/G] \ar[r]^-{\cong} \ar[d] &\mM_U  \ar[d] \\
			\uU\ssslash G \ar[r]^-{\cong} & U.
		}
	\end{align*}
	Here
	$G=\Aut(x)$ 
	for a unique closed point $x \in \pi_{\mM}^{-1}(y)$, 
	$\uU=\Spec R$ is an affine $\mathbb{C}$-scheme with $G$-action, 
	and 
	$\uU\ssslash G=\Spec R^G$.   
\end{thm}

The above result can be extended to derived stacks, whose proof will 
be given in Subsection~\ref{subsub:extend} (cf.~\cite[Lemma~3.2.5]{HalpK32}). 
\begin{prop}\label{prop:extend}
	In the situation of Theorem~\ref{thm:AHR}, 
	there is an equivalence 
	\begin{align*}
		\fM_U \sim [\fU/G],
		\end{align*}
where  
	$[\fU/G]$ is a derived stack 
	associated with a $G$-equivariant 
	tuple $(Y, V, s)$ in Definition~\ref{def:tuple}, such that
	$\uU=t_0(\fU)$ and $G$ are given as in Theorem~\ref{thm:AHR}. 
		Moreover for $l\in \Pic(\mM)_{\mathbb{R}}$, 
	by replacing $U$ with a Zariski open neighborhood of 
	$0\in U$, the pull-back $\iota_{\fM}^{\ast}(l) \in \Pic([\uU/G])_{\mathbb{R}}$
	is extended to a $\mathbb{R}$-line bundle on $[Y/G]$. 
\end{prop}

\subsection{Categorical DT theory via Verdier quotients}\label{subsec:catVer}
\subsubsection{Definition of DT category}\label{subsub:defDT}
Let $\fM$ be a quasi-smooth and QCA derived stack, and 
take a  conical closed substack
\begin{align*}
\zZ \subset \nN=t_0(\Omega_{\mathfrak{M}}[-1]). 
\end{align*}
Here similarly to before, $\zZ$ is called conical 
if it is closed under the fiberwise $\mathbb{C}^{\ast}$-action on $\nN \to \mM$. 
Let $\alpha \colon \fU \to \fM$ be a smooth morphism as in (\ref{map:alpha}). 
By Lemma~\ref{lem:beta0}, we have the associated conical closed subscheme
\begin{align}\notag
\alpha^{\ast}\zZ \cneq 
\alpha^{\diamondsuit} (\alpha^{\spadesuit})^{-1}(\zZ) \subset t_0(\Omega_{\mathfrak{U}}[-1])
=\mathrm{Crit}(w).
\end{align}
Here we have written $\mathfrak{U}$ as (\ref{frak:U}), and 
$w$ is defined in the digram (\ref{diagram:KZ}). 
By Definition~\ref{defi:ssuport2}, we have the
associated subcategory 
$\cC_{\alpha^{\ast}\zZ} \subset D^b_{\rm{coh}}(\mathfrak{U})$. 
It is proved in~\cite[Section~7]{MR3300415}
that the subcategory $\cC_Z$ satisfies 
functorial properties with respect to smooth
morphisms. 
So given a diagram (\ref{dia:smooth}),
the pull-back functors restrict to functors
\begin{align}\label{pull:C}
	\rho^{\ast} \colon \cC_{{\alpha'}^{\ast}\zZ} \to 
	\cC_{\alpha^{\ast}\zZ}, \ 
	\rho^{\rm{ind}\ast}  \colon \Ind \cC_{\alpha^{\ast}\zZ} \to 
	\Ind \cC_{{\alpha'}^{\ast}\zZ}. 
\end{align}
By taking the limit, we have the following definition:

\begin{defi}\label{def:CZ2}\emph{(\cite{MR3300415})}
For a conical closed substack $\zZ \subset \nN$, 
we define 
\begin{align}\notag
\cC_{\zZ} \cneq 
\lim_{\fU \stackrel{\alpha}{\to} \fM}
\cC_{\alpha^{\ast}\zZ}\subset D^b_{\rm{coh}}(\mathfrak{M}), \ 
\Ind \cC_{\zZ} \cneq \lim_{\fU \stackrel{\alpha}{\to} \fM}
\Ind \cC_{\alpha^{\ast}\zZ}\subset \Ind D^b_{\rm{coh}}(\mathfrak{M}). 
\end{align}
\end{defi}
By~\cite[Corollary~8.2.6]{MR3300415}, we have adjoint pairs 
\begin{align}\label{singular:adjoint}
	\xymatrix{
		\Ind \cC_{\zZ}
		\ar@<0.5ex>[r]^-{i} &
		\ar@<0.5ex>[l]^-{\Gamma_{\zZ}}
		\Ind \Dbc(\fM_2)
	},
\end{align}
where $i$ is the natural inclusion and $i \dashv \Gamma_{\zZ}$. 

We denote by $\nN^{\rm{ss}}$ the complement of 
$\zZ$
\begin{align*}
\nN^{\rm{ss}} \cneq \nN \setminus \zZ
\end{align*}
which is a $\mathbb{C}^{\ast}$-invariant 
 open substack of $\nN$. 
We define the $\mathbb{C}^{\ast}$-equivariant 
DT category
for $\nN^{\rm{ss}}$ as follows: 

\begin{defi}\label{defi:catDT}
We define the triangulated category 
$\dD \tT^{\mathbb{C}^{\ast}}(\nN^{\rm{ss}})$
by the Verdier quotient
\begin{align}\label{quot:D}
\dD \tT^{\mathbb{C}^{\ast}}(\nN^{\rm{ss}}) \cneq 
D^b_{\rm{coh}}(\mathfrak{M})/\cC_{\zZ}.
\end{align}
\end{defi}

The quotient category (\ref{quot:D}) admits a dg-enhancement
by taking dg-quotients of dg-categories 
developed in~\cite{MR1667558, MR2028075}. 
In general for a dg-category $\dD$ and 
its dg-subcategory $\dD' \subset \dD$, 
its Drinfeld dg-quotient 
$\dD/\dD'$ 
is obtained from $\dD$ by formally adding 
degree $-1$ morphisms (see~\cite{MR2028075})
\begin{align*}
\epsilon_U \colon U \to U, \ d(\epsilon_U)=\id_U
\end{align*}
for each $U \in \dD'$. 
Let 
\begin{align*}
L\cC_{\zZ} \subset L_{\rm{coh}}(\mathfrak{M})
\end{align*}
be the dg subcategory consisting of objects 
which are isomorphic to objects in $\cC_{\zZ}$ in 
the homotopy category. 
\begin{defi}\label{def:catDT:dg}
We define
the dg-category $\mathcal{DT}^{\mathbb{C}^{\ast}}_{\rm{dg}}(\nN^{\rm{ss}})$
to be the Drinfeld quotient
\begin{align*}
\dD\tT^{\mathbb{C}^{\ast}}_{\rm{dg}}(\nN^{\rm{ss}}) 
 \cneq 
L_{\rm{coh}}(\mathfrak{M})/L\cC_{\zZ}.
\end{align*}
\end{defi}
The dg-category 
$\mathcal{DT}^{\mathbb{C}^{\ast}}_{\rm{dg}}(\nN^{\rm{ss}})$
is a dg-enhancement of $\mathcal{DT}^{\mathbb{C}^{\ast}}(\nN^{\rm{ss}})$, i.e. 
we have the canonical 
equivalence (see~\cite[Theorem~3.4]{MR2028075})
\begin{align*}\dD\tT^{\mathbb{C}^{\ast}}(\nN^{\rm{ss}}) 
\stackrel{\sim}{\to}
\mathrm{Ho}(\dD\tT^{\mathbb{C}^{\ast}}_{\rm{dg}}(\nN^{\rm{ss}})). 
\end{align*}

Let $\omega_{\fM}\cneq \det(\mathbb{L}_{\fM})[\rank(\mathbb{L}_{\fM})]$ be 
the dualizing line bundle of $\fM$. 
We denote by $\mathbb{D}_{\fM}^{\rm{Serre}}$ the Serre dualizing functor
\begin{align}\label{Serre}
	\mathbb{D}^{\rm{Serre}}_{\fM} \colon 
	\Dbc(\fM)\stackrel{\sim}{\to}
	\Dbc(\fM)^{\rm{op}}, \ 
	E \mapsto \hH om(E, \omega_{\fM}). 
\end{align}
Here the above functor preserves coherence as 
$\fM$ is quasi-smooth, and 
also preserves the singular supports
by~\cite[Proposition~4.7.2]{MR3300415}. 
Therefore we have the Serre duality equivalence
\begin{align*}
	\mathbb{D}^{\rm{Serre}}_{\nN} \colon 
	\dDT^{\C}(\nN^{\rm{ss}}) \stackrel{\sim}{\to}
	\dDT^{\C}(\nN^{\rm{ss}})^{\rm{op}}. 
\end{align*}
For another quasi-smooth derived stack $\fM'$, 
a conical closed subset $\zZ'$ and its complement ${\nN'}^{\rm{ss}}$
\begin{align*}
	\zZ' \subset t_0(\Omega_{\fM'}[-1]), \ 
	{\nN'}^{\rm{ss}}=t_0(\Omega_{\fM'}[-1]) \setminus \zZ'
\end{align*}
we define the tensor product of DT categories as 
\begin{align*}
	\dDT^{\C}(\nN^{\rm{ss}}) \otimes \dDT^{\C}({\nN'}^{\rm{ss}})
	\cneq \Dbc(\fM \times \fM')/\cC_{q^{\ast}\zZ \cup {q'}^{\ast}\zZ'}
\end{align*}
where $q, q'$ are projections from 
$\Omega_{\fM}[-1] \times \Omega_{\fM'}[-1]$ onto corresponding 
factors. 

\subsubsection{$\wdDT$-version}\label{subsub:DThat}
We can also define other 
kind of DT categories, defined by taking limits 
of quotient categories. 
For a diagram (\ref{dia:smooth}), 
since  
the pull-back $\rho^{\ast} \colon \Dbc(\fU) \to \Dbc(\fU')$
restrict to the functor (\ref{pull:C}), 
we have the induced functors
\begin{align*}
	\rho^{\ast} \colon 
	L_{\coh}(\fU)/L\cC_{\alpha^{\ast}\zZ}
	\to L_{\coh}(\fU')/L\cC_{{\alpha'}^{\ast}\zZ'}, \ 
	\rho^{\ast} \colon \Dbc(\fU)/\cC_{\alpha^{\ast}\zZ}
	\to \Dbc(\fU')/\cC_{{\alpha'}^{\ast}\zZ'}.
\end{align*}
By taking their limits, 
we have the following definition: 
\begin{defi}\label{def:DThat}
We define the dg-category 
$\wdDT^{\C}_{\rm{dg}}(\nN^{\rm{ss}})$
and the 
triangulated category 
$\wdDT^{\C}(\nN^{\rm{ss}})$ to be 
\begin{align}\label{def:qlimit}
	\wdDT^{\C}_{\rm{dg}}(\nN^{\rm{ss}}) \cneq 
	\lim_{\fU \stackrel{\alpha}{\to} \fM}
	\left(	L_{\coh}(\fU)/L\cC_{\alpha^{\ast}\zZ} \right), \
	\wdDT^{\C}(\nN^{\rm{ss}}) \cneq
	\lim_{\fU \stackrel{\alpha}{\to} \fM}
	\left(\Dbc(\fU)/\cC_{\alpha^{\ast}\zZ}\right). 
\end{align}
\end{defi}

\begin{rmk}
	By the equivalence (\ref{equiv:MF}), 
	the limit (\ref{def:qlimit}) may be regarded as 
	a gluing of derived categories of factorizations
	\begin{align*}
	\wdDT^{\C}(\nN^{\rm{ss}})=	\lim_{\fU \stackrel{\alpha}{\to} \fM}
		\mathrm{MF}_{\coh}^{\C}(V^{\vee} \setminus \alpha^{\ast}\zZ, w). 
	\end{align*}
\end{rmk}
\begin{rmk}\label{rmk:nfunctor}
There is a natural functor
\begin{align*}
	\dDT^{\C}(\nN^{\rm{ss}}) \to \wdDT^{\C}(\nN^{\rm{ss}}),
\end{align*}
but it is not clear whether this is fully-faithful nor 
essentially surjective in general. 
In Section~\ref{sec:compare} we will 
address the above question in the case that 
$\Ind \cC_{\zZ}$ is compactly generated. 
\end{rmk}

For a conical closed substack $\zZ \subset \nN=t_0(\Omega_{\fM}[-1])$
and its complement $\nN^{\rm{ss}}=\nN \setminus \zZ$, 
we define the ind-completion of the DT category 
$\dDT^{\C}(\nN^{\rm{ss}})$ by 
\begin{align*}
	\Ind \dDT^{\C}(\nN^{\rm{ss}}) \cneq \Ind(\dDT^{\C}(\nN^{\rm{ss}})).  
\end{align*}
On the other hand, a similar construction in Definition~\ref{def:DThat} applies to define limits 
\begin{align*}
	\lim_{\fU \stackrel{\alpha}{\to} \fM}
	\Ind \left( \Dbc(\fU)/\cC_{\alpha^{\ast}\zZ} \right), \ 
	\lim_{\fU \stackrel{\alpha}{\to} \fM}
	\left(\Ind \Dbc(\fU)/\Ind \cC_{\alpha^{\ast}\zZ}\right). 
\end{align*}
The following proposition will be proved in Subsection~\ref{subsec:eqDT}:
\begin{prop}\label{lem:indDT}
	Suppose that $\Ind \cC_{\zZ}$ is compactly generated 
	with $(\Ind \cC_{\zZ})^{\rm{cp}}=\cC_{\zZ}$. 
	Then we have equivalences 
	\begin{align*}
		\Ind \dDT^{\C}(\nN^{\rm{ss}}) \stackrel{\sim}{\to}
		\Ind \Dbc(\fM)/\Ind \cC_{\zZ} \stackrel{\sim}{\to}
		\lim_{\fU \stackrel{\alpha}{\to} \fM}
		\Ind \left( \Dbc(\fU)/\cC_{\alpha^{\ast}\zZ} \right)
		\stackrel{\sim}{\to} 
		\lim_{\fU \stackrel{\alpha}{\to} \fM}
		\left(\Ind \Dbc(\fU)/\Ind \cC_{\alpha^{\ast}\zZ}\right). 
	\end{align*}
	Moreover we have fully-faithful functors with dense images
	\begin{align*}
		\dDT^{\C}(\nN^{\rm{ss}}) \hookrightarrow 
		\wdDT^{\C}(\nN^{\rm{ss}}) \hookrightarrow \left( \Ind \dDT^{\C}(\nN^{\rm{ss}})\right)^{\rm{cp}}. 
	\end{align*}
\end{prop}
Using the above proposition, we also have the following lemma: 
\begin{lem}\label{lem:orth:N}
	In the situation of Proposition~\ref{lem:indDT}, let 
	$\zZ_i \subset t_0(\Omega_{\fM}[-1])$ for $i=1, 2$ be 
	conical closed substacks such that $\zZ_1 \cap \zZ_2=\zZ$. 
	Then for $\eE_i \in \cC_{\zZ_i}$, we have 
	\begin{align*}
		\Hom_{\dDT^{\C}(\nN^{\rm{ss}})}(\eE_1, \eE_2)=0. 
	\end{align*}
\end{lem}
\begin{proof}
	By replacing $\eE_1$ if necessary, 
	any morphism $\eE_1 \to \eE_2$ in 
	$\dDT^{\C}(\nN^{\rm{ss}})$ is represented 
	by a morphism $\phi \colon \eE_1 \to \eE_2$ in 
	$\Dbc(\fM)$. 
	The morphism $\phi$ 
	factors through $\eE_1 \to \Gamma_{\zZ_1}(\eE_2) \to \eE_2$,
	where $\Gamma_{\zZ_1}$ is the right adjoint of the 
	natural fully-faithful functor 
	$\Ind \cC_{\zZ_1} \hookrightarrow \Ind \Dbc(\fM)$
	in (\ref{singular:adjoint}).  
	As $\Gamma_{\zZ_{1}}(\eE_2)$ have singular
	support contained in 
	$\zZ_1 \cap \zZ_2=\zZ$ (which is obvious from the construction 
	of the right adjoint in~\cite[Lemma~3.3.7]{MR3300415}), it follows that 
	$\phi=0$ in $\Ind \Dbc(\fM)/\Ind \cC_{\zZ}$. 
	Then the lemma follows from the first equivalence in Proposition~\ref{lem:indDT}. 
\end{proof}

\subsubsection{Replacement of the quotient category}\label{subsub:replace}
The derived moduli stacks we will consider 
are typically not quasi-compact, 
and it will be useful if we can 
replace it by its quasi-compact open substack. 
In order to 
show that such a replacement makes sense, 
we need 
to check that the resulting quotient 
categories are independent of the choice of
a quasi-compact open substack. 

Let $\wW \subset \mM$ be a closed 
substack, and take the open 
derived substack
\begin{align*}
	\mathfrak{M}_{\circ} \subset \mathfrak{M}
\end{align*}
whose truncation is $\mM \setminus \wW$. 
Let us take a conical closed substack $\zZ \subset \nN =  t_0(\Omega_{\mathfrak{M}}[-1])$. 
We have the following conical closed substack
\begin{align*}
	\zZ_{\circ} \cneq \zZ \setminus p_0^{-1}(\wW)
	\subset \nN_{\circ} \cneq t_0(\Omega_{\mathfrak{M}_{\circ}}[-1]).
\end{align*}
Note that $\nN_{\circ}=\nN \setminus p_0^{-1}(\wW)$, and 
we have the following open immersion
\begin{align}\label{open:N}
	\nN_{\circ}^{\rm{ss}} \cneq 
	\nN_{\circ} \setminus \zZ_{\circ} \hookrightarrow
	\nN \setminus \zZ =\nN^{\rm{ss}}. 
\end{align}
The following lemma will be proved in Lemma~\ref{lem:dstack:support}
and Lemma~\ref{lem:replace}. 
\begin{lem}\label{lem:replace0}
	Suppose that $p_0^{-1}(\wW) \subset \zZ$.
	Then 
	the restriction functors 
	give equivalences
	\begin{align*}
		\mathcal{DT}^{\mathbb{C}^{\ast}}
		(\nN^{\rm{ss}}) \stackrel{\sim}{\to}
		\mathcal{DT}^{\mathbb{C}^{\ast}}
		(\nN_{\circ}^{\rm{ss}}), \ 
		\wdDT^{\C}
		(\nN^{\rm{ss}}) \stackrel{\sim}{\to}
		\wdDT^{\C}
		(\nN_{\circ}^{\rm{ss}}).
	\end{align*}
Moreover if $p_0^{-1}(\wW)=\zZ$, 
we have equivalences
\begin{align}\notag
	\mathcal{DT}^{\mathbb{C}^{\ast}}(\nN^{\rm{ss}})
	\stackrel{\sim}{\to} \wdDT^{\C}(\nN^{\rm{ss}}) \stackrel{\sim}{\to}
	D^b_{\rm{coh}}(\mathfrak{M}_{\circ}). 
\end{align}
	\end{lem}

\subsubsection{$\mathbb{C}^{\ast}$-rigidification}\label{subsec:rig}
In some case  
our derived stack $\mathfrak{M}$
has $\mathbb{C}^{\ast}$-automorphisms at any point, 
and 
in this case $\mM=t_0(\mathfrak{M})$ is never a Deligne-Mumford 
stack. In this case, 
we rigidify $\mathbb{C}^{\ast}$-automorphisms
and define  
the DT category
via the $\mathbb{C}^{\ast}$-rigidified derived stack. 
 
Suppose that we have a closed 
embedding 
into the inertia stack
\begin{align}\label{inertia}
(\mathbb{C}^{\ast})_{\mM} \hookrightarrow I_{\mM}\cneq \Delta \times_{\mM \times \mM} \Delta.
\end{align}
Here $\Delta$ is the diagonal. 
Under the presence of the 
embedding (\ref{inertia}), 
by~\cite[Theorem~A.1]{MR2427954} there exists an Artin 
stack
$\mM^{\mathbb{C}^{\ast}\mathchar`-\mathrm{rig}}$
together with a map 
\begin{align}\label{M:gerb}
\varrho_{\mM} \colon 
\mM \to 
\mM^{\mathbb{C}^{\ast}\mathchar`-\mathrm{rig}}
\end{align}
uniquely characterized by the properties that 
$\varrho_{\mM}$ is a $\mathbb{C}^{\ast}$-gerbe
(see~\cite{MR2309155} for the notion of gerbes)
and for any map $\xi \colon T \to \mM$ from a $\mathbb{C}$-scheme $T$, 
the homomorphism of group schemes
\begin{align*}
\Aut_T(\xi) \to \Aut_T(\varrho_{\mM} \circ \xi)
\end{align*}
is surjective with kernel 
$(\mathbb{C}^{\ast})_T$. In the above case, 
any object $\fF \in \Coh(\mM)$ admits 
a $\mathbb{C}^{\ast}$-action induced by the 
inertia action (see~\cite[Section~2.1]{MR2309155}).

The above $\C$-rigidification can be extended to derived 
stacks as follows (see~\cite[Lemma~3.4.9]{HalpK32}). 
Suppose that there is a weak action of $B\C$ on $\fM$ such that the 
induced map on the inertia stack
$(\C)_{\fM} \hookrightarrow I_{\fM}$
is a closed immersion. 
Then we have 
the $\mathbb{C}^{\ast}$-gerbe
\begin{align}\label{derived:gerbe}
\varrho_{\mathfrak{M}} \colon 
\mathfrak{M} \to \mathfrak{M}^{\mathbb{C}^{\ast}\mathchar`-\rm{rig}}
\end{align}
whose classical truncation is identified with (\ref{M:gerb}). 
 
Let $\Lambda$ be the character group 
\begin{align*}
\Lambda=\Hom_{\mathbb{Z}}(\mathbb{C}^{\ast}, \mathbb{C}^{\ast})
\cong \mathbb{Z}.
\end{align*}
For $\lambda \in \Lambda$, we define
the subcategories 
\begin{align*}
\Coh(\mM)_{\lambda} \subset \Coh(\mM), \ 
D_{\rm{coh}}^{b}(\mathfrak{M})_{\lambda}
\subset D_{\rm{coh}}^{b}(\mathfrak{M})
\end{align*}
to be consisting of sheaves with 
$\mathbb{C}^{\ast}$-weight $\lambda$, 
objects of cohomology sheaves 
with 
$\mathbb{C}^{\ast}$-weight $\lambda$, respectively. 
Then 
we have the decompositions 
\begin{align}\label{decompose:chi}
\Coh(\mM)=\bigoplus_{\lambda \in \Lambda}\Coh(\mM)_{\lambda}, \ 
D_{\rm{coh}}^{b}(\mathfrak{M})= 
\bigoplus_{\lambda \in \Lambda}D_{\rm{coh}}^{b}(\mathfrak{M})_{\lambda}. 
\end{align}
The above decompositions are obvious if 
$\varrho_{\mathfrak{M}}$ is a trivial $\mathbb{C}^{\ast}$-gerbe. 
In general, a generalization 
of the arguments in~\cite[Theorem~4.7, Theorem~5.4]{DBOMS} 
easily implies the above decompositions. 
The above decompositions lift to 
dg-enhancements. 
Let $L_{\rm{coh}}(\mathfrak{M})_{\lambda} \subset L_{\rm{coh}}(\mathfrak{M})$
be the dg-subcategory consisting of objects of 
cohomology sheaves with $\mathbb{C}^{\ast}$ weight $\lambda$. 
We have the weak equivalence
\begin{align*}
\bigoplus_{\lambda \in \Lambda} L_{\rm{coh}}(\mathfrak{M})_{\lambda}
\stackrel{\sim}{\to}
L_{\rm{coh}}(\mathfrak{M}). 
\end{align*}

Suppose that the $\mathbb{C}^{\ast}$-weight 
on $\mathbb{L}_{\mathfrak{M}}|_{\mM}$ is zero. 
Then  
the stack 
$\nN=t_0(\Omega_{\mathfrak{M}}[-1])$
also satisfies 
$(\mathbb{C}^{\ast})_{\nN} \subset I_{\nN}$.
Therefore we have its $\mathbb{C}^{\ast}$-rigidification
$\varrho \colon \nN \to \nN^{\mathbb{C}^{\ast}\mathchar`-\mathrm{rig}}$. 
Indeed we have
\begin{align*}
\nN^{\mathbb{C}^{\ast}\mathchar`-\mathrm{rig}} =
t_0(\Omega_{\mathfrak{M}^{\mathbb{C}^{\ast}\mathchar`-\mathrm{rig}} }[-1])
\end{align*}
 and we have 
the Cartesian square
\begin{align}\notag
\xymatrix{
\nN \ar[r]^-{\varrho_{\nN}} \ar[d]\ar@{}[dr]|\square & 
\nN^{\mathbb{C}^{\ast}\mathchar`-\mathrm{rig}} \ar[d] \\
\mM \ar[r]^-{\varrho_{\mM}} & \mM^{\mathbb{C}^{\ast}\mathchar`-\rm{rig}}. 
}
\end{align}
Let us take a conical closed substack
and its complement
\begin{align*}
\zZ^{\mathbb{C}^{\ast}\mathchar`-\mathrm{rig}}
\subset \nN^{\mathbb{C}^{\ast}\mathchar`-\mathrm{rig}}, \ 
(\nN^{\rm{ss}})^{\mathbb{C}^{\ast}\rig}=\nN^{\mathbb{C}^{\ast}\rig} \setminus 
\zZ^{\mathbb{C}^{\ast}\mathchar`-\mathrm{rig}}
\end{align*}
and also set $\zZ=\varrho_{\nN}^{-1}(\zZ^{\mathbb{C}^{\ast}\mathchar`-\mathrm{rig}})$, 
$\nN^{\rm{ss}}=\nN \setminus \zZ$. 
We define the $\lambda$-twisted version of
DT categories as follows: 
\begin{defi}
We 
define $\mathcal{DT}^{\mathbb{C}^{\ast}}(\nN^{\rm{ss}})_{\lambda}$, 
$\mathcal{DT}_{\rm{dg}}^{\mathbb{C}^{\ast}}(\nN^{\rm{ss}})_{\lambda}$
by the Verdier/Drinfeld quotients 
\begin{align*}
&\mathcal{DT}^{\mathbb{C}^{\ast}}(\nN^{\rm{ss}})_{\lambda} \cneq 
D^b_{\rm{coh}}(\mathfrak{M})_{\lambda}/(D^b_{\rm{coh}}(\mathfrak{M})_{\lambda} \cap 
\cC_{\zZ}), \\
&\mathcal{DT}_{\rm{dg}}^{\mathbb{C}^{\ast}}(\nN^{\rm{ss}})_{\lambda}
\cneq L_{\rm{coh}}(\mathfrak{M})_{\lambda}/
(L_{\rm{coh}}(\mathfrak{M})_{\lambda}
\cap L\cC_{\zZ}). 
\end{align*}
\end{defi}
Note that $\mathcal{DT}_{\rm{dg}}^{\mathbb{C}^{\ast}}(\nN^{\rm{ss}})_{\lambda}$ is a 
dg-enhancement of $\mathcal{DT}^{\mathbb{C}^{\ast}}(\nN^{\rm{ss}})_{\lambda}$. 
We have the following proposition:
\begin{prop}\label{prop:rigid}
(i) 
We have the equivalence (weak equivalence)
\begin{align}\label{decompose:chi2}
\bigoplus_{\lambda \in \Lambda}\mathcal{DT}^{\mathbb{C}^{\ast}}(\nN^{\rm{ss}})
_{\lambda} \stackrel{\sim}{\to}
\mathcal{DT}^{\mathbb{C}^{\ast}}(\nN^{\rm{ss}}), \ 
\bigoplus_{\lambda \in \Lambda}\mathcal{DT}_{\rm{dg}}^{\mathbb{C}^{\ast}}(\nN^{\rm{ss}})
_{\lambda} \stackrel{\sim}{\to}
\mathcal{DT}_{\rm{dg}}^{\mathbb{C}^{\ast}}(\nN^{\rm{ss}}). 
\end{align}

(ii) We have the equivalence (weak equivalence)
\begin{align}\label{equiv:rigidify}
\varrho_{\mathfrak{M}}^{\ast} \colon 
\mathcal{DT}^{\mathbb{C}^{\ast}}((\nN^{\rm{ss}})^{\mathbb{C}^{\ast}\rig})
\stackrel{\sim}{\to}\mathcal{DT}^{\mathbb{C}^{\ast}}(\nN^{\rm{ss}})_{\lambda=0}, \ 
\varrho_{\mathfrak{M}}^{\ast} \colon 
\mathcal{DT}_{\rm{dg}}^{\mathbb{C}^{\ast}}((\nN^{\rm{ss}})^{\mathbb{C}^{\ast}\rig})
\stackrel{\sim}{\to}\mathcal{DT}_{\rm{dg}}^{\mathbb{C}^{\ast}}(\nN^{\rm{ss}})_{\lambda=0}. 
\end{align}

(iii) 
Suppose that (\ref{derived:gerbe}) is a trivial 
$\mathbb{C}^{\ast}$-gerbe, i.e. 
$\mathfrak{M}=\mathfrak{M}^{\mathbb{C}^{\ast}\rig} \times B\mathbb{C}^{\ast}$. 
Then we have an equivalence (weak equivalence)
\begin{align}\label{trivial:gerbe}
\mathcal{DT}^{\mathbb{C}^{\ast}}(\nN^{\rm{ss}})_{\lambda=0}
\stackrel{\sim}{\to}
\mathcal{DT}^{\mathbb{C}^{\ast}}(\nN^{\rm{ss}})_{\lambda}, \
\mathcal{DT}_{\rm{dg}}^{\mathbb{C}^{\ast}}(\nN^{\rm{ss}})_{\lambda=0}
\stackrel{\sim}{\to}
\mathcal{DT}_{\rm{dg}}^{\mathbb{C}^{\ast}}(\nN^{\rm{ss}})_{\lambda}. 
\end{align}
\end{prop}
\begin{proof}
In all of (i), (ii), (iii), it is enough to show the equivalences for 
triangulated DT categories.  

(i) For $\fF \in D^b_{\rm{coh}}(\mathfrak{M})$, 
we take the decomposition $\fF=\oplus_{\lambda}\fF_{\lambda}$
where $\fF_{\lambda}$ is $\mathbb{C}^{\ast}$-weight $\lambda$-part. 
Then
$\Hom^{\ast}(\fF, \fF)$ is the direct sum of 
$\Hom^{\ast}(\fF_{\lambda}, \fF_{\lambda})$, 
so by the definition of singular supports 
$\fF$ is an object in $\cC_{\zZ}$ if and only 
if each $\fF_{\lambda}$ is an object in 
$\cC_{\zZ}$. 
Therefore the decomposition (\ref{decompose:chi}) 
restricts to the decomposition
\begin{align*}
\cC_{\zZ}=\bigoplus_{\lambda \in \Lambda} (D^b_{\rm{coh}}(\mathfrak{M})_{\lambda} \cap \cC_{\zZ}).
\end{align*}
By taking quotients, we have the equivalence (\ref{decompose:chi2}). 

(ii) We first show that the functor
\begin{align}\label{rho:equiv:2}
D_{\rm{coh}}^b(\mathfrak{M}^{\mathbb{C}^{\ast}\mathchar`-\mathrm{rig}})
\stackrel{\varrho_{\mathfrak{M}}^{\ast}}{\to} D^b_{\rm{coh}}(\mathfrak{M})_{\lambda=0}
\end{align}
is an equivalence. 
We have the natural transformation
$\id\to \varrho_{\mathfrak{M}\ast} \varrho_{\mathfrak{M}}^{\ast}$, 
which is an isomorphism
as $\varrho_{\mathfrak{M}}$ is \'etale locally trivial $\mathbb{C}^{\ast}$-gerbe. 
Therefore (\ref{rho:equiv:2}) is fully-faithful.
Moreover for an object $E \in D^b_{\rm{coh}}(\mathfrak{M})_{\lambda=0}$, 
we have $E=0$ if $\varrho_{\mathfrak{M}\ast}E=0$, again because 
$\varrho_{\mathfrak{M}}$ is \'etale locally trivial $\mathbb{C}^{\ast}$-gerbe. 
Therefore (\ref{rho:equiv:2}) is also essentially surjective, 
and hence (\ref{rho:equiv:2}) is equivalence. 
Then (\ref{rho:equiv:2})
 obviously restricts to the equivalence between $\cC_{\zZ^{\mathbb{C}^{\ast}\mathchar`-\mathrm{rig}}}$ and 
$D^b_{\rm{coh}}(\mathfrak{M})_{\lambda=0} \cap \cC_{\zZ}$. 
By taking quotients, we obtain the equivalence (\ref{equiv:rigidify}). 

(iii) Under the assumption, there exists a line bundle $\oO_{\mathfrak{M}}(\lambda)$
on $\mathfrak{M}$
given by the pull-back of $\oO_{B\mathbb{C}^{\ast}}(\lambda)$, where 
the latter is the one dimensional $\mathbb{C}^{\ast}$-representation 
with weight $\lambda$. 
The equivalence (\ref{trivial:gerbe}) is given by taking 
the tensor product with $\oO_{\mathfrak{M}}(\lambda)$. 
\end{proof}

\begin{rmk}\label{rmk:twisted}
If $\mM^{\mathbb{C}^{\ast}\rig}$ is a scheme, 
then the $\mathbb{C}^{\ast}$-gerbe (\ref{M:gerb})
is classified by a Brauer class
$\alpha \in \mathrm{Br}(\mM^{\mathbb{C}^{\ast}\rig})$. 
Then by~\cite[Proposition~2.1.3.3]{MR2309155}, 
the category $\Coh(\mM)_{\lambda}$ is equivalent to the 
category of 
$\alpha^{\lambda}$-twisted 
coherent sheaves on $\mM^{\mathbb{C}^{\ast}\rig}$
in the sense of C\v{a}ld\v{a}raru~\cite{MR2700538}. 
Similarly the $\lambda$-twisted version 
$\mathcal{DT}^{\mathbb{C}^{\ast}}(\nN^{\rm{ss}})_{\lambda}$
may be interpreted as $\alpha^{\lambda}$-twisted 
analogue for $\mathcal{DT}^{\mathbb{C}^{\ast}}
((\nN^{\rm{ss}})^{\mathbb{C}^{\ast}\rig})$. 
\end{rmk}

Similar argument also applies to define 
$\wdDT^{\C}(\nN^{\rm{ss}})_{\lambda}$
and $\wdDT^{\C}_{\rm{dg}}(\nN^{\rm{ss}})_{\lambda}$, with 
the properties as in Proposition~\ref{prop:rigid}. 
As we will not use them, we omit details.

\subsubsection{Functoriality of DT categories}
Let $\fM_1$, $\fM_2$ be quasi-smooth and QCA
derived stacks, and take a morphism
$f \colon \fM_1 \to \fM_2$. 
First suppose that 
$f$ is a quasi-smooth morphism. 
In this case, the morphism
$f^{\diamondsuit}$ in the diagram (\ref{diagram:induced}) 
is a closed immersion by Lemma~\ref{lem:beta0}. 
Therefore for any 
conical closed substack $\zZ_2 \subset t_0(\Omega_{\fM_2}[-1])$, 
we have the conical closed substack
\begin{align}\notag
	f^{\diamondsuit}(f^{\spadesuit})^{-1}(\zZ_2) \subset
	t_0(\Omega_{\fM_1}[-1]). 
\end{align}
By~\cite[Lemma~8.4.2]{MR3300415}
(also see~\cite[Lemma~2.2]{THtype}), 
the
functor (\ref{f:back}) 
restricts to the functor 
\begin{align}\label{funct1}
	f^{\ast} \colon 
	\cC_{\zZ_2} \to \cC_{f^{\diamondsuit}(f^{\spadesuit})^{-1}(\zZ_2)}.
	\end{align}

Next suppose that $f\colon \fM_1 \to \fM_2$ is a proper morphism.
In this case, the morphism 
$f^{\spadesuit}$ in the diagram (\ref{diagram:induced})
is a proper morphism of stacks. 
Therefore for any conical closed
substack $\zZ_1 \subset t_0(\Omega_{\fM_1}[-1])$, 
we have the 
conical closed substack 
\begin{align}\notag
	f^{\spadesuit} (f^{\diamondsuit})^{-1}(\zZ_1) \subset 
	t_0(\Omega_{\fM_2}[-1]). 
\end{align}
	Then the functor (\ref{f:push}) restricts to the functor 
	(see~\cite[Lemma~8.4.5]{MR3300415})
	\begin{align}\label{funct2}
		f_{\ast} \colon \cC_{\zZ_1}
		\to \cC_{f^{\spadesuit} (f^{\diamondsuit})^{-1}(\zZ_1)}. 
		\end{align}

Let $\fN$ be another quasi-smooth and QCA derived stack 
with 
a diagram
\begin{align}\label{diagram:N}
	\fM_1 \stackrel{f_1}{\leftarrow} \fN \stackrel{f_2}{\to} \fM_2. 
\end{align}
Suppose that $f_1$ is quasi-smooth and $f_2$ is a 
proper morphism. 
Then 
we have the functor
\begin{align}\label{FM:P}
	F \cneq f_{2\ast} \circ f_1^{\ast} \colon 
	\Dbc(\fM_1) \stackrel{f_1^{\ast}}{\to} \Dbc(\fN) 
	\stackrel{f_{2\ast}}{\to} \Dbc(\fM_2). 
\end{align}
Let $\ev$ be the morphism
\begin{align*}
	\ev \cneq (f_1, f_2) \colon \fN \to \fM_1 \times \fM_2,
\end{align*}
and $\Omega_{\ev}[-2]$ the $(-2)$-shifted conormal stack
\begin{align*}
	\Omega_{\ev}[-2] \cneq 
	\Spec_{\fM_1 \times \fM_2}(S(\mathbb{T}_{\ev}[2])) \to \fM_1 \times \fM_2. 
	\end{align*}
Then we have the diagram (see~\cite[Section~2.4]{THtype})
\begin{align}\label{diagram:fN}
	\xymatrix{
		& & t_0(\Omega_{\ev}[-2])
		\ar@{}[dd]|\square
		\ar[ld]_-{g_1} \ar[rd]^-{g_2} \ar@/_30pt/[lldd]_-{h_1} 
		\ar@/^30pt/[rrdd]^-{h_2} & & \\
		& f_1^{\ast}t_0(\Omega_{\fM_1}[-1]) \ar[ld]_-{f^{\spadesuit}_1} 
		\ar[rd]^-{f^{\diamondsuit}_1} & & f_2^{\ast}t_0(\Omega_{\fM_2}[-1]) \ar[ld]_-{f^{\diamondsuit}_2}
		\ar[rd]^-{f^{\spadesuit}_2} & \\
		t_0(\Omega_{\fM_1}[-1]) & & t_0(\Omega_{\fN}[-1]) & & t_0(\Omega_{\fM_2}[-1]).
	}
\end{align}
Here the middle square is Cartesian. 
Note that $f_1^{\diamondsuit}$ and $g_2$ are closed immersions as $f_1$ is quasi-smooth, 
and $f^{\spadesuit}_2$ is proper as $f_2$ is. 
Therefore $h_2$ is also proper. 
For 
a conical closed substack $\zZ_1 \subset t_0(\Omega_{\fM_1}[-1])$, 
we have 
\begin{align}\label{Z:h}
	f_2^{\spadesuit}(f_2^{\diamondsuit})^{-1}f_1^{\diamondsuit}(f_1^{\spadesuit})^{-1}(\zZ_1)
	=h_2(h_1)^{-1}(\zZ_1)
	\end{align}
which is a conical closed substack in 
$t_0(\Omega_{\fM_2}[-1])$. 
By (\ref{funct1}), (\ref{funct2}), it follows that the functor $F$ in (\ref{FM:P})
restricts to the functor 
\begin{align*}
	F \colon \cC_{\zZ_1} \to \cC_{h_2(h_1)^{-1}(\zZ_1)}. 
	\end{align*}
Therefore we have the following: 

\begin{prop}\emph{(\cite[Proposition~2.4]{THtype})}\label{prop:sendC}
	Suppose that 
	a conical closed substack 
	$\zZ_2 \subset t_0(\Omega_{\fM_2}[-1])$
	contains (\ref{Z:h}), or equivalently the diagram (\ref{diagram:fN}) restricts 
	to the diagram
	\begin{align*}
		\nN_1^{\rm{ss}} \stackrel{h_1}{\leftarrow} h_2^{-1}(\nN_2^{\rm{ss}}) 
		\stackrel{h_2}{\to} \nN_2^{\rm{ss}}. 
		\end{align*}
	Here $\nN_i=t_0(\Omega_{\fM_i}[-1])$ and $\nN_i^{\rm{ss}} \cneq \nN_i \setminus \zZ_i$. 
	Then the functor (\ref{FM:P})
	sends $\cC_{\zZ_1}$ to $\cC_{\zZ_2}$. 
	In particular,
	the functor (\ref{FM:P}) descends to the functor 
	\begin{align}\label{F:descend}
		F \colon \dDT^{\C}(\nN_1^{\rm{ss}}) \to \dDT^{\C}(\nN_2^{\rm{ss}}). 
		\end{align}
\end{prop}
By noting (\ref{ind:M}), the ind-completion of 
(\ref{FM:P}) gives rise to the functor 
\begin{align*}
	F^{\ind} \colon \cneq 
	f_{2\ast}^{\ind} \circ f_1^{\ind\ast} \colon 
	\Ind \Dbc(\fM_2) \stackrel{f_1^{\ind \ast}}{\to}
	\Ind \Dbc(\fN) \stackrel{f_{2\ast}^{\ind}}{\to}
	\Ind \Dbc(\fM_2). 
\end{align*}
It admits the continuous right adjoint functor (see Subsection~\ref{subsub:QCA})
\begin{align}\label{def:FR}
	F^R \cneq 
	f_{1\ast}^{\ind} \circ f_2^! \colon \Ind \Dbc(\fM_2) \stackrel{f_2^{!}}{\to}
	\Ind\Dbc(\fN) \stackrel{f_{1\ast}^{\rm{ind}}} \to 
	\Ind \Dbc(\fM_1). 
\end{align}
We expect that, similarly to Proposition~\ref{prop:sendC}, 
the above functor $F^R$ sends $\Ind \cC_{\zZ_2}$ to $\Ind \cC_{\zZ_1}$
if 
$\zZ_2$ is contained in (\ref{Z:h}). 
This is less obvious since 
$f_1^{\spadesuit}$, $f_2^{\diamondsuit}$ are not necessary proper.
Indeed the argument in Proposition~\ref{prop:sendC} only 
implies that $F^R$ restricts to the functor
\begin{align*}
	F^R \colon \Ind \cC_{\zZ_2} \to \Ind \cC_{\zZ_1'}, \ 
	\zZ_1'=\overline{f_1^{\spadesuit}(f_1^{\diamondsuit})^{-1}\overline{(f_2^{\diamondsuit}
		(f_2^{\spadesuit})^{-1}(\zZ_2))}}
	\end{align*}
and $\zZ_1'$ may not be contained in $\zZ_1$. 
If the diagram (\ref{diagram:N})
is represented by some digram of equivariant 
tuples, we can directly 
show that $F^R$ sends $\Ind \cC_{\zZ_2}$ to $\Ind \cC_{\zZ_1}$ by 
discussing in the Koszul dual side.  

Let $(Y_i, V_i, s_i)$ for $i=1, 2, 3$
be $G_i$-equivariant tuples,  
and consider morphisms of equivariant tuples
with respect to some algebraic group homomorphisms 
$G_3 \to G_i$ for $i=1, 2$
\begin{align}\label{dia:equiv0}
	\xymatrix{
		V_1 \ar[d] & V_3 \ar[r]^-{g_2} \ar[l]_-{g_1} \ar[d] & V_2 \ar[d] \\
		Y_1 \ar@/^10pt/[u]^-{s_1} & Y_3 \ar[r]_-{f_2} 
		\ar[l]^-{f_1}\ar@/^10pt/[u]^-{s_3} 
		& Y_2. \ar@/_10pt/[u]_-{s_2}
	}
\end{align}
The above diagram 
induces the diagram of derived stacks
\begin{align}\label{dia:Uf}
	[\fU_1/G_1] \stackrel{\mathbf{f}_1}{\leftarrow}
	[\fU_3/G_3] \stackrel{\mathbf{f}_2}{\to} [\fU_2/G_2]. 
\end{align}
The diagram (\ref{dia:equiv0}) also induces the diagram
(see the construction of the diagram (\ref{diagram:dual}))
\begin{align}\label{ind:diag}
	\xymatrix{		
	&	& \wW \ar@/_10pt/[dd]_-{q_3} 
		\ar@/_20pt/[lldd]_-{q_1} \ar@/^20pt/[rrdd]^-{q_2}
		\ar@/^10pt/[ddd]_-{\overline{w}_3}
		\ar@{}[dd]|\square  \ar[ld]_-{r_1} \ar[rd]^-{r_2}  &  & \\
	&	f_1^{\ast}\vV_1^{\vee} \ar@/_10pt/[rdd]_-{\overline{w}_1}
		\ar[dl]_-{f_1} \ar[rd]^-{g_1} &  & f_2^{\ast}\vV_2^{\vee}
		\ar@/^10pt/[ldd]^-{\overline{w}_2} \ar[dr]^-{f_2} \ar[ld]_-{g_2} &\\
	\vV_1^{\vee} \ar[rrd]_-{w_1}	&
		& \vV_3^{\vee} \ar[d]_-{w_3} &  &\vV_2^{\vee}  \ar[lld]^-{w_2} \\
	&	& \mathbb{C}. &   &
	}
\end{align}
Here as before we set $\yY_i=[Y_i/G_i]$ and $\vV_i=[V_i/G_i]$. 
We have the following proposition: 
\begin{prop}\label{prop:ind:R}
	Suppose that the diagram (\ref{diagram:N}) is represented by 
	a diagram (\ref{dia:Uf}) induced by a diagram (\ref{dia:equiv0})
		such that 
	$f_2 \colon \yY_3 \to \yY_1$ is proper, 
	the Cartesian in (\ref{ind:diag}) is a derived Cartesian 
	and $\wW$ is a smooth stack. 
	Assume that a conical closed substack 
	$\zZ_2 \subset t_0(\Omega_{\fM_2}[-1])$ is contained in (\ref{Z:h}), or 
	equivalently the diagram (\ref{diagram:fN}) restricts to the diagram 
	\begin{align*}
		\nN_1^{\rm{ss}} \stackrel{h_1}{\leftarrow} h_1^{-1}(\nN_1^{\rm{ss}}) 
		\stackrel{h_2}{\to} \nN_2^{\rm{ss}}.
	\end{align*}
	Then the functor $F^R$ descends to the 
	functor
	\begin{align*}
		F^R \colon 
		\Ind \dDT^{\C}(\nN_2^{\rm{ss}}) \to 
		\Ind \dDT^{\C}(\nN_1^{\rm{ss}})
	\end{align*}
	which is a right adjoint of the ind-completion of 
	the functor (\ref{F:descend}). 
\end{prop}
\begin{proof}
	We show that the functor $F^R$ in (\ref{def:FR})
	sends $\Ind \cC_{\zZ_2}$ to $\Ind \cC_{\zZ_1}$. 
	By Lemma~\ref{lem:ind}, Lemma~\ref{lem:commute1} and Lemma~\ref{lem:commute2}, 
	it is enough to show that the composition
	\begin{align*}
		\MF_{\qcoh}^{\C}(\vV_2^{\vee}, w_2) 
		&\stackrel{f_2^!}{\to}
		\MF_{\qcoh}^{\C}(f_2^{\ast}\vV_2^{\vee}, \overline{w}_2) 
		\stackrel{g_{2\ast}}{\to}
		\MF_{\qcoh}^{\C}(\vV_3^{\vee}, w_3) \\
		&	\stackrel{g_1^{\ast}}{\to}
		\MF_{\qcoh}^{\C}(f_1^{\ast}\vV_1^{\vee}, w_1) 
		\stackrel{f_{1\ast}}{\to}
		\MF_{\qcoh}^{\C}(\vV_1^{\vee}, w_1) 
	\end{align*}
	sends $\MF_{\qcoh}^{\C}(\vV_2^{\vee}, w_2)_{\zZ_2}$ to 
	$\MF_{\qcoh}^{\C}(\vV_1^{\vee}, w_1)_{\zZ_1}$. 
	By the assumption and the derived base change, 
	the above composition is isomorphic to the 
	composition
	\begin{align*}
		\MF_{\qcoh}^{\C}(\vV_2^{\vee}, w_2) 
		\stackrel{f_2^!}{\to}
		\MF_{\qcoh}^{\C}(f_2^{\ast}\vV_2^{\vee}, \overline{w}_2) 
		\stackrel{r_2^{\ast}}{\to}
		\MF_{\qcoh}^{\C}(\wW, \overline{w}_3) 
		\stackrel{q_{1\ast}}{\to}
		\MF_{\qcoh}^{\C}(\vV_1^{\vee}, w_1).  
	\end{align*}
	Let us take 
	an object $A \in \MF_{\qcoh}^{\C}(\vV_2^{\vee}, w_2)$. 
	Then
	$A$ is supported on $\mathrm{Crit}(w_2)$, so 
	$r_2^{\ast}f_2^{!}(A)$ is supported on 
	$q_2^{-1}(\Crit(w_2))$. 
	Also since $g_{2\ast}f_2^{!}A$ is supported on 
	$\mathrm{Crit}(w_3)$, we have 
	\begin{align*}
		q_{1\ast}(r_2^{\ast}f_2^!(A)|_{\wW \setminus q_3^{-1}(\Crit(w_3))})
		\cong f_{1\ast}g_1^{\ast}(g_{2\ast}f_2^!(A)|_{\vV_3^{\vee} \setminus 
			\Crit(w_3)}) \cong 0. 
	\end{align*}
	From the distinguished triangle in (\ref{dist:support})
	\begin{align*}
		\Gamma_{q_3^{-1}(\Crit(w_3))}(r_2^{\ast}f_2^!(A))
		\to r_2^{\ast}f_2^!(A) \to 
		r_2^{\ast}f_2^!(A)|_{\wW \setminus q_3^{-1}(\Crit(w_3))}
	\end{align*}
	we have the isomorphism 
	\begin{align}\label{isom:q1}
		q_{1\ast}	\Gamma_{q_3^{-1}(\Crit(w_3))}(r_2^{\ast}f_2^!(A))
		\stackrel{\cong}{\to}
		q_{1\ast}r_2^{\ast}f_2^!(A). 
	\end{align}
	Since the above object is supported on $\Crit(w_1)$, the object
	\begin{align}\label{obj:G}
		B \cneq \Gamma_{q_1^{-1}(\Crit(w_1))} \circ \Gamma_{q_3^{-1}(\Crit(w_3))}
		(r_2^{\ast}f_2^!(A))
		\in \MF_{\qcoh}^{\C}(\wW, \overline{w}_3)
	\end{align}
	is supported on $\cap_{i=1}^3 q_i^{-1}(\Crit(w_i))$
	such that $q_{1\ast}B$ is isomorphic to the object (\ref{isom:q1}). 
	
	From the diagram (\ref{diagram:fN}) and Lemma~\ref{lem:critical}, 
	we have 
	\begin{align}\label{ev:-2}
		t_0(\Omega_{\rm{ev}}[-2])
		=\bigcap_{i=1}^3 q_i^{-1}(\Crit(w_i)) \subset \wW
	\end{align}
	and $h_i$ is the restriction of $q_i$ to 
	the above closed substack. 
	Now suppose that $h_2^{-1}(\zZ_2) \subset h_1^{-1}(\zZ_1)$ and $A$
	is supported on $\zZ_2$. 
	Then the object (\ref{obj:G}) is supported on 
	$h_2^{-1}(\zZ_2)$, hence on $h_1^{-1}(\zZ_1)$. 
	Therefore the object (\ref{isom:q1}) is supported on $\zZ_1$. 
	
	 Since the
	functor (\ref{def:FR}) sends $\Ind \cC_{\zZ_2}$
	to $\Ind \cC_{\zZ_1}$, it descends to the functor
	\begin{align*}
		\Ind \Dbc(\fM_2)/\Ind \cC_{\zZ_2} \to 
		\Ind \Dbc(\fM_1)/\Ind \cC_{\zZ_1}. 
	\end{align*}
	Therefore we obtain the desired functor by Lemma~\ref{lem:gcomplete} and 
	Proposition~\ref{lem:indDT}. 
	\end{proof}


\subsection{Comparison with cohomological/numerical DT invariants}\label{subsec:cohDT}
In this section, 
we discuss the relationship between 
DT categories in the previous 
section with 
the cohomological DT invariants whose foundation 
is established in~\cite{MR3353002}. 
In~\cite{MR3815168}, it is proved that 
the periodic cyclic homology of 
the category of matrix factorization is 
isomorphic to the $\mathbb{Z}/2$-graded hypercohomology 
of the perverse sheaf of the vanishing cycles, 
and the discussion here is based on this result. 
Then we show that, under some technical assumptions, 
we relate the Euler characteristic of 
the periodic cyclic homology of our DT category
with the numerical DT invariant defined by the Behrend function.

\subsubsection{Periodic cyclic homologies (review)}
Here we recall the definition of periodic cyclic 
homologies for small dg-categories introduced 
by Keller~\cite{MR1667558}. 
Our convention here is due to~\cite[Section~3.1]{MR3815168}. 

Let
$\tT$ be a $(\mathbb{Z}$ or $\mathbb{Z}/2)$-graded
small dg-category. 
Recall that the 
Hochschild complex 
$(\mathrm{Hoch}(\tT), b)$
is 
defined to be 
\begin{align*}
	\mathrm{Hoch}(\tT)&=\bigoplus_{E \in \tT} 
	\Hom_{\tT}(E, E)  \oplus \\
	&\bigoplus_{\begin{subarray}{c}n\ge 1, \\
			E_0, \ldots, E_n \in \tT
	\end{subarray}} \Hom_{\tT^e}(E_n, E_0) \otimes 
	\Hom_{\tT}(E_{n-1}, E_n)[1] \otimes 
	\cdots \otimes \Hom_{\tT}(E_0, E_1)[1].
\end{align*}
Here $\tT^e$ 
is the dg-category defined by formally adding a 
closed identity 
morphism $e_E$ for each
$E \in \tT$, i.e. 
$\mathrm{Ob}(\tT^e)=\mathrm{Ob}(\tT)$
and 
\begin{align*}
	\Hom_{\tT^e}(E, E')=\left\{ 
	\begin{array}{ll}
		\Hom_{\tT}(E, E'), & E \neq E', \\
		\Hom_{\tT}(E, E) \oplus \mathbb{C} \cdot e_E, 
		& E=E'. 
	\end{array}
	\right. 
\end{align*}  
The differential $b$ is the sum
$b=b_1+b_2$, 
where $b_1$ is given by
\begin{align*}
	b_1(a_n, a_{n-1}, \ldots, a_0)=
	\sum_{i=0}^{n} \pm (a_n \ldots, da_{i}, \ldots, a_0) 
\end{align*}
and $b_2$ is given by
\begin{align*}
	b_2(a_n, a_{n-1}, \ldots, a_0)=
	\pm (a_0 a_n, a_{n-1}, \ldots, a_1)+
	\sum_{i=0}^{n-1}\pm
	(a_{n}, \ldots, a_{i+1} a_i, \ldots, 
	a_0). 
\end{align*}
Moreover there is a 
Connes differential $B$ of degree $-1$ defined by
\begin{align*}
	B(a_n, a_{n-1}, \ldots, a_0)=
	\sum_{i=0}^n \pm (e_{E_{i}}, a_{i-1}, \ldots, a_0, 
	a_{n}, \ldots, a_i)
\end{align*}
if $a_n \in \Hom_{\tT}(E_n, E_0)$
and zero if $a_n \in \mathbb{C} \cdot e_{E_{0}}$. 
Here we refer to~\cite[Example, Section~3.1]{MR3815168} for 
choices of signs. 
The triple 
\begin{align}\label{triple}
	(\mathrm{Hoch}(\tT), b, B)
\end{align}
is a mixed complex, i.e. 
dg-module over $\mathbb{C}[\epsilon]/\epsilon^2$
with $\deg \epsilon =-1$, where $\epsilon$
acts by $B$. 

The 
periodic cyclic homology 
of $\tT$ is defined to be
\begin{align}\label{PHom}
	\mathrm{HP}_{\ast}(\tT) \cneq H^{\ast}(\mathrm{Hoch}(\tT)\lgakko u \rgakko, b+uB). 
\end{align}
Here 
$u$ is a formal variable of degree $2$. 
The periodic cyclic homology (\ref{PHom}) is also 
equipped with a natural connection 
$\nabla_{u}$, which makes (\ref{PHom}) 
a $(\mathbb{Z}$ or $\mathbb{Z}/2$)-graded
vector bundle over the formal punctured disc
$\Spf \mathbb{C}\lgakko u \rgakko$
with a connection (see~\cite[Example, Section~3.3]{MR3815168} for details). 
Note that $\mathrm{HP}_{\ast}(\tT)$ is $2$-periodic
since $u \colon \mathrm{HP}_{\ast}(\tT) \to \mathrm{HP}_{\ast+2}(\tT)$
is an isomorphism. 

Below for a $\mathbb{Z}$-graded object $V^{\bullet}$, 
we set $V^{\mathbb{Z}/2}$ to be $V^{\bullet}$ regarded as 
$\mathbb{Z}/2$-graded object, i.e. 
\begin{align}\label{regard:Z2}
	V^{\mathbb{Z}/2}=V^{\rm{even}} \oplus V^{\rm{odd}}. 
\end{align}
For a $\mathbb{Z}$-graded dg-category $\tT$, 
we can regard it as a $\mathbb{Z}/2$-graded dg-category $\tT^{\mathbb{Z}/2}$
by setting
\begin{align*}
	\Hom_{\tT^{\mathbb{Z}/2}}(E, F)
	=\Hom_{\tT}(E, F)^{\mathbb{Z}/2}. 
\end{align*}
Then 
we have the obvious identities 
\begin{align*}
	\mathrm{HP}_{\ast}(\tT^{\mathbb{Z}/2})=
	\mathrm{HP}_{\ast}(\tT)^{\mathbb{Z}/2}=
	(\mathrm{HP}_0(\tT) \oplus \mathrm{HP}_1(\tT))\lgakko u \rgakko. 
\end{align*}

\subsubsection{Periodic cyclic homologies for derived categories of factorizations}
Let $M$ be a smooth quasi-projective variety 
with a regular function $w \colon M \to \mathbb{C}$. 
Then we have the perverse sheaf of vanishing cycles 
(see~\cite[Theorem~5.2.21]{MR2050072})
\begin{align*}
	\phi_{w} \cneq \phi_{w}(\mathbb{Q}_M[\dim M]) \in \mathrm{Perv}(\mathrm{Crit}(w))
\end{align*}
with a monodromy automorphism $T_{w} \in \Aut(\phi_{w})$. 

On the other hand, 
let $\mathrm{MF}_{\rm{coh}}^{\mathbb{Z}/2}(M, w)$
be the $\mathbb{Z}/2$-periodic triangulated categorify of 
(non-$\mathbb{C}^{\ast}$-equivariant) 
coherent factorizations of $w \colon M \to \mathbb{C}$ (see Remark~\ref{rmk:def:Z2}). 
We have its dg-enhancement 
$L_{\rm{coh}}^{\mathbb{Z}/2}(M, w)$, 
defined to be the $\mathbb{Z}/2$-graded 
dg-category of coherent factorizations 
of $w$, 
localized by acyclic factorizations. 
Then the following result is proved in~\cite{MR3815168}:
\begin{thm}\emph{(\cite[Theorem~1.1]{MR3815168})}\label{thm:cyclic}
	We have an isomorphism
	\begin{align}\notag
		(\mathrm{HP}_{\ast}(L_{\rm{coh}}^{\mathbb{Z}/2}(M, w)), \nabla_u) 
		\cong (H^{\ast+\dim M}(M, \phi_{w})
		\otimes_{\mathbb{Q}}\mathbb{C}\lgakko u \rgakko, 
		d+T_{w}/u)
	\end{align}
	as $\mathbb{Z}/2$-graded 
	vector bundles on 
	$\Spf \mathbb{C}\lgakko u \rgakko$ with connections. 
\end{thm}

Suppose that $\mathbb{C}^{\ast}$ acts on $M$ 
which restricts to the trivial action on 
$\mu_2 \subset \mathbb{C}^{\ast}$, 
and assume that $w$ is of weight two. 
Then 
as in Subsection~\ref{subsub:fact}, 
we have the derived category of $\C$-equivariant 
factorizations $\mathrm{MF}_{\rm{coh}}^{\mathbb{C}^{\ast}}(M, w)$,
and its 
dg-enhancement $L_{\rm{coh}}^{\mathbb{C}^{\ast}}(M, w)$
defined to be the 
dg-category of
$\mathbb{C}^{\ast}$-equivariant coherent 
factorizations of $w$ 
localized by acyclic factorizations. 
Then by the following lemma together with Theorem~\ref{thm:cyclic}, 
we can relate the periodic cyclic homology of 
$L_{\rm{coh}}^{\mathbb{C}^{\ast}}(M, w)$ with the 
vanishing cycle cohomology:

\begin{lem}\label{lem:Zgraded}
	We have an isomorphism as $\mathbb{Z}/2$-graded 
	vector bundles on 
	$\Spf \mathbb{C}\lgakko u \rgakko$ with connections:
	\begin{align*}
		(\mathrm{HP}_{\ast}(L_{\rm{coh}}^{\mathbb{C}^{\ast}}(M, w)), 
		\nabla_{u})^{\mathbb{Z}/2}
		\cong (\mathrm{HP}_{\ast}(L_{\rm{coh}}^{\mathbb{Z}/2}(M, w)), \nabla_{u}). 
	\end{align*}
	Here $(-)^{\mathbb{Z}/2}$ means regarding $\mathbb{Z}$-graded 
	object as a $\mathbb{Z}/2$-graded object (see (\ref{regard:Z2})). 
\end{lem}
\begin{proof}
	Since $M$ is smooth quasi-projective, 
	by Sumihiro's theorem~\cite{MR0337963}
	$M$ is covered by $\mathbb{C}^{\ast}$-invariant
	affine open subsets.
	Then the argument
	in~\cite[Proposition~5.1]{MR3815168}
	implies that we can assume that $M$ is affine, 
	$M=\Spec R$ for a $\mathbb{C}$-algebra $R$. 
	Note that the $\mathbb{C}^{\ast}$-action on $M$ yields the 
	$\mathbb{Z}$-grading $R^{\bullet}$ on $R$ such that $w \in R$ 
	is degree two, and $R^{\rm{odd}}=0$. 
	
	In this case, 
	it is proved in~\cite[(4.2), (4.3), (4.4)]{MR3815168}
	that the Hochschild complex of 
	$L_{\rm{coh}}^{\mathbb{Z}/2}(M, w)$
	is, as a mixed complex, 
	quasi-isomorphic to the 
	Hochschild complex of the second kind (see~\cite[Section~2.4]{MR2931331})
	of the 
	$\mathbb{Z}/2$-graded curved dg-algebra
	$(R^{\mathbb{Z}/2}, 0, w)$, 
	where the differential is zero and $w$ is the curvature. 
	As we will explain below,  
	the same arguments in~\cite[(4.2), (4.3), (4.4)]{MR3815168}
	literally work in the $\mathbb{Z}$-graded case, 
	and we have the following: 
	the dg-category 
	$L_{\rm{coh}}^{\mathbb{C}^{\ast}}(M, w)$
	is equivalent to the 
	dg-category of 
	finitely generated 
	$\mathbb{Z}$-graded 
	curved dg-modules over $(R^{\bullet}, 0, w)$
	and its Hochschild complex is, 
	as a mixed complex, quasi-isomorphic to 
	the Hochschild complex 
	of the second kind $\mathrm{Hoch}^{\rm{I\hspace{-.1em}I}}(R^{\bullet}, 0, w)$
	of $(R^{\bullet}, 0, w)$. 
	Explicitly, it is written as
	\begin{align*}
		\mathrm{Hoch}^{\rm{I\hspace{-.1em}I}}(R^{\bullet}, 0, w)=
		R^{\bullet} \oplus \prod_{n\ge 1}
		\overbrace{R^{\bullet} \otimes R^{\bullet}[1] \otimes \cdots \otimes 
			R^{\bullet}[1]}^{n+1}
	\end{align*}
	with 
	differential the sum of the 
	Hochschild differential
	and the twisted differential which involves $w$, 
	see~\cite[Section~3.1]{MR2931331} for details. 
	Then 
	$\mathrm{Hoch}^{\rm{I\hspace{-.1em}I}}(R^{\bullet}, 0, w)^{\mathbb{Z}/2}$
	is nothing but the Hochschild complex of the second kind for 
	$(R^{\mathbb{Z}/2},  0, w)$. 
	Therefore we have the desired isomorphism.  
	
	Here we explain that the 
	arguments of~\cite[(4.2), (4.3), (4.4)]{MR3815168}
	work in the above $\mathbb{Z}$-graded case. 
	Indeed~\cite[(4.2), (4.4)]{MR3815168}
	are consequences of~\cite[Proposition~2.13, Proposition~3.10]{MR3815168}, 
	which refer to~\cite[Lemma~1.5.A, Section~2.4]{MR2931331}.
	But the latter reference is formulated for both of $\mathbb{Z}$ and $\mathbb{Z}/2$ grading, 
	so the same argument applies. 
	Also~\cite[(4.3)]{MR3815168}
	is a consequence of~\cite[Proposition~(3.12)]{MR3815168}, 
	which refers to~\cite[Corollary~4.8A]{MR2931331}.
	The latter reference is formulated only for the $\mathbb{Z}/2$-grading
	case,
	but we can easily modify the proof for the $\mathbb{Z}$-graded
	case by using 
	$\mathbb{Z}$-graded version of 
	Orlov's equivalence between the category of 
	graded matrix factorizations 
	and graded 
	singularity category (see~\cite[Theorem~3.10]{Orsin}). 
\end{proof}

\subsubsection{Conjectural relation with cohomological DT theory}
Let $\mathfrak{M}$ be a quasi-smooth QCA derived stack. 
Below, we use notation in Subsection~\ref{subsec:shifted:cotangent}. 
For a smooth morphism $\alpha \colon \mathfrak{U} \to \mathfrak{M}$
with a presentation $\mathfrak{U}=\Spec \rR(V\to Y, s)$
as in (\ref{frak:U}),
we have the perverse sheaf of vanishing cycles
\begin{align}\notag
	\phi_{w} \cneq \phi_{w}(\mathbb{Q}_{V^{\vee}}[\dim V^{\vee}])
	\in \Perv(\mathrm{Crit}(w)) 
\end{align} 
associated with the diagram (\ref{diagram:KZ}), 
equipped with a monodromy automorphism 
$T_{w} \in \Aut(\phi_{\omega})$. 
By~\cite{MR3353002}, 
we can glue these perverse sheaves to a global 
perverse sheaf
on $\nN=t_0(\Omega_{\mathfrak{M}}[-1])$ with an automorphism
\begin{align}\label{perv:N}
	\phi_{\nN} \in \Perv(\nN), \ 
	T_{\nN} \in \Aut(\phi_{\nN})
\end{align}
once we choose an \textit{orientation data} of the 
d-critical locus $\nN$. 
Here an orientation data is a choice of a square root 
line bundle of the virtual canonical 
line bundle
\begin{align*}
	K_{\nN}^{\rm{vir}}=\det(\mathbb{L}_{\Omega_{\mathfrak{M}}[-1]}|_{\nN})
	\in \Pic(\nN).
\end{align*}
In our case,
the following lemma is well-known:
\begin{lem}\label{lem:square}
	We have a canonical isomorphism 
	$K_{\nN}^{\rm{vir}} \cong p_0^{\ast}\det(\mathbb{L}_{\mathfrak{M}}|_{\mM})^{\otimes 2}$. \end{lem}
\begin{proof}
	The lemma follows from the 
	distinguished triangle
	\begin{align*}
		p^{\ast}\mathbb{L}_{\mathfrak{M}} \to \mathbb{L}_{\Omega_{\mathfrak{M}}[-1]}
		\to \mathbb{L}_{\Omega_{\mathfrak{M}}[-1]/\mathfrak{M}}=p^{\ast}
		\mathbb{T}_{\mathfrak{M}}[1]. 
	\end{align*}
\end{proof}

By the above lemma, 
we can take a canonical 
choice of orientation data
$K_{\nN}^{\rm{vir}, 1/2}=
p_0^{\ast}\det(\mathbb{L}_{\mathfrak{M}}|_{\mM})$, 
and we define (\ref{perv:N}) by this choice of orientation. 
Let $\zZ \subset \nN$
be a conical closed substack such that 
the complement $\nN^{\rm{ss}} \cneq \nN \setminus \zZ$
is a scheme (not stack), which is of finite 
type by the QCA assumption of $\mathfrak{M}$. 
Then 
the cohomological DT invariant is 
defined to be the hypercohomology
\begin{align}\label{DT:coh}
	H^{\ast}(\nN^{\rm{ss}} , \phi_{\nN}|_{\nN^{\rm{ss}} }). 
\end{align}
Based on Theorem~\ref{thm:cyclic} and Lemma~\ref{lem:Zgraded}, 
we expect the following relation of the periodic cyclic 
homologies of DT categories with 
cohomological DT invariants: 

\begin{conj}\label{conj:periodic}
	We have an isomorphism of $\mathbb{Z}/2$-graded 
	vector bundles on 
	$\Spf \mathbb{C}\lgakko u \rgakko$ with connections
	\begin{align*}
		(\mathrm{HP}_{\ast}
		(\mathcal{DT}_{\rm{dg}}^{\mathbb{C}^{\ast}}(\nN^{\rm{ss}})), \nabla_u)^{\mathbb{Z}/2} \cong (H^{\ast+\vdim \mM}(\nN^{\rm{ss}}, \phi_{\nN}|_{\nN^{\rm{ss}}}) \otimes_{\mathbb{Q}}\mathbb{C}\lgakko u \rgakko, 
		d+T_{\nN}/u). 
	\end{align*}
	Here $\vdim \mM$ is the rank of 
	$\mathbb{L}_{\mathfrak{M}}|_{\mM}$. 
\end{conj}
\subsubsection{Relation with numerical DT invariants}
In the situation of the previous 
subsection, 
the hypercohomology (\ref{DT:coh}) is finite dimensional 
and the numerical DT invariant is defined by
\begin{align*}
	\DT(\nN^{\rm{ss}}) \cneq 
	\chi(H^{\ast}(\nN^{\rm{ss}}, \phi_{\nN}|_{\nN^{\rm{ss}}}))
	=\int_{\nN^{\rm{ss}}} \nu_{\nN} \ d\chi.
\end{align*}
Here $\nu_{\nN} \colon x \mapsto \chi(\phi_{\nN}|_{x})$
is a constructible function on $\nN^{\rm{ss}}$, 
called \textit{Behrend function}~\cite{MR2600874}. 
In this case, the isomorphism in Conjecture~\ref{conj:periodic}
in particular implies 
the numerical identity between  
$\DT(\nN^{\rm{ss}})$ and the 
Euler characteristics of the periodic cyclic homology of 
$\mathcal{DT}_{\rm{dg}}^{\mathbb{C}^{\ast}}(\nN^{\rm{ss}})$. 
Here we show that the numerical version of 
Conjecture~\ref{conj:periodic} is true 
under some technical assumptions on $\mathfrak{M}$. 
We first show the following lemma: 
\begin{lem}\label{lem:cyclic}
	Let $\mathfrak{M}=\mathfrak{M}_1 \cup \mathfrak{M}_2$
	be a Zariski open cover of $\mathfrak{M}$. 
	We set $\mM_i=t_0(\mathfrak{M}_i)$, 
	$\mM_{12}=t_0(\mathfrak{M}_{12})$ and 
	\begin{align*}
		\nN_i^{\rm{ss}}
		\cneq p_0^{-1}(\mM_i) \cap \nN^{\rm{ss}}, \ 
		\nN_{12}^{\rm{ss}} \cneq 
		p_0^{-1}(\mM_{12}) \cap \nN^{\rm{ss}}. 
	\end{align*}
	Then we have the distinguished triangle 
	in the derived category of mixed complexes
	\begin{align*}
		\mathrm{Hoch}(\mathcal{DT}^{\mathbb{C}^{\ast}}_{\rm{dg}}(\nN^{\rm{ss}})) \to 
		\mathrm{Hoch}(\mathcal{DT}_{\rm{dg}}^{\mathbb{C}^{\ast}}(\nN_{1}^{\rm{ss}}))
		\oplus
		\mathrm{Hoch}(\mathcal{DT}_{\rm{dg}}^{\mathbb{C}^{\ast}}(\nN_{2}^{\rm{ss}})) 
		\to \mathrm{Hoch}(\mathcal{DT}_{\rm{dg}}^{\mathbb{C}^{\ast}}
		(\nN_{12}^{\rm{ss}})). 
	\end{align*}
\end{lem}
\begin{proof}
	The proof is similar to~\cite[Proposition~5.1]{MR3815168}. 
	We set $\wW=\mM \setminus \mM_1 \subset \mM_2$, 
	and define $L_{\rm{coh}}(\mathfrak{M})_{\wW}$
	to be the dg-subcategory of $L_{\rm{coh}}(\mathfrak{M})$ consisting 
	of objects whose cohomology sheaves are supported on $\wW$. 
	Then 
	by Lemma~\ref{lem:exact:Z}, 
	for 
	$\zZ_i=p_0^{-1}(\mM_i) \cap \zZ$ we have the exact
	sequences of dg-categories, i.e. 
	the induced sequences of the homotopy 
	categories are localization sequences
	\begin{align*}
		&L_{\rm{coh}}(\mathfrak{M})_{\wW}/L\cC_{p_0^{-1}(\wW)\cap \zZ} \to 
		\mathcal{DT}^{\mathbb{C}^{\ast}}_{\rm{dg}}(\nN^{\rm{ss}}) \to 
		\mathcal{DT}^{\mathbb{C}^{\ast}}_{\rm{dg}}(\nN_{1}^{\rm{ss}}), \\
		&L_{\rm{coh}}(\mathfrak{M}_2)_{\wW}/L\cC_{p_0^{-1}(\wW)\cap \zZ_2} \to 
		\mathcal{DT}^{\mathbb{C}^{\ast}}_{\rm{dg}}(\nN_{2}^{\rm{ss}}) \to 
		\mathcal{DT}^{\mathbb{C}^{\ast}}_{\rm{dg}}(\nN_{12}^{\rm{ss}}).
	\end{align*}
	By~\cite[Theorem~3.1]{MR1492902}, 
	the assignment of dg-categories to
	the triples (\ref{triple})
	takes exact sequences 
	to the distinguished triangles of mixed complexes. 
	Therefore 
	we have the distinguished triangles of mixed complexes
	\begin{align}\label{triangle:mixed}
		&\mathrm{Hoch}(L_{\rm{coh}}(\mathfrak{M})_{\wW}/L\cC_{p_0^{-1}(\wW)\cap \zZ}) \to 
		\mathrm{Hoch}(\mathcal{DT}^{\mathbb{C}^{\ast}}_{\rm{dg}}(\nN^{\rm{ss}})) \to 
		\mathrm{Hoch}(\mathcal{DT}^{\mathbb{C}^{\ast}}_{\rm{dg}}(\nN_{1}^{\rm{ss}})), \\
		\notag
		&\mathrm{Hoch}(L_{\rm{coh}}(\mathfrak{M}_2)_{\wW}/L\cC_{p_0^{-1}(\wW)\cap \zZ_2}) \to 
		\mathrm{Hoch}(\mathcal{DT}^{\mathbb{C}^{\ast}}_{\rm{dg}}(\nN_{2}^{\rm{ss}})) \to 
		\mathrm{Hoch}(\mathcal{DT}^{\mathbb{C}^{\ast}}_{\rm{dg}}(\nN_{12}^{\rm{ss}})). 
	\end{align}
	Since $\wW \subset \mM_2$, the restriction functor 
	gives a weak equivalence
	\begin{align*}
		L_{\rm{coh}}(\mathfrak{M})_{\wW}/L\cC_{p_0^{-1}(\wW)\cap \zZ}
		\stackrel{\sim}{\to}L_{\rm{coh}}(\mathfrak{M}_2)_{\wW}/L\cC_{p_0^{-1}(\wW)\cap \zZ_2}.
	\end{align*}
	By~\cite[Theorem~2.6]{MR1492902}, 
	the mixed complexes (\ref{triple}) are quasi-isomorphisms under 
	weak equivalences of dg-categories. 
	Therefore 
	we have the quasi-isomorphism of mixed complexes
	\begin{align*}
		\mathrm{Hoch}(L_{\rm{coh}}(\mathfrak{M})_{\wW}/L\cC_{p_0^{-1}(\wW)\cap \zZ})
		\stackrel{\sim}{\to} \mathrm{Hoch}(L_{\rm{coh}}(\mathfrak{M}_2)_{\wW}/L\cC_{p_0^{-1}(\wW)\cap \zZ_2}).
	\end{align*}
	The lemma holds from distinguished triangles (\ref{triangle:mixed})
	together with the above quasi-isomorphism. 
\end{proof}

For a finite dimensional $\mathbb{Z}/2$-graded vector space 
$V=V_0 \oplus V_1$ over $\mathbb{C}\lgakko u \rgakko$, 
its Euler characteristic 
is defined by 
\begin{align*}
	\chi(V) \cneq \dim_{\mathbb{C}\lgakko u \rgakko} V_0 -
	\dim_{\mathbb{C}\lgakko u \rgakko} V_1.
\end{align*}
We have the following proposition: 
\begin{prop}\label{prop:num}
	Let $\mathfrak{M}$ be a 
	quasi-smooth QCA derived stack
	and $\zZ \subset \nN=t_0(\Omega_{\mathfrak{M}}[-1])$
	is a conical closed substack. 
	Suppose that 
	for each $\mathbb{C}$-valued point $x \in \mathfrak{M}$, 
	there is a Zariski open neighborhood $x \in \mathfrak{M}' \subset \mathfrak{M}$
	satisfying the following conditions: 
	\begin{enumerate}
		\item $\mathfrak{M}'$ is of the form
		$\mathfrak{M}'=[\mathfrak{U}/G]$, 
		where $\mathfrak{U}=\Spec_Y \rR(V \to Y, s)$
		for a smooth quasi-projective scheme $Y$ and a section $s$
		on a vector bundle $V$ on it, and an affine algebraic group $G$
		acts on $Y$, $V$, such that $s$ is $G$-equivariant. 
		\item
		Let $\nN'=\nN \times_{t_0(\mathfrak{M})} t_0(\mathfrak{M}')$
		which is a conical
		closed substack in $[V^{\vee}/G]$, 
		and set ${\nN'}^{\rm{ss}}=\nN' \cap \nN^{\rm{ss}}$, 
		$\zZ'=\nN' \cap \zZ$.  
		Then there is a $\mathbb{C}^{\ast}$-invariant 
		quasi-projective 
		open substack $\wW \subset [V^{\vee}/G] \setminus \zZ'$
		satisfying ${\nN'}^{\rm{ss}}=\nN' \cap \wW$. 
	\end{enumerate}
	Then 
	the periodic cyclic homology 
	$\mathrm{HP}_{\ast}
	(\mathcal{DT}_{\rm{dg}}^{\mathbb{C}^{\ast}}(\nN^{\rm{ss}}))$ 
	is finite dimensional
	over $\mathbb{C}\lgakko u \rgakko$, and 
	we have the identity
	\begin{align}\notag
		\chi(\mathrm{HP}_{\ast}(\mathcal{DT}_{\rm{dg}}^{\mathbb{C}^{\ast}}(\nN^{\rm{ss}})))=(-1)^{\vdim \mM} \DT(\nN^{\rm{ss}}). 
	\end{align}
\end{prop}
\begin{proof}
	By the 
	QCA assumption on $\mathfrak{M}$, 
	it is covered by finite number of 
	Zariski open substacks 
	of the form $\mathfrak{M}'$ satisfying 
	the conditions (i), (ii). 
	Therefore by Lemma~\ref{lem:cyclic} and 
	the induction on the number of open covers, 
	we can assume that 
	$\mathfrak{M}=\mathfrak{M}'$. 
	Then $\nN=[\mathrm{Crit}(w)/G]$
	where $w \colon V^{\vee} \to \mathbb{C}$ is given 
	as in the diagram (\ref{diagram:KZ}), which is $G$-invariant 
	by the condition (i), 
	and $\zZ$ is written as $[Z/G]$ for a $G$-invariant conical 
	closed subset $Z \subset \mathrm{Crit}(w)$. 
	
By Proposition~\ref{prop:koszul:Z}, we have the equivalence 
	\begin{align*}
		\mathcal{DT}^{\mathbb{C}^{\ast}}(\nN^{\rm{ss}}) 
		=D^b_{\rm{coh}}([\mathfrak{U}/G])/\cC_{[Z/G]} 
		\stackrel{\sim}{\to} 
		\mathrm{MF}_{\rm{coh}}^{\mathbb{C}^{\ast}}([(V^{\vee} \setminus Z)/G], w). 
	\end{align*}
	We take $\wW$ as in the condition (ii), 
	and write $\wW=[W/G]$ 
	for a $G$-invariant open subscheme $W \subset V^{\vee} \setminus Z$
	on which the $G$-action is free. 
	Since $\mathrm{Crit}(w) \setminus Z=\mathrm{Crit}(w|_{W})$, 
	the restriction functor 
	gives an equivalence (see~(\ref{rest:equiv}))
	\begin{align*}
		\mathrm{MF}_{\rm{coh}}^{\mathbb{C}^{\ast}}([(V^{\vee} \setminus Z)/G], w)
		\stackrel{\sim}{\to} \mathrm{MF}_{\rm{coh}}^{\mathbb{C}^{\ast}}(\wW, w). 
	\end{align*}
	The above equivalence obviously lifts to a weak equivalence of the dg-enhancements, 
	so by~\cite[Theorem~2.6]{MR1492902}
	we have the quasi-isomorphism
	of mixed complexes
	\begin{align*}
		\mathrm{Hoch}(
		\mathcal{DT}_{\rm{dg}}^{\mathbb{C}^{\ast}}(\nN^{\rm{ss}}))
		\stackrel{\sim}{\to}
		\mathrm{Hoch}(L_{\rm{coh}}^{\mathbb{C}^{\ast}}(\wW, w)). 
	\end{align*}
	Moreover we have 
	\begin{align*}
		\vdim \mM=\dim Y-\rank(V)-\dim G, \ 
		\dim \wW=\dim Y+ \rank(V)-\dim G, 
	\end{align*}
	hence $\vdim \mM \equiv \dim \wW \pmod{2}$. 
	Therefore 
	the proposition holds by applying Theorem~\ref{thm:cyclic}
	and Lemma~\ref{lem:Zgraded}
	for $w \colon \wW \to \mathbb{C}$. 
\end{proof}

\begin{rmk}\label{rmk:condition}
	The assumption of Proposition~\ref{prop:num} is automatically 
	satisfied if $t_0(\mathfrak{M})$ is a quasi-projective 
	scheme by~\cite[Theorem~4.1]{BBJ}. 
	In this case, the algebraic group $G$ can be taken 
	to be trivial. 
\end{rmk}

\subsection{DT categories for local surfaces}\label{subsec:catDTsurf}
In this section, we introduce DT categories 
for compactly supported 
(semi)stable coherent sheaves on local surfaces, 
following the constructions in Section~\ref{subsec:catVer}. 
We then formulate a conjecture 
on the wall-crossing formula of DT categories of 
one dimensional stable sheaves from the viewpoint of 
d-critical birational geometry. 
\subsubsection{Derived moduli stacks of coherent sheaves on surfaces}
Let $S$ be a smooth projective surface over $\mathbb{C}$. 
We consider the 
derived Artin stack
\begin{align}\label{dstack:MS}
\mathfrak{Perf}_S \colon dAff^{op} \to SSets
\end{align}
which sends an affine derived scheme 
$T$ to the $\infty$-groupoid of 
perfect complexes on $T \times S$, 
constructed in~\cite{MR2493386}. 
We have the open substack
\begin{align*}
\mathfrak{M}_S \subset \mathfrak{Perf}_S
\end{align*}
corresponding to perfect complexes on $S$ 
quasi-isomorphic to 
coherent sheaves on $S$. 
Since
any object in $\Coh(S)$ is perfect as $S$ is smooth,
the derived Artin stack $\mathfrak{M}_S$ is 
the derived  
moduli stack of objects in 
$\Coh(S)$. 

Let $\mM_S \cneq t_0(\mathfrak{M}_S)$ and 
take the universal families
\begin{align}\label{F:universal}
\mathfrak{F} \in D_{\rm{coh}}^b(S \times \mathfrak{M}_S), \ 
\fF \cneq \mathfrak{F}|_{S \times \mM_S} \in D_{\rm{coh}}^b(S \times \mM_S). 
\end{align}
Then the perfect obstruction theory 
on $\mM_S$ induced by the 
cotangent complex of 
$\mathfrak{M}_S$ is given by 
\begin{align}\label{perf:obs1}
&\eE^{\bullet}=\left(\dR p_{\mM\ast} 
\dR \hH om_{S \times \mM_S}(\fF, \fF)[1]\right)^{\vee}
\to \tau_{\ge -1} \mathbb{L}_{\mM_S}.
\end{align}
Here $p_{\mM} \colon S \times \mM_S \to \mM_S$ is the projection. 
By the above description of the perfect obstruction
theory, the derived moduli stack $\mathfrak{M}_S$ is quasi-smooth.

Let $N(S)$ be the numerical Grothendieck group of 
$S$
\begin{align}\notag
N(S) \cneq K(S)/\equiv
\end{align}
where $F_1, F_2 \in K(S)$
satisfies $F_1 \equiv F_2$ if $\ch(F_1)=\ch(F_2)$. 
Then $N(S)$ is a finitely generated free abelian 
group. 
We have the decompositions into open and closed substacks
\begin{align*}
\mathfrak{M}_S=\coprod_{v \in N(S)} \mathfrak{M}_S(v), \ 
\mM_S=\coprod_{v \in N(S)} \mM_S(v)
\end{align*}
where each component 
corresponds to sheaves $F$
with $[F]=v$. 

Note that the automorphism 
group of a sheaf $F$ on $S$
contains a one dimensional torus 
$\mathbb{C}^{\ast} \subset \Aut(F)$
given by the scalar multiplication. 
Therefore 
the inertia stack $I_{\fM_S}$ of $\fM_S$
admits an embedding 
$(\mathbb{C}^{\ast})_{\fM_S} \subset I_{\fM_S}$. 
As in Subsection~\ref{subsec:rig}, 
we have the 
$\mathbb{C}^{\ast}$-rigidification 
\begin{align}\label{C:rig}
\mathfrak{M}_S(v) \to 
\mathfrak{M}_S^{\mathbb{C}^{\ast}\mathchar`-\mathrm{rig}}(v).
\end{align}

\subsubsection{Moduli stacks of compactly supported sheaves on local 
surfaces}
For a smooth projective surface $S$, 
we consider its total space of the canonical line bundle: 
\begin{align*}
X=\mathrm{Tot}_S(\omega_S) \stackrel{\pi}{\to}S.
\end{align*}
Here $\pi$ is the projection. 
We denote by 
$\Coh_{\rm{cpt}}(X) \subset \Coh(X)$
the subcategory of compactly supported coherent 
sheaves on $X$. 
We consider the 
classical Artin stack
\begin{align*}
\mM_X : 
Aff^{\rm{op}} \to Groupoid
\end{align*}
whose $T$-valued points for $T \in Aff$
form the groupoid of $T$-flat families of 
objects in 
$\Coh_{\rm{cpt}}(X)$. 
We have the decomposition into open and closed 
substacks
\begin{align*}
\mM_X=\coprod_{v\in N(S)} \mM_X(v)
\end{align*}
where $\mM_X(v)$ corresponds to compactly 
supported
sheaves $F$ on $X$ with $[\pi_{\ast}F]=v$. 
By pushing forward to $S$, we have the natural 
morphism 
\begin{align}\label{pi:ast}
\pi_{\ast} \colon \mM_X(v) \to \mM_S(v), \ 
F \mapsto \pi_{\ast}F. 
\end{align}
 The following lemma is well-known (cf.~\cite[Lemma~5.4]{MR3842061}). 
\begin{lem}\label{lem:moduliX}
We have an isomorphism of stacks over $\mM_S(v)$
\begin{align}\label{isom:etaM}
\eta \colon 
\mM_X(v) \stackrel{\cong}{\to} t_0(\Omega_{\mathfrak{M}_S(v)}[-1])
=\mathrm{Obs}^{\ast}(\eE^{\bullet}|_{\mM_S(v)}). 
\end{align}
\end{lem}
\begin{proof}
The lemma is easily proved 
by noting that the fiber of (\ref{pi:ast}) 
at $[F] \in \mM_S(v)$ is the space of 
$\oO_X$-module structures on $F$, 
that is 
\begin{align*}
(\pi_{\ast})^{-1}([F])=
\Hom(F, F \otimes \omega_S)=\Ext^2(F, F)^{\vee}=
\Spec S(\hH^1(\eE^{\bullet \vee}|_{[F]})).
\end{align*}
\end{proof}

Similarly to (\ref{C:rig}), 
we have the $\mathbb{C}^{\ast}$-rigidification 
\begin{align*}
\mM_X(v) \to \mM_X(v)^{\C\rig}.
\end{align*}
An argument similar to Lemma~\ref{lem:moduliX}
shows the isomorphism
\begin{align}\label{isom:rigid:S}
\mM_X(v)^{\mathbb{C}^{\ast}\mathchar`-\mathrm{rig}} \stackrel{\cong}{\to}
t_0
(\Omega_{\mathfrak{M}_S(v)^{\C\rig}}[-1]). 
\end{align}

\subsubsection{Definition of DT category for local surfaces}\label{subsec:catDT:sigma}
Let $A(S)_{\mathbb{C}}$ be the ample cone
\begin{align*}
	A(S)_{\mathbb{C}} \cneq \{B+iH : H \mbox{ is ample}\}
	\subset \mathrm{NS}(S)_{\mathbb{C}}. 
	\end{align*}
For $\sigma=B+iH \in A(S)_{\mathbb{C}}$ and 
 an object $F \in \Coh_{\rm{cpt}}(X)$, 
its $B$-twisted reduced Hilbert polynomial is defined by
\begin{align}\label{reduced:Hilb}
\overline{\chi}_{\sigma}(F, m) \cneq 
\frac{\chi(\pi_{\ast}F \otimes \oO_S(-B+mH))}{c} \in \mathbb{Q}[m]. 
\end{align}
Here $c$ is the coefficient of the highest degree term of 
the polynomial $\chi(\pi_{\ast}F \otimes \oO_S(-B+mH))$ in $m$. 
\begin{rmk}
Note that the notation $\oO_S(-B+mH)$ makes sense only when 
$\sigma$ is an integral class. But the right hand side of (\ref{reduced:Hilb})
is determined by the numerical class of $\sigma$, so 
it also makes sense for any $\sigma$. 
\end{rmk}
The following definition gives the notion of 
$B$-twisted $H$-Gieseker stability
for compactly supported sheaves on $X$. 
\begin{defi}\label{def:sigmastab}
An object $F \in \Coh_{\rm{cpt}}(X)$ is 
$\sigma$-(semi)stable if it is a pure dimensional 
sheaf, and
for any subsheaf $0 \neq F' \subsetneq F$ we have
$\overline{\chi}_{\sigma}(F', m) <
 (\le) \overline{\chi}_{\sigma}(F, m)$
for $m\gg 0$. 
\end{defi}

We have the open substacks
\begin{align}\notag
\mM_{X}^{\sigma \st}(v)
\subset 
\mM_{X}^{\sigma}(v)
\subset \mM_X(v)
\end{align}
corresponding to $\sigma$-stable sheaves, $\sigma$-semistable sheaves, 
respectively. 
By the GIT construction of 
the moduli stack $\mM_X^{\sigma}(v)$, it admits a good moduli 
space
\begin{align}\label{gmoduli:M}
	\mM_X^{\sigma}(v) \to M_X^{\sigma}(v)
	\end{align}
where $M_X^{\sigma}(v)$ is a quasi-projective 
scheme whose closed points correspond to $\sigma$-polystable objects, i.e. 
direct sum of $\sigma$-stable objects with the same $B$-twisted
reduced Hilbert polynomials. 
Suppose that 
the following condition holds: 
\begin{align}\label{cond:ss=st}
\mM_{X}^{\sigma \st}(v)
=
\mM_{X}^{\sigma}(v).
\end{align}
In this case, 
the morphism (\ref{gmoduli:M}) is a $\C$-gerbe
so that 
\begin{align*}
M_{X}^{\sigma}(v) =
\mM_{X}^{\sigma}(v)^{\mathbb{C}^{\ast} \rig}
\end{align*}
holds. 
Note that it 
admits a  
$\mathbb{C}^{\ast}$-action induced by the fiberwise
weight two  
$\mathbb{C}^{\ast}$-action on $\pi \colon X \to S$. 

The moduli stack of $\sigma$-semistable sheaves 
$\mM_X^{\sigma}(v)$ is of finite 
type, while $\mathfrak{M}_S(v)$ is not quasi-compact in general, so 
in particular it is not QCA. 
So we take a quasi-compact open derived substack
$\mathfrak{M}_{S}(v)_{\rm{qc}} \subset \mathfrak{M}_S(v)$
satisfying the condition  
\begin{align}\label{take:open}
\pi_{\ast}(\mM_{X}^{\sigma}(v)) \subset t_0(\mathfrak{M}_{S}(v)_{\rm{qc}}). 
\end{align}
Here $\pi_{\ast}$ is the morphism (\ref{pi:ast}). 
Note that $\fM_S(v)_{\rm{qc}}$ is QCA. 
By the isomorphism $\eta$ in Lemma~\ref{lem:moduliX}, 
 we have the conical closed substack
\begin{align*}
\zZ^{\sigma\mathchar`-\rm{us}}
\cneq t_0(\Omega_{\mathfrak{M}_S(v)_{\rm{qc}}}[-1])
\setminus \eta(\mM_{X}^{\sigma}(v)) \subset t_0(\Omega_{\mathfrak{M}_S(v)_{\rm{qc}}}[-1]).
\end{align*}
By taking the $\mathbb{C}^{\ast}$-rigidification
and  
using the isomorphism (\ref{isom:rigid:S}), 
we also have the conical closed substack
\begin{align*}
(\zZ^{\sigma \us})^{\mathbb{C}^{\ast} \rig}
\subset t_0(\Omega_{\mathfrak{M}_S(v)^{\C\rig}_{\rm{qc}}}[-1]).
\end{align*}
By Definition~\ref{defi:catDT}, 
the DT category for 
$M_{X}^{\sigma}(v)$ is defined as follows: 
\begin{defi}\label{catDT:stable}
The $\C$-equivariant DT category for $\mM_X^{\sigma}(v)$
is defined as 
\begin{align*}
\dDT^{\C}(\mM_{X}^{\sigma}(v)) \cneq 
\Dbc(\fM_S(v)_{\rm{qc}})/\cC_{\zZ^{\sigma \us}}. 
	\end{align*}
When the condition (\ref{cond:ss=st})
holds, the $\mathbb{C}^{\ast}$-equivariant DT category for 
the moduli space $M_{X}^{\sigma}(v)$ 
is defined by 
\begin{align*}
&\mathcal{DT}^{\mathbb{C}^{\ast}}(M_{X}^{\sigma}(v)) \cneq 
D_{\rm{coh}}^b
(\mathfrak{M}_S(v)^{\C\rig}_{\rm{qc}})
/\cC_{(\zZ^{\sigma \us})^{\mathbb{C}^{\ast} \rig}}.
\end{align*}
The dg-enhancements, $\lambda$-twisted version, $\wdDT$-version and
ind-version  
are similarly defined following Subsections~\ref{subsub:defDT}, \ref{subsub:DThat}, \ref{subsec:rig}. 
\end{defi}
\begin{rmk}\label{rmk:open:independent}
By Lemma~\ref{lem:replace0}, 
the DT categories in Definition~\ref{catDT:stable}
are independent of 
a choice of a quasi-compact open substack 
$\mathfrak{M}_S(v)_{\rm{qc}}$ 
of $\mathfrak{M}_S(v)$ 
satisfying (\ref{take:open}), up to equivalence. 
\end{rmk}

\begin{rmk}\label{rmk:DTtwist}
The DT category 
$\dDT^{\C}(\mM_X^{\sigma}(v))$ decomposes into the direct
sum 
\begin{align}\notag
	\dDT^{\C}(\mM_X^{\sigma}(v))=
	\bigoplus_{\lambda \in \mathbb{Z}} 
	\dDT^{\C}(\mM_X^{\sigma}(v))_{\lambda}
\end{align}
where each summand 
corresponds to $\lambda$-twisted version (see Subsection~\ref{subsec:rig}). 
\end{rmk}
Note that we have the fully-faithful functor
\begin{align*}
i_{\ast} \colon \Coh(S) \to \Coh_{\rm{cpt}}(X)
\end{align*}
where $i$ is the zero section of 
$\pi \colon X \to S$. 
The image of the above functor is 
closed under subobjects and quotients. 
Therefore an object $F \in \Coh(S)$ is 
$\sigma$-(semi)stable if and only if $i_{\ast}F$
is $\sigma$-(semi)stable 
in Definition~\ref{def:sigmastab}. 
Let us also take open substacks 
\begin{align}\label{dstack:surface}
\mM_S^{\sigma \st}(v) \subset \mM_{S}^{\sigma}(v) \subset \mM_S(v), \ 
 \fM_S^{\sigma \st}(v) 
 \subset \mathfrak{M}_{S}^{\sigma}(v) \subset \mathfrak{M}_S(v)
 \end{align}
consisting of $\sigma$-(semi)stable sheaves on $S$.  
\begin{lem}\label{lem:open:Sss}
For the map $\pi_{\ast}$ in (\ref{pi:ast}), we have 
\begin{align}\label{open:Sss}
	(\pi_{\ast})^{-1}( \mM_{S}^{\sigma \st}(v)) \subset 
	\mM_{X}^{\sigma \st}(v), \ 
(\pi_{\ast})^{-1}( \mM_{S}^{\sigma}(v)) \subset 
 \mM_{X}^{\sigma}(v). 
\end{align}
\end{lem}
\begin{proof}
For a compactly supported coherent sheaf $E$ on $X$, 
suppose that it is not $\sigma$-(semi)stable. 
Then a destabilizing subsheaf $F \subset E$
gives a subsheaf $\pi_{\ast}F \subset \pi_{\ast}E$
as $\pi$ is affine, which destabilizes
$\pi_{\ast}E$. 
\end{proof}

\begin{lem}\label{lem:ss:dt}
Suppose that (\ref{cond:ss=st}) holds,  
and the left open immersion (\ref{open:Sss}) is an isomorphism. 
Then we have the equivalence
\begin{align*}
\mathcal{DT}^{\mathbb{C}^{\ast}}(M_{X}^{\sigma}(v)) \stackrel{\sim}{\to}
D^{b}_{\rm{coh}}(\mathfrak{M}_{S}^{\sigma}(v)^{\mathbb{C}^{\ast}\rig})
\end{align*}
and the following identity
\begin{align}\notag
\DT(M_{X}^{\sigma}(v))=(-1)^{\vdim \mM_S(v)+1} \chi(\mathrm{HP}_{\ast}(\mathcal{DT}_{\rm{dg}}^{\mathbb{C}^{\ast}}
(M_{X}^{\sigma}(v)))). 
\end{align}
\end{lem}
\begin{proof}
The first statement follows from Lemma~\ref{lem:replace0}.
The second statement follows from Proposition~\ref{prop:num}
since 
 $\mM_{S}^{\sigma \st}(v)^{\C\rig}$
 is a quasi-projective scheme (see Remark~\ref{rmk:condition}). 
\end{proof}

\subsubsection{Categorical DT theory for one dimensional stable sheaves}\label{subsec:catDT:one}
Here we focus on the case of 
moduli spaces of one dimensional stable sheaves on $X$. 
We denote by 
\begin{align*}
\Coh_{\le 1}(X) \subset \Coh_{\rm{cpt}}(X)
\end{align*}
the  
abelian subcategory consisting 
of sheaves $F$ with $\dim \Supp(F) \le 1$. 
We also denote by $N_{\le 1}(S)$ the subgroup of $N(S)$
spanned by 
sheaves $F \in \Coh_{\le 1}(S)$. 
Note that we have an isomorphism
\begin{align}\notag
N_{\le 1}(S) \stackrel{\cong}{\to} 
\mathrm{NS}(S) \oplus \mathbb{Z}, \ 
F \mapsto (l(F), \chi(F))
\end{align}
where $l(F)$ is the fundamental one cycle of $F$. 
Below we identify an element $v \in N_{\le 1}(S)$
with $(\beta, n) \in \mathrm{NS}(S) \oplus \mathbb{Z}$
by the above isomorphism. 
We often write $[F]=(\beta, n)$ for 
$F \in \Coh_{\le 1}(S)$ by the above isomorphism. 

Under the above notation, we can interpret
the $\sigma$-(semi)stability on $\Coh_{\le 1}(S)$ 
for $\sigma \in A(S)_{\mathbb{C}}$ in terms 
of Bridgeland stability conditions~\cite{MR2373143}. 
Namely for $F \in \Coh_{\le 1}(X)$, it is 
$\sigma$-(semi)stable for $\sigma=B+iH$
if and only if it is 
Bridgeland (semi)stable with respect to the central 
charge
\begin{align*}
Z_{\sigma} \colon N_{\le 1}(S) \to \mathbb{C}, \ 
(\beta, n) \mapsto -n+(B+iH)\beta,
\end{align*}
i.e. for any subsheaf $0 \neq F' \subsetneq F$ we have 
$\arg Z_{\sigma}(F') <(\le) \arg Z_{\sigma}(F)$, equivalently 
$\mu_{\sigma}(F')<(\le) \mu_{\sigma}(F)$
where $\mu_{\sigma}(F)$ is 
\begin{align}\label{def:slope}
\mu_{\sigma}(F) \cneq -\frac{\Ree Z_{\sigma}(\pi_{\ast}F)}
{\Imm Z_{\sigma}(\pi_{\ast}F)}=\frac{n-B \cdot \beta}{H \cdot \beta},
\end{align}
where $[\pi_{\ast}F]=(\beta, n)$. 
We give a description of 
$\mathcal{DT}^{\mathbb{C}^{\ast}}(M_{X}^{\sigma}(v))$
for $v \in N_{\le 1}(S)$ in a simple example. 

\begin{exam}\label{exam:-3}
Suppose that $C \subset S$ is a $(-3)$-curve, i.e. 
\begin{align*}
C \cong \mathbb{P}^1, \ C^2=-3, \ 
\omega_S|_{C}\cong \oO_C(1).
\end{align*}  
Let us take $v=(2[C], 1) \in N_{\le 1}(S)$
and $\sigma=iH$ for an ample divisor 
$H$ on $S$. 
Then one can 
easily show that 
any compactly supported 
$\sigma$-stable 
coherent sheaf $E$ on $X$
with $[\pi_{\ast}E]=v$
satisfies $\pi_{\ast}E \cong F$, 
where $F \cneq \oO_C \oplus \oO_C(-1)$. 
Let us set
\begin{align*}
V_F \cneq \Hom(F, F \otimes \omega_S), \ 
G_F \cneq \Aut(F). 
\end{align*}
Then we have
an open immersion $\mM_{X}^{\sigma}(v) \subset [V_F/G_F]$, 
so we can take $\mathfrak{M}_{S}(v)^{\C\rig}_{\rm{qc}}$
to be 
\begin{align*}
\mathfrak{M}_{S}(v)^{\C\rig}_{\rm{qc}}
=[\Spec S(V_F[1])/\mathbb{P}G_F]. 
\end{align*}
Here $\mathbb{P}G_F=G_F/\mathbb{C}^{\ast}$, 
where $\mathbb{C}^{\ast}$ acts on $F$ by scalar
multiplication. 
Note that we have 
\begin{align}\label{describe:VF}
V_F=\left( \begin{array}{cc}
H^0(\oO_C(1)) & H^0(\oO_C) \\
H^0(\oO_C(2)) & H^0(\oO_C(1))
\end{array}
\right)
=\left( \begin{array}{cc}
\mathbb{C}^2 & \mathbb{C} \\
\mathbb{C}^3 & \mathbb{C}^2
\end{array}
\right). 
\end{align}
Let $V_F' \subset V_F$ be the codimension one 
linear subspace given by the zero locus of 
the top right component of (\ref{describe:VF}).
Then one can show that
\begin{align*}
(\zZ^{\sigma\us})^{\mathbb{C}^{\ast} \rig}
=[V_F'/\mathbb{P}G_F] 
= [V_F/\mathbb{P}G_F] \setminus \mM_{X}^{\sigma}(v). 
\end{align*}
We have $V_F=V_F' \oplus V_F''$ for 
the one dimensional subspace $V_F'' \subset V_F$
generated by the top right component of (\ref{describe:VF}), 
so 
\begin{align*}
\Spec S(V_F[1])=\mathfrak{U}' \times \mathfrak{U}'', \ 
\mathfrak{U}'=\Spec S(V_F'[1]), \ 
\mathfrak{U}''=\Spec S(V_F''[1]).
\end{align*}
Then from Example~\ref{exam:simple} (i), 
the subcategory 
\begin{align*}
\cC_{(\zZ^{\sigma\us})^{\mathbb{C}^{\ast} \rig}}
\subset D^b_{\rm{coh}}(\mathfrak{M}_{S}(v)^{\C\rig}_{\rm{qc}})
\end{align*}
may be
the subcategory of $\mathbb{P}G_F$-equivariant 
coherent sheaves on $\Spec S(V_F[1])$, 
which lie on the thick closure of 
$D^b_{\rm{coh}}(\mathfrak{U}') \boxtimes \mathrm{Perf}(\mathfrak{U}'')$. 
By taking the quotients, 
the DT category may be heuristically described as
\begin{align*}
\mathcal{DT}^{\mathbb{C}^{\ast}}(M_{X}^{\sigma}(v))=
\left(D^b_{\rm{coh}}(\mathfrak{U}') \otimes D_{\rm{sg}}(\mathfrak{U}'')
\right)^{\mathbb{P}G_F}. 
\end{align*}
Here $(-)^{\mathbb{P}G_F}$ refers to the $\mathbb{P}G_F$-equivariant category. 
\end{exam}

We formulate
wall-crossing formula of DT categories
of one dimensional stable sheaves. 
We fix a primitive element 
$v=(\beta, n) \in N_{\le 1}(S)$
such that $\beta>0$. 
Here we write $\beta>0$ if $\beta=[C]$ for a non-zero 
effective divisor $C$ on $S$. 
For each decomposition 
\begin{align*}
	v=v_1+v_2, \ 
	v_i=(\beta_i, n_i), \
	\beta_i>0
\end{align*}
we define 
\begin{align*}
	W_{v_1, v_2}& \cneq \{\sigma \in A(S)_{\mathbb{C}} : 
	\mu_{\sigma}(v_1)=\mu_{\sigma}(v_2)\} \\
	&=\{ B+iH \in A(S)_{\mathbb{C}} : 
	(n_1 \beta_2-n_2 \beta_1) \cdot H=B\beta_1 \cdot H\beta_2
	-B\beta_2 \cdot H\beta_1\}. 
\end{align*}
Since $v$ is primitive, 
$W_{v_1, v_2}\subsetneq A(S)_{\mathbb{C}}$ and 
$W_{v_1, v_2}$ is a real codimension one 
hypersurface in $A(S)_{\mathbb{C}}$. 
For a fixed $v$, the set of 
hypersurfaces $W_{v_1, v_2}$ are called 
\textit{walls}. 
It is easy to see that the walls are locally 
finite. Also each connected component 
\begin{align*}
	\cC \subset A(S)_{\mathbb{C}} \setminus \bigcup_{v_1+v_2=v}
	W_{v_1, v_2}
\end{align*}
is called a \textit{chamber}. 
From the construction of walls, the moduli stacks 
$\mM_{S}^{\sigma}(v)$, $\mM_{X}^{\sigma}(v)$ are constant 
if $\sigma$ is contained in a chamber, but may change 
when $\sigma$ crosses a wall. 
Moreover if $\sigma$ lies in a chamber, 
they consist of $\sigma$-stable sheaves, i.e. 
the condition (\ref{cond:ss=st}) holds
for both of $X$ and $S$. 

Let us take 
$\sigma \in A(S)_{\mathbb{C}}$ which lies on a wall 
and take $\sigma_{\pm} \in A(S)_{\mathbb{C}}$ which 
lie on adjacent chambers. 
Then by~\cite[Theorem~8.3]{Toddbir}, 
the wall-crossing diagram
\begin{align*}
\xymatrix{
M_X^{\sigma_+}(v) \ar[rd] & & \ar[ld] M_X^{\sigma_-}(v) \\
& M_X^{\sigma}(v) &
}
	\end{align*}
is a \textit{d-critical flop}, which 
is a d-critical analogue of usual flops in birational geometry.  
Therefore 
following the discussion in Subsection~\ref{intro:motivation}, 
we propose the following conjecture: 
\begin{conj}\label{conj1}
There exists an equivalence of $\mathbb{C}^{\ast}$-equivariant
DT categories:
\begin{align}\notag
\dD\tT^{\mathbb{C}^{\ast}}(M_{X}^{\sigma_+}(v)) \stackrel{\sim}{\to}
\dD\tT^{\mathbb{C}^{\ast}}(M_{X}^{\sigma_-}(v)).
\end{align}
In particular $\dDT^{\C}(M_X^{\sigma}(v))$ is independent of a choice of 
generic $\sigma$ up to equivalence. 
\end{conj}

\begin{rmk}
The numerical DT invariants counting one dimensional 
(semi)stable sheaves are known as genus zero 
Gopakumar-Vafa invariants. It is well-known that 
they are independent of a choice of a 
stability condition (see~\cite[Theorem~6.16]{JS}). 
By Lemma~\ref{lem:ss:dt}, the equivalence in 
Conjecture~\ref{conj1}
recovers this fact under the assumption
in Proposition~\ref{prop:num}.  
\end{rmk}

Let us take $v=(\beta, n) \in N_{\le 1}(S)$
such that $\beta$ is irreducible (see Subsection~\ref{subsec:notation}). 
Following the notation of~\cite{MR2552254}, we
denote the open substacks
of pure one dimensional sheaves by 
\begin{align}\label{M:pureone}
\mM_n(X, \beta) \subset \mM_X(v), \ 
M_n(X, \beta) \subset \mM_X(v)^{\C\rig},
\ \mathfrak{M}_n(S, \beta) \subset \mathfrak{M}_S(v),
\end{align}
respectively. 
As $\beta$ is irreducible, the
stacks (\ref{M:pureone})
coincide with the substacks of $\sigma$-stable sheaves
on $X$, $S$
respectively, for any choice of $\sigma$. 
We have the following: 
\begin{lem}\label{lem:M:irred}
For $v=(\beta, n)\in N_{\le 1}(S)$ such that $\beta$ is irreducible, 
we have 
the equivalences
\begin{align*}
\mathcal{DT}^{\mathbb{C}^{\ast}}(\mM_n(X, \beta))_{\lambda}
\stackrel{\sim}{\to}D^b_{\rm{coh}}(\mathfrak{M}_n(S, \beta))_{\lambda}, \ 
\mathcal{DT}^{\mathbb{C}^{\ast}}(M_n(X, \beta))
\stackrel{\sim}{\to} D^b_{\rm{coh}}(\mathfrak{M}_n(S, \beta)^{\C\rig}). 
\end{align*}
Moreover we have the identity
\begin{align}\label{id:M:HP}
\chi(\mathrm{HP}_{\ast}(\mathcal{DT}_{\rm{dg}}^{\mathbb{C}^{\ast}}(M_n(X, \beta)))
=(-1)^{\beta^2+1}N_{n, \beta}. 
\end{align}
Here $N_{n, \beta}=\DT(M_n(X, \beta))$ is the DT invariant 
counting one dimensional stable sheaves on $X$. 
\end{lem}
\begin{proof}
By the condition that $\beta$ is irreducible, the 
assumption of Lemma~\ref{lem:ss:dt} is satisfied (see~\cite[Lemma~5.4]{MR3842061}). 
Moreover 
by the Riemann-Roch theorem, we have  
\begin{align*}
\vdim \mM_S(v)=-\chi(F, F)=\beta^2
\end{align*}
for a one dimensional sheaf
$F$ on $S$ with $[F]=(\beta, n)$.
Therefore the lemma holds from Lemma~\ref{lem:ss:dt}. 
\end{proof}

Later we will also need the following lemma: 
\begin{lem}\label{lem:periodic}
For an irreducible class $\beta$
and $n \in \mathbb{Z}$, 
we have the following: 
\begin{enumerate}
\item By setting 
$d=\mathrm{g.c.d.}\{ \lvert \beta \cdot D \rvert : 
D \in \mathrm{NS}(S), \beta \cdot D \neq 0 \}$, 
there is an equivalence
\begin{align}\label{equiv:one:1}
\mathcal{DT}^{\mathbb{C}^{\ast}}(\mM_n(X, \beta))_{\lambda}
\stackrel{\sim}{\to}
\mathcal{DT}^{\mathbb{C}^{\ast}}(\mM_{n+d}(X, \beta))_{\lambda}.
\end{align}
\item By setting $c=\mathrm{g.c.d.}(d, n)$, 
there is an equivalence
\begin{align}\label{equiv:one:2}
\mathcal{DT}^{\mathbb{C}^{\ast}}(\mM_n(X, \beta))_{\lambda}
\stackrel{\sim}{\to}
\mathcal{DT}^{\mathbb{C}^{\ast}}(\mM_{n}(X, \beta))_{\lambda+c}.
\end{align}
\item There is an equivalence
\begin{align}\label{equiv:one:3}
\mathcal{DT}^{\mathbb{C}^{\ast}}(\mM_n(X, \beta))_{\lambda}
\stackrel{\sim}{\to}
\mathcal{DT}^{\mathbb{C}^{\ast}}(\mM_{-n}(X, \beta))_{-\lambda}. 
\end{align}
\end{enumerate}
\end{lem}
\begin{proof}
(i) 
Let us take $L \in \Pic(S)$ such that 
$\beta \cdot c_1(L)=d$.
Then the equivalence (\ref{equiv:one:1})
follows from the equivalence
of derived stacks
$\otimes L  \colon 
\mathfrak{M}_n(S, \beta)\stackrel{\sim}{\to} \mathfrak{M}_{n+d}(S, \beta)$.

(ii) 
Let $L \in \Pic(S)$ be as above, 
$\mathfrak{F}$ the universal family on 
$S \times \mathfrak{M}_n(S, \beta)$, 
and take 
\begin{align*}
\det p_{\mathfrak{M}\ast}(L^{\otimes k} \boxtimes \mathfrak{F})
\in \Pic(\mathfrak{M}_n(S, \beta)). 
\end{align*}
Here $p_{\mathfrak{M}}$ is the projection onto $\mathfrak{M}_n(S, \beta)$. 
The inertia $\mathbb{C}^{\ast}$-weight of the above line bundle is 
$kd+n$. 
Therefore there is a line bundle 
$\lL$ on $\mathfrak{M}_n(S, \beta)$ whose 
inertia $\mathbb{C}^{\ast}$-weight is $c$. 
Indeed by writing $c=ad+bn$ for $a, b \in \mathbb{Z}$, 
we can take $\lL$ to be
\begin{align*}
	\lL=(\det p_{\mathfrak{M}\ast}(L \boxtimes \mathfrak{F}))^{\otimes a}
	\otimes (\det p_{\mathfrak{M}\ast}(\mathfrak{F}))^{\otimes b-a}.
	\end{align*}
Then the equivalence (\ref{equiv:one:2}) is 
obtained by taking the tensor product with $\lL$. 

(iii) The equivalence (\ref{equiv:one:3}) follows 
from the equivalence of derived stacks
$\mathfrak{M}_n(S, \beta) \stackrel{\sim}{\to}
\mathfrak{M}_{-n}(S, \beta)$
given by $F \mapsto \eE xt^1(F, \omega_S)$
(see Subsection~\ref{subsec:dualpair}).
\end{proof}

\section{Categorical DT theory for D0-D2-D6 bound states}\label{sec:catMNOP}
The purpose of this section is to 
construct DT categories for stable 
D0-D2-D6 bound states on 
local surfaces. 
For an arbitrary CY 3-fold $X$, 
the category of D0-D2-D6 bound states is defined by 
\begin{align*}
	\aA_X\cneq \langle \oO_{X}, \Coh_{\le 1}(X)[-1] \rangle_{\rm{ex}},
	\end{align*}
which is known to be an abelian category. 
The above category was first introduced in~\cite{Tcurve1}
in order to 
prove MNOP/PT correspondence conjecture. 
The category $\aA_X$ contains several interesting 
geometric objects which have been studied in the context of 
curve counting DT theory, e.g. ideal sheaves
of one or zero dimensional closed subschemes, 
Pandharipande-Thomas stable pairs, etc. 

Indeed there is a stability parameter $t \in \mathbb{R}$
which defines $\mu_t^{\dag}$-stability on $\aA_X$, 
and the above MNOP/PT moduli spaces are realized as 
certain stable objects with respect to $\mu_t^{\dag}$-stability.
The wall-crossing formula of the associated DT invariants 
was used in~\cite{Tolim2, Tsurvey} to prove the rationality conjecture 
of the generating series of PT invariants. 
For each $t \in \mathbb{R}$ and a numerical class
$(\beta, n)$, where $\beta$ is a curve class and $n \in \mathbb{Z}$, 
we have the moduli space of $\mu_t^{\dag}$-stable objects in $\aA_X$
with numerical class $(\beta, n)$, denoted as $P_n^t(X, \beta)$. 
We will introduce the corresponding DT category for the local surface $X=\mathrm{Tot}_S(\omega_S)$, denoted by
\begin{align*}
	\dDT^{\C}(P_n^t(X, \beta)). 
	\end{align*}

The key ingredient for the construction 
is to show that
the moduli stack of rank one objects in $\aA_X$ is
isomorphic to the dual obstruction cone 
over the moduli stack 
of pairs $(F, \xi)$ for $F \in \Coh_{\le 1}(S)$ and 
$\xi \colon \oO_S \to F$. 
This is not a trivial statement since a priori there is no 
obvious map from the moduli stack of objects in $\aA_X$ to 
the moduli stack of pairs on $S$. 
The main idea is to give an alternative description of 
the category of $\aA_X$ purely in terms of the surface $S$.
More precisely we show that giving a rank one object 
in $\aA_X$ is equivalent to giving a diagram
\begin{align}\label{diagram:BS:over}
	\xymatrix{
		0 \ar[r] & \oO_S \ar[r] \ar@{.>}[rd]& \uU \ar[r] 
		\ar[d] & F \otimes \omega_S^{-1} 
		\ar[r] & 0 \\
		&   & F  & 
	}
\end{align}
where 
$F \in \Coh_{\le 1}(S)$ and the top 
sequence is an exact sequence of 
coherent sheaves on $S$. 
The relevant map
to the moduli stack of pairs on $S$ is given by
sending the above diagram to the dotted arrow. 

We will then formulate categorical wall-crossing for
the above DT categories, based on d-critical analogue of D/K 
equivalence conjecture. Among them, 
for generic $t_1>t_2>0$ we propose the existence of a 
fully-faithful functor
\begin{align*}
	\dDT^{\C}(P_n^{t_2}(X, \beta)) \hookrightarrow 
	\dDT^{\C}(P_n^{t_1}(X, \beta)).
	\end{align*}
In particular there should be a 
fully-faithful functor from PT category to MNOP category. 

The main result in this section is to prove 
the above conjecture when $\beta$ is an irreducible curve class. 
In this case, we can also describe the semiorthogonal complement of the 
above fully-faithful functor, which recovers the wall-crossing formula 
of the PT invariants, and also gives a rationality statement of the 
generating series of PT categories in the Grothendieck group of 
triangulated categories. 
The strategy is to 
interpret the wall-crossing diagram for an irreducible $\beta$ in terms of 
Koszul duality diagram studied by Mirkovi\'{c}-Riche~\cite{MrRi, MrRi2}. 
Indeed we will show a general result concerning 
semiorthogonal decompositions associated with linear Koszul pairs in~\cite{MrRi, MrRi2}, 
and apply it to prove the above 
conjecture for an irreducible $\beta$. 
A more general case of the above conjecture will be studied in 
Section~\ref{sec:window:DT} via window theorem, and in Section~\ref{sec:cat:hall} via 
categorified Hall products. 

The organization of this section is as follows. 
In Section~\ref{subsec:moduli:D026},
we introduce moduli stacks of pairs on surfaces
and moduli stacks of D0-D2-D6 bound states on local surfaces. 
In Section~\ref{subsec:catDTPT}, we introduce 
DT categories for stable D0-D2-D6 bound states, 
and formulate our main conjecture. 
In Section~\ref{subsec:exam:-1-1}, we give an example 
of our conjecture for a local $(-1, -1)$-curve. 
In Section~\ref{subsec:duality}, we prove some duality statement 
of DT categories under change of stability $t \mapsto -t$, 
and use this to prove the main conjecture when the curve class is irreducible. 
In Section~\ref{subsec:cat026}, we study the category $\aA_X$ for the local surface $X$ 
in detail and give its alternative description in terms of diagrams (\ref{diagram:BS:over}). 
In Section~\ref{subsec:sod:koszul}, we prove
the existence of semiorthogonal decompositions 
associated with linear Koszul dual pairs which is used in the proof of
main conjecture for irreducible curve classes.

\subsection{Moduli stacks of D0-D2-D6 bound states}\label{subsec:moduli:D026}
\subsubsection{Moduli stacks of pairs}\label{subsec:pairs}
For a smooth projective surface $S$, let 
$\mathfrak{M}_S$ 
be the 
derived moduli stack of coherent sheaves on $S$
considered in (\ref{dstack:MS}), 
and $\mathfrak{F}$ the universal 
object (\ref{F:universal}). 
We 
define the derived stack $\mathfrak{M}_S^{\dag}$ by 
\begin{align*}
\rho^{\dag} \colon 
\mathfrak{M}_S^{\dag} \cneq 
\Spec_{\mathfrak{M}_S}S((p_{\mathfrak{M}\ast} \mathfrak{F})^{\vee})
\to \mathfrak{M}_S. 
\end{align*}
Here $p_{\mathfrak{M}} \colon S \times \mathfrak{M}_S \to \mathfrak{M}_S$
is the projection. 
Its classical truncation is a 1-stack, 
given by
\begin{align}\label{MS:dag}
\mM_S^{\dag}\cneq t_0(\mathfrak{M}_S^{\dag})
=\Spec_{\mM_S}(S(\hH^0((\dR p_{\mM\ast} \fF)^{\vee})). 
\end{align}
From the above description, 
the $T$-valued points of 
$\mM_S^{\dag}$ 
form the groupoid of pairs 
\begin{align*}
(F_T, \xi), \ 
\oO_{S \times T} \stackrel{\xi}{\to} F_T
\end{align*}
where $F_T \in \Coh(S \times T)$ is flat over $T$
and $\xi$ is a morphism of coherent sheaves. 
The isomorphisms of $\mM_S^{\dag}(T)$
are given by isomorphisms of 
$T$-flat families of coherent sheaves on $S$
which commute with sections. 
In particular, we have the universal pair
on $\mM_S^{\dag}$
\begin{align*}
\iI^{\bullet}=(\oO_{S \times \mM_S^{\dag}} \to \fF).
\end{align*}
The 
obstruction theory on $\mM_S^{\dag}$
induced by the cotangent complex of 
$\mathfrak{M}_S^{\dag}$ is given by 
\begin{align}\label{perf:obs2}
&\eE^{\dag \bullet}=\left(\dR p_{\mM\ast} 
\dR \hH om_{S \times \mM_S^{\dag}}(\iI^{\bullet}, \fF)\right)^{\vee}
\to \tau_{\ge -1} \mathbb{L}_{\mM_S^{\dag}}.
\end{align}
Also we have the decompositions into
open and closed substacks
\begin{align*}
\mathfrak{M}_S^{\dag}=\coprod_{v \in N(S)} \mathfrak{M}_S^{\dag}(v), \ 
\mM_S^{\dag}=\coprod_{v \in N(S)} \mM_S^{\dag}(v),
\end{align*}
where each component corresponds to pairs $(\oO_S \to F)$
such that $[F]=v$. 

\begin{lem}\label{lem:dag:qsmooth}
For $v \in N_{\le 1}(S)$, 
the derived stack $\mathfrak{M}_S^{\dag}(v)$ is quasi-smooth.  
Moreover the fiber of 
\begin{align}\label{fiber:Obs}
t_0(\Omega_{\mathfrak{M}_S^{\dag}(v)}[-1])=\mathrm{Obs}^{\ast}(\eE^{\dag\bullet}) \to \mM_S^{\dag}(v)
\end{align}
at $(\xi \colon \oO_S \to F)$ is given by 
$\Hom(F \otimes \omega_S^{-1}, I^{\bullet}[1])$, 
where $I^{\bullet}$ is the two term complex $(\oO_S \stackrel{\xi}{\to} F)$
such that $\oO_S$ is located in degree zero. 
\end{lem}
\begin{proof}
For the two term complex $I^{\bullet}=(\oO_S \to F)$, 
we have the distinguished triangle
\begin{align}\notag
\dR \Gamma(F) \to \RHom(I^{\bullet}, F)\to \RHom(F, F)[1].
\end{align}
Therefore if $\dim \Supp(F) \le 1$ then 
$\eE^{\dag \bullet}|_{(\oO_S \to F)}$ has 
cohomological amplitude $[-1, 1]$. 
The fiber of the morphism (\ref{fiber:Obs}) is dual 
to the obstruction space $\Hom^1(I^{\bullet}, F)$, 
which is $\Hom(F \otimes \omega_S^{-1}, I^{\bullet}[1])$
by Serre duality
(see Lemma~\ref{lem:cone} for $T$-valued fibers). 
\end{proof}

\subsubsection{Moduli stacks of D0-D2-D6 bound states}
\label{subsec:moduliD026}
Here we
introduce the category of 
D0-D2-D6 bound states on the local surface 
$X=\mathrm{Tot}_S(\omega_S)$, 
and discuss the relationship of their moduli 
spaces with the dual obstruction cone over the 
stack of pairs $\mM_S^{\dag}$. 
Note that we have the compactification of $X$
\begin{align*}
X \subset \overline{X} \cneq \mathbb{P}_S(\omega_S \oplus \oO_S).
\end{align*}

\begin{defi}\label{defi:state}
The category of D0-D2-D6 bound states on the 
non-compact CY 3-fold $X=\mathrm{Tot}_S(\omega_S)$ 
is defined by the extension closure
in $D^b_{\rm{coh}}(\overline{X})$
\begin{align*}
\aA_{X} \cneq 
\langle \oO_{\overline{X}}, \Coh_{\le 1}(X)[-1]
\rangle_{\rm{ex}}. 
\end{align*}
Here we regard $\Coh_{\le 1}(X)$
 as a subcategory of $\Coh(\overline{X})$
by the push-forward of the open 
immersion $X \subset \overline{X}$. 
\end{defi}

The arguments in~\cite[Lemma~3.5, Proposition~3.6]{Tcurve1} 
show
that 
$\aA_{X}$ is the heart of 
a triangulated subcategory 
of $D^b_{\rm{coh}}(\overline{X})$ generated by 
$\oO_{\overline{X}}$ and $\Coh_{\le 1}(X)$. 
In particular, $\aA_{X}$ is an abelian category. 
There is a group homomorphism
\begin{align*}
\cl \colon K(\aA_X) \to \Gamma \cneq \mathbb{Z} \oplus N_{\le 1}(S)
\end{align*}
characterized by the condition 
that $\cl(\oO_X)=(1, 0)$
and $\cl(F[-1])=(0, [\pi_{\ast}F])$
for $F \in \Coh_{\le 1}(X)$. 

Note that 
an object $E \in \aA_X$ is of rank 
one if and only if 
we have 
an isomorphism 
$\dL i_{\infty}^{\ast}E \cong \oO_{S_{\infty}}$, 
where 
$i_{\infty} \colon 
S_{\infty} \hookrightarrow \overline{X}$
is the divisor at the infinity. 
We define the (classical) moduli stack of rank one 
objects in $\aA_X$
to be the 2-functor
\begin{align*}
\mM_X^{\dag} \colon Aff^{op} \to 
Groupoid
\end{align*}
whose $T$-valued points for $T \in Aff$
form the groupoid
of data
\begin{align}\notag
\eE_T \in D^b_{\rm{coh}}(\overline{X} \times T), \ 
\dL(i_{\infty} \times \id_T)^{\ast} \eE_T \stackrel{\cong}{\to}
\oO_{S_{\infty} \times T}
\end{align}
such that for any closed point $x \in T$, 
we have 
\begin{align*}
\eE_x \cneq \dL i_x^{\ast} \eE \in \aA_X, \ 
i_x \colon \overline{X} \times \{x\} \hookrightarrow \overline{X} \times T.
\end{align*} 
The isomorphisms of the groupoid $\mM_X^{\dag}(T)$ are 
given by isomorphisms of objects $\eE_T$ which commute with 
the trivializations at the infinity. 

We have the decomposition of $\mM_X^{\dag}$ 
into open and closed substacks
\begin{align*}
\mM_X^{\dag}=\coprod_{v \in N_{\le 1}(S)} \mM_X^{\dag}(v)
\end{align*}
where $\mM_X^{\dag}(v)$ corresponds to 
$E \in \aA_X$ with $\cl(E)=(1, v)$. 
The following result will be proved in 
Section~\ref{subsec:cat026} (see Corollary~\ref{lem:moduli:isom} and Proposition~\ref{thm:stack}): 
\begin{thm}\label{thm:D026}
(i) 
For $v \in N_{\le 1}(S)$, 
the stack $\mM_X^{\dag}(v)$
is isomorphic to the stack of diagrams of coherent 
sheaves on $S$
\begin{align}\label{intro:dia:BS}
\xymatrix{
0 \ar[r] & \oO_{S}
 \ar[r] \ar@{.>}[rd]_-{\xi} & \uU \ar[r]
\ar[d]^-{\phi} & F \otimes \omega_S^{-1} 
\ar[r] & 0 \\
&   & F  & 
}
\end{align}
where the top sequence is an exact sequence 
and $F$ has numerical class $v$. 
In particular, there  
exists a natural 
morphism 
\begin{align}\label{mor:p0dag}
	\pi_{\ast}^{\dag} \colon \mM_X^{\dag}(v) \to \mM_S^{\dag}(v)
	\end{align}
sending a diagram (\ref{intro:dia:BS}) to the pair 
$(\oO_S \stackrel{\xi}{\to} F)$.

(ii) There is an isomorphism of stacks over $\mM_S^{\dag}(v)$
\begin{align}\label{isom:dag}
\eta^{\dag} \colon 
\mM_X^{\dag}(v) \stackrel{\cong}{\to} 
t_0(\Omega_{\mathfrak{M}_S^{\dag}(v)}[-1]) =
\mathrm{Obs}^{\ast}(\eE^{\dag \bullet}|_{\mM_S^{\dag}(v)}) 
\end{align}
which, over the pair 
$(\oO_S \stackrel{\xi}{\to} F)$, 
sends a diagram (\ref{intro:dia:BS})
to the morphism $F \otimes \omega_S^{-1} \to I^{\bullet}[1]$ given by
(see Lemma~\ref{lem:dag:qsmooth})
\begin{align}\notag
F \otimes \omega_S^{-1}[-1] \stackrel{\sim}{\leftarrow} (\oO_S \to \uU)
\stackrel{(\id, \phi)}{\longrightarrow} (\oO_S \stackrel{\xi}{\to} F)=I^{\bullet}. 
\end{align}
\end{thm}

\begin{rmk}\label{rmk:dia:BS}
Note that for a pair $(\oO_X \to F)$
with $F \in \Coh_{\le 1}(X)$, we have 
the associated object (see Lemma~\ref{lem:pair})
\begin{align}\label{pair:AX}
(\oO_{\overline{X}} \to F) \in \aA_X
\end{align}
where $\oO_{\overline{X}}$ is located in degree zero. 
The map $\pi_{\ast}^{\dag}$ sends such an object
to the associated pair $(\oO_S \to \pi_{\ast}F)$. 
However a rank one object in $\aA_X$ may not 
be necessary written of the form (\ref{pair:AX}). 
As we will see in Lemma~\ref{lem:pair},
 a diagram (\ref{intro:dia:BS}) corresponds to 
an object of the form (\ref{pair:AX}) if and only if 
the top sequence of (\ref{intro:dia:BS}) splits. 
\end{rmk}

\subsection{DT category for D0-D2-D6 bound states}\label{subsec:catDTPT}
\subsubsection{Categorical MNOP/PT theories}
For $(\beta, n) \in N_{\le 1}(S)=\mathrm{NS}(S) \oplus \mathbb{Z}$, 
we denote by 
\begin{align}\notag
I_n(X, \beta)
\end{align}
the moduli space of closed subschemes 
$C \subset X$, where $C$ is compactly supported 
with $\dim C \le 1$ satisfying $[\pi_{\ast}\oO_C]=(\beta, n)$. 
The moduli space $I_n(X, \beta)$ is a quasi-projective scheme, 
and considered by Maulik-Nekrasov-Okounkov-Pandharipande~\cite{MR2264664}
 in their 
formulation of GW/DT correspondence conjecture. 

On the other hand, a PT stable pair consists of a pair~\cite{MR2545686}
\begin{align}\label{PTpair}
(F, s), \  F \in \Coh_{\le 1}(X), \ s \colon \oO_X \to F
\end{align}
such that $F$ is pure one dimensional and $s$ is 
surjective in dimension one. 
For $(\beta, n) \in N_{\le 1}(S)$, 
we denote by 
\begin{align}\notag
P_n(X, \beta)
\end{align}
the moduli space 
of PT stable pairs (\ref{PTpair})
satisfying $[\pi_{\ast}F]=(\beta, n)$. 
The moduli space of stable pairs $P_n(X, \beta)$ is 
known to be a quasi-projective scheme. 

We have the open immersions
\begin{align*}
I_n(X, \beta) \subset \mM_X^{\dag}(\beta, n), \ 
P_n(X, \beta) \subset \mM_X^{\dag}(\beta, n)
\end{align*}
sending a subscheme $C \subset X$, a pair $(F, s)$ to 
two term complexes $(\oO_{\overline{X}} \twoheadrightarrow \oO_C)$, 
$(\oO_{\overline{X}} \stackrel{s}{\to} F)$ respectively. 
They are quasi-projective schemes, so there is a 
quasi-compact derived open substack 
$\mathfrak{M}_S^{\dag}(\beta, n)_{\rm{qc}}$
in $\mathfrak{M}_S^{\dag}(\beta, n)$ 
such that 
\begin{align*}
\pi_{\ast}^{\dag}(I_n(X, \beta)) \subset 
t_0(\mathfrak{M}_S^{\dag}(\beta, n)_{\rm{qc}}), \ 
\pi_{\ast}^{\dag}(P_n(X, \beta)) \subset 
t_0(\mathfrak{M}_S^{\dag}(\beta, n)_{\rm{qc}}). 
\end{align*}
Then by the isomorphism (\ref{isom:dag}), 
we have the conical closed 
substacks in $t_0(\Omega_{\mathfrak{M}_S^{\dag}(\beta, n)_{\rm{qc}}}[-1])$
\begin{align*}
\zZ^{I \us} \cneq  t_0(\Omega_{\mathfrak{M}_S^{\dag}(\beta, n)_{\rm{qc}}}[-1])\setminus \eta^{\dag}(I_n(X, \beta)), \ 
\zZ^{P\us} \cneq
 t_0(\Omega_{\mathfrak{M}_S^{\dag}(\beta, n)_{\rm{qc}}}[-1]) 
\setminus \eta^{\dag}(P_n(X, \beta)). 
\end{align*}
By Definition~\ref{defi:catDT}, 
the corresponding DT categories are 
defined as follows: 
\begin{defi}\label{cat:DTPT}
The $\mathbb{C}^{\ast}$-equivariant
MNOP/PT categories 
are defined  
to be 
\begin{align*}
&\dD \tT^{\mathbb{C}^{\ast}}(I_n(X, \beta)) \cneq 
D^b_{\rm{coh}}(\mathfrak{M}_S^{\dag}(\beta, n)_{\rm{qc}})/ 
\cC_{\zZ^{I\us}}, \\
&\dD \tT^{\mathbb{C}^{\ast}}(P_n(X, \beta)) \cneq 
D^b_{\rm{coh}}(\mathfrak{M}_S^{\dag}(\beta, n)_{\rm{qc}})/ 
\cC_{\zZ^{P\us}}.
\end{align*}
The dg-enhancements, $\wdDT$-version and 
ind-version  
are similarly defined following Subsections~\ref{subsub:defDT}, \ref{subsub:DThat}. 
\end{defi}

Similarly to Remark~\ref{rmk:open:independent}, 
the above definition is independent of a choice of 
$\mathfrak{M}_S^{\dag}(\beta, n)_{\rm{qc}}$. 
As proved in~\cite{Tcurve1}, 
 two moduli space $I_n(X, \beta)$, $P_n(X, \beta)$ 
are related by wall-crossing in $\aA_X$. 
Let $\tT_n(X, \beta)$ be the open substack 
\begin{align}\label{stack:Tn}
	\tT_n(X, \beta) \subset \mM_X^{\dag}(\beta, n)
	\end{align}
corresponding to $\eE \in \aA_X$ such that $\hH^1(\eE)$ is zero 
dimensional.
We have the following lemma proved in~\cite{Tcurve1}.  
\begin{lem}\emph{(\cite[Lemma~3.11 (ii)]{Tcurve1})}\label{lem:MNOP/PTwall}
For a rank one object $\eE \in \aA_X$, 
$\hH^1(\eE)$ is zero dimensional if and only if it is 
isomorphic to a two term complex $(s \colon \oO_{\overline{X}} \to F)$ 
where $F \in \Coh_{\le 1}(X)$ and $s$ is surjective in dimension one. 
\end{lem}
The above lemma implies that 
the 
stack $\tT_n(X, \beta)$
is isomorphic to the 
moduli stack of pairs $(F, s)$
for
$F \in \Coh_{\le 1}(X)$ and $(s \colon \oO_{\overline{X}} \to F)$
such that $s$ is surjective in dimension one, and 
$[\pi_{\ast}F]=(\beta, n)$. 
As discussed in~\cite[Theorem~B.1]{Toddbir}, it 
admits a good moduli space
\begin{align*}
	\pi_{\tT} \colon 
	\tT_n(X, \beta) \to T_n(X, \beta)
	\end{align*}
together with a wall-crossing diagram 
\begin{align*}
	\xymatrix{
	I_n(X, \beta) \ar[rd] & & P_n(X, \beta) \ar[ld] \\
	& T_n(X, \beta) &
} 
	\end{align*}
which is a d-critical flip. 
Therefore 
following the discussion in Subsection~\ref{intro:motivation}, 
we propose the following conjecture:

\begin{conj}\label{conj:DT/PT}
There exists a fully-faithful functor
\begin{align*}
\dD \tT^{\mathbb{C}^{\ast}}(P_n(X, \beta)) \hookrightarrow
\dD \tT^{\mathbb{C}^{\ast}}(I_n(X, \beta)). 
\end{align*}
\end{conj}

\begin{rmk}\label{rmk:non-compact}
We can similarly formulate MNOP/PT categories 
for a non-compact surface $S$, by replacing 
$\mathfrak{M}_S$ by the derived moduli stack 
of compactly supported coherent sheaves on $S$. 
Then we can similarly formulate Conjecture~\ref{conj:DT/PT}
and study it (see Section~\ref{subsec:exam:-1-1}). 
The same also applies to Conjecture~\ref{conj2} below. 
\end{rmk}

\subsubsection{Moduli stacks of stable D0-D2-D6 bound states}\label{subsec:catDT:stableD026}
Let $S$ be a smooth projective surface, and 
$X=\mathrm{Tot}_S(\omega_S)$
as before. 
Following~\cite{Tolim2, Tsurvey, Toddbir}, 
we introduce a one parameter family of 
weak stability conditions on 
the abelian category $\aA_X$ given in Definition~\ref{subsec:moduliD026}. 
Below, we fix an element $\sigma=iH \in A(S)_{\mathbb{C}}$
for an ample divisor $H$ on $S$. 
For each $t \in \mathbb{R}$, we define the map 
\begin{align*}
	\mu_t^{\dag} \colon 
	\Gamma=\mathbb{Z} \oplus \mathrm{NS}(S) \oplus \mathbb{Z}
	\to \mathbb{R} \cup \{\infty\}, \ 
	(r, \beta, n) \mapsto \left\{ \begin{array}{cl}
		t, & r\neq 0, \\
		n/(\beta \cdot H), & r=0.
	\end{array} \right. 
\end{align*}
Note that we have 
\begin{align*}
	\mu_t^{\dag}(0, \beta, n)=\mu_{H}(\beta, n)
	\cneq \mu_{\sigma=iH}(\beta, n)
	\end{align*}
where the latter is defined in (\ref{def:slope}). 
For an object $E \in \aA_X$, we set 
$\mu_t^{\dag}(E) \cneq \mu_t^{\dag}(\cl(E))$. 
\begin{defi}\label{def:stability}
	An object $E \in \aA_X$ is $\mu_t^{\dag}$-(semi)
	stable if for any exact sequence 
	$0 \to E' \to E \to E'' \to 0$ in 
	$\aA_X$ we have the inequality 
	$\mu_t^{\dag}(E')<(\le) \mu_t^{\dag}(E'')$. 
\end{defi}
We have the substack
\begin{align*}
	\pP_n^t(X, \beta) \subset 
 \mM_X^{\dag}(\beta, n)
\end{align*}
corresponding to
$\mu_t^{\dag}$-semistable objects. 
The result of~\cite[Proposition~3.17]{Tolim2} shows that 
the above substack is an open substack 
of finite type. 
Moreover there is a finite set of \textit{walls}
$W \subset \mathbb{Q}$ such that
$\pP_n^t(X, \beta)$ is constant if $t$ lies in a connected 
component of $\mathbb{R} \setminus W$, called \textit{chamber}. 
By~\cite{AHLH}, the moduli stack $\pP_n^t(X, \beta)$ admits a good moduli space
\begin{align*}
	\pi_{\pP} \colon 
	\pP_n^t(X, \beta) \to P_n^t(X, \beta)
\end{align*} 
where closed points of $P_n^t(X, \beta)$ correspond 
to $\mu_t^{\dag}$-polystable objects, i.e. direct sums
of $\mu_t^{\dag}$-stable objects. 
If $t \notin W$, then 
$\pP_n^t(X, \beta)$ consists of $\mu_t^{\dag}$-stable objects
and the good moduli space morphism
is an isomorphism 
\begin{align}\label{isom:chamber}
	\pi_{\pP} \colon 
	\pP_n^t(X, \beta) \stackrel{\cong}{\to}P_n^t(X, \beta), \ 
	t \notin W. 
	\end{align}
\begin{rmk}\label{rmk:rigidify}
	Contrary to the case of moduli stacks of semistable sheaves in (\ref{gmoduli:M}), 
	the trivial $\C$-autmorphisms are rigidified in the definition of the stack 
	$\pP_n^t(X, \beta)$ by the trivialization at $S_{\infty}$. Therefore the 
	morphism (\ref{isom:chamber}) for $t \notin W$ is an isomorphism, not 
	$\C$-gerbe. 
	\end{rmk}
For each $\beta \in \mathrm{NS}(S)$, we set
\begin{align}\label{nbeta}
	n(\beta) \cneq \mathrm{min}\left\{ \chi(\oO_C) : 
	\begin{array}{l}
		C \subset X \mbox{ is a compactly supported } \\
		\mbox{closed subscheme with } \pi_{\ast}[C] \le \beta
		\end{array}  \right\}. 
\end{align}
The following lemma will be used later. 
\begin{lem}\label{lem:bound}
	If $\pP_n^t(X, \beta) \neq \emptyset$ for $t>0$
	then $n \ge n(\beta)$. 
\end{lem}
\begin{proof}
	Let $\eE \in \aA_X$ correspond to a point in $\pP_n^t(X, \beta)$.
	Then there is an exact sequence 
	\begin{align*}
		0 \to \hH^0(\eE) \to \eE \to \hH^1(\eE)[-1] \to 0.
	\end{align*}
Since $\hH^0(\eE)$ is a rank one torsion free sheaf which is trivial 
on $S_{\infty}$, we have 
$\hH^0(\eE)=I_C$ for an ideal sheaf of a compactly supported closed 
	subscheme $C \subset X$. Then as $\pi_{\ast}[C] \le \beta$, 
	we have $\chi(\oO_C) \ge n(\beta)$. 
	Moreover $\hH^1(\eE) \in \Coh_{\le 1}(X)$ 
	and it satisfies
	$\mu_H(\hH^1(\eE)) \ge t>0$
	by the $\mu_t^{\dag}$-stability. 
	Therefore $n=\chi(\oO_C)+\chi(\hH^1(\eE)) \ge n(\beta)$. 
\end{proof}

\begin{rmk}\label{rmk:nbeta}
	We have $n(\beta)>-\infty$ by~\cite[Lemma~3.10]{Tolim}. 
	Moreover if $\beta$ is reduced then 
	\begin{align*}
		n(\beta)=-\frac{1}{2}\beta(\beta+K_S). 
	\end{align*}
	Indeed for a closed subscheme $C \subset X$ with class $\beta$, 
	the morphism $\oO_S \to \pi_{\ast}\oO_C$ is generically surjective. 
	Since any Cohen-Macaulay curve in $S$ with 
	class $\beta$ has Euler characteristic $-\beta(\beta+K_S)/2$, 
	we have $\chi(\oO_C) \ge -\beta(\beta+K_S)/2$. 
\end{rmk}
\subsubsection{DT categories for semistable D0-D2-D6 bound states}
Similarly to Subsection~\ref{subsec:catDTPT}, there is a 
quasi-compact derived open substack 
$\mathfrak{M}_S^{\dag}(\beta, n)_{\rm{qc}}$
in $\mathfrak{M}_S^{\dag}(\beta, n)$ 
such that 
\begin{align}\label{p0(P)}
	\pi_{\ast}^{\dag}(P_n^t(X, \beta)) \subset 
	t_0(\mathfrak{M}_S^{\dag}(\beta, n)_{\rm{qc}}). 
\end{align}
Then by the isomorphism (\ref{isom:dag}), 
we have the conical closed 
substack in $t_0(\Omega_{\mathfrak{M}_S^{\dag}(\beta, n)_{\rm{qc}}}[-1])$
\begin{align*}
	\zZ^{t\us}
	\cneq t_0(\Omega_{\mathfrak{M}_S^{\dag}(\beta, n)_{\rm{qc}}}[-1])
	\setminus 
	\eta^{\dag}(\pP_n^{t}(X, \beta))
	\subset t_0(\Omega_{\mathfrak{M}_S^{\dag}(\beta, n)_{\rm{qc}}}[-1]).
\end{align*}
By Definition~\ref{defi:catDT}, we have the 
following definition of the DT category for $P_n^t(X, \beta)$:

\begin{defi}\label{defi:catDT:A}
	For $(\beta, n) \in N_{\le 1}(S)$, 
we define the $\mathbb{C}^{\ast}$-equivariant 
	DT category for $\pP_n^t(X, \beta)$
by 
	\begin{align*}
		&\mathcal{DT}^{\mathbb{C}^{\ast}}(\pP_n^t(X, \beta)) \cneq 
		D_{\rm{coh}}^b(\mathfrak{M}_S^{\dag}(\beta, n)_{\rm{qc}})/
		\cC_{\zZ^{t\us}}. 
	\end{align*}
If $t \in \mathbb{R}$ lies in a chamber, then we also 
write it as $\dDT^{\C}(P_n^t(X, \beta))$ by the isomorphism (\ref{isom:chamber}). 
The dg-enhancements, $\wdDT$-version and
ind-version  
are similarly defined following Subsections~\ref{subsub:defDT}, \ref{subsub:DThat}, \ref{subsec:rig}. 
\end{defi}

Below we will see the relation of the above
DT categories with PT categories introduced in Definition~\ref{cat:DTPT}. 
We recall that for $\lvert t \rvert \gg 0$, 
the moduli space $P_n^t(X, \beta)$ is related
to the moduli space of PT
stable pairs as follows:  

\begin{thm}\emph{(\cite[Theorem~3.21]{Tolim2})}\label{thm:PT=M}
	For $(\beta, n)\in N_{\le 1}(S)$, we have the following: 
	
	(i) For $t\gg 0$, we have the isomorphism
	\begin{align*}
		P_n(X, \beta) \stackrel{\cong}{\to} P_n^t(X, \beta), \ (F, s) \mapsto 
		(\oO_{\overline{X}} \stackrel{s}{\to} F).
	\end{align*}
	
	(ii) For $t\ll 0$, we have the isomorphism
	\begin{align*}
		P_{-n}(X, \beta) \stackrel{\cong}{\to} P_n^t(X, \beta), \ (F, s) \mapsto \mathbb{D}_{\overline{X}}(\oO_{\overline{X}} \stackrel{s}{\to} F).
	\end{align*}
\end{thm}
By Theorem~\ref{thm:PT=M} (i), 
for $v=(\beta, n)$ 
we have the obvious identity
\begin{align}\label{catDT:tlarge}
	\mathcal{DT}^{\mathbb{C}^{\ast}}(P_n^t(X, \beta))
	=\mathcal{DT}^{\mathbb{C}^{\ast}}(P_n(X, \beta)), \ t\gg 0.  
\end{align}
On the other hand, we also have an equivalence
\begin{align}\label{catDT:tsmall}
	\mathcal{DT}^{\mathbb{C}^{\ast}}(P_n^t(X, \beta))
	\stackrel{\sim}{\to}
	\mathcal{DT}^{\mathbb{C}^{\ast}}(P_{-n}(X, \beta)), \ t\ll 0.  
\end{align}
The equivalence (\ref{catDT:tsmall}) is less obvious 
than (\ref{catDT:tlarge}) since the isomorphism in 
Theorem~\ref{thm:PT=M} (ii) is given through the 
derived dual $\mathbb{D}_{\overline{X}}$. 
Indeed the equivalence (\ref{catDT:tsmall}) is a 
consequence of a more general duality statement 
in Theorem~\ref{thm:duality}. 

\begin{rmk}\label{rmk:infty}
	By formally extending the $\mu_t^{\dag}$-stability in Definition~\ref{def:stability}
	for $t=\infty$ and $t=\infty+0$, we have 
	\begin{align}\notag
		\pP_n^{t=\infty \pm 0}(X, \beta)
		=P_n^{t=\infty \pm 0}(X, \beta)=
		\begin{cases}
			I_n(X, \beta), & t=\infty+0, \\
			P_n(X, \beta), & t=\infty-0. 
		\end{cases}
	\end{align}
	Moreover the moduli stack at the wall $t=\infty$
	coincides with the stack in (\ref{stack:Tn})
	\begin{align*}
		\tT_n(X, \beta) =
		\pP_n^{t=\infty}(X, \beta) \subset \mM_X^{\dag}(\beta, n). 
	\end{align*}
\end{rmk}

\subsubsection{Moduli stacks of semistable pairs}
In some case, the DT category for D0-D2-D6 bound states on 
$X$ is related to the derived category of derived moduli stack 
of some stable pairs on $S$. 
Let $\aA_S$ be the abelian category of pairs on $S$
\begin{align*}
	W \otimes \oO_S \to F
\end{align*}
where $W$ is a finite dimensional vector space and $F \in \Coh_{\le 1}(S)$. The set 
of morphisms is given by commutative diagrams
\begin{align*}
	\xymatrix{
		W \otimes \oO_S \ar[r] \ar[d] & F \ar[d] \\
		W' \otimes \oO_S \ar[r] & F'. 
	}
\end{align*}
Here the left arrow is induced by a linear map $W \to W'$. 
We have the group homomorphism
\begin{align*}
	\cl \colon K(\aA_S) \to \mathbb{Z} \oplus N_{\le 1}(S)
\end{align*}
determined by $\cl(\oO_S \to 0)=(1, 0)$ and 
$\cl(0 \to F)=(0, [F])$. 
For an object 
$E \in \aA_S$, we define 
$\mu_t^{\dag}(E) \cneq \mu_t^{\dag}(\cl(E))$. 
\begin{defi}
	An object $E \in \aA_S$ is $\mu_t^{\dag}$-(semi)stable if for any 
	exact sequence
	$0 \to E' \to E \to E'' \to 0$ in $\aA_S$ we have
	$\mu_t^{\dag}(E')<(\le) \mu_t^{\dag}(E'')$. 
\end{defi}
For each $(\beta, n) \in N_{\le 1}(S)$, we denote by 
\begin{align}\label{dmoduli:PS}
	\mathfrak{P}_n^t(S, \beta) \subset \fM^{\dag}_n(S, \beta), \ 
	\pP_n^t(S, \beta) \subset \mM_n^{\dag}(S, \beta)
\end{align}
the open substacks of $\mu_t^{\dag}$-semistable 
objects $E \in \aA_S$ with $\cl(E)=(1, \beta, n)$. 
By Le Potier’s GIT construction
of moduli spaces of
semistable coherent systems~\cite[Theorem~4.11]{LeP}, 
we have the good moduli space
\begin{align*}
\pi_{\pP} \colon \pP_n^t(S, \beta) \to P_n^t(S, \beta)
\end{align*}
where $P_n^t(S, \beta)$ is a projective scheme
whose closed points correspond to $\mu_t^{\dag}$-polystable 
objects in $\aA_S$. 
Similarly to Subsection~\ref{subsec:catDT:stableD026}, 
there is a wall-chamber structure on
$\mathbb{R}$ such that the above good moduli 
space morphism is an isomorphism if $t$ lies in a chamber. 

\begin{lem}\label{lem:shiftP}
	We have 
	the open immersion 
	\begin{align}\label{Pst:open}
		(\pi_{\ast}^{\dag})^{-1}(\pP_n^t(S, \beta)) \subset \pP_n^t(X, \beta).
		\end{align}
	If the above inclusion is the identity, 
	we have the equivalence
	\begin{align}\label{equiv:Pnt}
		\dDT^{\C}(\pP_n^t(X, \beta))
		\stackrel{\sim}{\to} 
		\Dbc(\fP_n^t(S, \beta)). 
		\end{align}
	\end{lem}
\begin{proof}
	The inclusion (\ref{Pst:open}) is obvious 
	from the correspondence between 
	rank objects in $\aA_X$ and 
	a diagram (\ref{intro:dia:BS}).
	Indeed if there is a destabilizing 
	sequence of a diagram in (\ref{intro:dia:BS}), 
	it also destabilizes the associated pair $(\oO_S \to F)$. 
	Then the equivalence (\ref{equiv:Pnt}) follows from Lemma~\ref{lem:replace0}. 
	\end{proof}

\subsubsection{Conjectural wall-crossing phenomena of DT categories}
\label{subsec:conj:AX}
Below we discuss a 
conjectural wall-crossing phenomena
of DT categories in 
Definition~\ref{defi:catDT:A} under 
change of weak stability conditions. 
We fix $\sigma=iH$
for an ample divisor $H$ and $(\beta, n) \in N_{\le 1}(S)$. 
Let us take $t \in \mathbb{R}$ which lies 
on a wall
and real numbers $t_{\pm}$
on adjacent chambers, 
\begin{align*}
	t \in W \cap \mathbb{Q}_{>0}, \ 
	t_{\pm}\cneq t \pm \varepsilon, \ 0<\varepsilon \ll 1. 
\end{align*}
Since $\pP_n^{t_{\pm}}(X, \beta) \subset \pP_n^t(X, \beta)$, 
we have the induced diagram 
on good moduli spaces
\begin{align}\label{dia:dflip}
	\xymatrix{
		P_n^{t_+}(X, \beta)
		\ar[rd] & & P_n^{t_-}(X, \beta) \ar[ld] \\
		&P_n^{t}(X, \beta). &
	}
\end{align}
By~\cite[Theorem~9.13]{Toddbir}, the above 
diagram is a \textit{d-critical flip}. 
Therefore 
following the discussion in Subsection~\ref{intro:motivation}, 
we propose the following conjecture:

\begin{conj}\label{conj2}
	There exists a fully-faithful functor
	\begin{align*}
		\dD\tT^{\mathbb{C}^{\ast}}(P_n^{t_-}(X, \beta))
		\hookrightarrow 
		\dD\tT^{\mathbb{C}^{\ast}}(P_n^{t_+}(X, \beta)). 
	\end{align*}
\end{conj}

The above conjecture implies that, 
for
$t_1>t_2>\cdots > t_N>0$
which lie on chambers we have a chain of 
fully-faithful functors: 
\begin{align}\label{FF:P1}
	\dD \tT^{\mathbb{C}^{\ast}}(P_n(X, \beta)) \hookleftarrow & \dD\tT^{\mathbb{C}^{\ast}}(P_n^{t_1}(X, \beta)) \hookleftarrow
	\cdots 
	\cdots \hookleftarrow \dD\tT^{\mathbb{C}^{\ast}}(P_n^{t_N}(X, \beta)). 
\end{align}

As for the wall-crossing at $t=0$, 
the diagram (\ref{dia:dflip}) is a d-critical flop
(see~\cite[Corollary~9.18]{Toddbir}), so we conjecture 
the following: 
\begin{conj}\label{conj:L}
	There exists an equivalence 
	\begin{align*}
		\dD \tT^{\mathbb{C}^{\ast}}(P_n^{\varepsilon}(X, \beta)) \stackrel{\sim}{\to} 
		\dD \tT^{\mathbb{C}^{\ast}}(P_n^{-\varepsilon}(X, \beta)), \ 0<\varepsilon \ll 1. 
	\end{align*}
\end{conj}
\subsection{Example: local $(-1, -1)$-curve}\label{subsec:exam:-1-1}
In this section, we prove Conjectures~\ref{conj:DT/PT}, \ref{conj2}
in the case that $S$ is the blow-up of $\mathbb{C}^2$ at 
the origin. 
Although $S$ is non-compact, 
Conjecture~\ref{conj:DT/PT} still makes sense 
(see~Remark~\ref{rmk:non-compact}). 
In this case, $X$ is the total space of $\oO_{\mathbb{P}^1}(-1)^{\oplus2}$, 
and 
we can explicitly describe our derived moduli stacks in terms of 
non-commutative crepant resolution~\cite{MR2057015}. 
In particular we have a global critical locus description of our moduli spaces on 
$X$, and we can reduce our problem to the comparison of 
factorization categories under variation of GIT quotients. 
Then using window theorem developed in~\cite{MR3327537, MR3895631}, we prove our assertion. 
\subsubsection{Local $(-1, -1)$-curve}
Let $S\to \mathbb{C}^2$ be the blow-up at the origin, or 
equivalently 
\begin{align}\label{loc:-1}
	S=\mathrm{Tot}_{\mathbb{P}^1}(\oO_{\mathbb{P}^1}(-1)) \to \mathbb{P}^1
\end{align}
the total space of $\oO_{\mathbb{P}^1}(-1)$ over $\mathbb{P}^1$. 
In this case, we have
\begin{align*}
	X=\mathrm{Tot}_S(\omega_S)=\mathrm{Tot}_{\mathbb{P}^1}(\oO_{\mathbb{P}^1}(-1)^{\oplus 2}). 
\end{align*}
We denote by $C \subset S$ the zero section of the projection (\ref{loc:-1}), which is a 
$(-1)$-curve. 
By setting $\eE_S=\oO_S \oplus \oO_S(-C)$
and $A_S=\End(\eE_S)$, 
we have the derived equivalence~\cite{MR2057015}
\begin{align}\notag
	\Phi_S \cneq \RHom_S(\eE_S, -) \colon 
	D^b_{\rm{coh}}(S) \stackrel{\sim}{\to} D^b(\modu A_S). 
\end{align}
It is easy to see that $A_S$ is the path algebra of
the following quiver $Q_S$ with relation
\begin{align*}
	\xymatrix{
		\bullet^{0}  \ar@/^5pt/[rr]^-{a_1} & & \bullet^{1} \ar@/^15pt/[ll]^-{b_2} \ar@/^5pt/[ll]^-{b_1}
	}, \quad  
	b_1 a_1 b_2=b_2 a_1 b_1. 
\end{align*}
Here the objects $\oO_C$, $\oO_C(-1)[1]$ are sent to the simple 
$A_S$-modules corresponding to the vertices $\bullet^0$, $\bullet^1$
respectively. 

Similarly 
by setting $\eE_X=\pi^{\ast}\eE_S$ and 
$A_X=\End(\eE_X)$, we have the derived equivalence~\cite{MR2057015}
\begin{align}\label{deq:X}
	\Phi_X \cneq \RHom_X(\eE_X, -) \colon 
	D^b_{\rm{coh}}(X) \stackrel{\sim}{\to} D^b(\modu A_X). 
\end{align}
Then $A_X$ is the path algebra of
the following quiver $Q_X$ with relation
$\partial W$ for the super-potential $W$
(see~\cite{MR2836398})
\begin{align*}
	\xymatrix{
		\bullet^{0} \ar@/^15pt/[rr]^-{a_2}  \ar@/^5pt/[rr]^-{a_1} & & \bullet^{1} \ar@/^15pt/[ll]^-{b_2} \ar@/^5pt/[ll]^-{b_1}
	}, \quad  
	W=a_2(b_1 a_1 b_2-b_2 a_1 b_1). 
\end{align*}

We will also consider framed quivers 
\begin{align*}
	\xymatrix{
		\bullet^{0}  \ar@/^5pt/[rr]^-{a_1} & & \bullet^{1} \ar@/^15pt/[ll]^-{b_2} \ar@/^5pt/[ll]^-{b_1} \\
		\bullet^{\infty} \ar[u]^-{\xi} & &
	}, 
	\quad 
	\xymatrix{
		\bullet^{0} \ar@/^15pt/[rr]^-{a_2}  \ar@/^5pt/[rr]^-{a_1} & & \bullet^{1} \ar@/^15pt/[ll]^-{b_2} \ar@/^5pt/[ll]^-{b_1} \\
		\bullet^{\infty} \ar[u]^-{\xi} & &
	}
\end{align*}
denoted as $Q_S^{\dag}$, $Q_X^{\dag}$ respectively. 
By setting $\pP_S=\Phi_S(\oO_S)$, giving a $Q_S^{\dag}$-representation 
with dimension vector $1$ at $\bullet^{\infty}$ and relation $b_1 a_1 b_2=b_2 a_1 b_1$ is 
equivalent to giving a pair $(\pP_S \to M)$ 
where $M \in \modu A_S$. 
Similarly for $\pP_X=\Phi_X(\oO_X)$, 
giving a $Q_X^{\dag}$-representation with dimension vector $1$
at $\bullet^{\infty}$ and relation $\partial W$ is 
equivalent to 
giving a pair $(\pP_X \to N)$ for $N \in \modu A_X$. 

\subsubsection{Moduli stacks of quiver representations}
We prepare some notation on moduli stacks of quiver representations. 
For $\vec{v}=(v_0, v_1) \in \mathbb{Z}_{\ge 0}^2$, let 
$V_0$, $V_1$ be vector spaces with $\dim V_i=v_i$. 
We set 
\begin{align*}
	&R_{Q_S^{\dag}}(\vec{v})=V_0 \oplus \Hom(V_0, V_1) \oplus \Hom(V_1, V_0)^{\oplus 2}, \\
	&R_{Q_X^{\dag}}(\vec{v}) =V_0 \oplus \Hom(V_0, V_1)^{\oplus 2}
	\oplus \Hom(V_1, V_0)^{\oplus 2}. 
\end{align*}
The group $G=\GL(V_0) \times \GL(V_1)$ naturally 
acts on $R_{Q_S^{\dag}}$, $R_{Q_X^{\dag}}$, 
and the quotient stacks 
\begin{align*}
	\mM_{Q_S^{\dag}}(\vec{v})=[R_{Q_S^{\dag}}(\vec{v})/G], \ 
	\mM_{Q_X^{\dag}}(\vec{v})=[R_{Q_X^{\dag}}(\vec{v})/G]
\end{align*}
are the $\mathbb{C}^{\ast}$-rigidified moduli 
stacks of $Q_S^{\dag}$ and $Q_X^{\dag}$-representations with dimension vector $(1,  \vec{v})$, 
where $1$ is the dimension vector at the vertex $\bullet^{\infty}$, 
$v_i$ is the dimension vector at $\bullet^i$. 

Let $s$ be the
map
\begin{align*}
	s \colon R_{Q_S^{\dag}}(\vec{v}) \to \Hom(V_1, V_0), \ 
	s(\Xi, A_1, B_1, B_2)=B_1 A_1 B_2-B_2 A_1 B_1. 
\end{align*}
Then for the derived zero locus $s^{-1}(0) \subset R_{Q_S^{\dag}}(\vec{v})$, 
the derived stack
\begin{align}\label{dstack:MA}
	\mathfrak{M}_{A_S}^{\dag}(\vec{v})=[s^{-1}(0)/G]
\end{align}
is the derived moduli stack of
pairs $(\pP_S \to M)$ in $\modu A_S$, 
where $M$ has dimension vector $\vec{v}$. 
Its classical truncation is denoted by $\mM_{A_S}^{\dag}(\vec{v})$. 

By the above description, the $(-1)$-shifted cotangent derived stack 
$\Omega_{\mathfrak{M}^{\dag}_{A_S}(\vec{v})}[-1]$ is the derived
critical locus of 
\begin{align}\label{funct:w}
	w \colon [R_{Q_X^{\dag}}(\vec{v})/G] \to \mathbb{C}, \ 
	w(\Xi, A_1, A_2, B_1, B_2)=\mathrm{tr}(A_2(B_1 A_1 B_2-B_2 A_1 B_1)). 
\end{align}
It follows that 
\begin{align*}
	\mM_{A_X}^{\dag}(\vec{v})\cneq t_0(\Omega_{\mathfrak{M}^{\dag}_{A_S}(\vec{v})}[-1])
	=\mathrm{Crit}(w) \subset \mM_{Q_X^{\dag}}(\vec{v})
\end{align*}
is the moduli stack of pairs $(\pP_X \to N)$ in $\modu A_X$, where the 
dimension vector of $N$ is $\vec{v}$. 
The fiberwise weight two $\mathbb{C}^{\ast}$-action on 
$\mM_{A_X}^{\dag}(\vec{v}) \to
\mM_{A_S}^{\dag}(\vec{v})$
is 
given by 
\begin{align}\label{C:action}
	t \cdot (\Xi, A_1, A_2, B_1, B_2)=(\Xi, A_1, t^2 A_2, B_1, B_2).
\end{align}

\subsubsection{Moduli stacks of semistable representations}
Next we discuss King's $\theta$-stability conditions on $Q_X^{\dag}$-representations~\cite{Kin}. 
Let us take 
\begin{align*}
	\theta=(\theta_{\infty}, \theta_0, \theta_1) \in \mathbb{Q}^3.
\end{align*}
For a dimension vector $(v_{\infty}, v_0, v_1)$ of 
$Q_X^{\dag}$, 
we set $\theta(v_{\infty}, v_0, v_1)=\theta_{\infty} v_{\infty}+\theta_{0} v_0+\theta_1 v_1$. 

\begin{defi}
	A $Q_X^{\dag}$-representation $E$ with dimension vector $(1, \vec{v})$ is 
	called $\theta$-(semi)stable if 
	$\theta(1, \vec{v})=0$ and 
	for any subrepresentation $0\neq F \subsetneq E$, 
	we have the inequality
	$\theta(v(F))<(\le) 0$, where 
	$v(-)$ is the dimension vector. 
\end{defi}
Below we fix 
$\vec{v}=(v_0, v_1)$, and take $\theta$ such that $\theta(1, \vec{v})=0$ 
holds. 
Then $\theta_{\infty}$ is determined by
$\theta_{\infty}=-\theta_0 v_0-\theta_1 v_1$, so 
we simply write $\theta=(\theta_0, \theta_1)$. 

We denote by 
\begin{align*}
	\mM_{Q_X^{\dag}}^{\theta}(\vec{v})=[R_{Q_X^{\dag}}^{\theta}(\vec{v})/G]
\end{align*}
the open substack in $\mM_{Q_X^{\dag}}(\vec{v})$
corresponding to $\theta$-semistable representations. 
The stack $\mM_{Q_X^{\dag}}^{\theta}(\vec{v})$ admits the good moduli space
\begin{align}\label{Q:good:m}
	\pi_{Q_X^{\dag}} \colon 
	\mM_{Q_X^{\dag}}^{\theta}(\vec{v}) \to M_{Q_X^{\dag}}^{\theta}(\vec{v})
	\cneq 
	R_{Q_X^{\dag}}^{\theta}(\vec{v}) \sslash G
\end{align}
where $M_{Q_X^{\dag}}^{\theta}(\vec{v})$ parametrizes $\theta$-polystable 
$Q_X^{\dag}$-representations with dimension vector $(1, \vec{v})$. 
Namely a point $p \in M_{Q_X^{\dag}}^{\theta}(\vec{v})$ corresponds 
to the direct sum 
\begin{align}\label{E:polystable}
	E=
	F_{\infty} \oplus 
	\bigoplus_{i=1}^l U_i \otimes F_i
\end{align}
where 
$\{F_{\infty}, F_1, \ldots, F_l\}$ are mutually non-isomorphic 
$\theta$-stable $Q_X^{\dag}$-representations
with $\theta(v(F_i))=0$, 
$F_{\infty}$ has dimension vector of the form $(1, \ast)$, 
each $U_i$ is a finite dimensional vector space, 
and $E$ has dimension vector $(1, \vec{v})$. 

For a point $p$ as above, we denote by $Q_p^{\dag}$ the \textit{Ext-quiver}
associated 
with the collection of $\theta$-stable objects $\{F_{\infty}, F_1, \ldots, F_l\}$.
Namely the set of vertices of $Q_p^{\dag}$ is $\{\infty, 1, \ldots, l\}$, 
and the number of arrows from $i$ to $j$ is 
the dimension of $\Ext^1(F_i, F_j)$.  
From the construction of $Q_p^{\dag}$, note that 
\begin{align}\label{equiver:arrow}
	\sharp (i \to j)=\sharp (j \to i), \ 1\le i, j \le l, \ 
	\sharp (\infty \to i)-\sharp(i\to \infty)=v_0^{(i)} \ge 0
\end{align}
where we have written $v(F_i)=(0, v_0^{(i)}, v_1^{(i)})$ for $1\le i\le l$. 
The quiver $Q_p^{\dag}$ contains the subquiver
$Q_p \subset Q_p^{\dag}$ given by the Ext-quiver 
of the sub collection $\{F_1, \ldots, F_l\}$. 

We
set $G_p=\Aut(E)$, which is 
identified with the stabilizer group of the action of 
$G$ on $R_{Q_X^{\dag}}$
at the point corresponding to $E$. 
The algebraic group $G_p$ acts on $\Ext^1(E, E)$ by conjugation. 
Then we set 
\begin{align*}
	\pi_{Q_p^{\dag}} \colon 
	\mM_{Q_p^{\dag}}(\vec{u})\cneq [\Ext^1(E, E)/G_p] \to 
	M_{Q_p^{\dag}}(\vec{u}) \cneq \Ext^1(E, E)\sslash G_p. 
\end{align*}
The stack $\mM_{Q_p^{\dag}}(\vec{u})$ is identified with the stack of 
$Q_p^{\dag}$-representations with dimension vector 
$(1, \vec{u})$, where $1$ is the dimension vector at $\infty$, 
$\vec{u}=(\dim U_1, \ldots, \dim U_l)$, and $M_{Q_p^{\dag}}(\vec{u})$ 
its its good moduli space. 
We also define the open substacks
\begin{align*}
	\mM_{Q_p^{\dag}}^+(\vec{u}) \subset \mM_{Q_p^{\dag}}(\vec{u}), \ 
	\mM_{Q_p^{\dag}}^-(\vec{u}) \subset \mM_{Q_p^{\dag}}(\vec{u})
\end{align*}
corresponding to 
$Q_p^{\dag}$-representations 
such that the images of the maps
$\mathbb{C} \to U_i$
associated with arrows $\infty \to i$ 
for all $1\le i \le l$
generate $\oplus_{i=1}^l U_i$ as $\mathbb{C}[Q_p]$-module 
(resp. dual of the maps $U_i \to \mathbb{C}$ 
associated with arrows $i \to \infty$ generate
$\oplus_{i=1}^l U_i^{\vee}$ as $\mathbb{C}[Q_p]$-module). 
Then we have the following: 
\begin{prop}\label{prop:analytic}
	There exist analytic open neighborhoods
	$p \in T \subset M_{Q_X^{\dag}}^{\theta}(\vec{v})$, 
	$0 \in U \subset M_{Q_p^{\dag}}(\vec{u})$, and 
	commutative isomorphisms
	\begin{align}\label{isom:analytic}
		\xymatrix{
			(\pi_{Q_p^{\dag}})^{-1}(U) \ar[r]^-{\cong} \ar[d] & (\pi_{Q_X^{\dag}})^{-1}(T) \ar[d] \\
			U \ar[r]^-{\cong} & T. 
		}
	\end{align}
	Moreover the top isomorphism restricts to the 
	isomorphism
	\begin{align}\label{isom:restrict}
		(\pi_{Q_p^{\dag}})^{-1}(U)  \cap \mM_{Q_p^{\dag}}^{\pm}(\vec{u})
		\stackrel{\cong}{\to} (\pi_{Q_X^{\dag}})^{-1}(T) \cap \mM_{Q_X^{\dag}}^{\theta_{\pm}}(\vec{v}). 
	\end{align}
	Here $\theta_{\pm}=\theta \mp \varepsilon(1, 1)$ for $0<\varepsilon \ll 1$. 
\end{prop}
\begin{proof}
	Since $\mM_{Q_X^{\dag}}^{\theta}(\vec{v})$ is a smooth stack, 
	the isomorphisms (\ref{isom:analytic}) follow from Luna's \'etale slice theorem. 
	The isomorphism (\ref{isom:restrict}) 
	follows from the similar argument of~\cite[Theorem~9.11]{Toddbir},
	so we omit details. 
\end{proof}

\subsubsection{Window subcategories}
By~\cite{Kin}, 
the semistable locus 
$R_{Q_X^{\dag}}^{\theta}(\vec{v})$ is the GIT $L_{\theta}$-semistable locus 
for a $G$-equivariant $\mathbb{Q}$-line bundle $L_{\theta}$ on $R_{Q_X^{\dag}}(\vec{v})$, determined by the 
rational $G$-character
\begin{align}\label{chi:theta}
	\chi_{\theta} \colon 
	G=\GL(V_0) \times \GL(V_1) \to \mathbb{C}^{\ast}, \  (g_0, g_1) \mapsto \det (g_0)^{-\theta_0} \det (g_1)^{-\theta_1}. 
\end{align}
Let us fix a maximal torus $T \subset G$ 
and a Weyl-invariant norm $\lvert \ast \rvert$
on $\Hom_{\mathbb{Z}}(\mathbb{C}^{\ast}, T)_{\mathbb{R}}$. 
There is a Kempf-Ness stratification (see Section~\ref{subsec:KN})
\begin{align}\notag
	R_{Q_X^{\dag}}(\vec{v})=
  S_1 \sqcup S_2 \sqcup \cdots
	\sqcup 	R_{Q_X^{\dag}}^{\theta}(\vec{v}). 
\end{align}
Let us take $\theta_{\pm}=\theta \mp \varepsilon (1, 1)$ for $0<\varepsilon \ll 1$. The KN stratifications
with respect to $L_{\theta_{\pm}}$ 
are finer than that with respect to $L_{\theta}$, so we have the 
stratifications
\begin{align*}
	R_{Q_X^{\dag}}^{\theta}(\vec{v})=
 S_{1}^{\pm} \sqcup S_{2}^{\pm} \sqcup \cdots \sqcup 	R_{Q_X^{\dag}}^{\theta_{\pm}}(\vec{v}) 
\end{align*}
with associated one parameter subgroups $\lambda_{\alpha}^{\pm} \colon \mathbb{C}^{\ast} \to T$
and $\lambda_{\alpha}^{\pm}$-fixed subset $Z_{\alpha}^{\pm} \subset S_{\alpha}^{\pm}$. 
We set
\begin{align*}
	\eta_{\alpha}^{\pm} \cneq \wt_{\lambda_{\alpha}^{\pm}}\det(N_{S_{\alpha}^{\pm}/R_{Q_X^{\dag}}^{\theta}(\vec{v})}|_{Z_{\alpha}^{\pm}}).
\end{align*}
We also set
\begin{align}\label{def:delta}
	\delta=L_{(1, 1)}^{\otimes \varepsilon}
	\in \Pic_G(R_{Q_X^{\dag}}^{\theta}(\vec{v}))_{\mathbb{R}}, \ 
	0<\varepsilon \ll 1. 
\end{align}
Then the window subcategories (see Theorem~\ref{thm:window})
\begin{align*}
	\wW_{\delta}^{\theta_{\pm}}
	 \subset \mathrm{MF}_{\rm{coh}}^{\mathbb{C}^{\ast}}(\mM_{Q_X^{\dag}}^{\theta}(\vec{v}), w)
\end{align*}
are defined to be the triangulated subcategories 
of factorizations $\pP$ such that for all 
$\alpha$ we have 
\begin{align}\label{wt:conifold}
	\wt_{\lambda_{\alpha}}(
	\pP|_{Z_{\alpha}})
	\subset  
	\wt_{\lambda_{\alpha}}(\delta|_{Z_{\alpha}})+
	\left[ -\frac{\eta_{\alpha}^{\pm}}{2}, \frac{\eta_{\alpha}^{\pm}}{2} \right). 
\end{align}
By Theorem~\ref{thm:window}, 
the compositions 
\begin{align}\label{defi:window}
\wW_{\delta}^{\theta_{\pm}} \hookrightarrow \mathrm{MF}_{\rm{coh}}^{\mathbb{C}^{\ast}}(\mM_{Q_X^{\dag}}^{\theta}, w)
	\twoheadrightarrow \mathrm{MF}_{\rm{coh}}^{\mathbb{C}^{\ast}}
	(\mM_{Q_X^{\dag}}^{\theta_{\pm}}, w)
\end{align}
are equivalences. 
Here the right arrow is a restriction functor, 
$w$ is the function (\ref{funct:w}), 
and the right hand sides are the derived categories of 
$(G \times \mathbb{C}^{\ast})$-equivariant 
coherent factorizations of $w$, where the 
$\mathbb{C}^{\ast}$-action is given by (\ref{C:action}). 
\begin{prop}\label{prop:window}
	We have $\wW_{\delta}^{\theta_{-}} \subset \wW_{\delta}^{\theta_{+}}$. 
	Hence we have a fully-faithful functor
	\begin{align}\notag
		\mathrm{MF}_{\rm{coh}}^{\C}
		(\mM_{Q_X^{\dag}}^{\theta_-}(\vec{v}), w)
		\hookrightarrow 
		\mathrm{MF}_{\rm{coh}}^{\mathbb{C}^{\ast}}
		(\mM_{Q_X^{\dag}}^{\theta_+}(\vec{v}), w).
	\end{align}
\end{prop}
\begin{proof}
	Let 
	$\wW^{\theta_{\pm}, \mathbb{Z}/2\mathbb{Z}}_{\delta}$ be the subcategories
	\begin{align*}
	\wW^{\theta_{\pm}, \mathbb{Z}/2\mathbb{Z}}_{\delta} \subset \mathrm{MF}_{\rm{coh}}^{\mathbb{Z}/2\mathbb{Z}}
	(\mM_{Q_X^{\dag}}^{\theta}(\vec{v}), w)	
		\end{align*}
	defined by the same condition (\ref{wt:conifold})
	for factorizations $\pP$ without 
	auxiliary $\C$-action. 
		Since the condition (\ref{wt:conifold}) holds for 
	$\pP$ if and only if the same condition holds 
	after forgetting the $\C$-action, 
	it is enough to show the inclusion 
	\begin{align}\label{inclu:coni}
		\wW_{\delta}^{\theta_-, \mathbb{Z}/2\mathbb{Z}} \subset \wW_{\delta}^{\theta_+, \mathbb{Z}/2\mathbb{Z}}.
		\end{align}
	
	Below we follow the same strategy as in~\cite[Theorem~5.2]{KoTo}. 
	Let us take a point 
	$p \in M_{Q_X^{\dag}}(\vec{v})$ corresponding to a 
	polystable object $E$ as in (\ref{E:polystable}), 
	and take isomorphisms as in 
	Proposition~\ref{prop:analytic}. 
	Let $w_p \colon \pi_{Q_p^{\dag}}^{-1}(U) \to \mathbb{C}$ be the pull-back 
	of $w$ under the top isomorphism in Proposition~\ref{prop:analytic}. 
	Since the statement (\ref{inclu:coni})
	is a local question on the base of the map (\ref{Q:good:m}), 
	it is enough to show a similar statement for 
	$(\pi_{Q_p^{\dag}}^{-1}(U), w_p)$. 
	Let us write $(\pi_{Q_p^{\dag}})^{-1}(U)=[R_p/G_p]$ for an 
	analytic open subset $R_p \subset \Ext^1(E, E)$, 
	and denote its intersection with $\mM_{Q_p^{\dag}}^{\pm}(\vec{u})$
	by $[R_{p}^{\pm}/G_p]$. 
	We have the KN stratifications
	\begin{align}\label{KN:local}
		R_p=S_{1, p}^{\pm} \sqcup 
		S_{2, p}^{\pm} \sqcup \cdots \sqcup R_p^{\pm}
	\end{align}
	with respect to $G_p$-characters $\chi_{\theta_{\pm}}|_{G_p}$. 
	Here $\chi_{\theta}$ is given by (\ref{chi:theta}), and we have restricted 
	it to $G_p$ by the inclusion $G_p \subset G$. 
The $G_p$-characters $\chi_{\theta_{\pm}}|_{G_p}$ are written as 
	\begin{align}\label{write:Gp}
		\chi_{\theta_{\pm}}|_{G_p} \colon 
		\prod_{i=1}^l \GL(U_i) \to \C, \ 
		(h_i) \mapsto \prod_{i=1}^l (\det h_i)^{\pm \varepsilon(v_0^{(i)}+v_1^{(i)})}. 
		\end{align}
		Let $\delta_p$ be the pull-back of (\ref{def:delta}) to 
	a $G_p$-equivariant $\mathbb{R}$-line bundle on $R_p$ under 
	the top isomorphism of (\ref{isom:analytic}). 
	Then we have the window subcategories
	\begin{align*}
		\widehat{\wW}_{\delta_p}^{\theta_{\pm}, \mathbb{Z}/2\mathbb{Z}} \subset \mathrm{MF}_{\rm{coh}}^{\mathbb{Z}/2\mathbb{Z}}([R_{p}/G_p], w_p)
	\end{align*}
	defined similarly to (\ref{defi:window})
	with respect to the KN stratifications (\ref{KN:local}).  
	By the property of $Q_p^{\dag}$ 
	given in (\ref{equiver:arrow})
	and the description of $\chi_{\theta_{\pm}}|_{G_p}$ in (\ref{write:Gp}), 
	we are in the same situation in Proposition~\ref{prop:inclu:W}, 
	so the inclusion $\widehat{\wW}_{\delta_p}^{\theta_-, \mathbb{Z}/2\mathbb{Z}} \subset \widehat{\wW}_{\delta_p}^{\theta_+, \mathbb{Z}/2\mathbb{Z}}$
	follows from Proposition~\ref{prop:inclu:W} (see Remark~\ref{rmk:Z2Z}). 
	Since this holds for any $p$, 
	we have (\ref{inclu:coni}). 
\end{proof}

\subsubsection{Proof of Conjectures~\ref{conj:DT/PT}, \ref{conj2} local $(-1, -1)$-curve}
We show the following: 
\begin{thm}\label{thm:catDTPT:loc}
	Conjecture~\ref{conj:DT/PT} is true for $S=\mathrm{Tot}_{\mathbb{P}^1}(\oO_{\mathbb{P}^1}(-1))$, i.e. for 
	$(d, n) \in \mathbb{Z}^2$ with $d\ge 0$, there is a 
	fully-faithful functor
	\begin{align*}
		\mathcal{DT}^{\mathbb{C}^{\ast}}(P_n(X, d[C])) \hookrightarrow 
		\mathcal{DT}^{\mathbb{C}^{\ast}}(I_n(X, d[C])). 
	\end{align*} 
\end{thm}
\begin{proof}
	For $(d, n) \in \mathbb{Z}^2$, we
	set $\vec{v}=(v_0, v_1)=(n, n-d)$, and 
	\begin{align*} 
		\mM^{\dag, \theta}_{A_X}(\vec{v})=
		\mM_{A_X}^{\dag}(\vec{v}) \cap \mM_{Q_X^{\dag}}^{\theta}(\vec{v}).
	\end{align*}
	For $\theta=(-1, 1)$, it is proved in~\cite{MR2836398}
	(see~Figure~1 in \textit{loc.~cit.~})
	that the equivalence (\ref{deq:X}) induce the isomorphisms
	\begin{align}\label{isom:IP:A}
		\Phi_{X\ast} \colon I_n(X, d[C]) \stackrel{\cong}{\to}
		 \mM^{\dag, \theta_+}_{A_X}(\vec{v}), \ \Phi_{X\ast} \colon 
		P_n(X, d[C]) \stackrel{\cong}{\to} \mM^{\dag, \theta_-}_{A_X}(\vec{v})
	\end{align}
	by sending a pair $(\oO_X \to F)$ 
	to $(\pP_X \to \Phi_X(F))$. 
	Let 
	\begin{align}\label{conifold:qc}
		\mathfrak{M}_S^{\dag}(d[C], n)_{\rm{qc}} 
		\subset \mathfrak{M}_S^{\dag}(d[C], n)
	\end{align}
	be the derived
	open substack of pairs $(\oO_S \to F)$ on $S$
	such that $[F]=(d[C], n)$, and
	satisfying $\Phi_S(F) \in \modu A_S$. 
	We have the open immersion of derived stacks
	\begin{align*}
		\Phi_{S\ast} \colon 
		\mathfrak{M}_S^{\dag}(d[C], n)_{\rm{qc}} \hookrightarrow
		\mathfrak{M}_{A_S}^{\dag}(\vec{v})
	\end{align*}
	defined by sending 
	$(\oO_S \to F)$ to $(\pP_S \to \Phi_S(F))$. 
	In particular, the derived open substack (\ref{conifold:qc}) 
	is quasi-compact. 
	These morphisms fit into
	the commutative diagram
	\begin{align}\label{dia:IPA}
		\xymatrix{
			I_n(X, d[C])  \ar@<-0.3ex>@{^{(}->}[r]^-{\Phi_{X\ast}}  \ar[d] & 
			\mM_{A_X}^{\dag}(\vec{v}) \ar[d] \\
			\mM_S^{\dag}(d[C], n)_{\rm{qc}} 
			\ar@<-0.3ex>@{^{(}->}[r]^-{\Phi_{S\ast}}
			& \mM_{A_S}^{\dag}(\vec{v}),
		} \ 
		\xymatrix{
			P_n(X, d[C])  \ar@<-0.3ex>@{^{(}->}[r]^-{\Phi_{X\ast}}  \ar[d] & 
			\mM_{A_X}^{\dag}(\vec{v}) \ar[d] \\
			\mM_S^{\dag}(d[C], n)_{\rm{qc}} 
			\ar@<-0.3ex>@{^{(}->}[r]^-{\Phi_{S\ast}}
			& \mM_{A_S}^{\dag}(\vec{v}).
		}
	\end{align}
Here $\mM_S^{\dag}(d[C], n)_{\rm{qc}}=t_0(\fM_S^{\dag}(d[C], n)_{\rm{qc}})$. 
	We set
	\begin{align*}
		\zZ^{\theta\us}\cneq 
		\mM_{A_X}^{\dag}(\vec{v}) \setminus 
		\mM^{\dag, \theta}_{A_X}(\vec{v}). 
	\end{align*}
	Then from the isomorphisms (\ref{isom:IP:A})
	and the diagrams (\ref{dia:IPA}), 
	we obtain the equivalences by Lemma~\ref{lem:replace0}
	\begin{align*}
		&\mathcal{DT}^{\mathbb{C}^{\ast}}(I_n(X, d[C]))
		\stackrel{\sim}{\to}
		D^b_{\rm{coh}}(\mathfrak{M}_{A_S}(\vec{v}))/
		\cC_{\zZ^{\theta_+ \us}}, \\ 
		&\mathcal{DT}^{\mathbb{C}^{\ast}}(P_n(X, d[C]))
		\stackrel{\sim}{\to}
		D^b_{\rm{coh}}(\mathfrak{M}_{A_S}(\vec{v}))/
		\cC_{\zZ^{\theta_- \us}}. 
	\end{align*}
	From the description of the derived moduli stack (\ref{dstack:MA})
	and Theorem~\ref{thm:knoer}, 
	we have the equivalence
	\begin{align*}
		D^b_{\rm{coh}}(\mathfrak{M}_{A_S}(\vec{v}))/\cC_{\zZ^{\theta_{\pm}\us}} \stackrel{\sim}{\to}
		\mathrm{MF}^{\mathbb{C}^{\ast}}_{\mathrm{coh}}(\mM_{Q_X^{\dag}}(\vec{v})
		\setminus \zZ^{\theta_{\pm}\us}, w).
	\end{align*} 
	Since the critical locus of $w$ in $\mM_{Q_X^{\dag}}^{\theta_{\pm}}(\vec{v})$
	is contained in $\mM_{Q_X^{\dag}}(\vec{v})
	\setminus \zZ^{\theta_{\pm}\us}$, 
	the following restriction functors give equivalences
	(see~(\ref{rest:equiv}))
	\begin{align*}
		\mathrm{MF}^{\mathbb{C}^{\ast}}_{\mathrm{coh}}(\mM_{Q_X^{\dag}}(\vec{v})
		\setminus \zZ^{\theta_{\pm}\us}, w)
		\stackrel{\sim}{\to}
		\mathrm{MF}^{\mathbb{C}^{\ast}}_{\mathrm{coh}}
		(\mM_{Q_X^{\dag}}^{\theta_{\pm}}(\vec{v}), w).
	\end{align*}
	Therefore we have the equivalences
	\begin{align}\notag
		&\mathcal{DT}^{\mathbb{C}^{\ast}}(I_n(X, d[C]))
		\stackrel{\sim}{\to}
		\mathrm{MF}_{\rm{coh}}^{\mathbb{C}^{\ast}}
		(\mM_{Q_X}^{\theta_+}(\vec{v}), w), \\ 
		\notag
		&\mathcal{DT}^{\mathbb{C}^{\ast}}(P_n(X, d[C]))
		\stackrel{\sim}{\to}
		\mathrm{MF}_{\rm{coh}}^{\mathbb{C}^{\ast}}
		(\mM_{Q_X}^{\theta_-}(\vec{v}), w), 
	\end{align}
	and the theorem follows from Proposition~\ref{prop:window}. 
\end{proof}

Also the same argument of Theorem~\ref{thm:catDTPT:loc} shows the following: 
\begin{thm}\label{thm:PTWCF:loc}
	Conjecture~\ref{conj2} is true 
	when $S \to \mathbb{C}^2$ is the blow-up at the origin. 
\end{thm}
\begin{proof}
	For $t \in \mathbb{Q}_{>0}$ which lies on a wall, 
		we set $\theta=(-t+1, t)$. 
	Then similarly to (\ref{isom:IP:A}), we can show that the 
	functor $\Phi_X$ induces the isomorphisms
	\begin{align*}
		\Phi_{X\ast} \colon 
		P_n^{t_{\pm}}(X, d[C]) \stackrel{\cong}{\to}
		\mM_{A_X}^{\dag, \theta_{\pm}}(\vec{v}). 
	\end{align*}
	The rest of the proof is exactly same as in Theorem~\ref{thm:catDTPT:loc}. 
\end{proof}

\subsection{Dualities of DT categories for D0-D2-D6 bound states}\label{subsec:duality}
\subsubsection{Moduli stacks of dual pairs}\label{subsec:dualpair}
Let $\mathfrak{M}_S(v)$ be the derived moduli 
stack of one dimensional coherent sheaves on $S$
with numerical class 
$v \in N_{\le 1}(S)$, with truncation 
$\mM_S(v)$. 
In this subsection, we focus on the 
open substack
\begin{align*}
	\mathfrak{M}_S^{\rm{pure}}(v) \subset \mathfrak{M}_S(v), \ 
	\mM_S^{\rm{pure}}(v) \subset \mM_S(v)
\end{align*}
consisting of pure one 
dimensional sheaves on $S$. 
For a pure one dimensional sheaf $F$ on $S$, 
we set  
\begin{align}\label{F:dual}
	F^{\vee} \cneq \mathbb{D}_S(F) \otimes \omega_S[1]
	=\eE xt^1_{\oO_S}(F, \omega_S).
\end{align}
Then $F^{\vee}$ is a pure one dimensional 
sheaf on $S$ with numerical class $v^{\vee}$.
Here for $v=(\beta, n) \in N_{\le 1}(S)$, we write 
$v^{\vee}=(\beta, -n)$. 
Hence 
we have the equivalence of derived stacks
\begin{align}\label{equiv:pure:one}
	\mathbb{D}_S(-) \otimes \omega_S[1]  \colon 
	\mathfrak{M}_S^{\rm{pure}}(v) \stackrel{\sim}{\to}
	\mathfrak{M}_S^{\rm{pure}}(v^{\vee}). 
\end{align}
Similarly we denote by 
\begin{align*}
	\mathfrak{M}_S^{\dag, \rm{pure}}(v) \subset \mathfrak{M}_S^{\dag}(v), \ 
	\mM_S^{\dag, \rm{pure}}(v) \subset \mM_S^{\dag}(v)
\end{align*}
the open (derived) substacks of 
pairs $(\oO_S \to F)$
such that $F$ is a pure one dimensional sheaf. 
We also define the derived stack $\mathfrak{M}_S^{\sharp, \rm{pure}}(v)$
over 
$\mathfrak{M}_S^{\rm{pure}}(v)$ to be
\begin{align*}
	\mathfrak{M}_S^{\sharp, \rm{pure}}(v) \cneq 
	\Spec_{\fM_S^{\rm{pure}}(v)} S(p_{\mathfrak{M}\ast}\mathfrak{F}[1]). 
\end{align*}
Here $\mathfrak{F}$ is the universal family 
as in Subsection~\ref{subsec:pairs}. 
Its classical truncation is given by
\begin{align*}
	\mM_S^{\sharp, \rm{pure}}(v) \cneq 
	t_0(
	\mathfrak{M}_S^{\sharp, \rm{pure}}(v))= 
	\Spec_{\mM_S^{\rm{pure}}(v)} S(\hH^1(\dR p_{\mM\ast}\fF)).
\end{align*}

\begin{lem}\label{lem:dual:sharp}
	(i) 
	The $T$-valued points of $\mM_S^{\sharp, \rm{pure}}(v)$
	form the groupoid of 
	pairs 
	\begin{align*}
		(F_T, \xi'), \ 
		F_T \stackrel{\xi'}{\to} \omega_S \boxtimes \oO_T[1]
	\end{align*}
	where $F_T$ is a $T$-valued point of $\mM_S^{\rm{pure}}(v)$
	and $\xi'$ is a
	morphism in $D^b_{\rm{coh}}(S \times T)$. 
	
	(ii) The fiber of the projection
	\begin{align}\label{fib:dag:sharp}
		t_0(\Omega_{\mathfrak{M}_S^{\sharp, \rm{pure}}(v)}[-1]) \to 
		\mM_S^{\sharp, \rm{pure}}(v)
	\end{align}
	at the pair $(F \stackrel{\xi'}{\to} \omega_S[1])$
	is given by $\Hom(\uU, F)$, 
	where $0 \to \oO_S \to \uU \to F\otimes \omega_S^{-1} \to 0$
	is the extension class of $\xi'$. 
\end{lem}
\begin{proof}
	(i) follows from the Serre duality 
	$H^1(F)^{\vee}=\Hom(F, \omega_S[1])$ for a one dimensional 
	sheaf $F$.
	As for (ii), 
	we have the distinguished triangle
	\begin{align*}
		(\RHom(F, F)[1])^{\vee} \to 
		\mathbb{L}_{\fM_S^{\sharp, \rm{pure}}(v)}|_{(F \to \omega_S[1])}
		\to \dR \Gamma(F[1]). 
		\end{align*}
	By Serre duality, we have 
	$(\RHom(F, F)[1])^{\vee} \cong \RHom(F, F \otimes \omega_S[1])$, 
	so $\mathbb{L}_{\fM_S^{\sharp, \rm{pure}}(v)}|_{(F \to \omega_S[1])}$
	is obtained by the cone of the morphism 
	$\dR \Gamma(F) \to \RHom(F, F\otimes \omega_S[1])$
	induced by $\xi'$. 
	Therefore $\mathbb{L}_{\fM_S^{\sharp, \rm{pure}}(v)}|_{(F \to \omega_S[1])}$
	is quasi-isomorphic to 
	$\RHom(\uU, F)[1]$. 
	The fiber of (\ref{fib:dag:sharp}) is given by 
	its $(-1)$-th cohomology, which is $\Hom(\uU, F)$. 
\end{proof}
We have the following lemma: 
\begin{lem}\label{lem:isom:DS}
	We have the equivalence of derived stacks
	\begin{align}\label{isom:DS}
		\mathbb{D}_S(-) \otimes \omega_S[1] 
		\colon \mathfrak{M}_S^{\dag, \rm{pure}}(v) \stackrel{\sim}{\to}
		\mathfrak{M}_S^{\sharp, \rm{pure}}(v^{\vee})
	\end{align}
	which on $\mathbb{C}$-valued points given by 
	\begin{align*}
		(\oO_S \to F) \mapsto \mathbb{D}_S(\oO_S \to F)
		 \otimes \omega_S[1]=
		(F^{\vee} \to \omega_S[1]).
	\end{align*}
	Moreover we have the commutative diagram
	\begin{align}\label{commu:dual}
		\xymatrix{
			\mathfrak{M}_S^{\dag, \rm{pure}}(v) \ar[d]_-{\mathbb{D}_S(-) \otimes \omega_S[1]}^-{\sim} \ar[r]^-{\rho^{\dag}} & \fM_S^{\rm{pure}}(v) \ar[d]_-{\mathbb{D}_S(-) \otimes \omega_S[1]}^-{\sim} \ar[r]^-{i^{\dag}} & \mathfrak{M}_S^{\dag, \rm{pure}}(v) 
			\ar[d]_-{\mathbb{D}_S(-) \otimes \omega_S[1]}^-{\sim} \\
			\mathfrak{M}_S^{\sharp, \rm{pure}}(v^{\vee})  \ar[r]_-{\rho^{\sharp}} & \fM_S^{\rm{pure}}(v^{\vee})  \ar[r]_-{i^{\sharp}} & \mathfrak{M}_S^{\sharp, \rm{pure}}(v^{\vee}).
		}
	\end{align}
	Here $\rho^{\dag}$, $\rho^{\sharp}$ are projections and 
	$i^{\dag}$, $i^{\sharp}$ are the zero sections. 
\end{lem}
\begin{proof}
	Under the equivalence (\ref{equiv:pure:one}), 
	the object $(p_{\mathfrak{M}\ast}\mathfrak{F})^{\vee}$
	on $\fM_S^{\rm{pure}}(v)$ 
	corresponds to $p_{\mathfrak{M}\ast}\mathfrak{F}[1]$
	on $\fM_S^{\rm{pure}}(v^{\vee})$, since 
	for a pure one dimensional sheaf $F$ on $S$
	we have 
	\begin{align*}
		\dR \Gamma(F)^{\vee}=\RHom(F, \omega_S[2])=\dR \Gamma(F^{\vee})[1]
	\end{align*}
	by the Serre duality. 
	Therefore we have the equivalence (\ref{isom:DS}). 
	The commutativity of the diagram (\ref{commu:dual}) is obvious 
	from the constructions. 
\end{proof}
We also denote by
\begin{align*}
	\mM_X^{\dag, \rm{pure}}(v) \subset \mM_X^{\dag}(v)
\end{align*}
the open substack of 
objects in the 
subcategory
\begin{align}\label{def:Apure}
	\aA_X^{\rm{pure}} \cneq \langle \oO_{\overline{X}}, \Coh_{\le 1}^{\rm{pure}}(X)[-1] 
	\rangle_{\rm{ex}} 
	\subset \aA_X
\end{align}
where $\Coh_{\le 1}^{\rm{pure}}(X)$ is
the category of compactly supported 
pure one dimensional coherent sheaves on $X$. 
Note that the map $\pi_{\ast}^{\dag}$ in Theorem~\ref{thm:D026}
restricts to the morphism 
$\pi_{\ast}^{\dag} \colon \mM_X^{\dag, \rm{pure}}(v) \to 
\mM_S^{\dag, \rm{pure}}(v)$
and the isomorphism (\ref{isom:dag}) restricts to the 
isomorphism over $\mM_S^{\dag, \rm{pure}}(v)$
\begin{align*}
	\eta^{\dag} \colon \mM_X^{\dag, \rm{pure}}(v) \stackrel{\cong}{\to}
	t_0(\Omega_{\mathfrak{M}_S^{\dag, \rm{pure}}(v)}[-1]). 
\end{align*}

We have a statement similar to Theorem~\ref{thm:D026}
for 
the stack $\mM_S^{\sharp, \rm{pure}}(v)$.
\begin{prop}\label{thm:D026v2}
	(i) 
	There exists a natural morphism 
	$\pi_{\ast}^{\sharp} \colon \mM_X^{\dag, \rm{pure}}(v) \to \mM_S^{\sharp, \rm{pure}}(v)$
	sending a diagram (\ref{intro:dia:BS})
	to the pair 
	$(F \stackrel{\xi'}{\to} \omega_S[1])$
	corresponding to the 
	extension class of the top sequence of (\ref{intro:dia:BS}). 
	
	(ii) There is an isomorphism over 
	$\mM_S^{\sharp, \rm{pure}}(v)$
	\begin{align*}
		\eta^{\sharp} \colon \mM_X^{\dag, \rm{pure}}(v) \stackrel{\cong}{\to}
		t_0(\Omega_{\mathfrak{M}_S^{\sharp, \rm{pure}}(v)}[-1]) 
	\end{align*}
	which, over the pair 
	$(F \stackrel{\xi'}{\to} \omega_S[1])$, 
	sends a diagram (\ref{intro:dia:BS}) 
	to the 
	morphism $\xi \colon \uU \to F$. 
\end{prop}
\begin{proof}
	The proof 
	for $\mathbb{C}$-valued points 
	is clear from Lemma~\ref{lem:dual:sharp}
	and Lemma~\ref{lem:obvious0}. 
	The argument for $T$-valued points 
	is similar to Lemma~\ref{lem:cone}, so we omit details. 
	\end{proof}

Note that we have obtained the commutative diagram
\begin{align}\label{diagram:dpair}
	\xymatrix{
		t_0(\Omega_{\mathfrak{M}_S^{\dag, \rm{pure}}(v)}[-1]) \ar[d]
		&\ar[l]_-{\eta^{\dag}}^-{\cong} \ar[r]^-{\eta^{\sharp}}_-{\cong} \mM_X^{\dag, \rm{pure}}(v)  \ar[ld]^-{\pi_{\ast}^{\dag}} \ar[rd]_-{\pi_{\ast}^{\sharp}}
		\ar[d]
		&t_0(\Omega_{\mathfrak{M}_S^{\sharp, \rm{pure}}(v)}[-1]) \ar[d] \\
		\mM_S^{\dag, \rm{pure}}(v) \ar[r]_-{\rho_0^{\dag}} & \mM_S^{\rm{pure}}(v) & \ar[l]^-{\rho_0^{\sharp}} \mM_S^{\sharp, \rm{pure}}(v). 
	}
\end{align}

On the other hand
the category $\aA_X^p$ is closed under the 
derived dual functor 
$\mathbb{D}_{\overline{X}}$, so we have 
an isomorphism of stacks
\begin{align}\label{isom:X:dual}
	\mathbb{D}_{\overline{X}} \colon 
	\mM_X^{\dag, \rm{pure}}(v) \stackrel{\cong}{\to} \mM_X^{\dag, \rm{pure}}(v^{\vee}).
\end{align}
\begin{lem}\label{lem:D026v2}
	We have the commutative diagram
	\begin{align}\label{commute:dual2}
		\xymatrix{
			\mM_X^{\dag, \rm{pure}}(v) \ar[r]^-{\mathbb{D}_{\overline{X}}}_-{\cong} 
			\ar[d]_{\eta^{\dag}}
			&
			\mM_X^{\dag, \rm{pure}}(v^{\vee}) \ar[d]^{\eta^{\sharp}} \\
			t_0(\Omega_{\mathfrak{M}_S^{\dag, \rm{pure}}(v)}[-1]) \ar[r]_-{\cong} & 
			t_0(\Omega_{\mathfrak{M}_S^{\sharp, \rm{pure}}(v^{\vee})}[-1]). 
		}
	\end{align}
	Here the bottom isomorphism is induced 
	by the equivalence (\ref{isom:DS}). 
\end{lem}
\begin{proof}
	The commutativity of the diagram (\ref{commute:dual2}) immediately 
	follows from the constructions of $\eta^{\dag}$, $\eta^{\sharp}$ and 
	Lemma~\ref{prop:duality}. 
\end{proof}

By Proposition~\ref{thm:D026v2}, 
the stack $\mM_X^{\dag, \rm{pure}}(v)$
admits two $\mathbb{C}^{\ast}$-actions, 
fiberwise $\mathbb{C}^{\ast}$-actions
with respect to $\pi_{\ast}^{\dag}$ and $\pi_{\ast}^{\sharp}$. 
These two $\mathbb{C}^{\ast}$-actions on $\mM_X^{\dag, \rm{pure}}(v)$
correspond to scaling actions on the maps $\uU \to F$, $\uU \to F\otimes \omega_S^{-1}$
in the diagram (\ref{intro:dia:BS}) respectively. 
We call a closed substack $\zZ \subset \mM_X^{\dag, \rm{pure}}(v)$
to be \textit{double conical} if it is closed under both of the above 
$\mathbb{C}^{\ast}$-actions. 

Let us take quasi-compact 
derived open substacks
\begin{align}\label{pure:qcompact}
	\mathfrak{M}_S(v)_{\rm{qc}} \subset \fM_S^{\rm{pure}}(v), \ 
	\mM_S(v)_{\rm{qc}} \subset \mM_S^{\rm{pure}}(v)
\end{align}
where $\mM_S(v)_{\rm{qc}}=t_0(\mathfrak{M}_S(v)_{\rm{qc}})$.  
Below we use the subscript $\rm{qc}$ to indicate the 
pull-back by 
the open immersion (\ref{pure:qcompact}), 
e.g. 
$\mathfrak{M}_S^{\dag}(v)_{\rm{qc}} \cneq 
\mathfrak{M}_S(v)_{\rm{qc}} \times_{\fM_S^{\rm{pure}}(v)} \mathfrak{M}_S^{\dag, \rm{pure}}(v)$. 
By pulling back the diagram (\ref{diagram:dpair}) 
via the open immersion (\ref{pure:qcompact}), 
we obtain the commutative diagram 
\begin{align}\label{diagram:dpair2}
	\xymatrix{
		t_0(\Omega_{\mathfrak{M}_S^{\dag}(v)_{\rm{qc}}}[-1]) \ar[d]
		&\ar[l]_-{\eta^{\dag}}^-{\cong} \ar[r]^-{\eta^{\sharp}}_-{\cong} \mM_X^{\dag}(v)_{\rm{qc}}  \ar[ld]^-{\pi_{\ast}^{\dag}} \ar[rd]_-{\pi_{\ast}^{\sharp}}
		\ar[d]
		&t_0(\Omega_{\mathfrak{M}_S^{\sharp}(v)_{\rm{qc}}}[-1]) \ar[d] \\
		\mM_S^{\dag}(v)_{\rm{qc}} \ar[r]_-{\rho_0^{\dag}} & \mM_S(v)_{\rm{qc}} & \ar[l]^-{\rho_0^{\sharp}} \mM_S^{\sharp}(v)_{\rm{qc}}. 
	}
\end{align}
Similarly 
we take 
a quasi-compact 
derived open substack
$\mathfrak{M}_S(v^{\vee})_{\rm{qc}} \subset \fM_S^{\rm{pure}}(v^{\vee})$
with truncation $\mM_S(v^{\vee})_{\rm{qc}}$, 
such that the equivalence (\ref{equiv:pure:one})
restricts to 
the equivalence
$\mathfrak{M}_S(v)_{\rm{qc}} \stackrel{\sim}{\to} 
\mathfrak{M}_S(v^{\vee})_{\rm{qc}}$. 
Then 
via the isomorphism $(\mathbb{D}_{\overline{X}},  \mathbb{D}_S(-) \otimes \omega_S)$ in
the diagrams (\ref{commu:dual}), (\ref{commute:dual2}), 
we see that the diagram (\ref{diagram:dpair2})
is isomorphic to the diagram
\begin{align}\label{diagram:dpair3}
	\xymatrix{
		t_0(\Omega_{\mathfrak{M}_S^{\sharp}(v^{\vee})_{\rm{qc}}}[-1]) \ar[d]
		&\ar[l]_-{\eta^{\sharp}}^-{\cong} \ar[r]^-{\eta^{\dag}}_-{\cong} \mM_X^{\dag}(v^{\vee})_{\rm{qc}}  \ar[ld]^-{\pi_{\ast}^{\sharp}} \ar[rd]_-{\pi_{\ast}^{\dag}}
		\ar[d]
		&t_0(\Omega_{\mathfrak{M}_S^{\dag}(v^{\vee})_{\rm{qc}}}[-1]) \ar[d] \\
		\mM_S^{\sharp}(v^{\vee})_{\rm{qc}} \ar[r]_-{\rho_0^{\sharp}} & \mM_S(v^{\vee})_{\rm{qc}} & \ar[l]^-{\rho_0^{\dag}} \mM_S^{\dag}(v^{\vee})_{\rm{qc}}. 
	}
\end{align}
We have the following corollary of Proposition~\ref{thm:D026v2}
and Proposition~\ref{prop:compare:Z}:

\begin{cor}\label{cor:duality}
	Let 
	$\zZ \subset \mM_X^{\dag}(v)_{\rm{qc}}$
	be a double conical closed substack. 
	
	(i) Let 
	$\mathbb{D}_{\overline{X}}(\zZ) \subset \mM_X^{\dag}(v^{\vee})_{\rm{qc}}$
	be its image under the isomorphism (\ref{isom:X:dual}). 
	We have the equivalence
	\begin{align*}
		D^b_{\rm{coh}}(\mathfrak{M}_S^{\dag}(v)_{\rm{qc}})/\cC_{\eta^{\dag}(\zZ)}
		\stackrel{\sim}{\to} D^b_{\rm{coh}}(\mathfrak{M}_S^{\sharp}(v^{\vee})_{\rm{qc}})
		/\cC_{\eta^{\sharp}(\mathbb{D}_{\overline{X}}(\zZ))}. 
	\end{align*}
	
	(ii) We have the equivalence
	\begin{align*}
		D^b_{\rm{coh}}(\mathfrak{M}_S^{\dag}(v)_{\rm{qc}})/\cC_{\eta^{\dag}(\zZ)}
		\stackrel{\sim}{\to} D^b_{\rm{coh}}(\mathfrak{M}_S^{\sharp}(v)_{\rm{qc}})
		/\cC_{\eta^{\sharp}(\zZ)}. 
	\end{align*}
\end{cor}
\begin{proof}
	As for (i), 
	we have the equivalence by Lemma~\ref{lem:isom:DS}
	\begin{align*}
		\Dbc(\fM_S^{\dag}(v)_{\rm{qc}}) \stackrel{\sim}{\to}
		\Dbc(\fM_S^{\sharp}(v^{\vee})_{\rm{qc}}). 
		\end{align*}
	Then 
	Lemma~\ref{lem:D026v2} implies that the above 
	equivalence restricts to the equivalence 
	between $\cC_{\eta^{\dag}(\zZ)}$ and 
	$\cC_{\eta^{\sharp}(\mathbb{D}_{\overline{X}}(\zZ))}$. 
	By taking the quotients, we obtain (i). 
	
	As for the equivalence in (ii),  
	it is straightforward to check that 
	the following diagram commutes: 
	\begin{align*}
		\xymatrix{
			& \mM_X^{\dag, \rm{pure}}(v) \ar[ld]^-{\cong}_-{\eta^{\dag}} 
			\ar[rd]_-{\cong}^-{\eta^{\sharp}} & \\
			t_0(\Omega_{\mathfrak{M}_S^{\dag, \rm{pure}}(v)}[-1]) 
			\ar[rr]^{\vartheta_0}_-{\cong} & &
			t_0(\Omega_{\mathfrak{M}_S^{\sharp, \rm{pure}}(v)}[-1]). 
		}
	\end{align*}
	Here $\vartheta_0$ is the isomorphism in Lemma~\ref{lem:equiv:shifted}. 
	Therefore the equivalence in (ii)
	follows from Proposition~\ref{prop:compare:Z}. 
\end{proof}

\subsubsection{Wall-crossing at $t<0$}
Let us return to the situation of Subsections~\ref{subsec:catDT:stableD026}, 
\ref{subsec:conj:AX}. 
We first give the following lemma: 
\begin{lem}\label{lem:t<0}
	The moduli stack of $\mu_t^{\dag}$-semistable objects
	$\pP_n^{t}(X, \beta)$
	is an 
	open substack in $\mM_X^{\dag, \rm{pure}}(\beta, n)$, 
	and the 
	isomorphism (\ref{isom:X:dual}) restricts to the isomorphisms
	\begin{align*}
		\mathbb{D}_{\overline{X}} \colon 
		\pP_n^{t}(X, \beta)
		\stackrel{\cong}{\to}
		\pP_{-n}^{-t}(X, \beta), \ 
		\mathbb{D}_{\overline{X}} \colon 
		P_n^t(X, \beta) \stackrel{\cong}{\to} P_{-n}^{-t}(X, \beta). 
	\end{align*}
\end{lem}
\begin{proof}
	The former statement 
	follows from Lemma~\ref{lem:onepure:AX}
	together with the $\mu_t^{\dag}$-stability. 
	The latter statement is 
	due to~\cite[Proposition~5.4 (iii)]{Tsurvey}. 
\end{proof}
Then 
we have the following duality statement 
as an application of
Corollary~\ref{cor:duality}:
\begin{thm}\label{thm:duality}
	If $t \in \mathbb{R}$ lies in a chamber,
	we have the equivalence
	\begin{align}\notag
		\mathcal{DT}^{\mathbb{C}^{\ast}}(P_n^t(X, \beta))
		\stackrel{\sim}{\to} \dD \tT^{\mathbb{C}^{\ast}}(P_{-n}^{-t}(X, \beta)). 
	\end{align}
	In particular 
	we have the equivalence (\ref{catDT:tsmall}). 
\end{thm}
\begin{proof}
	Below we use the notation of the diagrams (\ref{diagram:dpair2}), (\ref{diagram:dpair3}). 
	Let us take a 
	quasi-compact open derived substack 
	$\mathfrak{M}_S(\beta, n)_{\rm{qc}}  \subset \mathfrak{M}_S^{\rm{pure}}(\beta, n)$ 
	such that $\mM_S^{\dag}(\beta, n)_{\rm{qc}}=t_0(\mathfrak{M}_S^{\dag}(\beta, n)_{\rm{qc}})$
	satisfies the condition (\ref{p0(P)}). 
	By definition, we have 
	\begin{align*}
		\mathcal{DT}^{\mathbb{C}^{\ast}}(P_n^t(X, \beta)) =
		D^b_{\rm{coh}}(\mathfrak{M}_S^{\dag}(\beta, n)_{\rm{qc}})/\cC_{\zZ^{t\us}}. 
	\end{align*}
	Note that 
	the unstable locus 
	$(\eta^{\dag})^{-1}(\zZ^{t\us})$ is a double conical closed 
	substack of $\mM_X^{\dag}(\beta, n)_{\rm{qc}}$, 
	since the two $\mathbb{C}^{\ast}$-actions preserve the $\mu_t^{\dag}$-stability. 
	By Lemma~\ref{lem:t<0}
	the isomorphism (\ref{isom:X:dual})
	restricted to $\mM_X^{\dag}(v)_{\rm{qc}}$
	sends
	$(\eta^{\dag})^{-1}(\zZ^{t\us})$
	to $(\eta^{\dag})^{-1}(\zZ^{-t\us})$. 
	Therefore applying Corollary~\ref{cor:duality} (i) for $v=(\beta, n)$ and 
	$\zZ=(\eta^{\dag})^{-1}(\zZ^{t\us})$, we have
	the equivalence
	\begin{align*}
		D^b_{\rm{coh}}(\mathfrak{M}_S^{\dag}(\beta, n)_{\rm{qc}})/\cC_{\zZ^{t\us}}
		\stackrel{\sim}{\to}
		D^b_{\rm{coh}}(\mathfrak{M}_S^{\sharp}(\beta, -n)_{\rm{qc}})/
		\cC_{\eta^{\sharp}(\eta^{\dag})^{-1}(\zZ^{-t\us})}. 
	\end{align*}
	Also applying Corollary~\ref{cor:duality} (ii) 
	for $v=(\beta, -n)$ and $\zZ=(\eta^{\dag})^{-1}(\zZ^{-t\us})$, we 
	have the equivalence
	\begin{align*}
		D^b_{\rm{coh}}(\mathfrak{M}_S^{\dag}(\beta, -n)_{\rm{qc}})/\cC_{\zZ^{-t\us}}
		\stackrel{\sim}{\to}
		D^b_{\rm{coh}}(\mathfrak{M}_S^{\sharp}(\beta, -n)_{\rm{qc}})/
		\cC_{\eta^{\sharp}(\eta^{\dag})^{-1}(\zZ^{-t\us})}. 
	\end{align*}
	The desired equivalence follows from the 
	above equivalences. 
\end{proof}

For $t_0<0$, the diagram (\ref{dia:dflip}) is a d-critical anti-flip. 
Therefore for $t<0$ wall-crossing, 
we expect a chain of fully-faithful functors
for $0>s_1>s_2>\cdots>s_M$ which lie on chambers
\begin{align*}
	\mathcal{DT}^{\mathbb{C}^{\ast}}(P_n^{s_1}(X, \beta)) 
	\hookrightarrow 
	\mathcal{DT}^{\mathbb{C}^{\ast}}(P_n^{s_2}(X, \beta)) 
	\hookrightarrow 
	\cdots 
	\hookrightarrow
	\mathcal{DT}^{\mathbb{C}^{\ast}}(P_n^{s_M}(X, \beta)). 
\end{align*}
Indeed 
Theorem~\ref{thm:duality} implies that the above chain of 
fully-faithful functors is equivalent to 
that of fully-faithful functors (\ref{FF:P1})
for $(\beta, -n)$. 

\subsubsection{Wall-crossing formula of categorical PT theories for irreducible curve classes}
We take $(\beta, n) \in N_{\le 1}(S)$ such that 
$\beta$ is an irreducible class and $n\ge 0$. 
As in Subsection~\ref{subsec:catDT:one}, 
we use the notation 
$\mM_n(X, \beta)$, $M_n(X, \beta)$ for 
moduli spaces of one dimensional stable sheaves on $X$
with numerical class $(\beta, n)$. 
In the irreducible $\beta$ case, 
there is only one wall 
$t=n/(H \cdot \beta)$
with respect to the $\mu_t^{\dag}$-stability, 
and the diagram (\ref{dia:dflip}) 
becomes (see~\cite[Section~9.6]{Toddbir})
\begin{align*}
	\xymatrix{
		P_n(X, \beta) \ar[rd]  & & P_{-n}(X, \beta) \ar[ld] \\
		& M_n(X, \beta). &
	}
\end{align*} 
By~\cite[Theorem~9.22]{Toddbir}, the above diagram is a simple d-critical flip if $n>0$, 
simple d-critical flop if $n=0$. 
Below we prove Conjecture~\ref{conj2} in this case. 

As $\beta$ is irreducible, the derived stack $\fM_S^{\rm{pure}}(\beta, n)$ 
coincides with the derived
stack of one dimensional stable sheaves on $S$ 
with numerical class $(\beta, n)$. 
In particular, all of the derived stacks 
$\fM_S^{\rm{pure}}(\beta, n)$, $\mathfrak{M}_S^{\dag, \rm{pure}}(\beta, n)$, 
$\mathfrak{M}_S^{\sharp}(\beta, n)$ are QCA. 
In the notation of the diagram (\ref{commu:dual}), we set 
\begin{align*}
	\mathfrak{P}_n(S, \beta) \cneq 
	\mathfrak{M}_S^{\dag, \rm{pure}}(\beta, n) \setminus i^{\dag}(\mathfrak{M}_S^{\rm{pure}}(\beta, n)), \ 
	\mathfrak{Q}_n(S, \beta) \cneq 
	\mathfrak{M}_S^{\sharp, \rm{pure}}(\beta, n) \setminus i^{\sharp}(\mathfrak{M}_S^{\rm{pure}}(\beta, n))
\end{align*}
which are derived open substacks of 
$\mathfrak{M}_S^{\dag, \rm{pure}}(\beta, n)$, $\mathfrak{M}_S^{\sharp, \rm{pure}}(\beta, n)$
respectively. 
By the assumption that $\beta$ is irreducible, 
the derived stack $\mathfrak{P}_n(S, \beta)$
is the derived moduli scheme of PT
stable pairs on $S$ with numerical class $(\beta, n)$, so 
its truncation 
\begin{align*}
	P_n(S, \beta) \cneq t_0(\mathfrak{P}_n(S, \beta))
\end{align*}
is the moduli 
space of PT stable pairs on $S$. 
\begin{lem}\label{lem:P:irred}
	For $(\beta, n) \in N_{\le 1}(S)$ such that $\beta$ is irreducible, 
	we have the equivalences
	\begin{align}\label{equiv:P:irred}
		\dD \tT^{\mathbb{C}^{\ast}}(P_n(X, \beta))
		\stackrel{\sim}{\to} D^b_{\rm{coh}}(\mathfrak{P}_n(S, \beta)), \ 
		\dD \tT^{\mathbb{C}^{\ast}}(P_{-n}(X, \beta)) 
		\stackrel{\sim}{\to} D^b_{\rm{coh}}
		(\mathfrak{Q}_n(S, \beta)). 
	\end{align}
	Moreover the following identity holds:
	\begin{align}\label{id:PT:HP}
		\chi(\mathrm{HP}_{\ast}
		(\mathcal{DT}_{\rm{dg}}^{\mathbb{C}^{\ast}}(P_n(X, \beta)))=
		(-1)^{n+\beta^2}P_{n, \beta}. 
	\end{align}
\end{lem}
\begin{proof}
	In~\cite[Lemma~5.13]{MR3842061}
	it is proved that, under the assumption that $\beta$ is irreducible, 
	$P_n(X, \beta)$ is the dual obstruction cone over
	$P_n(S, \beta)$. 
	It follows that we have 
	\begin{align*}
		P_n(X, \beta)=(\pi_{\ast}^{\dag})^{-1}(P_n(S, \beta))
	\end{align*}
	for the map $\pi_{\ast}^{\dag} \colon \mM_X^{\dag}(\beta, n) \to \mM_S^{\dag}(\beta, n)$
	in Theorem~\ref{thm:D026}. 
	Therefore the first equivalence 
	(\ref{equiv:P:irred}) follows from Lemma~\ref{lem:replace0}
	as $\mathfrak{M}_S^{\dag, \rm{pure}}(\beta, n)$ is QCA. 
	The second equivalence of (\ref{equiv:P:irred})
	follows from the first equivalence for $(\beta, -n)$
	together with the equivalence 
	of derived schemes
	\begin{align*}
		\mathfrak{P}_{-n}(S, \beta)
		\stackrel{\sim}{\to}
		\mathfrak{Q}_n(S, \beta)
	\end{align*}
	induced by the equivalences in the diagram (\ref{commu:dual}). 
	
	As for the identity (\ref{id:PT:HP}), 
	note that
	the virtual dimension of 
	$P_n(S, \beta)$ is calculated by the Riemann-Roch theorem
	\begin{align*}
		\vdim P_n(S, \beta)=\chi(\oO_S \to F, F)=\beta^2+n
	\end{align*}
	for a stable pair $(\oO_S \to F)$ on $S$
	with $[F]=(\beta, n)$. 
	Then the identity (\ref{id:PT:HP}) follows from Proposition~\ref{prop:num}, 
	since $\pi_{\ast}^{\dag}(P_n(X, \beta))=P_n(S, \beta)$ and 
	$P_n(S, \beta)$ is a projective scheme (see Remark~\ref{rmk:condition}).  
\end{proof}

The following result gives a proof of Conjecture~\ref{conj2} in the 
irreducible curve class case: 

\begin{thm}\label{thm:PT:irredu}
	For $(\beta, n) \in N_{\le 1}(S)$ such that $\beta$ is irreducible 
	and $n\ge 0$, 
	there exists a fully-faithful 
	functor 
	\begin{align*}
		\Phi_P \colon 
		\mathcal{DT}^{\mathbb{C}^{\ast}}(P_{-n}(X, \beta))
		\hookrightarrow \mathcal{DT}^{\mathbb{C}^{\ast}}(P_n(X, \beta))
	\end{align*}
	and 
	a semiorthogonal decomposition
	\begin{align}\label{SOD:PT}
		\dD \tT^{\mathbb{C}^{\ast}}(P_n(X, \beta))
		=\langle \Upsilon_{-n+1}, \ldots, \Upsilon_{0}, \Imm \Phi_{P} \rangle
	\end{align}
	such that each $\dD_{\lambda}$ is equivalent to 
	$\mathcal{DT}^{\mathbb{C}^{\ast}}(\mM_n(X, \beta))_{\lambda}$. 
\end{thm}
\begin{proof}
	By Lemma~\ref{lem:P:irred}, 
	the result follows from 
	Theorem~\ref{thm:SOD:M}
	by applying it for $\mathfrak{M}=\mathfrak{M}_S^{\rm{pure}}(\beta, n)$, 
	$\mathfrak{E}=(p_{\mathfrak{M}\ast}\mathfrak{F})^{\vee}$. 
	Here we note that, 
	since the object $\mathfrak{E}$ has weight $-1$ with respect to the 
	embedding $(\mathbb{C}^{\ast})_{\mM} \subset I_{\mM}$
	given by the 
	scaling action $\mathbb{C}^{\ast} \subset \Aut(F)$, 
	we need to apply Theorem~\ref{thm:SOD:M}
	for the composition
	of the above embedding 
	$(\mathbb{C}^{\ast})_{\mM} \subset I_{\mM}$
	with the isomorphism 
	$(\mathbb{C}^{\ast})_{\mM} \stackrel{\cong}{\to} (\mathbb{C}^{\ast})_{\mM}$
	sending $x$ to $x^{-1}$. 
	Then $e=n \ge 0$, and 
	the weights in the semiorthogonal decomposition 
	(\ref{SOD:PT}) have opposite signs from those in Theorem~\ref{thm:SOD:M}. 
\end{proof}

The above result is regarded as a categorification of 
the wall-crossing formula 
\begin{align}\label{PT:formula}
	P_{n, \beta}-P_{-n, \beta}=(-1)^{n-1} n N_{n, \beta}
\end{align}
proved in~\cite{MR2552254}, which is a special case of the formula (\ref{intro:PT:WCF}). 
Indeed we have the following: 
\begin{prop}\label{rmk:wcf}
	In the situation of Theorem~\ref{thm:PT:irredu}, 
	suppose that there is a divisor $D$ on $S$ such that 
	$(D\cdot \beta, n)$ is coprime. 
	Then the semiorthogonal decomposition (\ref{SOD:PT}) implies the formula (\ref{PT:formula}). 
\end{prop}
\begin{proof}
	The assumption 
	together with 
	Proposition~\ref{prop:rigid} and Lemma~\ref{lem:periodic} (ii)
	imply that  
	each $\dD_{\lambda}$ is equivalent to 
	$\mathcal{DT}^{\mathbb{C}^{\ast}}(M_n(X, \beta))$. 
	On the other hand, as we discussed in the proof of Lemma~\ref{lem:cyclic}, 
	the assignment of dg-categories to mixed complexes of 
	Hochschild complexes takes 
	exact sequences to distinguished triangles 
	(see~\cite[Theorem~3.1]{MR1492902}). 
	Therefore 
	the semiorthogonal decomposition (\ref{SOD:PT}) implies the formula
	\begin{align*}
		\chi(\mathrm{HP}_{\ast}(\dD \tT^{\mathbb{C}^{\ast}}_{\rm{dg}}(P_n(X, \beta))))
		-\chi(\mathrm{HP}_{\ast}(\dD \tT^{\mathbb{C}^{\ast}}_{\rm{dg}}(P_{-n}(X, \beta))))
		=n \cdot \chi(\mathrm{HP}_{\ast}(\dD \tT^{\mathbb{C}^{\ast}}_{\rm{dg}}(M_n(X, \beta)))). 
	\end{align*}
	The formula (\ref{PT:formula}) 
	follows from the above identity together with 
	(\ref{id:M:HP}), 
	(\ref{id:PT:HP}). 
\end{proof}

\begin{exam}\label{exam:Hilb}
	Suppose that $H^1(\oO_S)=0$. 
	Then there is unique $L \in \Pic(S)$ such that 
	$c_1(L)=\beta$. 
	Let $\lvert L \rvert$ be the complete linear system, and 
	\begin{align*}
		\pi_{\cC} \colon \cC \to \lvert L \rvert
	\end{align*}
	the universal
	curve. We 
	also denote by  
	$g \in \mathbb{Z}$ the arithmetic 
	genus of curves in $\lvert L \rvert$, 
	i.e. $g=1+\beta(K_S+\beta)/2$. Then 
	we have an isomorphism (see~\cite[Appendix]{MR2552254})
	\begin{align*}
		P_n(S, \beta) \cong \cC^{[n+g-1]}.
	\end{align*}
	Here the right hand side is $\pi_{\cC}$-relative Hilbert scheme
	of $(n+g-1)$-points on $\cC$. 
	Moreover $P_n(S, \beta)$ has only 
	locally complete intersection singularities, and 
	we have a distinguished triangle
	in $D^b_{\rm{coh}}(P_n(S, \beta))$ (see~\cite[Section~5.9]{MR3842061})
	\begin{align*}
		\vV[1] \to \mathbb{L}_{\mathfrak{P}_n(S, \beta)}|_{P_n(S, \beta)}
		\to \mathbb{L}_{P_n(S, \beta)}
	\end{align*}
	for a locally free sheaf $\vV$ on $P_n(S, \beta)$
	with rank $h^1(L)$. 
	Therefore 
	the closed immersion $P_n(S, \beta) \hookrightarrow \mathfrak{P}_n(S, \beta)$
	is an equivalence if and only if $h^1(L)=0$. 
	
	If $h^1(L)=0$, then  
	the semiorthogonal decomposition in Theorem~\ref{thm:PT:irredu} is equivalent to the semiorthogonal decomposition
	for $n\ge 0$
	\begin{align}\label{SOD:Hilb}
		D^b_{\rm{coh}}(\cC^{[n+g-1]})=
		\langle \Upsilon_{-n+1}, \ldots, \Upsilon_0, D^b_{\rm{coh}}(\cC^{[-n+g-1]}) \rangle. 
	\end{align}
	Here each $\dD_{\lambda}$ is the bounded 
	derived category of 
	twisted coherent sheaves on the relative compactified Jacobian 
	$\overline{J} \to \lvert L \rvert$. 
	The above semiorthogonal decomposition generalizes~\cite[Corollary~5.10]{TodDK}, 
	where a similar semiorthogonal decomposition is proved under some more additional 
	assumptions. 
	On the other hand if $h^1(L) \neq 0$, 
	the derived scheme $\mathfrak{P}_n(S, \beta)$
	has non-trivial derived structures, and 
	the semiorthogonal decomposition in Theorem~\ref{thm:PT:irredu} is different from 
	(\ref{SOD:Hilb}). 
\end{exam}

\subsubsection{Application to the rationality}
The wall-crossing formula (\ref{PT:formula})
was used in~\cite{MR2552254} to show the rationality of the generating 
series of PT invariants. 
Using the semiorthogonal decomposition in Theorem~\ref{thm:PT:irredu}, 
we also have some rationality statement. 
First we give the following definition:
\begin{defi}
	The Grothendieck group of triangulated 
	categories  
	$K(\Delta \mathchar`-\mathop{\mathrm{Cat}})$
	is the abelian group generated by 
	equivalence classes of triangulated categories, 
	with relations 
	$[\aA]=[\bB]+[\cC]$ for 
	semiorthogonal decompositions
	$\aA=\langle \bB, \cC \rangle$. 
\end{defi}

\begin{rmk}\label{rmk:pretri}
	The Grothendieck group of pre-triangulated categories 
	introduced in~\cite{MR2051435} is a refined version of 
	$K(\Delta \mathchar`-\mathop{\mathrm{Cat}})$. 
	An advantage of the former is the existence of the product 
	structure. As we will not need the product structure below, 
	we use the simpler version $K(\Delta \mathchar`-\mathop{\mathrm{Cat}})$. 
\end{rmk}

Let
$\mathbb{Q}(q)^{\rm{inv}} \subset \mathbb{Q}\lgakko q \rgakko$
be the subspace of 
rational functions invariant under $q \leftrightarrow 1/q$. 
We have the following corollary of Theorem~\ref{thm:PT:irredu}: 
\begin{cor}\label{cor:PT:rational}
	Suppose that $\beta$ is irreducible. 
	Then the generating series 
	\begin{align*}
		P_{\beta}^{\rm{cat}}(q) \cneq 
		\sum_{n \in \mathbb{Z}}
		[\mathcal{DT}^{\mathbb{C}^{\ast}}(P_n(X, \beta))]
		q^n \in K(\Delta \mathchar`-\mathop{\mathrm{Cat}})\lgakko q \rgakko
	\end{align*}
	lies in 
	$K(\Delta \mathchar`-\mathop{\mathrm{Cat}}) \otimes_{\mathbb{Z}}
	\mathbb{Q}(q)^{\rm{inv}}$. 
\end{cor}
\begin{proof}
	For a fixed irreducible class $\beta$, we set
	\begin{align*}
		c_{n, \lambda} \cneq [\mathcal{DT}^{\mathbb{C}^{\ast}}(\mM_n(X, \beta))_{-\lambda}] \in K(\Delta \mathchar`-\mathop{\mathrm{Cat}}). 
	\end{align*}
	Then the semiorthogonal decomposition in Theorem~\ref{thm:PT:irredu}
	implies the identity in $K(\Delta \mathchar`-\mathop{\mathrm{Cat}})$
	\begin{align}\label{relation:PT}
		[\mathcal{DT}^{\mathbb{C}^{\ast}}(P_n(X, \beta))]
		=[\mathcal{DT}^{\mathbb{C}^{\ast}}(P_{-n}(X, \beta))]
		+\sum_{\lambda=0}^{n-1} c_{n, \lambda} 
	\end{align}
	for $n\ge 0$.
	We consider the following generating series
	\begin{align*}
		N_{\beta}^{\rm{cat}}(q) \cneq 
		\sum_{n\ge 0} \sum_{\lambda=0}^{n-1} c_{n, \lambda} q^n
		\in K(\Delta \mathchar`-\mathop{\mathrm{Cat}})\lkakko q \rkakko. 
	\end{align*}
	Then by the relation (\ref{relation:PT}), we have 
	\begin{align*}
		P_{\beta}^{\rm{cat}}(q)-N_{\beta}^{\rm{cat}}(q)=
		[\mathcal{DT}^{\mathbb{C}^{\ast}}(P_0(X, \beta))] +
		\sum_{n>0}[\mathcal{DT}^{\mathbb{C}^{\ast}}(P_{-n}(X, \beta))](q^n+q^{-n}). 
	\end{align*}
	The right hand side 
	lies in $K(\Delta \mathchar`-\mathop{\mathrm{Cat}}) \otimes_{\mathbb{Z}}
	\mathbb{Q}(q)^{\rm{inv}}$.
	Therefore it is enough to show that 
	$N_{\beta}^{\rm{cat}}(q)$ lies in 
	$K(\Delta \mathchar`-\mathop{\mathrm{Cat}}) \otimes_{\mathbb{Z}}
	\mathbb{Q}(q)^{\rm{inv}}$
	
	Let $d\in \mathbb{Z}_{>0}$ be as in Lemma~\ref{lem:periodic}. 
	Then by Lemma~\ref{lem:periodic}, we have the identities
	\begin{align}\label{identity:c}
		c_{n, \lambda}=c_{n+d, \lambda}=c_{n, \lambda+n}=c_{n, \lambda+d}=c_{-n, -\lambda}. 
	\end{align}
	Let us set
	\begin{align*}
		C_n \cneq \sum_{\lambda=0}^{d-1}c_{n, \lambda}, \ 
		\overline{C}_n \cneq \sum_{\lambda=0}^{n-1} c_{n, \lambda}. 
	\end{align*}
	Then using the relations (\ref{identity:c}), a straightforward computation
	shows that 
	\begin{align*}
		N_{\beta}^{\rm{cat}}(q)=C_0 \frac{q^d}{(1-q^d)^2}+
		\sum_{n=1}^{d-1}C_n (q^n+q^{-n})\frac{q^d}{2(1-q^d)^2}
		+\sum_{n=1}^{d-1} \overline{C}_n
		\frac{q^n-q^{d-n}}{2(1-q^d)}.
	\end{align*}
	Since the right hand side lies in 
	$K(\Delta \mathchar`-\mathop{\mathrm{Cat}}) \otimes_{\mathbb{Z}}
	\mathbb{Q}(q)^{\rm{inv}}$, 
	the assertion holds. 
\end{proof}

\subsection{The category of D0-D2-D6 bound states}\label{subsec:cat026}
The purpose of this section is to prove Theorem~\ref{thm:D026}. 
As we mentioned in Subsection~\ref{intro:subsec:dtpt}, the key ingredient is 
to relate rank one objects in the 
category of D0-D2-D6 bound states on the local surface
with the diagram (\ref{intro:dia:BS}) on the surface. 
This is similar to Diaconescu's description~\cite{MR2888981} of 
D0-D2-D6 bound states on local curves (i.e. non-compact 
CY 3-folds given by total spaces of 
split rank two vector bundles on smooth projective curves)
in terms of ADHM sheaves on curves
(though the details of the comparison with $\aA_X$ are not available 
in literatures).  
\subsubsection{Notation of local surfaces}
Let $S$ be a smooth projective surface over $\mathbb{C}$, 
and $X=\mathrm{Tot}_S(\omega_S)$ be the total 
space of its canonical line bundle. 
We take its compactification $\overline{X}$ and 
consider the diagram
\begin{align}\notag
	\xymatrix{
		X \ar@<-0.3ex>@{^{(}->}[r]^-{j} \ar[d]_-{\pi} &
		\overline{X}=\mathbb{P}_S(\omega_S \oplus \oO_S) 
		\ar@<-0.5ex>[ld]_-{\overline{\pi}} \\
		S \ar@<-0.5ex>[ru]_-{i_{0}, i_{\infty}}. &
	}
\end{align}
Here $\pi$, $\overline{\pi}$ are projections, 
$j$ is an open immersion, 
$i_{0}$ is the zero section of $\pi$ and $i_{\infty}$ is the 
section of $\overline{\pi}$ at the infinity $\overline{X} \setminus X$. 
We denote by
\begin{align*}
	S_0 \cneq i_0(S), \ S_{\infty} \cneq i_{\infty}(S).
\end{align*}
Below we often identity $S_0$, $S_{\infty}$ with $S$
by the morphisms $i_0$, $i_{\infty}$ respectively. 
Note that we have
\begin{align}\label{O(S)}
	\oO_{\overline{X}}(1)=\oO_{\overline{X}}(S_{\infty}), \ 
	\Omega_{\overline{X}/S}=\oO_{\overline{X}}(-S_0-S_{\infty}), \ 
	\overline{\pi}^{\ast}\omega_S=\oO_{\overline{X}}(S_0-S_{\infty}). 
\end{align}
By taking the restrictions to both of 
$S_0$ and $S_{\infty}$,
we have the canonical 
isomorphism
\begin{align}\label{isom:push}
	\dR \overline{\pi}_{\ast}\oO_{\overline{X}}(S_{\infty})
	\stackrel{\cong}{\to}
	\dR \overline{\pi}_{\ast}\oO_{S_{0}}(S_{\infty}) \oplus
	\dR \overline{\pi}_{\ast}\oO_{S_{\infty}}(S_{\infty})
	=\oO_S \oplus \omega_S^{-1}. 
\end{align}
The above isomorphism will be often used below. 

\subsubsection{The category $\bB_S$}
We define the category $\bB_S$ whose objects consist of 
diagrams
\begin{align}\label{diagram:BS}
	\xymatrix{
		0 \ar[r] & \vV \ar[r]^-{\alpha} & \uU \ar[r]^-{\psi}
		\ar[d]_-{\phi} & F \otimes \omega_S^{-1} 
		\ar[r] & 0 \\
		&   & F  & 
	}
\end{align}
where $\vV \in \langle \oO_S \rangle_{\rm{ex}}$, 
$F \in \Coh_{\le 1}(S)$ and the top 
sequence is an exact sequence of 
coherent sheaves on $S$. 
The morphisms between two diagrams (\ref{diagram:BS}) are
given by commutative diagrams 
\begin{align}\notag
	\xymatrix{
0 \ar[r] & \vV \ar[d]_-{a_{\vV}} \ar[r]^-{\alpha} & \uU \ar[d]_-{a_{\uU}}\ar[r]^-{\psi} & F \otimes \omega_S^{-1} 
\ar[d]_-{a_{F} \otimes \id} \ar[r] & 0 \\
0 \ar[r] & \vV' \ar[r]^-{\alpha'} & \uU' \ar[r]^-{\psi'} & F' \otimes \omega_S^{-1} \ar[r] & 0,
} \quad 
	\xymatrix{
\uU \ar[d]_-{a_{\uU}}\ar[r]^-{\phi} & F 
	\ar[d]_-{a_F}  \\
\uU' \ar[r]^-{\phi'} & F',
}
	\end{align}
\begin{lem}\label{lem:BS}
	The category $\bB_S$ is an abelian category. 
\end{lem}
\begin{proof}
	Because $\Hom(F, \oO_S)=0$ for $F \in \Coh_{\le 1}(S)$, 
	the snake lemma easily implies that the termwise kernels and 
	cokernels give diagrams in $\bB_S$. 
\end{proof}
For a diagram (\ref{diagram:BS}), its rank is 
defined to be the rank of the sheaf $\vV$. 
We denote by 
\begin{align*}
	\bB_S^{\le 1} \subset \bB_S
\end{align*}
the subcategory 
of $\bB_S$ 
consisting of diagrams (\ref{diagram:BS})
with rank less than or equal to one. 
For a rank one diagram, we have another
equivalent way to give it. 
\begin{lem}\label{diagram:equivalent}
	Giving a rank one diagram (\ref{diagram:BS}) is equivalent to giving 
	a pair $(\oO_S \stackrel{\xi}{\to} F)$ for $F \in \Coh_{\le 1}(S)$
	together with a morphism 
	$\vartheta \colon F \otimes \omega_S^{-1} \to I^{\bullet}[1]$
	in $D^b_{\rm{coh}}(S)$. Here $I^{\bullet}=(\oO_S \stackrel{\xi}{\to} F)$
	is the two term complex such that $\oO_S$ is located in degree zero. 
\end{lem}
\begin{proof}
	First suppose that we are given a digram (\ref{diagram:BS}) with $\vV=\oO_S$. 
	Then we have the associated pair 
	\begin{align*}
		(\xi \colon \oO_S \stackrel{\alpha}{\to} \uU \stackrel{\phi}{\to} F). 
	\end{align*}
	The top sequence in (\ref{diagram:BS}) gives a 
	quasi-isomorphism $(\oO_S \stackrel{\alpha}{\to} \uU)[1] \stackrel{\sim}{\to}
	F \otimes \omega_S^{-1}$
	and we have the morphism of complexes
	\begin{align*}
		\xymatrix{
			\oO_S \ar[r]^-{\alpha} \ar[d]_-{\id} & \uU \ar[d]^-{\phi} \\
			\oO_S \ar[r]^-{\xi}  & F. 
		}
	\end{align*}
	The above diagram gives a morphism 
	$\vartheta \colon F \otimes \omega_S^{-1} \to I^{\bullet}[1]$
	in $D^b_{\rm{coh}}(S)$. 
	
	Conversely, suppose that we are given a pair 
	$(\xi \colon \oO_S \to F)$ together with a morphism 
	$\vartheta \colon F \otimes \omega_S^{-1} \to I^{\bullet}[1]$. 
	Then we have the commutative diagram
	\begin{align}\label{dia:OIF}
		\xymatrix{
			& \oO_S\ar[r]^-{\alpha}\ar[d]^-{\id} & \uU\ar[r]^-{\psi} \ar@{.>}[d]^-{\phi}& F \otimes \omega_S^{-1} \ar[r] \ar[d]^-{\vartheta} & \oO_S[1] \ar[d]^-{\id} \\
			I^{\bullet} \ar[r]^-{\beta}& \oO_S\ar[r]^-{\xi} &
			F \ar[r] & I^{\bullet}[1] \ar[r] & \oO_S[1]. 
		}
	\end{align}
	Here horizontal sequences are distinguished triangles. 
	Therefore there exists a morphism $\phi \colon \uU \to F$ which 
	makes the above diagram commutative. 
	We need to show that $\phi$ is uniquely determined by 
	the above commutativity. 
	Suppose that there exists another $\phi' \colon \uU \to F$ which 
	makes the above diagram commutative. Below we show that 
	$\phi=\phi'$. 
	By the commutativity of (\ref{dia:OIF}), $\phi-\phi'$ is written as 
	\begin{align*}
		\phi-\phi'=\xi \circ \gamma, \ \uU \stackrel{\gamma}{\to} \oO_S \stackrel{\xi}{\to} F
	\end{align*}
	for some morphism $\gamma \colon \uU \to \oO_S$. 
	We have $\phi=\phi'$ if $\xi=0$, so we may assume that $\xi \neq 0$. 
	The commutativity of (\ref{dia:OIF}) implies that 
	\begin{align*}
		\xi \circ \gamma \circ \alpha
		=(\phi-\phi') \circ \alpha
		=0, \ 
		\oO_S \to F. 
	\end{align*}
	Therefore $\gamma \circ \alpha \colon \oO_S \to \oO_S$ 
	is written as
	\begin{align*}
		\gamma \circ \alpha=\beta \circ w, \ 
		\oO_S \stackrel{w}{\to} I^{\bullet} \stackrel{\beta}{\to} \oO_S
	\end{align*}
	for some morphism $w \colon \oO_S \to I^{\bullet}$. 
	However from the 
	distinguished triangle $F[-1] \to I^{\bullet} \to \oO_S$
	we have the exact sequence
	\begin{align*}
		0=\Hom(\oO_S, F[-1]) \to \Hom(\oO_S, I^{\bullet}) \to \mathbb{C} \to \Hom(\oO_S, F)
	\end{align*}
	where the right arrow takes $1$ to $\xi$, which is injective by the assumption $\xi\neq 0$. 
	Therefore $\Hom(\oO_S, I^{\bullet})=0$, hence $w=0$. 
	It follows that $\gamma \circ \alpha=0$, hence $\gamma$ is written as
	\begin{align*}
		\gamma=\iota \circ \psi, \ \uU \stackrel{\psi}{\to} F \otimes \omega_S^{-1} 
		\stackrel{\iota}{\to} \oO_S
	\end{align*}
	for some morphism $\iota$. But $\iota$ must be a zero map as $F$ is a torsion sheaf. 
	Therefore $\gamma=0$ and $\phi=\phi'$ follows. 
\end{proof}

We also have the following description of rank
one objects in $\bB_S$, which is obvious. 
\begin{lem}\label{lem:obvious0}
	Giving a rank one diagram (\ref{diagram:BS}) is equivalent to 
	giving a pair $(\xi \colon F\otimes \omega_S^{-1} \to \oO_S[1])$
	together with a morphism $\phi \colon \uU \to F$, 
	where $0 \to \oO_S \to \uU \to F\otimes \omega_S^{-1} \to 0$
	is the extension determined by $\xi$. 
		\end{lem}

\subsubsection{The functor from $\aA_X$ to $\bB_S$}
We have the following natural 
diagram in $D^b_{\rm{coh}}(\overline{X})$
\begin{align}\label{diagram:S}
	\xymatrix{
		\oO_{S_{\infty}} \ar[r] & \oO_{\overline{X}}(-S_0 -S_{\infty})[1] \ar[r] \ar[d]
		& \oO_{\overline{X}}(-S_0)[1] \\
		& \oO_{\overline{X}}(-S_{\infty})[1] &
	}
\end{align}
where the top arrow is a distinguished triangle. 
For $E \in D^b_{\rm{coh}}(\overline{X})$, by 
taking the tensor product with (\ref{diagram:S})
and push-forward to $S$, 
we obtain the following diagram in $D^b_{\rm{coh}}(S)$
\begin{align}\label{diagram:Phi(E)}
	\xymatrix{
		\dR \overline{\pi_{\ast}}(E|_{S_{\infty}})
		\ar[r] & \dR \overline{\pi_{\ast}}(E(-S_0 -S_{\infty})[1]) \ar[r] \ar[d]
		& \dR \overline{\pi_{\ast}}(E(-S_0)[1]) \\
		& \dR \overline{\pi_{\ast}}(E(-S_{\infty})[1]) &
	} 
\end{align}
Let $\aA_X$ be the abelian subcategory of 
$D^b_{\rm{coh}}(\overline{X})$ defined in Definition~\ref{defi:state}. 
We have the following lemma: 
\begin{lem}\label{lem:Phi}
	If $E \in \aA_{X}$, then the diagram (\ref{diagram:Phi(E)})
	determines an object in $\bB_S$. 
\end{lem}
\begin{proof}
	We check that 
	\begin{align*}
		\dR \overline{\pi_{\ast}}(E|_{S_{\infty}}) \in \langle 
		\oO_S \rangle_{\rm{ex}}, \ 
		\dR \overline{\pi_{\ast}}(E(-S_{\infty})[1]) \in \Coh_{\le 1}(S). 
	\end{align*}
	The former one follows from the definition of $\aA_X$ and 
	the latter one follows from 
	$\dR \overline{\pi}_{\ast}(\oO_{\overline{X}}(-S_{\infty}))=0$. 
	Moreover noting that 
	$\overline{\pi}^{\ast}\omega_S=\oO_{\overline{X}}(S_0 -S_{\infty})$, 
	we have 
	\begin{align*}
		\dR \overline{\pi_{\ast}}(E(-S_0)[1])
		= \dR \overline{\pi_{\ast}}(E(-S_{\infty})[1])
		\otimes \omega_S^{-1}. 
	\end{align*}
	Therefore the lemma follows. 
\end{proof}

Let $\aA_X^{\le 1} \subset \aA_X$ be the subcategory 
consisting of objects $E$ with $\rank(E) \le 1$. 
By Lemma~\ref{lem:Phi}, we have the functors 
\begin{align}\notag
	\Phi \colon \aA_X \to \bB_S, \ 
	\Phi \colon \aA_X^{\le 1} \to \bB_S^{\le 1}
\end{align}
sending $E \in \aA_X$ to the diagram (\ref{diagram:Phi(E)}).

\subsubsection{The functor from $\bB_S$ to $D^b_{\rm{coh}}(\overline{X})$}
As for the other direction of functors, we define 
\begin{align}\label{funct:Psi}
	\Psi \colon \bB_S \to D^b_{\rm{coh}}(\overline{X})
\end{align}
by sending a diagram (\ref{diagram:BS})
to the two term complex
\begin{align}\label{Psi(dia)}
	\left(\overline{\pi}^{\ast}\uU \stackrel{\eta}{\to} 
	\overline{\pi}^{\ast}F \otimes \oO_{\overline{X}}(S_{\infty}) \right). 
\end{align}
Here 
$\overline{\pi}^{\ast}\uU$ is located in degree zero and 
$\eta$ is determined by the adjoint of the following map
\begin{align*}
	(\phi, \psi) \colon 
	\uU \to \dR \overline{\pi}_{\ast}(\overline{\pi}^{\ast}F \otimes \oO_{\overline{X}}(S_{\infty}))
	\stackrel{\cong}{\to}
	F \oplus (F\otimes \omega_S^{-1})
\end{align*}
where $\phi$, $\psi$ are maps in the diagram (\ref{diagram:BS}), 
and the second arrow is given by the projection 
formula together with the isomorphism (\ref{isom:push}). 
\begin{lem}\label{lem:isom1}
	For $E \in \aA_X$, we have a functorial isomorphism
	$\Psi \circ \Phi(E) \stackrel{\cong}{\to} E$. 
\end{lem}
\begin{proof}
	The object $\Psi \circ \Phi(E)$ is given by the two 
	term complex
	concentrated in degree $[-1, 0]$
	\begin{align}\label{eta:E}
		\left(\overline{\pi}^{\ast} \dR \overline{\pi}_{\ast}
		E(-S_0-S_{\infty}) \stackrel{\eta}{\to} 
		\overline{\pi}^{\ast} (\dR \overline{\pi}_{\ast}
		E(-S_{\infty}))(S_{\infty}) \right).
	\end{align}
	Here $\eta$ is adjoint to the map
	\begin{align}\notag
		\dR \overline{\pi}_{\ast}
		E(-S_0-S_{\infty})
		&\to \dR \overline{\pi}_{\ast}E(-S_{\infty}) \otimes (\oO_S \oplus \omega_S^{-1}) \\\notag&\stackrel{\cong}{\to}\dR \overline{\pi}_{\ast}E(-S_{0}) \oplus 
		\dR \overline{\pi}_{\ast}E(-S_{\infty})
	\end{align}
	induced by the sum of the natural inclusions
	$\oO_{\overline{X}}(-S_0-S_{\infty}) \to 
	\oO_{\overline{X}}(-S_0) \oplus \oO_{\overline{X}}(-S_{\infty})$. 
	On the other hand, we have the Koszul resolution of the 
	diagonal 
	$\Delta \subset \overline{X} \times_S \overline{X}$
	\begin{align*}
		0 \to \Omega_{\overline{X}/S}(1) \boxtimes \oO_{\overline{X}}(-1)
		\to \oO_{\overline{X}\times_S \overline{X}} \to \oO_{\Delta} \to 0. 
	\end{align*}
	Note that by (\ref{O(S)}), we have 
	\begin{align*}
		\Omega_{\overline{X}/S}(1) \boxtimes \oO_{\overline{X}}(-1)=
		\oO_{\overline{X}}(-S_0) \boxtimes \oO_{\overline{X}}(-S_{\infty}).
	\end{align*}
	By pulling back $E(-S_{\infty})$
	to $\overline{X} \times_S \overline{X}$
	by the first projection, tensoring with 
	the above exact sequence and pushing forward
	by the second projection, we obtain the distinguished triangle
	\begin{align*}
		\overline{\pi}^{\ast} \dR \overline{\pi}_{\ast}
		E(-S_0-S_{\infty}) \to
		\overline{\pi}^{\ast} (\dR \overline{\pi}_{\ast}
		E(-S_{\infty}))(S_{\infty})
		\to E. 
	\end{align*}
	It is straightforward to check that the first arrow is 
	given by $\eta$ in (\ref{eta:E}). 
	Therefore the lemma holds. 
\end{proof}

\begin{lem}\label{lem:funct:Psi2}
	The functor $\Psi$ in (\ref{funct:Psi}) restricts to the functor 
	\begin{align*}
		\Psi \colon \bB_S^{\le 1} \to \aA_X^{\le 1}. 
	\end{align*}
\end{lem}
\begin{proof}
	For a diagram (\ref{diagram:BS}),
	let $E$ be the two term complex 
	(\ref{Psi(dia)})
	\begin{align*}
		E=\left(\overline{\pi}^{\ast}\uU \stackrel{\eta}{\to} 
		\overline{\pi}^{\ast}F \otimes \oO_{\overline{X}}(S_{\infty}) \right). 
	\end{align*}
	Below we show that if the rank of the diagram (\ref{diagram:BS}) is less than 
	or equal to one, then we have $E \in \aA_{X}^{\le 1}$. 
	Note that $E$ is concentrated in degrees $[0, 1]$. 
	By restricting $E$ to $S_{\infty}$, 
	we have the distinguished triangle
	\begin{align*}
		\dL i_{\infty}^{\ast}E \to \uU \stackrel{\psi}{\to}
		F \otimes \omega_S^{-1}. 
	\end{align*} 
	The above sequence is isomorphic to the top 
	sequence in (\ref{diagram:BS}), therefore 
	we have 
	\begin{align}\label{E:infty}
		\dL i_{\infty}^{\ast}E =\vV \in \langle \oO_S \rangle_{\rm{ex}}.
	\end{align}
	In particular, 
	$\hH^1(E)=\Cok(\eta)$ is zero on $S_{\infty}$, so 
	it is supported away from $S_{\infty}$. 
	On the other hand, we have the surjection 
	$\overline{\pi}^{\ast}F(S_{\infty}) \twoheadrightarrow \Cok(\eta)$, 
	so we have
	\begin{align*}
		\Supp(\hH^1(E)) \subset \overline{\pi}^{-1}(\Supp(F)).
	\end{align*}
	Therefore if $\hH^1(E)$ has two dimensional support, it is of the form 
	$\overline{\pi}^{-1}(Z)$ for a one dimensional support $Z \subset \Supp(F)$. 
	It contradicts to that $\hH^1(E)$ is supported away from $S_{\infty}$, 
	so  
	$\hH^1(E)$ is at most one dimensional and we have  
	$\hH^1(E) \in \Coh_{\le 1}(X)$. 
	
	Next we show that $\hH^0(E)=\Ker(\eta)$ is a torsion free sheaf on 
	$\overline{X}$. 
	By (\ref{E:infty}), the sheaf $\hH^0(E)$ is a locally free sheaf 
	in a neighborhood of $S_{\infty}$. Suppose that 
	$\hH^0(E)$ has a torsion subsheaf. Then it is a subsheaf of the 
	torsion part of $\overline{\pi}^{\ast}\uU$, therefore its 
	support is of the form $\overline{\pi}^{-1}(Z)$ for $Z \subset S$
	with $\dim Z \le 1$. It contradicts to that 
	$\hH^0(E)$ is locally free near $S_{\infty}$, hence $\hH^0(E)$ is torsion 
	free. 
	
	Now suppose that the diagram (\ref{diagram:BS}) has rank one, so 
	$E$ is of rank one. As $\hH^0(E)$ is torsion free, we have the 
	exact sequence of sheaves
	\begin{align*}
		0 \to \hH^0(E) \to \hH^0(E)^{\vee\vee} \to Q \to 0
	\end{align*}
	for $Q \in \Coh_{\le 1}(\overline{X})$. 
	As $\hH^0(E)$ is a line bundle in a neighborhood of $S_{\infty}$, 
	the sheaf $Q$ is supported away from $S_{\infty}$, i.e. 
	$Q \in \Coh_{\le 1}(X)$. 
	We also have an isomorphism 
	$\hH^0(E)^{\vee\vee} \cong \oO_{\overline{X}}$, as it is 
	a rank one reflexive sheaf. 
	Now the object 
	$E$ is filtered by objects $Q[-1]$, $\oO_{\overline{X}}$, $\hH^1(E)[-1]$, 
	so it is an object in $\aA_X^{\le 1}$. 
	The rank zero case is easier and we omit details. 
\end{proof}

\begin{thm}\label{thm:equiv}
	The functor $\Phi \colon \aA_X^{\le 1} \to \bB_S^{\le 1}$
	is an equivalence of categories whose quasi-inverse is given 
	by $\Psi \colon \bB_S^{\le 1} \to \aA_X^{\le 1}$. 
\end{thm}
\begin{proof}
	By Lemma~\ref{lem:isom1}, it is enough to show that there exists
	an isomorphism of functors $\Phi \circ \Psi \cong \id$. 
	For a diagram (\ref{diagram:BS}), let $E$ be the object given by 
	(\ref{Psi(dia)}). 
	Then we have the distinguished triangle
	\begin{align*}
		\overline{\pi}^{\ast}F \otimes \oO_{\overline{X}}(S_{\infty})[-1]
		\to E \to \overline{\pi}^{\ast}\uU.
	\end{align*}
	Using the above triangle, we see that 
	there exist natural isomorphisms
	\begin{align*}
		\dR \overline{\pi}_{\ast}(E(-S_{0}-S_{\infty})[1])
		\stackrel{\cong}{\to} \uU \otimes \dR \overline{\pi}_{\ast}
		\oO_{\overline{X}}(-S_0 -S_{\infty})[1] \stackrel{\cong}{\to}
		\uU. 
	\end{align*}
	Similarly we have 
	isomorphisms
	\begin{align*}
		\vV \stackrel{\cong}{\to} 
		\dR \overline{\pi}_{\ast}(E|_{S_{\infty}}), \ 
		F \stackrel{\cong}{\to}
		\dR \overline{\pi}_{\ast}(E(-S_{\infty})[1]), \ 
		F\otimes \omega_S^{-1} \stackrel{\cong}{\to}
		\dR \overline{\pi}_{\ast}(E(-S_{0})[1]).
	\end{align*}
	Then it is straightforward to check that 
	the diagram (\ref{diagram:Phi(E)}) is isomorphic to the original 
	diagram (\ref{diagram:BS}). 
	This shows that $\Phi \circ \Psi \cong \id$. 
\end{proof}

Finally in this subsection, we give a 
characterization of rank one objects in 
$\aA_X$ given by pairs $(\oO_X \to F)$
in terms of objects in $\bB_S$: 
\begin{lem}\label{lem:pair}
	For a pair $(\oO_X \to F)$ with 
	$F \in \Coh_{\le 1}(X)$, it determines an
	object $(\oO_{\overline{X}} \to F)$ in $\aA_X^{\le 1}$. 
	Under the equivalence in Theorem~\ref{thm:equiv},
	a rank one diagram (\ref{diagram:BS}) 
	corresponds to an object in $\aA_X$ of the form
	$(\oO_{\overline{X}} \to F)$
	if and only if the top sequence of (\ref{diagram:BS})
	splits. 
\end{lem}
\begin{proof}
	For a pair $(\oO_X \to F)$ with $F \in \Coh_{\le 1}(X)$, we have 
	the exact sequence in $\aA_X$
	\begin{align*}
		0 \to F[-1] \to (\oO_{\overline{X}} \to F) \to \oO_{\overline{X}} \to 0.
	\end{align*}
	The first statement follows from the above sequence. 
	By the above exact sequence, it is clear that a rank one object
	$E \in \aA_X$ is of the form $(\oO_{\overline{X}} \to F)$
	if and only if it admits a surjection 
	$E \twoheadrightarrow \oO_{\overline{X}}$
	in $\aA_{X}$. 
	Therefore via the equivalence in Theorem~\ref{thm:equiv}, 
	such an object corresponds to a 
	diagram (\ref{diagram:BS}) with $\vV=\oO_S$,
	which admits a surjection to the following diagram
	\begin{align*}
		\xymatrix{
			0 \ar[r] & \oO_S \ar[r]^{\id}  & \oO_S \ar[r] \ar[d]& 0 \ar[r] & 0 \\
			& & 0.  & &
		}
	\end{align*}
	The latter condition is equivalent to that 
	the top sequence of the 
	diagram (\ref{diagram:BS}) with $\vV=\oO_S$ splits. 
\end{proof}

\subsubsection{Moduli stacks of rank one objects in $\aA_X$}
Let $\mM_X^{\dag}$ be the moduli stack 
of rank one objects in $\aA_X$, 
defined in Subsection~\ref{subsec:moduliD026}. 
Also let $\mM_S^{\dag}$ be the moduli stack of 
pairs $(\oO_S \to F)$ considered in Subsection~\ref{subsec:pairs}, 
and denote by 
$\mM_{S, \le 1}^{\dag} \subset \mM_S^{\dag}$
the open and closed open substack of pairs $(\oO_S \to F)$
such that $F \in \Coh_{\le 1}(S)$. 
A family version 
of the arguments in the previous section 
immediately yields the 
following corollary of Theorem~\ref{thm:equiv}:
\begin{cor}\label{lem:moduli:isom}
	The stack $\mM_X^{\dag}$ is isomorphic to 
	the stack whose $T$-valued points
	for an affine $\mathbb{C}$-scheme $T$
	form  
	the groupoid of diagrams
	\begin{align}\label{diagram:BST}
		\xymatrix{
			0 \ar[r] & \oO_{S \times T}
			\ar[r] \ar@{.>}[rd]_-{\xi} & \uU_T \ar[r]^-{\psi}
			\ar[d]^-{\phi} & F_T \boxtimes \omega_S^{-1} 
			\ar[r] & 0 \\
			&   & F_T.   & 
		}
	\end{align}
	Here the top arrow is an exact sequence 
	of sheaves on $S \times T$ and 
	$F_T \in \Coh(S \times T)$ is $T$-flat such that 
	$F|_{S \times \{x\}} \in \Coh_{\le 1}(S)$ for any 
	closed point $x \in T$. 
	The isomorphisms are given by 
	\begin{align}\notag
		\xymatrix{
			0 \ar[r] & \oO_{S \times T}\ar[d]_-{\id} \ar[r] & \uU_T \ar[d]_-{a_{\uU}}^-{\cong}\ar[r]^-{\psi} & F_T \boxtimes \omega_S^{-1} 
			\ar[d]_-{a_{F} \otimes \id}^-{\cong} \ar[r] & 0 \\
			0 \ar[r] & \oO_{S\times T} \ar[r] & \uU_T' \ar[r]^-{\psi'} & F_T' \boxtimes \omega_S^{-1} \ar[r] & 0,
		} \quad 
		\xymatrix{
			\uU_T \ar[d]_-{a_{\uU}}^-{\cong}\ar[r]^-{\phi} & F_T 
			\ar[d]_-{a_F}^-{\cong}  \\
			\uU_T' \ar[r]^-{\phi'} & F_T'.
		}
	\end{align}
		In particular by sending a diagram (\ref{diagram:BST})
	to the pair $\xi \colon \oO_{S \times T} \to F_T$, we obtain 
	the natural morphism
	\begin{align}\notag
		\pi_{\ast}^{\dag} \colon 
		\mM_X^{\dag} \to \mM_{S, \le 1}^{\dag}. 
	\end{align}
\end{cor}

Let $\eE^{\dag \bullet}$ be the perfect obstruction 
theory on $\mM_{S, \le 1}^{\dag}$ given in (\ref{perf:obs2}). 
We consider the associated dual obstruction cone
\begin{align*}
	\mathrm{Obs}^{\ast}(\eE^{\dag \bullet}) \to \mM_{S, \le 1}^{\dag}.
\end{align*}

\begin{lem}\label{lem:cone}
	For an affine 
	$\mathbb{C}$-scheme $T$, 
	the $T$-valued points of $\mathrm{Obs}^{\ast}(\eE^{\dag \bullet})$ 
	form the groupoid of data 
	\begin{align}\notag
		\left\{ (\oO_{S \times T} \stackrel{\xi}{\to}
		F_T), \  \xi' \colon F_T \boxtimes \omega_S^{-1}
		\to I_{T}^{\bullet}[1] \right\}
	\end{align}
	where $(\oO_{S\times T} \stackrel{\xi}{\to} F_T)$
	gives a $T$-valued point of $\mM_{S, \le 1}^{\dag}$, 
	$I_T^{\bullet}=(\oO_{S\times T} \stackrel{\xi}{\to} F_T)$ is a two 
	term complex with $\oO_{S \times T}$ located in degree zero, and $\xi'$
	is a morphism in $D^b_{\rm{coh}}(S \times T)$. 
\end{lem}
\begin{proof}
	By its definition, the
	$T$-valued points of 
	$\mathrm{Obs}^{\ast}(\eE^{\dag\bullet})$ 
	form the groupoid of data 
	\begin{align*}
		\left\{ f \colon T \to \mM_{S, \le 1}^{\dag}, \ f^{\ast}\hH^1((\eE^{\dag\bullet})^{\vee}) \to \oO_T  \right\}. 
	\end{align*}
	Since $(\eE^{\dag\bullet})^{\vee}$ is perfect of cohomological amplitude $[-1, 1]$, 
	we have 
	$f^{\ast}\hH^1((\eE^{\dag\bullet})^{\vee})=
	\hH^1(\dL f^{\ast}(\eE^{\dag\bullet})^{\vee})$. 
	Also we have  
	\begin{align*}
		\dL f^{\ast}\eE^{\dag\bullet} &=
		\left(\dR p_{T\ast} \dR \hH om_{S \times T}(I_{T}^{\bullet}, F_T)\right)^{\vee}
		\\
		&\cong \dR p_{T\ast} \dR \hH om_{S \times T}(F_T \boxtimes \omega_S^{-1}, I_T^{\bullet}[2]).
	\end{align*}
	Here $p_T \colon S \times T \to T$ is the projection and 
	the last isomorphism follows from the Grothendieck duality. 
	Using above, we see that 
	\begin{align*}
		\Hom(f^{\ast}\hH^1((\eE^{\dag\bullet})^{\vee}), \oO_T) & =
		\Hom(\hH^1(\dL f^{\ast}(\eE^{\dag\bullet})^{\vee}), \oO_T) \\
		&=\Hom(\dL f^{\ast} (\eE^{\dag\bullet})^{\vee}, \oO_T[-1]) \\
		&= \Hom(\oO_T[1], \dL f^{\ast} \eE^{\dag\bullet}) \\
		&=\Hom(F_T \boxtimes \omega_S^{-1}, I_T^{\bullet}[1]). 
	\end{align*}
	Therefore the lemma holds. 
\end{proof}

Now we finish the proof of Theorem~\ref{thm:D026}:  
\begin{prop}\label{thm:stack}
	We have an isomorphism over $\mM_{S, \le 1}^{\dag}$
	\begin{align*}
		\mM_X^{\dag}\stackrel{\cong}{\to} \mathrm{Obs}^{\ast}(\eE^{\dag\bullet}). 
	\end{align*}
\end{prop}
\begin{proof}
	For each affine $\mathbb{C}$-scheme $T$, 
	a $T$-valued point of $\mM_X^{\dag}$ is identified with a diagram 
	(\ref{diagram:BST}) by Lemma~\ref{lem:moduli:isom}.
	A $T$-flat version of the argument of Lemma~\ref{diagram:equivalent}
	shows that giving a diagram (\ref{diagram:BST}) is equivalent to 
	giving a $T$-valued point 
	$(\oO_{S \times T} \stackrel{\xi}{\to} F_T)$
	of $\mM_{S, \le 1}^{\dag}$ together with a morphism in $D^b_{\rm{coh}}(S \times T)$
	\begin{align*}
		\xi' \colon F_T \boxtimes \omega_S^{-1} \to I_T^{\bullet}[1], \ 
		I_T^{\bullet}=(\oO_{S \times T} \stackrel{\xi}{\to} F_T)
	\end{align*}
	where $\oO_{S \times T}$ located in degree zero. 
	By Lemma~\ref{lem:cone}, this is equivalent to giving 
	a $T$-valued point of $\mathrm{Obs}^{\ast}(\eE^{\dag\bullet})$. 
\end{proof}

\subsubsection{Comparison of dualities}
Let 
$\aA_X^{\rm{pure}} \subset \aA_X$ be
defined in (\ref{def:Apure}), and 
$\bB_S^{\rm{pure}} \subset \bB_S$
the subcategory consisting of diagrams (\ref{diagram:BS})
such that $F$ is a pure one dimensional sheaf. 
We set
\begin{align*}
	\aA_{X}^{\rm{pure}, \le 1} \cneq \aA_X^{\le 1} \cap \aA_X^{\rm{pure}}, \ 
	\bB_S^{\rm{pure}, \le 1} \cneq \bB_S^{\le 1} \cap \bB_S^{\rm{pure}}.
\end{align*}
We will give two lemmas on the above subcategories.

\begin{lem}\label{lem:onepure:AX}
	An object $E \in \aA_X^{\le 1}$ is an object 
	in $\aA_X^{\rm{pure}, \le 1}$ if and only if 
	$\Hom(Q[-1], E)=0$ for any zero dimensional sheaf $Q$ on $X$. 
\end{lem}
\begin{proof}
	The only if direction is obvious. As for the if direction, 
	the statement is obvious for rank zero objects. 
	Let us take a rank one
	object $E \in \aA_X$ 
	satisfying $\Hom(Q[-1], E)=0$ for any zero 
	dimensional sheaf $Q$ on $X$, and take the maximal 
	zero dimensional subsheaf
	$T \subset \hH^1(E)$. 
	Then $F=\hH^1(E)/T$
	is a pure one dimensional sheaf, and we have 
	the surjection $E \twoheadrightarrow F[-1]$ in $\aA_X$. 
	Let $E'$ be the kernel of the above surjection. 
	Then $\hH^0(E')$ is zero dimensional, 
	hence Lemma~\ref{lem:MNOP/PTwall}
	implies that $E'$ is isomorphic to a two term 
	complex $(\oO_{\overline{X}} \to F')$ for a one dimensional 
	sheaf $F'$. If $F'$ is not pure, then 
	$E'$ contains a subobject of the form $Q[-1]$ 
	for a zero dimensional sheaf $Q$, which contradicts to the 
	assumption. 
	Therefore $F'$ is pure, and $E \in \aA_X^{\rm{pure}, \le 1}$ holds. 
\end{proof}

\begin{lem}\label{lem:pure:Phi}
	The equivalence $\Phi$ in Theorem~\ref{thm:equiv}
	restricts to the equivalence 
	$\Phi \colon \aA_{X}^{\rm{pure}, \le 1} \stackrel{\sim}{\to} 
	\bB_S^{\rm{pure}, \le 1}$. 
\end{lem}
\begin{proof}
	The construction of $\Phi$ obviously implies that 
	it takes $\aA_X^{\rm{pure}, \le 1}$ to $\bB_S^{\rm{pure}, \le 1}$. 
	Note that we have 
	\begin{align*}
		\Hom(\Phi(Q[-1]), \bB_S^{\rm{pure}, \le 1})=\Hom(Q[-1], \Psi(\bB_S^{\rm{pure}, \le 1}))=0
	\end{align*}
	for any zero dimensional sheaf $Q$ on $X$.
	Therefore by Lemma~\ref{lem:onepure:AX}, 
	the functor $\Psi$ in Lemma~\ref{lem:funct:Psi2}
	sends $\bB_S^{\rm{pure}, \le 1}$ to $\aA_X^{\rm{pure}, \le 1}$. 
\end{proof}

The category $\aA_X^{\rm{pure}, \le 1}$ is closed under
the derived dual $\mathbb{D}_{\overline{X}}$. 
Below we describe the corresponding duality on $\bB_S^{\rm{pure}, \le 1}$
under the equivalence in Lemma~\ref{lem:pure:Phi}. 
For a diagram (\ref{diagram:BS}) such that $F$ is a pure one 
dimensional sheaf, let us apply
the derived dual functor $\mathbb{D}_S(-)$ on it. 
Then 
we obtain the diagram
\begin{align*}
	\xymatrix{
		\vV^{\vee} & \ar[l] \mathbb{D}_S(\uU) & \ar[l] F^{\vee}[-1] \\
		&  F^{\vee} \otimes \omega_S^{-1}[-1] \ar[u] \ar@{.>}[lu]^-{\xi^{\vee}}&
	}
\end{align*}
where $F^{\vee}$ is given as in (\ref{F:dual}). 
By taking the cones, we obtain the following diagram
\begin{align}\label{dual:diagram}
	\xymatrix{
		0\ar[r] & \vV^{\vee} \ar[r] & \Cone(\xi^{\vee}) \ar[r] \ar[d]
		& F^{\vee} \otimes \omega_S^{-1} \ar[r] & 0 \\
		& & F^{\vee}. & &
	}
\end{align}
Here the top sequence corresponds to the extension 
class $\xi^{\vee}$. 
By sending a diagram (\ref{diagram:BS})
to the diagram (\ref{dual:diagram}), we 
obtain the functor
\begin{align*}
	\mathbb{D}_{\bB} \colon 
	\bB_S^{\rm{pure}} \to (\bB_S^{\rm{pure}})^{\rm{op}}. 
\end{align*}
The above functor obviously preserves
the subcategory $\bB_S^{\rm{pure}, \le 1}$. 
It is easy to see that 
$\mathbb{D}_{\bB} \circ \mathbb{D}_{\bB}=\id$, 
so in particular $\mathbb{D}_{\bB}$ is an equivalence. 

\begin{lem}\label{prop:duality}
	The following diagram is commutative:
	\begin{align*}
		\xymatrix{
			\aA_X^{\rm{pure}, \le 1} \ar[r]^-{\mathbb{D}_{\overline{X}}} \ar[d]_-{\Phi} & 
			(\aA_X^{\rm{pure}, \le 1})^{\rm{op}} \ar[d]^-{\Phi} \\
			\bB_S^{\rm{pure}, \le 1} \ar[r]^-{\mathbb{D}_{\bB}} &
			(\bB_S^{\rm{pure}, \le 1})^{\rm{op}}. 
		}
	\end{align*} 
\end{lem}
\begin{proof}
	We only check the commutativity for rank 
	one object $E \in \aA_{X}^{\rm{pure}, \le 1}$. 
	The diagram $\mathbb{D}_{\bB} \circ \Phi(E)$ is
	by the construction 
		\begin{align}\notag
			\xymatrix{
			\oO_S
				\ar[r] & \mathbb{D}_S \circ \dR \overline{\pi_{\ast}}(E) \ar[r] \ar[d]
				& \mathbb{D}_S \circ \dR \overline{\pi_{\ast}}(E(-S_{\infty})[1]) \\
				& \mathbb{D}_S \circ \dR \overline{\pi_{\ast}}(E(-S_{0})[1]). &
			} 
		\end{align}
By the Grothendieck duality 
$\dR \overline{\pi}_{\ast}(\mathbb{D}_X(-) \otimes \Omega_{\overline{X}/S}[1])
\cong \mathbb{D}_S \circ \dR \overline{\pi}_{\ast}(-)$, 
the above diagram is nothing but $\Phi \circ \mathbb{D}_{\overline{X}}(E)$.  
\end{proof}

\subsection{Semiorthogonal decomposition via Koszul duality}\label{subsec:sod:koszul}
Here we establish a general statement about semiorthogonal 
decomposition associated with linear Koszul duality 
studied in~\cite{MrRi, MrRi2}, which is required for the proof of 
Theorem~\ref{thm:PT:irredu}. 
We first review linear Koszul duality
in the local case, globalize it, and then show the 
existence of certain semiorthogonal decomposition under the linear Koszul duality. We also prove 
a certain comparison result of 
singular supports under linear Koszul duality, which is essential 
for the proof of Corollary~\ref{cor:duality}. 
\subsubsection{Linear Koszul duality: local case}
Here we review the linear Koszul duality 
proved in~\cite{MrRi, MrRi2}, and 
prove its variants. 
Suppose that $Y$ is a smooth affine $\mathbb{C}$-scheme, 
and take a two term complex of vector bundles on it
\begin{align}\label{vect:vV}
\eE^{{}}=(\eE^{-1} \stackrel{\phi}{\to} \eE^0).
\end{align}
We take a trivial $\mathbb{C}^{\ast}$-action on $Y$
and weight two $\mathbb{C}^{\ast}$-actions on $\eE^{-1}$ and $\eE^0$. 
Let $\aA^{\dag}$, 
$\aA^{\natural}$ and $\aA^{\sharp}$ be the sheaves of dg-algebras on $Y$
defined by
\begin{align*}
\aA^{\dag} \cneq S_{\oO_Y}(\eE), \ \aA^{\natural} \cneq S_{\oO_Y}(\eE^{\vee}[-1]), \ 
\aA^{\sharp} \cneq S_{\oO_Y}(\eE^{\vee}[1]).
\end{align*}
For $\star \in \{\dag, \natural, \sharp\}$,
we denote by 
$D_{\rm{fg}}^{\mathbb{C}^{\ast}}(\aA^{\star} \mathchar`-\mathrm{mod})$ 
the $\mathbb{C}^{\ast}$-equivariant 
derived category of finitely generated dg-modules over $\aA^{\star}$. 
\begin{thm}\emph{(\cite{MrRi, MrRi2})}\label{thm:lkoszul}
There exists an equivalence
\begin{align}\label{equiv:koszul}
\overline{\Phi}_Y^{\rm{op}} \colon 
D_{\rm{fg}}^{\mathbb{C}^{\ast}}(\aA^{\dag} \mathchar`-\mathrm{mod})^{\rm{op}}
\stackrel{\sim}{\to} 
D_{\rm{fg}}^{\mathbb{C}^{\ast}}(\aA^{\sharp} \mathchar`-\mathrm{mod})
\end{align}
sending $\aA^{\dag}$ to $\oO_Y$ and $\oO_Y$ to $\aA^{\sharp}$. 
\end{thm}
Below we recall the construction of the equivalence (\ref{equiv:koszul})
and its quasi-inverse. 
For a $\mathbb{C}^{\ast}$-equivariant 
dg-module $M$ over $\aA^{\dag}$, 
we write its $\mathbb{C}^{\ast}$-weight $j$ part by $M_j$
which is a direct summand of $M$ as $\oO_Y$-module. 
Then its $\oO_Y$-dual 
$\mathbb{D}_Y(M) \cneq \oplus_{j} \mathbb{D}_Y(M_j)$ 
is naturally equipped with 
a $\mathbb{C}^{\ast}$-equivariant 
dg-module structure over $\aA^{\dag}$, where 
$\mathbb{D}_Y(M_j)$ is of $\mathbb{C}^{\ast}$-weight $-j$. 
We set
\begin{align*}
\kK \cneq \aA^{\natural}\otimes_{\oO_Y} \aA^{\dag}, \ 
\kK' \cneq \mathbb{D}_Y(\aA^{\dag}) \otimes_{\oO_Y} \aA^{\natural}
\end{align*}
with differentials given by 
\begin{align*}
d_{\kK}=d_{\aA^{\natural}\otimes_{\oO_Y} \aA^{\dag}}+\eta, \ 
d_{\kK'}=d_{\mathbb{D}_Y(\aA^{\dag})\otimes_{\oO_Y} \aA^{\natural}}+\eta'
\end{align*}
where 
$\eta \in \eE^{\vee} \otimes \eE$,
$\eta' \in \eE \otimes \eE^{\vee}$
are degree one elements of 
$\aA^{\natural}\otimes_{\oO_Y} \aA^{\dag}$, 
$\aA^{\dag} \otimes_{\oO_Y} \aA^{\natural}$, 
corresponding to  
the identity map $\id \colon \eE \to \eE$. 
Then both of $\kK$, $\kK'$ are dg-modules over 
$\aA^{\natural}\otimes_{\oO_Y} \aA^{\dag}$, 
$\aA^{\dag} \otimes_{\oO_Y} \aA^{\natural}$, 
respectively. 
For a $\mathbb{C}^{\ast}$-equivariant dg-module $M$
over $\aA^{\dag}$ and a $\mathbb{C}^{\ast}$-equivariant dg-module $N$
over $\aA^{\natural}$,
we set
\begin{align*}
\Phi'_Y(M)=\kK \otimes_{\aA^{\dag}} M, \ 
\Psi'_Y(N)=\kK' \otimes_{\aA^{\natural}}N. 
\end{align*} 
Then $\Phi'_Y(M)$ is a dg-module over $\aA^{\natural}$, 
and $\Psi'_Y(N)$ is a dg-module over $\aA^{\dag}$. 
\begin{rmk}\label{rmk:Morita}
The functor $\Phi'_Y$ is described in another way as follows. 
Let $\overline{\kK}=S(\eE[1]) \otimes_{\oO_Y}\aA^{\dag}$
be the Koszul resolution of $\oO_Y$ by free 
$\aA^{\dag}$-modules. 
Then we have 
\begin{align}\notag
\Phi_Y'(M)=\Hom_{\aA^{\dag}}(\overline{\kK}, M)=\RHom_{\aA^{\dag}}(\oO_Y, M),
\end{align}
i.e. $\Phi_Y'$ is a derived Morita-type functor
with respect to $\oO_Y$. 
\end{rmk}
Finally there is a regrading equivalence $\xi$
\begin{align}\label{regrading}
\xi \colon 
D_{\rm{fg}}^{\mathbb{C}^{\ast}}(\aA^{\natural} \mathchar`-\mathrm{mod})
\stackrel{\sim}{\to}
D_{\rm{fg}}^{\mathbb{C}^{\ast}}(\aA^{\sharp} \mathchar`-\mathrm{mod})
\end{align}
sending a $\mathbb{C}^{\ast}$-equivariant 
$\aA^{\natural}$-module $M$ 
to the $\mathbb{C}^{\ast}$-equivariant 
$\aA^{\sharp}$-module 
$\xi(M)$ given by 
$\xi(M)_{j}^i=M_{j}^{i-j}$, where 
$i$ is the cohomological grading and $j$ is the 
$\mathbb{C}^{\ast}$-weight. 
The equivalence $\overline{\Phi}_Y^{\rm{op}}$ and its quasi-inverse
$\overline{\Psi}_Y^{\rm{op}}$
are given by
\begin{align*}
\overline{\Phi}_Y^{\rm{op}}=\xi \circ \Phi_Y' \circ \mathbb{D}_Y, \ 
\overline{\Psi}_Y^{\rm{op}} \cneq 
\mathbb{D}_Y \circ \Psi_Y' \circ \xi^{-1}. 
\end{align*}

We will give some variants of Theorem~\ref{thm:lkoszul}. 
We set the following affine derived schemes over $Y$: 
\begin{align}\notag
Y^{\dag} \cneq \Spec \aA^{\dag}, \ Y^{\sharp} 
\cneq \Spec \aA^{\sharp}. 
\end{align}
Instead of the weight two action of 
$\mathbb{C}^{\ast}$ on $\eE$, we 
take the weight one action on $\eE$, 
and take the quotient stacks
$[Y^{\dag}/\mathbb{C}^{\ast}]$, 
$[Y^{\sharp}/\mathbb{C}^{\ast}]$.
\begin{prop}\label{thm:koszul}
There is an equivalence
\begin{align}\label{equiv:koszul:A}
\Phi_Y^{\rm{op}} \colon D^b_{\rm{coh}}([Y^{\dag}/\C])^{\rm{op}}
\stackrel{\sim}{\to} D^b_{\rm{coh}}([Y^{\sharp}/\C]). 
\end{align}
\end{prop}
\begin{proof}
We denote by 
$[Y^{\dag} \sslash \C]$, 
$[Y^{\sharp} \sslash \C]$ the stack 
quotients by the weight two $\mathbb{C}^{\ast}$-actions on $\eE^{\bullet}$. 
The result of Theorem~\ref{thm:lkoszul}
means the equivalence
\begin{align}\label{equiv:koszul:2}
\overline{\Phi}_Y^{\rm{op}} \colon 
D^b_{\rm{coh}}([Y^{\dag} \sslash \C])^{\rm{op}}
\stackrel{\sim}{\to} D^b_{\rm{coh}}([Y^{\sharp} \sslash \C]). 
\end{align}
Now the natural map
$\mathbb{C}^{\ast} \to \mathbb{C}^{\ast}$
given by $t \mapsto t^2$ induces the map
\begin{align*}
\rho \colon [Y^{\star} \sslash \mathbb{C}^{\ast}] \to [Y^{\star}/\mathbb{C}^{\ast}], \ 
\star \in \{\dag, \sharp\}
\end{align*}
which is a $\mu_2$-gerbe. 
Then similarly to (\ref{decompose:chi})
we have the decomposition
\begin{align}\label{decompose:mu2}
D^b_{\rm{coh}}([Y^{\star} \sslash \mathbb{C}^{\ast}])=
D^b_{\rm{coh}}([Y^{\star} \sslash \mathbb{C}^{\ast}])_{\lambda=0}
\oplus D^b_{\rm{coh}}([Y^{\star} \sslash \mathbb{C}^{\ast}])_{\lambda=1} 
\end{align}
such that $D^b_{\rm{coh}}([Y^{\star} \sslash \mathbb{C}^{\ast}])_{\lambda=0}$
is equivalent to $D^b_{\rm{coh}}([Y^{\star}/\mathbb{C}^{\ast}])$
via $\rho^{\ast}$. The equivalence (\ref{equiv:koszul:2})
preserves the decomposition (\ref{decompose:mu2}),
as it preserves the parity of the $\mathbb{C}^{\ast}$-weights. 
Therefore we 
obtain the equivalence (\ref{equiv:koszul:A}). 
Explicitly, $\Phi_Y^{\rm{op}}$ and its quasi-inverse $\Psi_Y^{\rm{op}}$ are given by
\begin{align}\label{equiv:covariant}
\Phi_Y^{\rm{op}}=\xi' \circ \Phi_Y' \circ \mathbb{D}_{Y}, \ 
\Psi_Y^{\rm{op}}=\mathbb{D}_Y \circ \Psi_Y' \circ \xi'. 
\end{align}
Here the regrading equivalence (\ref{regrading}) is replaced by
$\xi'$ determined by $\xi'(M)_{j}^i=M_{j}^{i-2j}$. 
\end{proof}

Let us take an affine derived scheme 
$\mathfrak{U}=\Spec \rR(V\to Y, s)$ 
for a vector bundle $V \to Y$ with a section $s$ as in 
(\ref{frak:U}).
We
take $\mathfrak{E}=\eE \otimes_{\oO_Y}\oO_{\mathfrak{U}}$
where $\eE$ is a two term complex of vector bundles on $Y$
as in (\ref{vect:vV}), which is 
equipped
with a weight one $\mathbb{C}^{\ast}$-action. 
We set 
\begin{align*}
\mathfrak{U}^{\dag} \cneq 
\Spec S_{\oO_{\mathfrak{U}}}(\mathfrak{E})
=\Spec \aA_{\mathfrak{U}}^{\dag}, \
\mathfrak{U}^{\sharp} \cneq 
\Spec S_{\oO_{\mathfrak{U}}}(\mathfrak{E}^{\vee}[1])
=\Spec \aA_{\mathfrak{U}}^{\sharp}.
\end{align*}
Here 
for a $\oO_Y$-module $(-)$, 
we have written 
$(-)_{\mathfrak{U}}=(-)\dotimes_{\oO_Y} \oO_{\mathfrak{U}}$, and 
we use the same notation below. 
We have the following Cartesian diagrams
for $\star \in \{\dag, \sharp\}$
\begin{align}\label{dia:Udag}
\xymatrix{
\mathfrak{U}^{\star} \ar[r]^-{\rho_{\mathfrak{U}}^{\star}} \ar[d]_-{j^{\star}} \ar@{}[dr]|\square & \mathfrak{U} \ar[d]^-{j} \ar[r]^-{i_{\mathfrak{U}}^{\star}} \ar@{}[dr]|\square & \mathfrak{U}^{\star} 
\ar[d]^-{j^{\star}} \\
Y^{\star}  \ar[r]_-{\rho_Y^{\star}} & Y \ar[r]_-{i_Y^{\star}} & Y^{\star}. 
}
\end{align}
Here $j$ is the closed immersion of affine derived schemes, 
$\rho_{\mathfrak{U}}^{\star}$, $\rho_Y^{\star}$ are projections and 
$i_{\mathfrak{U}}^{\star}$, $i_{Y}^{\star}$ are their zero sections. 
 We have the following slight 
generalization of Proposition~\ref{thm:koszul}:
\begin{lem}\label{prop:koszul:U}
There is an equivalence
\begin{align}\label{equiv:koszul:U}
\Phi_{\mathfrak{U}}^{\rm{op}} \colon 
D^b_{\rm{coh}}([\mathfrak{U}^{\dag}/\mathbb{C}^{\ast}])^{\rm{op}}
\stackrel{\sim}{\to} 
D^b_{\rm{coh}}([\mathfrak{U}^{\sharp}/\mathbb{C}^{\ast}])
\end{align}
which fits into the commutative diagram
\begin{align}\label{dia:DM}
\xymatrix{
D^b_{\rm{coh}}([\mathfrak{U}^{\dag}/\mathbb{C}^{\ast}])^{\rm{op}} \ar[r]^-{\Phi_{\mathfrak{U}}^{\rm{op}}} \ar[d]_-{j^{\dag}_{\ast}} &D^b_{\rm{coh}}([\mathfrak{U}^{\sharp}/\mathbb{C}^{\ast}]) 
 \ar[d]^-{j^{\sharp}_{\ast}} \\
D^b_{\rm{coh}}([Y^{\dag}/\C])^{\rm{op}} \ar[r]_-{\widetilde{\Phi}_{Y}^{\rm{op}}}  &D^b_{\rm{coh}}([Y^{\sharp}/\C]). 
}
\end{align}
Here $\widetilde{\Phi}_Y^{\rm{op}} \cneq \Phi_Y^{\rm{op}} \circ \otimes_{\oO_Y} \det V^{\vee}[-\rank V]$. 
\end{lem}
\begin{proof}
Since 
the equivalence in Proposition~\ref{thm:koszul} 
is defined over $Y$, the equivalence $\Phi_{\mathfrak{U}}$ 
is 
obtained by pulling back the equivalence in Proposition~\ref{thm:koszul}
by the closed immersion of derived schemes
$j \colon \mathfrak{U} \to Y$
(cf.~\cite[Theorem~6.4]{MR2801403}). 

More precisely, the proof is as follows. 
For a dg-module $M$ over $\aA_{\mathfrak{U}}^{\dag}$ 
and a dg-module $N$ over $\aA_{\mathfrak{U}}^{\sharp}$, 
we set
\begin{align*}
\Phi_{\mathfrak{U}}'(M)=
\kK_{\mathfrak{U}} \otimes_{\aA_{\mathfrak{U}}^{\dag}}
M, \ 
\Psi_{\mathfrak{U}}'(N)=
\kK'_{\mathfrak{U}} \otimes_{\aA_{\mathfrak{U}}^{\natural}}
N. 
\end{align*}
Then similarly to (\ref{equiv:covariant}), 
the functors $\Phi_{\mathfrak{U}}^{\rm{op}}$ and $\Psi_{\mathfrak{U}}^{\rm{op}}$ are
defined to be
\begin{align}\label{equiv:U}
\Phi_{\mathfrak{U}}^{\rm{op}}\cneq 
\xi' \circ \Phi_{\mathfrak{U}}'\circ 
\mathbb{D}_{\mathfrak{U}}, \ 
\Psi_{\mathfrak{U}}^{\rm{op}}\cneq 
\mathbb{D}_{\mathfrak{U}}
\circ \Psi_{\mathfrak{U}}'\circ {\xi'}^{-1}.
\end{align}
Here for a dg-module $M$ over $\aA_{\mathfrak{U}}^{\dag}$, 
its $\oO_{\mathfrak{U}}$-dual 
$\mathbb{D}_{\mathfrak{U}}(M)$ is naturally
regarded as a dg-module over $\aA_{\mathfrak{U}}^{\dag}$. 
Note that $\Phi_{\mathfrak{U}}^{\rm{op}}(M)$, $\Psi_{\mathfrak{U}}^{\rm{op}}(N)$
are dg-modules over $\aA_{\mathfrak{U}}^{\sharp}$, 
$\aA_{\mathfrak{U}}^{\dag}$ respectively.
 
We first show that $\Phi^{\rm{op}}_{\mathfrak{U}}$ 
determines a functor (\ref{equiv:koszul:U}), i.e. it preserves the 
coherence and boundedness.  
Let us take an object 
$M \in D^b_{\rm{coh}}([\mathfrak{U}^{\dag}/\mathbb{C}^{\ast}])$. 
By the constructions of $\Phi_{Y}^{\rm{op}}$ and 
$\Phi_{\mathfrak{U}}^{\rm{op}}$, 
using the projection formula 
together with
the Grothendieck duality on $\Dbc(\fU)$
\begin{align*}
	j_{\ast} \circ \mathbb{D}_{\fU}(-)
\cong \mathbb{D}_Y \circ (j_{\ast}(-) \otimes \det V^{\vee}[-\rank V]),
\end{align*}
we see that there is 
a quasi-isomorphism of dg-modules over $\aA^{\sharp}$
\begin{align*}
j_{\ast}^{\sharp} \circ \Phi_{\mathfrak{U}}^{\rm{op}}(M)
\stackrel{\sim}{\to} \Phi_Y^{\rm{op}} \circ (j^{\dag}_{\ast}(M)
\otimes \det V^{\vee}[-\rank V]). 
\end{align*}
Therefore 
by Proposition~\ref{thm:koszul}, 
we have $j^{\sharp}_{\ast} \circ \Phi_{\mathfrak{U}}^{\rm{op}}(M) \in D^b_{\rm{coh}}([Y^{\sharp}/\C])$. 
It follows that we have 
$\Phi_{\mathfrak{U}}^{\rm{op}}(M) \in D^b_{\rm{coh}}([\mathfrak{U}^{\sharp}/\mathbb{C}^{\ast}])$, 
hence we have the functor (\ref{equiv:koszul:U})
and the commutative diagram (\ref{dia:DM}). 

Similarly $\Psi^{\rm{op}}_{\mathfrak{U}}$
determines the functor 
\begin{align*}
\Psi_{\mathfrak{U}}^{\rm{op}} \colon D^b_{\rm{coh}}([\mathfrak{U}^{\sharp}/\mathbb{C}^{\ast}])
\to D^b_{\rm{coh}}([\mathfrak{U}^{\dag}/\mathbb{C}^{\ast}])^{\rm{op}}.
\end{align*}
In order to show that 
$\Psi_{\mathfrak{U}}^{\rm{op}} \circ \Phi_{\mathfrak{U}}^{\rm{op}}$
and $\Phi_{\mathfrak{U}}^{\rm{op}} \circ \Psi_{\mathfrak{U}}^{\rm{op}}$
are identity functors, it is enough to show that
\begin{align*}
\kK'_{\mathfrak{U}} \otimes_{\aA^{\natural}_{\mathfrak{U}}}
\kK_{\mathfrak{U}} \cong 
\aA^{\dag}_{\mathfrak{U}}, \ 
\kK_{\mathfrak{U}} \otimes_{\aA^{\dag}_{\mathfrak{U}}}
\kK_{\mathfrak{U}} \cong 
\aA^{\natural}_{\mathfrak{U}}
\end{align*}
as dg-modules over 
$\aA^{\dag}_{\mathfrak{U}} \otimes_{\oO_{\mathfrak{U}}} \aA^{\dag}_{\mathfrak{U}}$, $\aA^{\natural}_{\mathfrak{U}} \otimes_{\oO_{\mathfrak{U}}} 
\aA^{\natural}_{\mathfrak{U}}$, respectively. 
These isomorphisms
follow by pulling back the isomorphisms
$\kK' \otimes_{\aA^{\natural}} \kK \cong \aA^{\dag}$, 
$\kK \otimes_{\aA^{\dag}} \kK' \cong \aA^{\natural}$
proved in~\cite[Proposition~1.3.1]{MrRi2} by the map 
$j \colon \mathfrak{U} \to Y$
respectively. 
\end{proof}

\subsubsection{Linear Koszul duality: global case}
Let $\mathfrak{M}$ be a quasi-smooth and QCA derived stack
with truncation $\mM=t_0(\mathfrak{M})$. 
Here we assume that it admits an embedding
$(\mathbb{C}^{\ast})_{\fM} \hookrightarrow I_{\fM}$
as in Subsection~\ref{subsec:rig}
so 
that 
we have the 
$\mathbb{C}^{\ast}$-rigidification 
$\mathfrak{M} \to \mathfrak{M}^{\mathbb{C}^{\ast} \rig}$. 
Let us take a perfect object
\begin{align*}
\mathfrak{E} \in \mathrm{Perf}(\mathfrak{M})
\end{align*}
which is of cohomological amplitude $[-1, 0]$, and 
of $\mathbb{C}^{\ast}$-weight one 
with respect to the inertia action. 
We define 
\begin{align*}
\mathfrak{M}^{\dag}=\Spec_{\mathfrak{M}}(S(\mathfrak{E})), \ 
\mathfrak{M}^{\sharp}=\Spec_{\mathfrak{M}}(S(\mathfrak{E}^{\vee}[1])).
\end{align*}
Below we will use the following diagram of derived stacks 
over $\mathfrak{M}^{\mathbb{C}^{\ast} \rig}$
\begin{align}\label{dia:stacs:dag}
\xymatrix{
\mathfrak{M}^{\dag}
\ar[r]_-{\rho^{\dag}} \ar[rd] & \ar@/_5pt/[l]_-{i^{\dag}}
\mathfrak{M} \ar[d] \ar@/^5pt/[r]^-{i^{\sharp}}& \ar[l]^-{\rho^{\sharp}}
 \ar[ld] \mathfrak{M}^{\sharp} \\
&  \mathfrak{M}^{\mathbb{C}^{\ast} \rig}. &
}
\end{align}
Here $\rho^{\dag}$, $\rho^{\sharp}$ are 
projections and $i^{\dag}$, $i^{\sharp}$ are their 
zero sections. 
The following is a global version of Proposition~\ref{thm:koszul}: 
\begin{prop}\label{prop:lKoszul}
There is an equivalence
\begin{align}\label{Phi:global}
\Phi_{\mathfrak{M}}^{\rm{op}} \colon
D^b_{\rm{coh}}(\mathfrak{M}^{\dag})^{\rm{op}} \stackrel{\sim}{\to} 
D^b_{\rm{coh}}(\mathfrak{M}^{\sharp}). 
\end{align}
\end{prop}
\begin{proof}
Let $\mathfrak{U}=\Spec \rR(V\to Y, s)$ be
an affine derived scheme 
as in (\ref{frak:U}), and 
take a smooth morphism 
$\alpha \colon \mathfrak{U} \to \mathfrak{M}^{\mathbb{C}^{\ast}\rig}$
such that 
$\mathfrak{M} \times_{\mathfrak{M}^{\mathbb{C}^{\ast}\rig}} \mathfrak{U}$
is a trivial $\mathbb{C}^{\ast}$-gerbe over $\mathfrak{U}$, i.e. 
it is equivalent to $[\mathfrak{U}/\mathbb{C}^{\ast}]$
where $\mathbb{C}^{\ast}$ acts on $\mathfrak{U}$ trivially.  
Then the pull-back of the diagram (\ref{dia:stacs:dag}) 
by $\alpha \colon \mathfrak{U} \to \mathfrak{M}^{\mathbb{C}^{\ast}\rig}$
is equivalent to the diagram
\begin{align}\label{dia:stacs:dag2}
\xymatrix{
[\mathfrak{U}^{\dag}/\mathbb{C}^{\ast}]
\ar[r] \ar[rd] & 
[\mathfrak{U}/\mathbb{C}^{\ast}] 
 \ar[d]& \ar[l]
 \ar[ld] [\mathfrak{U}^{\sharp}/\mathbb{C}^{\ast}] \\
&  \mathfrak{U}. &
}
\end{align}
By shrinking $\mathfrak{U}$ if necessary, 
we may assume that the object
$\alpha^{\ast}\mathfrak{E}$
on $[\mathfrak{U}/\mathbb{C}^{\ast}]$ 
is quasi-isomorphic to $\eE \otimes_{\oO_Y}\oO_{\mathfrak{U}}$
for a two term complex of vector bundles (\ref{vect:vV})
on $Y$ with 
$\mathbb{C}^{\ast}$-weight one. 
Then 
we have the equivalence $\Phi_{\mathfrak{U}}^{\rm{op}}$
in Lemma~\ref{prop:koszul:U}, 
which naturally lifts
to a weak equivalence of dg-categories by the construction (\ref{equiv:U})
 \begin{align}\label{PhiU:dg}
\Phi_{\mathfrak{U}, \rm{dg}}^{\rm{op}} \colon L_{\rm{coh}}([\mathfrak{U}^{\dag}/\mathbb{C}^{\ast}])^{\rm{op}}
\stackrel{\sim}{\to}  L_{\rm{coh}}([\mathfrak{U}^{\sharp}/\mathbb{C}^{\ast}]). 
\end{align}

Note that for $\star \in \{\dag, \sharp\}$, 
we can write
\begin{align*}
L_{\rm{coh}}(\mathfrak{M}^{\star})=
\lim_{\mathfrak{U} \to \mathfrak{M}^{\mathbb{C}^{\ast} \rig}}
L_{\rm{coh}}([\mathfrak{U}^{\star}/\mathbb{C}^{\ast}]).
\end{align*}
Here the limit is taken for all 
smooth morphisms $\alpha \colon \fU \to \fM^{\C\rig}$. 
We will see that
the dg-functor
$\Phi_{\mathfrak{U}, \rm{dg}}^{\rm{op}}$
 in (\ref{PhiU:dg})
is naturally globalized to a dg-functor
\begin{align}\label{Phi:global2}
\Phi_{\mathfrak{M}, \rm{dg}}^{\rm{op}}
=\xi' \circ \Phi'_{\mathfrak{M}} \circ
\mathbb{D}_{\mathfrak{M}}  \colon 
L_{\rm{coh}}(\mathfrak{M}^{\dag})^{\rm{op}} \to L_{\rm{coh}}(\mathfrak{M}^{\sharp}).
\end{align}

Below we explain each term of (\ref{Phi:global2}) in the global setting. 
First for a $S(\mathfrak{E})$-module $\fF$ on $\mathfrak{M}$, 
its $\oO_{\mathfrak{M}}$-dual is given by
\begin{align*}
\mathbb{D}_{\mathfrak{M}}(\fF)=\bigoplus_{\lambda \in \mathbb{Z}}
\mathbb{D}_{\mathfrak{M}}(\fF_{\lambda})
\end{align*}
where $\fF_{\lambda}$ is the weight $\lambda$ part of $\fF$. 
Then $\mathbb{D}_{\mathfrak{M}}(\fF)$ is naturally 
equipped with a $S(\mathfrak{E})$-module structure. 
 Next, we set 
$\overline{\kK}_{\mathfrak{M}}$ to be the totalization of the Koszul complex 
in the dg-category $L_{\rm{qcoh}}(\mathfrak{M})$
\begin{align*}
\cdots \to \bigwedge^2\mathfrak{E} \otimes_{\oO_{\mathfrak{M}}}  S(\mathfrak{E}) \to 
\mathfrak{E} \otimes_{\oO_{\mathfrak{M}}} S(\mathfrak{E}) \to S(\mathfrak{E}).
\end{align*}
The above object naturally defines an object in  
$ L_{\rm{coh}}(\mathfrak{M}^{\dag})$
which is equivalent to $i^{\dag}_{\ast}\oO_{\mathfrak{M}}$. 
Then for a $S(\mathfrak{E})$-module $\fF$ on $\mathfrak{M}$,  
we set
\begin{align*}
\Phi_{\mathfrak{M}}'(\fF)=\hH om_{S(\mathfrak{E})}(\overline{\kK}_{\mathfrak{M}}, \fF)
\end{align*}
which is equipped with a 
natural $S(\mathfrak{E}^{\vee}[-1])$-module structure. 
Note that the
above object represents $\dR \hH om_{\mathfrak{M}^{\dag}}(i^{\dag}_{\ast}
\oO_{\mathfrak{M}}, \fF)$
(see~Remark~\ref{rmk:Morita}). 
Finally for a $S(\mathfrak{E}^{\vee}[-1])$-module on $\mathfrak{M}$, 
the regrading functor $\xi'$
is given by 
\begin{align*}
\xi'(\fF)=\bigoplus_{\lambda}\fF_{\lambda}[-2\lambda],
\end{align*}
 which is equipped with a natural
$S(\mathfrak{E}^{\vee}[1])$-module structure. 
The composition $\xi' \circ \Phi_{\mathfrak{M}}' \circ \mathbb{D}_{\mathfrak{M}}$ preserves objects with bounded 
coherent cohomologies, 
since this is true 
 locally on 
$\mathfrak{M}^{\mathbb{C}^{\ast} \rig}$
by Lemma~\ref{prop:koszul:U}.
Therefore 
we obtain the functor (\ref{Phi:global2}).  

Similarly we can globalize $\Psi_{\mathfrak{U}}^{\rm{op}}$ in (\ref{equiv:U})
to a dg-functor $\Psi_{\mathfrak{M}, \rm{dg}}^{\rm{op}}$
from $L_{\rm{coh}}(\mathfrak{M}^{\sharp})$ to $L_{\rm{coh}}(\mathfrak{M}^{\dag})^{\rm{op}}$. 
By taking the induced functors on homotopy categories, we obtain the functors
\begin{align}\label{Phi:Psi:global}
\Phi_{\mathfrak{M}}^{\rm{op}} \colon D^b_{\rm{coh}}(\mathfrak{M}^{\dag})^{\rm{op}} \to
D^b_{\rm{coh}}(\mathfrak{M}^{\sharp}), \ 
\Psi_{\mathfrak{M}}^{\rm{op}} \colon D^b_{\rm{coh}}(\mathfrak{M}^{\sharp}) \to
D^b_{\rm{coh}}(\mathfrak{M}^{\dag})^{\rm{op}}.
\end{align}
They are adjoints each other, and
quasi-inverses each other 
as this is true locally on 
$\mathfrak{M}^{\mathbb{C}^{\ast} \rig}$
by Lemma~\ref{prop:koszul:U}. 
In particular, the functors (\ref{Phi:Psi:global}) are equivalences. 
\end{proof}

Since $\fM^{\dag}$ is quasi-smooth, 
the derived dual $\mathbb{D}_{\mathfrak{M}^{\dag}}$ is 
a contravariant autoequivalence of $D^b_{\rm{coh}}(\mathfrak{M}^{\dag})$. 
Therefore we also have a covariant equivalence, 
\begin{align}\label{aut:covariant}
\Phi_{\mathfrak{M}}\cneq \Phi_{\mathfrak{M}}^{\rm{op}} \circ
\mathbb{D}_{\mathfrak{M}^{\dag}} \colon 
D^b_{\rm{coh}}(\mathfrak{M}^{\dag}) \stackrel{\sim}{\to}
D^b_{\rm{coh}}(\mathfrak{M}^{\sharp})
\end{align}
whose quasi-inverse 
is $\Psi_{\mathfrak{M}}\cneq \mathbb{D}_{\mathfrak{M}^{\dag}} \circ \Psi_{\mathfrak{M}}^{\rm{op}}$.

\subsubsection{Semiorthogonal decomposition}
In the diagram (\ref{dia:stacs:dag}), we set
\begin{align*}
\mathfrak{M}^{\dag}_{\circ} \cneq \mathfrak{M}^{\dag} \setminus i^{\dag}(\mathfrak{M}), \ 
\mathfrak{M}^{\sharp}_{\circ} \cneq \mathfrak{M}^{\sharp} \setminus i^{\sharp}(\mathfrak{M})
\end{align*}
which are open substacks of $\mathfrak{M}^{\dag}$, $\mathfrak{M}^{\sharp}$
respectively. 
We have the following proposition, which is a consequence of~\cite{HalpRem}. 

\begin{prop}\label{prop:sod}
For each $\lambda \in \mathbb{Z}$, the functors
\begin{align}\label{funct:i}
i^{\dag}_{\ast} \colon D^b_{\rm{coh}}(\mathfrak{M})_{\lambda} \to 
D_{\rm{coh}}^b(\mathfrak{M}^{\dag}), \ 
i^{\sharp}_{\ast} \colon D^b_{\rm{coh}}(\mathfrak{M})_{\lambda} \to 
D_{\rm{coh}}^b(\mathfrak{M}^{\sharp})
\end{align}
are fully-faithful.
Moreover we have semiorthogonal decomposition
\begin{align}\label{SOD:i}
&D^b_{\rm{coh}}(\mathfrak{M}^{\dag})=
\langle \ldots, 
i^{\dag}_{\ast}D^b_{\rm{coh}}(\mathfrak{M})_{-1}, \dD^{\dag}, 
i^{\dag}_{\ast}D^b_{\rm{coh}}(\mathfrak{M})_{0}, 
i^{\dag}_{\ast}D^b_{\rm{coh}}(\mathfrak{M})_{1}, \ldots
\rangle, \\
\notag
&D^b_{\rm{coh}}(\mathfrak{M}^{\sharp})=
\langle \ldots, 
i^{\sharp}_{\ast}D^b_{\rm{coh}}(\mathfrak{M})_{1}, \dD^{\sharp}, 
i^{\sharp}_{\ast}D^b_{\rm{coh}}(\mathfrak{M})_{0}, 
i^{\sharp}_{\ast}D^b_{\rm{coh}}(\mathfrak{M})_{-1}, \ldots
\rangle
\end{align}
such that the restriction functors give 
equivalences 
\begin{align*}
\dD^{\dag} \stackrel{\sim}{\to} 
D^b_{\rm{coh}}(\mathfrak{M}^{\dag}_{\circ}), \ 
\dD^{\sharp} \stackrel{\sim}{\to} 
D^b_{\rm{coh}}(\mathfrak{M}^{\sharp}_{\circ}).
\end{align*}
\end{prop}
\begin{proof}
The proposition is a consequence 
of~\cite[Theorem~3.2]{HalpRem}
by applying it 
for the $\Theta$-stratifications 
$\mathfrak{M}^{\dag}=\mathfrak{M}^{\dag}_{\circ} \cup i^{\dag}(\mathfrak{M})$, 
$\mathfrak{M}^{\sharp}=\mathfrak{M}^{\dag}_{\circ} \cup i^{\sharp}(\mathfrak{M})$.
Here we note that the weights in the above semiorthogonal 
components
have opposite signs for $\mathfrak{M}^{\dag}$ and $\mathfrak{M}^{\sharp}$, 
because the $\mathbb{C}^{\ast}$-weights on $\mathfrak{E}$ and 
$\mathfrak{E}^{\vee}[1]$ have the opposite sign. 
\end{proof}

\begin{prop}\label{lem:rhoast}
The functors
\begin{align}\label{funct:ast}
\rho^{\dag\ast} \colon 
D^b_{\rm{coh}}(\mathfrak{M})_{\lambda} \to D^b_{\rm{coh}}(\mathfrak{M}^{\dag}), \ 
\rho^{\sharp\ast} \colon 
D^b_{\rm{coh}}(\mathfrak{M})_{\lambda} \to D^b_{\rm{coh}}(\mathfrak{M}^{\sharp})
\end{align}
are fully-faithful. 
Moreover
for the equivalence $\Phi_{\mathfrak{M}}^{\rm{op}}$ in (\ref{Phi:global}), 
we have 
\begin{align*}
\Phi_{\mathfrak{M}}^{\rm{op}}
(\rho^{\dag\ast}D^b_{\rm{coh}}(\mathfrak{M})_{\lambda})
=i^{\sharp}_{\ast} D^b_{\rm{coh}}(\mathfrak{M})_{-\lambda}, \ 
\Phi_{\mathfrak{M}}^{\rm{op}}(i^{\dag}_{\ast} 
D^b_{\rm{coh}}(\mathfrak{M})_{\lambda})
=\rho^{\sharp\ast}D^b_{\rm{coh}}(\mathfrak{M})_{-\lambda}. 
\end{align*}
\end{prop}
\begin{proof}
The right adjoint functor of $\rho^{\dag\ast}$ in (\ref{funct:ast})
is given by $(\rho^{\dag}_{\ast})_{\lambda}$, where
$(-)_{\lambda}$ means taking the weight $\lambda$ part. 
There is a natural transform
\begin{align*}
(-) \to (\rho^{\dag}_{\ast})_{\lambda}\circ \rho^{\dag\ast}(-)
=\left((-) \otimes_{\oO_{\mathfrak{M}}}S(\mathfrak{E}) \right)_{\lambda}
\end{align*}
which is obviously an isomorphism as $\mathfrak{E}$ is of 
$\mathbb{C}^{\ast}$-weight one. 
Therefore $\rho^{\dag\ast}$ is fully-faithful, and 
the proof for $\rho^{\sharp\ast}$ is similar. 

As for the second statement, the statement is local 
on $\mathfrak{M}^{\mathbb{C}^{\ast} \rig}$, so 
we can assume that 
$\mathfrak{M}=[\mathfrak{U}/\mathbb{C}^{\ast}]$ where 
$\mathfrak{U}$ is given as in (\ref{frak:U})
with trivial $\mathbb{C}^{\ast}$-action, 
and $\mathfrak{E}=\eE \otimes_{\oO_Y} \oO_{\mathfrak{U}}$
for a two term complex of vector bundles $\eE$ as in (\ref{vect:vV}). 
We first consider the case that $\mathfrak{U}=Y$. 
For $m \in \mathbb{Z}$
and $M \in D_{\rm{qcoh}}([Y/\mathbb{C}^{\ast}])$, 
we write $M(m)=M \otimes_{\mathbb{C}} \mathbb{C}(m)$ where 
$\mathbb{C}(m)$ is the one dimensional $\mathbb{C}^{\ast}$-representation 
with weight $m$. 
In the notation of the proof of Proposition~\ref{thm:koszul}, we have  
\begin{align*}
\Phi_{Y}^{\rm{op}}(\oO_Y(\lambda))=
\xi' \circ \Phi' \circ \mathbb{D}_Y(\oO_Y(\lambda))=
\oO_{Y^{\sharp}}(-\lambda). 
\end{align*}
Similarly we have 
$\Phi_Y^{\rm{op}}(\oO_{Y^{\dag}}(\lambda))=\oO_Y(-\lambda)$.
Since $D^b_{\rm{coh}}(Y)$ is locally 
generated by $\oO_Y$, 
for $\fF \in D^b_{\rm{coh}}([Y/\mathbb{C}^{\ast}])_{\lambda})$
we have 
\begin{align}\label{Phi:AF}
\Phi_{Y}^{\rm{op}}
(\rho_Y^{\dag\ast}\fF)
\in i^{\sharp}_{Y\ast} D^b_{\rm{coh}}([Y/\mathbb{C}^{\ast}])_{-\lambda}, \
\Phi_{Y}^{\rm{op}}(i^{\dag}_{Y\ast} \fF) \in 
 \rho_Y^{\sharp\ast}D^b_{\rm{coh}}([Y/\mathbb{C}^{\ast}])_{-\lambda}. 
\end{align}

In general, let us take 
an object $\fF \in D^b_{\rm{coh}}([\mathfrak{U}/\mathbb{C}^{\ast}])_{\lambda}$. 
By the commutative diagrams (\ref{dia:Udag}), (\ref{dia:DM}) together with (\ref{Phi:AF}), we have 
\begin{align*}
&j^{\sharp}_{\ast}\Phi_{\mathfrak{U}}^{\rm{op}} (\rho_{\mathfrak{U}}^{\dag \ast}\fF)
\cong \widetilde{\Phi}_Y^{\rm{op}}(j^{\dag}_{\ast}\rho_{\mathfrak{U}}^{\dag\ast}\fF)
\cong \widetilde{\Phi}_Y^{\rm{op}}(\rho_Y^{\dag\ast}j_{\ast}\fF)
 \in i_{Y\ast}^{\sharp}D^b_{\rm{coh}}([Y/\mathbb{C}^{\ast}])_{-\lambda}, \\
&j^{\sharp}_{\ast}\Phi_{\mathfrak{U}}^{\rm{op}}(i_{\mathfrak{U}\ast}^{\dag}\fF)
\cong \widetilde{\Phi}_Y^{\rm{op}}(j^{\dag}_{\ast}i_{\mathfrak{U}\ast}^{\dag}\fF)
\cong \widetilde{\Phi}_Y^{\rm{op}}(i_{Y\ast}^{\dag}j_{\ast}\fF) \in \rho_Y^{\sharp \ast}D^b_{\rm{coh}}([Y/\mathbb{C}^{\ast}])_{-\lambda}. 
\end{align*}
Then by Lemma~\ref{lem:push0} below, we have the inclusions
\begin{align}\notag
&\Phi_{\mathfrak{U}}^{\rm{op}}
(\rho_{\mathfrak{U}}^{\dag\ast}D^b_{\rm{coh}}([\mathfrak{U}/\mathbb{C}^{\ast}])_{\lambda})
\subset i^{\sharp}_{\mathfrak{U}\ast} D^b_{\rm{coh}}([\mathfrak{U}/\mathbb{C}^{\ast}])_{-\lambda}, \\ 
\notag
&\Phi_{\mathfrak{U}}^{\rm{op}}(i^{\dag}_{\mathfrak{U}\ast} 
D^b_{\rm{coh}}([\mathfrak{U}/\mathbb{C}^{\ast}])_{\lambda})
\subset \rho_{\mathfrak{U}}^{\sharp\ast}D^b_{\rm{coh}}([\mathfrak{U}/\mathbb{C}^{\ast}])_{-\lambda}
\end{align}
respectively. 
By applying the same argument for 
$\Psi_{\mathfrak{M}}^{\rm{op}}$ given in (\ref{Phi:Psi:global}), we see that the
above inclusions
are identities. Therefore the proposition holds. 
\end{proof}

Here we have used the following lemma. 
\begin{lem}\label{lem:push0}
In the proof of Proposition~\ref{lem:rhoast}, 
we have the following: 
\begin{enumerate}
\item 
An 
 object $M \in D^b_{\rm{coh}}([\mathfrak{U}^{\sharp}/\mathbb{C}^{\ast}])$
is an object in 
$i_{\mathfrak{U}\ast}^{\sharp}D^b_{\rm{coh}}([\mathfrak{U}/\mathbb{C}^{\ast}])_{\lambda}$
if and only if $j^{\sharp}_{\ast}M$ 
is an object in $i_{Y\ast}^{\sharp}D^b_{\rm{coh}}([Y/\mathbb{C}^{\ast}])_{\lambda}$. 
\item An object 
$M \in D^b_{\rm{coh}}([\mathfrak{U}^{\sharp}/\mathbb{C}^{\ast}])$ is an object in 
$\rho_{\mathfrak{U}}^{\sharp\ast} D^b_{\rm{coh}}([\mathfrak{U}/\mathbb{C}^{\ast}])_{\lambda}$
if and only if $j^{\sharp}_{\ast}M$ is an object in $\rho_Y^{\sharp\ast} D^b_{\rm{coh}}([Y/\mathbb{C}^{\ast}])_{\lambda}$. 
\end{enumerate}
\end{lem}
\begin{proof}
As for (i), since the functors (\ref{funct:i}) are fully-faithful, 
the subcategory
$i^{\sharp}_{\ast}D^b_{\rm{coh}}(\mathfrak{M})_{\lambda}$
in $D^b_{\rm{coh}}(\mathfrak{M}^{\sharp})$ is identified with 
the subcategory 
of objects 
$\fF \in D^b_{\rm{coh}}(\mathfrak{M}^{\sharp})$
such that each $\hH^i(\fF)$ is an $\oO_{\mM}$-module with 
$\mathbb{C}^{\ast}$-weight $\lambda$. 
Then (i) follows from this fact. 

As for (ii), 
suppose that 
 $M$ satisfies the latter condition. 
Then the natural map
\begin{align*}
j^{\sharp}_{\ast}\rho_{\mathfrak{U}}^{\sharp \ast}  (\rho_{\mathfrak{U}\ast}^{\sharp})_{\lambda}M
\cong 
\rho_Y^{\sharp \ast}  (\rho_{Y\ast}^{\sharp})_{\lambda}  j^{\sharp}_{\ast}M 
\to j^{\sharp}_{\ast}M
\end{align*}
is an isomorphism. 
Here the first isomorphism follows from base change 
with respect to the diagram (\ref{dia:Udag}). 
Therefore the natural map 
$\rho_{\mathfrak{U}}^{\sharp \ast}  (\rho_{\mathfrak{U}\ast}^{\sharp})_{\lambda}M \to M$
is also an isomorphism, 
since 
the functor $j^{\sharp}_{\ast} \colon 
 D^b_{\rm{coh}}([\mathfrak{U}^{\sharp}/\mathbb{C}^{\ast}])
 \to  D^b_{\rm{coh}}([Y^{\sharp}/\C])$
 is conservative. 
It follows that $M$ is an object in 
$\rho_{\mathfrak{U}}^{\sharp\ast} D^b_{\rm{coh}}([\mathfrak{U} /\mathbb{C}^{\ast}])_{\lambda}$. 
\end{proof}

Using Proposition~\ref{lem:rhoast}, we show the following lemma. 
\begin{lem}\label{lem:sod:perf}
We have the semiorthogonal decomposition
\begin{align}\label{SOD:perf}
D^b_{\rm{coh}}(\mathfrak{M}^{\dag})=
\langle \ldots, 
\rho^{\dag\ast}D^b_{\rm{coh}}(\mathfrak{M})_{1}, 
\rho^{\dag\ast}D^b_{\rm{coh}}(\mathfrak{M})_{0}, \tT^{\dag}, 
\rho^{\dag\ast}D^b_{\rm{coh}}(\mathfrak{M})_{-1}, \ldots
\rangle
\end{align}
together with an equivalence $\tT^{\dag} \stackrel{\sim}{\to} D^b_{\rm{coh}}(\mathfrak{M}^{\sharp}_{\circ})$. 
\end{lem}
\begin{proof}
We apply the functor 
$\Psi_{\mathfrak{M}}^{\rm{op}}$ in (\ref{Phi:Psi:global})
 to the 
semiorthogonal decomposition of $D^b_{\rm{coh}}(\mathfrak{M}^{\sharp})$ in Proposition~\ref{prop:sod}. 
By Proposition~\ref{lem:rhoast}, we
obtain the semiorthogonal decomposition (\ref{SOD:perf})
for $\tT^{\dag}=\Psi_{\mathfrak{M}}^{\rm{op}}(\dD^{\sharp})$. 
Then $\tT^{\dag}$ is equivalent to $D^b_{\rm{coh}}(\mathfrak{M}^{\sharp}_{\circ})^{\rm{op}}$, and applying $\mathbb{D}_{\mathfrak{M}^{\sharp}_{\circ}}$
we obtain the equivalence $\tT^{\dag} \stackrel{\sim}{\to} D^b_{\rm{coh}}(\mathfrak{M}^{\sharp}_{\circ})$. 
\end{proof}

We set 
$e \in \mathbb{Z}$ to be
\begin{align*}
e \cneq \rank(\mathfrak{E}|_{\mM})=
\wt_{\mathbb{C}^{\ast}}(\det \mathfrak{E}|_{\mM}).
\end{align*}

\begin{lem}\label{lem:dual:e}
For each $\lambda \in \mathbb{Z}$, we have 
\begin{align*}
\mathbb{D}_{\mathfrak{M}^{\dag}}(i^{\dag}_{\ast} D^b_{\rm{coh}}(\mathfrak{M})_{\lambda})
= i^{\dag}_{\ast}D^b_{\rm{coh}}(\mathfrak{M})_{-\lambda-e}. 
\end{align*}
\end{lem}
\begin{proof}
For $\fF \in D^b_{\rm{coh}}(\mathfrak{M})_{\lambda}$, we have 
\begin{align*}
\dR \hH om_{\mathfrak{M}^{\dag}}(i^{\dag}_{\ast}\fF, \oO_{\mathfrak{M}^{\dag}})
\cong i^{\dag}_{\ast}
\dR \hH om_{\mathfrak{M}}(\fF, i^{\dag !}\oO_{\mathfrak{M}^{\dag}}) 
 \cong i^{\dag}_{\ast}
\mathbb{D}_{\mathfrak{M}}(\fF) \otimes \det \mathfrak{E}^{\vee}[-e]. 
\end{align*}
The lemma holds since $\mathbb{D}_{\mathfrak{M}}(\fF)$ is of $\mathbb{C}^{\ast}$-weight $-\lambda$
and $\det \mathfrak{E}^{\vee}$ is of $\mathbb{C}^{\ast}$-weight 
$-e$. 
\end{proof}

The following is the main result in this section.

\begin{thm}\label{thm:SOD:M}
Suppose that $e\ge 0$. Then we have the semiorthogonal decomposition of the form
\begin{align*}
D^b_{\rm{coh}}(\mathfrak{M}_0^{\dag})=
\langle \rho^{\dag\ast}D^b_{\rm{coh}}(\mathfrak{M})_{e-1}, \ldots, 
 \rho^{\dag\ast}D^b_{\rm{coh}}(\mathfrak{M})_{0}, D^b_{\rm{coh}}(\mathfrak{M}_{\circ}^{\sharp})
\rangle. 
\end{align*}
\end{thm}
\begin{proof}
By Proposition~\ref{prop:sod} and Lemma~\ref{lem:sod:perf}, it is 
enough to show 
the identity of subcategories in $D^b_{\rm{coh}}(\mathfrak{M}^{\dag})$
\begin{align}\label{id:D=T}
\dD^{\dag}=\langle \rho^{\dag\ast}D^b_{\rm{coh}}(\mathfrak{M})_{e-1}, \ldots, 
 \rho^{\dag\ast}D^b_{\rm{coh}}(\mathfrak{M})_{0}, \tT^{\dag}
\rangle. 
\end{align}
The above identity follows from the 
same argument of~\cite[Theorem~2.5]{Orsin}, 
by replacing $\sS_{\lambda}$, $\pP_{\lambda}$ in \textit{loc.~cit.~}by 
$i^{\dag}_{\ast}D^b_{\rm{coh}}(\mathfrak{M})_{\lambda}$, 
$\rho^{\dag\ast}D^b_{\rm{coh}}(\mathfrak{M})_{\lambda}$
respectively. 
For each $\lambda \in \mathbb{Z}$, we set $\sS_{\lambda}$, 
$\pP_{\lambda}$ as above, and 
define $\sS_{\le \lambda}$, $\pP_{\le \lambda}$ to be the triangulated 
subcategories of $D^b_{\rm{coh}}(\mathfrak{M}^{\dag})$
generated by $\sS_{\lambda'}$, $\pP_{\lambda'}$ for $\lambda' \le \lambda$
 respectively. 
The subcategories $\sS_{>\lambda}$, $\pP_{>\lambda}$ are defined 
similarly. We also define
$D^b_{\rm{coh}}(\mathfrak{M}^{\dag})_{\ge \lambda}$ to be the 
subcategory of $D^b_{\rm{coh}}(\mathfrak{M}^{\dag})$ consisting 
of objects $M$ such that $\rho^{\dag}_{\ast}M \in 
D_{\rm{qcoh}}(\mathfrak{M})_{\ge \lambda}$. 
Then by Lemma~\ref{lem:sod:easy} below, 
we have the semiorthogonal decompositions
\begin{align}\label{sod:easy}
D^b_{\rm{coh}}(\mathfrak{M}^{\dag})=\langle \sS_{<0}, 
D^b_{\rm{coh}}(\mathfrak{M}^{\dag})_{\ge 0} \rangle 
=\langle D^b_{\rm{coh}}(\mathfrak{M}^{\dag})_{\ge 0}, \pP_{<0} \rangle. 
\end{align}
By comparing with semiorthogonal decomposition in (\ref{SOD:i}), (\ref{SOD:perf}),
we have
\begin{align}\label{SOD:sp0}
D^b_{\rm{coh}}(\mathfrak{M}^{\dag})_{\ge 0}
=\langle \dD^{\dag}, \sS_{\ge 0}\rangle = \langle \pP_{\ge 0}, \tT^{\dag} 
\rangle.
\end{align}
Therefore we have the semiorthogonal decomposition
\begin{align}\label{SOD:sp1}
D^b_{\rm{coh}}(\mathfrak{M}^{\dag})
=\langle \sS_{<0}, \pP_{\ge 0}, \tT^{\dag}\rangle. 
\end{align}
On the other hand, we apply 
$\mathbb{D}_{\mathfrak{M}^{\dag}} \circ \otimes_{\mathbb{C}}\mathbb{C}(-e+1)$
to the semiorthogonal decomposition (\ref{SOD:i}), (\ref{SOD:perf}). 
By noting 
that $\mathbb{D}_{\mathfrak{M}^{\dag}}(\sS_{\ge \lambda})
=\sS_{\le -\lambda-e}$ by Lemma~\ref{lem:dual:e}, 
we have 
\begin{align*}
D^b_{\rm{coh}}(\mathfrak{M}^{\dag})=\langle \sS_{<0}, 
\mathbb{D}_{\mathfrak{M}^{\dag}}(\dD^{\dag}(-e+1)), \sS_{\ge 0} \rangle 
=\langle \pP_{\ge e}, \mathbb{D}_{\mathfrak{M}^{\dag}}(\tT^{\dag}(-e+1)), \pP_{<e}\rangle.
\end{align*}
By comparing with (\ref{SOD:i}), we 
have $\dD^{\dag}=\mathbb{D}_{\mathfrak{M}^{\dag}}(\dD^{\dag}(-e+1))$. 
Similarly applying $\mathbb{D}_{\mathfrak{M}^{\dag}} \circ \otimes_{\mathbb{C}}\mathbb{C}(-e+1)$
to the semiorthogonal decomposition (\ref{SOD:sp0}), we have 
\begin{align*}
\mathbb{D}_{\mathfrak{M}^{\dag}}(D^b_{\rm{coh}}(\mathfrak{M})_{\ge -e+1})
=\langle \sS_{<0}, \mathbb{D}_{\mathfrak{M}^{\dag}}(\dD^{\dag}(-e+1))\rangle
=\langle\mathbb{D}_{\mathfrak{M}^{\dag}}(\tT^{\dag}(-e+1)), \pP_{<e}\rangle.  
\end{align*}
It follows that we have
\begin{align*}
D^b_{\rm{coh}}(\mathfrak{M}^{\dag})=\langle \pP_{\ge e}, \sS_{<0}, 
\dD^{\dag}\rangle
=\langle \sS_{<0}, \pP_{\ge e}, \dD^{\dag} \rangle. 
\end{align*}
Here the second identity follows from
 $\Hom(\pP_{\ge e}, \sS_{<0})=0$ since $e\ge 0$. By comparing with (\ref{SOD:sp1}), we obtain the identity (\ref{id:D=T}). 
\end{proof}

We have used the following lemma, which is a generalization of~\cite[Lemma~2.3]{Orsin}. 
\begin{lem}\label{lem:sod:easy}
	We have semiorthogonal decompositions (\ref{sod:easy}). 
	\end{lem}
\begin{proof}
	We first show the left semiorthogonal decomposition in (\ref{sod:easy}). 
	There is no non-zero morphism from $\Dbc(\fM^{\dag})_{\ge 0}$ to 
	$\sS_{<0}$ since locally on $\fM$ any object in $\Dbc(\fM^{\dag})$ admits 
	a bounded above resolution of free $S(\fE)$-modules with non-negative $\C$-weights. 
	For any object $A \in \Dbc(\fM^{\dag})$, we have the split distinguished triangle 
	\begin{align}\label{rho:split}
		(\rho^{\dag}_{\ast}A)_{\ge 0} \to  \rho^{\dag}_{\ast}A \to
			(\rho^{\dag}_{\ast}A)_{< 0}
			\end{align} 
		in $D_{\qcoh}(\fM)$
	such that $(\rho^{\dag}_{\ast}A)_{\ge 0}$ has non-negative $\C$-weights 
	and $(\rho^{\dag}_{\ast}A)_{< 0}$ has negative $\C$-weights. 
	Since $S(\fE)$ has non-negative $\C$-weights, the $S(\fE)$-module structure 
	on $\rho^{\dag}_{\ast}A$ induces
	$S(\fE)$-module structures on $(\rho^{\dag}_{\ast}A)_{\ge 0}$, $(\rho^{\dag}_{\ast}A)_{< 0}$.
	Therefore the  
	distinguished triangle (\ref{rho:split}) lifts to a distinguished triangle 
	$A_{\ge 0} \to A \to A_{<0}$ 
	in $\Dbc(\fM^{\dag})$ such that 
	$A_{\ge 0} \in \Dbc(\fM^{\dag})_{\ge 0}$, $A_{<0} \in \sS_{<0}$. 
	
	We next show the right semiorthogonal decomposition in (\ref{sod:easy}). 
	It is obvious that there is no non-zero morphism from $\pP_{<0}$ to $\Dbc(\fM^{\dag})_{\ge 0}$ from 
	the definition of $\Dbc(\fM^{\dag})_{\ge 0}$. 
	For an object $A\in \Dbc(\fM^{\dag})$, let 
	$\lambda$ be the minimum $\C$-weight of $\rho^{\dag}_{\ast}A$
	and assume that $\lambda<0$. 
	We have the morphism 
	$\rho^{\dag \ast}(\rho^{\dag}_{\ast}A)_{\lambda} \to A$, and let $A'$ be its cone. 
	Then we have the distinguished triangle in $D_{\qcoh}(\fM)$
	\begin{align*}
		(\rho^{\dag}_{\ast}A)_{\lambda} \otimes S(\fE) \to \rho^{\dag}_{\ast}A \to \rho^{\dag}_{\ast}A'. 
		\end{align*}
	From the above distinguished triangle, the minimum $\C$-weight of 
	$\rho^{\dag}_{\ast}A'$ is strictly bigger than $\lambda$. 
Therefore by repeating this argument, we 
find a distinguished triangle 
$A_{<0}' \to A \to A_{\ge 0}'$ such that 
$A_{<0}' \in \pP_{<0}$ and $A_{\ge 0} \in \Dbc(\fM^{\dag})_{\ge 0}$. 
	
	\end{proof}
\subsubsection{Singular supports under linear Koszul duality}
\label{subsec:support:koszul}
Let $\mathfrak{M}^{\dag}$, $\mathfrak{M}^{\sharp}$ be the 
derived stacks as in the diagram (\ref{dia:stacs:dag}). 
Since $\mathbb{L}_{\mathfrak{M}^{\dag}/\mathfrak{M}}=\rho^{\dag\ast} \mathfrak{E}$ and 
$\mathbb{L}_{\mathfrak{M}^{\sharp}/\mathfrak{M}}=
\rho^{\sharp\ast} \mathfrak{E}^{\vee}[1]$, 
we have natural equivalences
\begin{align}\label{nat:Omega}
\Omega_{\mathfrak{M}^{\dag}/\mathfrak{M}}[-1] 
\stackrel{\sim}{\to} 
 \mathfrak{M}^{\dag} \times_{\mathfrak{M}} \mathfrak{M}^{\sharp}
\stackrel{\sim}{\leftarrow} 
\Omega_{\mathfrak{M}^{\sharp}/\mathfrak{M}}[-1].  
\end{align}
We show that the 
above natural equivalence lifts to 
the absolute $(-1)$-shifted cotangent stacks
on classical truncations:
\begin{lem}\label{lem:equiv:shifted}
There is a natural isomorphism 
$\vartheta_0 \colon t_0(\Omega_{\mathfrak{M}^{\dag}}[-1]) \stackrel{\cong}{\to}
t_0(\Omega_{\mathfrak{M}^{\sharp}}[-1])$
which fits into the commutative diagram
\begin{align}\label{eq:theta}
\xymatrix{
t_0(\Omega_{\mathfrak{M}^{\dag}}[-1]) \ar[rd]^-{\iota^{\dag}} \ar[rr]_-{\vartheta_0}^-{\cong} \ar[d] & & 
t_0(\Omega_{\mathfrak{M}^{\sharp}}[-1]) \ar[d] \ar[ld]_-{\iota^{\sharp}}  \\
t_0(\Omega_{\mathfrak{M}^{\dag}/\mathfrak{M}}[-1]) \ar[r]^-{\cong} & 
t_0(\mathfrak{M}^{\dag} \times_{\mathfrak{M}} \mathfrak{M}^{\sharp})
& 
\ar[l]_-{\cong} t_0(\Omega_{\mathfrak{M}^{\sharp}/\mathfrak{M}}[-1])
} 
\end{align}
Here the vertical arrows are natural 
maps induced by $\mathbb{L}_{\mathfrak{M}^{\dag}} \to 
\mathbb{L}_{\mathfrak{M}^{\dag}/\mathfrak{M}}$, 
$\mathbb{L}_{\mathfrak{M}^{\sharp}} \to 
\mathbb{L}_{\mathfrak{M}^{\sharp}/\mathfrak{M}}$, 
and 
the bottom arrows are induced by (\ref{nat:Omega}). 
\end{lem}
\begin{proof}
We first consider the situation of the diagram (\ref{dia:stacs:dag2}), i.e. 
$\mathfrak{M}=[\mathfrak{U}/\mathbb{C}^{\ast}]$ where 
$\mathfrak{U}=\Spec \rR(V\to Y, s)$ is given as in (\ref{frak:U})
with trivial $\mathbb{C}^{\ast}$-action, 
and $\mathfrak{E}=\eE \otimes_{\oO_Y} \oO_{\mathfrak{U}}$
for a two term complex of vector bundles $\eE$ as in (\ref{vect:vV}). 
We set $\eE_0 \cneq (\eE^0)^{\vee}$ and 
$\eE_1 \cneq (\eE^1)^{\vee}$. 
Note that the structure sheaves of $\mathfrak{U}^{\dag}$
and $\mathfrak{U}^{\sharp}$ are Koszul complexes
\begin{align*}
&\oO_{\mathfrak{U}^{\dag}}=S_{\oO_{\mathfrak{U}}}(\mathfrak{E})=
\rR(\eE_0 \times_Y \eE_1 \times_Y V \to \eE_0, (\phi, s)), \\ 
&\oO_{\mathfrak{U}^{\sharp}}=
S_{\oO_{\mathfrak{U}}}(\mathfrak{E}^{\vee}[1])=
\rR(\eE^{-1} \times_Y \eE^0 \times_Y V \to \eE^{-1}, (\phi^{\vee}, s))
\end{align*}
where differentials are induced by
\begin{align*}
\eE^{-1} \oplus V^{\vee} \stackrel{(\phi, s)}{\to}
\eE^0 \oplus \oO_Y \subset \oO_{\eE_0}, \ 
\eE_0 \oplus V^{\vee} \stackrel{(\phi^{\vee}, s)}{\to}
\eE_1 \oplus \oO_Y \subset \oO_{\eE^{-1}}
\end{align*}
respectively. 
By the above descriptions, both of 
$(-1)$-shifted cotangent derived schemes
$\Omega_{\mathfrak{U}^{\dag}}[-1]$, $\Omega_{\mathfrak{U}^{\sharp}}[-1]$
are given by
the derived critical locus of the function
\begin{align}\label{funct:koszul:w}
w \colon \eE_0 \times_Y \eE^{-1} \times_Y V^{\vee} \to \mathbb{C}, \ 
w(x, e, e', v)=\langle \phi|_x(e'), e \rangle +\langle s(x), v\rangle
\end{align}
for $x \in Y, e\in \eE_0|_{x}, e' \in \eE^{-1}|_{x}, v \in V^{\vee}|_{x}$.
Therefore the isomorphism $\vartheta_0$
is given by the composition of isomorphisms 
\begin{align}\label{isom:theta0}
	t_0(\Omega_{\fU^{\dag}}[-1]) \stackrel{\cong}{\to}
	\Crit(w) \stackrel{\cong}{\leftarrow} t_0(\Omega_{\fU^{\sharp}}[-1]). 
	\end{align}
Also both of $\Omega_{\mathfrak{U}^{\dag}/\mathfrak{U}}[-1]$
and $\Omega_{\mathfrak{U}^{\sharp}/\mathfrak{U}}[-1]$ are the 
derived zero locus of 
\begin{align*}
(\phi^{\vee}, \phi, s) \colon 
\eE_0 \times_Y \eE^{-1} \to \eE_1 \times_Y \eE^0 \times_Y V
\end{align*}
and
the left and right maps in (\ref{eq:theta}) 
are induced by 
the projection $\eE_0 \times_Y \eE^{-1} \times_Y V \to 
\eE_0 \times_Y \eE^{-1}$. 
Therefore the diagram (\ref{eq:theta}) commutes. 

In order to globalize the above isomorphism, it is enough 
to show that the isomorphism $\vartheta_0$ is independent of a
presentation of $\fU$ and $\eE$. 
Here we check the latter independence. 
The check for the former independence is similarly discussed. 
Let 
\begin{align}\label{qis:twoterm}
	\eE=(\eE^{-1} \to \eE^0) \to 
	\eE'=({\eE'}^{-1} \to {\eE'}^{0})
	\end{align}
be a quasi-isomorphism of two term complexes of 
vector bundles on $Y$. 
We set 
\begin{align*}
	\fE'=\eE' \otimes_{\oO_Y} \oO_{\fU}, \ 
	{\fU'}^{\dag}=\Spec S_{\oO_{\fU}}(\fE'), \ 
	{\fU'}^{\sharp}=\Spec 
S_{\oO_{\fU}}({\fE'}^{\vee}[1]).
\end{align*}
We also have the function
$w'$ on $\eE_0' \times_Y {\eE'}^{-1} \times_Y V^{\vee}$
as in (\ref{funct:koszul:w}), and isomorphisms 
\begin{align}\label{isom:theta1}
	t_0(\Omega_{{\fU'}^{\dag}}[-1]) \stackrel{\cong}{\to} \Crit(w')
	\stackrel{\cong}{\leftarrow} 
	t_0(\Omega_{{\fU'}^{\sharp}}[-1])
	\end{align}
as in (\ref{isom:theta0}).
On the other hand, the
quasi-isomorphism (\ref{qis:twoterm}) 
induces equivalences 
${\fU'}^{\dag} \stackrel{\sim}{\to} \fU^{\dag}$
and $\fU^{\sharp} \stackrel{\sim}{\to} {\fU'}^{\sharp}$, which induce
isomorphisms
\begin{align}\label{isom:Omega:induced}
	t_0(\Omega_{{\fU'}^{\dag}}[-1]) \stackrel{\cong}{\to}
	t_0(\Omega_{\fU^{\dag}}[-1]), \ 
		t_0(\Omega_{{\fU'}^{\sharp}}[-1]) \stackrel{\cong}{\to}
	t_0(\Omega_{\fU^{\sharp}}[-1]). 
	\end{align}
Via the isomorphisms in (\ref{isom:theta0}), (\ref{isom:theta1}),
the above isomorphisms give two isomorphisms
$\Crit(w') \stackrel{\cong}{\to} \Crit(w)$, and we need to check that 
they are indeed the same isomorphism.  
Namely we claim that there is an isomorphism 
$\gamma \colon \Crit(w') \stackrel{\cong}{\to} \Crit(w)$
such that the following diagram commutes
\begin{align}\label{gamma:commute}
	\xymatrix{
	t_0(\Omega_{{\fU'}^{\dag}}[-1]) \ar[r]^-{\cong} \ar[d]^-{\cong}& 
	\Crit(w') \ar[d]^-{\cong}_-{\gamma}
& \ar[l]_-{\cong} 
t_0(\Omega_{{\fU'}^{\sharp}}[-1]) \ar[d]^-{\cong} \\
	t_0(\Omega_{{\fU}^{\dag}}[-1]) \ar[r]^-{\cong} & 
\Crit(w) 
& \ar[l]_-{\cong} 
t_0(\Omega_{{\fU}^{\sharp}}[-1]).
}
	\end{align}
Here the horizontal arrows are given by (\ref{isom:theta0}), (\ref{isom:theta1}), 
and the left and right vertical arrows are given by (\ref{isom:Omega:induced}). 

The quasi-isomorphism $\eE \to \eE'$
induces the diagram 
\begin{align}\notag
	\xymatrix{
	\eE_0' \times_Y {\eE'}^{-1} \times_Y V^{\vee}
	& \ar[l]_-{h'}
	 \eE_0' \times_Y \eE^{-1} \times_Y V^{\vee}
	\ar[r]^-{h} &
	 \eE_0 \times_Y \eE^{-1} \times_Y V^{\vee}.
	}
	\end{align}
By Lemma~\ref{lem:critical} and Lemma~\ref{lem:beta0},
we have the isomorphism 
\begin{align*}
\xymatrix{
	\gamma \colon 
\Crit(w') & \ar[l]_-{\cong}^-{h'} 
(h')^{-1}(\Crit(w')) \cap h^{-1}(\Crit(w)) 
\ar[r]^-{\cong}_-{h} & \Crit(w),
}
	\end{align*}
which fits into the commutative diagram (\ref{gamma:commute}). 
Therefore we have the independence of $\vartheta_0$ for 
a presentation of $\eE$. 
\end{proof}

We have the isomorphism by Lemma~\ref{lem:equiv:shifted}
\begin{align*}
\vartheta_0 \colon  t_0(\Omega_{\mathfrak{M}^{\dag}}[-1])
 \stackrel{\cong}{\to}
t_0(\Omega_{\mathfrak{M}^{\sharp}}[-1]).
\end{align*}
In particular 
they admit two fiberwise $\mathbb{C}^{\ast}$-actions with respect to 
projections to $t_0(\mathfrak{M}^{\dag})$, $t_0(\mathfrak{M}^{\sharp})$ respectively. 
We have the following proposition: 
\begin{prop}\label{prop:compare:Z}
Let $\zZ \subset t_0(\Omega_{\mathfrak{M}^{\dag}}[-1])$
be a double conical closed substack. 
Under the equivalence $\Phi_{\mathfrak{M}}$ in (\ref{aut:covariant}), 
we have $\Phi_{\mathfrak{M}}(\cC_{\zZ})=\cC_{\vartheta_0(\zZ)}$. 
Therefore we have the equivalence
\begin{align*}
\Phi_{\mathfrak{M}} \colon D^b_{\rm{coh}}(\mathfrak{M}^{\dag})/\cC_{\zZ}
\stackrel{\sim}{\to} D^b_{\rm{coh}}(\mathfrak{M}^{\sharp})/
\cC_{\vartheta_0(\zZ)}. 
\end{align*}
\end{prop}
\begin{proof}
The statement is local 
on $\mathfrak{M}^{\mathbb{C}^{\ast} \rig}$, so 
we can assume
the situation of the diagram (\ref{dia:stacs:dag2}). 
Below, we use notation of the
proof of Lemma~\ref{lem:equiv:shifted}. 
By the local computation in \textit{loc.~cit.~}, 
singular supports for objects in 
$D^b_{\rm{coh}}(\mathfrak{U}^{\star})$
with $\star \in \{\dag, \sharp\}$
are closed subsets
in $\eE_0 \times_Y \eE^{-1} \times_Y V^{\vee}$. 
By Lemma~\ref{lem:ssuport}, 
they are determined by natural maps to relative 
Hochschild cohomologies
\begin{align}\label{map:HH}
\eE^0 \oplus \eE_1 \oplus V
\to \mathrm{HH}^{\ast}(\mathfrak{U}^{\dag}/\eE_0), \ 
\eE^0 \oplus \eE_1 \oplus V
\to \mathrm{HH}^{\ast}(\mathfrak{U}^{\sharp}/\eE^{-1})
\end{align}
respectively. 

Let $\Phi_{\mathfrak{U}}$ be the equivalence
\begin{align*}
\Phi_{\mathfrak{U}}=\Phi_{\mathfrak{U}}^{\rm{op}} \circ \mathbb{D}_{\mathfrak{U}^{\dag}}
 \colon D^b_{\rm{coh}}([\mathfrak{U}^{\dag}/\mathbb{C}^{\ast}])
\stackrel{\sim}{\to} 
D^b_{\rm{coh}}([\mathfrak{U}^{\sharp}/\mathbb{C}^{\ast}]). 
\end{align*}
From the construction of the equivalence
$\Phi_{\mathfrak{U}}$, it commutes with the $\mathbb{C}^{\ast}$-weight shift
up to cohomological shift (cf.~\cite[Theorem~1.7.1]{MrRi2}):
we have $\Phi_{\mathfrak{U}}((-)[m](j))=\Phi_{\mathfrak{U}}(-)[m+2j](j)$. 
So the equivalence $\Phi_{\mathfrak{U}}$ identifies the natural 
transforms
\begin{align}\label{koszul:nat}
\mathrm{Nat}_{D^b_{\rm{coh}}([\mathfrak{U}^{\dag}/\mathbb{C}^{\ast}])}(\id, \id[m](j))
\stackrel{\sim}{\to} \mathrm{Nat}_{D^b_{\rm{coh}}([\mathfrak{U}^{\sharp}/\mathbb{C}^{\ast}])}(\id, \id[m+2j](j)). 
\end{align}
On the other hand for $\star \in \{\dag, \sharp\}$, 
we have
\begin{align*}
\mathrm{HH}^{\ast}(\mathfrak{U}^{\star}/Y) &\cneq \Hom^{\ast}_{\mathfrak{U}^{\star}\times_Y \mathfrak{U}^{\star}}
(\Delta_{\ast}\oO_{\mathfrak{U}^{\star}}, \Delta_{\ast}\oO_{\mathfrak{U}^{\star}}) \\
&=\bigoplus_{(m, j) \in \mathbb{Z}^2}
\Hom_{[\mathfrak{U}^{\star}/\mathbb{C}^{\ast}] \times_{[Y/\mathbb{C}^{\ast}]} [\mathfrak{U}^{\star}/\mathbb{C}^{\ast}]}
(\Delta_{\ast}\oO_{\mathfrak{U}^{\star}}, \Delta_{\ast}\oO_{\mathfrak{U}^{\star}}[m](j)).
\end{align*}
Therefore 
as in the proof of Proposition~\ref{prop:koszul:Z}, 
one can lift the direct sum of (\ref{koszul:nat}) 
for all $(m, j)$
to the isomorphism of vector spaces
\begin{align*}
\Phi_{\mathfrak{U}}^{\rm{HH}} \colon 
\mathrm{HH}^{\ast}(\mathfrak{U}^{\dag}/Y)
\stackrel{\cong}{\to}
\mathrm{HH}^{\ast}(\mathfrak{U}^{\sharp}/Y)
\end{align*}
compatible with maps to natural 
transforms 
in (\ref{koszul:nat}). 
Note that $\Phi_{\mathfrak{U}}^{\rm{HH}}$ does not preserve the 
cohomological grading $\ast$. 

By ignoring cohomological grading, 
it is enough to show that the following 
diagram is commutative: 
\begin{align}\label{dia:HH:commute}
\xymatrix{
\eE^0 \oplus \eE_1 \oplus V \ar@{=}[d]
\ar[r] & \mathrm{HH}^{\ast}(\mathfrak{U}^{\dag}/\eE_0) \ar[r]  & 
\mathrm{HH}^{\ast}(\mathfrak{U}^{\dag}/Y) \ar[d]^-{\Phi_{\mathfrak{U}}^{\rm{HH}}} \\
\eE^0 \oplus \eE_1 \oplus V 
\ar[r] & \mathrm{HH}^{\ast}(\mathfrak{U}^{\sharp}/\eE^{-1}) \ar[r] & 
\mathrm{HH}^{\ast}(\mathfrak{U}^{\sharp}/Y). 
}
\end{align}

Here the left horizontal arrows are the maps (\ref{map:HH}), and the 
right horizontal arrows are natural maps via the projections 
$\eE_0 \to Y$, $\eE^{-1} \to Y$. 
The relative Hochschild cohomologies 
$\mathrm{HH}^{\ast}(\mathfrak{U}^{\star}/Y)$
and the map $\Phi_{\mathfrak{U}}^{\rm{HH}}$ can be computed 
by the similar argument as in Proposition~\ref{prop:koszul:Z}. 
Namely using the automorphism
 $(x, y)\mapsto (x+y, x-y)/2$
on $V^{\oplus 2}$ and $\eE^{\oplus 2}$, 
we have equivalences
\begin{align*}
&\mathfrak{U}^{\dag} \times_Y \mathfrak{U}^{\dag}
\stackrel{\sim}{\to}
Y^{\dag} \times_Y Y^{\dag} \times_Y \mathfrak{U}^{\flat} \times_Y
 \mathfrak{U}, \\
 &\mathfrak{U}^{\sharp} \times_Y \mathfrak{U}^{\sharp}
\stackrel{\sim}{\to}
Y^{\sharp} \times_Y Y^{\sharp} \times_Y \mathfrak{U}^{\flat} 
\times_Y \mathfrak{U}
\end{align*}
where $\oO_{\mathfrak{U}^{\flat}}=S(V^{\vee}[1])$ 
with zero differential. 
Under the above equivalences, the 
objects $\Delta_{\ast}\oO_{\mathfrak{U}^{\dag}}$, 
$\Delta_{\ast}\oO_{\mathfrak{U}^{\sharp}}$ correspond 
to $\oO_{Y^{\dag}} \boxtimes \oO_{Y} \boxtimes \oO_Y \boxtimes 
\oO_{\mathfrak{U}}$, 
$\oO_{Y} \boxtimes \oO_{Y^{\sharp}} \boxtimes \oO_Y \boxtimes 
\oO_{\mathfrak{U}}$
respectively. 
Therefore $\mathrm{HH}^{\ast}(\mathfrak{U}^{\dag}/Y)$, 
$\mathrm{HH}^{\ast}(\mathfrak{U}^{\sharp}/Y)$
are computed by cohomologies of the complexes
\begin{align}\label{compute:HH}
&\RHom_{Y^{\dag}}(\oO_{Y^{\dag}}, \oO_{Y^{\dag}}) \dotimes_{\oO_Y}
\RHom_{Y^{\dag}}(\oO_{Y}, \oO_{Y}) \dotimes_{\oO_Y}
\RHom_{\mathfrak{U}^{\flat}}(\oO_Y, \oO_Y) \dotimes_{\oO_Y} \oO_{\mathfrak{U}}, \\
\notag
&\RHom_{Y^{\sharp}}(\oO_{Y}, \oO_{Y}) \dotimes_{\oO_Y}
\RHom_{Y^{\sharp}}(\oO_{Y}^{\sharp}, \oO_{Y}^{\sharp}) \dotimes_{\oO_Y}
\RHom_{\mathfrak{U}^{\flat}}(\oO_Y, \oO_Y) \dotimes_{\oO_Y} \oO_{\mathfrak{U}}
\end{align}
respectively. 
By taking Koszul resolutions 
\begin{align*}
S(\eE[1]) \otimes_{\oO_Y}\oO_{Y^{\dag}} \stackrel{\sim}{\to} \oO_Y, \ 
S(\eE^{\vee}[2]) \otimes_{\oO_Y} \oO_{Y^{\sharp}} \stackrel{\sim}{\to} \oO_Y, \ S(V^{\vee}[2]) \otimes_{\oO_Y}\oO_{\mathfrak{U}^{\flat}} \stackrel{\sim}{\to}\oO_Y
\end{align*}
the complexes (\ref{compute:HH})
are quasi-isomorphic to 
\begin{align}\label{compute:HH2}
&S(\eE) \otimes_{\oO_Y} S(\eE^{\vee}[-1]) \otimes_{\oO_Y} S(V[-2]) \otimes_{\oO_Y}\oO_{\mathfrak{U}}, \\
\notag
&S(\eE[-2]) \otimes_{\oO_Y} S(\eE^{\vee}[1]) \otimes_{\oO_Y} S(V[-2]) \otimes_{\oO_Y}\oO_{\mathfrak{U}}
\end{align}
respectively. 
Through the above identifications, the map 
$\Phi_{\mathfrak{U}}^{\mathrm{HH}}$ is obtained by 
naturally identifying the complexes in (\ref{compute:HH2}) 
ignoring cohomological 
gradings. 

On the other hand from the constructions of 
top horizontal arrows of (\ref{dia:HH:commute})
in Subsection~\ref{subsec:ssuport}, 
it is straightforward to see that 
the compositions of 
left and right horizontal arrows in (\ref{dia:HH:commute})
are 
induced by maps to (\ref{compute:HH2})
given by 
\begin{align*}
\eE^0\oplus 
\eE_1 \oplus V \ni 
(x, y, z) \mapsto 
x \otimes y \otimes z\otimes 1.
\end{align*}
This applies to both of the top and bottom horizontal arrows in
 (\ref{dia:HH:commute}). 
Therefore 
we conclude that the diagram (\ref{dia:HH:commute}) commutes. 
\end{proof}

\section{Window theorem for DT categories}\label{sec:window:DT}
In this section, we
prove window theorem for DT categories
for $(-1)$-shifted cotangent stacks, 
and apply it to prove some conjectures proposed in
the previous sections. 
For a quasi-smooth and QCA derived stack $\fM$
with $\mM=t_0(\fM)$
and $l \in \Pic(\mM)_{\mathbb{R}}$, 
its pull-back to the $(-1)$-shifted cotangent 
determines the $l$-semistable locus
\begin{align*}
	\nN^{l\sss} \subset \nN \cneq t_0(\Omega_{\fM}[-1])
	\end{align*}
which is an open substack of $\nN$. 
We expect the existence of a triangulated subcategory 
$\wW(\fM) \subset \Dbc(\fM)$ such that the composition 
\begin{align}\label{subintro:window}
	\wW(\fM) \hookrightarrow \Dbc(\fM) \twoheadrightarrow 
	\dDT^{\C}(\nN^{l\sss})
	\end{align}
is an equivalence. 
We will develop such a theory in the case that 
$\mM$ admits a good moduli space
$\mM \to M$. More precisely 
we show that, given a symmetric structure $\mathbb{S}$ on $\fM$
there exists a subcategory $\wW(\fM)$
as above such that 
the composition (\ref{subintro:window}) is 
fully-faithful, and an equivalence if 
$l$ is also compatible with $\mathbb{S}$. 
Here we refer to Subsection~\ref{subsec:sym} for the 
notions of symmetric structures and compatibility with them. 

An idea for the construction is as follows. 
Since we have the good moduli space $\mM \to M$, 
the category $\Dbc(\fM)$ is 
obtained as a limit of $\Dbc(\fM_U)$ for 
all \'{e}tale maps $U \to M$ as in 
Theorem~\ref{thm:AHR}, where $\fM_U$ is obtained in 
Proposition~\ref{prop:extend}
so it is of the form $[\fU/G]$ for a 
reductive algebraic group $G$ and a $G$-equivariant tuple 
$(Y, V, s)$. Then we have the 
equivalence by 
Theorem~\ref{thm:knoer}
\begin{align}\label{equiv:knoer:5}
	\Phi \colon \Dbc([\fU/G])
	\stackrel{\sim}{\to}
	\MF_{\coh}^{\C}([V^{\vee}/G], w). 
	\end{align}
We then construct some subcategory 
(called intrinsic window subcategory)
$\wW(\fM_U)$ in $\Dbc([\fU/G])$ 
such that the equivalence (\ref{equiv:knoer:5}) restricts to the 
fully-faithful functor
\begin{align*}
	\Phi \colon \wW(\fM_U) \hookrightarrow
	\wW^l([V^{\vee}/G], w), 
\end{align*}
which is an equivalence if $l$ is compatible with $\bS$. 
Here the right hand side is the 
window subcategory for derived categories of factorizations 
constructed 
in~\cite{MR3327537, MR3895631}. 
We then show that intrinsic window 
subcategory is independent of a 
presentation of $\fM_U$ as $[\fU/G]$, 
so that they glue to 
define the subcategory 
\begin{align*}
	\wW(\fM) \cneq \lim_{U \to M} \wW(\fM_U) \subset \Dbc(\fM). 
	\end{align*}
The above subcategory gives a desired window subcategory. 
The proof of the presentation independence of $\wW(\fM_U)$ 
and its comparison with window subcategory under Koszul duality 
are technically hard parts of this section, which 
will be discussed in Section~\ref{subsec:intrinsic}. 

As applications of window theorem
for DT categories, we prove conjectures in 
the previous sections under some setting. 
The first one is a proof of 
Conjecture~\ref{conj1}, when 
strictly semistable sheaves on $X$
 push-forward to semistable sheaves 
 on $S$ at the wall. The next one is 
 a proof of 
 Conjecture~\ref{conj:DT/PT} when the curve class is reduced. 
 The latter can be also proved using categorified Hall products
 in Section~\ref{sec:cat:hall}. 
 
 The organization of this section is as follows. 
 In Section~\ref{subsec:window:git}, we review the original 
 window theorem for derived categories of factorizations 
 with respect to an action of a reductive algebraic group 
 on a smooth affine scheme, and prove some of its variants. 
 In Section~\ref{subsec:intrinsic}, we 
 introduce intrinsic window subcategories 
 and prove the window theorem for DT categories. 
 In Section~\ref{subsec:app1}, we apply 
 the window theorem for Conjecture~\ref{conj1}. 
 In Section~\ref{subsec:app2}, we apply the window theorem 
 for Conjecture~\ref{conj:DT/PT}.

\subsection{Window theorem for GIT quotient}\label{subsec:window:git}
\subsubsection{Kempf-Ness stratification}\label{subsec:KN}
Let $G$ be a reductive algebraic group, 
with maximal torus $T \subset G$. 
We always denote by $M$ the character lattice of $T$
and $N$ the cocharacter lattice of $T$, i.e. 
\begin{align*}
	M=\Hom_{\mathbb{Z}}(T, \mathbb{C}^{\ast}), \ 
	N=\Hom_{\mathbb{Z}}(\mathbb{C}^{\ast}, T). 
\end{align*}
The subspace $M_{\mathbb{R}}^W \subset M_{\mathbb{R}}$ is 
defined to be the 
Weyl-invariant subspace, which is identified with 
$\Pic(BG)_{\mathbb{R}}$. 
Note that $M$, $N$ are finitely generated free abelian 
groups with a perfect pairing
\begin{align*}
	\langle -, - \rangle \colon M \times N \to \mathbb{Z}. 
\end{align*}

Below we follow the convention of~\cite[Section~2.1]{MR3327537} for 
Kempf-Ness stratification associated with GIT quotients. 
Let $Y$ be a smooth affine variety with a $G$-action. 
For an element 
$l \in 
\Pic([Y/G])_{\mathbb{R}}$, 
we have the open subset of $l$-semistable points
\begin{align}\label{Yssl}
	Y^{l\sss} \subset Y.
\end{align}
By the Hilbert-Mumford criterion, 
$Y^{l\sss}$ is characterized 
by the set of points $y \in Y$ such that 
for any one parameter subgroup $\lambda \colon \mathbb{C}^{\ast} \to G$
such that 
the limit $
z=\lim_{t\to 0}\lambda(t)(y)$
exists in $Y$, we 
have $\wt(l|_{z})\ge 0$. 

We will often take 
$l \in M_{\mathbb{R}}^W=\Pic(BG)_{\mathbb{R}}$
and regard 
it as an element of $\Pic([Y/G])_{\mathbb{R}}$
by the pull-back of the natural morphism
\begin{align*}
	a \colon 
	[Y/G] \to BG.
\end{align*} 
In this case, the condition 
$\wt(l|_{z})\ge 0$ is equivalent to 
$\langle l, \lambda \rangle \ge 0$. 
In some cases we may assume that $l$
is pulled back from $\Pic(BG)_{\mathbb{R}}$
by the following lemma:
\begin{lem}\label{lem:Gchar}
	Let $x \in Y$ be fixed by $G$ and set
	$y=p(x)$ for the quotient morphism 
	$p \colon Y \to Y\ssslash G$. 
	Let $\mu \colon BG \to [Y/G]$ be the map sending 
	a point to $x$ and identity on stabilizer groups. 
	Then 
	for any $l\in \Pic([Y/G])$
	there exists a Zariski open subset 
	$x\in U \subset Y\ssslash G$
	such that $l|_{[p^{-1}(U)/G]}$ is 
	isomorphic to 
	$a^{\ast}\mu^{\ast}(l)|_{[p^{-1}(U)/G]}$. 
\end{lem}
\begin{proof}
	Let us set
	$\lL \cneq l \otimes a^{\ast}\mu^{\ast}(l)^{-1} \in \Pic([Y/G])$.
	Then $\mu^{\ast}\lL$ is a trivial line bundle on $BG$, or in other word
	$H^0(\lL|_{x})^G=\mathbb{C}$. Since $Y$ is affine and $G$ is reductive, 
	the functor $H^0(-)^G$ on $\Coh([Y/G])$ is an exact functor. 
	Therefore by applying the above functor to the surjection 
	$\lL \twoheadrightarrow \lL|_{x}$, 
	we obtain a surjection 
	$H^0(\lL)^G \twoheadrightarrow H^0(\lL|_x)^G$.
	In particular there is $s \in H^0(\lL)^G$ which is non-zero on $x$. 
	It gives a $G$-equivariant 
	map $s \colon \oO_Y \to \lL$, and let $Z \subset Y$ be the 
	zero locus of $s$. Then $Z$ is a $G$-invariant closed subset of $Y$
	which does not contain $x$. 
	However the map $[Y/G] \to Y\ssslash G$ is a good moduli space for 
	$[Y/G]$, so in particular it is universally closed (see~\cite[Theorem~4.16]{MR3237451}). 
	Therefore there exists a Zariski open subset $y\in U \subset Y\ssslash G$
	such that $Z \cap p^{-1}(U)=\emptyset$, 
	which implies that $s$ is an isomorphism on $p^{-1}(U)$. 
\end{proof}

By fixing a Weyl-invariant 
norm $\lvert \ast \rvert$ on $N_{\mathbb{R}}$, 
we have the associated Kempf-Ness (KN) stratification 
\begin{align}\label{KN:strata}
	Y=S_{1} \sqcup S_{2} \sqcup \cdots \sqcup S_N \sqcup Y^{l\sss}. 
\end{align}
Here for each $\alpha$ there exists a 
one parameter subgroup $\lambda_{\alpha} \colon \mathbb{C}^{\ast} \to
T \subset G$, a connected component
$Z_{\alpha}$ of the 
$\lambda_{\alpha}$-fixed part 
of 
$Y \setminus \cup_{\alpha'<\alpha} S_{\alpha'}$
such that 
\begin{align*}
	S_{\alpha}=G \cdot Y_{\alpha}, \ 
	Y_{\alpha}\cneq \{ y \in Y: 
	\lim_{t \to 0}\lambda_{\alpha}(t)(y) \in Z_{\alpha}\}. 
\end{align*}
We call $Z_{\alpha}$ the \textit{center} and $Y_{\alpha}$ the \textit{attracting locus}. 
Moreover by setting the slope to be
\begin{align}\label{slope:mu}
	\mu_{\alpha} \cneq -\frac{
		\wt(l|_{Z_{\alpha}})}{\lvert \lambda_{\alpha} \rvert} \in \mathbb{R}
\end{align}
we have 
the inequalities
$\mu_1>\mu_2>\cdots>0$.
By taking the quotient stacks of the 
stratification (\ref{KN:strata}), we have the 
stratification of the quotient stack $\yY=[Y/G]$
\begin{align}\notag
	\yY=\sS_{1} \sqcup \sS_2 \sqcup \cdots \sqcup \sS_N \sqcup \yY^{l\sss}.
\end{align} 

Using Hilbert-Mumford criterion, 
the notion of semistability can be generalized to 
an arbitrary Artin stack $\yY$
and $l \in \Pic(\yY)_{\mathbb{R}}$
(see~\cite[Definition~1.13]{HalpTheta}). 
Namely 
a point $p \in \yY$ is \textit{$l$-semistable} 
if for any map 
$f \colon [\mathbb{A}^1/\C] \to \yY$
with $f(1) \sim p$, we have 
$\wt(f(0)^{\ast}(l)) \ge 0$. 
Here $\C$ acts on $\mathbb{A}^1$ by weight one. 
The set of $l$-semistable points is denoted by 
\begin{align*}
	\yY^{l\sss} \subset \yY.
\end{align*}
When $\yY=[Y/G]$, then 
$\yY^{l\sss}=[Y^{l\sss}/G]$ 
where $Y^{l\sss}$ is the GIT semistable locus (\ref{Yssl}).

\subsubsection{Semiorthogonal decomposition via KN stratification}
In the setting of Subsection~\ref{subsec:KN}, 
suppose furthermore that 
$Y$ is a smooth affine variety.  
Given $l\in \Pic([Y/G])_{\mathbb{R}}$, 
we have a KN-stratification (\ref{KN:strata})
with one parameter subgroup 
$\lambda_{\alpha} \colon \C \to G$. 
We have the following subgroups in $G$
\begin{align*}
	&	P_{\alpha} \cneq \{ g \in G : 
	\mbox{ there exists } 
	\lim_{t \to 0} \lambda_{\alpha}(t) g \lambda_{\alpha}^{-1}(t) \in G\}, \\
	&	G_{\alpha} \cneq \{ g \in G : \lambda_{\alpha}(t) g \lambda_{\alpha}^{-1}(t)=g\}. 
\end{align*}
We have the following diagram (see~\cite[Definition~2.2]{MR3327537})
\begin{align}\label{dia:YZ}
	\xymatrix{
		[Y_{\alpha}/P_{\alpha}] \ar[r]^-{\cong} \ar[d] & [S_{\alpha}/G] \ar[dl]_{p_{\alpha}} \ar@<-0.3ex>@{^{(}->}[r]^-{q_{\alpha}}
		& \left[\left(Y\setminus \cup_{\alpha'<\alpha} S_{\alpha'}\right)/G \right] \ar[d]_-{w} \\
		[Z_{\alpha}/G_{\alpha}] \ar@<-0.3ex>@{^{(}->}[rru]_-{i_{\alpha}} \ar[rr]_{w|_{Z_{\alpha}}} & & \mathbb{C}. 
	}
\end{align}
Here the left arrow is induced by morphisms $Y_{\alpha} \to Z_{\alpha}$ and 
$P_{\alpha} \to G_{\alpha}$ given by
taking $t\to 0$ limit of the action of $\lambda_{\alpha}(t)$
for $t \in \C$, and $i_{\alpha}, q_{\alpha}$ are induced by the 
embedding $Z_{\alpha} \hookrightarrow Y$, $S_{\alpha} \hookrightarrow Y$ respectively. 
Let $\eta_{\alpha} \in \mathbb{Z}$ be defined by 
\begin{align}\label{eta:alpha}
	\eta_{\alpha} \cneq \wt_{\lambda_{\alpha}}(\det(N_{S_{\alpha}/Y}^{\vee}|_{Z_{\alpha}})).\end{align}
As in Subsection~\ref{subsub:fact}, 
let $\C$ acts on $Y$ which commutes with the $G$-action, 
 and $w \colon Y \to \C$ is a 
$G$-invariant function with $\C$-weight two. 
We will use the following version of window theorem. 
\begin{thm}\label{thm:window}\emph{(\cite{MR3327537, MR3895631})}
	For each $\alpha$, we take $k_{\alpha} \in \mathbb{R}$. 
	
	(i) For each $j \in \mathbb{Z}$, the composition 
	\begin{align*}
		q_{\alpha \ast} p_{\alpha}^{\ast} \colon
		\MF^{\C}_{\coh}([Z_{\alpha}/G_{\alpha}], w|_{Z_{\alpha}})_{\lambda_{\alpha} \mathchar`- \wt= j}
		\to \MF^{\C}_{\coh}([S_{\alpha}/G], w|_{S_{\alpha}})
		\to \MF^{\C}_{\coh}([(Y \setminus \cup_{\alpha'<\alpha}S_{\alpha'})/G], w)
	\end{align*}
	is fully-faithful, whose essential image is denoted by 
	$\Upsilon_{j}^{l, \alpha}$. 
	
	(ii) 
	There exist semiorthogonal decomposition 
	\begin{align}\notag
		\MF_{\coh}^{\C}([(Y \setminus \cup_{\alpha'<\alpha}S_{\alpha'})/G], w)
		=\langle \ldots, \Upsilon_{\lceil k_{\alpha} \rceil -2}^{l, \alpha}, \Upsilon_{\lceil k_{\alpha} \rceil -1}^{l, \alpha},  
		\wW_{k_{\alpha}}^{l, \alpha}, \Upsilon_{\lceil k_{\alpha} \rceil}^{l, \alpha}, \Upsilon_{\lceil k_{\alpha} \rceil +1}^{l, \alpha}, \ldots
		\rangle. 
	\end{align}	
	Here $\wW_{k_{\alpha}}^{l, \alpha}$ consists of 
	factorizations 
	$(\pP, d_{\pP})$ as in 
	(\ref{factorization})
	satisfying that 
	\begin{align}\label{condition:P}
		i_{\alpha}^{\ast}(\pP, d_{\pP}) \in 
		\bigoplus_{j \in [k_{\alpha}, k_{\alpha}+\eta_{\alpha})}
		\mathrm{MF}_{\coh}^{\C}([Z_{\alpha}/G_{\alpha}], w|_{Z_{\alpha}})_{\lambda_{\alpha} \mathchar`- \wt= j}. 
	\end{align}
	
	(iii) The composition functor
	\begin{align*}
		\wW_{k_{\alpha}}^{l, \alpha} \hookrightarrow 
		\MF_{\coh}^{\C}([(Y \setminus \cup_{\alpha'<\alpha}S_{\alpha'})/G], w)
		\twoheadrightarrow 
		\MF_{\coh}^{\C}([(Y \setminus \cup_{\alpha' \le \alpha}S_{\alpha'})/G], w)
	\end{align*}
	is an equivalence 
\end{thm}	
As a consequence of Theorem~\ref{thm:window}, we have the following 
window theorem.
Let
\begin{align*}
	\wW_{k_{\bullet}}^l([Y/G], w) \subset 
	\MF_{\coh}^{\C}([Y/G], w)
	\end{align*}
be the subcategory of objects $(\pP, d_{\pP})$
satisfying the condition (\ref{condition:P}) for all $\alpha$. 
Then the composition functor 
\begin{align*}
	\wW_{k_{\bullet}}^l([Y/G], w) \hookrightarrow 
	\MF_{\coh}^{\C}([Y/G], w)
	\twoheadrightarrow
	\MF_{\coh}^{\C}([Y^{l\sss}/G], w)
	\end{align*}
is an equivalence. 
The above window subcategory depends on a choice of $k_{\bullet}$. 
For $\delta \in \Pic([Y/G])_{\mathbb{R}}$, 
we often use the following special choice
\begin{align}\label{window:special}
	\wW_{\delta}^l([Y/G], w) \cneq 
		\wW_{k_{\bullet}}^l([Y/G], w), \ 
		k_{\alpha}=-\frac{1}{2}\eta_{\alpha}+\wt_{\lambda_{\alpha}}(\delta). 
	\end{align}

We will apply Theorem~\ref{thm:window} for a KN stratification 
of $\Crit(w)$
\begin{align*}
	\Crit(w)=S_1' \sqcup S_2' \sqcup \cdots \sqcup S_N' \sqcup \Crit(w)^{l\sss}
\end{align*}
in the following way. 
After discarding 
KN strata $S_{\alpha} \subset Y$ 
with $\Crit(w) \cap S_{\alpha}=\emptyset$, 
the above filtration is obtained by
restricting a KN filtration (\ref{KN:strata}) 
for $Y$ to $\Crit(w)$. 
Let $\lambda_{\alpha} \colon \C \to G$ be a 
one parameter subgroup for $S_{\alpha}'$
with center $Z_{\alpha}' \subset S_{\alpha}'$. 
We define 
$\overline{Z}_{\alpha} \subset Y$ to 
be the connected component 
of the $\lambda_{\alpha}$-fixed part of $Y$
which contains $Z_{\alpha}'$, 
and $\overline{Y}_{\alpha} \subset Y$ is the set of 
point $y \in Y$ with 
$\lim_{t\to 0}\lambda_{\alpha}(t)y \in \overline{Z}_{\alpha}$. 
Similarly to (\ref{dia:YZ}), 
we have the diagram
\begin{align}\label{dia:closure}
	\xymatrix{
		[\overline{Y}_{\alpha}/P_{\alpha}] \ar[r]^-{\overline{q}_{\alpha}} \ar[d]_-{\overline{p}_{\alpha}} & [Y/G] \\
		[\overline{Z}_{\alpha}/G_{\alpha}] \ar@<-0.3ex>@{^{(}->}[ur]_-{\overline{i}_{\alpha}}. 	
	}
\end{align}
By noting the equivalence (\ref{equiv:crit})
together with $\Crit(w) \cap \overline{Z}_{\alpha}=
\Crit(w|_{\overline{Z}_{\alpha}})$, the result of 
Theorem~\ref{thm:window}
implies the semiorthogonal decomposition
\begin{align}\notag
	\MF_{\coh}^{\C}([(Y \setminus \cup_{\alpha'<\alpha}S'_{\alpha'})/G], w)
	=\langle \ldots, \Upsilon_{\lceil k_{\alpha} \rceil-2}^{l, \alpha}, \Upsilon_{\lceil k_{\alpha} \rceil-1}^{l, \alpha},  
	\wW_{k_{\alpha}}^{l, \alpha}, \Upsilon_{\lceil k_{\alpha} \rceil}^{l, \alpha}, \Upsilon_{\lceil k_{\alpha} \rceil +1}^{l, \alpha}, \ldots
	\rangle
\end{align}	
with equivalences
\begin{align}\label{equiv:bar}
	\overline{q}_{\alpha \ast} \overline{p}_{\alpha}^{\ast} \colon
	\MF^{\C}_{\coh}([(\overline{Z}_{\alpha} \setminus \cup_{\alpha'<\alpha}S'_{\alpha'})/G_{\alpha}], w|_{\overline{Z}_{\alpha}})_{\lambda_{\alpha} \mathchar`- \wt= j}
	\stackrel{\sim}{\to}\Upsilon_j^{l, \alpha}. 
\end{align}
The subcategory $\wW_{k_{\alpha}}^{l, \alpha}$ 
consists of factorizations $(\pP, d_{\pP})$ such that 
\begin{align*}
	\overline{i}_{\alpha}^{\ast}(\pP, d_{\pP})|_{[(\overline{Z}_{\alpha} \setminus 
		\cup_{\alpha'<\alpha}S_{\alpha'}')/G_{\alpha}]} \in 
	\bigoplus_{j \in [k_{\alpha}, k_{\alpha}+\overline{\eta}_{\alpha})}
	\mathrm{MF}_{\coh}^{\C}([(\overline{Z}_{\alpha} \setminus 
	\cup_{\alpha'<\alpha}S_{\alpha'}')/G_{\alpha}], w|_{\overline{Z}_{\alpha}})_{\lambda_{\alpha} \mathchar`- \wt= j}
\end{align*}
where $\overline{\eta}_{\alpha}=\wt_{\lambda_{\alpha}}\det (\mathbb{L}_{\overline{q}_{\alpha}})^{\vee}|_{\overline{Z}_{\alpha}}$, 
and the composition functor
\begin{align*}
	\wW_{k_{\alpha}}^{l, \alpha} \hookrightarrow 
	\MF_{\coh}^{\C}([(Y \setminus \cup_{\alpha'<\alpha}S'_{\alpha'})/G], w)
	\twoheadrightarrow 
	\MF_{\coh}^{\C}([(Y \setminus \cup_{\alpha' \le \alpha}S'_{\alpha'})/G], w)
\end{align*}
is an equivalence.

\subsubsection{Magic window subcategories}
\label{subsec:magic}
Suppose that $Y$ is an affine space, i.e. 
$Y=\mathbb{A}^n$ for some $n$, 
and 
it is a $G$-representation. 
Let us take a decomposition into a direct sum of 
$G$-representations
\begin{align}\label{sym:Y}
	Y=\mathbb{S} \oplus \mathbb{U}
\end{align}
such that $\mathbb{S}$ is a symmetric, i.e. 
$\mathbb{S} \cong \mathbb{S}^{\vee}$ as $G$-representations. 
We call such a decomposition 
$Y=\bS\oplus \bU$ as a 
\textit{symmetric structure} of $Y$, and 
refer to it as $\bS$. 

For each
one parameter subgroup $\lambda \colon \C \to T$ and an 
element 
$\mathbb{L} \in K_0(\mathrm{Rep}(T))=\mathbb{Z}[M]$, 
we define
$\mathbb{L}^{\lambda>0}$ to be the projection of this class 
onto the subspace spanned by weights which pair positively with $\lambda$. 
We define 
$\nabla_{\mathbb{S}} \subset M_{\mathbb{R}}$ to be
\begin{align}\notag
	\nabla_{\mathbb{S}} \cneq \left\{ \chi \in M_{\mathbb{R}} : 
	\langle \chi, \lambda \rangle \in
	\left[-\frac{1}{2}\left\langle 
	\mathbb{L}^{\lambda>0}_{[\mathbb{S}/G]}|_{0}, \lambda \right\rangle, 
	\frac{1}{2}\left\langle\mathbb{L}^{\lambda>0}_{[\mathbb{S}/G]}|_{0}, \lambda
	\right\rangle \right]
	\mbox{ for all } \lambda \colon \mathbb{C}^{\ast} \to T  
	\right\}. 
\end{align}
Here 
$\mathbb{L}_{[\mathbb{S}/G]}|_{0}$ 
is the cotangent complex of the quotient stack 
$[\mathbb{S}/G]$ restricted to 
the origin, whose $K$-theory class is 
\begin{align*}
	\mathbb{L}_{[\mathbb{S}/G]}|_{0}=[\mathbb{S}^{\vee}]-[\mathfrak{g}^{\vee}] 
	\in K_0(\mathrm{Rep}(T)).
\end{align*}
Here $\mathfrak{g}$ is the Lie algebra of $G$, 
and regarded as a $T$-representation by the adjoint 
representation.
If $\lambda=\lambda_{\alpha}$
for the KN stratification (\ref{KN:strata}), 
then we have the 
(in)equalities
\begin{align}\label{id:eta}
\left\langle 
	\mathbb{L}^{\lambda>0}_{[\mathbb{S}/G]}|_{0}, \lambda
	\right\rangle \le \left\langle 
	\mathbb{L}^{\lambda>0}_{[Y/G]}|_{0}, \lambda
		\right\rangle=\eta_{\alpha}.
\end{align}
Here $\eta_{\alpha}$ is defined in (\ref{eta:alpha}), and 
the second identity is due to~\cite[Equation~(4)]{MR3327537}. 

We introduce some conditions for elements in $M_{\mathbb{R}}^W$. 
Let $\overline{\Sigma}_{\bS} \subset M_{\mathbb{R}}$ be 
the convex hull of the $T$-characters of 
$\bigwedge^{\ast}(\bS)$. 
\begin{defi}\label{def:generic}
	For $l, \delta \in M_{\mathbb{R}}^W$, we say 
	\begin{enumerate}
		\item
		$l$ is 
		$\bS$-\textit{generic} if it is contained in the linear span of 
		$\overline{\Sigma}_{\bS}$ 
		but is not parallel to any face of $\overline{\Sigma}_{\bS}$. 
			\item 
		$\delta$ is $l$-\textit{generic} 
		if $\langle \delta, \lambda \rangle \notin \mathbb{Q}$
		for any one parameter subgroup
		$\lambda \colon \C \to T$
		such that $\langle l, \lambda\rangle \neq 0$. 
		\item $l$ is compatible with 
		the symmetric structure $\bS$
		in (\ref{sym:Y}) 
		if we have 
		\begin{align*}
			Y^{l\sss}=\bS^{l\sss} \oplus \bU.
		\end{align*}
	\end{enumerate}
\end{defi}
The genericity condition for (i) is introduced in~\cite[Section~2]{HLKSAM}. 
As for (ii), for example if $l \in M_{\mathbb{Q}}^W$ then 
$\delta=\varepsilon \cdot l$ is $l$-generic 
for any $\varepsilon \in \mathbb{R} \setminus \mathbb{Q}$.
Note that if $\delta$ is $l$-generic, 
then $\langle \delta, \lambda_{\alpha}\rangle \notin \mathbb{Q}$
for any one parameter subgroup 
$\lambda_{\alpha}$ which appears in a KN stratification (\ref{KN:strata}).  
Later we will use the following lemma: 
\begin{lem}\label{rmk:generic}
	For $l_1, l_2 \in M_{\mathbb{R}}^W$, 
	there is an uncountable dense subset $U \subset \mathbb{R}^2$
	such that for $(\alpha_1, \alpha_2) \in U$, 
	$\alpha_1 l_1+\alpha_2 l_2$ is $l_1$-generic 
	and $l_2$-generic. 
\end{lem}
\begin{proof}
	For a one parameter subgroup $\lambda \colon \C \to T$
	such that either $\langle l_1, \lambda \rangle \neq 0$ or 
	$\langle l_2, \lambda \rangle \neq 0$, 
	the set of $(\alpha_1, \alpha_2) \in \mathbb{R}^2$ such that 
	$\alpha_1 \langle l_1, \lambda \rangle +\alpha_2 \langle l_2, \lambda \rangle \notin \mathbb{Q}$
	is a complement of countable number of lines in $\mathbb{R}^2$, so it is uncountable and dense. 
	As such one parameter subgroups $\lambda$
	are also countable many, we obtain the lemma. 
\end{proof}

The magic window subcategory is introduced in~\cite{HLKSAM} in order to give a
stability independent description of window subcategories. 
Here we define its variant as follows: 
\begin{defi}\label{def:mwindow}
For $\delta \in M_{\mathbb{R}}^W$, 
the subcategory (called \textit{magic window over $\bS$})
\begin{align*}
	\wW_{\delta}^{\rm{mag}/\bS}([Y/G], w) \subset 
	\mathrm{MF}_{\coh}^{\mathbb{C}^{\ast}}([Y/G], w)
\end{align*}
is defined to be 
split generated by factorizations 
$(\pP, d_{\pP})$ as in 
(\ref{factorization}),  
such that 
$\pP$ is isomorphic to $W \otimes_{\mathbb{C}}\oO_Y$
for a $(G\times \mathbb{C}^{\ast})$-representation $W$
whose $(T \times \{1\})$-weights are contained in $\delta+\nabla_{\bS}$. 
\end{defi}
The following result
is proved in the proof of~\cite[Proposition~2.6]{KoTo}, 
which is itself based on magic window theorem in~\cite{HLKSAM, MR3698338}. 
Its formal fiber version will be proved in Proposition~\ref{cor:magic}
by the same argument. 
\begin{prop}\emph{(\cite[Proposition~2.6]{KoTo})}\label{thm:HLS}
Let us take 
	$l, \delta \in M_{\mathbb{R}}^W$ such that 
	$\delta$ is $l$-generic.
	Then
	we have the inclusion 
	\begin{align*}
		\wW_{\delta}^{\rm{mag}/\bS}([Y/G], w) \subset \wW^l_{\delta}([Y/G], w)
	\end{align*}
	which is identity 
	if $l$ is $\bS$-generic and compatible with
$\bS$. 
\end{prop}

\subsubsection{Window subcategories for formal completions}
\label{subsec:fib}
Let $Y$ be a smooth affine scheme with an action of a reductive 
algebraic group $G$, and $V\to Y$ a $G$-equivariant vector bundle. 
In this paper, we often work with the formal fiber of the map
\begin{align*}
	[V/G] \to [Y/G] \to Y\ssslash G.
\end{align*}
We will use the following notation. 
For $y \in Y\ssslash G$, 
we denote by $\widehat{Y}_y \ssslash G$ be the formal completion 
of $Y\ssslash G$ at $y$, 
\begin{align*}
	\widehat{Y}_y \ssslash G \cneq \Spec \widehat{\oO}_{Y\ssslash G, y}. 
\end{align*}
Then by setting 
\begin{align*}
	\widehat{Y}_y \cneq Y\times_{Y\ssslash G} \widehat{Y}_y 
	\ssslash G, \ 
	\widehat{V}_y \cneq V \times_{Y\ssslash G} \widehat{Y}_y \ssslash G, 
\end{align*} 
we have the Cartesian squares
\begin{align}\label{form:square}
	\xymatrix{
		[\widehat{V}_y/G] \ar[r] \ar[d] \diasquare &
		[\widehat{Y}_y/G] \ar[r] \ar[d] \diasquare & \widehat{Y}_y \ssslash G
		\ar[d] \\
		[V/G] \ar[r] & [Y/G] \ar[r] & Y\ssslash G. 
	}
\end{align}
The KN stratification of $V$
with respect to the above $G$-action on 
$V$ and $l\in M_{\mathbb{R}}^W$
is restricted to a stratification on $\widehat{V}_y$
\begin{align}\label{KN:induced}
	\widehat{V}_y= \widehat{S}_1 
	\sqcup \widehat{S}_2 \sqcup \cdots \sqcup \widehat{S}_N \sqcup \widehat{V}_y^{l\sss}. 
\end{align}

Suppose furthermore that 
$Y$ is a finite dimensional $G$-representation, and 
$V \to Y$ a $G$-equivariant 
vector bundle. Then the total space of $V$ is a
direct sum of $G$-representations $V|_{0} \oplus Y$
and the projection $V\to Y$ is identified with 
the second projection.  
Below
we take a symmetric structure $\bS$ of $V$, 
\begin{align}\label{sym:V}
	V|_0 \oplus Y=\bS\oplus \bU
\end{align}
and 
prove a version of Theorem~\ref{thm:HLS}
for the formal fiber 
$[\widehat{V}_y/G]$ at 
$y=0 \in Y\ssslash G$. 
Let 
$\C$ acts on the fibers of $V \to Y$
and 
$\widehat{w}_0 \colon [\widehat{V}_0/G] \to \mathbb{C}$
a function of $\C$-weight two. 
Then for each $l,\delta \in M_{\mathbb{R}}^W$, the 
window subcategory
\begin{align}\label{window:formal}
	\wW_{\delta}^l([\widehat{V}_0/G], \widehat{w}_0)
	\subset \mathrm{MF}_{\coh}^{\C}([\widehat{V}_0/G], \widehat{w}_0)
\end{align}
is 
defined in the same way as (\ref{window:special}), 
using the induced 
stratification (\ref{KN:induced}).  
The magic window subcategory
\begin{align*}
	\wW_{\delta}^{\rm{mag}/\bS}([\widehat{V}_0/G], \widehat{w}_0)
	\subset \mathrm{MF}_{\coh}^{\C}([\widehat{V}_0/G], \widehat{w}_0)
\end{align*}
is also defined to be split generated by 
$(\pP, d_{\pP})$ where $\pP$
is isomorphic to $W \otimes_{\mathbb{C}}\oO_{\widehat{V}_0}$
for a $(G\times \mathbb{C}^{\ast})$-representation $W$
whose $(T \times \{1\})$-weights are contained in $\delta+\nabla_{\bS}$. 
We will use the following variant of Theorem~\ref{thm:HLS}: 
\begin{prop}\label{cor:magic}
Let us take 
	$l, \delta \in M_{\mathbb{R}}^W$
	such that 
	$\delta$ is $l$-generic. 
	Then 
	we have the inclusion 
	\begin{align*}
		\wW_{\delta}^{\rm{mag}/\bS}([\widehat{V}_0/G], \widehat{w}_0) \subset 
		\wW_{\delta}^l([\widehat{V}_0/G], \widehat{w}_0)
	\end{align*}
	which is identity if 
	$l$ is $\bS$-generic and compatible with $\bS$. 
	\end{prop}
\begin{proof}
	The proof is almost same as Proposition~\ref{thm:HLS}, 
	given in~\cite[Proposition~2.6]{KoTo}. 
	Let $\lambda_{\alpha} \colon \C \to T$
	be a one parameter subgroup which appears in a KN 
	stratification of $V$ with respect to the $G$-action on it, 
	and set $\eta_{\alpha}$ as in (\ref{eta:alpha}) for the above KN stratification. 
	By the assumption that $\delta$ is $l$-generic, 
	we have 
	$\langle \delta, \lambda_{\alpha}
	\rangle \pm \eta_{\alpha}/2 \notin \mathbb{Z}$. 
	Together with the
	(in)equality in (\ref{id:eta})
	\begin{align*}
		\left\langle \mathbb{L}^{\lambda_{\alpha}>0}_{[\bS/G]}|_{0}, \lambda_{\alpha} \right\rangle
		\le \left\langle \mathbb{L}^{\lambda_{\alpha}>0}_{[V/G]}|_{0}, \lambda_{\alpha} \right \rangle=\eta_{\alpha}
	\end{align*}
	we have the inclusion  
	\begin{align}\label{inc:window}
		\wW_{\delta}^{\rm{mag}}([\widehat{V}_0/G], \widehat{w}_0) \subset 
		\wW_{\delta}^l([\widehat{V}_0/G], \widehat{w}_0)
	\end{align}
	by the definition of these window subcategories. 
	The composition 
	\begin{align}\label{compose:magic}
		\wW_{\delta}^l([\widehat{V}_0/G], \widehat{w}_0)\hookrightarrow
		\mathrm{MF}_{\coh}^{\mathbb{C}^{\ast}}([\widehat{V}_0/G], \widehat{w}_0)
		\twoheadrightarrow \mathrm{MF}_{\coh}^{\mathbb{C}^{\ast}}
		([\widehat{V}_0^{l\sss}/G], \widehat{w}_0)
	\end{align}
	is an equivalence by Theorem~\ref{thm:window}.
	Therefore the composition
	\begin{align}\label{compose:formal}
		\wW_{\delta}^{\rm{mag}/\bS}([\widehat{V}_0/G], \widehat{w}_0)\hookrightarrow
		\mathrm{MF}_{\coh}^{\mathbb{C}^{\ast}}([\widehat{V}_0/G], \widehat{w}_0)
		\twoheadrightarrow \mathrm{MF}_{\coh}^{\mathbb{C}^{\ast}}
		([\widehat{V}_0^{l\sss}/G], \widehat{w}_0)
	\end{align}
	is fully-faithful. 
	
	Now suppose that $l$ is $\bS$-generic 
	and compatible with $\bS$, where $\bS$
	is the symmetric structure in 
	(\ref{sym:V}). 
	Let $\Delta_{\bS}$ be the set of isomorphism classes of finite 
	dimensional $(G \times \C)$-representations $W$ 
	whose $(T \times \{1\})$-weights are contained in 
	$\delta+\nabla_{\bS}$. 
	Using the $\bS$-genericity of $l$, 
		it is proved in 
	\cite[Proposition~3.11]{HLKSAM} that 
	any vector bundle $W'\otimes \oO_{\bS}$ on $\bS$ 
	for a 
	finite dimensional $(G\times \C)$-representation 
	$W'$
	admits a $(G\times \C)$-equivariant resolution 
	by vector bundles $W \otimes \oO_{\bS}$ 
	for $W \in \Delta_{\bS}$, 
	modulo 
	sheaves on $\bS$ supported on $l$-unstable locus.
	By pulling it back to $V$ via the projection 
	$V \to \bS$ and using the assumption that 
	$l$ is compatible with $\bS$, 
	any vector bundle $W'\otimes \oO_{V}$ for a 
	finite dimensional $(G\times \C)$-representation 
	$W'$
	admits a $(G\times \C)$-equivariant resolution 
	by vector bundles $W \otimes \oO_{V}$ 
	for $W \in \Delta_{\bS}$, 
	modulo 
	sheaves on $V$ supported on $l$-unstable locus.
	By restricting it to $\widehat{V}_0$, the same
	also applies to 
	$W'\otimes \oO_{\widehat{V}_0}$. 
	Therefore $\Dbc([\widehat{V}_0^{l\sss}/(G\times \C)])$ is split 
	generated by 
	vector bundles $W \otimes \oO_{\widehat{V}_0^{l\sss}}$
	for $W \in \Delta_{\bS}$. 
	It follows that for any object $(\pP, d_{\pP})$ in 
	$\mathrm{MF}_{\coh}^{\mathbb{C}^{\ast}}
	([\widehat{V}_0^{l\sss}/G], \widehat{w}_0)$, 
	the underlying $(G\times \C)$-equivariant sheaf $\pP$
	is a direct summand of a bounded complex
	$\wW=W^{\bullet} \otimes \oO_{\widehat{V}_0^{l\sss}}$
	for $W^i \in \Delta_{\bS}$ in the derived category 
	$\Dbc([\widehat{V}_0^{l\sss}/(G\times \C)])$. 
	Note that we have 
	\begin{align*}
		\Ext^{>0}_{[\widehat{V}_0^{l\sss}/(G\times \C)]}(W_1 \otimes \oO_{\widehat{V}_0^{l\sss}}, W_2 \otimes 
		\oO_{\widehat{V}_0^{l\sss})}) \stackrel{\cong}{\to}
		\Ext^{>0}_{[\widehat{V}_0/(G\times \C)]}(W_1 \otimes \oO_{\widehat{V}_0}, W_2 
		\otimes \oO_{\widehat{V}_0})=0
	\end{align*}
	for $W_1, W_2 \in \Delta_{\bS}$.
	Here the first isomorphism follows 
	since the composition (\ref{compose:magic}) 
	without super-potential case is also fully-faithful, and the second identity follows 
	since $\widehat{V}_0$ is affine and $G$ is reductive. 
	Therefore using the resolution property of 
	factorizations (see~\cite[Theorem~3.9]{MR3488035}, \cite[Lemma~4.10]{ADS}), 
	the differential 
	$d_{\pP}$ can be lifted to
	a differential $d_{\wW}$ on 
	the totalization of $\wW$
	so that $(\wW, d_{\wW}) \in 
	\mathrm{MF}_{\coh}^{\mathbb{C}^{\ast}}
	([\widehat{V}_0^{l\sss}/G], \widehat{w}_0)$
	and $(\pP, d_{\pP})$ is a direct summand 
	of $(\wW, d_{\wW})$. 
	Therefore 
	the
	composition (\ref{compose:formal})
	is essentially surjective, 
	so the inclusion (\ref{inc:window}) is identity. 
\end{proof}

The following is an obvious corollary of Proposition~\ref{cor:magic}: 
\begin{cor}\label{corcor:magic}
In the setting of Proposition~\ref{cor:magic}, 
let $l, l', \delta \in M_{\mathbb{R}}^W$ satisfy that 
$\delta$ is $l$-generic, $l'$-generic, and $l$ is $\bS$-generic 
and compatible with $\bS$. 
Then we have the inclusion 
\begin{align*}
\wW_{\delta}^{l}([\widehat{V}_0/G], \widehat{w}_0) \subset 
\wW_{\delta}^{l'}([\widehat{V}_0/G], \widehat{w}_0)	
	\end{align*}
which is identity if $l'$ is also $\bS$-generic and compatible with $\bS$. 
\end{cor}

\subsubsection{KN stratifications for some representations of quivers}\label{KN:quiver}
We will use the following example of KN stratifications 
for moduli spaces of representations of quivers. 
Let 
$Q$ be a quiver 
\begin{align*}
	Q=(V(Q), E(Q), s, t), \ s, t \colon E(Q) \to V(Q)
\end{align*}
where $V(Q)$ is the set of vertices, $E(Q)$ is the set of 
edges and $s, t$ are the maps which correspond to the source and 
target of the edge.
Its dual quiver
$Q^{\vee}$ is defined by 
\begin{align*}
	Q^{\vee}=(V(Q), E(Q), s^{\vee}, t^{\vee}), \ 
	s^{\vee}=t, \ t^{\vee}=s. 
\end{align*}
For a quiver $Q$, we denote by $E_{i, j}$ the set of 
$e \in E(Q)$ such that $s(e)=i$ and $t(e)=j$. 
A quiver $Q$ is called \textit{symmetric}
if $\sharp E_{i, j}=\sharp E_{j, i}$ for all $i, j \in V(Q)$. 

Let $\mathrm{Rep}(Q)$ be the category of finite dimensional 
$Q$-representations. 
We have the map taking the dimension vectors
\begin{align*}
	\mathbf{dim} \colon 
	K(\mathrm{Rep}(Q)) \to 
	\Gamma_Q \cneq \bigoplus_{i\in V(Q)} \mathbb{Z} \cdot \mathbf{e}_i. 
\end{align*}
For $\vec{v}=(v_i)_{i \in V(Q)} \in \Gamma_Q$, 
we set $\lvert \vec{v} \rvert \cneq \sum_{i} v_i$. 

Let us fix $0 \in V(Q)$ and 
take the full subquiver $Q_0 \subset Q$ whose 
vertex set $V(Q_0)$ is $V(Q) \setminus \{0\}$. 
For each $i \in V(Q_0)$, let $V_i$ be a finite dimensional 
vector space with dimension $v_i$. 
For $\vec{v}=(v_i)_{i \in V(Q_0)}$, we set
\begin{align*}
	R^{\dag}(Q_0, \vec{v}) \cneq
	\bigoplus_{e \in E(Q_0)} \Hom(V_{s(e)}, V_{t(e)}) 
	\oplus \bigoplus_{0 \to i} V_i \oplus \bigoplus_{i \to 0}V_i^{\vee}
	\oplus \mathbb{C}^{\sharp(0 \to 0)}. 
\end{align*}
The algebraic group $G=\prod_{i\in V(Q_0)} \GL(V_i)$ acts on 
$R^{\dag}(Q_0, \vec{v})$ by the conjugation. 
The quotient stack 
\begin{align*}
	\mM^{\dag}(Q_0, \vec{v})=
	\left[ R^{\dag}(Q_0, \vec{v})/G \right]
\end{align*}
is the $\C$-rigidified moduli stack of 
$Q$-representations of dimension vector $(1, \vec{v})$, where $1$ is the dimension 
vector at $0$. 

Let us take $\theta=(\theta_i)_{i \in V(Q_0)} \in \mathbb{Z}^{V(Q_0)}$
and set 
\begin{align}\notag
	\chi_{\theta} \colon G \to \C, \ 
	(g_i) \mapsto \prod_{i\in V(Q_0)} \det(g_i)^{\theta_i}. 
\end{align}
We also set $\chi_0 \cneq \chi_{\theta_i=1}$, i.e. 
$\chi_0((g_i)_{i \in V(Q_0)})=\prod_{i \in V(Q_0)} \det(g_i)$. 
Let us fix basis of $V_i$
as $\{e_{i}^{(j)} : 1\le j \le v_i\}$
and take the maximal torus
\begin{align*}
	T=\prod_{i \in V(Q_0)} \prod_{j=1}^{v_i} \C \cdot e_{i}^{(j)\vee}
	\otimes e_i^{(j)} \subset G. 
	\end{align*}
For a one parameter subgroup $\lambda \colon \C \to T$
given by $t \mapsto (t^{\lambda_i^{(j)}})$, we set
\begin{align}\label{norm:lambda}
	\lvert \lambda \rvert \cneq 
	\sqrt{\sum_{i\in V(Q_0)} \sum_{j=1}^{v_i} (\lambda_i^{(j)})^2}
	\end{align}
which determines a Wely-invariant norm on $N_{\mathbb{R}}$. 

We describe the KN stratifications for $\chi_0^{\pm}$
and the norm (\ref{norm:lambda}).  
Let us take a point 
\begin{align}\label{point:x}
	x=((A_e)_{e \in E(Q_0)}, (b_i)_{i\in V(Q_0)}, (c_i)_{i\in V(Q_0)}, u)
	\in R^{\dag}(Q_0, \vec{v}), 
\end{align}
where each elements are 
\begin{align*}
	A_e \in \Hom(V_{s(e)}, V_{t(e)}), \ 
	b_i \in V_i, \ c_i \in V_i^{\vee}, \ u \in \mathbb{C}^{\sharp(0 \to 0)}.
		\end{align*}
	Then $(A_e)_{e \in E(Q_0)}$ determine the 
	$\mathbb{C}[Q_0]$-module structure on 
	$\oplus_{i\in V(Q_0)}V_i$, 
	and $(A_e^{\vee})_{e \in E(Q_0)}$ determine the 
	$\mathbb{C}[Q_0^{\vee}]$-module structure on 
	$\oplus_{i\in V(Q_0)}V_i^{\vee}$.
We define the following sub vector spaces
\begin{align*}
	M_x^+ &\cneq \mathbb{C}[Q_0] \langle 
	b_i : i \in V(Q_0) \rangle \subset \oplus_{i \in V(Q_0)}V_i, \\
	M_x^- &\cneq \mathbb{C}[Q_0^{\vee}] \langle 
	c_i : i \in V(Q_0) \rangle \subset \oplus_{i \in V(Q_0)}V_i^{\vee}.  
\end{align*}
Namely 
$M_x^+$ is the subspace generated by $b_i$ as $\mathbb{C}[Q_0]$-module, 
and $M_x^-$ is the subspace generated by $c_i$ as $\mathbb{C}[Q_0^{\vee}]$-module. 
For $\vec{w}=(w_i)_{i\in V(Q_0)}$ with $w_i \in \mathbb{Z}_{\ge 0}$, 
we set 
\begin{align*}
	S_{\vec{w}}^{\pm} \cneq \{ x \in R^{\dag}(Q_0, \vec{v}) : 
	\mathbf{dim}(M_x^{\pm})=\vec{w}
	\}, \ 
	S_k^{\pm} \cneq \coprod_{\lvert \vec{w} \rvert=k} S_{\vec{w}}^{\pm}. 
\end{align*}

\begin{lem}\label{lem:KN}
	The stratifications 
	\begin{align*}
		R^{\dag}(Q_0, \vec{v})=S_0^{\pm} \sqcup S_1^{\pm} \sqcup \cdots \sqcup S_{\lvert \vec{v} \rvert-1}^{\pm}
		\sqcup S_{\lvert \vec{v} \rvert}^{\pm}
	\end{align*}
	are KN stratifications of $R^{\dag}(Q_0, \vec{v})$
	with respect to $\chi_0^{\pm 1}$ and the norm (\ref{norm:lambda}), where 
	$S_{\lvert \vec{v} \rvert}^{\pm}$ are semistable loci 
	$R^{\dag}(Q_0, \vec{v})^{\chi_0^{\pm 1} \sss}$. 
	The corresponding one parameter 
	subgroup 
	$\lambda_{\vec{w}}^{\pm}$
	for $S_{\vec{w}}^{\pm}$ with 
	$\lvert \vec{w} \rvert< \lvert \vec{v} \rvert$ 
can be taken as 
	\begin{align*}
		\lambda_{\vec{w}}^{\pm} \colon \C \to T, \ 
		t \mapsto (t^{\lambda_i^{(j)}}), \ 
		\lambda_i^{(j)}=\begin{cases}
			\mp 1, & 1\le j \le w_i, \\
			0, & w_i<j\le v_i.
			\end{cases}
		\end{align*}
\end{lem}
\begin{proof}
	We only prove
	the case of $\chi_0$. 
	For a one parameter 
	subgroup $\lambda \colon t \mapsto t^{(\lambda_i^{(j)})}$, 
	the slope in (\ref{slope:mu}) is given by 
	\begin{align*}
		\mu_{\lambda} \cneq -\frac{\sum_{i, j}\lambda_i^{(j)}}
		{\sqrt{\sum_{i, j}(\lambda_i^{(j)})^2}}.
		\end{align*}
	Suppose by induction that
	stratas $S_0^+ \sqcup \cdots \sqcup S_{k-1}^+$
	are chosen and let $\lambda \colon t \mapsto t^{(\lambda_i^{(j)})}$
	be a one parameter subgroup for next strata. 
	We set
	$w_i$ to be the number of $1\le j \le v_i$ such that 
	$\lambda_i^{(j)}=0$. 
	Then for $\vec{w}=(w_i)_{i \in V(Q_0)}$, 
	if $\lvert \vec{w} \rvert <k$ then 
	the $\lambda$-fixed locus is contained 
	in $S_0^+ \sqcup \cdots \sqcup S_{k-1}^+$. 
	Therefore $\lvert \vec{w} \rvert \ge k$. 
	Among such $\lambda$, 
	the slope $\mu_{\lambda}$ takes the maximal value 
	$\sqrt{\lvert\vec{v}\rvert-k}$ only if $\lvert \vec{w} \rvert=k$ and 
	$\lambda_i^{(j)}=c$ for some constant $c<0$ if 
	$\lambda_i^{(j)} \neq 0$. 
	We have the decomposition $V_i=W_i \oplus W_i^{\perp}$, where
	\begin{align*}
		W_i=\bigoplus_{\lambda_i^{(j)} =0}
		\mathbb{C} \cdot e_i^{(j)}, \ 
		W_i^{\perp}=\bigoplus_{\lambda_i^{(j)} \neq 0}
		\mathbb{C} \cdot e_i^{(j)}. 
		\end{align*}
	Then the $\lambda$-fixed locus 
consists of $x$ as in (\ref{point:x})
such that $b_i \in W_i$, $c_i \in W_i^{\vee}$ and 
$A_{ij} \in \Hom(W_i, W_j) \oplus \Hom(W_i^{\perp}, W_j^{\perp})$. 
The attracting locus is then 
$b_i \in W_i$ and $A_{ij}(W_i) \subset W_i$. 
By applying the $G$-action, we
obtain the strata $S_{\vec{w}}^+$. 
\end{proof}

We also have the following lemma: 
\begin{lem}\label{lem:theta:sss}
	For $\theta=(\theta_i)_{i \in V(Q_0)} \in \mathbb{Z}^{V(Q_0)}$ with $\theta_i>0$, we have 
	\begin{align*}
			R^{\dag}(Q_0, \vec{v})^{\chi_{\theta}^{\pm}\sss}
		=R^{\dag}(Q_0, \vec{v})^{\chi_0^{\pm}\sss}.
		\end{align*}
		\end{lem}
	\begin{proof}
			From the correspondence between GIT stability and $\theta$-stability 
		proved in~\cite[Theorem~4.1]{Kin}, 
		the $\chi_{\theta}$-semistable locus in 
		$R^{\dag}(Q_0, \vec{v})$
		corresponds to $Z_{\theta}$-semistable representations, 
		where $(\mathrm{Rep}(Q), Z_{\theta})$ 
		are Bridgeland stability conditions~\cite{MR2373143} with central 
		charges given by 
		\begin{align*}
			Z_{\theta} \colon 
			\Gamma_{Q} \to \mathbb{C}, \ 
			\mathbf{e}_i \mapsto \theta_i+\sqrt{-1} \ (i \neq 0), 
			\mathbf{e}_0 \mapsto \theta_0 +\sqrt{-1}.
		\end{align*}
	Here $\theta_0$ is determined by 
	$\theta_0+\sum_{i \in V(Q_0)} \theta_i v_i=0$. 	
	Namely a $Q$-representation $R$ is $Z_{\theta}$-semistable 
	if and only if for any subrepresentation $R' \subset R$, we have 
	$\arg Z_{\theta}(R') \le \arg Z_{\theta}(R)$. 
	Then by the condition $\theta_i>0$, 
	a $Q$-representation $R$ of dimension vector $(1, \vec{v})$ is 
	$Z_{\theta}$-semistable if and only if there is no 
	non-trivial surjection $R \twoheadrightarrow R''$
	where $R''$ has dimension vector $(0, \ast)$. 
	Indeed if there is such a surjection, then 
	$\arg Z_{\theta}(R'') < \arg Z_{\theta}(R)$ which 
	destabilizes $R$. Conversely if there is no such a surjection, then 
	any subobject $R' \subset R$ has dimension vector $(0, \ast)$
	so $\arg Z_{\theta}(R')< \arg Z_{\theta}(R)$ holds. 
	In other words, 
	 a $Q$-representation of dimension vector $(1, \vec{v})$ is 
	$Z_{\theta}$-semistable if and only if it corresponds to 
	a point $x$ as in (\ref{point:x}) so that 
	$M_x^+=\oplus_{i \in V(Q_0)}V_i$. 
	This is independent of a choice of $\theta_i>0$, so 
	the stability for $\theta$ is 
	equivalent to the stability for $\theta_i=1$. 
	The case of $\chi_{\theta}^{-1}=\chi_{-\theta}$-stability is similarly discussed. 
		\end{proof}
\begin{rmk}\label{rmk:HN}
	Although the lemma implies that semistable locus does not depend on $\theta_i>0$, the 
	KN stratification depends on $\theta$. 
		The KN stratifications in Lemma~\ref{lem:KN} are nothing but Harder-Narasimhan 
	filtrations with respect to $Z_{\theta}$-stability
	for $\theta_i=\pm 1$. 
\end{rmk}

Suppose that we have the following 
condition 
\begin{align}\label{cond:Eij}
	\sharp E_{i, j} -\sharp E_{j, i}\begin{cases}
		=0, & i, j \in V(Q_0), \\
		\ge 0, & i=0, j \in V(Q_0).
		\end{cases}
	\end{align}
In particular, 
$Q_0$ is symmetric. 
	We fix subsets $E_{0, i}' \subset E_{0, i}$ for $i \in V(Q_0)$
such that 
$\sharp E_{0, i}'=\sharp E_{i, 0}$. 
Then we have the symmetric structure 
$R^{\dag}(Q_0, \vec{v})=\mathbb{S} \oplus \mathbb{U}$,
where $\mathbb{S}$ and $\mathbb{U}$ are given by 
\begin{align}\label{sym:SU}
	\mathbb{S} &=
	\bigoplus_{e \in E(Q_0)} \Hom(V_{s(e)}, V_{t(e)}) 
	\oplus \bigoplus_{(0 \to i) \in E_{0, i}'} V_i \oplus \bigoplus_{(i \to 0) \in E_{i, 0}}V_i^{\vee}
	\oplus \mathbb{C}^{\sharp(0 \to 0)}, \\
	\notag
	\mathbb{U} &=\bigoplus_{(0 \to i) \in E_{0, i} \setminus E_{0, i}'} V_i. 
\end{align}
\begin{lem}\label{lem:quiver:generic}
	For $\theta=(\theta_i)_{i\in V(Q_0)} \in \mathbb{Z}^{V(Q_0)}$ with $\theta_i>0$, 
	the $G$-characters $\chi_{\theta}^{\pm}$ are $\bS$-generic, 
	and $\chi_{\theta}^{-}$ is compatible with $\bS$.
	\end{lem}
\begin{proof}
Let $Q' \subset Q$ be a symmetric subquiver 
such that $V(Q')=V(Q)$, and the edges of $Q'$ 
are the same as $Q$
except that the set of edges from $0$ to $i$ is $E_{0, i}'$. 
Then Lemma~\ref{lem:qgeneric} below 
shows that $\chi_{\theta}$ is $\bS$-generic, 
where we apply it for 
the quiver $Q'$ with dimension vector $(1, \vec{v})$
and 
$\theta_0=-\sum_{i\in V(Q_0)}\theta_i v_i$.  
Therefore $\chi_{\theta}^{-1}$ is also $\bS$-generic. 

By Lemma~\ref{lem:KN} and Lemma~\ref{lem:theta:sss},
the $\chi_{\theta}^{-1}$-stability on 
$R^{\dag}(Q_0, \vec{v})$ imposes constraints only on $V_i^{\vee}$ and $\End(V_i)$-factors, 
and does not impose any constraint on $V_i$-factors. 
The same also applies to the $\chi_{\theta}^{-1}$-stability 
on $\bS$. Since $\bS$ is obtained by extracting some of $V_i$-factors
from $R^{\dag}(Q_0, \vec{v})$,  
the condition 
\begin{align*}
	R^{\dag}(Q_0, \vec{v})^{\chi_{\theta}^{-1} \sss}=\bS^{\chi_{\theta}^{-1}\sss} \oplus \bU
\end{align*} 
is satisfied. Therefore $\chi_{\theta}^{-1}$ is compatible with $\bS$. 
\end{proof}

Let $\C$ acts on 
$R^{\dag}(Q_0, \vec{v})$
which commutes with the $G$-action, 
and 
$w \colon R^{\dag}(Q_0, \vec{v}) \to \mathbb{C}$ is a $G$-invariant function
with $\C$-weight two. 
	As an application of Proposition~\ref{thm:HLS}, we have the following: 
\begin{prop}\label{prop:inclu:W}
	For $\theta=(\theta_i)_{i \in V(Q_0)} \in \mathbb{Z}^{V(Q_0)}$
	with $\theta_i>0$
	and $\delta \in M_{\mathbb{R}}^W$ which is 
	$\chi_{\theta}$-generic (e.g. $\delta=\varepsilon \cdot \chi_{\theta}$ with $\varepsilon \in \mathbb{R} \setminus \mathbb{Q}$), 
	we have the inclusion 
	in $\MF_{\coh}^{\C}(\mM^{\dag}(Q_0, \vec{v}), w)$
	\begin{align}\label{inclu:W}
		\wW_{\delta}^{\chi_{\theta}^{-1}}(\mM^{\dag}(Q_0, \vec{v}), w) \subset 
		\wW_{\delta}^{\chi_{\theta}}(\mM^{\dag}(Q_0, \vec{v}), w). 
		\end{align}
	In particular, we have a fully-faithful functor 
	\begin{align*}
		\MF_{\coh}^{\C}(\mM^{\dag}(Q_0, \vec{v})^{\chi_{\theta}^{-1} \sss}, w)
		\hookrightarrow 	\MF_{\coh}^{\C}(\mM^{\dag}(Q_0, \vec{v})^{\chi_{\theta} \sss}, w).
		\end{align*}
	\end{prop}
	\begin{proof}
		By Lemma~\ref{lem:quiver:generic}, 
		we can apply Proposition~\ref{thm:HLS} (also see Corollary~\ref{corcor:magic}) and obtain the inclusion (\ref{inclu:W}).  
The last assertion follows from Theorem~\ref{thm:window}. 
\end{proof}

\begin{rmk}\label{rmk:Z2Z}
	The $\C$-action on $R^{\dag}(Q_0^{\dag}, \vec{v})$ is not essential
	in the above arguments, so 
	we also have $\mathbb{Z}/2\mathbb{Z}$-version of Proposition~\ref{prop:inclu:W}
\begin{align*}
	\wW_{\delta}^{\chi_{\theta}^{-1}, \mathbb{Z}/2\mathbb{Z}}(\mM^{\dag}(Q_0, \vec{v}), w) \subset 
	\wW_{\delta}^{\chi_{\theta}, \mathbb{Z}/2\mathbb{Z}}(\mM^{\dag}(Q_0, \vec{v}), w). 
	\end{align*}
Here both sides are window subcategories in 
$\MF_{\coh}^{\mathbb{Z}/2\mathbb{Z}}(\mM^{\dag}(Q_0, \vec{v}), w)$
defined by the same condition as (\ref{condition:P}) replacing $\C$ with $\mathbb{Z}/2\mathbb{Z}$. 
	Moreover as in Proposition~\ref{cor:magic}, we also have the 
	formal fiber version
	\begin{align*}
	\wW_{\delta}^{\chi_{\theta}^{-1}, \mathbb{Z}/2\mathbb{Z}}(\widehat{\mM}^{\dag}(Q_0, \vec{v})_y, w) \subset 
	\wW_{\delta}^{\chi_{\theta}, \mathbb{Z}/2\mathbb{Z}}(\widehat{\mM}^{\dag}(Q_0, \vec{v})_y, w).
	\end{align*}
Here both sides are window subcategories in 
$\MF_{\coh}^{\mathbb{Z}/2\mathbb{Z}}(\widehat{\mM}^{\dag}(Q_0, \vec{v})_y, w)$
for $y \in R^{\dag}(Q_0, \vec{v})\ssslash G$ 
defined by the same condition (\ref{condition:P}). 
	\end{rmk}

We have used the following lemma. 
\begin{lem}\label{lem:qgeneric}
	Let $Q$ be a symmetric quiver
	whose vertex set, edge set, are 
	denoted by $V(Q)$, $E(Q)$, respectively.
	Let $\vec{v}=(v_i)_{i\in V(Q)}$ 
	be a dimension vector of $Q$, 
	and $\{\theta_i\}_{i\in V(Q)}$ with $\theta_i \in \mathbb{Z}$
	satisfy that 
	\begin{align}\label{theta:v}
		\sum_{i\in V(Q)} \theta_i \cdot v_i=0, \ 
		\sum_{i\in V(Q)} \theta_i \cdot v_i' \neq 0
	\end{align}
	for any $\vec{v}'=(v_i')_{i\in V(Q)}$ such that 
	$0<\vec{v}'<\vec{v}$. 
	Here $0<\vec{v}'$ means $v_i'\ge 0$ for any $i\in V(Q)$
	and $\vec{v}' \neq 0$, 
	and $\vec{v}'<\vec{v}$ means $\vec{v}-\vec{v}'>0$. 
	Let $V_i$ for $i\in V(Q)$ be vector spaces with dimension $v_i$. 
	Then for $G=\prod_{i \in V(Q)} \GL(V_i)$
	and $G'=G/\C$, where $\C \subset G$ is the diagonal torus, 
	the $G'$-character
	\begin{align*}
		\chi_{\theta} \colon 
		G' \to \mathbb{C}^{\ast}, \
		(g_i)_{i\in V(Q)} \mapsto \prod_{i\in V(Q)} \det(g_i)^{\theta_i}
	\end{align*}
	is $\bS$-generic 
	with respect to
	the symmetric 
	$G'$-representation 
	\begin{align}\label{GrepV}
		\bS=
		\bigoplus_{e \in E(Q)} \Hom(V_{s(e)}, V_{t(e)}).
	\end{align}
	Here 
	$G'$ acts on (\ref{GrepV}) by conjugation. 
\end{lem}
\begin{proof}
	The proof is a slight modification of~\cite[Lemma~3.3]{KoTo}. 
	The maximal torus of $G'$ is given by $T'=T/\C$
	where 
	$T=\prod_{i\in V(Q)}T_i$ for the maximal torus 
	$T_i \subset \GL(V_i)$, 
	and 
	the character lattice $M'$ of $M$ is given by the kernel of 
	$M \to \mathbb{Z}$ dual to the diagonal embedding $\C \to T$. 
	Let 
	$\{e_{i}^{(j)} : 1\le j \le v_i\}$
	be a basis of $V_i$. 
	By fixing $i_0 \in V(Q)$ and $1\le k_0 \le v_{i_0}$, 
	we can write $M'_{\mathbb{R}}$ as
	\begin{align*}
		M_{\mathbb{R}}'=\bigoplus_{i \in V(Q)} \bigoplus_{1\le k\le v_i}
		\mathbb{R}(e_{i}^{(k)}-e_{i_0}^{(k_0)}).
	\end{align*}

	Let $\gamma_1, \ldots, \gamma_d \in M'$
	be the $T'$-weights of the
	$G'$-representation 
	(\ref{GrepV}).  
	Then any non-zero $T'$-character $\gamma_j$ 
	is of the form
	$e_{i}^{(k)}-e_{i'}^{(k')}$. 
		By~\cite[Proposition~2.1]{HLKSAM}, 
	the genericity of 
	$\chi_{\theta}$ is equivalent to 
	that for any proper subspace $H\subsetneq M'_{\mathbb{R}}$, there is a 
	one parameter subgroup $\lambda \colon \C \to T'$
	such that $\langle \gamma_j, \lambda \rangle=0$ 
	for any $\gamma_j \in H$
	and $\langle\chi_{\theta}, \lambda \rangle \neq 0$. 
Given $H$ as above, we set $\lambda \colon \C \to T'$ be to be 
	\begin{align*}
		\lambda(t)=(t^{\lambda_{i}^{(k)}})_{i \in V(Q), 1\le k \le v_i}, \ 
		\lambda_{i}^{(k)}=\left\{\begin{array}{cc}
			0, & \mbox{ if }e_{i}^{(k)}-e_{i_0}^{(k_0)} \in H, \\
			1, & \mbox{ if } e_{i}^{(k)}-e_{i_0}^{(k_0)} \notin H.
		\end{array}   \right. 
	\end{align*}
	Then $\langle\gamma_j, \lambda\rangle=0$ for any 
	$\gamma_j \in H$.
	As $\chi_{\theta}=
	\sum_{i, k}\theta_i \cdot (e_{i}^{(k)}-e_{i_0}^{(k_0)})$, we have 
	\begin{align*}
		\langle \chi_{\theta}, \lambda \rangle=
		\sum_{i \in V(Q)}
		\theta_i \cdot \sharp\{ 1\le k\le v_i : 
		e_{i}^{(k)}-e_{i_0}^{(k_0)} \notin H\} \neq 0.
	\end{align*}
	Here the latter inequality follows from (\ref{theta:v}). 
	Therefore $\chi_{\theta}$ is $\bS$-generic. 
\end{proof}

\subsection{Intrinsic window subcategories}\label{subsec:intrinsic}

\subsubsection{Symmetric structures of derived stacks}\label{subsec:sym}
Let $\fM$ be a quasi-smooth and QCA derived stack
such that $\mM=t_0(\fM)$ admits a good moduli 
space $\pi_{\mM} \colon \mM \to M$. 
Recall that for any closed point $x \in \mM$, 
its automorphism group $\Aut(x)$ is a reductive algebraic group
by~\cite[Proposition~12.14]{MR3237451}.  
We introduce the notion of symmetric structures 
for derived stacks. 
\begin{defi}\label{def:symmetric}
	A symmetric structure $\bS$ 
	of $\fM$ is a choice of 
	symmetric structures of 
	$\hH^0(\mathbb{T}_{\fM}|_{x}) \oplus \hH^1(\mathbb{T}_{\fM}|_{x})^{\vee}$
	at each closed point $x \in \mM$, i.e. 
	a direct sum
	of $\Aut(x)$-representations
	\begin{align}\label{sym:M}
		\hH^0(\mathbb{T}_{\fM}|_{x}) \oplus \hH^1(\mathbb{T}_{\fM}|_{x})^{\vee}
		=\bS_x \oplus \bU_x
	\end{align}
	such that $\bS_x$ is a symmetric 
	$\Aut(x)$-representation. 
	
	A derived stack $\fM$ is called symmetric if 
	$\hH^0(\mathbb{T}_{\fM}|_{x}) \oplus \hH^1(\mathbb{T}_{\fM}|_{x})^{\vee}$
	is a symmetric $\Aut(x)$-representation for any closed point $x \in \mM$. 
	In this case, we have 
	a symmetric structure (\ref{sym:M})
	such that $\bU_x=0$, 
	which we call a maximal symmetric structure. 
\end{defi}

Recall that we defined $G$-equivariant tuple $(Y, V, s)$
in Definition~\ref{def:tuple}. 
In this section, we always assume that $G$ is reductive. 
We will use the following lemma on symmetric 
structures.  
\begin{lem}\label{lem:symmetric}
	Let $(Y, V, s)$ be a $G$-equivariant tuple for a reductive 
	$G$, and $\fM=[\fU/G]$ the associated 
	derived stack as in Definition~\ref{def:tuple}. 
			Suppose that 
	a symmetric structure 
	of $\fM$
	is given as in (\ref{sym:M}). 
	Let $x \in [\uU/G]$ be a closed point, 
	and $G_x \subset G$ the stabilizer subgroup of $x$. 
	Then 
	there is a symmetric structure 
	of the $G_x$-representation $T_x Y \oplus V|_{x}^{\vee}$
	of the form 
	\begin{align*}
		T_x Y \oplus V|_{x}^{\vee}\cong 
		(\bS_x \oplus P_x \oplus P_x^{\vee} \oplus \mathfrak{g}/\mathfrak{g}_x) \oplus 
		\bU_x
	\end{align*}
	for some $G_x$-representation $P_x$. 
	Here $\mathfrak{g}$, $\mathfrak{g}_x$ are the Lie algebras of $G$, $G_x$, 
	and $\bS_x \oplus P \oplus P^{\vee} \oplus \mathfrak{g}/\mathfrak{g}_x$
	is the symmetric part. 
\end{lem}
\begin{proof}
	The tangent complex of $\fM$ at $x$ is given by 
	\begin{align*}
		\mathbb{T}_{\fM}|_{x}=(\mathfrak{g} \to  T_x Y \stackrel{ds|_{x}}{\to}
		V|_x). 
	\end{align*}
	Here the kernel of the first map is $\mathfrak{g}_x$. 
	Since $G_x$ is reductive, we have 
	an isomorphism of complexes of $G_x$-representations
	\begin{align*}
		(T_x Y \stackrel{ds|_{x}}{\to}
		V|_x) \cong (\Ker ds|_{x} \stackrel{0}{\to} \Cok ds|_x) \oplus (P_x \stackrel{\id}{\to} P_x)
	\end{align*}
	for some $G_x$-representation $P_x$. 
	We also have a splitting $\Ker ds|_{x} \cong \hH^0(\mathbb{T}_{\fM}|_{x}) \oplus 
	\mathfrak{g}/\mathfrak{g}_x$
	as $G_x$-representations. 
	Then we have 
	\begin{align*}
		T_x Y \oplus V|_{x}^{\vee} \cong 
		\hH^0(\mathbb{T}_{\fM}|_{x}) \oplus \hH^1(\mathbb{T}_{\fM}|_{x})^{\vee}
		\oplus P_x \oplus P_x^{\vee} \oplus \mathfrak{g}/\mathfrak{g}_x. 
	\end{align*}
	Therefore the lemma holds. 
\end{proof}

We next introduce 
the notion of `formal neighborhood theorem'
for quasi-smooth derived stacks with good moduli spaces. 

\begin{defi}\label{def:anthm}
	We say that $\fM$ satisfies
	formal neighborhood theorem if 
	for any closed point $x \in \mM$
	with $y=\pi_{\mM}(x) \in M$, 
	there exists an $\Aut(x)$-equivariant morphism
	\begin{align}\label{kappa}
		\kappa \colon \widehat{\hH^0(\mathbb{T}_{\fM}|_{x})}_0 \to \hH^1(\mathbb{T}_{\fM}|_{x})
	\end{align}
	with $\kappa(0)=0$ 
	such that, by setting $\widehat{\nN}_0 \hookrightarrow 
	\widehat{\hH^0(\mathbb{T}_{\fM}|_{x})}_0$
	to be
	the classical zero locus of $\kappa$, 
	we have commutative isomorphisms
	\begin{align}\label{dia:analytic}
		\xymatrix{
			[\widehat{\nN}_0/\Aut(x)] \ar[r]^-{\cong} \ar[d]  & \mM \times_{M}  \Spec \widehat{\oO}_{M, y}
			\ar[d] \\
			\widehat{\nN}_0 \ssslash \Aut(x) \ar[r]^-{\cong} & \Spec \widehat{\oO}_{M, y}. 
		}
	\end{align} 
	Here the top isomorphism sends $0$ to $x$, and identity on 
	stabilizer groups at these points. 
\end{defi}

\begin{rmk}
	Similarly to (\ref{form:square}), 
	the 
	scheme $\widehat{\hH^0(\mathbb{T}_{\fM}|_{x})}_0$
	is defined to be the formal fiber
	of the 
	morphism 
	\begin{align*}
		\hH^0(\mathbb{T}_{\fM}|_{x}) \to \hH^0(\mathbb{T}_{\fM}|_{x})\ssslash \Aut(x)
	\end{align*}
	at the origin. 
	This should not be confused with the 
	formal completion 
	of $\hH^0(\mathbb{T}_{\fM}|_{x})$
	at the origin. 
\end{rmk}

The following lemma is proved along with 
the similar argument of Proposition~\ref{prop:cdga2}, 
so we omit its proof. 

\begin{lem}\label{lem:analytic}
	Suppose that $\fM$ satisfies formal neighborhood theorem. 
	Then the diagram (\ref{dia:analytic}) can be extended to 
	a Cartesian diagram
	\begin{align}\notag
		\xymatrix{
			[\widehat{\nN}_0/\Aut(x)] \ar[r] \ar@<-0.3ex>@{^{(}->}[d] \diasquare & \mM 
			\ar@<-0.3ex>@{^{(}->}[d] \\
			[\widehat{\mathfrak{N}}_0/\Aut(x)] \ar[r] & \fM. 
		} 
	\end{align}
	Here the vertical arrows are 
	closed immersions given by taking the classical truncations
	and $\widehat{\mathfrak{N}}_0$ is the derived zero locus of (\ref{kappa}). 
\end{lem}

We also introduce some conditions 
for $\mathbb{R}$-line bundles on the classical 
stack $\mM=t_0(\fM)$. 
For a closed point $x \in \mM$, 
we denote by $\mu_x \colon B\Aut(x) \to \mM$
the map sending a point to $x$ and identity 
on the automorphism groups. 
For $l \in \Pic(\mM)_{\mathbb{R}}$, we set 
$l_x \cneq \mu_x^{\ast}l \in \Pic(B\Aut(x))_{\mathbb{R}}$.
Recall the conditions for 
the characters of 
reductive algebraic groups in Definition~\ref{def:generic}.  
\begin{defi}\label{def:stack:gen}
	For $l, \delta \in \Pic(\mM)_{\mathbb{R}}$, we say 
	\begin{enumerate}
		\item 
		$l$ is $\bS$-generic 
		for a symmetric structure $\bS$ given in (\ref{sym:M})
		if for any closed point $x \in \mM$, 
		the $\mathbb{R}$-line bundle 
		$l_x$ on $B\Aut(x)$ 
		is $\bS_x$-generic. 
			
		\item $\delta$ is 
		$l$-generic if $\delta_x$ 
		is $l_x$-generic for any closed point $x \in \mM$. 
		
		\item $l$ is compatible with the symmetric structure
		$\bS$ in (\ref{sym:M}) if 
		$l_x$ is compatible with the symmetric structure 
		(\ref{sym:M}) for
		any closed point $x \in \mM$. 
	\end{enumerate}
\end{defi} 
For an $\mathbb{R}$-line bundle $l \in \Pic(\mM)_{\mathbb{R}}$, 
we have the open substack of $l$-semistable locus 
\begin{align*}
	\nN^{l\sss} \subset \nN=t_0(\Omega_{\fM}[-1]). 
	\end{align*}
Here we have regarded $l$ as an $\mathbb{R}$-line bundle bundle on $\nN$ 
by the pull-back for the projection $\nN \to \mM$. 
We now state the main result in this section: 
\begin{thm}\label{thm:equivalence}
	Let $\fM$ be a quasi-smooth and QCA derived stack
	with a good moduli space 
	$\mM=t_0(\fM) \to M$.  
	Suppose that 
	a symmetric structure $\bS$ of $\fM$ is given as (\ref{sym:M}),
	and $\fM$  
	satisfies formal neighborhood theorem. 
	Let us take $l, \delta \in \Pic(\mM)_{\mathbb{R}}$ 
	such that $\delta$ is $l$-generic. 
	Then there exists a triangulated subcategory 
	$\wW_{\delta}^{\intt/\bS}(\fM) \subset \Dbc(\fM)$
	such that
		the composition
	\begin{align*}
		\Theta_l \colon 
		\wW_{\delta}^{\intt/\bS}(\fM) \hookrightarrow 
		\Dbc(\fM) \to 
		\dDT^{\C}(\nN^{l\sss})
		\end{align*}
	is fully-faithful, 
	which is an equivalence if 
	$l$ is $\bS$-generic and compatible with 
	$\bS$. 
\end{thm}
\begin{proof}
	The proof will be given in Subsection~\ref{subsec:proofthm}
	(see Theorem~\ref{thm:equivalence2}). 
\end{proof}

We have the following corollary of the 
above theorem.  

\begin{cor}\label{cor:equivalence}
	Under the assumption of Theorem~\ref{thm:equivalence}, 
	let us take 
	$l_1, l_2\in \Pic(\mM)_{\mathbb{R}}$
	such that $l_1$ is compatible with the 
	symmetric structure $\bS$. 
	Then 
	there exists a fully-faithful functor 
	\begin{align}\notag
		\Theta_{l_1, l_2} \colon 
		\dDT^{\C}(\nN^{l_1 \sss})
				\hookrightarrow 
				\dDT^{\C}(\nN^{l_2 \sss})
	\end{align}
	which is an equivalence if $l_2$ is also compatible with 
	$\bS$. In particular $\Theta_{l_1, l_2}$ is an equivalence 
	if $\fM$ is symmetric and $l_i$ are generic 
	with respect to the maximal symmetric structure. 
\end{cor}
\begin{proof}
	By Lemma~\ref{rmk:generic}, 
	for each closed point $x \in \mM$ 
	there is an uncountable many $(\varepsilon_1, \varepsilon_2) \in \mathbb{R}^2$
	such that for 
	$\delta=\varepsilon_1 l_1+\varepsilon_2 l_2$, 
	$\delta_x$ is $l_{1, x}$-generic and $l_{2, x}$-generic. 
	Since the set of isomorphism classes of reductive algebraic groups
	$\Aut(x)$ together with their representations
	$\hH^0(T_{\fM}|_{x}) \oplus \hH^1(T_{\fM}|_{x})^{\vee}$ 
	and their symmetric structures 
	is at most countable many, 
	we can take $(\varepsilon_1, \varepsilon_2)$ to be independent of $x$. 
	Then the desired fully-faithful functor $\Theta_{l_1, l_2}$ is 
	given by $\Theta_{l_2}\circ \Theta_{l_1}^{-1}$
	for the above choice of $\delta$, which is an equivalence if $l_2$
	is also compatible with $\bS$. 
	In the case that $\fM$ is symmetric, then each 
	$l_i$ is compatible with the maximal symmetric structure, 
	so $\Theta_{l_1, l_2}$ is an equivalence. 
\end{proof}

\begin{rmk}\label{rmk:commute}
	In the situation of Corollary~\ref{cor:equivalence}, 
	let $\fM' \subset \fM$ be an open immersion 
	with
	$\mM'=t_0(\fM')$, 
	$\nN'=t_0(\Omega_{\fM'}[-1])$. 
	We set 
	${\nN'}^{i} \cneq \nN^{l_i \sss} \times_{\mM} \mM'$. 
If ${\nN'}^{1}={\nN'}^{2}$, then we have the commutative diagram
	\begin{align}\label{commute:M'}
		\xymatrix{
			\dDT^{\C}(\nN^{l_1 \sss})
			\ar[r]^-{\Theta_{l_1, l_2}}
			\ar[d] & 	\dDT^{\C}(\nN^{l_2 \sss}) \ar[d] \\
				\dDT^{\C}({\nN'}^{1}) \ar[r]^-{\sim} & \dDT^{\C}({\nN'}^{2}).
		}
	\end{align}
	Here the vertical arrows are restriction functors, 
	and the bottom arrow is a natural equivalence 
	given by ${\nN'}^{1}={\nN'}^{2}$. 
	The commutative diagram (\ref{commute:M'})
	follows 
	since the compositions
	\begin{align*}
		\wW_{\delta}^{\intt/\bS}(\fM) \hookrightarrow 
		\Dbc(\fM) \to
		\dDT^{\C}(\nN^{l_i \sss}) \to 
		\dDT^{\C}({\nN'}^{i})
		\end{align*}
			are identified if
			${\nN'}^{1}={\nN'}^{2}$. 
\end{rmk}

\subsubsection{Intrinsic window subcategories}
In this subsection, we 
define the intrinsic window subcategory of $\Dbc([\fU/G])$
for a derived stack $[\fU/G]$ as in Definition~\ref{def:tuple}
in terms of 
weight conditions for objects in $\Dbc([\fU/G])$ 
under the push-forward to $[Y/G]$. 

We first prepare some notation. 
Let $\fM$ be a quasi-smooth derived stack
such that $\mM=t_0(\fM)$ admits a good moduli space 
$\mM \to M$. 
Assume that $\fM$ satisfies 
the formal neighborhood theorem (see Definition~\ref{def:anthm}) 
and fix its symmetric structure (\ref{sym:M}). 
For each map $\lambda \colon B\C \to \mM$, 
it induces the map 
of good moduli spaces
$\Spec \mathbb{C} \to M$
whose image is denoted by $y(\lambda) \in M$. 
We define $x(\lambda) \in \mM$ to be the unique 
closed point of the fiber of $\pi_{\mM} \colon \mM \to M$
at $y(\lambda)$. 
Then by the diagram (\ref{dia:analytic}), $\lambda$ factors through 
the map
\begin{align*}
	\lambda \colon B\C \to [\widehat{\nN}_0/\Aut(x(\lambda))]. 
\end{align*}
Note that any $\Aut(x(\lambda))$-representation 
induces a vector bundle on 
$[\widehat{\nN}_0/\Aut(x(\lambda))]$, so 
in particular $\bU_{x(\lambda)}$ in the decomposition (\ref{sym:M}) 
is regarded as a vector bundle on 
$[\widehat{\nN}_0/\Aut(x(\lambda))]$. 
Under the above preparation, we introduce the following definition:
\begin{defi}\label{defi:mu}
	We define $\mu^{\pm}_{\lambda} \in \mathbb{Z}$ to be
	\begin{align*}
		\mu^+_{\lambda} \cneq 
		\wt \det((\lambda^{\ast}\bU_{x(\lambda)})^{\wt >0}), \
		\mu^-_{\lambda} \cneq 
		\wt \det((\lambda^{\ast}\bU_{x(\lambda)})^{\wt <0}). 
	\end{align*}
\end{defi}

\begin{rmk}\label{rmk:mu}
	The integers $\mu^{\pm}_{\lambda}$ are independent of choices of 
	isomorphisms in (\ref{dia:analytic}) by Lemma~\ref{lem:rest0}. 
\end{rmk}

Let $(Y, V, s)$ be a $G$-equivariant tuple 
for a reductive $G$,  
and consider the associated derived stack 
$[\fU/G]$. 
Note that we have the following diagram
\begin{align}\label{dia:afst}
	\xymatrix{
		& \vV \ar@{=}[r] & [V/G]\ar[d] \\
		[\fU/G] \ar@<-0.3ex>@{^{(}->}[r]^-{j} &  \yY \ar@{=}[r] &
		[Y/G] \ar@/_10pt/[u]_s
	}
\end{align}
such that $[\fU/G]$ is the derived zero locus of $s$. 
Below we also fix a symmetric structure $\bS$ of $[\fU/G]$, 
i.e. decompositions 
\begin{align}\label{sym:U}
	\hH^0(\mathbb{T}_{[\fU/G]}|_{x}) \oplus 
	\hH^1(\mathbb{T}_{[\fU/G]}|_{x})^{\vee} =
	\bS_x \oplus \bU_x
\end{align}
of $\Aut(x)$-representations 
for each closed point $x \in [\uU/G]$
such that $\bS_x$ is symmetric. 
The intrinsic window subcategory 
with respect to the above symmetric structure is given 
as follows. 
\begin{defi}\label{defi:iwindow}
	For an element $\delta \in \Pic([\uU/G])_{\mathbb{R}}$, 
	we define 
	\begin{align}\label{def:int}
		\wW^{\rm{int}/\bS}_{\delta}([\fU/G]) \subset \Dbc([\fU/G])
	\end{align}
	to
	be the triangulated subcategory 
	consisting of $\eE \in \Dbc([\fU/G])$
	such that for any morphism
	$\lambda \colon B\mathbb{C}^{\ast} \to [\uU/G]$
	we have 
	\begin{align}\label{cond:int}
		\mathrm{wt}(\lambda^{\ast}j_{\ast}\eE)
		\subset 
		\mathrm{wt}(\lambda^{\ast}\delta)
		+\left[ \frac{1}{2}
		\mathrm{wt} \det ((\lambda^{\ast}\mathbb{L}_{\vV}|_{\yY})^{\rm{wt}<0})-
		\frac{1}{2}\mu_{\lambda}^-, 
		\frac{1}{2}
		\mathrm{wt} \det ((\lambda^{\ast}\mathbb{L}_{\vV}|_{\yY})^{\rm{wt}>0})-
		\frac{1}{2}\mu_{\lambda}^+
		\right]. 
	\end{align}
	Here $\mathbb{L}_{\vV}|_{\yY}$ is the cotangent complex on $\vV$
	restricted to the zero section of $\vV \to \yY$, 
	and we have used the same notation $\lambda$ for the 
	composition $B\C \stackrel{\lambda}{\to} [\uU/G] \stackrel{j}{\hookrightarrow} \yY$.  
\end{defi}

In what follows, for an equivalence of derived stacks 
$\bff \colon [\fU/G] \stackrel{\sim}{\to} [\fU'/G']$, we always
take the symmetric structure $\bS'$ of 
$[\fU'/G']$ induced by 
the symmetric structure 
(\ref{sym:U}) and 
isomorphisms 
$\bff^{\ast}\hH^i(\mathbb{T}_{[\fU'/G']})
\stackrel{\cong}{\to} 
\hH^i(\mathbb{T}_{[\fU/G]})$. 
We have the following 
lemma. 
\begin{lem}\label{lem:obvious}
	Suppose that we have a commutative diagram
	\begin{align}\notag
		\xymatrix{
			[\uU/G] \ar@<-0.3ex>@{^{(}->}[r] \ar[d]^-{\cong}_-{f} &
			[\fU/G] \ar@<-0.3ex>@{^{(}->}[r]^-{j} \ar[d]^-{\sim}_-{\bff} & 
			[Y/G] \ar[d]^-{\cong}_-{f} \ar@/^10pt/[r]^-{s} & [V/G] \ar[l] \ar[d]^-{\cong}_-{g}\\
			[\uU'/G'] \ar@<-0.3ex>@{^{(}->}[r] &
			[\fU'/G'] \ar@<-0.3ex>@{^{(}->}[r]^-{j'} & [Y'/G'] \ar[r] 
			\ar@/_10pt/[r]_-{s'} & [V'/G']. \ar[l]
		}
	\end{align}
	Here $(Y', V', s')$ is a $G'$-equivariant tuple
	for a reductive $G'$, 
	and $[\fU'/G']$ is the associated 
	derived stack in Definition~\ref{def:tuple}. 
	The 
	right vertical arrow is an 
	isomorphism of vector bundles, 
	the equivalence $\bff$ is induced by the 
	right commutative isomorphisms, and the left 
	vertical arrow is the induced isomorphism on 
	classical truncations. 
	
	For $\delta' \in \Pic([\uU'/G'])_{\mathbb{R}}$, 
	we take $\delta=f^{\ast}\delta' \in \Pic([\uU/G])_{\mathbb{R}}$. 
	Then we have the equivalences
	\begin{align*}
		\bff_{\ast} \colon \wW_{\delta}^{\intt/\bS}([\fU/G])\stackrel{\sim}{\to}
		\wW_{\delta'}^{\intt/\bS'}([\fU'/G']), \ 
		\bff^{\ast} \colon \wW_{\delta'}^{\intt/\bS'}([\fU'/G'])\stackrel{\sim}{\to}
		\wW_{\delta}^{\intt/\bS}([\fU/G]). 
	\end{align*}
\end{lem}
\begin{proof}
	The lemma is obvious 
	since the category 
	$\wW^{\intt/\bS}_{\delta}([\fU/G])$ is 
	defined in terms of 
	intrinsic properties of isomorphism
	classes of the diagram of stacks (\ref{dia:afst}). 
\end{proof}

Below we show that the category 
$\wW^{\rm{int}/\bS}_{\delta}([\fU/G])$ is also 
independent of possibly non-isomorphic 
presentation (\ref{dia:afst})
for a fixed $G$.  
Let $(Y', V', s')$ be another 
$G$-equivariant tuple, and 
suppose that we have a morphism 
of $G$-equivariant tuples (see Definition~\ref{def:eq:morphism})
\begin{align}\label{dia:VY0}
	\xymatrix{
		V \ar[r]^-{g} \ar[d] & V' \ar[d] \\
		Y \ar[r]^-{f} \ar@/^10pt/[u]^-{s}
		& Y', \ar@/_10pt/[u]_-{s'}&
	} \ 
	\xymatrix{
		\vV \ar[r]^-{g} \ar[d] & \vV' \ar[d] \\
		\yY \ar[r]^-{f} \ar@/^10pt/[u]^-{s}
		& \yY'. \ar@/_10pt/[u]_-{s'}&
	}
\end{align}
Here the right diagram is obtained from 
the left one by taking the quotients by $G$, 
where $\yY=[Y/G]$, $\yY'=[\yY'/G]$, 
$\vV=[V/G]$ and $\vV'=[V'/G]$. 
We assume that the induced morphism 
of derived stacks is an equivalence
\begin{align}\label{equiv:fU}
	\bff \colon [\fU/G] \stackrel{\sim}{\to} 
	[\fU'/G]
\end{align}
where 
both sides are the associated derived stacks in Definition~\ref{def:tuple}. 
As in the diagram (\ref{diagram:dual}), the 
diagram (\ref{dia:VY0}) also induces the 
following $G$-equivariant commutative diagram
\begin{align}\label{diagram:dual2}
	\xymatrix{
		&  \mathbb{C} & \\
		V^{\vee} \ar[ur]^-{w} \ar[d]_-{p} & f^{\ast}{V'}^{\vee} \ar[l]_-{g} 
		\ar[r]^-{f}  \ar@{}[dr]|\square
		\ar[d]_-{p''}
		\ar[u]^{\overline{w}}
		& {V'}^{\vee} \ar[d]^-{p'} \ar[ul]_-{w'} \\
		Y & Y \ar@{=}[l] \ar[r]^-{f} & Y'. 
	}
\end{align}
Here $w'$ is defined as in (\ref{def:w})
from $(Y', s')$. 
We prepare two lemmas.

\begin{lem}\label{lem:stack:imm}
	Suppose that we have a morphism of $G$-equivariant tuples 
	(\ref{dia:VY0}) which induces 
	an equivalence of derived stacks (\ref{equiv:fU}). 
	Moreover suppose that the morphism $f$ in the diagram (\ref{dia:VY0})
	is a closed 
	immersion. 
	Then for an object $\eE \in \Dbc(\yY)$
	and $\gamma \in \mathbb{R}$, 
	we have the condition 
	\begin{align}\notag
		\mathrm{wt}(\lambda^{\ast}\eE)
		\subset \gamma
		+\left[ \frac{1}{2}
		\mathrm{wt} \det ((\lambda^{\ast}\mathbb{L}_{\vV}|_{\yY})^{\rm{wt}<0}), 
		\frac{1}{2}
		\mathrm{wt} \det ((\lambda^{\ast}\mathbb{L}_{\vV}|_{\yY})^{\rm{wt}>0})
		\right]
	\end{align}
	for any 
	$\lambda \colon B\mathbb{C}^{\ast} \to \yY$
	if and only if 
	$f_{\ast}\eE \in \Dbc(\yY')$ satisfies the condition 
	\begin{align}\notag
		\mathrm{wt}({\lambda'}^{\ast}f_{\ast}\eE)
		\subset 
		\gamma
		+\left[ \frac{1}{2}
		\mathrm{wt} \det (({\lambda'}^{\ast}\mathbb{L}_{\vV'}|_{\yY'})^{\rm{wt}<0}), 
		\frac{1}{2}
		\mathrm{wt} \det (({\lambda'}^{\ast}\mathbb{L}_{\vV'}|_{\yY'})^{\rm{wt}>0})
		\right]
	\end{align}
	for any $\lambda \colon B\C \to \yY$, where 
	$\lambda' =\lambda \circ f \colon 
	B\mathbb{C}^{\ast} \to \yY'$. 
\end{lem}
\begin{proof}
	Since $f$ is a closed immersion, 
	for $E \in \Coh(\yY)$
	we have 
	\begin{align*}
		\hH^{-k}(f^{\ast}f_{\ast}E)
		=\tT or_k^{\oO_{\yY'}}(f_{\ast}E, f_{\ast}\oO_{\yY})
		\cong E \otimes \bigwedge^k N^{\vee}_{\yY/\yY'}.
	\end{align*}
	By the above isomorphism, 
	for any $\eE \in \Dbc(\yY)$
	the object $f^{\ast}f_{\ast}\eE \in \Dbc(\yY)$
	fits into a finite sequence of distinguished triangles
	\begin{align}\label{filt:gr}
		\xymatrix{
			\cdots \ar[r] & \eE^{-2} \ar[rr] & & \eE^{-1} \ar[rr] 
			\ar[dl] & & \eE^0=f^{\ast}f_{\ast}\eE \ar[dl] \\
			& & \pP^{-1}\ar@{.>}[ul]^{[1]} & & \pP^{0} \ar@{.>}[ul]^{[1]}&
		}
	\end{align}
	such that $\pP^{-k} \cong 
	\eE \otimes \bigwedge^k N^{\vee}_{\yY/\yY'}[k]$. 
	For an object $\fF \in \Dbc(B\C)$, we denote by 
	$\wt^{\max}(\fF) \in \mathbb{Z}$
	(resp.~$\wt^{\min}(\fF) \in \mathbb{Z}$) the 
	maximal (resp.~minimal)
	$\C$-weight of $\hH^{\bullet}(\fF)$. 
	By the distinguished triangles (\ref{filt:gr}), 
	for any map $\lambda \colon B\mathbb{C}^{\ast} \to \yY$
	we have 
	\begin{align}\label{wt:max}
		\wt^{\max}({\lambda'}^{\ast}f_{\ast}\eE)
		=\wt^{\max}(\lambda^{\ast}\eE)+\wt \det
		((\lambda^{\ast}N_{\yY/\yY'}^{\vee})^{\wt>0}).
	\end{align}
	Moreover since the diagram (\ref{dia:VY0})
	induces an equivalence of derived zero loci (\ref{equiv:fU}), by comparing 
	the cotangent complexes of $[\fU/G]$ and $[\fU'/G]$
	we have the identity in $K(B\mathbb{C}^{\ast})$
	\begin{align*}
		\lambda^{\ast}\mathbb{L}_{\yY}-\lambda^{\ast}\vV^{\vee}
		={\lambda'}^{\ast}\mathbb{L}_{\yY'}-{\lambda'}^{\ast}{\vV'}^{\vee}. 
	\end{align*}
	Therefore we have the identities in $K(B\mathbb{C}^{\ast})$
	\begin{align*}
		{\lambda'}^{\ast}\mathbb{L}_{\vV'}|_{\yY'}&=
		{\lambda'}^{\ast}\mathbb{L}_{\yY'}+{\lambda'}^{\ast}{\vV'}^{\vee} \\
		&= \lambda^{\ast}\mathbb{L}_{\yY}+{\lambda'}^{\ast}{\vV'}^{\vee}
		-\lambda^{\ast}\vV^{\vee} +{\lambda'}^{\ast}{\vV'}^{\vee} \\
		&=\lambda^{\ast}(\mathbb{L}_{\yY}+\vV^{\vee})
		+2(\lambda^{\ast}{\vV'}^{\vee}-\lambda^{\ast}\vV^{\vee}) \\
		&=\lambda^{\ast}\mathbb{L}_{\vV}|_{\yY}+2\lambda^{\ast}N_{\yY/\yY'}^{\vee}. 
	\end{align*}
	From (\ref{wt:max}), it follows that 
	\begin{align*}
		\wt^{\max}({\lambda'}^{\ast}f_{\ast}\eE)=\wt^{\max}(\lambda^{\ast}\eE)
		+\frac{1}{2}\wt \det(({\lambda'}^{\ast}\mathbb{L}_{\vV'}|_{\yY'})^{\wt>0})
		-\frac{1}{2}\wt \det((\lambda^{\ast}\mathbb{L}_{\vV}|_{\yY})^{\wt>0}).
	\end{align*}
	Similarly we have 
	\begin{align*}
		\wt^{\min}({\lambda'}^{\ast}f_{\ast}\eE)=\wt^{\min}(\lambda^{\ast}\eE)
		+\frac{1}{2}\wt \det(({\lambda'}^{\ast}\mathbb{L}_{\vV'}|_{\yY'})^{\wt<0})
		-\frac{1}{2}\wt \det((\lambda^{\ast}\mathbb{L}_{\vV}|_{\yY})^{\wt<0}).
	\end{align*}
	The lemma follows from the above two identities. 
\end{proof}

Later we will reduce some statements 
to the case that $f$ is a closed immersion 
using the following lemma: 
\begin{lem}\label{lem:closed}
	Suppose that we have a morphism of $G$-equivariant tuples
	(\ref{dia:VY0}) which induces
	an equivalence of derived stacks (\ref{equiv:fU}). 
	Then there exists
	another $G$-equivariant tuple $(Y'', V'', s'')$
	and morphisms of $G$-equivariant tuples
		\begin{align}\label{dia:YVV}
		\xymatrix{
			V \ar[r]^-{g''} \ar[d] & V'' \ar[d] \\
			Y \inclusion^-{f''} \ar@/^10pt/[u]^-{s}
			& Y'', \ar@/_10pt/[u]_-{s''}&
		} \ 
		\xymatrix{
			V' \ar[r]^-{g'} \ar[d] & V'' \ar[d] \\
			Y' \inclusion^-{f'} \ar@/^10pt/[u]^-{s'}
			& Y'', \ar@/_10pt/[u]_-{s''}&
		}
	\end{align}
satisfying the followings: 
	\begin{enumerate}
		\item The diagrams (\ref{dia:YVV})
		induce equivalences of derived stacks
		\begin{align}\notag
			\bff'' \colon [\fU/G] \stackrel{\sim}{\to}
			[\fU''/G], \
			\bff' \colon [\fU'/G] \stackrel{\sim}{\to}
			[\fU''/G]
		\end{align}
		which commute with 
		the equivalence (\ref{equiv:fU})
		in the $\infty$-category of derived stacks, i.e. 
		$\bff'\circ \bff \sim \bff''$. 
		Here $[\fU''/G]$ is the derived 
		stack associated with $G$-equivariant tuple $(Y'', V'', s'')$.  
		\item 
		The morphisms $f'$, $f''$ are closed immersions. 
	\end{enumerate} 
\end{lem}
\begin{proof}
	The proof will be given in Subsection~\ref{subsec:lem4.9}. 
\end{proof}

The following proposition shows that the intrinsic window 
subcategories are independent of presentations as 
derived critical loci for a fixed $G$. 
\begin{prop}\label{prop:compare}
	Suppose that we have a commutative diagram
	\begin{align}\label{diagram:compare}
		\xymatrix{
			BG \ar[d]^-{\cong}_-{h} &
			\ar[l] [\fU/G] \ar@<-0.3ex>@{^{(}->}[r]^-{j} \ar[d]^-{\sim}_-{\bff} & 
			[Y/G] \ar@/^10pt/[r]^-{s} & [V/G] \ar[l] \\
			BG &
			\ar[l]  [\fU'/G] \ar@<-0.3ex>@{^{(}->}[r]^-{j'} & [Y'/G] 
			\ar@/_10pt/[r]_-{s'} & [V'/G] \ar[l]
		}
	\end{align}
	where $\bff$ is an equivalence of derived stacks
	and $[\fU/G]$, $[\fU'/G]$ are derived zero loci
	of $s$, $s'$, and 
	the left horizontal arrows 
	are
	given by canonical $G$-torsors 
	$\fU \to [\fU/G]$, $\fU' \to [\fU'/G]$. 
	Then by setting $\delta=\bff^{\ast}\delta'$
	for $\delta' \in \Pic([\uU'/G])_{\mathbb{R}}$, 
	we have the equivalences
	\begin{align}\label{equiv:int}
		\bff_{\ast} \colon \wW_{\delta}^{\intt/\bS}([\fU/G])\stackrel{\sim}{\to}
		\wW_{\delta'}^{\intt/\bS'}([\fU'/G]), \ 
		\bff^{\ast} \colon \wW_{\delta'}^{\intt/\bS'}([\fU'/G])\stackrel{\sim}{\to}
		\wW_{\delta}^{\intt/\bS}([\fU/G]). 
	\end{align}
\end{prop}
\begin{proof}
	We first remark that since $\bff$ is an equivalence, 
	the functors 
	\begin{align*}
		\bff_{\ast} \colon \Dbc([\fU/G]) \to \Dbc([\fU'/G]), \ 
		\bff^{\ast} \colon \Dbc([\fU'/G]) \to \Dbc([\fU/G])
	\end{align*}
	are equivalences which are quasi-inverse each other. 
	So we have the left equivalence in (\ref{equiv:int}) if and only if 
	we have the right equivalence in (\ref{equiv:int}). 
	Moreover let 
	\begin{align*}
		\bff \colon [\fU/G] \stackrel{\bff'}{\to} [\fU''/G] \stackrel{\bff''}{\to} [\fU'/G]
	\end{align*}
	be a factorization of $\bff$, i.e. 
	$\bff \sim \bff'' \circ \bff'$, 
	such that $\bff'$, $\bff''$ are equivalences
	of derived stacks. 
	Then 
	if two of three pairs $(\bff_{\ast}, \bff^{\ast})$, $(\bff'_{\ast}, {\bff'}^{\ast})$, 
	$(\bff''_{\ast}, {\bff''}^{\ast})$ satisfy the equivalences (\ref{equiv:int}),
	then the rest of 
	them also satisfies (\ref{equiv:int}). 
	
	By taking the pull-back 
	of the bottom horizontal diagram in 
	(\ref{diagram:compare})
	via $h \colon BG \to BG$, 
	we have the commutative diagram 
	\begin{align}\notag
		\xymatrix{
			BG \ar[d]^-{=}_-{\id} &
			\ar[l] [\fU/G] \ar@<-0.3ex>@{^{(}->}[r]^-{j} \ar[d]^-{\sim}_-{\bff\circ h^{-1}} & 
			[Y/G] \ar@/^10pt/[r]^-{s} & [V/G] \ar[l] \\
			BG \ar[d]_-{h}^-{\cong}&
			\ar[l]  [\fU'/G]\ar[d]_-{h}^-{\sim} \ar@<-0.3ex>@{^{(}->}[r]^-{j'} & [Y'/G] \ar[d]_-{h}^-{\cong}
			\ar@/_10pt/[r]_-{s'} & [V'/G]\ar[d]_-{h}^-{\cong} \ar[l] \\
			BG &
			\ar[l]  [\fU'/G] \ar@<-0.3ex>@{^{(}->}[r]^-{j'} & [Y'/G] 
			\ar@/_10pt/[r]_-{s'} & [V'/G]. \ar[l]
		}
	\end{align}
	Here the $G$-actions 
	on $\fU'$, $Y'$ and $V'$ 
	in the middle horizontal diagram 
	are twisted by $h \in \Aut(G)$. 
	By the above remark together with Lemma~\ref{lem:obvious},
	we may assume that $h=\id$. 
	
	By Lemma~\ref{lem:4.11} below and 
	the first remark in the proof of this proposition, we can assume that 
	$\bff$ fits into a commutative diagram 
	\begin{align}\notag
		\xymatrix{
			[\fU/G] \ar@<-0.3ex>@{^{(}->}[r]^-{j} \ar[d]^-{\sim}_-{\bff} & 
			[Y/G] \ar[d]_-{f} \ar@/^10pt/[r]^-{s} & [V/G] \ar[l] \ar[d] \\
			[\fU'/G] \ar@<-0.3ex>@{^{(}->}[r]^-{j'} & [Y'/G] \ar[r] 
			\ar@/_10pt/[r]_-{s'} & [V'/G]. \ar[l]
		}
	\end{align}
Here the right diagram is induced by a morphism of $G$-equivariant tuples. 
	Then using Lemma~\ref{lem:closed}
	and the above mentioned remark, 
	we may also assume that $f$ is a closed immersion. 
	Then the left equivalence in (\ref{equiv:int}) follows from 
	Lemma~\ref{lem:stack:imm}, noting the presentation independence of 
	$\mu_{\lambda}^{\pm}$. 
	Therefore 
	the proposition holds.  
\end{proof}

We have used the following lemma:
\begin{lem}\label{lem:4.11}
	Suppose that $h=\id$ in the diagram (\ref{diagram:compare}).
	Then there exists a 
	$G$-equivariant tuple $(\widetilde{Y}, \widetilde{V}, \widetilde{s})$
	and morphisms of $G$-equivariant tuples
	\begin{align}\label{dia:equiv}
		\xymatrix{
			V \ar[d] & \widetilde{V} \ar[r] \ar[l] \ar[d] & V' \ar[d] \\
			Y \ar@/^10pt/[u]^-{s} & \widetilde{Y} \ar[r]^-{\widetilde{f}'} 
			\ar[l]_-{\widetilde{f}} \ar@/^10pt/[u]^-{\widetilde{s}} 
			& Y'\ar@/_10pt/[u]_-{s'}
		}
	\end{align}
	such that, by setting $\widetilde{\fU}$ to be 
	the derived zero locus of $\widetilde{s}$, 
	the above diagram induces equivalences
	$\widetilde{\bff} \colon [\widetilde{\fU}/G] \stackrel{\sim}{\to} [\fU/G]$, 
	$\widetilde{\bff}' \colon [\widetilde{\fU}/G] \stackrel{\sim}{\to} [\fU'/G]$
	which commute with $\bff$, i.e. 
	$\bff \circ \widetilde{\bff} \sim \widetilde{\bff}'$. 
\end{lem}
\begin{proof}
	The proof will be given in Subsection~\ref{subsec:proof:4.11}. 
\end{proof}

\subsection{Window theorem for DT categories}
In this section, we compare intrinsic 
window subcategories with the original window subcategories on 
derived categories of factorizations under Koszul duality in 
Theorem~\ref{thm:knoer}, and use 
it to prove window theorem for DT categories.
\subsubsection{Window subcategories under Koszul duality (linear case)}
Let $(Y, V, s)$ be a $G$-equivariant tuple for a reductive $G$, 
and $[\fU/G]$ is the associated derived stack. 
In this subsection, 
we assume that $Y=\mathbb{A}^n$ is a $G$-representation and 
$s(0)=0$. 
Let $\bS$ be a symmetric structure
of $[\fU/G]$ as in (\ref{sym:U}). 
So for each closed point $x \in [\uU/G]$, 
we have 
\begin{align*}
	\hH^0(\mathbb{T}_{[\fU/G]}|_{x}) \oplus 
	\hH^{1}(\mathbb{T}_{[\fU/G]}|_{x})^{\vee} =\mathbb{S}_x \oplus 
	\mathbb{U}_x.  
	\end{align*}
Let $G_x \subset G$ be the stabilizer subgroup of $x$ 
which is also reductive. 
By writing the total space of $V \to Y$ 
as a direct sum of $G$-representations 
$V|_{0} \oplus Y$, 
by Lemma~\ref{lem:symmetric}
we have the induced symmetric structure 
on the total space of $V^{\vee} \to Y$
as $G_x$-representations 
\begin{align}\label{sym:Vdual}
	V|_0^{\vee} \oplus Y \cong (\bS_x \oplus P_x \oplus P_x^{\vee} \oplus \mathfrak{g}/\mathfrak{g}_x) \oplus 
	\bU_x
\end{align}
for some $G_x$-representation $P_x$. 
We denote by 
$\widetilde{\bS}_x=\bS_x \oplus P_x \oplus P^{\vee}_x \oplus \mathfrak{g}/\mathfrak{g}_x$
its symmetric part. 

As in the diagram (\ref{dia:afst}), we set 
$\yY=[Y/G]$ and $\vV=[V/G]$. 
For a one parameter subgroup $\lambda \colon \C \to T_x \cneq T \cap G_x$, we 
use the same notation $\lambda \colon B\C \to \yY$ for the 
corresponding map sending a point to $x$, 
and the map on stabilizer groups is given by $\lambda \colon \C \to T_x$. 
It factors through $\lambda \colon \C \to [\uU/G]$, 
and we have $\mu_{\lambda}^{\pm} \in \mathbb{Z}$ as in Definition~\ref{defi:mu}. 
We have the following lemma: 
\begin{lem}\label{lem:eta}
	For a one parameter subgroup $\lambda \colon \C \to T_x$
	and the corresponding map $\lambda \colon B\C \to \yY$,  
	we have the identities
	\begin{align}\label{id:1}
		&
		\mathrm{wt} \det ((\lambda^{\ast}\mathbb{L}_{\vV}|_{\yY})^{\rm{wt}>0})-\mu_{\lambda}^+
		=\left\langle
		\mathbb{L}^{\lambda>0}_{[\widetilde{\bS}_x/G]}, \lambda 
		\right\rangle
		+ \wt \lambda^{\ast}K_{\yY}, \\
		\label{id:2}
		&
		\mathrm{wt} \det ((\lambda^{\ast}\mathbb{L}_{\vV}|_{\yY})^{\rm{wt}<0})
		-\mu_{\lambda}^-
		=-\left\langle
		\mathbb{L}^{\lambda>0}_{[\widetilde{\bS}_x/G]}, \lambda 
		\right\rangle
		+\wt \lambda^{\ast}K_{\yY}.
	\end{align}
Here (though $\widetilde{S}_x$ is not $G$-representation)
we denoted $\mathbb{L}^{\lambda>0}_{[\widetilde{\bS}_x/G]}
\cneq (\widetilde{\bS}_x^{\vee})^{\lambda>0}-(\mathfrak{g}^{\vee})^{\lambda>0}$
as an element of $K(BT_x)$,  
and $K_{\yY} \cneq \det(\mathbb{L}_{\yY}) \in \Pic(\yY)$. 
\end{lem}
\begin{proof}
	Since $\widetilde{\bS}_x$ is symmetric, we have the identity in 
	$K(BG_x)$
	\begin{align*}
		V|_0^{\vee}+Y-\bU_x=V|_0+Y^{\vee}-\bU_x^{\vee}. 
	\end{align*}
	Therefore for any $\lambda \colon \C \to T_x$
	we have 
	$\langle Y, \lambda \rangle=\langle V|_0+\bU_x, \lambda \rangle$. 
	Using the above identity, we 
	have the identities
	\begin{align*}
		\mathrm{wt} \det ((\lambda^{\ast}\mathbb{L}_{\vV}|_{\yY})^{\rm{wt}>0})
		&=\langle (V|_0^{\vee})^{\lambda >0}+
		(Y^{\vee})^{\lambda >0} -(\mathfrak{g}^{\vee})^{\lambda>0}, \lambda \rangle \\
		&=\langle V|_0^{\lambda >0}+(Y^{\vee})^{\lambda>0} -(\mathfrak{g}^{\vee})^{\lambda>0}, \lambda \rangle-\langle V|_{0}, \lambda \rangle  \\
		&=\langle (\widetilde{\bS}_x^{\vee})^{\lambda>0}-(\mathfrak{g}^{\vee})^{\lambda>0}, \lambda \rangle +\langle (\bU_x^{\vee})^{\lambda>0}, \lambda \rangle 
		-\langle Y, \lambda\rangle +\langle \bU_x, \lambda \rangle \\
		&=\left\langle
		\mathbb{L}^{\lambda>0}_{[\widetilde{\bS}_x/G]}, \lambda 
		\right\rangle
		+ \wt \lambda^{\ast}K_{\yY}+\mu_{\lambda}^+. 
	\end{align*}
	Therefore the identity (\ref{id:1}) holds. 
	The identity (\ref{id:2}) also holds by the same computation. 
\end{proof}

Recall that we have the Koszul duality equivalence 
\begin{align}\label{koszul:window}
	\Phi \colon \Dbc([\fU/G]) \stackrel{\sim}{\to}
	\MF_{\coh}^{\C}([V^{\vee}/G], w)
	\end{align}
from Theorem~\ref{thm:knoer}. 
The symmetric structure (\ref{sym:Vdual}) for $x=0$
gives a symmetric structure $V^{\vee}=\widetilde{\bS}_0 \oplus \bU_0$
of $G$-representations, and 
we have the magic window subcategory over $\widetilde{\bS}_0$
in the RHS of (\ref{koszul:window}) 
as in Definition~\ref{def:mwindow}. 
We also have another window subcategory in the RHS of (\ref{koszul:window})
defined as in (\ref{window:special}). 
So for $\delta, l \in M_{\mathbb{R}}^W$, 
we have subcategories (see Proposition~\ref{thm:HLS})
\begin{align*}
	\wW^{\rm{mag}/\widetilde{\bS}_0}_{\delta}([V^{\vee}/G], w)
	\subset \wW^l_{\delta}([V^{\vee}/G], w)
	\subset 	\MF_{\coh}^{\C}([V^{\vee}/G], w). 
	\end{align*}
On the other hand, 
we have subcategories 
\begin{align*}
		\wW_{\delta}^{\rm{int}/\bS}([\fU/G]) \subset 
	\wW_{\delta}^{\rm{int}/\bS_0}([\fU/G]) \subset \Dbc([\fU/G]).
\end{align*}
Here the intermediate one is 
defined by the condition (\ref{cond:int})
for all $\lambda \colon B\C \to [\uU/G]$ which is mapped to 
$0$ by the composition $[\uU/G] \to \uU\ssslash G$. 
The following proposition gives a comparison of these 
window subcategories 
under the equivalence (\ref{koszul:window}).  
\begin{prop}\label{prop:int}
	We take $l, \delta \in M_{\mathbb{R}}^W$
	such that $\delta$ is $l$-generic. 
	Then under the equivalence $\Phi$ in (\ref{koszul:window}), we have 
	\begin{align}\label{inc:windows}
			&	\Phi(\wW^{\rm{int}/\bS}_{\delta}([\fU/G]))
		\subset \wW^l_{\delta+K_{\yY}/2}([V^{\vee}/G], w), \\
		\label{inc:windows2}
		&\wW^{\rm{mag}/\widetilde{\bS}_0}_{\delta+K_{\yY}/2}([V^{\vee}/G], w) \subset 
			\Phi(\wW^{\rm{int}/\bS_0}_{\delta}([\fU/G])). 
	\end{align}
\end{prop}
\begin{proof}
	We first show the inclusion (\ref{inc:windows}). 
	Let $\lambda=\lambda_{\alpha}
	\colon \C \to T$ be a one parameter subgroup 
	which appears in a KN stratification (\ref{KN:strata}) for the 
	$G$-action on $V^{\vee}$ with respect to $l \in M_{\mathbb{R}}^W$, 
	and $(V^{\vee})^{\lambda} \to Y^{\lambda}$ the 
	restriction of the projection $V \to Y$ to the 
	$\lambda$-fixed loci.  
	Let $\eta_{\alpha}$ be defined as in (\ref{eta:alpha}) 
	for the $G$-action on $V^{\vee}$. 
		By the $l$-genericity of $\delta$, 
	we have
	$\langle \delta+K_{\yY}/2, \lambda\rangle \pm \eta_{\alpha}/2 \notin \mathbb{Z}$. 
	Therefore for an object $\eE \in \wW_{\delta}^{\intt/\bS}([\fU/G])$, 
	it is enough to show that 
	\begin{align}\label{F:restrict}
		\Phi(\eE)|_{(V^{\vee})^{\lambda}}
		\in \bigoplus_{j\in I}\mathrm{MF}_{\coh}^{\C}([(V^{\vee})^{\lambda}/G^{\lambda}], w|_{(V^{\vee})^{\lambda}})_{\lambda \mathchar`- \wt= j}. 
	\end{align} 
	Here $G^{\lambda} \subset G$ is the center of $\lambda$, and 
	$I \subset \mathbb{R}$ is the interval
	\begin{align*}
		I=\left\langle \delta +\frac{K_{\yY}}{2}, \lambda \right\rangle+
		\left[ -\frac{1}{2}\eta_{\alpha}, \frac{1}{2}\eta_{\alpha}   \right]. 
		\end{align*}
		By restricting the left hand side of (\ref{F:restrict}) 
	to the zero section $Y^{\lambda} \hookrightarrow (V^{\vee})^{\lambda}$
	and noting that $\omega|_{Y^{\lambda}}=0$, 
	we obtain the object
	\begin{align*}
		\Phi(\eE)|_{Y^{\lambda}} \in 
		\mathrm{MF}_{\coh}^{\C}([Y^{\lambda}/G^{\lambda}], 0).
	\end{align*}
	By Lemma~\ref{lem:useful}, 
	under the tautological equivalence 
	$\Dbc([Y^{\lambda}/G^{\lambda}]) \stackrel{\sim}{\to}
	\mathrm{MF}_{\coh}^{\C}([Y^{\lambda}/G^{\lambda}], 0)$
	we have $\Phi(\eE)|_{Y^{\lambda}} \cong (j_{\ast}\eE)|_{Y^{\lambda}}$. 
	Then by Lemma~\ref{lem:eta}
and noting that (see (\ref{id:eta}))
		\begin{align*}
		\left \langle \mathbb{L}^{\lambda>0}_{[\widetilde{\bS}_x/G]}|_{0}, \lambda 
		\right \rangle \le 
		\left \langle  \mathbb{L}^{\lambda>0}_{[V^{\vee}/G]}|_{0} , 
		\lambda \right \rangle=\eta_{\alpha}
	\end{align*}
		the condition (\ref{cond:int}) implies that 
	\begin{align}\label{F:restrict2}
		\Phi(\eE)|_{Y^{\lambda}} \in 
		\bigoplus_{j\in I}\mathrm{MF}_{\coh}^{\C}([Y^{\lambda}/G^{\lambda}], 0)_{\lambda \mathchar`- \wt= j}.
	\end{align}
	By comparing (\ref{F:restrict}) with (\ref{F:restrict2}), it is enough 
	to show that the pull-back by 
	the zero section 
	\begin{align*} 
		0_{Y^{\lambda}}^{\ast} \colon  
		\mathrm{MF}_{\coh}^{\C}([(V^{\vee})^{\lambda}/G^{\lambda}], w|_{(V^{\vee})^{\lambda}}) \to 
		\mathrm{MF}_{\coh}^{\C}([Y^{\lambda}/G^{\lambda}], 0)
	\end{align*}
is conservative. 
	This also follows from Lemma~\ref{lem:useful}, 
	since the push-forward $j_{\ast}$ in the diagram (\ref{dia:useful})
is conservative as $j$ is a closed immersion. 
	
	We next show the inclusion (\ref{inc:windows2}). 
	Let us take an object 
	$\pP \in \wW_{\delta+K_{\yY}/2}^{\rm{mag}/\widetilde{\bS}_0}
	([V^{\vee}/G], w)$. 
	By Lemma~\ref{lem:useful} and 
	the definition of the magic window subcategory, 
	the object $j_{\ast}\Phi^{-1}(\pP) \in \Dbc([Y/G])$ is split 
	generated by $W \otimes \oO_Y$ for $G$-representations 
	$W$ whose $T$-weights are contained in $\delta+K_{\yY}/2+\nabla_{\widetilde{\bS}_0}$. 
	It follows that, by Lemma~\ref{lem:eta},
	for any map $\lambda \colon B\C \to [\uU/G]$
	which is mapped to $0$ by $[\uU/G] \to \uU\ssslash G$
	the object $\lambda^{\ast}j_{\ast}\Phi^{-1}(\pP)$
	satisfies the weight condition (\ref{cond:int}). 
	Therefore we have $\Phi^{-1}(\pP) \in \wW_{\delta}^{\intt/\bS_0}([\fU/G])$
	from its definition. 
\end{proof}

\subsubsection{Window subcategories under Koszul duality (formal fiber case)}
We also have the formal fiber version of the results 
in the previous subsections. 
Let $Y$ be a smooth affine scheme with a $G$-action,
where $G$ is a reductive algebraic group. 
We take the formal completion of $Y\ssslash G$ at $y$, 
that is 
$\widehat{Y}_y \ssslash G  \cneq \Spec \widehat{\oO}_{Y\ssslash G, y}$
as in 
Subsection~\ref{subsec:fib}. 
Let $V \to Y$ be a $G$-equivariant vector bundle, 
and 
\begin{align*}
	[\widehat{V}_y/G] \to [\widehat{Y}_y/G]
\end{align*}
the formal fibers at $y$ as in the diagram (\ref{form:square}). 
Let $\widehat{s}_y$ be a section 
of the above vector bundle and 
$[\widehat{\fU_y}/G]$ the derived zero locus of $\widehat{s}_y$. 
Similarly to (\ref{diagram:UA}),
we have the commutative diagram  
\begin{align}\label{diagram:UAK}
	\xymatrix{
		& & [\widehat{V}_y/G] \ar[d] & [\widehat{V}_y^{\vee}/G] \ar[r]^-{\widehat{w}_y} \ar[d] & \mathbb{A}^1 \\
		[\widehat{\uU}_y/G] \ar@<-0.3ex>@{^{(}->}[r] \ar[d]_-{\pi_{\uU}}
		& [\widehat{\fU}_y/G] \ar@<-0.3ex>@{^{(}->}[r]^-{j} \ar[rd]_-{\pi_{\fU}} & [\widehat{Y}_y/G] 
		\ar[d]_-{\pi_{Y}} \ar@/_10pt/[u]_-{\widehat{s}_y} \ar@{=}[r]& [\widehat{Y}_y/G]  \ar@/_10pt/[u]_-{0}
		\ar[d]_-{\pi_Y} & \\
		\widehat{\uU}_y \ssslash G \ar@<-0.3ex>@{^{(}->}[rr]
		& & \widehat{Y}_y \ssslash G \ar@{=}[r] & \widehat{Y}_y\ssslash G. &
	}
\end{align}

The proof of Theorem~\ref{thm:knoer} applies to this 
setting, so we have the equivalence
\begin{align}\label{equiv:K}
	\widehat{\Phi}_y \colon \Dbc([\widehat{\fU}_y/G]) \stackrel{\sim}{\to}
	\mathrm{MF}_{\coh}^{\mathbb{C}^{\ast}}
	([\widehat{V}_y^{\vee}/G], \widehat{w}_y).
\end{align}
Let $x$ be the unique closed point in $[\widehat{\uU}_y/G]$.
Then for a symmetric structure $\bS$ 
of $[\fU/G]$ as in (\ref{sym:U}), 
its restriction to $x$ determines the 
symmetric structure $\bS_x$ of 
$[\widehat{\fU}_y/G]$.  
For $\delta \in \Pic([\uU/G])_{\mathbb{R}}$, 
the intrinsic 
window subcategory
\begin{align}\label{window:K}
	\wW_{\delta}^{\intt/\bS_x}([\widehat{\fU}_y/G])
	\subset \Dbc([\widehat{\fU}_y/G])
\end{align}
is defined similarly to (\ref{def:int}), using 
the closed immersion $[\widehat{\fU}_y/G] \hookrightarrow 
[\widehat{Y}_y/G]$
and the symmetric structure (\ref{sym:U}) at $x$. 
We have the following lemma which 
relates window subcategories and those on formal 
fibers. 
\begin{lem}\label{lem:rest:an}
	Let $s$ be a section of 
	$[V/G] \to [Y/G]$, and 
	$[\fU/G] \hookrightarrow [Y/G]$ its derived zero locus. 
	For each $y \in Y\ssslash G$, 
	let $\widehat{s}_y$ be 
	the section of $[\widehat{V}_y/G] \to [\widehat{Y}_y/G]$
	induced from $s$. 
	Then for an object $\eE \in \Dbc([\fU/G])$, 
	we have 
	\begin{align*}
		\eE \in \wW_{\delta}^{\intt/\bS}([\fU/G]), \quad  
		(resp.~\Phi(\eE) \in \wW^l_{\delta+K_{\yY}/2}
		([V^{\vee}/G], w))
	\end{align*}
	if and only if for any $y \in \uU\ssslash G$ 
	we have
	\begin{align*}
		\widehat{\eE}_y \in \wW_{\delta}^{\intt/\bS_x}([\widehat{\fU}_y/G]), \quad  
		(resp.~\widehat{\Phi}_y(\widehat{\eE}_y) \in \wW^l_{\delta+K_{\yY}/2}([\widehat{V}_y^{\vee}/G], \widehat{w}_y)).
	\end{align*}
	Here $\widehat{\eE}_y$ is the pull-back of $\eE$ to $[\widehat{\fU}_y/G]$, 
	and the pull-backs of $l, \delta, K_{\yY}$ to the formal fibers are 
	also denoted as $l, \delta, K_{\yY}$.  
\end{lem}
\begin{proof}
	By the construction of the equivalence in Theorem~\ref{thm:knoer}, 
	we have the commutative diagram
	\begin{align}\notag
		\xymatrix{
			\Dbc([\fU/G]) \ar[r]^-{\Phi} \ar[d] & 
			\mathrm{MF}_{\coh}^{\mathbb{C}^{\ast}}([V^{\vee}/G], w) \ar[d] \\
			\Dbc([\widehat{\fU}_y/G]) \ar[r]^-{\widehat{\Phi}_y}  & 
			\mathrm{MF}_{\coh}^{\mathbb{C}^{\ast}}([\widehat{V}_y^{\vee}/G], 
			\widehat{w}_y).
		}
	\end{align}
	Here the vertical arrows are pull-back functors. 
	The lemma follows from the above commutative diagram, 
	since the defining conditions 
	of the relevant window subcategories
	are local on $\uU\ssslash G$.  
\end{proof}

Suppose that $Y=\mathbb{A}^n$ is a $G$-representation and 
take 
the formal fibers at $0 \in Y\ssslash G$. 
Let $\widetilde{\bS}_0$ be the symmetric structure 
as in (\ref{sym:Vdual}) for $x=0$. 
In this case, we have the following formal fiber version of 
Proposition~\ref{prop:int}. 
\begin{prop}\label{prop:int2}
	We take $l, \delta \in M_{\mathbb{R}}^W$
	such that $\delta$ is $l$-generic. 
	Then under the equivalence $\widehat{\Phi}_0$ in (\ref{equiv:K}), 
	we have 
	\begin{align*}
		\wW^{\rm{mag}/\widetilde{\bS}_0}_{\delta+K_{\yY}/2}([\widehat{V}_0^{\vee}/G], \widehat{w}_0) \subset 
		\widehat{\Phi}_0(\wW^{\rm{int}/\bS_0}_{\delta}([\widehat{\fU}_0/G]))
		\subset \wW^l_{\delta+K_{\yY}/2}([\widehat{V}_0^{\vee}/G], \widehat{w}_0).
	\end{align*}
	In particular if 
	furthermore $l$ is $\widetilde{\bS}_0$-generic and compatible 
	with $\widetilde{\bS}_0$, 
	then 
	\begin{align*}
		\wW^{\rm{mag}/\widetilde{\bS}_0}_{\delta+K_{\yY}/2}([\widehat{V}_0^{\vee}/G], \widehat{w}_0) =
		\widehat{\Phi}_0(\wW^{\rm{int}/\bS_0}_{\delta}([\widehat{\fU}_0/G]))
		= \wW^l_{\delta+K_{\yY}/2}([\widehat{V}_0^{\vee}/G], \widehat{w}_0).
	\end{align*}
\end{prop}
\begin{proof}
	The argument of Proposition~\ref{prop:int}
	applies verbatim. Here we have a chain of inclusions 
	since 
	$\wW^{\rm{int}/\bS_0}_{\delta}([\widehat{\fU}_0/G])
	=\wW^{\rm{int}/\bS}_{\delta}([\widehat{\fU}_0/G])$
	for the formal fiber case. 
	Then the 
	second statement follows from Proposition~\ref{cor:magic}. 
\end{proof}

We have the following formal fiber version of Proposition~\ref{prop:compare}:
\begin{lem}\label{lem:window1.5}
	Let $Y$, $Y'$ be smooth affine schemes of finite presentation 
	over $\mathbb{C}$ with $G$-actions, 
	and 
	$V\to Y$, $V'\to Y'$ be $G$-equivariant vector bundles. 
	Suppose that we have the following diagram
	for $y \in Y \ssslash G$, $y' \in Y'\ssslash G$
	\begin{align}\label{dia:window1.5}
		\xymatrix{
			[\widehat{\fU}_y/G] \ar@<-0.3ex>@{^{(}->}[r]^-{j} \ar[d]^-{\sim}_-{\widehat{\bff}} & 
			[\widehat{Y}_y/G] \ar@/^10pt/[r]^-{\widehat{s}_y} & [\widehat{V}_y/G] \ar[l] \\
			[\widehat{\fU}_{y'}'/G] \ar@<-0.3ex>@{^{(}->}[r]^-{j'} & [\widehat{Y}_{y'}'/G] 
			\ar@/_10pt/[r]_-{\widehat{s}_{y'}'} & [\widehat{V}_{y'}'/G] \ar[l]
		}
	\end{align}
	where $\widehat{\bff}$ is an equivalence of derived stacks
	and $[\widehat{\fU}_y/G]$, $[\widehat{\fU}_{y'}'/G]$ are derived zero loci
	of the sections $\widehat{s}_y$, $\widehat{s}_{y'}'$, respectively. 
	We assume that 
	$\widehat{s}_y(x)=0$, $\widehat{s}_{y'}'(x')=0$,
	$G=\Aut(x)=\Aut(x')$, 
	where $x, x'$ are unique closed points of 
	$[\widehat{Y}_y/G]$, $	[\widehat{Y}_{y'}/G]$ respectively. 
Then for $\delta' \in \Pic([\widehat{\uU}_{y'}'/G])_{\mathbb{R}}$
and $\delta=\widehat{\bff}^{\ast}\delta'$, we have equivalences
\begin{align*}
	\widehat{\bff}_{\ast} \colon 
	\wW^{\intt/\bS_x}_{\delta}([\widehat{\fU}_y/G]) \stackrel{\sim}{\to}
	\wW^{\intt/\bS_{x'}'}_{\delta'}([\widehat{\fU}_{y'}'/G]), \ 
\widehat{\bff}^{\ast} \colon \wW^{\intt/\bS_{x'}'}_{\delta'}([\widehat{\fU}_{y'}'/G])
\stackrel{\sim}{\to}
\wW^{\intt/\bS_x}_{\delta}([\widehat{\fU}_y/G]).
\end{align*}
\end{lem}
\begin{proof}
	The proof of Proposition~\ref{prop:compare} applies 
	verbatim, using Lemma~\ref{lem:tuple:formal} and  
	Lemma~\ref{lem:tuple:formal2}, 
	instead of Lemma~\ref{lem:4.11} and Lemma~\ref{lem:closed}. 
		One subtle difference is that, in the setting of lemma, 
	we have $\widehat{\bff}(x)=x'$
	and this implies that 
	$\widehat{\bff}$ fits into a left commutative diagram 
	in (\ref{diagram:compare}). 
	Namely, let $\widehat{\bff}(x) \colon BG \to BG$
	be the induced morphism at the closed points. 
	Then the diagram 
	\begin{align}\label{commute:BG}
		\xymatrix{
			[\widehat{\fU}_y/G] \ar[r] \ar[d]_-{\widehat{\bff}} & BG \ar[d]^{\widehat{\bff}(x)} & \\
			[\widehat{\fU}'_{y'}/G] \ar[r] & BG
		}
	\end{align}
	commutes. 
	Here the horizontal arrows are 
	given by canonical $G$-torsors 
	$\widehat{\fU}_y \to [\widehat{\fU}_y/G]$, 
	$\widehat{\fU}_{y'}' \to [\widehat{\fU}_{y'}'/G]$. 
	The commutative diagram (\ref{commute:BG}) follows from 
	Lemma~\ref{lem:rest0} below. 
\end{proof}

We have used the following lemma: 

\begin{lem}\label{lem:rest0}
For morphisms $f, f' \colon [\widehat{\fU}_y/G] \to BG$, 
suppose that $f \circ \mu \cong f' \circ \mu$ 
as morphisms $BG \to BG$, where 
$\mu \colon BG \to [\widehat{\fU}_y/G]$ sends 
the point to $x$ and identity on the stabilizer groups.
Then we have  $f \sim f'$.  
\end{lem}
\begin{proof}
The proof will be given in Subsection~\ref{subsec:lemG}.
\end{proof}

\subsubsection{Window subcategories under Koszul duality (affine case)}
Let $G$ be a reductive algebraic group, 
and $(Y, V, s)$ a $G$-equivariant tuple 
as in Definition~\ref{def:tuple}. 
We consider 
the associated derived stack 
$[\fU/G]$
with a symmetric structure $\bS$ as in 
(\ref{sym:U}). 
In Proposition~\ref{prop:int}, we compared 
window subcategories under Koszul duality when $Y$ 
is a $G$-representation. By applying the results for the formal 
fibers and using \'{e}tale slice theorem, we prove a similar comparison 
result for an affine $Y$. We first prove the following 
proposition: 

\begin{prop}\label{prop:rest:an}
	Let us take 
	$l, \delta \in \Pic([\uU/G])_{\mathbb{R}}$
	such that $\delta$ is $l$-generic, 
	and they are extended to $\mathbb{R}$-line bundles on $[Y/G]$ which 
	we use the same notation $l, \delta$. 
	Then the equivalence $\Phi$ in Theorem~\ref{thm:knoer}
	restricts to the fully-faithful functor  
	\begin{align}\label{id:locglob}
		\Phi \colon 
		\wW_{\delta}^{\intt/\bS}([\fU/G]) \hookrightarrow 
		\wW^l_{\delta+K_{\yY}/2}([V^{\vee}/G], w). 
	\end{align} 
Here we have regarded $\mathbb{R}$-line bundles on $Y$ as $\mathbb{R}$-line bundles on 
$V^{\vee}$ by the pull-back of the projection 
$V^{\vee} \to Y$. 	
\end{prop}
\begin{proof}
	Let us take a closed point $y \in \uU\ssslash G$, 
	and a unique closed point $x \in [\uU/G]$ 
	in the fiber of $[\uU/G] \to \uU\ssslash G$ at $y$.
	We denote by $G_x \subset G$ the 
	stabilizer subgroup of $x$, which is also reductive. 
Below we take a representative of $x$ in $\uU$ and write $x \in \uU$. 
	By Luna's \'{e}tale slice theorem
	for the $G$-action on $Y$
	(see~\cite{MR0342523, AHR}), 
	there is a $G_x$-invariant 
	locally closed subscheme $x \in Z \subset Y$
	and Cartesian diagrams 
	\begin{align}\notag
		\xymatrix{
			[(V, x)/G] \ar[d]_-{} & [(V|_{Z}, x)/G_x] \ar@{}[ld]|\square
			\ar@{}[rd]|\square
			\ar[l] \ar[r] \ar[d]_-{\pi_Z} & [(T_x (V|_{Z}), 0)/G_x] \ar[d]^-{} \\
			[(Y, x)/G] \ar[d]_-{\pi_Y} & [(Z, x)/G_x] \ar@{}[ld]|\square
			\ar@{}[rd]|\square
			\ar[l] \ar[r] \ar[d]_-{\pi_Z} & [(T_x Z, 0)/G_x] \ar[d]^-{\pi_{T}} \\
			(Y \ssslash G, y) & (Z \ssslash G_x, y) 
			\ar[l] \ar[r] & (T_x Z \ssslash G_x, 0). 
		}
	\end{align}
	Here each horizontal arrows are \'{e}tale morphisms, and 
	$T_x Z, T_x(V|_Z)$ are the Zariski tangent spaces of $Z$, 
	$V|_{Z}$ at $x$, where 
	we regard $x$ as a point of $V|_{Z}$ by the zero section 
	of $V|_{Z} \to Z$.   
	Note that $Z$ is smooth since $Y$ is smooth and 
	$[Z/G_x] \to [Y/G]$ is \'{e}tale. 
	Also note that $T_x(V|_{Z}) =V|_x \oplus T_x Z$ as $G_x$-representations. 
	
	By taking the formal fibers of left arrows 
	at $y \in Y\ssslash G$
	and 
	the right arrows at $0 \in T_x Z\ssslash G_x$, 
	we obtain the commutative diagram (see the diagram~(\ref{diagram:UAK}) 
	for the notation)
	\begin{align}\label{diagram:analytic}
		\xymatrix{
			[\widehat{V}_y/G] \ar[r]^-{\cong} \ar[d] 
			&   [(V|_{x} \times \widehat{(T_x Z)}_0)/G_x] \ar[d] \\
			[\widehat{Y}_y/G] \ar[r]^-{\cong} \ar@/^10pt/[u]^-{\widehat{s}_y}
			& [(\widehat{T_x Z})_0/G_x]
			\ar@/_10pt/[u]_{\widehat{t}_0}.
		}
	\end{align}
	Here $\widehat{s}_y$ is induced by the section $s \colon Y \to V$, 
	and $\widehat{t}_0$ is defined by the commutative diagram 
	(\ref{diagram:analytic}). 
	In particular, 
	$[\widehat{\fU}_y/G]$ is
	equivalent to the derived zero locus of $\widehat{t}_0$.
	By Lemma~\ref{lem:rest:an}, 
	in order to show that $\Phi$ restricts to the functor 
	(\ref{id:locglob}), it is enough to show that 
	the functor $\widehat{\Phi}_y$ in (\ref{equiv:K}) restricts to the functor 
		\begin{align}\label{Phiy:rest:form}
	\widehat{\Phi}_y \colon 
	\wW_{\delta}^{\intt/\bS_x}([\widehat{\fU}_y/G]) \to 
	\wW^l_{\delta+K_{\yY}/2}([\widehat{V}_y^{\vee}/G], \widehat{w}_y).
	\end{align}
	Here both sides are the 
window subcategories (\ref{window:K}), (\ref{window:formal}). 
	By the commutative diagram (\ref{diagram:analytic})
	and Lemma~\ref{lem:obvious}, we can reduce the above claim 
	to the corresponding claim in the right hand side of (\ref{diagram:analytic}), 
	which is the case of the formal fiber for a linear representation. 
	Namely let $[\widehat{\fT}_0/G_x]$ be the derived zero locus of $\widehat{t}_0$, 
	and consider the Koszul duality equivalence in Theorem~\ref{thm:knoer}
	\begin{align*}
	\widehat{\Phi}_T \colon 	\Dbc([\widehat{\fT}_0/G_x]) \stackrel{\sim}{\to}
		\MF_{\coh}^{\C}([(V^{\vee}|_{x} \times \widehat{(T_x Z)}_0)/G_x], \widehat{w}_0).
		\end{align*}
	Here $\widehat{w}_0$ is defined from $\widehat{t}_0$ as in (\ref{def:w}). 
	Then Proposition~\ref{prop:int2} implies that 
	$\widehat{\Phi}_T$ restricts to the functor
	\begin{align}\label{window:PhiT}
		\widehat{\Phi}_T \colon 	\wW_{\delta}^{\rm{int}/\bS_x}([\widehat{\fT}_0/G_x]) \to
	\wW_{\delta+K_{\yY}/2}^{l}([(V^{\vee}|_{x} \times \widehat{(T_x Z)}_0)/G_x], \widehat{w}_0).	
		\end{align}
	Here we have regarded $\delta, l, \bS_x, K_{\yY}$ as objects in the right hand side of (\ref{diagram:analytic})
	by the isomorphisms in (\ref{diagram:analytic}). 
	By Lemma~\ref{lem:obvious}, we conclude that
	 $\widehat{\Phi}_y$ restricts 
	to the functor (\ref{Phiy:rest:form}) as desired. 
	Therefore the functor $\Phi$ in Theorem~\ref{thm:knoer} restricts 
	to the functor (\ref{id:locglob}), which is fully-faithful.

\end{proof}

In the next lemma, we show that if 
the Koszul duality equivalence 
restricts to the equivalence of 
window subcategories for some presentation as a derived zero 
locus, then the 
same property also holds for other presentations. 
\begin{lem}\label{lem:window}
	In the situation of Proposition~\ref{prop:compare}, 
	assume that $h=\id$ in the diagram (\ref{diagram:compare}). 
Let $\Phi$, $\Phi'$ be equivalences in Theorem~\ref{thm:knoer}
	applied for $[\fU/G]$, $[\fU'/G]$. 
	We
	take $l, \delta\in M_{\mathbb{R}}^W$ 
	such that $\delta$ is $l$-generic. 
	Then 
	$\Phi$ restricts to the 
	equivalence
	\begin{align*}
		\Phi \colon \wW^{\intt/\bS}_{\delta}([\fU/G]) \stackrel{\sim}{\to} \wW^l_{\delta+K_{\yY}/2}([V^{\vee}/G], w)
	\end{align*}
	if and only if $\Phi'$ restricts to 
	the equivalence
	\begin{align*}
		\Phi' \colon \wW^{\intt/\bS'}_{\delta}([\fU'/G]) \stackrel{\sim}{\to}
		\wW^l_{\delta +K_{\yY'}/2}([{V'}^{\vee}/G], w').
	\end{align*}
\end{lem}
\begin{proof}
	As in the proof of Proposition~\ref{prop:compare}, 
	we may assume that $\bff$ is induced by a $G$-equivariant 
	diagram (\ref{dia:VY0}) such that $f$ is a closed immersion. 
	Then in the notation of the diagram (\ref{diagram:dual2}), 
	we have the commutative diagram by Lemma~\ref{lem:commute1}
	\begin{align}\label{commute:KZ}
		\xymatrix{
			\Dbc([\fU/G])
			\ar[r]^-{\Phi} 
			\ar[d]_-{\bff_{\ast}} &
			\mathrm{MF}_{\coh}^{\mathbb{C}^{\ast}}([V^{\vee}/G], w)
			\ar[d]^-{f_{\ast} g^{\ast}} \\
			\Dbc([\fU'/G]) 
			\ar[r]^-{\Phi'} 
			& \mathrm{MF}_{\coh}^{\mathbb{C}^{\ast}}([{V'}^{\vee}/G], w'). 
		}
	\end{align}
	Together with Proposition~\ref{prop:rest:an}, 
	we have the commutative diagram
	\begin{align}\label{dia:MF:W}
		\xymatrix{
			\wW_{\delta}^{\intt/\bS}([\fU/G])
			\ar@<-0.3ex>@{^{(}->}[r]^-{\Phi} \ar[d]^-{\sim}_-{\bff_{\ast}} & \wW^l_{\delta+K_{\yY}/2}([V^{\vee}/G], w) \ar@<-0.3ex>@{^{(}->}[r]
			\ar@/^20pt/[rr]^-{\mathrm{res}}
			&
			\mathrm{MF}_{\coh}^{\mathbb{C}^{\ast}}([V^{\vee}/G], w) \ar[r] \ar[d]^-{\sim}_-{f_{\ast}g^{\ast}} &
			\mathrm{MF}_{\coh}^{\mathbb{C}^{\ast}}([(V^{\vee})^{l\sss}/G], w) \ar@{.>}[d]_-{\Theta}\\
			\wW_{\delta}^{\intt/\bS'}([\fU'/G])
			\ar@<-0.3ex>@{^{(}->}[r]^-{\Phi'} & \wW^l_{\delta +K_{\yY'}/2}([{V'}^{\vee}/G], w') \ar@<-0.3ex>@{^{(}->}[r] \ar@/_20pt/[rr]_{\mathrm{res}'}&
			\mathrm{MF}_{\coh}^{\mathbb{C}^{\ast}}([{V'}^{\vee}/G], w') \ar[r] &
			\mathrm{MF}_{\coh}^{\mathbb{C}^{\ast}}([({V'}^{\vee})^{l\sss}/G], w'). 
		}
	\end{align}
	We show that there is an equivalence $\Theta$ in the 
	dotted arrow which makes the above diagram commutative. 
	We set
	\begin{align*}
		Z^{l\us}=\Crit(w)\setminus \Crit(w)^{l\sss}, \
		{Z'}^{l\us}=\Crit(w') \setminus \Crit(w')^{l\sss}. 
	\end{align*}
	Since we have $\mathrm{Crit}(w)^{l\sss}=(V^{\vee})^{l\sss} \cap \Crit(w)$, 
	we have the open immersion
	\begin{align}\label{res:W}
		(V^{\vee})^{l\sss} \hookrightarrow 
		V^{\vee} \setminus Z^{l\us}
	\end{align}
	such that we have 
	\begin{align*}
		\Crit(w) \cap (V^{\vee})^{l\sss}
		=\Crit(w) \cap (V^{\vee} \setminus Z^{l\us})
		=\Crit(w)^{l\sss}. 
	\end{align*}
	Since a derived 
	factorization category depends only on an open 
	neighborhood of the critical locus 
	(see (\ref{rest:equiv})), 
	the restriction along the open immersion (\ref{res:W}) 
	gives an equivalence 
	\begin{align}\label{MF:restrict}
		\mathrm{MF}_{\coh}^{\mathbb{C}^{\ast}}([(V^{\vee} \setminus Z^{l\us})/G], w)
		\stackrel{\sim}{\to} \mathrm{MF}_{\coh}^{\mathbb{C}^{\ast}}
		([(V^{\vee})^{l\sss}/G], w). 
	\end{align}
	On the other hand, the equivalence of derived stacks $\bff$ in the diagram 
	(\ref{diagram:compare})
	induces the isomorphism
	\begin{align*}
		[\Crit(w)/G] \stackrel{\cong}{\to} [\Crit(w')/G]
	\end{align*}
	which sends 
	a conical closed substack
	$\zZ^{l\us} \cneq [Z^{l\us}/G]$
	to ${\zZ'}^{l\us} \cneq [{Z'}^{l\us}/G]$. 
	Therefore the equivalence 
	$\bff_{\ast} \colon \Dbc([\fU/G]) \stackrel{\sim}{\to}
	\Dbc([\fU'/G])$ restricts to the equivalence
	\begin{align*}
		\bff_{\ast} \colon \cC_{\zZ^{l\us}} \stackrel{\sim}{\to}
		\cC_{{\zZ'}^{l\us}}.
	\end{align*}
	By Proposition~\ref{prop:koszul:Z}
	and the commutative diagram (\ref{commute:KZ}), 
	the equivalence $f_{\ast}g^{\ast}$ in the diagram (\ref{dia:MF:W})
	restricts to the equivalence
	\begin{align*}
		f_{\ast}g^{\ast} \colon 
		\mathrm{MF}_{\coh}^{\mathbb{C}^{\ast}}([V^{\vee}/G], w)_{\zZ^{l\us}}
		\stackrel{\sim}{\to}\mathrm{MF}_{\coh}^{\mathbb{C}^{\ast}}([{V'}^{\vee}/G], w')_{{\zZ'}^{l\us}}.
	\end{align*}
	By taking the Verdier quotients as in (\ref{quot:MF})
	and using (\ref{MF:restrict}), 
	we obtain the desired equivalence $\Theta$. 
	
	Note that the functors $\mathrm{res}$, $\mathrm{res}'$ in the diagram (\ref{dia:MF:W})
	are equivalences 
	by Theorem~\ref{thm:window}. 
	Using the equivalence $\Theta$, 
	we have the commutative diagram
	\begin{align*}
		\xymatrix{
			\wW_{\delta}^{\intt/\bS}([\fU/G]) \ar@<-0.3ex>@{^{(}->}[r]^-{\Phi} \ar[d]_-{\bff_{\ast}}^-{\sim}
			& \wW^l_{\delta+K_{\yY}/2}([V^{\vee}/G], w) \ar[d]^-{\mathrm{res'}^{-1} \circ \Theta \circ \mathrm{res}}_-{\sim} \\
			\wW_{\delta}^{\intt/\bS'}([\fU'/G]) \ar@<-0.3ex>@{^{(}->}[r]^-{\Phi'} & \wW^l_{\delta+K_{\yY'}/2}([{V'}^{\vee}/G], w'). 
		}
	\end{align*}
	The lemma follows from the above commutative diagram. 
\end{proof}

We also have the following formal fiber version 
of Lemma~\ref{lem:window}
(see (\ref{diagram:UAK}) for the notation of formal fibers):
\begin{lem}\label{lem:window2}
	In the setting of Lemma~\ref{lem:window1.5}, 
	assume that $\widehat{\bff}(x)=\id$ in the diagram (\ref{commute:BG}). 
	We 
	take $l, \delta \in M_{\mathbb{R}}^W$ such that $\delta$ is $l$-generic. 
	Then the functor 
	$\widehat{\Phi}_y$ in (\ref{equiv:K}) restricts to the equivalence  
	\begin{align*}
		\widehat{\Phi}_y \colon 
		\wW^{\intt/\bS_x}_{\delta}([\widehat{\fU}_y/G]) \stackrel{\sim}{\to}
		\wW^l_{\delta+K_{\yY}/2}([\widehat{V}_y^{\vee}/G], \widehat{w}_y)
	\end{align*}
	if and only if $\widehat{\Phi}_{y'}'$ restricts to the 
	equivalence
	\begin{align*}
		\widehat{\Phi}_{y'}' \colon 
		\wW^{\intt/\bS_{x'}'}_{\delta}([\widehat{\fU}_{y'}'/G]) \stackrel{\sim}{\to}
		\wW^l_{\delta +K_{\yY'}/2}([\widehat{V}_{y'}^{'\vee}/G], \widehat{w}_{y'}').\end{align*}
\end{lem}
\begin{proof}
	The proof of Proposition~\ref{prop:rest:an} shows that the functor 
	$\widehat{\Phi}_y$
	restricts to the fully-faithful 
	functor 
	\begin{align*}
		\widehat{\Phi}_y \colon 
	\wW^{\intt/\bS_x}_{\delta}([\widehat{\fU}_y/G]) \hookrightarrow
	\wW^l_{\delta+K_{\yY}/2}([\widehat{V}_y^{\vee}/G], \widehat{w}_y),
	\end{align*}
and the same also applies to $\widehat{\Phi}_{y'}$. 
Therefore the 
argument of Lemma~\ref{lem:window} applies verbatim, 
using Lemma~\ref{lem:window1.5} instead of Proposition~\ref{prop:compare}, 
and also using Lemma~\ref{lem:tuple:formal} and  
Lemma~\ref{lem:tuple:formal2}, 
instead of Lemma~\ref{lem:4.11} and Lemma~\ref{lem:closed}
for the reduction to the case of closed immersion. 
\end{proof}

Using the above argument for formal fibers, 
we prove the following proposition for affine case:
\begin{prop}\label{prop:rest:an2}
	In the setting of Proposition~\ref{prop:rest:an},
	suppose that $[\fU/G]$ satisfies formal neighborhood 
	theorem, $l$ is $\bS$-generic and compatible with $\bS$. 
	Then the functor (\ref{id:locglob}) 
	is an equivalence. 
	\end{prop}
\begin{proof}
	We use the notation in the proof of Proposition~\ref{prop:rest:an}.
	By Lemma~\ref{lem:rest:an}, it is enough to prove that 
	the functor $\widehat{\Phi}_y$ in (\ref{Phiy:rest:form}) 
	is an equivalence 
		\begin{align}\label{Phiy:equiv}
		\widehat{\Phi}_y \colon 
		\wW_{\delta}^{\intt/\bS_x}([\widehat{\fU}_y/G]) \stackrel{\sim}{\to}
		\wW^l_{\delta+K_{\yY}/2}([\widehat{V}_y^{\vee}/G], \widehat{w}_y). 
	\end{align}
		We set  
	\begin{align*}
		Y' \cneq \hH^0(\mathbb{T}_{[\fU/G]}|_{x}), \ 
		V' \cneq \hH^1(\mathbb{T}_{[\fU/G]}|_{x}) \oplus \hH^0(\mathbb{T}_{[\fU/G]}|_{x}).
	\end{align*} 
	Note that $Y'$ is a $G_x$-representation, 
	and $V'$ is regarded as a
	$G_x$-equivariant vector bundle on $Y'$
	by the second projection $V' \to Y'$. 
	As we assume that $[\fU/G]$ 
	satisfies the formal neighborhood theorem, 
	Lemma~\ref{lem:analytic} implies the following: 
	there exists a diagram of formal fibers at $0 \in Y' \ssslash G_x$
	\begin{align}\label{dia:VUY}
		\xymatrix{
			& [\widehat{V}_0'/G_x]\ar[d] \\
			[\widehat{\fU}_0'/G_x] \ar@<-0.3ex>@{^{(}->}[r]^-{} &
			[\widehat{Y}_0'/G_x] \ar@/_10pt/[u]_{\widehat{s}_0'}
		}
	\end{align}
	where $[\widehat{\fU}_0'/G_x]$ is the derived zero locus of the section 
	$\widehat{s}_0'$,
	such that there is an equivalence 
	\begin{align}\label{equiv:analytic}
		[\widehat{\fU}_y/G] \stackrel{\sim}{\to} [\widehat{\fU}_0'/G_x]
	\end{align}
	where $x \in \uU$ corresponds to 
	$0 \in Y'$. 
	By Lemma~\ref{lem:Gchar}, 
	the $\mathbb{R}$-line bundles
	$l|_{\widehat{\uU}_y}, \delta|_{\widehat{\uU}_y} \in \Pic([\widehat{\uU}_y/G])_{\mathbb{R}}$
	correspond to 
	$l_x, \delta_x \in \Pic(BG_x)_{\mathbb{R}}$
	under the equivalence (\ref{equiv:analytic}).
	Note that by the genericity assumption on $l$, $\delta$, 
	the element $l_x$ is $\bS_x$-generic
	and $\delta_x$ is $l_x$-generic. 
	
	Let $\widehat{\Phi}_0'$ be the Koszul duality equivalence 
	in Theorem~\ref{thm:knoer} applied for the diagram (\ref{dia:VUY})
	\begin{align*}
		\widehat{\Phi}_0' \colon \Dbc([\widehat{\fU}_0'/G_x]) \stackrel{\sim}{\to}
		\mathrm{MF}_{\coh}^{\mathbb{C}^{\ast}}([\widehat{V}_0^{'\vee}/G_x], \widehat{w}_0').
	\end{align*} 
	Here $\widehat{w}_0'$ is defined from 
	$\widehat{s}_0'$ in the diagram (\ref{dia:VUY}) as in (\ref{def:w}). 
	By Proposition~\ref{prop:int2}, 
	the genericity condition on $l_x$ and $\delta_x$, and the assumption that 
	$l_x$ is compatible with $\bS_x$, 
	the equivalence $\widehat{\Phi}_0'$ restricts to the equivalence
	\begin{align}\label{Phi0'}
		\widehat{\Phi}_0' \colon 
		\wW_{\delta_x}^{\intt/\bS_x}([\widehat{\fU}_0'/G_x])\stackrel{\sim}{\to}
		\wW^{l_x}_{\delta_x +K_{\yY'}/2}
		([\widehat{V}_0^{'\vee}/G_x], \widehat{w}_0'). 
	\end{align}
	On the other hand, 
	by the equivalence (\ref{equiv:analytic}) and the diagram (\ref{diagram:analytic}), 
	we have an equivalence 
	\begin{align*}
		[\widehat{\fU}_0'/G_x] \sim [\widehat{\mathfrak{T}}_0/G_x],
	\end{align*}
which is $\id \colon BG_x \to BG_x$ on closed points. 
	By Lemma~\ref{lem:window2} and the equivalence (\ref{Phi0'}),
	we conclude that the
	functor (\ref{window:PhiT}) is an equivalence. 
		The
	equivalence (\ref{Phiy:equiv}) now
	follows from the diagram (\ref{diagram:analytic}) and Lemma~\ref{lem:obvious}. 
	\end{proof}

We have the following corollary of the above 
proposition: 

\begin{cor}\label{cor:composition}
	In the situation of Proposition~\ref{prop:rest:an},
	let $\zZ^{l\us} \subset [\Crit(w)/G]$ be 
	the conical closed substack of $l$-unstable points. 
	Then 
	the composition
	\begin{align}\notag
		\wW_{\delta}^{\intt/\bS}([\fU/G]) \hookrightarrow 
		\Dbc([\fU/G]) \twoheadrightarrow 
		\Dbc([\fU/G])/\cC_{\zZ^{l\us}}
	\end{align} 
	is 
	fully-faithful, which is an equivalence if $l$ is 
	$\bS$-generic and compatible 
	with $\bS$. 
\end{cor}
\begin{proof}
	By Proposition~\ref{prop:rest:an}, 
	we have the commutative diagram 
	\begin{align*}
		\xymatrix{
			\wW_{\delta}^{\intt/\bS}([\fU/G]) \inclusion \ar[d]_-{\Phi}
			& \Dbc([\fU/G]) \ar@{>>}[r] \ar[d]_-{\Phi}^-{\sim}
			& \Dbc([\fU/G])/\cC_{\zZ^{l\us}} \ar[d]_-{\Phi}^-{\sim} \\
			\wW^l_{\delta+K_{\yY}/2}([V^{\vee}/G], w) \inclusion & 
			\mathrm{MF}_{\coh}^{\mathbb{C}^{\ast}}([V^{\vee}/G], w])
			\ar@{>>}[r] &
			\mathrm{MF}_{\coh}^{\mathbb{C}^{\ast}}([V^{\vee}/G] \setminus \zZ^{l\us}, 
			w). 
		}
	\end{align*}
	Since have an equivalence (\ref{MF:restrict}), 
	the bottom composition is an equivalence by Theorem~\ref{thm:window}. 
	Therefore the corollary follows from Proposition~\ref{prop:rest:an}. 
\end{proof}

\subsubsection{Proof of window theorem for DT categories}\label{subsec:proofthm}
Finally, we 
give a proof of Theorem~\ref{thm:equivalence}
by taking the limits of  
the results in the previous subsections. 
Let $\fM$ be a quasi-smooth derived stack 
such that $\mM=t_0(\fM)$ admits a good moduli 
space $\mM \to M$
and satisfies the formal neighborhood theorem. 
Let $\bS$ be a symmetric structure of $\fM$ as in (\ref{sym:M}). 
Note that for an \'{e}tale morphism 
$\iota_{\fM} \colon \fM_U \to \fM$ in the diagram (\ref{Cartesian:U}), 
we have the induced symmetric structure $\bS_U$ 
since $\iota_{\fM}$ induces the equivalences of tangent complexes 
at each closed points. 
Using Proposition~\ref{prop:extend}, the definition of 
intrinsic window subcategory 
in $\Dbc(\fM)$ is defined as a globalization of Definition~\ref{defi:window}: 
\begin{defi}\label{defi:intwin}
	For $\delta \in \Pic(\mM)_{\mathbb{R}}$, we 
	define the triangulated subcategory
	\begin{align*}
		\wW_{\delta}^{\intt/\bS}(\fM) \subset \Dbc(\fM)
	\end{align*}
	to be consisting of objects $\eE \in \Dbc(\fM)$ 
	such that for any 
	\'{e}tale morphism $\iota \colon U \to M$ from an affine 
	scheme $U$ which fits into a 
	diagram (\ref{Cartesian:U}),
	we have $\iota_{\fM}^{\ast}\eE \in \wW_{\iota_{\fM}^{\ast}\delta}^{\intt/\bS_U}
	(\fM_U)$. 
\end{defi}

In the following, we show that 
the above intrinsic window subcategory 
gives a desired subcategory in Theorem~\ref{thm:equivalence}.
\begin{thm}\label{thm:equivalence2}
	Let us take $l, \delta \in \Pic(\mM)_{\mathbb{R}}$ 
	such that $l$ is $\delta$ is $l$-generic. 
	Then 
	the composition
	\begin{align}\label{compose:Thetal}
		\Theta_l \colon 
		\wW_{\delta}^{\intt/\bS}(\fM) \hookrightarrow 
		\Dbc(\fM) \to
		\dDT^{\C}(\nN^{l\sss})
		\end{align}
	is fully-faithful, which is an equivalence 
	if $l$ is $\bS$-generic and compatible with $\bS$. 
\end{thm}
\begin{proof}
	Let
	$\dD_{{\rm{\acute{e}t}}/M}$
	be the category of \'{e}tale 
	morphisms $\iota \colon U \to M$ 
	as in Subsection~\ref{subsubsec:gmoduli}, and 
	$\dD'_{{\rm{\acute{e}t}}/M} \subset \dD_{{\rm{\acute{e}t}}/M}$
	the subcategory 
	satisfying 
	the condition in Theorem~\ref{thm:AHR}. 
	For each $(\iota \colon U \to M) \in \dD'_{{\rm{\acute{e}t}}/M}$, 
	we have the induced 
	\'{e}tale morphism 
	$\iota_{\fM} \colon \fM_U \to \fM$ in
	the diagram (\ref{Cartesian:U}). 
	Moreover for any morphism $\rho \colon U' \to U$
	in $\dD'_{{\rm{\acute{e}t}}/M}$, 
	we have the induced \'{e}tale morphism 
	$\rho_{\fM} \colon \fM_{U'} \to \fM_{U}$
	as in the diagram (\ref{dia:rho}). 
	Since 
	we have the \'{e}tale cover
	$\coprod_{(U\stackrel{\iota}{\to}M) \in \dD'_{{\rm{\acute{e}t}}/M}}
	\fM_U \stackrel{\iota_{\fM}}{\to} \fM$
	of $\fM$, 
	we have an equivalence 
	\begin{align}\label{DbM}
		\Dbc(\fM) \stackrel{\sim}{\to} 
		\lim_{(U\stackrel{\iota}{\to} M)\in \dD'_{{\rm{\acute{e}t}}/M}}
		\Dbc(\fM_U). 
	\end{align}
	By Lemma~\ref{lem:pull-back} below, 
	for a morphism $\rho \colon U' \to U$ 
	the pull-back
	$\rho_{\fM}^{\ast}$ restricts to the functor
	\begin{align*}
		\rho_{\fM}^{\ast} \colon \wW^{\intt/\bS_U}_{\iota_{\fM}^{\ast}\delta}(\fM_U)
		\to \wW^{\intt/\bS_{U'}}_{\iota_{\fM}^{'\ast}\delta}
		(\fM_{U'}).
		\end{align*}
	Therefore from (\ref{DbM}) and the definition of 
	$\wW_{\delta}^{\intt/\bS}(\fM)$, 
	the equivalence (\ref{DbM}) restricts 
	to the equivalence 
	\begin{align}\label{equiv:Theta}
		\wW_{\delta}^{\intt/\bS}(\fM) \stackrel{\sim}{\to}
		\lim_{(U\stackrel{\iota}{\to} M)\in \dD'_{{\rm{\acute{e}t}}/M}} \wW_{\iota_{\fM}^{\ast}\delta}^{\intt/\bS_U}(\fM_U). 
	\end{align}
	On the other hand
	the assumption on $\fM$ 
	together with the genericity of $l$, $\delta$ imply 
	that, 
	for each $(U\stackrel{\iota}{\to} M) \in \dD'_{{\rm{\acute{e}t}}/M}$, 
	the derived stack 
	$\fM_U$
	together with $\iota_{\fM}^{\ast}l$, $\iota_{\fM}^{\ast}\delta$
	satisfy the assumption of Proposition~\ref{prop:rest:an}.
	Therefore  
	by Corollary~\ref{cor:composition},  
	the composition 
	\begin{align}\label{compose:WM}
		\wW_{\iota_{\fM}^{\ast}\delta}^{\intt/\bS_U}(\fM_U)
		\hookrightarrow \Dbc(\fM_U) \twoheadrightarrow 
		\Dbc(\fM_U)/\cC_{\iota_{\fM}^{\ast}\zZ_{l\us}}
	\end{align}
	is fully-faithful, and an equivalence if $l$ is 
	compatible with $\bS$.  
	By taking the limit for 
	$(U \stackrel{\iota}{\to} M) \in \dD'_{{\rm{\acute{e}t}}/M}$
	and using the equivalences (\ref{DbM}), (\ref{equiv:Theta}), 
	the composition 
	\begin{align*}
		\wW_{\delta}^{\intt/\bS}(\fM) \hookrightarrow 
		\Dbc(\fM) \to 
		\lim_{(U \stackrel{\iota}{\to} M)\in \dD'_{{\rm{\acute{e}t}}/M}}
		\left(\Dbc(\fM_U)/\cC_{\iota_{\fM}^{\ast}\zZ_{l\us}}\right)
	\end{align*}
	is full-faithful, and an equivalence if $l$ is compatible with $\bS$. 
	Now the above composition functor factors as 
	\begin{align*}
	\wW_{\delta}^{\intt/\bS}(\fM) \to \dDT^{\C}(\nN^{l\sss}) \to 
	\lim_{(U \stackrel{\iota}{\to} M)\in \dD'_{{\rm{\acute{e}t}}/M}}
	\left(\Dbc(\fM_U)/\cC_{\iota_{\fM}^{\ast}\zZ_{l\us}}\right). 
		\end{align*}
	Here the first arrow is the composition functor (\ref{compose:Thetal}). 
	The right arrow is fully-faithful 
	by Lemma~\ref{lem:inter}, and the above composition is fully-faithful, 
	so the first arrow is also fully-faithful. 
	Moreover if $l$ is compatible with $\bS$, then 
	the first arrow is also an equivalence since each arrow is 
	fully-faithful and the composition is an equivalence. 
	\end{proof}

We have used the following lemma:  
\begin{lem}\label{lem:pull-back}
	Let $[\fU/G]$, $[\fU'/G']$ be derived stacks of as in Definition~\ref{def:tuple}
	for reductive $G$, $G'$. 
	Suppose that we have a commutative diagram 
	\begin{align}\label{dia:etale}
		\xymatrix{
			\uU' \ssslash G' \ar[d]^-{f} \diasquare
			& [\uU'/G'] \inclusion \ar[l] \ar[d]^-{f} \diasquare
			& [\fU'/G'] \ar[d]^-{\bff} \\
			\uU\ssslash G & [\uU/G] \inclusion \ar[l] & [\fU/G]
		}
	\end{align}
	where each square is a Cartesian and the vertical arrows 
	are \'{e}tale. 
	Then for $\delta \in \Pic([\uU/G])_{\mathbb{R}}$
	and $\delta'=f^{\ast}\delta \in \Pic([\uU'/G'])_{\mathbb{R}}$, 
	the functor $\bff^{\ast} \colon \Dbc([\fU/G]) \to \Dbc([\fU'/G'])$
	restricts to the functor
	\begin{align*}
		\bff^{\ast} \colon \wW_{\delta}^{\intt/\bS}([\fU/G]) \to 
		\wW_{\delta'}^{\intt/\bS'}([\fU'/G']). 
	\end{align*}
	Here $\bS'$ is induced from $\bS$ by the 
	\'{e}tale morphism $f$. 
\end{lem}
\begin{proof}
	For a closed point $y' \in \uU'\ssslash G'$
	and $y=f(y') \in \uU\ssslash G$, 
	the diagram (\ref{dia:etale}) induces 
	an equivalence 
	\begin{align}\notag
		\widehat{\bff}_y \colon 
		[\widehat{\fU}'_{y'}/G'] \stackrel{\sim}{\to} [\widehat{\fU}_y/G].
	\end{align}
	Here we have used the notation in (\ref{diagram:UAK}). 
	Let $x \in \uU$, $x' \in \uU'$ be closed points
	in the closed orbits of the fibers of 
	$\uU \to \uU \ssslash G$, 
	$\uU' \to \uU'\ssslash G'$
	at $y$, $y'$, respectively. 
	Let $x\in Z \subset Y$, $x' \in Z' \subset Y'$ be \'{e}tale 
	slices 
	as in the proof of 
	Proposition~\ref{prop:rest:an}.
	Then by the diagram (\ref{diagram:analytic}), 
	we have the diagram 
	\begin{align}\notag
		\xymatrix{
			[\widehat{\fU}'_{y'}/G'] \ar@<-0.3ex>@{^{(}->}[r]^-{j'} 
			\ar[d]^-{\sim}_-{\widehat{\bff}_y} & 
			[\widehat{(T_{x'} Z')}_{0}/G'_{x'}]
			\ar@/^20pt/[r]^-{\widehat{t}_0'} & 
			[V'|_{x'} \oplus \widehat{(T_{x'}Z')}_{0}/G'_{x'}] \ar[l] \\
			[\widehat{\fU}_y/G] \ar@<-0.3ex>@{^{(}->}[r]^-{j} & 
			[\widehat{(T_{x} Z)}_{0}/G_{x}]
			\ar@/_20pt/[r]_-{\widehat{t}_0} & 
			[(V|_{x} \oplus \widehat{(T_x Z)}_0)/G_x] \ar[l]
		}
	\end{align}
	such that 
	$[\widehat{\fU}'_{y'}/G']$, 
	$[\widehat{\fU}_y/G]$ are equivalent to derived zero loci 
	of $\widehat{t}_0'$, $\widehat{t}_0$ respectively. 
	
	Let us take an object $\eE \in \wW_{\delta}^{\intt/\bS}([\fU/G])$. 
	Then we have 
	$\widehat{\eE}_y \in \wW_{\delta}^{\intt/\bS_x}([\widehat{\fU}_y/G])$
	by Lemma~\ref{lem:rest:an}. 
	Since $f$ induces the isomorphism 
	$G'_{x'} \stackrel{\cong}{\to} G_x$, 
	by Lemma~\ref{lem:obvious} and 
Lemma~\ref{lem:window1.5}
	we 
	conclude 
	that 
	\begin{align*}
		\widehat{(\bff^{\ast}\eE)}_{y'}=
		\widehat{\bff}_y^{\ast}(\widehat{\eE}_y)
		\in \wW_{\delta'}^{\intt/\bS'_{x'}}([\widehat{\fU}_{y'}'/G']).
	\end{align*}
	Since this holds for any $y' \in \uU'\ssslash G'$, 
	we have 
	$\bff^{\ast}\eE \in \wW_{\delta'}^{\intt/\bS'}([\fU'/G'])$
	by Lemma~\ref{lem:rest:an}. 
\end{proof}

\subsection{Application to equivalences of DT categories for 
one dimensional stable sheaves}\label{subsec:app1}
In this section, we use Corollary~\ref{cor:equivalence} to give
an application to Conjecture~\ref{conj1}. 
\subsubsection{Derived moduli stacks of one dimensional semistable sheaves on surfaces}\label{subsec:stackS}
Let $S$ be a smooth projective surface. 
As in Section~\ref{subsec:catDTsurf}, 
we consider the derived stack 
$\fM_S$ of coherent sheaves on $S$. 
We take a stability condition 
$\sigma=B+iH \in A(S)_{\mathbb{C}}$
and an element $v \in N_{\le 1}(S)$. 
As in (\ref{dstack:surface}), we have the derived open substack
\begin{align*}
	\fM_S^{\sigma}(v) \subset \fM_S
	\end{align*}
consisting of one dimensional 
$\sigma$-semistable sheaves with numerical class 
$v$. 
\begin{lem}\label{lem:stackS}
	Both of derived stacks $\fM_S^{\sigma}(v)$ and $\fM_S^{\sigma}(v)^{\C\rig}$
	are symmetric and satisfy formal neighborhood theorem. 
\end{lem}
\begin{proof}
	A closed point $x \in \mM_S^{\sigma}(v)$ corresponds to a 
	$\sigma$-polystable sheaf $F$ on $S$, which is of the form
	\begin{align}\label{pstable:F}
		F=\bigoplus_{i=1}^m V_i \otimes F_i
	\end{align}
	where each $F_i$ is a $\sigma$-stable sheaf on $S$, 
	$V_i$ is a finite dimensional vector space 
	such that $F_i$ is not isomorphic to $F_j$ for $i\neq j$, 
	and $\mu_{\sigma}(F_i)=\mu_{\sigma}(F_j)$ for all $i$, $j$, 
	where $\mu_{\sigma}$ is defined in (\ref{def:slope}). 
	By the description of the cotangent complex (\ref{perf:obs1}), 
	we have 
	\begin{align}\notag
		\hH^0(\mathbb{T}_{\fM_S^{\sigma}(v)}|_{x}) \oplus \hH^1(\mathbb{T}_{\fM_S^{\sigma}(v)}|_{x})^{\vee} 
		&= \Ext_S^1(F, F) \oplus \Ext_S^2(F, F)^{\vee} \\
		\label{rep:tan}
		&=\bigoplus_{a, b} \Hom(V_a, V_b) \otimes (\Ext_S^1(F_a, F_b) \oplus 
		\Ext_S^2(F_b, F_a)^{\vee}). 
	\end{align}
	The automorphism group of $\mM_S^{\sigma}(v)$ at $x$ is given by  
	\begin{align*}
		\Aut(x)=\Aut(F)=\prod_{i=1}^m \GL(V_i)
	\end{align*}
	and its acts on (\ref{rep:tan}) by the conjugation. 
	The dual representation of (\ref{rep:tan}) is given by 
	\begin{align*}
		\bigoplus_{a, b} \Hom(V_a, V_b) \otimes (\Ext_S^1(F_b, F_a)^{\vee} \oplus 
		\Ext_S^2(F_a, F_b)). 
	\end{align*}
	Therefore in order to show that (\ref{rep:tan}) is a symmetric representation
	of $\Aut(x)$, 
	we need to show that 
	\begin{align}\label{id:ext12}
		\ext_S^1(F_a, F_b)+\ext_S^2(F_b, F_a)=\ext_S^1(F_b, F_a)+\ext_S^2(F_a, F_b). 
	\end{align}
	By the Riemann-Roch theorem and the stability for $F_a$, 
	by writing $[F_a]=(\beta_a, n_a)$ we
	have
	\begin{align*}
		\ext_S^1(F_a, F_b)-\ext_S^2(F_a, F_b)=\delta_{ab}+\beta_a \cdot \beta_b
	\end{align*}
	which is symmetric in $a$ and $b$. 
	Therefore (\ref{id:ext12}) holds, and 
	$\fM_S^{\sigma}(v)$ is symmetric. 
	Similarly (\ref{rep:tan}) is a symmetric 
	$\Aut(x)/\mathbb{C}^{\ast}$-representation, 
	so $\fM_S^{\sigma}(v)^{\C\rig}$ is symmetric. 
	
	The fact that 
	the derived stack $\fM_S^{\sigma}(v)$ satisfies the formal 
	neighborhood theorem follows from Theorem~\ref{thm:formal}. 
	Here note that Theorem~\ref{thm:formal} is formulated for 
	Gieseker stability (i.e. $\sigma=iH$ for an ample divisor $H$), 
	but the same argument applies for 
	one dimensional $\mu_{\sigma}$-semistable sheaves 
	by the existence of the good moduli space for $\mM_S^{\sigma}(v)$
	(see~\cite[Lemma~7.4]{MR3811778}). 
	Then the derived stack $\fM_S^{\sigma}(v)^{\C\rig}$ also satisfies
	the formal neighborhood theorem 
	by taking the $\C$-rigidifications of 
	top isomorphism in the diagram (\ref{dia:fthm}). 
\end{proof}

\begin{rmk}\label{rmk:squiver}
	Let 
	$X=\mathrm{Tot}_S(\omega_S)$
	and
	$i \colon S \hookrightarrow X$ the zero section. 
	The $\Aut(x)$-representation 
	(\ref{rep:tan}) is isomorphic to 
	the conjugate $\Aut(x)$-action on 
	\begin{align*}
		\Ext_X^1(i_{\ast}F, i_{\ast}F)=
		\bigoplus_{e \in E(Q_{E_{\bullet}})}\Hom(V_{s(e)}, V_{t(e)})
	\end{align*}
	where $Q_{E_{\bullet}}$ is the Ext-quiver associated with the 
	collection $E_{\bullet}=(i_{\ast}F_1, \ldots, i_{\ast}F_k)$
	(see Subsection~\ref{subsec:Equiver}). 
	This is because of the isomorphisms
	\begin{align*}
		\Ext_X^1(i_{\ast}F_a, i_{\ast}F_b) &\cong 
		\Ext_S^1(i^{\ast}i_{\ast}F_a, F_b) \\
		&\cong \Ext_S^1(F_a, F_b) \oplus \Hom(F_a, F_b \otimes \omega_S) \\
		&\cong \Ext_S^1(F_a, F_b) \oplus \Ext_S^2(F_b, F_a)^{\vee}. 
	\end{align*}
\end{rmk}

\subsubsection{Line bundles on moduli stacks}


We now define some 
line bundle on $\mM_S(v)$ associated with 
an integral class $\sigma \in A(S)_{\mathbb{C}}$. 
\begin{defi}\label{def:line}
	For an integral class
	$\sigma=B+iH \in A(S)_{\mathbb{C}}$
	such that $H$ is an effective class, 
	we define 
	$l(\sigma) \in \Pic(\mM_S(v))$ by 
	\begin{align*}
		l(\sigma)=(\det \dR p_{\mM\ast} 
		(\fF \boxtimes \oO_S(-B)))^{-\beta \cdot H} \otimes 
		(\det \dR p_{\mM \ast}(\fF \boxtimes \oO_{H}))^{n-B \cdot \beta}. 
	\end{align*}
	Here $\fF$ is a universal sheaf (\ref{F:universal}). 
	Its pull-back to $\mM_X(v)$, 
	and also its restriction to an open 
	substack of $\mM_{X}(v)$
	are also denoted by $l(\sigma)$. 
\end{defi}
The line bundle in Definition~\ref{def:line} descends 
to the line bundle in the $\C$-rigidification: 
\begin{lem}\label{lem:descend}
	The line bundles $l(\sigma)$ on $\mM_{S}(v)$, 
	$\mM_{X}(v)$ descend
	to line bundles on $\mM_{S}^{\C\rig}(v)$, 
	$\mM_{X}^{\C\rig}(v)$.  
\end{lem}
\begin{proof}
	At each point $[F] \in \mM_{S}(v)$, the 
	inertial $\C$-weight of $l(\sigma)|_{[F]}$ is 
	\begin{align*}
		-(\beta \cdot H) \chi(F \otimes \oO_S(-B))
		+(n-B \cdot \beta)\chi(F \otimes \oO_H) &=-(\beta \cdot H)(n-B \cdot \beta)+(n-B \cdot \beta)(\beta \cdot H) \\
		&=0. 
	\end{align*}
	Therefore $l(\sigma)$ descends to $\mM^{\C\rig}_{S}(v)$. 
	The case for $\mM_{X}(v)$ follows from the same argument. 
\end{proof}

Suppose that $\sigma \in A(S)_{\mathbb{C}}$ lies on a 
wall and take $\sigma_{\pm} =B_{\pm}+iH_{\pm}
\in A(S)_{\mathbb{C}}$ which 
lie on its adjacent chambers. 
Note that we have open immersions 
\begin{align*}
	\mM_{X}^{\sigma_{\pm}}(v) \subset \mM_{X}^{\sigma}(v).
\end{align*}
Since each chamber contains dense rational points, 
by taking small deformations of $\sigma_{\pm}$ and rescaling 
we may assume 
that $B_{\pm}, H_{\pm}$ are integral 
and $H_{\pm}$ are effective 
without changing $\mM_{X}^{\sigma_{\pm}}(v)$. 
Then
by Definition~\ref{def:line}, 
we have the line bundles $l(\sigma_{\pm})$ on $\mM_S^{\sigma}(v)$, 
$\mM_{X}^{\sigma}(v)$. 
By Lemma~\ref{lem:descend}, they descend to 
line bundles on $\mM_{S}^{\sigma}(v)^{\C\rig}$, 
$\mM_{X}^{\sigma}(v)^{\C\rig}$, which 
we also denote by $l(\sigma_{\pm})$. 

The following lemma shows that 
the open substacks 
$\mM_{X}^{\sigma_{\pm}}(v) \subset \mM_{X}^{\sigma}(v)$
coincide with $l(\sigma_{\pm})$-semistable loci. 
\begin{lem}\label{lem:chamber}
	We have the identity of open substacks in $\mM_{X}^{\sigma}(v)$, 
	\begin{align}\label{id:Mstable}
		\mM_{X}^{\sigma}(v)^{l(\sigma_{\pm}) \sss}=
		\mM_{X}^{\sigma_{\pm}}(v). 
	\end{align}
\end{lem}
\begin{proof}
	It is enough to prove the identity (\ref{id:Mstable}) on each fiber of the 
	good moduli space morphism $\pi_{\mM_X} \colon \mM_{X}^{\sigma}(v) \to M_{X}^{\sigma}(v)$. 
	Let $y \in M_{X}^{\sigma}(v)$ corresponds to a $\sigma$-polystable sheaf on 
	$X$ of the form
	\begin{align}\label{pstable:X}
		E=\bigoplus_{i=1}^m V_i \otimes E_i. 
	\end{align}
	Here each $V_i$ is a finite dimensional vector space and
	$\{E_1, \ldots, E_m\}$ are mutually non-isomorphic $\sigma$-stable sheaves. 
	Let $Q_{E_{\bullet}}$ be the Ext-quiver
	associated with the collection $(E_1, \ldots, E_m)$
	(see Subsection~\ref{subsec:Equiver}).  
	Then the fiber of $\pi_{\mM_X}$ at $y$ is 
	the closed substack of 
	the nilpotent $Q_{E^{\bullet}}$-representations 
	with dimension vector $(\dim V_i)_{1\le i\le m}$
	(see Subsection~\ref{subsec:Equiver})
	\begin{align}\label{qmoduli}
		\left[\left\{ \oplus_{\begin{subarray}{c}
				(i\to j) \in Q_{E_{\bullet}}
		\end{subarray}}
		\Hom(V_i, V_j) \right\}^{\rm{nil} }
		/G \right]. 
	\end{align}
	Here $G=\prod_{i=1}^m \GL(V_i)$
	and the subscript $`\rm{nil}'$ means nilpotent 
	$Q_{E_{\bullet}}$-representations. 
	Let us write $[\pi_{\ast}E_i]=(\beta_i, n_i)$
	in $N_{\le 1}(S)$. 
	We define the following 
	group homomorphisms $W_{\pm}$
	\begin{align*}
		W_{\pm} \colon K(Q_{E_{\bullet}}) \stackrel{\mathbf{dim}}{\to} 
		\bigoplus_{i=1}^m \mathbb{Z} \cdot \mathbf{e}_i \to \mathbb{C}. 
	\end{align*}
	Here the first arrow is taking the dimension vector, 
	and the second arrow is given by
	\begin{align*}
		\mathbf{e}_i \mapsto -n_i+B_{\pm} \cdot \beta_i +(H_{\pm} \cdot \beta_i)\sqrt{-1}
		\in \mathbb{C}. 
	\end{align*}
	Then $W_{\pm}$ determine Bridgeland stability conditions~\cite{MR2373143}
	on the abelian category of 
	finite dimensional $Q_{E_{\bullet}}$-representations:  
	a finite dimensional $Q_{E_{\bullet}}$-representation $R$
	is $W_{\pm}$-(semi)stable if for any non-zero 
	subrepresentation $R' \subsetneq R$, we have 
	\begin{align*}
		\arg W_{\pm}(R')<(\le) \arg W_{\pm}(R)
	\end{align*}
	in $(0, \pi]$. 
	By~\cite[Lemma~7.8]{MR3811778}, 
	the 
	intersection $\pi_{\mM_X}^{-1}(p) \cap \mM_{X}^{\sigma_{\pm}}(v)$
	corresponds to 
	$W_{\pm}$-semistable $Q_{E_{\bullet}}$-representations 
	inside the stack (\ref{qmoduli}). 
	An easy calculation shows that a $Q_{E_{\bullet}}$-representation $R$
	of dimension vector $(\dim V_i)_{1\le i\le m}$
	is $W_{\pm}$-(semi)stable if and only if 
	for any non-zero subrepresentation $R' \subsetneq R$,
	we have 
	\begin{align*}
		\theta_{\pm}(R') \cneq \sum_{i=1}^m 
		\theta_{\pm, i} \cdot r_i>(\ge) 0 =\theta^{\pm}(R).
	\end{align*}
	Here $(r_i)_{1\le i\le m}$ is the dimension vector of $R'$
	and $\theta_{\pm, i} \in \mathbb{Z}$ is given by 
	\begin{align}\label{theta:pm}
		\theta_{\pm, i}=(B_{\pm} \cdot \beta_i-n_i) \cdot (H_{\pm} \cdot 
		\beta)+(n-B_{\pm} \cdot \beta) \cdot (H_{\pm} \cdot \beta_i). 
	\end{align}
	By the relation of $\theta$-stability and GIT stability 
	proved by King~\cite[Theorem~4.1]{Kin},
	the $\theta_{\pm}$-(semi)stable loci in (\ref{qmoduli}) 
	correspond to 
	GIT (semi)stable loci with respect to the characters
	\begin{align}\label{G:character}
		G \to \mathbb{C}^{\ast}, \ (g_i)_{1\le i\le m}
		\mapsto \prod_{i=1}^m \det(g_i)^{\theta_{\pm, i}}. 
	\end{align} 
	
	On the other hand, 
	let $x \in \mM_{X}^{\sigma}(v)$ be the closed point 
	corresponding to the polystable sheaf (\ref{pstable:X}). 
	Then we have $G=\Aut(x)$. 
	In the notation of Definition~\ref{def:stack:gen}, 
	the pull-backs 
	$l(\sigma_{\pm})_x \in \Pic(BG)$ 
	are described as 
	\begin{align*}
		l(\sigma_{\pm})_x
		&=\det \dR \Gamma\left( \bigoplus_{i=1}^m V_i
		\otimes \pi_{\ast}E_i \otimes \oO_S(-B_{\pm})\right)^{-H_{\pm} \cdot \beta}
		\otimes \det \dR \Gamma\left(\bigoplus_{i=1}^m V_i\otimes \pi_{\ast}E_i \otimes
		\oO_{H_{\pm}}\right)^{n-B_{\pm} \cdot \beta} \\
		&=\bigotimes_{i=1}^m 
		(\det V_i)^{-(H_{\pm} \cdot \beta) \cdot \chi(\pi_{\ast}E_i \otimes \oO_S(-B_{\pm}))}
		\otimes \bigotimes_{i=1}^m (\det V_i)^{(H_{\pm}\cdot \beta_i)\cdot (n-B_{\pm} \cdot \beta)} \\
		&=\bigotimes_{i=1}^m \det(V_i)^{\theta_{\pm, i}}. 
	\end{align*}
	Therefore $l(\sigma_{\pm})_x$ are induced 
	by the $G$-characters (\ref{G:character}). 
	Together with using Lemma~\ref{lem:Gchar}, 
	the line bundles 
	$l(\sigma_{\pm})$ on $\mM_{X}^{\sigma}(v)$
	are
	induced by the $G$-characters 
	(\ref{G:character})
	on the fiber $\pi_{\mM_X}^{-1}(p)$, 
	so the identity (\ref{id:Mstable}) holds on 
	$\pi_{\mM_X}^{-1}(p)$. 
\end{proof}

Since the derived stacks $\fM_S^{\sigma}(v)$, 
$\fM_S^{\sigma}(v)^{\C\rig}$ 
are symmetric by Lemma~\ref{lem:stackS}, 
we take their maximal symmetric structures $\bS$
as in Definition~\ref{def:symmetric}. 
The following lemma shows that the line bundles $l(\sigma_{\pm})$ satisfy the 
genericity condition in Definition~\ref{def:stack:gen}. 

\begin{lem}\label{lem:lgeneric}
	The line bundles 
	$l(\sigma_{\pm}) \in \Pic(\mM^{\sigma}_{S}(v)^{\C\rig})$ are $\bS$-generic.
\end{lem}
\begin{proof} 
	Let us take a closed point $x \in \mM^{\sigma}_{S}(v)^{\C\rig}$
	corresponding to a polystable sheaf (\ref{pstable:F}). 
	Then $G' \cneq \Aut(x)=G/\C$ where 
	$\C \subset G=\prod_{i=1}^m \GL(V_i)$ is the diagonal torus. 
	From the proof of Lemma~\ref{lem:chamber}, 
	the element $l(\sigma_{\pm})_x \in \Pic(B\Aut(x))$
	corresponds 
	to a $G'$-character of the form (\ref{G:character}), 
	where
	$\theta_{\pm}$ are given as in (\ref{theta:pm})
	for $[F_i]=(\beta_i, n_i)$. 
	Here note that the $G$-character (\ref{G:character}) descends 
	to the $G'$-character since it restricts to the trivial 
	character on the diagonal torus $\C \subset G$
	(see Lemma~\ref{lem:descend}). 
	By the assumption that $\sigma_{\pm}$
	do not lie on walls, we have  
	$\theta_{\pm}(\vec{v}') \neq 0$
	for any $0<\vec{v}'<\vec{v}$
	where $\vec{v}=(\dim V_i)_{1\le i\le m}$. 
	Then the lemma follows from Remark~\ref{rmk:squiver} and 
	Lemma~\ref{lem:qgeneric}.  
\end{proof}

\subsubsection{Equivalences of DT categories for one dimensional 
stable sheaves}
Applying Corollary~\ref{cor:equivalence}, we have the following 
result which gives an evidence of Conjecture~\ref{conj1}.  
\begin{thm}\label{thm:surface}
	Let $\sigma \in A(S)_{\mathbb{C}}$ lies on a wall 
	with respect to $v \in N_{\le 1}(S)$
	and $\sigma_{\pm} \in A(S)_{\mathbb{C}}$ lie on 
	its adjacent chambers. 
	Moreover assume that 
	\begin{align}\label{assume:sss}
		\mM_{X}^{\sigma}(v) \subset 
		\pi_{\ast}^{-1}(\mM_S^{\sigma}(v)). 
	\end{align}
	Then there exists an equivalence
	\begin{align}\notag
		\dDT^{\C}(M_{X}^{\sigma_{+}}(v))\stackrel{\sim}{\to}
		\dDT^{\C}(M_{X}^{\sigma_{-}}(v)). 
	\end{align}
\end{thm}
\begin{proof}
The condition (\ref{assume:sss}) implies that 
$\mM_X^{\sigma}(v)=\pi_{\ast}^{-1}(\mM_S^{\sigma}(v))$
by Lemma~\ref{lem:open:Sss}. 
Therefore 
by Lemma~\ref{lem:chamber}, 
we have 
\begin{align*}
	M_X^{\sigma_{\pm}}(v)=t_0(\Omega_{\fM_S^{\sigma}(v)^{\C\rig}}[-1])^{l(\sigma_{\pm}) \sss}. 
	\end{align*}
Therefore the theorem is a consequence of Lemma~\ref{lem:stackS}, 
Lemma~\ref{lem:lgeneric} and Corollary~\ref{cor:equivalence}. 
\end{proof}

If we impose some further assumption, the 
result of Theorem~\ref{thm:surface} is described 
in terms of 
derived moduli spaces of stable sheaves on $S$. 
\begin{cor}\label{cor:preserve}
	Under the assumption of Theorem~\ref{thm:surface}, 
	suppose furthermore that 
	the following condition also holds: 
	\begin{align}\label{assume:preserve}
		\mM_{X}^{\sigma_{\pm}}(v) \subset 
		\pi_{\ast}^{-1}(\mM_S^{\sigma_{\pm}}(v)). 
	\end{align}
	Then we have an equivalence 
	\begin{align}\label{equiv:MS}
		\Theta_{\sigma_{+}, \sigma_-} \colon 
		\Dbc(\fM_S^{\sigma_{+}}(v)^{\C\rig})
		\stackrel{\sim}{\to}
		\Dbc(\fM_S^{\sigma_-}(v)^{\C\rig})
	\end{align}
	such that we have the commutative diagram 
	\begin{align}\label{commute:M'3}
		\xymatrix{
			\Dbc(\fM_S^{\sigma_{+}}(v)^{\C\rig}) 
			\ar[r]^-{\Theta_{\sigma_{+}, \sigma^-}}_-{\sim}
			\ar[d] & 
			\Dbc(\fM_S^{\sigma_-}(v)^{\C\rig}) \ar[d] \\
			\Dbc(\fM_S^{\sigma\st}(v)^{\C\rig}) \ar[r]^-{\id}_-{\sim} &
			\Dbc(\fM_S^{\sigma\st}(v)^{\C\rig}).
		}
	\end{align}
	Here the vertical arrows are restriction functors. 
\end{cor}
\begin{proof}
	If the inclusion (\ref{assume:preserve}) holds, 
	then it 
	is identity by Lemma~\ref{lem:open:Sss},
	and
	we have 
	equivalences
	\begin{align}\notag
			\dDT^{\C}(M_{X}^{\sigma_{\pm}}(v))
			\stackrel{\sim}{\to}
		\Dbc(\fM_S^{\sigma_{\pm}}(v)^{\C\rig})	
	\end{align}
	by Lemma~\ref{lem:ss:dt}. 
		Therefore by Theorem~\ref{thm:surface}, we have the 
	equivalence (\ref{equiv:MS}). 
	The commutative diagram (\ref{commute:M'3}) follows from Remark~\ref{rmk:commute}. 
\end{proof}

\subsubsection{Examples and applications}\label{subsec:exam}
In this subsection, we give several examples where 
the conditions (\ref{assume:sss}), (\ref{assume:preserve}) hold so that 
we can apply Corollary~\ref{cor:preserve}. 
We first consider the case of 
$v=(\beta, n) \in N_{\le 1}(S)$ such that $\beta$ is a reduced
class. 
Note that $v$ is primitive for any $n$. 
\begin{lem}\label{lem:reduced}
	For $v=(\beta, n) \in N_{\le 1}(S)$, 
	suppose that $\beta$ is reduced. Then 
	the condition (\ref{assume:sss}) is satisfied for any
	$\sigma\in A(S)_{\mathbb{C}}$. 
\end{lem}
\begin{proof}
	Note that giving a compactly supported coherent sheaf on 
	$X$ is equivalent to giving a pair $(F, \theta)$, 
	where $F \in \Coh(S)$ and $\theta \in \Hom(F, F\otimes \omega_S)$. 
	The push-forward $\pi_{\ast}$ sends such a pair $(F, \theta)$ to $F$. 
	
	Let $E \in \Coh_{\le 1}(X)$ be a $\sigma$-semistable sheaf
	which 
	corresponds to a pair $(F, \theta)$
	such that $[F]=(\beta, n)$. 
	Suppose that $F$ is not $\sigma$-semistable.  
	Then there exists an exact sequence 
	$0 \to F' \to F \to F'' \to 0$ in $\Coh(S)$
	such that $F'$, $F''$ are pure one dimensional 
	sheaves and $\mu_{\sigma}(F')>\mu_{\sigma}(F'')$. 
	We consider the following diagram
	\begin{align}\label{dia:reduced}
		\xymatrix{
			0\ar[r] & F' \ar[r] \ar@{.>}[d]& F \ar[r] \ar[d]^-{\theta} 
			& F'' \ar[r] \ar@{.>}[d] 
			& 0 \\
			0 \ar[r] & F' \otimes \omega_S \ar[r] & F\otimes \omega_S \ar[r] & F'' \otimes 
			\omega_S \ar[r] & 0. 
		}
	\end{align}
	By the assumption, the sheaf $F$ is scheme theoretically 
	supported on a reduced divisor on $S$. 
	Therefore the supports of $F'$ and $F''$ do not have common 
	irreducible components, hence
	we have $\Hom(F', F'' \otimes \omega_S)=0$. 
	Then there exist dotted arrows in the diagram (\ref{dia:reduced})
	which makes the diagram (\ref{dia:reduced}) commutative. 
	However then the diagram (\ref{dia:reduced}) destabilizes $E$, 
	which is a contradiction. 
\end{proof}
By the above lemma, 
in the reduced case
the condition (\ref{assume:sss}) is satisfied 
for any $\sigma\in A(S)_{\mathbb{C}}$.  
Therefore by Corollary~\ref{cor:preserve}, we have the following: 

\begin{cor}\label{cor:reduced}
	Let $S$ be a smooth projective surface, and 
	take $v=(\beta, n) \in N_{\le 1}(S)$ such that $\beta$
	is a reduced class. Then for any 
	$\sigma_{\pm}\in A(S)_{\mathbb{C}}$
	which do not lie walls, we have an equivalence
	\begin{align*}
		\Dbc(\fM_S^{\sigma_{+}}(v)^{\C\rig})
		\stackrel{\sim}{\to}
		\Dbc(\fM_S^{\sigma_-}(v)^{\C\rig}).
	\end{align*}
\end{cor}

We also have the following lemma 
for the conditions (\ref{assume:sss}), (\ref{assume:preserve})
to hold: 
\begin{lem}\label{lem:wS}
	Suppose that 
	$c_1(\omega_S) \in \mathbb{R}_{\le 0} \cdot H$
	for some $H \in A(S)_{\mathbb{R}}$. 
	Then 
	for $\sigma=iH$, 
	both of the conditions (\ref{assume:sss}), 
	(\ref{assume:preserve})
	hold for any primitive $v$. 
\end{lem}
\begin{proof}
	We follow the proof of Lemma~\ref{lem:reduced}.
	As for the condition (\ref{assume:sss}), let 
	$E \in \Coh_{\le 1}(X)$ be a 
	$\sigma$-semistable sheaf
	which 
	corresponds to a pair $(F, \theta)$
	such that $[F]=(\beta, n)$. 
	Suppose that $F$ is not $\sigma$-semistable.  
	Then by taking Harder-Narasimhan filtration in $\sigma$-stability, 
	there exists an exact sequence 
	\begin{align}\label{exact:F}
		0 \to F' \to F \to F'' \to 0
	\end{align}
	in $\Coh(S)$
	such that 
	$\mu_\sigma^{\rm{min}}(F')>\mu_\sigma^{\rm{max}}(F'')$. 
	Here $\mu_\sigma^{\rm{max}}(-)$ (resp.~$\mu_\sigma^{\rm{min}}(-)$) is the 
	maximal (resp.~minimal)
	slope among 
	the Harder-Narasimhan factors of $(-)$.
	By the assumption $c_1(\omega_S) \in \mathbb{R}_{\le 0} \cdot H$, 
	taking the tensor product with $\omega_S$ preserves
	$\sigma$-stability. 
	Therefore we have 
	\begin{align*}
		\mu_\sigma^{\rm{min}}(F')>
		\mu_{\sigma}^{\rm{max}}(F'') \ge 
		\mu_\sigma^{\rm{max}}(F'' \otimes \omega_S),
	\end{align*}
	hence the vanishing $\Hom(F', F'' \otimes \omega_S)=0$
	holds. Then we obtain the diagram (\ref{dia:reduced}), 
	and a contradiction. 
	
	As for the condition (\ref{assume:preserve}), the same argument 
	as above applies to the case of $c_1(\omega_S)=0$. 
	So we assume that $c_1(\omega_S) \in \mathbb{R}_{<0} \cdot H$. 
	Suppose that $E=(F, \theta)$ is $\sigma_{\pm}$-semistable 
	but $F$ is not $\sigma_{\pm}$-semistable. 
	Then similarly to above, we have an exact sequence (\ref{exact:F})
	such that $\mu_{\sigma_{\pm}}^{\rm{min}}(F')>\mu_{\sigma_{\pm}}^{\rm{max}}(F'')$. 
	Since $E$ is $\sigma$-semistable, 
	$F$ is also $\sigma$-semistable by the above argument, therefore 
	$F'$ and $F''$ are also $\sigma$-semistable with the same $\mu_{\sigma}$-slope. 
	Then $F'' \otimes \omega_S$ is also $\sigma$-semistable
	with 
	\begin{align*}
		\mu_{\sigma}(F')
		\ge \mu_{\sigma}(F'')
		>
		\mu_{\sigma}(F'' \otimes \omega_S),
	\end{align*}
	hence $\Hom(F', F'' \otimes \omega_S)=0$ holds. Similarly to 
	above, we obtain a contradiction. 
\end{proof}

A surface $S$ with $c_1(\omega_S)=0$ is classified 
into four types: K3 surface, abelian surface, Enriques surface and 
bielliptic surface. 
In the above case, 
by Corollary~\ref{cor:preserve} and Lemma~\ref{lem:wS} we have the following: 
\begin{cor}\label{cor:triv}
	Suppose that $S$ is a smooth projective surface 
	satisfying
	$c_1(\omega_S)=0$ in $H^2(S, \mathbb{R})$. 
	Then 
	for any 
	primitive $v \in N_{\le 1}(S)$ and 
	$\sigma_{\pm}\in A(S)_{\mathbb{C}}$
	which do not lie on walls, there exists an equivalence
	\begin{align}\label{equiv:triv}
		\Dbc(\fM_S^{\sigma_{+}}(v)^{\C\rig})
		\stackrel{\sim}{\to}
		\Dbc(\fM_S^{\sigma^-}(v)^{\C\rig}).
	\end{align}
\end{cor}
\begin{rmk}\label{rmk:K3}
	In the case that $S$ is a K3 surface,  
	the classical truncations
	\begin{align*}
		M_{S}^{\sigma_{\pm}}(v) \cneq t_0(\fM_{S}^{\sigma_{\pm}}(v))
	\end{align*}
	are holomorphic symplectic manifolds (\cite{Mu2, MR3194493}). 
	A derived equivalence
	\begin{align}\label{equiv:hol}
		\Dbc(M_{S}^{\sigma_{+}}(v)) \stackrel{\sim}{\to} 
		\Dbc(M_{S}^{\sigma_{-}}(v))
	\end{align}
	is a special case of Halpern-Leistner's result~\cite{HalpK3, HalpK32}.
	Since the closed immersion 
	$M_{S}^{\sigma_{\pm}}(v) \hookrightarrow \fM_{S}^{\sigma_{\pm}}(v)$
	is not an equivalence of derived stacks, 
	the equivalence (\ref{equiv:triv}) does not directly imply 
	(\ref{equiv:hol}). 
	This is caused by 
	the existence of surjections 
	$\Ext^2_S(E, E) \twoheadrightarrow \mathbb{C}$ for any 
	coherent sheaf $E$ on $S$. 
	In order to obtain an equivalence (\ref{equiv:hol}), we 
	replace $\fM_{S}^{\C\rig}(v)$ with 
	another derived stack $\fM_S^{\mathrm{rig}}(v)$
	obtained from $\fM_S^{\C\rig}(v)$
	by getting rid of the 
	above 
	one dimensional obstruction space from $\fM_{S}^{\C\rig}(v)$
	as in~\cite[Proposition~3.4.7]{HalpK32}.
	Then 
	we can recover the Halpern-Leistner's 
	equivalence (\ref{equiv:hol}) from the argument of Theorem~\ref{thm:surface}. 
\end{rmk} 

In the case of Enriques surface, Sacca~\cite{Sacca}
proved the following:
\begin{thm}\emph{(\cite{Sacca})}\label{thm:Sacca}
	Let $S$ be a general Enriques surface 
	and $\lvert C \rvert$ be a linear system 
	on it
	which contains an irreducible 
	divisor $C \subset S$ with 
	arithmetic genus $g\ge 2$. 
	We take $v=([C], n) \in \mathrm{NS}(S) \oplus \mathbb{Z}$ such that 
	$n\neq 0$ and 
	$([C], 2n)$ is coprime.  
	Then for generic
	$H_{\pm} \in A(S)_{\mathbb{R}}$, 
	the moduli spaces  
	$M_{S}^{H_{\pm}}(v)$  
	are smooth $(\beta^2+1)$-dimensional birational 
	Calabi-Yau manifolds. 
\end{thm}

By Corollary~\ref{cor:triv}, we have an equivalence of 
Sacca's Calabi-Yau manifolds for different polarizations: 
\begin{cor}\label{cor:Enriques}
	For a general Enriques surface $S$, let $v \in N_{\le 1}(S)$ 
	be as in Theorem~\ref{thm:Sacca}. 
	Then for generic $H_{\pm} \in A(S)_{\mathbb{R}}$, there exists
	an equivalence
	\begin{align}\label{equiv:Sacca}
		\Dbc(M_{S}^{H_{+}}(v)) \stackrel{\sim}{\to}
		\Dbc(M_{S}^{H_{-}}(v)). 
	\end{align}
\end{cor}
\begin{proof}
	Let $H \in A(S)_{\mathbb{R}}$ lies on a wall, and $H_{\pm}$ lie on 
	its adjacent chambers. 
	It is enough to show an equivalence (\ref{equiv:Sacca}) for 
	such $H_{\pm}$. 
	It is proved in~\cite[Lemma~2.5]{Sacca} that 
	there is no obstruction space for 
	$H_{\pm}$-stable sheaf $E$
	such that $[E] \in M_{S}^{H_{\pm}}(v)$, i.e. 
	$\Ext^2(E, E)=0$. Therefore for $\sigma_{\pm}=iH_{\pm}$, 
	the closed immersions
	\begin{align*}
		M_{S}^{H_{\pm}}(v)=
		t_0(\fM_{S}^{\sigma_{\pm}}(v)^{\C\rig}) \hookrightarrow 
		\fM_{S}^{\sigma_{\pm}}(v)^{\C\rig}
	\end{align*}
	are equivalences.  
	Therefore the corollary follows from Corollary~\ref{cor:triv}. 
\end{proof}

Let $S$ be a del-Pezzo surface, 
i.e. $-K_S$ is ample. 
In this case, we have the following: 

\begin{cor}\label{cor:delPezzo}
	Let $S$ be a del-Pezzo surface. 
	Then for any primitive $v\in N_{\le 1}(S)$
	and generic perturbations $\sigma_{\pm}$ of $-iK_S$, 
	we have a derived 
	equivalence of smooth projective varieties
	\begin{align*}
		\Dbc(M_{S}^{\sigma_{+}}(v)) \stackrel{\sim}{\to} 
		\Dbc(M_{S}^{\sigma_{-}}(v)). 
	\end{align*}
\end{cor}
\begin{proof}
	The moduli stack $\mM_S^{\sigma}(v)$ is a smooth stack 
	without obstruction space, i.e. for any Gieseker 
	$-K_S$-semistable sheaf $E$ on $S$
	with $[E]=v$, we have 
	\begin{align*}
		\Ext^2(E, E)=\Hom(E, E\otimes \omega_S)^{\vee}=0.
	\end{align*}
	Therefore 
	for generic perturbations $\sigma_{\pm}$ of $-iK_S$, 
	the moduli 
	spaces $M_{S}^{\sigma_{\pm}}(v)$ are smooth projective varieties 
	such that the closed immersions
	\begin{align*}
		M_{S}^{\sigma_{\pm}}(v) =t_0(\fM_{S}^{\sigma_{\pm}}(v)^{\C\rig}) \hookrightarrow 
		\fM_{S}^{\sigma_{\pm}}(v)^{\C\rig}
	\end{align*}
	are equivalences.  
	Therefore the corollary follows from 
	Corollary~\ref{cor:preserve}
	and Lemma~\ref{lem:wS}. 
\end{proof}

\subsection{Application to categorical MNOP/PT correspondence}\label{subsec:app2}
Here we apply the result of Corollary~\ref{cor:equivalence}
to prove Conjecture~\ref{conj:DT/PT}
for reduced curve classes. 
Below we use the notation of Section~\ref{subsec:catDTPT}. 
\subsubsection{The moduli stack of semistable objects on MNOP/PT wall}
Let $S$ be a smooth projective surface, and 
consider the moduli stacks of 
pairs $\mM_S^{\dag}(v)$, 
moduli stacks of D0-D2-D6 bound states 
$\mM_X^{\dag}(v)$ considered in Subsection~\ref{subsec:pairs}, Subsection~\ref{subsec:moduliD026} 
respectively. 
For $v=(\beta, n) \in N_{\le 1}(S)$, 
we consider open substacks
\begin{align}\notag
\tT_n(S, \beta) \subset \mM_S^{\dag}(\beta, n), \ 
\tT_n(X, \beta) \subset \mM_X^{\dag}(\beta, n)
\end{align}
corresponding to pairs $(\oO_S \stackrel{\xi}{\to} F)$
such that $\Cok(\xi)$ is at most zero dimensional, 
objects $\eE \in \aA_X$ such that $\hH^1(\eE)$
is at most zero dimensional, respectively. 
Note that the latter moduli stack $\tT_n(X, \beta)$ is considered in (\ref{stack:Tn})
 as the stack of 
semistable objects on the MNOP/PT wall. 
Similarly to Lemma~\ref{lem:shiftP}, 
we have the open immersion 
\begin{align}\label{inc:sharp}
	(\pi_{\ast}^{\dag})^{-1}(\tT_n(S, \beta)) \subset \tT_n(X, \beta). 
	\end{align}
We define some line bundles on 
the above moduli stacks. 
\begin{defi}
	We define $l_I, l_P \in \Pic(\mM_S^{\dag})$ to be
	\begin{align*}
		l_I \cneq \det (\dR p_{\mM^{\dag}\ast}\fF), \ 
		I_P \cneq \det (\dR p_{\mM^{\dag}\ast}\fF)^{-1}. 
	\end{align*}
Here $\oO_{S \times \mM_S^{\dag}} \to \fF$ be the universal pair. 
	The restrictions of $l_I$, $l_P$ to $\tT_n(S, \beta)$ and 
	pull-backs via 
	$\pi_{\ast}^{\dag} \colon \tT_n(X, \beta) \to \mM_S^{\dag}$
	are also denoted by $l_I$, $l_P$. 
\end{defi}
We have the following proposition: 
\begin{prop}\label{prop:DTPT:line}
	For $l_I, l_P \in \Pic(\tT_n(X, \beta))$, 
	we have the identities
	\begin{align}\label{id:IP}
		I_n(X, \beta)=\tT_n(X, \beta)^{l_I \sss}, \ 
		P_n(X, \beta)=\tT_n(X, \beta)^{l_P \sss}. 
	\end{align}
\end{prop}
\begin{proof}
	It is enough to show the identities (\ref{id:IP}) on the fibers of
	good moduli space morphisms
	$\pi_{\tT} \colon \tT_n(X, \beta)
	\to T_n(X, \beta)$. 
	A closed point $y\in T_n(X, \beta)$ corresponds 
	to a polystable object on MNOP/PT wall (see~\cite[Appendix~B]{Toddbir})
	\begin{align}\label{pstable:wall}
		I_C \oplus \bigoplus_{i=1}^m V_i \otimes \oO_{p_i}[-1]. 
	\end{align}
	Here $I_C \subset \oO_{\overline{X}}$ 
	is the ideal sheaf of a Cohen-Macaulay curve 
	$C \subset X$, 
	$p_1, \ldots, p_m \in X$ are distinct points 
	and $V_i$ is a finite dimensional vector space. 
	Let $Q$ be the Ext-quiver associated 
	with the collection $\{I_C, \oO_{p_1}[-1], \ldots, \oO_{p_m}[-1]\}$. 
	Then the vertex set $V(Q)$ is $\{0, 1, \ldots, m\}$, 
	and the edge set $E(Q)$ is given by 
	\begin{align*}
		\sharp(0 \to i)=\hom(I_C, \oO_{p_i}), \ 
		\sharp(i \to 0)=\ext^2(\oO_{p_i}, I_C), \ 
		\sharp(0 \to 0)=\ext^1(I_C, I_C)
	\end{align*}
	and also for $1\le i, j \le m$
	\begin{align*}
		\sharp(i \to j)=
		\left\{  \begin{array}{ll}
			3, & (i=j) \\
			0, & (i\neq j). 
		\end{array}
		\right.
	\end{align*}
By the argument in Subsection~\ref{subsec:Equiver}, 
	we have the closed immersion 
	\begin{align}\label{emb:MNOP/PT}
		\pi_{\tT}^{-1}(y) 
		\hookrightarrow 
		\left[\left\{(V_i)^{\sharp(0\to i)} \oplus 
		(V_i^{\vee})^{\sharp(i \to 0)} 
		\oplus \bigoplus_{i=1}^m \End(V_i)^{\oplus 3}
		\oplus \mathbb{C}^{\sharp(0 \to 0)}
		\right\}^{\rm{nil}} /G\right]. 
	\end{align}
	Here the RHS is the moduli stack of 
	nilpotent $Q$-representations
	with dimension vector $(1, \{\dim V_i\}_{1\le i\le k})$.
	The algebraic group 
	$G=\prod_{i=1}^m\GL(V_i)$ acts on 
	$V_i$, $V_i^{\vee}$ in a standard way, 
	on $\End(V_i)$ by the conjugation and on 
	$\mathbb{C}^{\sharp(0 \to 0)}$ trivially. 
	
	Let $x \in \tT_n(X, \beta)$ be a closed point 
	corresponding to the object (\ref{pstable:wall}).
	Then we have $G=\Aut(x)$ and 
	\begin{align*}
		(l_I)_{x}=\det \dR \Gamma\left(\oO_C \oplus \bigoplus_{i=1}^m 
		V_i \otimes \oO_{p_i}\right) 
		\cong \bigotimes_{i=1}^m \det V_i
	\end{align*} 
	as an element of $\Pic(BG)$, i.e. 
	it is the 
	determinant character
	\begin{align}\label{det:char}
		\chi_0 \colon G \to \C,
		(g_i)_{1\le i\le m} \mapsto \prod_{i=1}^m \det(g_i).
	\end{align}
	
	Let $e_i$ for $1\le i\le m$ be the simple 
	$Q$-representation corresponding to the vertex $i$, 
	and 
	$Q_{0} \subset Q$ be the full 
	subquiver 
	whose vertex set is $\{1, \ldots, m\}$. 
	By Lemma~\ref{lem:KN},  
	a $Q$-representation $R$
	with dimension vector $(1, \{\dim V_i\}_{1\le i\le m})$
	is $\chi_0$-(semi)stable 
	if and only if 
	the images of $\mathbb{C} \to V_i$ for 
	$(0 \to i)$ generate
	$\oplus_{i=1}^m V_i$ as a $\mathbb{C}[Q_{0}]$-module.
	This is equivalent to that  
	$\Hom(R, e_i)=0$ for all $1\le i\le m$. 
	On the other hand, the fiber $\pi_{\tT}^{-1}(y)$ parametrizes 
	objects in the extension closure 
	\begin{align}\label{I:closure}
			\langle I_C, \oO_{p_1}[-1], \ldots, \oO_{p_m}[-1]\rangle_{\rm{ex}}.
	\end{align}
Under the embedding (\ref{emb:MNOP/PT}), an 
object $\eE$ in (\ref{I:closure}) corresponds to 
a $Q$-representation $R$ with $\Hom(R, e_i)=0$ if and only if 
	$\Hom(\eE, \oO_{p_i}[-1])=0$. 
	The object
	$\eE$ is represented by a pair $(\oO_{\overline{X}} \stackrel{\xi}{\to}E)$
	for a generically surjective $\xi$
	by Lemma~\ref{lem:MNOP/PTwall} , 
	and the condition $\Hom(\eE, \oO_{p_i}[-1])=0$ for all $i$ 
	is equivalent to that $\xi$ is surjective. 
	Therefore the identity for $I_n(X, \beta)$
	holds on $\pi_{\tT}^{-1}(y)$. 
	
	The identity for $P_n(X, \beta)$ is similar. 
	By Lemma~\ref{lem:KN}, 
	a $Q$-representation $R$
	with dimension vector $(1, \{\dim V_i\}_{1\le i\le m})$
	is $\chi_0^{-1}$-(semi)stable 
	if and only if 
	the images of duals of 
	$V_i \to \mathbb{C}$ for 
	$(i \to 0)$ generate
	$\oplus_{i=1}^m V_i^{\vee}$ as a $\mathbb{C}[Q_{0}]$-module, 
	which is equivalent to that  
	$\Hom(e_i, R)=0$ for all $1\le i\le m$. 
	Under the embedding (\ref{emb:MNOP/PT}), an object $\eE$ in (\ref{I:closure})
	corresponds to a $Q$-representation $R$ with $\Hom(e_i, R)=0$ if and only if 
	$\Hom(\oO_{p_i}[-1], \eE)=0$. 
	As $\eE$ is represented by a pair $(\oO_{\overline{X}} \stackrel{\xi}{\to}E)$
	for a generically surjective $\xi$, 
	the above condition is equivalent to that $E$ is
	pure. Therefore the identity for $P_n(X, \beta)$
	holds on $\pi_{\tT}^{-1}(y)$. 
\end{proof}

We have the derived open substack
\begin{align}\notag
\fT_n(S, \beta) \subset \fM_S^{\dag}(\beta, n)
\end{align}
whose classical truncation is $\tT_n(S, \beta)$. 
In the following proposition, we show the 
existence of a symmetric structure on $\fT_n(S, \beta)$
such that $l_P$ is compatible with it. 
\begin{prop}\label{prop:PTS}
	There exists a symmetric structure $\bS$ on 
	$\fT_n(S, \beta)$ such that 
	$l_I, l_P$ are $\bS$-generic 
	and $l_P$ is compatible with $\bS$. 
\end{prop}
\begin{proof}
	A closed point $x$ of the stack $T_n(S, \beta)$ corresponds to 
	a direct sum of pairs
	\begin{align*}
		(\oO_S \to F)=(\oO_S \twoheadrightarrow \oO_C) \oplus 
		\bigoplus_{i=1}^m V_i \otimes (0 \to \oO_{p_i}). 
	\end{align*}
	Here $C \subset S$ is a Cohen-Macaulay curve, 
	$p_1, \ldots, p_m \in S$ are distinct points and $V_i$ is a 
	finite dimensional vector space. 
	Then $\Aut(x)=\prod_{i=1}^m \GL(V_i)$. 
	From the description of the cotangent complex (\ref{perf:obs2}), 
	we have 
	\begin{align*}
		\mathbb{T}_{\fT_n(S, \beta)}|_{x}
		=\RHom\left(\oO_S(-C) \oplus \bigoplus_{i=1}^m (V_i \otimes \oO_{p_i})[-1], 
		\oO_C \oplus \bigoplus_{i=1}^m V_i \otimes \oO_{p_i} \right).
	\end{align*}
	By the above, 
	an easy computation shows that 
	\begin{align}\label{tan:sharp}
		&\hH^0(\mathbb{T}_{\fT_n(S, \beta)}|_{x})
		\oplus \hH^1(\mathbb{T}_{\fT_n(S, \beta)}|_{x})^{\vee} \\
		\notag&=\left(\bigoplus_{i=1}^m 
		\Ext_S^2(\oO_{p_i}, \oO_C)^{\vee} \oplus H^0(\oO_S(C)|_{p_i})\right)
		\otimes V_i 
		\oplus \left( \bigoplus_{i=1}^m \Ext_S^1(\oO_{p_i}, \oO_C)
		\right) \otimes V_i^{\vee} \\
		\notag& \quad \oplus \bigoplus_{i=1}^m \End(V_i)^{\oplus 3}
		\oplus H^0(\oO_C(C)) \oplus H^1(\oO_C(C))^{\vee}. 
	\end{align}
	We take the the decomposition $\bS_x \oplus \bU_x$ 
	of (\ref{tan:sharp})
	to be 
	\begin{align*}
		\bS_x&=\left(\bigoplus_{i=1}^m
		\Ext_S^2(\oO_{p_i}, \oO_C)^{\vee})\right)
		\otimes V_i 
		\oplus \left( \bigoplus_{i=1}^m \Ext_S^1(\oO_{p_i}, \oO_C)
		\right) \otimes V_i^{\vee} \\
		& \quad \oplus \bigoplus_{i=1}^m \End(V_i)^{\oplus 3}
		\oplus H^0(\oO_C(C)) \oplus H^1(\oO_C(C))^{\vee}, \\
		\bU_x&=\bigoplus_{i=1}^m 
		H^0(\oO_S(C)|_{p_i})
		\otimes V_i.
	\end{align*}
	As $C$ is Cohen-Macaulay
	we have $\Hom(\oO_{p_i}, \oO_C)=0$, 
	so  
	the Riemann-Roch theorem shows that 
	$\ext_S^1(\oO_{p_i}, \oO_C)=\ext_S^2(\oO_{p_i}, \oO_C)$.
	Therefore $\bS_x$ is a symmetric $\Aut(x)$-representation, 
	and the decomposition $\bS_x \oplus \bU_x$ of (\ref{tan:sharp})
	gives its symmetric structure. 
	
	Similarly to Remark~\ref{rmk:squiver}, 
	the $\Aut(x)$-representation (\ref{tan:sharp})
	is obtained as the space of $Q$-representations 
	with dimension vector $\{1, \{\dim V_i\}_{1\le i\le m}\}$, where 
	$Q$ is the Ext-quiver associated with 
	the collection 
	\begin{align*}
		\{I_{i(C)}, i_{\ast}\oO_{p_1}[-1], \ldots, i_{\ast}\oO_{p_m}[-1]\}.
	\end{align*} 
	Here $i \colon S \hookrightarrow X$ is the zero section. 
	Therefore from the proof of Proposition~\ref{prop:DTPT:line}, the 
	$(l_I)_x$-stability 
	on (\ref{tan:sharp}) is
	given by the determinant character
	$\chi_{0}$ in (\ref{det:char}), and 
	$(l_P)_x$ is given by $\chi_0^{-1}$. 
	By Lemma~\ref{lem:quiver:generic}, these characters are 
	$\bS_x$-generic, and $(l_P)_x$ is compatible with 
	$\bS_x$. 
		Therefore $l_I$, $l_P$ are $\bS$-generic, and $l_P$ is 
		compatible with $\bS$. 
\end{proof}

\subsubsection{Proof of categorical MNOP/PT correspondence}
In the case of reduced curve class, we have the following lemma: 
\begin{lem}\label{lem:red:PT}
	If $\beta$ is a reduced class, then 
	the inclusion (\ref{inc:sharp}) is the identity. 
	So we have 
	\begin{align}\label{id:TS=TX}
		\tT_n(X, \beta)=t_0(\Omega_{\fT_n(S, \beta)}[-1]). 
		\end{align}
\end{lem}
\begin{proof}
	By Lemma~\ref{lem:MNOP/PTwall}, a $\mathbb{C}$-valued point 
	in the right hand side of (\ref{inc:sharp})
	corresponds to a pair $(s \colon \oO_{\overline{X}} \to F)$
	such that $\Cok(s)$ is at most zero dimensional. 
	If $\beta$ is reduced, then $s$ is non-zero at each generic point of the support of $F$. 
	Therefore  
	$\pi_{\ast}s \colon \oO_S \to \pi_{\ast}F$ is non-zero at each 
	generic point of the support of $\pi_{\ast}F$. 
	Since the fundamental one cycle of $\pi_{\ast}F$ is reduced, this 
	implies that $\Cok(\pi_{\ast}s)$ is at most zero dimensional. 
\end{proof}
As an application of Corollary~\ref{cor:equivalence}, 
we give the following application to Conjecture~\ref{conj:DT/PT}:  
\begin{thm}\label{thm:catDTPT}
	Suppose that $\beta$ is a reduced class. 
	Then there exists a fully-faithful functor
	\begin{align*}
		\dDT^{\C}(P_n(X, \beta)) \hookrightarrow 
		\dDT^{\C}(I_n(X, \beta)). 
	\end{align*}
\end{thm}
\begin{proof}
	By Lemma~\ref{lem:red:PT} and Lemma~\ref{lem:replace0}, we can take 
a quasi-compact derived open substack in 
Definition~\ref{cat:DTPT}
 to be
	\begin{align*}
		\fM_S^{\dag}(\beta, n)_{\rm{qc}}=\fT_n(S, \beta). 
	\end{align*}
By (\ref{id:TS=TX}) and Proposition~\ref{prop:DTPT:line}, we have 
\begin{align*}
	I_n(X, \beta)=t_0(\Omega_{\fT_n(S, \beta)}[-1])^{l_I \sss}, \ 
		P_n(X, \beta)=t_0(\Omega_{\fT_n(S, \beta)}[-1])^{l_P \sss}. 
	\end{align*}
	As we mentioned before $\tT_n(S, \beta)$ admits a good moduli 
	space, and $\fT_n(S, \beta)$ satisfies formal neighborhood 
	theorem by Lemma~\ref{lem:PS:formal} (by applying it for $t=\infty$). 
	Then 
	by Proposition~\ref{prop:PTS}, 
	we can apply
	Corollary~\ref{cor:equivalence}
	to conclude the theorem. 
\end{proof}

\begin{rmk}\label{rmk:MNOP/PT}
	For a stable pair $(\oO_X \to E)$ on $X$, 
	if $E$ has a reduced support then it push-forwards to 
	a stable pair $(\oO_S \to \pi_{\ast}E)$ on $S$. 
	Therefore $P_n(X, \beta)=t_0(\Omega_{\fP_n(S, \beta)}[-1])$
	for a reduced class $\beta$, 
		where $\mathfrak{P}_n(S, \beta)$ is the derived 
	moduli space of stable pairs on $S$. 
Therefore 
	we have an equivalence by Lemma~\ref{lem:replace0}
	\begin{align*}
		\dDT^{\C}(P_n(X, \beta))
		\stackrel{\sim}{\to}
		\Dbc(\mathfrak{P}_n(S, \beta)).
	\end{align*}
	On the other hand, 
	a surjection $\oO_X \twoheadrightarrow F$ does not
	push-forward
	to a surjection $\oO_S \twoheadrightarrow \pi_{\ast}F$ in general, 
	even if
	the support of $F$ is reduced. 
	So for a reduced class $\beta$, the $\C$-equivariant MNOP category 
	$\dDT^{\C}(I_n(X, \beta))$ is not necessary 
	equivalent to $\Dbc(\mathfrak{I}_n(S, \beta))$, 
	where $\mathfrak{I}_n(S, \beta)$ is the derived moduli 
	space of closed subschemes in $S$.  
\end{rmk}

\section{Semiorthogonal decompositions via categorified Hall products}\label{sec:cat:hall}
In this section, we use categorified Hall products for 
DT categories to give an approach toward 
Conjectures~\ref{conj:DT/PT}, \ref{conj2}. 
The two dimensional categorified Hall algebras are introduced by 
Porta-Sala~\cite{PoSa} in order to categorify 
cohomological Hall algebras for surfaces introduced by Kapranov-Vasserot~\cite{KaVa2}. 
For a smooth projective surface $S$, 
the Porta-Sala categorified Hall product 
is defined by pull-back/push-forward 
diagram associated with the derived moduli stack $\fM_S^{\ext}(v_{\bullet})$
of short exact sequences
\begin{align*}
	0 \to F_1 \to F_3 \to F_2 \to 0, \ [F_i]=v_i.
	\end{align*}
It is equipped with evaluation morphisms 
$\ev_i \colon \fM_S^{\ext}(v_{\bullet}) \to \fM_S(v_i)$, 
so we have a correspondence diagram 
\begin{align*}
	\xymatrix{
\fM_S^{\ext}(v_{\bullet}) \ar[r]^-{\ev_3} 
\ar[d]_-{(\ev_1, \ev_2)} & \fM_S(v_3) \\
\fM_S(v_1) \times \fM_S(v_2).  &	
}
	\end{align*}
The Porta-Sala categorified Hall product is given by the functor 
\begin{align*}
	\ev_{3\ast}(\ev_1, \ev_2)^{\ast} \colon 
	\Dbc(\fM_S(v_1)) \times \Dbc(\fM_S(v_2))
	\to \Dbc(\fM_S(v_3)). 
	\end{align*}

In a separate paper~\cite{THtype}, we proved that 
Porta-Sala categorified Hall product is compatible with 
singular supports, so that it descends to 
the categorified Hall product on DT categories
for the local surface $X=\mathrm{Tot}_S(\omega_S)$
\begin{align*}
	\dDT^{\C}(\mM_X^{\sigma}(v_1)) \times \dDT^{\C}(\mM_X^{\sigma}(v_2))
	\to \dDT^{\C}(\mM_X^{\sigma}(v_3)).
	\end{align*}
A similar construction yields an action of 
DT categories for one or zero dimensional semistable sheaves on $X$
to those for stable D0-D2-D6 bound states on $X$.  
Here we explain these actions 
specializing to MNOP/PT wall. 
Let $\tT_n(X, \beta)$ be the moduli stack of 
semistable objects on the MNOP/PT wall, and $\mM_i(X)$ be the 
moduli stack of zero dimensional sheaves on $X$ with length $i$. 
Then we will see that there is a left/right action of DT categories for 
$\mM_i(X)$ to those for $\tT_n(X, \beta)$:
\begin{align*}
\bigoplus_{i\ge 0} \dDT^{\C}(\mM_i(X)) \acts
\bigoplus_{n \in \mathbb{Z}} \dDT^{\C}(\tT_n(X, \beta))  \lacts
\bigoplus_{i \ge 0} \dDT^{\C}(\mM_i(X)). 
	\end{align*}
The above actions descend to those on left/right actions on 
MNOP/PT categories, and in particular the right actions on PT categories 
are studied in~\cite{THtype} in detail. 
On the other hand, the above actions
do not descend to right/left actions on MNOP/PT categories, 
which we focus in this section. 

Our main conjecture in Conjecture~\ref{conj:SOD}
specialized to the MNOP/PT wall 
yields a conjectural semiorthogonal decomposition of $\dDT^{\C}(\tT_n(X, \beta))$
described in Theorem~\ref{conj:intro:dtpt}, which we will prove 
when $\beta$ is reduced.  
Although the proof works only for reduced $\beta$, 
this approach may give an ansatz toward Conjecture~\ref{conj:DT/PT}
for any curve class. In particular it gives another proof of Theorem~\ref{thm:catDTPT}. 
Our main conjecture in Conjecture~\ref{conj:SOD} extends 
the above construction to further wall-crossing of 
D0-D2-D6 bound states, which 
gives an ansatz toward Conjecture~\ref{conj2}. 
By proving the above conjecture, we will settle 
Conjecture~\ref{conj2} when $\beta$ is a reduced class
and $t>m(\beta)$ where $m(\beta)$ only depends on $\beta$. 

The organization of this section is as follows. 
In Section~\ref{subsec:catHall}, we recall categorified Hall products on 
surfaces and local surfaces. 
In Section~\ref{subsec:mainconj}, we formulate the main conjecture on 
semiorthogonal decompositions via categorified Hall products. 
In Section~\ref{subsec:proof}, we prove our conjectures under the assumption 
that the semistability on $X$ is preserved under push-forward to $S$
at the wall.

\subsection{Categorified Hall products}\label{subsec:catHall}
\subsubsection{Derived moduli stacks of extensions}
Let $S$ be a smooth projective surface. 
We define the derived moduli stack 
of exact sequences of coherent sheaves on $S$, 
following~\cite[Section~3]{PoSa}.
It is given by the derived Artin stack 
\begin{align}\notag
	\mathfrak{M}_S^{\rm{ext}} \colon dAff^{op} \to SSets
\end{align}
which sends an affine derived scheme 
$T$ to the $\infty$-groupoid of fiber 
sequences of 
perfect complexes on $T \times S$, 
\begin{align}\label{seq:F}
	\ffF_1 \to \ffF_3 \to \ffF_2
\end{align}
whose restrictions to $t_0(T) \times S$ are 
flat families of exact sequences of 
coherent sheaves on $S$ over $t_0(T)$. 
The classical truncation of $\fM_S^{\ext}$ is denoted by 
$\mM_S^{\ext} \cneq t_0(\fM_S^{\ext})$. 
We have the decompositions into open and closed substacks
\begin{align*}
	\fM_S^{\ext}=\coprod_{v_{\bullet}=(v_1, v_2)} 
	\fM_S^{\ext}(v_{\bullet}), \ 
	\mM_S^{\ext}=\coprod_{v_{\bullet}=(v_1, v_2)} \mM_S^{\ext}(v_{\bullet})
\end{align*}
where each component corresponds to 
exact sequences $0\to F_1 \to F_3 \to F_2 \to 0$
with $[F_i]=v_i$ for $i=1, 2$. 

By sending a sequence (\ref{seq:F}) to $\ffF_i$, 
we have the evaluation morphisms
\begin{align*}
	\mathrm{ev}_i \colon \mathfrak{M}_S^{\rm{ext}}(v_{\bullet})
	\to \mathfrak{M}_S(v_i), \ 1\le i\le 3 
\end{align*}
where we set $v_3=v_1+v_2$. 
Below we use the following diagram
\begin{align}\label{dia:extS}
	\xymatrix{
		\fM_S^{\ext}(v_{\bullet}) \ar[r]^-{\ev_3} \ar[d]_-{(\ev_1, \ev_2)} & \fM_S(v_3)
		\\
		\fM_S(v_1) \times \fM_S(v_2). & 
	}
\end{align}
The morphism $(\ev_1, \ev_2)$ is quasi-smooth (see~\cite[Corollary~3.10]{PoSa}), 
so we have the well-defined functor
\begin{align*}
	(\ev_1, \ev_2)^{\ast} \colon 
	\Dbc(\mathfrak{M}_S(v_1) \times \mathfrak{M}_S(v_2)) \to 
	\Dbc(\mathfrak{M}_S^{\ext}(v_{\bullet})). 
\end{align*}
The
horizontal morphism in (\ref{dia:extS})
is a proper morphism, so we have the well-defined functor
\begin{align*}
	\ev_{3\ast} \colon \Dbc(\mathfrak{M}_S^{\ext}(v_{\bullet})) \to 
	\Dbc(\mathfrak{M}_S(v_3)). 
\end{align*}
The 
Porta-Sala categorified Hall product is given by 
the composition~\cite{PoSa}
\begin{align}\label{cat:COHA}
	\Dbc(\mathfrak{M}_S(v_1)) \times 
	\Dbc(\mathfrak{M}_S(v_2)) 
	\stackrel{\boxtimes}{\to}
	\Dbc(\mathfrak{M}_S(v_1) \times \mathfrak{M}_S(v_2))
	\stackrel{\ev_{3\ast}(\ev_1, \ev_2)^{\ast}}{\longrightarrow}
	\Dbc(\mathfrak{M}_S(v_3))
\end{align}
which categorifies
cohomological Hall algebra on surfaces by Kapranov-Vasserot~\cite{KaVa2}.

Similarly we consider the 
classical Artin stack
of short exact sequences of 
compactly supported coherent sheaves on $X$.
It is given by the 2-functor
\begin{align*}
	\mM_X ^{\ext} \colon  
	Aff^{\rm{op}} \to Groupoid
\end{align*}
whose $T$-valued points for $T \in Aff$
form the groupoid of 
exact sequences of coherent sheaves on $X \times T$, 
\begin{align*}
	0 \to \eE_1 \to \eE_3 \to \eE_2 \to 0, \ 
	\eE_i \in \mM_X(T).
\end{align*}
We have the decomposition into open and closed substacks
\begin{align*}
	\mM_X^{\ext}=\coprod_{v_{\bullet}=(v_1, v_2)}
	\mM_X^{\ext}(v_{\bullet})
\end{align*}
where each component corresponds to 
exact sequences $0 \to E_1 \to E_3 \to E_2 \to 0$
with $[\pi_{\ast}E_i]=v_i$. 
We have the evaluation morphisms
\begin{align*}
	\ev_{i}^X \colon \mM_X^{\ext}(v_{\bullet}) \to \mM_X(v_i), \ 
	E_{\bullet} \mapsto E_i
\end{align*}
and obtain the diagram
\begin{align}\label{dia:extX}
	\xymatrix{
		\mM_X^{\ext}(v_{\bullet})
		\ar[r]^-{\ev_3^X} \ar[d]_-{(\ev_1^X, \ev_2^X)} & \mM_X(v_3) \\
		\mM_X(v_1) \times \mM_X(v_2). &
	}
\end{align}
We also have the morphism given by the push-forward along 
$\pi \colon X \to S$
\begin{align}\notag
	\pi_{\ast} \colon \mM_X^{\ext}(v_{\bullet})
	\to \mM_S^{\ext}(v_{\bullet}), \ 
	E_{\bullet} \mapsto \pi_{\ast}E_{\bullet}.
\end{align}
Here note that $\pi_{\ast}$ preserves the 
exact sequences of coherent sheaves 
as $\pi$ is affine.

\subsubsection{Categorified Hall algebras for local surfaces}
Below we fix an ample divisor $H$ on $S$. 
For $(\beta, n) \in N_{\le 1}(S)$, 
we use the following notation for moduli stacks of semistable sheaves
(see Section~\ref{subsec:catDTsurf})
\begin{align*}
	\mM_n^H(X, \beta) \cneq \mM_X^{\sigma=iH}(\beta, n), \ 
	\fM_n^H(S, \beta) \cneq \fM_S^{\sigma=iH}(\beta, n). 
	\end{align*}
For each 
$t \in \mathbb{Q} \cup \{\infty\}$, we set
\begin{align*}
	N_{\le 1}(S)_{t} \cneq 
	\{ v\in N_{\le 1}(S) : \mu_H(v)=t\}. 
\end{align*}
Here $\mu_H(v)=n/(H \cdot \beta)$ for $v=(\beta, n)$, which is $\infty$
if $H \cdot \beta=0$. 
The following result in~\cite{THtype}
shows that the categorified Hall products (\ref{cat:COHA})
descend to those on DT categories. 
\begin{thm}\label{thm:COHA}\emph{(\cite[Theorem~3.5]{THtype})}
	For $v_{\bullet}=(v_1, v_2) \in N_{\le 1}(S)_{t}^{\times 2}$
	with $v_3=v_1+v_2$, $v_i=(\beta_i, n_i)$, 
	the functor (\ref{cat:COHA}) descends to the functor
	\begin{align}\notag
		\mathcal{DT}^{\mathbb{C}^{\ast}}(\mM_{n_1}^H(X, \beta_1))
		\times 
		\mathcal{DT}^{\mathbb{C}^{\ast}}(\mM_{n_2}^H(X, \beta_2))
		\to \mathcal{DT}^{\mathbb{C}^{\ast}}(\mM_{n_3}^H(X, \beta_3)). 
	\end{align}
\end{thm}
\begin{proof}
	For the reader's convenience, we 
	 give an outline of the proof in~\cite[Theorem~3.5]{THtype}. 
	Let $\ev$ be the morphism 
	\begin{align*}
		\ev \cneq (\ev_1, \ev_2, \ev_3) \colon 
		\fM_S^{\ext} \to \fM_S(v_1) \times \fM_S(v_2) \times \fM_S(v_3). 
		\end{align*}
	Then we have the following diagram by (\ref{diagram:fN})
	\begin{align*}
		\xymatrix{
		t_0(\Omega_{\ev}[-2]) \ar[r] \ar[d] & t_0(\Omega_{\fM_S(v_3)}[-1]) \\
		t_0(\Omega_{\fM_S(v_1)}[-1])
		\times t_0(\Omega_{\fM_S(v_2)}[-1]). 
	}
		\end{align*}
	Then one can show (see~\cite[Proposition~3.1]{THtype}) that the above diagram is isomorphic 
	to the diagram (\ref{dia:extX}). 
	Moreover from the definition of $H$-stability, 
	the diagram (\ref{dia:extX}) restricts to the diagram
	\begin{align}\notag
		\xymatrix{
			(\ev_3^X)^{-1}(\mM_{n_3}^H(X, \beta_3))
			\ar[r]^-{\ev_3^X} \ar[d]_-{(\ev_1^X, \ev_2^X)} & \mM_{n_3}^H(X, \beta_3) \\
		\mM_{n_1}^H(X, \beta_1) \times \mM_{n_2}^H(X, \beta_2). &
		}
	\end{align}
	Therefore the claim follows from Proposition~\ref{prop:sendC}.
	\end{proof}

\subsubsection{Derived moduli stacks of extensions of pairs}
Let $\fM_S^{\dag}$ be the derived moduli stack of pairs 
in Subsection~\ref{subsec:pairs}. 
We define the derived stack 
$\fM_S^{\ext, \dag}$ by the 
Cartesian square
\begin{align*}
	\xymatrix{
		\fM_S^{\ext, \dag} \ar@{}[rd]|\square
		\ar[r] \ar[d]_-{(\ev_1^{\dag}, \ev_2^{\dag})} 
		& \fM_S^{\ext} \ar[d]^-{(\ev_1, \ev_2)} \\
		\fM_S^{\dag} \times \fM_S \ar[r]^-{(\rho^{\dag}, \id)} & 
		\fM_S \times \fM_S. 
	}
\end{align*}
For $T \in dAff$, the $T$-valued points of $\fM_S^{\dag}$ form the 
$\infty$-groupoid of diagrams
\begin{align}\label{diagram:dag}
	\xymatrix{
		\oO_{S \times T} \ar[d]_-{\xi} &  &  \\
		\ffF_1 \ar[r] & \ffF_3 \ar[r] & \ffF_2.
	}
\end{align}
Here the bottom sequence is a $T$-valued point of 
$\fM_S^{\ext}$. 
Let us take $v_{\bullet}=(v_1, v_2) \in N_{\le 1}(S)^{\times 2}$
and $v_3=v_1+v_2$. 
We have the open and closed derived substack
\begin{align*}
	\fM_S^{\ext, \dag}(v_{\bullet}) \subset \fM_S^{\ext, \dag}
\end{align*}
corresponding to the diagram
(\ref{diagram:dag}) such that 
the bottom sequence is a $T$-valued point of $\fM_S^{\ext}(v_{\bullet})$. 
By~\cite[Lemma~4.3]{THtype}, 
the derived stack $\fM_S^{\ext, \dag}(v_{\bullet})$ is quasi-smooth. 
We have the diagram
\begin{align}\label{dia:MSdag}
	\xymatrix{
		\fM_S^{\ext, \dag}(v_{\bullet}) \ar[r]^-{\ev_3^{\dag}} \ar[d]_-{(\ev_1^{\dag}, \ev_2^{\dag})}
		& \fM_S^{\dag}(v_3) \\
		\fM_S^{\dag}(v_1) \times \fM_S(v_2). 
	}
\end{align}
Here $\ev_3^{\dag}$ 
is obtained by sending 
a diagram (\ref{diagram:dag})
to the composition 
$\oO_{S \times T} \stackrel{\xi}{\to} \ffF_1 \to \ffF_3$. 
Note that the vertical arrow is quasi-smooth
and the horizontal arrow is proper (see~\cite[Lemma~4.4]{THtype}). 
Therefore 
the diagram (\ref{dia:MSdag}) 
induces the functor
\begin{align}\label{FM:extP}
	\ev_{3\ast}^{\dag}(\ev_1^{\dag}, \ev_2^{\dag})^{\ast}
	\colon \Dbc(\fM_S^{\dag}(v_1) \times \fM_S(v_2))
	\to \Dbc(\fM_S^{\dag}(v_3)). 
\end{align}

Similarly let $v_{\bullet}'=(v_2, v_1)$ and 
define the derived stack 
$\fM_S^{\ext, \ddag}(v_{\bullet}')$ by the 
Cartesian square
\begin{align*}
	\xymatrix{
		\fM_S^{\ext, \ddag}(v_{\bullet}') \ar[r]^-{\ev_3^{\ddag}} \ar[d] 
		\ar@{}[rd]|\square
		& \fM_S^{\dag}(v_3) \ar[d]^-{\rho^{\dag}} \\
		\fM_S^{\ext}(v_{\bullet}') \ar[r]^-{\ev_3} & \fM_S(v_3). 
	}
\end{align*}
For $T \in dAff$, the $T$-valued points of $\fM_S^{\dag}$ form the 
$\infty$-groupoid of diagrams
\begin{align}\label{diagram:ddag}
	\xymatrix{
		& \oO_{S \times T}\ar[d]^-{\xi}  &  \\
		\ffF_2 \ar[r] & \ffF_3 \ar[r] & \ffF_1.
	}
\end{align}
Here the bottom sequence is a $T$-valued point of 
$\fM_S^{\ext}(v_{\bullet}')$. 
We have the following diagram
\begin{align}\label{diagram:ddag2}
	\xymatrix{
		\fM_S^{\ext, \ddag}(v_{\bullet}) \ar[r]^{\ev_3^{\ddag}} 
		\ar[d]_{(\ev_2^{\ddag}, \ev_1^{\ddag})} & \fM_S^{\dag}(v_3) \\
		\fM_S(v_2) \times \fM_S^{\dag}(v_1). & 
	}
\end{align}
Here $\ev_2^{\ddag}$ sends a diagram (\ref{diagram:ddag}) to $\ffF_2$, 
and $\ev_1^{\ddag}$ sends a diagram (\ref{diagram:ddag})
to the composition 
$\oO_{S \times T} \stackrel{\xi}{\to} \ffF_3 \to \ffF_1$. 
Note that $\ev_3^{\ddag}$ is proper by definition. 
By~\cite[Lemma~7.1]{THtype}, 
the morphism $(\ev_2^{\ddag}, \ev_1^{\ddag})$ is quasi-smooth, 
in particular the derived stack 
$\fM_S^{\ext, \ddag}(v_{\bullet}')$ is quasi-smooth. 
Therefore the diagram (\ref{diagram:ddag2}) induces the functor
\begin{align}\label{FM:extP2}
	\ev_{3\ast}^{\ddag}(\ev_2^{\ddag}, \ev_1^{\ddag})^{\ast}
	\colon \Dbc(\fM_S(v_2) \times \fM_S^{\dag}(v_1))
	\to \Dbc(\fM_S^{\dag}(v_3)). 
\end{align}

\subsubsection{Moduli stacks of extensions in $\aA_X$}
Let $\mM_X^{\dag}$ be the moduli stack of D0-D2-D6 bound states 
in Subsection~\ref{subsec:moduliD026}. 
We
take $v_{\bullet}=(v_1, v_2) \in N_{\le 1}(S)^{\times 2}$
with $v_3=v_1+v_2$, 
and define the classical 
stack
\begin{align*}
	\mM_X^{\ext, \dag}(v_{\bullet}) \colon Aff^{op} \to Groupoid
\end{align*}
by sending $T \in Aff$ to the groupoid 
of distinguished triangles
\begin{align*}
	\eE_1 \stackrel{j}{\to} \eE_3 \to \eE_2[-1]
\end{align*}
where $(\eE_i, \lambda_i)$ for $i=1, 3$ are $T$-valued points of 
$\mM_X^{\dag}(v_i)$, 
$\eE_2$ is a $T$-valued point of $\mM_X(v_2)$
and $j$ commutes with trivializations $\lambda_i$ at $S_{\infty}$. 
We also have the evaluation morphisms
\begin{align}\label{dia:Xext:dag}
	\xymatrix{
		\mM_X^{\ext, \dag}(v_{\bullet}) 
		\ar[r]^-{\ev_3^{X, \dag}} \ar[d]_-{(\ev_1^{X, \dag}, \ev_2^{X, \dag})} 
		& \mM_X^{\dag}(v_3) \\
		\mM_X^{\dag}(v_1) \times \mM_X(v_2) & 
	}
\end{align}
where $\ev_i^{X, \dag}$ sends
$\eE_{\bullet}$ to $\eE_i$. 

For $v_{\bullet}'=(v_2, v_1)$, we 
also define the classical 
stack
\begin{align*}
	\mM_X^{\ext, \ddag}(v_{\bullet}') \colon Aff^{op} \to Groupoid
\end{align*}
by sending $T \in Aff$ to the groupoid 
of distinguished triangles
\begin{align*}
	\eE_2[-1] \to \eE_3 \stackrel{j'}{\to} \eE_1
\end{align*}
where $(\eE_i, \lambda_i)$ for $i=1, 3$ are $T$-valued points of 
$\mM_X^{\dag}(v_i)$, 
$\eE_2$ is a $T$-valued point of $\mM_X(v_2)$
and $j'$ commutes with trivializations $\lambda_i$ 
at $S_{\infty}$.  
We also have the evaluation morphisms
\begin{align}\label{dia:Xext:dagA}
	\xymatrix{
		\mM_X^{\ext, \ddag}(v_{\bullet}') 
		\ar[r]^-{\ev_3^{X, \ddag}} \ar[d]_-{(\ev_2^{X, \ddag}, \ev_1^{X, \ddag})} 
		& \mM_X^{\dag}(v_3) \\
		\mM_X(v_2) \times \mM_X^{\dag}(v_1) & 
	}
\end{align}
where $\ev_i^{X, \ddag}$ sends
$\eE_{\bullet}$ to $\eE_i$. 
For substacks $\zZ_1 \subset \mM_X^{\dag}(v_1)$ and 
$\zZ_2 \subset \mM_X(v_2)$, from the diagrams (\ref{dia:Xext:dag}), (\ref{dia:Xext:dagA})
we define 
\begin{align}\label{ast:X}
	&\zZ_1 \ast \zZ_2 \cneq \ev_3^{X, \dag}(\ev_1^{X, \dag}, \ev_2^{X, \dag})^{-1}(\zZ_1 \times \zZ_2), \\
	\notag
	&\zZ_2 \ast \zZ_1 \cneq \ev_3^{X, \ddag}(\ev_2^{X, \ddag}, \ev_1^{X, \ddag})^{-1}(\zZ_2 \times \zZ_1). 
\end{align}
Note that they are closed substacks in $\mM_X^{\dag}(v_3)$, $\mM_X^{\ddag}(v_3)$
if $\zZ_i$ are closed substacks, 
since $\ev_3^{X, \dag}$, $\ev_3^{X, \ddag}$ are proper. 

The above stacks of extensions in $\aA_X$
are realized as $(-2)$-shifted conormal stacks as follows. 
Let $\ev^{\dag}$, $\ev^{\ddag}$ be the following morphisms in the diagrams (\ref{dia:MSdag}), (\ref{diagram:ddag2})
\begin{align*}
	&\ev^{\dag}=(\ev_1^{\dag}, \ev_2^{\dag}, \ev_3^{\dag}) \colon 
	\fM_S^{\ext, \dag}(v_{\bullet}) \to 
	\fM_S^{\dag}(v_1) \times \fM_S(v_2) \times \fM_S^{\dag}(v_3), \\
	&\ev^{\ddag}=(\ev_1^{\ddag}, \ev_2^{\ddag}, \ev_3^{\ddag}) \colon 
	\fM_S^{\ext, \ddag}(v_{\bullet}) \to 
	\fM_S^{\dag}(v_1) \times \fM_S(v_2) \times \fM_S^{\dag}(v_3).
\end{align*}
Then similarly to Theorem~\ref{thm:COHA}, 
the diagrams (\ref{dia:Xext:dag}), (\ref{dia:Xext:dagA}) are isomorphic to the diagrams 
\begin{align*}
		&\xymatrix{
		t_0(\Omega_{\ev^{\dag}}[-2]) \ar[r] \ar[d] & t_0(\Omega_{\fM^{\dag}_S(v_3)}[-1]) \\
		t_0(\Omega_{\fM^{\dag}_S(v_1)}[-1])
		\times t_0(\Omega_{\fM_S(v_2)}[-1]), 
	}  \\
	&\xymatrix{
	t_0(\Omega_{\ev^{\ddag}}[-2]) \ar[r] \ar[d] & t_0(\Omega_{\fM^{\dag}_S(v_3)}[-1]) \\
	t_0(\Omega_{\fM_S(v_2)}[-1])
	\times t_0(\Omega_{\fM_S^{\dag}(v_1)}[-1]),
}
	\end{align*}
respectively (see~\cite[Proposition~4.7]{THtype}). 
Therefore similarly to Theorem~\ref{thm:COHA}, 
Proposition~\ref{prop:sendC} yields the following: 
\begin{prop}\label{prop:inducedF}
	\emph{(\cite[Corollary~4.8]{THtype})}
	Let 
	$\mM_X^{\dag}(v_i)^{\circ} \subset \mM_X(v_i)$
	for $i=1, 3$ and 
	$\mM_X(v_2)^{\circ} \subset \mM_X(v_2)$ be open 
	substacks of finite type. We regard them as open substacks 
	in $t_0(\Omega_{\fM_S^{\dag}(v_i)}[-1])$, 
	$t_0(\Omega_{\fM_S(v_2)}[-1])$ by 
	the isomorphisms (\ref{isom:dag}), (\ref{isom:etaM}) respectively. 
	
	(i) Suppose that the diagram (\ref{dia:Xext:dag})
	restricts to the diagram
	\begin{align}\label{dia:Xext:dag2}
		\xymatrix{
			(\ev_3^{X, \dag})^{-1}(\mM_X^{\dag}(v_3)^{\circ})
			\ar[r]^-{\ev_3^{X, \dag}} \ar[d]_-{(\ev_1^{X, \dag}, \ev_2^{X, \dag})} 
			& \mM_X^{\dag}(v_3)^{\circ} \\
			\mM_X^{\dag}(v_1)^{\circ} \times \mM_X(v_2)^{\circ}. & 
		}
	\end{align}
	Then the functor (\ref{FM:extP}) descends to the functor 
	\begin{align*}
		\dDT^{\C}(\mM_X^{\dag}(v_1)^{\circ})
		\otimes \dDT^{\C}(\mM_X(v_2)^{\circ}) \to \dDT^{\C}(\mM_X^{\dag}(v_3)^{\circ}). 
	\end{align*}
	(ii) Suppose that the diagram (\ref{dia:Xext:dagA})
	restricts to the diagram
	\begin{align}\notag
		\xymatrix{
			(\ev_3^{X, \ddag})^{-1}(\mM_X^{\ddag}(v_3)^{\circ})
			\ar[r]^-{\ev_3^{X, \ddag}} \ar[d]_-{(\ev_2^{X, \ddag}, \ev_1^{X, \ddag})} 
			& \mM_X^{\ddag}(v_3)^{\circ} \\
			\mM_X(v_2)^{\circ} \times \mM_X^{\dag}(v_1)^{\circ}. & 
		}
	\end{align}
	Then the functor (\ref{FM:extP2}) descends to the functor 
	\begin{align*}
		\dDT^{\C}(\mM_X(v_2)^{\circ})
		\otimes \dDT^{\C}(\mM_X^{\dag}(v_1)^{\circ}) \to \dDT^{\C}(\mM_X^{\dag}(v_3)^{\circ}). 
	\end{align*}
\end{prop}

The following is an analogue of Theorem~\ref{thm:COHA}
for moduli stacks of pairs, which is essentially 
proved in~\cite{THtype}. 
\begin{thm}\label{thm:PTfunct}
	We take $(v_1, v_2) \in N_{\le 1}(S)^{\times 2}$
	for $v_i=(\beta_i, n_i)$
	such that $v_2 \in N_{\le 1}(S)_t$. 
	
	(i) The functor (\ref{FM:extP}) 
	descends to the functor 
	\begin{align}\label{func:DTind}
		\dDT^{\C}(\pP_{n_1}^t(X, \beta_1)) \otimes 
		\dDT^{\C}(\mM^H_{n_2}(X, \beta_2)) 
		\to \dDT^{\C}(\pP_{n_1+n_2}^t(X, \beta_1+\beta_2)).
	\end{align}
	(ii) The functor (\ref{FM:extP2}) 
	descends to the functor 
	\begin{align}
		\label{funct:DTind2}
		\dDT^{\C}(\mM^H_{n_2}(X, \beta_2)) \otimes 
		\dDT^{\C}(\pP_{n_1}^t(X, \beta_1)) 
		\to \dDT^{\C}(\pP_{n_1+n_2}^t(X, \beta_1+\beta_2)).
	\end{align}
\end{thm}
\begin{proof}
	By 
	Lemma~\ref{lem:mutstab} below, 
	the diagrams (\ref{dia:Xext:dag}), 
	(\ref{dia:Xext:dag2}) restrict 
	to the diagrams  
	\begin{align}\label{dia:Xext:P}	
		\xymatrix{
			(\ev_3^{X, \dag})^{-1}(\pP_{n_1+n_2}^t(X, \beta_1+\beta_2))
			\ar[r]^-{\ev_3^{X, \dag}} \ar[d]_-{(\ev_1^{X, \dag}, \ev_2^{X, \dag})} 
			& \pP_{n_1+n_2}^t(X, \beta_1+\beta_2) \\
			\pP_{n_1}^t(X, \beta_1) \times \mM_{n_2}^H(X, \beta_2), & 
		} \\
		\notag
		\xymatrix{
			(\ev_3^{X, \ddag})^{-1}(\pP_{n_1+n_2}^t(X, \beta_1+\beta_2))
			\ar[r]^-{\ev_3^{X, \ddag}} \ar[d]_-{(\ev_2^{X, \ddag}, \ev_1^{X, \ddag})} 
			& \pP_{n_1+n_2}^t(X, \beta_1+\beta_2) \\
			\mM_{n_2}^H(X, \beta_2) \times	\pP_{n_1}^t(X, \beta_1). & 
		}
	\end{align}	
	Then we have the desired functors by Proposition~\ref{prop:inducedF}. 
\end{proof}

We have used the following lemma, which is obvious 
from the definition of $\mu_t^{\dag}$-stability
\begin{lem}\label{lem:mutstab}
	(i) Let $0 \to E_1 \to E_3 \to F_2[-1] \to 0$ be an exact sequence in 
	$\aA_X$ such that $\rank(E_1)=\rank(E_3)=1$ and 
	$F_2 \in \Coh_{\le 1}(X)$ satisfies $\mu_H(F_2)=t$. 
	Then $E_3$ is $\mu_t^{\dag}$-semistable if and only if 
	$E_1$ is $\mu_t^{\dag}$-semistable and $F_2$ is $\mu_H$-semistable. 
	
	(ii) 	Let $0 \to F_2[-1] \to E_3 \to E_1 \to 0$ be an exact sequence in 
	$\aA_X$ such that $\rank(E_1)=\rank(E_3)=1$ and 
	$F_2 \in \Coh_{\le 1}(X)$ satisfies $\mu_H(F_2)=t$. 
	Then $E_3$ is $\mu_t^{\dag}$-semistable if and only if 
	$E_1$ is $\mu_t^{\dag}$-semistable and $F_2$ is $\mu_H$-semistable. 
\end{lem}

\subsection{Conjectural semiorthogonal decompositions
via categorified Hall products}\label{subsec:mainconj}
\subsubsection{Stratifications of $\pP_n^t(X, \beta)$}
In the situation of Subsection~\ref{subsec:catDT:stableD026}, 
let us take $t \in \mathbb{R}_{>0} \cup \{\infty\}$.
Then for $t_{\pm}=t\pm \varepsilon$ for $0<\varepsilon \ll 1$, 
Conjectures~\ref{conj:DT/PT}, \ref{conj2} 
claim the existence of a fully-faithful 
functor 
	\begin{align*}
		\dDT^{\C}(P_{n}^{t_-}(X, \beta)) \hookrightarrow 
		\dDT^{\C}(P_n^{t_+}(X, \beta)). 
	\end{align*}
We refine the above conjecture 
using the categorified Hall products for 
DT categories. 
For $(\beta, n) \in N_{\le 1}(S)$ and $t \in \mathbb{R} \cup \{\infty\}$, 
we define $\mathbf{D}_{(\beta, n)}^t$ to be the 
set of decompositions of $(\beta, n)$
\begin{align*}
	\mathbf{D}_{(\beta, n)}^t \cneq \left\{(\beta, n)=(\beta_1, n_1)+(\beta_2, n_2) : (\beta_2, n_2) \in N_{\le 1}(S)_t, 
	\pP_{n_1}^{t}(X, \beta_1) \times \mM_{n_2}^H(X, \beta_2) \neq \emptyset \right\}.
\end{align*}
There is a partial order on 
$\mathbf{D}^t_{(\beta, n)}$ given by 
\begin{align}\label{porder}
	(\beta_1', n_1')+(\beta_2', n_2') > (\beta_1, n_1)+(\beta_2, n_2)
\end{align}
if there is a decomposition 
$(\beta_1, n_1)=(\beta_1'', n_1'')+(\beta_2'', n_2'')$ in 
$\mathbf{D}_{(\beta_1, n_1)}^t$ such that
\begin{align*}
	(\beta_1', n_1')=(\beta_1'', n_1''), \ 
	(\beta_2', n_2')=(\beta_2, n_2)+(\beta_2'', n_2''). 
\end{align*}
It defines the decomposition of $\mathbf{D}_{(\beta, n)}^t$
\begin{align*}
	\mathbf{D}_{(\beta, n)}^t=
	\mathbf{D}_{(\beta, n)}^{t, 1} \amalg
	\cdots 
	\amalg \mathbf{D}_{(\beta, n)}^{t, N}
\end{align*}
defined inductive so that
\begin{align*}
	\mathbf{D}_{(\beta, n)}^{t, i}
	\subset \mathbf{D}_{(\beta, n)}^t                                                                                                                                                                                                                                                        \setminus \bigcup_{j<i} \mathbf{D}_{(\beta, n)}^{t, j}
\end{align*}
is the set of maximal elements with respect to the partial order 
(\ref{porder}).
In the following, we set 
\begin{align*}
	\mathbf{D}_{(\beta, n)}^{t, \le l} \cneq 
	\mathbf{D}_{(\beta, n)}^{t, 1} \amalg
	\mathbf{D}_{(\beta, n)}^{t, 2} \amalg \cdots 
	\amalg \mathbf{D}_{(\beta, n)}^{t, l}. 
\end{align*}
For each $1\le l \le N$, we
define 
\begin{align*}
	\zZ^{t_{+}\us}_{l} &\cneq	\bigcup_{\begin{subarray}{c}(\beta, n)=(\beta_1, n_1)+(\beta_2, n_2) \\
			\in \mathbf{D}_{(\beta, n)}^{t, \le l}
	\end{subarray}}
	\pP_{n_1}^t(X, \beta_1) \ast \mM_{n_2}^H(X, \beta_2), \\
	\zZ^{t_-\us}_{l} &\cneq	\bigcup_{\begin{subarray}{c}(\beta, n)=(\beta_1, n_1)+(\beta_2, n_2) \\
			\in \mathbf{D}_{(\beta, n)}^{t, \le l}
	\end{subarray}}
	\mM_{n_2}^H(X, \beta_2) \ast
	\pP_{n_1}^t(X, \beta_1).	
\end{align*}
Here $\ast$ is defined in (\ref{ast:X}). 
Note that the substacks 
$\zZ^{t_{\pm}\us}_{l}$ are closed substacks 
in $\pP_n^t(X, \beta)$. 
We have the following lemma. 
\begin{lem}\label{lem:strata}
	We have the following identities
	\begin{align}\label{Z:disjoint}
		\zZ^{t_{+}\us}_{l}		&=\coprod_{\begin{subarray}{c}(\beta, n)=(\beta_1, n_1)+(\beta_2, n_2) \\
				\in \mathbf{D}_{(\beta, n)}^{t, \le l}
		\end{subarray}} \pP_{n_1}^{t_{+}}(X, \beta_1) \ast 
		\mM_{n_2}^H(X, \beta_2),
		\\
		\notag
		\zZ^{t_{-}\us}_{l}	&=\coprod_{\begin{subarray}{c}(\beta, n)=(\beta_1, n_1)+(\beta_2, n_2) \\
				\in \mathbf{D}_{(\beta, n)}^{t, \le l}
		\end{subarray}}\mM_{n_2}^H(X, \beta_2) \ast
		\pP_{n_1}^{t_{-}}(X, \beta_1).
	\end{align}
\end{lem}
\begin{proof}
	By the Harder-Narasimhan filtrations in 
	$\mu_{t_{\pm}}^{\dag}$-stability, 
	we have 
	\begin{align}\label{HN:strata}
		\pP_n^t(X, \beta)&=\coprod_{\begin{subarray}{c}(\beta, n)=(\beta_1, n_1)+(\beta_2, n_2) \\
				\in \mathbf{D}_{(\beta, n)}^{t}\end{subarray}}
		\pP_{n_1}^{t_{+}}(X, \beta_1) \ast \mM_{n_2}^H(X, \beta_2) \\
		\notag
		&=\coprod_{\begin{subarray}{c}(\beta, n)=(\beta_1, n_1)+(\beta_2, n_2) \\
				\in \mathbf{D}_{(\beta, n)}^{t}\end{subarray}}
		\mM_{n_2}^H(X, \beta_2) \ast
		\pP_{n_1}^{t_{-}}(X, \beta_1). 
	\end{align}
	In particular the RHS in (\ref{Z:disjoint}) are disjoint unions. 
	Since $\pP_{n}^{t_{\pm}}(X, \beta) \subset \pP_n^t(X, \beta)$, 
	the RHS in (\ref{Z:disjoint}) are contained in the LHS. 
	On the other hand applying (\ref{HN:strata}) for $(\beta_1, n_1)$, 
	we see that the LHS is contained in the RHS. 
\end{proof}

\begin{lem}\label{lem:exP}
	Let $(\beta, n)=(\beta_1, n_1)+(\beta_2, n_2)$ be a decomposition in 
	$\mathbf{D}_{(\beta, n)}^{t, l}$. 
	Then in the diagrams (\ref{dia:Xext:P}), we have 
	\begin{align}\label{ev-1:P}
		&(\ev_3^{X, \dag})^{-1}\left(\pP_{n}^t(X, \beta) \setminus \zZ_{l-1}^{t_{+}\us}\right)
		=(\ev_1^{X, \dag}, \ev_2^{X, \dag})^{-1}
		\left(\pP_{n_1}^{t_+}(X, \beta_1) \times \mM_{n_2}^H(X, \beta_2)
		\right), \\
		\notag
		&(\ev_3^{X, \ddag})^{-1}\left(\pP_{n}^t(X, \beta) \setminus \zZ_{l-1}^{t_{-}\us}\right)
		=(\ev_2^{X, \ddag}, \ev_1^{X, \ddag})^{-1}
		\left(\mM_{n_2}^H(X, \beta_2) \times \pP_{n_1}^{t_-}(X, \beta_1) 
		\right).
	\end{align}
\end{lem}
\begin{proof}
	We only prove the first identity. 
	Let us take an exact sequence 
	in $\aA_X$
	\begin{align}\label{exact:E1F2}
		0 \to E_1 \to E \to F_2[-1] \to 0
	\end{align}
	such that
	$\cl(E_1)=(1, \beta_1, n_1)$ and $F_2 \in \Coh_{\le 1}(X)$ with 
	$\cl(F_2[-1])=(0, \beta_2, n_2)$. 
	If $E$ corresponds to a point in 
	$\pP_n^t(X, \beta) \setminus \zZ^{t_{+}\us}_{l-1}$, then 
	$E_1$ is $\mu_{t_+}^{\dag}$-semistable. 
	Indeed if otherwise, there is an exact sequence 
	\begin{align}\label{exact:E1F2'}
		0 \to E_1'' \to E_1 \to F_2''[-1] \to 0
	\end{align}
	in $\aA_X$ such that $E_1''$ is $\mu_{t_+}^{\dag}$-semistable and 
	$0 \neq F_2'' \in \Coh_{\le 1}(X)$ 
	satisfies $\cl(F_2''[-1])=(0, \beta_2'', n_2'')$, where 
	$(\beta_2'', n_2'') \in N_{\le 1}(S)_t$. 
	Let $(\beta_1', n_1')=(\beta_1'', n_1'')$
	and $(\beta_2', n_2')=(\beta_2, n_2)+(\beta_2'', n_2'')$. 
	Then we have 
	\begin{align*}
		(\beta_1', n_1')+(\beta_2', n_2') >(\beta_1, n_1)+(\beta_2, n_2), 
	\end{align*}
	so the LHS is contained in $\mathbf{D}_{(\beta, n)}^{t, \le l-1}$. 
	By combining exact sequences (\ref{exact:E1F2}), (\ref{exact:E1F2'}), 
	we obtain the exact sequences in $\aA_X$ and $\Coh_{\le 1}(X)$
	\begin{align*}
		0 \to E_1' \to E \to F_2'[-1] \to 0, \ 
		0 \to F_2 \to F_2' \to F_2'' \to 0
	\end{align*}
	where $E_1'=E_1''$ so that $\cl(E_1')=(1, \beta_1', n_1')$ 
	and $\cl(F_2'[-1])=(0, \beta_2', n_2')$. 
	This contradicts to 
	the assumption that 
	$E$ does not correspond to a point in $\zZ^{t_{+}\us}_{l-1}$. 
	Therefore the LHS of (\ref{ev-1:P}) is contained in the RHS. 
	
	Conversely suppose that an exact sequence (\ref{exact:E1F2}) is given 
	such that $E_1$ is $\mu_{t_+}$-semistable and $F_2$ is $\mu_H$-semistable. 
	Suppose by contradiction that $E \in \zZ^{t_{+}\us}_{l-1}$. 
	Since (\ref{exact:E1F2}) is a Hadar-Narasimhan filtration in $\mu_{t_+}$-stability, 
	the uniqueness of Harder-Narasimhan filtration together with Lemma~\ref{lem:strata} imply that 
	$(\beta, n)=(\beta_1, n_1)+(\beta_2, n_2)$ is a decomposition in $\mathbf{D}_{(\beta, n)}^{t, \le l-1}$.
	This is a contradiction, so the RHS of (\ref{ev-1:P}) is contained in the LHS. 
\end{proof}

\subsubsection{Conjectures}

For each $1\le l \le N$, 
let $(\beta, n)=(\beta_1, n_1)+(\beta_2, n_2)$ be a decomposition in 
$\mathbf{D}_{(\beta, n)}^{t, l}$.
By Proposition~\ref{prop:inducedF}
and Lemma~\ref{lem:exP}, 
the functors (\ref{func:DTind}), (\ref{funct:DTind2}) descend to functors
\begin{align}\label{descend}
	&\ast \colon 
	\dDT^{\C}(\pP_{n_1}^{t_+}(X, \beta_1)) \otimes 
	\dDT^{\C}(\mM_{n_2}^H(X, \beta_2))
	\to \dDT^{\C}\left(\pP^t_{n}(X, \beta) \setminus \zZ^{t_{+}\us}_{l-1} \right), \\
	\notag	&\ast \colon \dDT^{\C}(\mM_{n_2}^H(X, \beta_2))
	\otimes \dDT^{\C}(\pP_{n_1}^{t_-}(X, \beta_1))
	\to \dDT^{\C}\left(\pP^t_{n}(X, \beta) \setminus \zZ^{t_{-}\us}_{l-1} \right). 
\end{align}

By Remark~\ref{rmk:DTtwist}, we have the decomposition 
into $j$-twisted parts
\begin{align*}
	\dDT^{\C}(\mM_{n}^H(X, \beta))
	=\bigoplus_{j \in \mathbb{Z}}\dDT^{\C}(\mM_{n}^H(X, \beta))_j. 
	\end{align*}
We denote by $i_j$, $\mathrm{pr}_j$ the
inclusion, projection with respect to the above decomposition 
\begin{align*}
	\xymatrix{
		\dDT^{\C}(\mM_{n}^H(X, \beta))_j
\ar@<0.5ex>[r]^-{i_j} & \ar@<0.5ex>[l]^-{\mathrm{pr}_j}
	\dDT^{\C}(\mM_{n}^H(X, \beta)).
}
\end{align*}
We propose the following conjecture. 
\begin{conj}\label{conj:SOD:l}
	For each $1\le l \le N$
	let $\mathbf{d} \in \mathbf{D}_{(\beta, n)}^{t, l}$ be 
	a decomposition $(\beta, n)=(\beta_1, n_1)+(\beta_2, n_2)$, and 
	take $k^{\pm}(\mathbf{d}) \in \mathbb{R}$. 
	Then for  
	$j \in \mathbb{Z}$, the
	functors (\ref{descend}) composed with $i_j$
	\begin{align}\label{funct:rest}
		&\ast_j \cneq \ast \circ i_j \colon 
		\dDT^{\C}(\pP_{n_1}^{t_+}(X, \beta_1)) \otimes 
		\dDT^{\C}(\mM_{n_2}^H(X, \beta_2))_j
		\to \dDT^{\C}\left(\pP^t_{n}(X, \beta) \setminus \zZ^{t_{+}\us}_{l-1} \right), \\
		\notag
		&\ast_j \cneq \ast \circ i_j \colon \dDT^{\C}(\mM_{n_2}^H(X, \beta_2))_j
		\otimes \dDT^{\C}(\pP_{n_1}^{t_-}(X, \beta_1))
		\to \dDT^{\C}\left(\pP^t_{n}(X, \beta) \setminus \zZ^{t_{-}\us}_{l-1} \right)
	\end{align}
	are fully-faithful. 
	By setting $\Upsilon^{\pm, \mathbf{d}}_j$ to be the essential images of the above functors, 
	we have the semiorthogonal decompositions
	\begin{align}\label{conj:sodv}
		\dDT^{\C}\left(\pP^t_{n}(X, \beta) \setminus \zZ^{t_{+}\us}_{l-1} \right)
		&=\left\langle 	 \bigoplus_{\mathbf{d} \in \mathbf{D}_{(\beta, n)}^{t, l}}\Upsilon^{+, \mathbf{d}}_{>k^+(\mathbf{d})}, 
		\wW^{+}_{l, k^+}, 	\bigoplus_{\mathbf{d} \in \mathbf{D}_{(\beta, n)}^{t, l}}\Upsilon^{+, \mathbf{d}}_{ \le k^+(\mathbf{d})}
		\right\rangle, \\	
		\notag
		\dDT^{\C}\left(\pP^t_{n}(X, \beta) \setminus \zZ^{t_{-}\us}_{l-1} \right)
			&=\left\langle 	 \bigoplus_{\mathbf{d} \in \mathbf{D}_{(\beta, n)}^{t, l}}\Upsilon^{-, \mathbf{d}}_{<k^-(\mathbf{d})}, 
		\wW^{-}_{l, k^-}, 	\bigoplus_{\mathbf{d} \in \mathbf{D}_{(\beta, n)}^{t, l}}\Upsilon^{-, \mathbf{d}}_{ \ge k^-(\mathbf{d})}
		\right\rangle	  
	\end{align}
	together with semiorthogonal decompositions 
	\begin{align}\label{sod:Phi}
		&\Upsilon^{+, \mathbf{d}}_{>k^+(\mathbf{d})}
		=\langle \ldots, \Upsilon^{+, \mathbf{d}}_{\lfloor k^+(\mathbf{d}) \rfloor +2}, \Upsilon^{+, \mathbf{d}}_{\lfloor k^+(\mathbf{d}) \rfloor +1} \rangle, \ 
		\Upsilon^{+, \mathbf{d}}_{\le k^+(\mathbf{d})}
		=\langle \Upsilon^{+, \mathbf{d}}_{\lfloor k^+(\mathbf{d}) \rfloor}, \Upsilon^{+, \mathbf{d}}_{\lfloor k^+(\mathbf{d}) \rfloor -1}, \ldots \rangle, \\
		\notag	&\Upsilon^{-, \mathbf{d}}_{<k^-(\mathbf{d})}
		=\langle \ldots, \Upsilon^{-, \mathbf{d}}_{\lceil k^-(\mathbf{d}) \rceil -2}, \Upsilon^{-, \mathbf{d}}_{\lceil k^-(\mathbf{d}) \rceil -1} \rangle, \ 
		\Upsilon^-_{\ge k^-(\mathbf{d})}
		=\langle \Upsilon^{-, \mathbf{d}}_{\lceil k^-(\mathbf{d}) \rceil}, \Upsilon^{-, \mathbf{d}}_{\lceil k^-(\mathbf{d}) \rceil +1}, \ldots \rangle
	\end{align}
	such that the composition functors
	\begin{align}\label{conj:compose}
		\wW^{\pm}_{l, k^{\pm}} \hookrightarrow 
		\dDT^{\C}\left(\pP^t_{n}(X, \beta) \setminus \zZ^{t_{\pm}\us}_{l-1} \right)
		\twoheadrightarrow 
		\dDT^{\C}\left(\pP^t_{n}(X, \beta) \setminus \zZ^{t_{\pm}\us}_{l} \right)
	\end{align}
	are equivalences. 
\end{conj}


The following conjecture 
holds if Conjecture~\ref{conj:SOD:l} holds for all $1\le l \le N$.

\begin{conj}\label{conj:SOD}
	Let $k^{\pm} \colon \mathbf{D}_{(\beta, n)}^t \to \mathbb{R}$ be maps. 
	There exist semiorthogonal decompositions
	\begin{align*}
		\mathcal{DT}^{\C}(\pP_n^t(X, \beta))=
		\left\langle \bigoplus_{\mathbf{d} \in \mathbf{D}_{(\beta, n)}^{t, 1}}
		\Upsilon^{+, \mathbf{d}}_{>k^+(\mathbf{d})}, \bigoplus_{\mathbf{d} \in \mathbf{D}_{(\beta, n)}^{t, 2}}
		\Upsilon^{+, \mathbf{d}}_{>k^+(\mathbf{d})}, \ldots, \bigoplus_{\mathbf{d} \in \mathbf{D}_{(\beta, n)}^{t, N}}
		\Upsilon^{+, \mathbf{d}}_{>k^+(\mathbf{d})}, \wW^{+}_{k^+}, \right. \\
		\left.\bigoplus_{\mathbf{d} \in \mathbf{D}_{(\beta, n)}^{t, N}}
		\Upsilon^{+, \mathbf{d}}_{\le k^+(\mathbf{d})}, \bigoplus_{\mathbf{d} \in \mathbf{D}_{(\beta, n)}^{t, N-1}}
		\Upsilon^{+, \mathbf{d}}_{\le k^+(\mathbf{d})}, \ldots, \bigoplus_{\mathbf{d} \in \mathbf{D}_{(\beta, n)}^{t, 1}}
		\Upsilon^{+, \mathbf{d}}_{\le k^+(\mathbf{d})}
		\right\rangle \\
		=\left\langle \bigoplus_{\mathbf{d} \in \mathbf{D}_{(\beta, n)}^{t, 1}}
		\Upsilon^{-, \mathbf{d}}_{<k^-(\mathbf{d})}, \bigoplus_{\mathbf{d} \in \mathbf{D}_{(\beta, n)}^{t, 2}}
		\Upsilon^{-, \mathbf{d}}_{<k^-(\mathbf{d})}, \ldots, \bigoplus_{\mathbf{d} \in \mathbf{D}_{(\beta, n)}^{t, N}}
		\Upsilon^{-, \mathbf{d}}_{<k^-(\mathbf{d})},
		\wW^{-}_{k^-}, \right. \\
		\left.\bigoplus_{\mathbf{d} \in \mathbf{D}_{(\beta, n)}^{t, N}}
		\Upsilon^{-, \mathbf{d}}_{\ge k^-(\mathbf{d})}, \bigoplus_{\mathbf{d} \in \mathbf{D}_{(\beta, n)}^{t, N-1}}
		\Upsilon^{-, \mathbf{d}}_{\ge k^-(\mathbf{d})}, \ldots, \bigoplus_{\mathbf{d} \in \mathbf{D}_{(\beta, n)}^{t, 1}}
		\Upsilon^{-, \mathbf{d}}_{\ge k^-(\mathbf{d})}
		\right\rangle		
	\end{align*}
	with semiorthogonal decomposition as in (\ref{sod:Phi}) and equivalences induced by categorified 
	Hall products 
	\begin{align*}
		&\ast_j \colon \mathcal{DT}^{\C}(P_{n_1}^{t_+}(X, \beta_1))
		\otimes \mathcal{DT}^{\C}(\mM_{n_2}^H(X, \beta_2))_{j}
		\stackrel{\sim}{\to}
		\Upsilon^{+, \mathbf{d}}_{j}, \\
		&\ast_j \colon \mathcal{DT}^{\C}(\mM_{n_2}^H(X, \beta_2))_{j}
		\otimes \mathcal{DT}^{\C}(P_{n_1}^{t_-}(X, \beta_1)) \stackrel{\sim}{\to}
		\Upsilon^{-, \mathbf{d}}_{j}	
	\end{align*}
	such that the composition functors
	\begin{align}\notag
		\wW^{\pm}_{k^{\pm}}
		\hookrightarrow \mathcal{DT}^{\C}(\pP_n^t(X, \beta))
		\to \mathcal{DT}^{\C}(P_n^{t_{\pm}}(X, \beta))
	\end{align}
	are equivalences. 
\end{conj}

In addition to the above conjecture, we also propose 
the following conjecture
which in particular implies Conjecture~\ref{conj2}. 
\begin{conj}\label{conj:FF}
	In the situation of Conjecture~\ref{conj:SOD}, 
	we set 
	\begin{align}\label{kpm}
		k^+(\mathbf{d})=\frac{1}{2}(n_2+\beta_1 \beta_2), \ 
		k^-(\mathbf{d})=-\frac{1}{2}\beta_1 \beta_2		
	\end{align}
where $\mathbf{d}$ is a decomposition $(\beta, n)=(\beta_1, n_1)+(\beta_2, n_2)$. 
Then 
	we have 
	$\wW^{-}_{k^-} \subset \wW^{+}_{k^+}$. 
	In particular, we have a fully-faithful functor 
	\begin{align*}
		\mathcal{DT}^{\C}(P_n^{t_-}(X, \beta))
		\hookrightarrow 
		\mathcal{DT}^{\C}(P_n^{t_+}(X, \beta)).
	\end{align*}
\end{conj}

We say that $t \in \mathbb{R}$ is a \textit{simple wall}
if $N=1$. In this case, we also propose the 
description of semiorthogonal complements of the fully-faithful 
embedding in Conjecture~\ref{conj:FF}. 

\begin{conj}\label{conj:simple}
	If $t \in \mathbb{R}_{>0}$ is a simple wall, 
	there exists a semiorthogonal decomposition 
	\begin{align}\label{SOD:simple}
		&\mathcal{DT}^{\C}(P_n^{t_+}(X, \beta)) \\
		\notag	&=\left \langle  \bigoplus_{\mathbf{d} \in \mathbf{D}_{(\beta, n)}^{t, 1}}
		\Upsilon^{-, \mathbf{d}}_{[-\frac{1}{2}\beta_1 \beta_2-\frac{1}{2}n_2, 
			-\frac{1}{2}\beta_1 \beta_2 )}, \mathcal{DT}^{\C}(P_n^{t_-}(X, \beta)), 
		\bigoplus_{\mathbf{d} \in \mathbf{D}_{(\beta, n)}^{t, 1}}
		\Upsilon^{-, \mathbf{d}}_{[-\frac{1}{2}\beta_1 \beta_2, 
			-\frac{1}{2}\beta_1 \beta_2+\frac{1}{2}n_2 )}
		\right \rangle. 	
	\end{align}
	Here $\mathbf{d}$ is a decomposition $(\beta, n)=(\beta_1, n_1)+(\beta_2, n_2)$ 
	in $\mathbf{D}_{(\beta, n)}^{t, 1}$, and 
	$\Upsilon_{[a, b)}^{-, \mathbf{d}}$
	admits semiorthogonal decomposition 
	\begin{align*}
		\Upsilon_{[a, b)}^{-, \mathbf{d}}=\left\langle \Upsilon_{\lceil a \rceil}^{-, \mathbf{d}}, \Upsilon_{\lceil a \rceil +1}^{-, \mathbf{d}}, 
		\ldots, \Upsilon_{\lceil b \rceil -1}^{-, \mathbf{d}} \right\rangle. 
	\end{align*} 
\end{conj}

\begin{rmk}\label{rmk:t=infty}
	The above Conjecture~\ref{conj:SOD}, Conjecture~\ref{conj:FF}
	for $t=\infty$ implies the semiorthogonal 
	decomposition in Theorem~\ref{conj:intro:dtpt} in the introduction. 
	Indeed let $n(\beta) \in \mathbb{Z}$ be as in (\ref{nbeta}).
	Then $\mathbf{D}_{(\beta, n)}^{\infty, l}$ consists of one element
	\begin{align}\notag
		\mathbf{D}_{(\beta, n)}^{\infty, l}=
		\left\{ (\beta, n)=(\beta, n(\beta)+l-1)+(0, n-n(\beta)-l+1)  \right\}. 
	\end{align} 
	In particular $t=\infty$ is a simple wall if and only if $n=n(\beta)+1$. 
	Note that 
	\begin{align*}
		P_n(X, \beta)=I_n(X, \beta), \ 
		n \le n(\beta),
		\end{align*}
	indeed they are empty for $n<n(\beta)$ by the definition of $n(\beta)$. 
		For $t=\infty$ and $n=n(\beta)+1$, 
	the semiorthogonal decomposition (\ref{SOD:simple})
	is 
		\begin{align}\notag
		\dDT^{\C}(I_n(X, \beta))
		=\langle 	\dDT^{\C}(P_n(X, \beta)), \dDT^{\C}(\mM_1(X))_0 \otimes 
		\dDT^{\C}(P_{n-1}(X, \beta))  \rangle,
	\end{align}
	where $\dDT^{\C}(\mM_1(X))_0$ is equivalent to $\Dbc(S)$. 
\end{rmk}

\subsubsection{The formally local descriptions of $\pP_n^t(X, \beta)$}
Let $\sS_l^{t_{\pm}}$ be substacks in $\pP_n^t(X, \beta)$ 
defined by 
\begin{align*}
	\sS_l^{t_{\pm}} \cneq 
	\zZ^{t_{\pm}\us}_{l} \setminus 
	\zZ^{t_{\pm}\us}_{l-1} \subset \pP_n^t(X, \beta). 
\end{align*}
Then we have the stratifications 
\begin{align}\label{strata:Pt}
	\pP_n^t(X, \beta)=\sS_{1}^{t_{\pm}} \sqcup 
	\sS_2^{t_{\pm}} \sqcup \cdots \sqcup \sS_N^{t_{\pm}} \sqcup \pP_n^{t_{\pm}}(X, \beta).
\end{align}
Here we interpret the above stratification in terms of KN stratifications 
formally locally on good moduli spaces of $\pP_n^t(X, \beta)$. 

Let us consider the good moduli space morphism
\begin{align*}
	\pi_{\pP} \colon 
	\pP_n^t(X, \beta) \to P_n^t(X, \beta). 
\end{align*}
For a closed point $p \in P_n^t(X, \beta)$, we consider 
its formal fiber
\begin{align*}
	\xymatrix{
		\widehat{\pP}_n^t(X, \beta)_p \ar[r] \ar[d] \diasquare
		& \pP_n^t(X, \beta) \ar[d] \\
		\Spec \widehat{\oO}_{P_n(X, \beta), p} \ar[r] & 
		P_n^t(X, \beta). 
	}
\end{align*}
As in~\cite{MR3811778} (also see~\cite[Theorem~9.11]{Toddbir} and Section~\ref{subsec:formal}) we can describe the 
above formal fiber
in terms of Ext-quivers with super-potentials 
associated with polystable objects
as follows. 
A closed point $p$ is represented by a $\mu_t^{\dag}$-polystable object of the form
\begin{align}\label{eE:polystable}
	\eE=\eE_0 \oplus \bigoplus_{i=1}^k W_i \otimes E_i[-1]
\end{align}
where $\eE_0$ is a rank one $\mu_t^{\dag}$-stable object in $\aA_X$, 
$(E_1, \ldots, E_k)$ are mutually non-isomorphic $\mu_H$-stable sheaves
in $\Coh_{\le 1}(X)$ with $\mu_H(E_i)=t$ and $W_i$ is a finite dimensional vector space. 
Let $Q_{\eE}$ be the Ext-quiver
associated with the collection 
\begin{align*}
	(\eE_0, E_1[-1], \ldots, E_k[-1]), 
\end{align*}
i.e. 
the set of vertices is $\{0, 1, \ldots, k\}$
and the number of arrows from $i$ to $j$ is $\ext^1(E_i, E_j)$ where 
we have set $E_0=\eE_0[1]$. 
Let $Q_{\eE, 0} \subset Q_{\eE}$
be the full sub quiver as in Subsection~\ref{KN:quiver}
whose vertex set is $\{1, \ldots, k\}$, which is 
the Ext-quiver for the collection $(E_1, \ldots, E_k)$. 
Let $\vec{w}=(\dim W_i)_{1\le i\le k}$ be the dimension 
vector for $Q_{\eE, 0}$, and set 
$G_{\eE}=\prod_{i=1}^k \GL(W_i)$. 
As in Subsection~\ref{KN:quiver}, we 
have the moduli stack of $Q_{\eE}$-representations 
with dimension vector $(1, \vec{w})$, 
given by the quotient stack together with its good moduli space
\begin{align}\label{gmoduli:R}
	\left[ R^{\dag}(Q_{\eE, 0}, \vec{w})/G_{\eE}   \right] \to 
	R^{\dag}(Q_{\eE, 0}, \vec{w}) \ssslash G_{\eE}. 
\end{align}
Then 
by~\cite[Theorem~9.11]{Toddbir},
there is a $G_{\eE}$-equivariant function 
on the formal fiber of (\ref{gmoduli:R}) at $0$
\begin{align*}
	w_{\eE} \colon \widehat{R}^{\dag}(Q_{\eE, 0}, \vec{w})_0
	\to \mathbb{C}
\end{align*}
such that we have the following commutative isomorphisms
\begin{align}\label{PT:formal}
	\xymatrix{
		\left[ \mathrm{Crit}(w_{\eE})/G_{\eE} \right]	\ar[r]^-{\cong} 
		\ar[d] & \widehat{\pP}_n^t(X, \beta)_p \ar[d] \\
		\mathrm{Crit}(w_{\eE}) \ssslash G_{\eE}
		\ar[r]^-{\cong} &
		\Spec \widehat{\oO}_{P_n(X, \beta), p}.
	}
\end{align}
We have the following lemma: 

\begin{lem}\label{lem:Pstrata}
	The stratifications (\ref{strata:Pt})
	pulled back via the top arrow in (\ref{PT:formal}) is the 
	KN stratification on $R^{\dag}(Q_{\eE, 0}, \vec{w})$
	restricted to $\Crit(w_{\eE})$ 
	with respect to 
	characters $\chi_0^{\pm}$, where $\chi_0$ is the determinant 
	character 
	\begin{align}\notag
		\chi_0 \colon G_{\eE}=\prod_{i=1}^k \GL(W_i) \to \C, \ 
		(g_i)_{1\le i \le k} \mapsto \prod_{i=1}^k \det (g_i).
	\end{align}
\end{lem}
\begin{proof}
	We only prove the lemma for $t_+$. 
	We first note that, on the formal fiber at $p \in P_n^t(X, \beta)$
	we can restrict decompositions $(\beta, n)=(\beta_1, n_1)+(\beta_2, n_2)$
	in $\mathbf{D}_{(\beta, n)}^{t}$ such that $(\beta_2, n_2)$ is 
	of the form 
	\begin{align}\label{restrict:beta}
		(\beta_2, n_2)=\sum_{i=1}^k w_{i}' \cdot [\pi_{\ast}E_i], \ 
		0 \le w_i' \le \dim W_i. 
	\end{align}
	Therefore after discarding strata which 
	do not intersect with $\widehat{\pP}_n^t(X, \beta)_p$
	and renumbering, 
	the stratification (\ref{strata:Pt}) restricts to the filtration 
	\begin{align*}
		\widehat{\pP}_n^t(X, \beta)_p=\sS_{0, p}^{t_{+}} \sqcup 
		\sS_{1, p}^{t_{+}} \sqcup \cdots \sqcup \sS_{\lvert\vec{w} \rvert-1, p}^{t_{+}} 
		\sqcup \widehat{\pP}_n^{t_{+}}(X, \beta)_p.
	\end{align*}
	Here $\cup_{i\le l} \sS_{i, p}^{t_{\pm}}$ 
	correspond to $\mu_t^{\dag}$-semistable object $\eE'$ 
	which admits an exact sequence
	\begin{align*}
		0 \to \eE_1' \to \eE' \to E_2'[-1] \to 0, \ 
		E_2' \in \Coh_{\le 1}(X)
	\end{align*}
	where $(\beta_2, n_2)=[\pi_{\ast}E_2']$
	is of the form 
	(\ref{restrict:beta}) such that 
	$\sum_{i=1}^k w_{i}' \ge \lvert \vec{w} \rvert -l$. 
	
	Under the top isomorphism (\ref{PT:formal}), 
	the above $\eE'$ corresponds to a $Q_{\eE}$-representation 
	$T$ which admits an exact sequence 
	\begin{align*}
		0 \to T_1 \to T \to T_2 \to 0
	\end{align*}
	such that $T_2$ is a $Q_{\eE, 0}$-representation 
	with $\lvert \mathbf{dim}(T_2) \rvert \ge \lvert \vec{w} \rvert -l$. 
	A $Q_{\eE}$-representation $T$ admits an exact sequence 
	as above if and only if it lies in 
	$\cup_{i\le l}S_{i}^+$, 
	where 
	$S_i^{+} \subset R^{\dag}(Q_{\eE, 0}, \vec{w})$
	is a KN strata with respect to $\chi_0$ given in Lemma~\ref{lem:KN}. 
	Therefore the lemma holds. 
\end{proof}

\subsection{Proofs of Conjectures~\ref{conj:SOD:l}, \ref{conj:SOD}, \ref{conj:FF}, \ref{conj:simple}}\label{subsec:proof}
In this section, we prove the conjectures in the previous section 
under Assumption~\ref{assum}. 
\subsubsection{Assumption}
Let us take $t \in \mathbb{R}_{>0} \cup \{\infty\}$ and
$(\beta, n) \in N_{\le 1}(S)$. 
Below we impose the following assumption: 
\begin{assum}\label{assum}
	For each decomposition $(\beta, n)=(\beta_1, n_1)+(\beta_2, n_2)$
	in $\mathbf{D}_{(\beta, n)}^{t}$, we assume that the morphisms
	(\ref{mor:p0dag}), (\ref{pi:ast})
	restrict to the morphisms
	\begin{align}\label{assum:push}
		\pi_{\ast}^{\dag} \colon 
		\pP_{n_1}^t(X, \beta_1) \to \pP_{n_1}^t(S, \beta_1), \ 
		\pi_{\ast} \colon \mM_{n_2}^{H\st}(X, \beta_2) \to 
		\mM_{n_2}^{H\st}(S, \beta_2). 
	\end{align}
\end{assum}
By Lemma~\ref{lem:shiftP} and Lemma~\ref{lem:open:Sss}, the 
above assumption implies that 
\begin{align*}
	\pP_{n_1}^t(X, \beta)=t_0(\Omega_{\fP_{n_1}^t(S, \beta_1)}[-1]), \ 
	\mM_{n_2}^H(X, \beta_2)=t_0(\Omega_{\fM_{n_2}^H(S, \beta_2)}). 
\end{align*}
Therefore we have equivalences
\begin{align}\label{shift:MP}
	\dDT^{\C}(\pP_{n_1}^t(X, \beta_1))
	\stackrel{\sim}{\to}\Dbc(\fP_{n_1}^t(S, \beta_1)), \ 
	\dDT^{\C}(\mM_{n_2}^H(X, \beta_2)) \stackrel{\sim}{\to}
	\Dbc(\fM_{n_2}^{H}(S, \beta_2)). 
\end{align}
Moreover as we have open immersions $P_{n_1}^{t_{\pm}}(X, \beta_2) 
\subset \pP_{n_1}^t(X, \beta_1)$, we have 
\begin{align}\label{assume:DTcat}
	\dDT^{\C}(P_{n_1}^{t_{\pm}}(X, \beta_1)) \stackrel{\sim}{\to}
	\Dbc(\fP_{n_1}^t(S, \beta_1))/\cC_{\zZ^{t_{\pm}\us}},
\end{align}
where $\zZ^{t_{\pm}\us}$ is the complement of 
$P_{n_1}^{t_{\pm}}(X, \beta_1)$ in 
$\pP_{n_1}^{t}(X, \beta_1)$.
\begin{rmk}
	In Assumption~\ref{assum}, we don't impose the 
	first condition in (\ref{assum:push}) for $t=t_{\pm}$, 
	so (\ref{assume:DTcat}) are not necessary equivalent to 
	$\dDT^{\C}(\fP_{n_1}^{t_{\pm}}(S, \beta_1))$. 
\end{rmk}
Note that the second condition in (\ref{assum:push}) is 
satisfied if $\beta$ is a reduced class by Lemma~\ref{lem:reduced}. 
We discuss the cases where the first condition in (\ref{assum:push})
is satisfied.  
We prepare the following two lemmas:  
\begin{lem}\label{lem:vanish}
	For each effective class $\beta \in \NS(S)$, 
	there exist $m(\beta) \in \mathbb{Q}$ such that 
	for any $H$-semistable sheaf $E \in \Coh_{\le 1}(X)$
	with $l(\pi_{\ast}E) \le \beta$ and $\mu_H(E)>m(\beta)$, we have 
	$H^{>0}(X, E)=0$. 	
\end{lem}
\begin{proof}
	For each $0\le i<\beta \cdot H$ and $k\in \mathbb{Z}$, we have the 
	isomorphism of stacks
	\begin{align*}
		\otimes \oO_X(k \pi^{\ast}H) \colon 
		\mM_i^H(X, \beta) \stackrel{\cong}{\to}
		\mM_{i+kH \cdot \beta}^H(X, \beta). 
	\end{align*}
	Since $\mM_i^H(X, \beta)$ is of finite type, 
	there exist $k(i) \in \mathbb{Z}$ such that 
	$H^{>0}(X, E \otimes \oO_X(k \pi^{\ast}H))=0$ for $k\ge k(i)$
	and $[E] \in \mM_i^H(X, \beta)$. 
	By setting 
	\begin{align*}
		m'(\beta) \cneq 
		\mathrm{max}
		\left\{ k(i)+\frac{i}{\beta \cdot H} : 0\le i < \beta \cdot H \right\}, \
		m(\beta) \cneq \mathrm{max}\left\{ m'(\beta') : 0<\beta' \le \beta \right\}
	\end{align*}
	the lemma holds. 
\end{proof}

\begin{lem}\label{lem:push}
	For an effective reduced class $\beta \in \NS(S)$, 
	let $m(\beta) \in \mathbb{Q}$ be defined 
	in Lemma~\ref{lem:vanish}. 
	Then for $t \in (m(\beta), \infty ]$, 
	the morphism (\ref{mor:p0dag}) restricts to the morphism 
	\begin{align*}
		\pi_{\ast}^{\dag} \colon	\pP_n^t(X, \beta) \to \pP_n^t(S, \beta). 
	\end{align*}
\end{lem}
\begin{proof}
	Let $\eE \in \aA_X$ be a rank one $\mu_t^{\dag}$-semistable object
	with $\cl(\eE)=(1, \beta, n)$. 
	Then we have the exact sequence 
	\begin{align*}
		0 \to \hH^0(\eE) \to \eE \to \hH^1(\eE)[-1] \to 0
	\end{align*}
	in $\aA_X$, where $\hH^1(\eE) \in \Coh_{\le 1}(X)$. 
	The $\mu_t^{\dag}$-stability of $\eE$ implies that 
	any Harder-Narasimhan factor $T$ of $\hH^1(\eE)$ with respect to 
	$\mu_H$-stability  
	satisfies $\mu_H(T) \ge t$. 
	From the definition of $m(\beta)$, we have 
	$H^1(X, \hH^1(\eE))=0$. 
	On the other hand as $\hH^0(\eE)$ is a rank one torsion free
	sheaf, it is the ideal sheaf $I_C$ for a compactly supported 
	closed subscheme $C \subset X$ with $\dim C \le 1$. The composition
	\begin{align*}
		\hH^1(\eE)[-2] \to \hH^0(\eE)=I_C \hookrightarrow \oO_{\overline{X}}
	\end{align*}
	vanishes by the Serre duality and the vanishing 
	$H^1(X, \hH^1(\eE))=0$. 
	So the morphism $\hH^0(\eE) \hookrightarrow \oO_X$
	factors through $\eE \to \oO_{\overline{X}}$. By taking the cones, 
	we obtain exact sequences in $\aA_X$ and $\Coh_{\le 1}(X)$
	\begin{align*}
		0 \to F[-1] \to \eE \to \oO_{\overline{X}} \to 0, \ 
		0 \to \oO_C \to F \to \hH^1(\eE) \to 0.
	\end{align*}
	Therefore $\eE$ is isomorphic to the two term complex
	$(s \colon \oO_{\overline{X}} \to F)$, and the morphism 
	$\pi_{\ast}^{\dag}$ in (\ref{mor:p0dag}) 
	sends it to $(\pi_{\ast}s \colon \oO_S \to \pi_{\ast}F)$. 
	
    Below we show that $(\pi_{\ast}s \colon \oO_S \to \pi_{\ast}F)$ is $\mu_t^{\dag}$-semistable
    in $\aA_S$. 
    Namely for any exact sequence in $\aA_S$
    \begin{align*}
    	0 \to (W_1 \otimes \oO_S \to F_1) \to 
    	(\oO_S \to \pi_{\ast}F) \to (W_2 \otimes \oO_S \to F_2) \to 0
    	\end{align*}
    we show the inequality 
    \begin{align}\label{ineq:IAS}
    	\mu_t^{\dag}(I_1) \le t \le \mu_t^{\dag}(I_2), \ 
    	I_i \cneq (W_i \otimes \oO_S \to F_i). 
    	\end{align}
    Note that $(W_1, W_2)$ is either $(0, \mathbb{C})$ or $(\mathbb{C}, 0)$. 
    	By the $\mu_t^{\dag}$-stability of $\eE$,
	we see that any Harder-Narasimhan factor $F'$ of $F$
	with respect to $\mu_H$-stability 
	satisfies $\mu_H(F') \le t$.  
	As $\beta$ is reduced, $\pi_{\ast}F'$ is also $\mu_H$-semistable
	with $\mu_H(\pi_{\ast}F') \le t$ by Lemma~\ref{lem:reduced}.
	Therefore (\ref{ineq:IAS})
	holds when $(W_1, W_2)=(0, \mathbb{C})$. 
		Also by pushing forward the exact sequence 
	$\oO_{\overline{X}} \to F \to \Cok(s) \to 0$, we 
	obtain the exact sequence
	\begin{align*}
		\oO_{S} \to \pi_{\ast}F \to \pi_{\ast}\Cok(s) \to 0. 
		\end{align*}
	Therefore we have the surjection 
	$\Cok(\pi_{\ast}s) \twoheadrightarrow \pi_{\ast}\Cok(s)$. 
	We take the exact sequence 
	\begin{align*}
		0 \to Q \to \Cok(\pi_{\ast}s) \to \pi_{\ast}\Cok(s) \to 0.
	\end{align*}
Since $\beta$ is reduced, 
 both of $\pi_{\ast}\Cok(s)$ and $\Cok(\pi_{\ast}s)$ has the same 
one dimensional reduced supports, so 
$Q$ must be zero dimensional. 
	As $\Cok(s)=\hH^1(\eE)$, it follows that any 
	Harder-Narasimhan factor $T'$ of $\Cok(\pi_{\ast}s)$ with respect to $\mu_H$-stability 
	satisfies $\mu_H(T') \ge t$. 	
	Therefore (\ref{ineq:IAS})
	holds when $(W_1, W_2)=(\mathbb{C}, 0)$. 
\end{proof}

For example, Assumption~\ref{assum} holds in the following cases. 
\begin{exam}\label{exam1}
	Let $\beta$ be a reduced curve class 
	and take $t \in (m(\beta), \infty)$. 
	Then the assumption (\ref{assum}) is satisfied by Lemma~\ref{lem:push}. 
	In this case, the first condition (\ref{assum:push}) 
	also holds for $t=t_{\pm}$
	so we have equivalences
	\begin{align*}
		\dDT^{\C}(P_{n_1}^{t_{\pm}}(X, \beta)) \stackrel{\sim}{\to}
	\Dbc(\fP^{t_{\pm}}_{n_1}(S, \beta_1)).
	\end{align*}
\end{exam}

\begin{exam}\label{exam2}
	Let $\beta$ be a reduced curve class 
	and take $t=\infty$. 
	Then the assumption (\ref{assum}) is satisfied by Lemma~\ref{lem:push}. 
	In this case, the first condition 
	(\ref{assum:push}) is satisfied for $t=t_{-}$, but not satisfied 
	for $t=t_{+}$. So we have 
	\begin{align*}
		\dDT^{\C}(I_n(X, \beta)) \neq 
		\Dbc(\fI_n(S, \beta)), \ 
		\dDT^{\C}(P_n(X, \beta)) \stackrel{\sim}{\to}
		\Dbc(\fP_n(S, \beta))		
		\end{align*}
	where $\fI_n(S, \beta)$ is the derived Hilbert scheme as in Remark~\ref{rmk:MNOP/PT}. 
\end{exam}

\begin{exam}\label{exam3}
	Let $\beta$ be an irreducible curve class. 
	In this case, there is only one wall $t=n/(H\cdot \beta)$,  
	and the assumption (\ref{assum}) is satisfied at this wall. 
	In this case, the first condition (\ref{assum:push})
	is satisfied for $t=t_{+}$, but not satisfied for $t=t_{-}$. 
	So we have 
	\begin{align*}
		\dDT^{\C}(P_n^{t_+}(X, \beta))\stackrel{\sim}{\to}
		\Dbc(\fP_n(S, \beta)), \ 
		\dDT^{\C}(P_n^{t_-}(X, \beta)) \neq 
		\Dbc(\fP_n^{t_-}(S, \beta)). 
		\end{align*}
	Indeed we have $\fP_n^{t_-}(S, \beta)=\emptyset$, but 
	$\dDT^{\C}(P_n^{t_-}(X, \beta))$ is not necessary zero
	(see Section~\ref{subsec:duality}). 
\end{exam}

\subsubsection{The formal local descriptions of $\pP_n^t(X, \beta)$ over
	$P_n^t(S, \beta)$}
Below we suppose that Assumption~\ref{assum} holds. 
Then by Lemma~\ref{lem:shiftP},  
we have the following commutative diagram
\begin{align}\label{dia:PXS}
	\xymatrix{
		t_0(\Omega_{\fP_n^t(S, \beta)}[-1]) \ar[r]^-{\cong} \ar[d]
		& \pP_n^t(X, \beta) \ar[r] \ar[rd] \ar[d]_-{\pi_{\ast}^{\dag}} & P_n^t(X, \beta) 
		\ar[d]_-{\pi_{\ast}^{\dag}} \\
		\pP_n^t(S, \beta) \ar@{=}[r] & \pP_n^t(S, \beta) \ar[r] & P_n^t(S, \beta). 	
	}
\end{align}
Here the right horizontal arrows are good moduli space morphisms, and the right 
vertical arrow is the induced morphism on good moduli spaces. 
By Lemma~\ref{lem:PS:formal}, 
the derived stack $\fP_n^t(S, \beta)$ satisfies the formal neighborhood 
theorem. 
Together with Lemma~\ref{lem:analytic}, we have the 
formal local description of $\fP_n^t(S, \beta)$
over the good moduli space as we discuss below. 

Let us take a closed point $p \in P_n^t(S, \beta)$, which corresponds to a 
$\mu_t^{\dag}$-polystable object in $\aA_S$
\begin{align}\label{pstable}
	I=(\oO_S \to F)=	(\oO_S \to F_0) \oplus \bigoplus_{i=1}^k V_i \otimes (0 \to F_i)
\end{align}
where $V_i$ is a finite dimensional vector space,
$I_0=(\oO_S \to F_0)$ is $\mu_t^{\dag}$-stable 
and $F_1, \ldots, F_k$ are mutually non-isomorphic
$\mu_H$-stable sheaves in $\Coh_{\le 1}(S)$ with $\mu_H(F_i)=t$ for $1\le i\le k$. 
The automorphism group of $I$ is 
\begin{align*}
	G \cneq \mathrm{Aut}(I)=\prod_{i=1}^k \GL(V_i). 
\end{align*}
By (\ref{perf:obs2}), the 
tangent complex of $\mathfrak{P}_n^t(S, \beta)$ at $I$ is given by 
$\RHom(I, F)$, which is a 
complex of $G$-representations. 
The $G$-action is determined by the decomposition
\begin{align*}
	\RHom(I, F)=&
	\RHom(I_0, F_0) \oplus \bigoplus_{i=1}^k \RHom(I_0, F_i)\otimes V_i \\
	&\oplus \bigoplus_{i=1}^k \RHom(F_i, F_0)[1]\otimes V_i^{\vee} \oplus 
	\bigoplus_{i, j} \RHom(F_i, F_j)[1] \otimes \Hom(V_i, V_j). 
\end{align*}
and the natural actions of $G$ to $V_i$. 
Let $\widehat{\Hom(I, F)_0}$ be the formal fiber 
of $\Hom(I, F) \to \Hom(I, F)\ssslash G$ at $0$ (see Subsection~\ref{subsec:fib} for the notation 
of the formal fiber). 
Let $\kappa$ be a $G$-equivariant Kuranishi map
\begin{align}\label{K:kappa}
	\kappa \colon \widehat{\Hom(I, F)_0} \to \Hom^1(I, F)
\end{align} 
and $\widehat{\mathfrak{N}}_0$ its derived zero locus. 
Then the derived stack $\mathfrak{P}^t_n(S, \beta)$ along the 
good moduli space morphism $\pP^t_n(S, \beta) \to P_n^t(S, \beta)$
is equivalent to the quotient derived stack 
\begin{align*}
	[\widehat{\mathfrak{N}}_0/G] \hookrightarrow 
	[\widehat{\Hom(I, F)_0}/G].
\end{align*}

Let $w$ be the function 
\begin{align*}
	w \colon \widehat{\Hom(I, F)}_0 \oplus 
	\Hom^1(I, F)^{\vee} \to \mathbb{C}
\end{align*}
determined from $\kappa$ by the construction (\ref{def:w}). 
Then from the diagram (\ref{dia:PXS}), 
we see that 
\begin{align}\label{fiber:p}
	t_0\left(\Omega_{[\widehat{\fN}_0/G]}[-1] \right)=\left[\Crit(w)/G \right]
\end{align}
is isomorphic to the formal fiber of 
the morphism $\pP_n^t(X, \beta) \to P_n^t(S, \beta)$ at $p$ in the diagram 
(\ref{dia:PXS}). 
\begin{rmk}\label{rmk:IF}
	The $G$-representation $\Hom(I, F) \oplus \Hom^1(I, F)^{\vee}$ is explicitly written as 
	\begin{align}\notag
		&\Hom(I_0, F_0) \oplus \Hom^1(I_0, F_0)^{\vee}
		\oplus \bigoplus_{i=1}^k \left( \Hom(I_0, F_i) \oplus \Ext^2(F_i, F_0)^{\vee}  \right)
		\otimes V_i \\
		\notag&\oplus \bigoplus_{i=1}^k \left( \Hom^1(I_0, F_i)^{\vee} \oplus \Ext^1(F_i, F_0)  \right) \otimes 
		V_i^{\vee}
		\oplus \bigoplus_{i, j}
		\left(\Ext^1(F_i, F_j) \oplus \Ext^1(F_j, F_i)^{\vee}  \right) \otimes \Hom(V_i, V_j). 
	\end{align}	
	Similarly to Remark~\ref{rmk:squiver}, the above 
	$G$-representation is the 
	space of representations of the 
	Ext-quiver associated 
	with the collection 
	\begin{align*}
		\{(\oO_X \to i_{\ast}F_0), i_{\ast}F_1[-1], \ldots, i_{\ast}F_k[-1]\}
	\end{align*}
	on the 3-fold $X=\mathrm{Tot}_S(\omega_S)$, with dimension 
	vector $(1, \{\dim V_i\}_{1\le i\le k})$. 
	Here $i \colon S \hookrightarrow X$ is the zero section.  
\end{rmk}	
By restricting 
$P_n^{t_{\pm}}(X, \beta)$ and 
$\zZ_{l}^{t_{\pm}\us}$ to the formal fiber (\ref{fiber:p}), we 
have open/closed 
substacks
\begin{align*}
	\left[\Crit(w)^{t_{\pm} \sss}/G \right] \subset 
	\left[\Crit(w)/G \right]
	\supset	\widehat{\zZ}_l^{t_{\pm} \us}. 
\end{align*}
We also set $\widehat{\sS}_l^{t_{\pm}} \cneq 
\widehat{\zZ}_l^{t_{\pm} \us} \setminus \widehat{\zZ}_{l-1}^{t_{\pm} \us}$. 
\begin{lem}\label{lem:KNS}
	The stratifications
	\begin{align*}
		\left[\Crit(w)/G \right]=
		\widehat{\sS}_1^{t_{\pm}} \sqcup \widehat{\sS}_2^{t_{\pm}} \sqcup \cdots 
		\sqcup \widehat{\sS}_N^{t_{\pm}} \sqcup \left[ \Crit(w)^{t_{\pm}\sss}/G \right]	
	\end{align*}
	are KN stratifications 
	with respect to $\chi_0^{\pm} \colon G \to \C$, 
	where $\chi_0$ is the determinant character 
\begin{align*}
	\chi_0 \colon G=\prod_{i=1}^k \GL(V_i) \to \C, \ (g_i)_{1\le i\le k} \mapsto 
	\prod_{i=1}^k \det(g_i). 
	\end{align*}
\end{lem}
\begin{proof}
	A closed point of the stack 
	$[\Crit(w)/G]$ corresponds to an object $\eE$ of the form 
	(\ref{eE:polystable}) such
	that $\pi_{\ast}\eE$ is $S$-equivalent to (\ref{pstable}). 
	By the first condition in (\ref{assum:push}), the 
	object $\pi_{\ast}\eE_0$ is $\mu_t^{\dag}$-semistable, so 
	it is $S$-equivalent to the object in $\aA_S$ of the form
	\begin{align*}
		(\oO_S \to F_0') \oplus \bigoplus_{i=1}^k V_i' \otimes (0 \to F_i),
		\end{align*}
	where $I_0'=(\oO_S \to F_0')$ is $\mu_{t}^{\dag}$-stable. 
	The second condition in (\ref{assum:push}) implies that each 
	$\pi_{\ast}E_i$ is $\mu_H$-stable. 
	Therefore by comparing the multiplicity of $F_i$, we have 
	\begin{align*}
		V_i =W_i' \oplus \bigoplus_{\pi_{\ast}E_j \cong F_i}W_j. 
		\end{align*}
		The map 
		$\Aut(\eE) \stackrel{\pi_{\ast}}{\to} \Aut(\pi_{\ast}\eE) 
		\subset \Aut(I)$
		is the 
		inclusion $G_{\eE} \hookrightarrow G$ given by 
	\begin{align*}
		\prod \GL(W_i) \to \prod \GL(V_i), \ 
		(g_i) \mapsto \left(\id_{W_i'} \oplus \bigoplus_{\pi_{\ast}E_j \cong F_i}g_j \right). 
	\end{align*}
	Therefore 
	the composition
$G_{\eE} \hookrightarrow G \stackrel{\chi_0}{\to} \C$
		coincides with the determinant character for $G_{\eE}$
	in Lemma~\ref{lem:Pstrata}.  
	It follows that the claim holds for 
	each fiber of the map 
	$[\Crit(w)/G] \to \Crit(w)\ssslash G$
	by Lemma~\ref{lem:Pstrata}, so the lemma follows.  
\end{proof}

\subsubsection{The formal local descriptions of the stack of exact
	sequences}
Let $(\beta, n)=(\beta_1, n_1)+(\beta_2, n_2)$ be 
a decomposition in $\mathbf{D}_{(\beta, n)}^{t}$. 
For $v_i=(\beta_i, n_i)$, we have the open substacks
in the diagrams (\ref{dia:MSdag}), (\ref{diagram:ddag2})
\begin{align*}
	&\fM_{S}^{\ext, \dag, t}(v_{\bullet}) 
	\cneq (\ev_3^{\dag})^{-1}(\mathfrak{P}_n^t(S, \beta)) 
	\subset 	\fM_{S}^{\ext, \dag}(v_{\bullet}), \\
	&\fM_{S}^{\ext, \ddag, t}(v_{\bullet}) 
	\cneq (\ev_3^{\ddag})^{-1}(\mathfrak{P}_n^t(S, \beta)) 
	\subset 	\fM_{S}^{\ext, \ddag}(v_{\bullet}).
\end{align*}
Similarly to (\ref{dia:Xext:P}), 
the diagrams (\ref{dia:MSdag}), (\ref{diagram:ddag2}) restrict to the diagrams
\begin{align}\label{dia:MSdagP}
	\xymatrix{
		\fM_S^{\ext, \dag, t}(v_{\bullet}) \ar[r]^-{\ev_3^{\dag}} \ar[d]_-{(\ev_1^{\dag}, \ev_2^{\dag})}
		& \mathfrak{P}_n^t(S, \beta) \\
		\mathfrak{P}_{n_1}^t(S, \beta_1) \times \mathfrak{M}_{n_2}^H(S, \beta_2),   
	}
	\xymatrix{
		\fM_S^{\ext, \dag, t}(v_{\bullet}) \ar[r]^-{\ev_3^{\ddag}} \ar[d]_-{(\ev_2^{\ddag}, \ev_1^{\ddag})}
		& \mathfrak{P}_n^t(S, \beta) \\
		\mathfrak{M}_{n_2}^H(S, \beta_2) \times	\mathfrak{P}_{n_1}^t(S, \beta_1) .    
	}
\end{align}
By taking the classical truncations
and good moduli spaces, we obtain the following commutative 
diagrams
\begin{align}\label{dia:MSdagP2}
	\xymatrix{
		\mM_S^{\ext, \dag, t}(v_{\bullet}) \ar[r]^-{\ev_3^{\dag}} \ar[d]_-{(\ev_1^{\dag}, \ev_2^{\dag})}
		& \pP_n^t(S, \beta) \ar[dd] \\
		\pP_{n_1}^t(S, \beta_1) \times \mM_{n_2}^H(S, \beta_2) \ar[d]  & 
		\\
		P_{n_1}^t(S, \beta_1) \times M_{n_2}^H(S, \beta_2) \ar[r]^-{\oplus} & P_n^t(S, \beta), 
	}
	\xymatrix{
		\mM_S^{\ext, \ddag, t}(v_{\bullet}) \ar[r]^-{\ev_3^{\ddag}} \ar[d]_-{(\ev_2^{\ddag}, \ev_1^{\ddag})}
		& \pP_n^t(S, \beta) \ar[dd] \\
		\mM_{n_2}^H(S, \beta_2) \times	\pP_{n_1}^t(S, \beta_1) \ar[d]  & 
		\\
		M_{n_2}^H(S, \beta_2) \times	P_{n_1}^t(S, \beta_1) \ar[r]^-{\oplus} & P_n^t(S, \beta). 	
	}
\end{align}
Here the bottom arrows are given by taking the direct sums of polystable objects, and 
the left bottom vertical arrow, the right vertical arrows are good moduli space morphisms.

Any point $(p^{(1)}, p^{(2)})$ of the bottom arrows in (\ref{dia:MSdagP2}) at $p \in P_n^t(S, \beta)$
corresponds to a direct sum decomposition of the object $I$ in (\ref{pstable})
\begin{align}\label{decom:I}
	I=I_0 \oplus  \bigoplus_{i=1}^k V_i^{(1)} \otimes (0 \to F_i)
	\oplus \bigoplus_{i=1}^k V_i^{(2)} \otimes (0 \to F_i). 
\end{align}
Here $V_i=V_i^{(1)}\oplus V_i^{(2)}$ and 
\begin{align*}
	I^{(1)}=(\oO_S \to F^{(1)}) &\cneq I_0 \oplus  \bigoplus_{i=1}^k V_i^{(1)} \otimes (0 \to F_i)
	\in P_{n_1}^t(S, \beta_1), \\
	F^{(2)} & \cneq \bigoplus_{i=1}^k V_i^{(2)} \otimes F_i 
	\in M_{n_2}^H(S, \beta_2),
\end{align*}
such that $p^{(1)}$ corresponds to $I^{(1)}$ and $p^{(2)}$ corresponds to $F^{(2)}$. 
We have the decomposition
\begin{align*}
	\RHom(I, F)=
	\RHom(I^{(1)}, F^{(1)}) &\oplus \RHom(I^{(1)}, F^{(2)}) \\
	&\oplus \RHom(F^{(2)}, F^{(1)})[1] \oplus 
	\RHom(F^{(2)}, F^{(2)})[1]. 
\end{align*}
Let $G^{(1)}$, $G^{(2)}$ be 
\begin{align*}
	G^{(1)}=\Aut(I^{(1)})=
	\prod_{i=1}^k \GL(V_i^{(1)}), \ 
	G^{(2)}=\Aut(F^{(2)})=\prod_{i=1}^k \GL(V_i^{(2)}). 
\end{align*}
Let $T^{(ij)}$ and $O^{(ij)}$ be defined by 
\begin{align*}
	&T^{(11)}=\Hom(I^{(1)}, F^{(1)}), \ 
	T^{(12)}=\Hom(I^{(1)}, F^{(2)}), \\ 
	&T^{(21)}=\Ext^1(F^{(2)}, F^{(1)}), \ 
	T^{(22)}=\Ext^1(F^{(2)}, F^{(2)}), \\
	&O^{(11)}=\Hom^1(I^{(1)}, F^{(1)}), \ 
	O^{(12)}=\Hom^1(I^{(1)}, F^{(2)}), \\ 
	&O^{(21)}=\Ext^2(F^{(2)}, F^{(1)}), \ 
	O^{(22)}=\Ext^2(F^{(2)}, F^{(2)}).
\end{align*}
We also write $T^{(i)}=T^{(ii)}$, $O^{(i)}=O^{(ii)}$ for 
simplicity. 
We have $G^{(i)}$-equivariant Kuranishi maps for $i=1, 2$
\begin{align*}
	\kappa^{(i)} \colon 
	\widehat{T}^{(i)}_0 \to O^{(i)}. 
\end{align*}
Let 
$\widehat{\mathfrak{N}}_0^{(1)}$, $\widehat{\mathfrak{N}}_0^{(2)}$
be derived zero loci of 
$\kappa^{(1)}$, $\kappa^{(2)}$ respectively. 
By Lemma~\ref{lem:analytic}, 
the product of quotient derived stacks
\begin{align*}
	[\widehat{\mathfrak{N}}_0^{(1)}/G^{(1)}] \times 
	[\widehat{\mathfrak{N}}_0^{(2)}/G^{(2)}] 
\end{align*}
is equivalent to the 
derived stack 
$\mathfrak{P}_{n_1}^t(S, \beta_1) \times 
\mathfrak{M}_{n_2}^H(S, \beta_2)$ 
along the formal fiber of the 
left bottom arrow of (\ref{dia:MSdagP2}) at 
$(p^{(1)}, p^{(2)})$. 
We have the function 
\begin{align*}
	w^{(i)} \colon \widehat{T}^{(i)}_0 \oplus O^{(i)\vee} \to \mathbb{C}
\end{align*}
determined by $\kappa^{(i)}$ by the construction (\ref{def:w}). 

We next give formal local description of the stack of exact 
sequences. 

\begin{lem}\label{lem:That}
	We have the following 
	\begin{align}\label{subsp}
		&\widehat{\Hom(I, F)}_0 \cap 
		\left( \bigoplus_{i\ge j}T^{(ij)}\right)
		=\widehat{T}^{(1)}_0 \oplus 
		\widehat{T}^{(2)}_0 \oplus T^{(21)}, \\ 
		&\notag
		\widehat{\Hom(I, F)}_0 \cap 
		\left( \bigoplus_{i\le j}T^{(ij)}\right)
		=\widehat{T}^{(2)}_0 \oplus 
		\widehat{T}^{(1)}_0 \oplus T^{(12)}. 
	\end{align}
\end{lem}
\begin{proof}
	We only prove the first identity. 
	Let us consider the composition
	\begin{align}\label{compose1}
		\bigoplus_{i\ge j}T^{(ij)} \hookrightarrow
		\Hom(I, F) \twoheadrightarrow
		\Hom(I, F)\ssslash G. 
	\end{align}
	Since any point in the LHS is specialized to a point
	in $T^{(1)} \oplus T^{(2)}$ by the action of the one parameter 
	subgroup 
	\begin{align*}
		\C \to G^{(1)} \times G^{(2)} \subset G, \ 
		t \mapsto (\id, t^{-1}\id)
	\end{align*}
	the composition (\ref{compose1}) factors through 
	\begin{align}\label{compose2}
		\bigoplus_{i\ge j}T^{(ij)} 
		\twoheadrightarrow 
		\bigoplus_{i=1}^2 T^{(i)} 
		\to 
		\Hom(I, F)\ssslash G. 			
	\end{align}
	The right arrow further factors as 
	\begin{align}\label{compose3}
		\bigoplus_{i=1}^2 T^{(i)} 
		\twoheadrightarrow T^{(1)}\ssslash G^{(1)} \times 
		T^{(2)} \ssslash G^{(2)} \to 
		\Hom(I, F)\ssslash G
	\end{align}
	The right morphism is finite by~\cite[Lemma~2.1]{MeRe}, whose preimage 
	of $0$ is $(0, 0)$. 
	Therefore the formal fiber of the composition (\ref{compose3}) at $0$
	is $\widehat{T}^{(1)}_0 \oplus \widehat{T}^{(2)}_0$. 
	By taking the formal fiber of (\ref{compose2}) at $0$, 
	we conclude that the formal fiber of (\ref{compose1}) is 
	given by the RHS of (\ref{subsp}). 
\end{proof}

By Lemma~\ref{lem:That}
and the fact that the Kuranishi map $\kappa$ in (\ref{K:kappa})
is $G$-equivariant, 
it 
restricts to morphisms
\begin{align}\notag
	&	\kappa^{\ext} \colon 
	\widehat{T}^{(1)}_0
	\oplus \widehat{T}^{(2)}_0
	\oplus T^{(21)}	\to 
	\bigoplus_{i\ge j}
	O^{(ij)}, \\
	&\notag	{\kappa'}^{\ext} \colon 
	\widehat{T}^{(2)}_0
	\oplus \widehat{T}^{(1)}_0
	\oplus T^{(12)}	\to 
	\bigoplus_{i\le j}
	O^{(ij)}. 
\end{align}
Let $\widehat{\mathfrak{N}}_0^{\ext}$, $\widehat{\fN}_0^{'\ext}$
be derived zero locus of $\kappa^{\ext}$, ${\kappa'}^{\ext}$ respectively. 
We have closed immersions
\begin{align*}
	\widehat{\mathfrak{N}}_0^{\ext} \hookrightarrow 
	\widehat{\mathfrak{N}}_0, \ 
	\widehat{\mathfrak{N}}_0^{'\ext} \hookrightarrow 
	\widehat{\mathfrak{N}}_0
\end{align*}
which correspond to deformations of $I$ which preserve
subobjects $I^{(1)} \subset I$, $(0 \to F^{(2)}) \subset I$
respectively. 
The deformations preserving 
these subobjects are identified 
by the actions of 
the subgroups
\begin{align*}
	G^{\ext} \cneq 
	\{g \in G : g(I^{(1)}) \subset I^{(1)}\}, \ 
	{G'}^{\ext} \cneq 
	\{g \in G : g(0 \to F^{(2)}) \subset (0 \to F^{(2)})\}
	\end{align*}
respectively. 
They are explicitly written as 
\begin{align*}
	&G^{\ext}=
	\prod_{i=1}^k \GL(V_i^{(1)}) \times 
	\prod_{i=1}^k \GL(V_i^{(2)}) 
	\times \prod_{i=1}^k \Hom(V_i^{(2)}, V_i^{(1)}), \\
	&{G'}^{\ext}=
	\prod_{i=1}^k \GL(V_i^{(2)}) \times 
	\prod_{i=1}^k \GL(V_i^{(1)}) 
	\times \prod_{i=1}^k \Hom(V_i^{(1)}, V_i^{(2)}). 
\end{align*}
Therefore
the derived stacks $\fM_S^{\ext, \dag, t}(v_{\bullet})$, 
$\fM_S^{\ext, \ddag, t}(v_{\bullet})$ 
along the formal 
fiber of the composition of the left arrows in (\ref{dia:MSdagP2}) at $(p^{(1)}, p^{(2)})$
are equivalent to the derived stacks
$[\widehat{\mathfrak{N}}_0^{\ext}/G^{\ext}]$, 
$[\widehat{\mathfrak{N}}_0^{\ext}/{G'}^{\ext}]$
respectively. 

By the above arguments, the diagrams (\ref{dia:MSdagP})
along the formal fibers at $p \in P_n^t(S, \beta)$ are
equivalent to the following diagrams
\begin{align}\label{dia:ext:N}
	\xymatrix{
		[\widehat{\mathfrak{N}}_0^{\ext}/G^{\ext}]	
		\ar[r]^-{\ev_3^{\dag}} \ar[d]_-{(\ev_1^{\dag}, \ev_2^{\dag})}
		& 	[\widehat{\mathfrak{N}}_0/G] \\
		[\widehat{\mathfrak{N}}^{(1)}_0/G^{(1)}]
		\times [\widehat{\mathfrak{N}}^{(2)}_0/G^{(2)}], 
		& 
	}
	\xymatrix{
		[\widehat{\mathfrak{N}}_0^{'\ext}/{G'}^{\ext}]	
		\ar[r]^-{\ev_3^{\ddag}} \ar[d]_-{(\ev_2^{\ddag}, \ev_1^{\ddag})}
		& 	[\widehat{\mathfrak{N}}_0/G] \\
		[\widehat{\mathfrak{N}}^{(2)}_0/G^{(2)}]
		\times [\widehat{\mathfrak{N}}^{(1)}_0/G^{(1)}]. 
		& 
	}
\end{align}
Here the horizontal arrows are induced by the inclusions of (\ref{subsp})
into $\widehat{\Hom(I, F)}_0$, 
and the vertical arrows are induced by the projections
from $\oplus_{i\ge j}T^{(ij)}$, $\oplus_{i\le j}T^{(ij)}$ onto 
$T^{(1)} \oplus T^{(2)}$, $T^{(2)} \oplus T^{(1)}$ respectively. 
Let $w^{\ext}$ and ${w'}^{\ext}$ be the functions
\begin{align*}
	w^{\ext} \colon \widehat{T}^{(1)}_0 \oplus 
	\widehat{T}^{(2)}_0 \oplus T^{(21)} \oplus 
	\bigoplus_{i\ge j}O^{(ij)\vee} \to \mathbb{C}, \quad
	{w'}^{\ext} \colon \widehat{T}^{(1)}_0 \oplus 
	\widehat{T}^{(2)}_0 \oplus T^{(12)} \oplus 
	\bigoplus_{i\le j}O^{(ij)\vee} \to \mathbb{C}
\end{align*}
determined from $\kappa^{\ext}$, ${\kappa'}^{\ext}$
by the construction (\ref{def:w}).
We have the following commutative diagrams
\begin{align}\label{dia:wext}
	\xymatrix{
		&   \left[ (\widehat{\Hom(I, F)}_0 \oplus 
		\Hom^1(I, F)^{\vee}) /G  \right]
		\ar@/^130pt/[ddd]_-{w} \\		
		\left[\left(\widehat{T}^{(1)}_0 \oplus 
		\widehat{T}^{(2)}_0 \oplus T^{(21)} \oplus  \bigoplus
		_{i\le j} O^{(ij)\vee} \right)/G^{\ext} \right]	\ar@/_110pt/[dd]_-{q_1} 
		\ar[r]_-{r_1} \ar[d]_-{r_2} \ar[ru]^-{q_2}\diasquare & \left[\left(\widehat{T}^{(1)}_0 \oplus 
		\widehat{T}^{(2)}_0 \oplus T^{(21)} \oplus 
		\Hom^1(I, F)^{\vee} \right)/G^{\ext} \right] \ar[u]_-{f_2} \ar[d]^-{g_2}  \\
		\left[\left(\widehat{T}^{(1)}_0 \oplus 
		\widehat{T}^{(2)}_0 \oplus T^{(21)} \oplus
		O^{(1)\vee} \oplus O^{(2)\vee}
		\right)/G^{\ext} \right]	
		\ar[r]_-{g_1} \ar[d]_-{f_1} & \left[\left(\widehat{T}^{(1)}_0 \oplus 
		\widehat{T}^{(2)}_0 \oplus T^{(21)} \oplus 
		\bigoplus_{i\ge j}O^{(ij)\vee}
		\right)/G^{\ext} \right] \ar[d]^-{w^{\ext}} \\
		\left[(\widehat{T}^{(1)}_0 \oplus O^{(1)\vee})/G^{(1)}  \right]
		\times 
		\left[(\widehat{T}^{(2)}_0 \oplus O^{(2)\vee})/G^{(2)}  \right] 
		\ar[r]_-{w^{(1)}+w^{(2)}} 
		& \mathbb{C}. 
	} 
\end{align} 
\begin{align}\notag
	\xymatrix{
		&   \left[ (\widehat{\Hom(I, F)}_0 \oplus 
		\Hom^1(I, F)^{\vee}) /G  \right]
		\ar@/^130pt/[ddd]_-{w'} \\		
		\left[\left(\widehat{T}^{(2)}_0 \oplus 
		\widehat{T}^{(1)}_0 \oplus T^{(12)} \oplus  \bigoplus
		_{i\ge j} O^{(ij)\vee} \right)/{G'}^{\ext} \right]	\ar@/_110pt/[dd]_-{q_1'} 
		\ar[r]_-{r_1'} \ar[d]_-{r_2'} \ar[ru]^-{q_2'}\diasquare & \left[\left(\widehat{T}^{(2)}_0 \oplus 
		\widehat{T}^{(1)}_0 \oplus T^{(12)} \oplus 
		\Hom^1(I, F)^{\vee} \right)/{G'}^{\ext} \right] \ar[u]_-{f_2'} \ar[d]^-{g_2'}  \\
		\left[\left(\widehat{T}^{(2)}_0 \oplus 
		\widehat{T}^{(1)}_0 \oplus T^{(12)} \oplus
		O^{(2)\vee} \oplus O^{(1)\vee}
		\right)/{G'}^{\ext} \right]	
		\ar[r]_-{g_1'} \ar[d]_-{f_1'} & \left[\left(\widehat{T}^{(2)}_0 \oplus 
		\widehat{T}^{(1)}_0 \oplus T^{(12)} \oplus 
		\bigoplus_{i\le j}O^{(ij)\vee}
		\right)/{G'}^{\ext} \right] \ar[d]^-{{w'}^{\ext}} \\
		\left[(\widehat{T}^{(2)}_0 \oplus O^{(2)\vee})/G^{(2)}  \right] 
		\times 
		\left[(\widehat{T}^{(1)}_0 \oplus O^{(1)\vee})/G^{(1)}  \right]
		\ar[r]_-{w^{(2)}+w^{(1)}} 
		& \mathbb{C}. 
	} 
\end{align} 
Here $g_1, r_1, f_2, q_2, g_1', r_1', f_2', q_2'$ are induced by embeddings into direct summands, 
and $f_1, g_2, r_2, q_1, f_1', g_2', r_2', q_1'$ are induced by projections. 
In order to simplify the notation, we write
the above diagrams as 
\begin{align}\label{dia:wext:simple}
	\xymatrix{	
		\xX_6	\ar@/_20pt/[dd]_-{q_1} 
		\ar[r]_-{r_1} \ar[d]_-{r_2} \ar@/^20pt/[rr]^-{q_2}\diasquare & 
		\xX_2 \ar[r]_-{f_2} \ar[d]^-{g_2} & \xX_1 \ar[dd]^-{w}  \\
		\xX_4	
		\ar[r]_-{g_1} \ar[d]_-{f_1} & \xX_3 \ar[dr]^-{w^{\ext}} & \\
		\xX_5^{(1)} \times \xX_5^{(2)}
		\ar[rr]_-{w^{(1)}+w^{(2)}}  & 
		& \mathbb{C},
	} 
	\quad
	\xymatrix{	
		\xX_6'	\ar@/_20pt/[dd]_-{q_1'} 
		\ar[r]_-{r_1'} \ar[d]_-{r_2'} \ar@/^20pt/[rr]^-{q_2'}\diasquare & 
		\xX_2' \ar[r]_-{f_2'} \ar[d]^-{g_2'} & \xX_1 \ar[dd]^-{w}  \\
		\xX_4'	
		\ar[r]_-{g_1'} \ar[d]_-{f_1'} & \xX_3' \ar[dr]^-{{w'}^{\ext}} & \\
		\xX_5^{(2)} \times \xX_5^{(1)}
		\ar[rr]_-{w^{(2)}+w^{(1)}}  & 
		& \mathbb{C}.
	} 
\end{align} 
e.g. $\xX_1=[(\widehat{\Hom}(I, F)_0 \oplus \Hom^1(I, F)^{\vee})/G]$, etc. 
We will also write $w_i \colon \xX_i \to \mathbb{C}$, $w_i' \colon \xX_i' \to \mathbb{C}$ for the induced 
functions on the above diagrams for $i=2, 4, 6$. 

Let $\lambda$ be the one parameter subgroup
\begin{align}\label{lambda}
	\lambda \colon \C \to G^{(1)} \times G^{(2)}, \ 
	t \mapsto (\id, t \cdot \id). 
\end{align}
Below we also regard $\lambda$ as one parameter subgroups 
for $G^{\ext}$, ${G'}^{\ext}$ and $G$ as 
they contain $G^{(1)} \times G^{(2)}$ as a subgroup. 
Note that the diagrams 
\begin{align}\label{dia:X}
	\xymatrix{
		\xX_6 \ar[r]^-{q_2} \ar[d]_-{q_1} & \xX_1 \\
		\xX_5^{(1)} \times \xX_5^{(2)}, &	
	}
	\quad 
	\xymatrix{
		\xX_6' \ar[r]^-{q_2'} \ar[d]_-{q_1'} & \xX_1 \\
		\xX_5^{(2)} \times \xX_5^{(1)} &	
	}
\end{align}
are identified with the diagram (\ref{dia:closure})
for $[Y/G]=\xX_1$ and $\lambda_{\alpha}=\lambda^{-1}, \lambda$
respectively. 
Let $\widehat{\ev}^{\dag}$, $\widehat{\ev}^{\ddag}$ be the morphisms 
from the diagrams (\ref{dia:ext:N}),
\begin{align*}
	\widehat{\ev}^{\dag} &\cneq
	 (\ev_1^{\dag}, \ev_2^{\dag}, \ev_3^{\dag}) \colon 
	[\widehat{\fN}_0^{\ext}/G^{\ext}] \to 
	[\widehat{\fN}_0^{(1)}/G^{(1)}] \times 
		[\widehat{\fN}_0^{(2)}/G^{(1)}] \times 
			[\widehat{\fN}_0/G], \\
			\widehat{\ev}^{\ddag} 
			&\cneq (\ev_1^{\ddag}, \ev_2^{\ddag}, \ev_3^{\ddag}) \colon 
			[\widehat{\fN}_0^{'\ext}/G^{\ext}] \to 
			[\widehat{\fN}_0^{(1)}/G^{(1)}] \times 
			[\widehat{\fN}_0^{(2)}/G^{(1)}] \times 
			[\widehat{\fN}_0/G]. 
	\end{align*} 
As in the proof of Proposition~\ref{prop:ind:R}
(see (\ref{ev:-2})), 
we have 
\begin{align*}
	&t_0(\Omega_{\widehat{\ev}^{\dag}}[-2])
	=(q_1^{-1}([\Crit(w^{(1)}+w^{(2)})/G^{(1)} \times G^{(2)}]) \cap 
	q_2^{-1}([\Crit(w)/G]))\times_{\xX_3} [\Crit(w^{\ext})/G^{\ext}], \\
	&t_0(\Omega_{\widehat{\ev}^{\ddag}}[-2])
	=({q_1'}^{-1}([\Crit(w^{(2)}+w^{(1)})/G^{(2)} \times G^{(1)}]) \cap 
	{q_2'}^{-1}([\Crit(w)/G])\times_{\xX_3'} [\Crit({w'}^{\ext})/{G'}^{\ext}).
	\end{align*}
In particular, 
the diagrams (\ref{dia:X}) restrict to diagrams 
	\begin{align}\label{dia:X2}
	&	\xymatrix{
			t_0(\Omega_{\widehat{\ev}^{\dag}}[-2]) \ar[r]^-{q_2} \ar[d]_-{q_1} & [\Crit(w)/G] \\
			[\Crit(w^{(1)})/G^{(1)}]\times [\Crit(w^{(2)})/G^{(2)}], &	
		} \\
	\notag	&		\xymatrix{
			t_0(\Omega_{\widehat{\ev}^{\ddag}}[-2])	 \ar[r]^-{q_2'} \ar[d]_-{q_1'} & [\Crit(w)/G] \\
			[\Crit(w^{(2)})/G^{(2)}]\times [\Crit(w^{(1)})/G^{(1)}],  
		}
	\end{align}
which give formal local descriptions of the diagrams (\ref{dia:Xext:P})
at $p\in P_n^t(S, \beta)$.

\subsubsection{Adjoint functors of categorified Hall products}
Let us consider the ind-completions of the functors (\ref{descend}) 
\begin{align}\label{descend:ind}
	&\ast \colon 
	\Ind \left(\dDT^{\C}(\pP_{n_1}^{t_+}(X, \beta_1)) \otimes 
	\dDT^{\C}(\mM_{n_2}^H(X, \beta_2)) \right)
	\to \Ind \dDT^{\C}\left(\pP^t_{n}(X, \beta) \setminus \zZ^{t_{+}\us}_{l-1} \right), \\
	\notag	&\ast \colon \Ind \left(\dDT^{\C}(\mM_{n_2}^H(X, \beta_2))
	\otimes \dDT^{\C}(\pP_{n_1}^{t_-}(X, \beta_1)) \right)
	\to \Ind \dDT^{\C}\left(\pP^t_{n}(X, \beta) \setminus \zZ^{t_{-}\us}_{l-1} \right). 
\end{align}

We show that they admit right adjoints under Assumption~\ref{assum}. 
\begin{lem}\label{lem:adjoint:ind}
	Under Assumption~\ref{assum}, 
	the functors (\ref{descend:ind}) admit
	right adjoints
	\begin{align*}
		&\ast^R \colon 
		\Ind \dDT^{\C}\left(\pP^t_{n}(X, \beta) \setminus \zZ^{t_{+}\us}_{l-1} \right) \to 
		\Ind\left(\dDT^{\C}(\pP_{n_1}^{t_+}(X, \beta_1)) \otimes 
		\dDT^{\C}(\mM_{n_2}^H(X, \beta_2))\right), \\
		\notag	&\ast^R \colon
		\Ind\dDT^{\C}\left(\pP^t_{n}(X, \beta) \setminus \zZ^{t_{-}\us}_{l-1} \right) \to 
		\Ind \left(\dDT^{\C}(\mM_{n_2}^H(X, \beta_2))
		\otimes \dDT^{\C}(\pP_{n_1}^{t_-}(X, \beta_1)) \right). 
	\end{align*}
\end{lem}
\begin{proof}
	We only prove the statement for $t_+$. 
	By the equivalences (\ref{shift:MP}), under the assumption the 
	functor (\ref{descend}) is a descendant 
	of the functor from the diagram (\ref{dia:MSdagP})
	\begin{align}\label{descendent:star}
		\ev_{3\ast}^{\dag}(\ev_1^{\dag}, \ev_2^{\dag})^{\ast}
		\colon \Dbc(\fP_{n_1}^t(S, \beta_1))
		\otimes \Dbc(\fM_{n_2}^H(S, \beta_2))
		\to \Dbc(\fP_{n}(S, \beta)). 
	\end{align}
	Its ind-completion admits a right adjoint functor (see Subsection~\ref{subsub:QCA})
	\begin{align}\notag
		(\ev_1^{\dag}, \ev_2^{\dag})_{\ast}^{\ind}(\ev_3^{\dag})^!
		\colon 
		\Ind \Dbc(\fP_n(S, \beta)) \to 
		\Ind \left(
		\Dbc(\fP_{n_1}^t(S, \beta_1))
		\otimes \Dbc(\fM_{n_2}^H(S, \beta_2)) \right). 
	\end{align}
	We see that the above functor restricts to the functor 
	\begin{align}\label{send:CZ}
		(\ev_1^{\dag}, \ev_2^{\dag})_{\ast}^{\ind}(\ev_3^{\dag})^! \colon 	
		\Ind \cC_{\zZ^{t_{+}\us}_{l-1}}
		\to \Ind \left( \cC_{\zZ^{t_{+}\us (1)} \times \mM_{n_2}^H(X, \beta_2)} \right). 
	\end{align}
Here we have set
\begin{align*}
	\zZ^{t_+\us (1)}=\pP_{n_1}^{t}(X, \beta_1)
	\setminus \pP_{n_1}^{t_+}(X, \beta_1).
	\end{align*}
	Indeed it is enough to show the above claim formally locally 
	on $P_n^t(S, \beta)$ in the diagram (\ref{dia:MSdagP2}), 
	i.e. the functor in the diagram (\ref{dia:ext:N})
	\begin{align}\label{formal:ev0}
		(\ev_1^{\dag}, \ev_2^{\dag})_{\ast}^{\ind}(\ev_3^{\dag})^! \colon 	
		\Ind \Dbc([\widehat{\fN}_0/G]) \to 
		\Ind \Dbc([\widehat{\fN}_0^{(1)}/G^{(1)}] \times [\widehat{\fN}_0^{(2)}/G^{(2)}])
	\end{align}
restricts to the functor 
	\begin{align}\label{formal:ev02}
	(\ev_1^{\dag}, \ev_2^{\dag})_{\ast}^{\ind}(\ev_3^{\dag})^! \colon 	
	\Ind \cC_{\widehat{\zZ}^{t_{+}\us}_{l-1}}
	\to \Ind \left( \cC_{\widehat{\zZ}^{t_{+}\us(1)} \times [\Crit(w^{(2)})/G^{(2)}]} \right). 
\end{align}
Here $\widehat{\zZ}^{t_+\us(1)} \subset [\Crit(w^{(1)})/G^{(1)}]$ is the 
pull-back of $\zZ^{t_+\us (1)}$ to $[\Crit(w^{(1)})/G^{(1)}]$. 
	By Lemma~\ref{lem:exP}, 
	we have the identity in the diagram (\ref{dia:X2})
	\begin{align*}
		q_1^{-1}(\widehat{\zZ}^{t_{+}\us (1)} \times [\Crit(w^{(2)})/G^{(2)}])=q_2^{-1}(\widehat{\zZ}^{t_{+}\us}_{l-1}). 
		\end{align*}
	Since the Cartesian diagram in (\ref{dia:wext}) is derived Cartesian 
	and $f_2$ is proper, 
	we can apply Proposition~\ref{prop:ind:R}
	to conclude that the functor (\ref{formal:ev0})
	restricts to the functor (\ref{formal:ev02}).  
	
	By Theorem~\ref{thm:compact},
	both sides in (\ref{send:CZ}) are compactly generated with 
	compact objects $\cC_{\zZ^{t_{+}\us}_{l-1}}$, 
	$\cC_{\zZ^{t_{+}\us (1)} \times \mM_{n_2}^H(X, \beta_2)}$
	respectively. 
	Therefore by  
	the equivalences in Proposition~\ref{lem:indDT}, the desired right 
	adjoint $\ast^R$ is obtained by taking the 
	Verdier quotients of both sides in (\ref{descendent:star}) by the subcategories 
	in (\ref{send:CZ}). 
\end{proof}

\begin{lem}\label{lem:radjoint}
	In the setting of Lemma~\ref{lem:adjoint:ind},
	the functors $\ast^R$ in Lemma~\ref{lem:adjoint:ind} composed with 
	the projection
	\begin{align*}
		\mathrm{pr}_j \colon 
		\Ind \dDT^{\C}(\mM_{n_2}^H(X, \beta_2)) \to 
		\Ind \dDT^{\C}(\mM_{n_2}^H(X, \beta_2))_{j}
	\end{align*} 
	restrict to the functors
	\begin{align}\notag
		&\ast^R_j =\mathrm{pr}_j \circ \ast^R 
		\colon \dDT^{\C}\left(\pP^t_{n}(X, \beta) \setminus \zZ^{t_{+}\us}_{l-1} \right) \to 
		\dDT^{\C}(\pP_{n_1}^{t_+}(X, \beta_1)) \otimes 
		\dDT^{\C}(\mM_{n_2}^H(X, \beta_2))_j, \\
		\notag
		&\ast^R_j =\mathrm{pr}_j \circ \ast^R \colon \dDT^{\C}\left(\pP^t_{n}(X, \beta) \setminus \zZ^{t_{-}\us}_{l-1} \right) \to
		\dDT^{\C}(\mM_{n_2}^H(X, \beta_2))_j
		\otimes \dDT^{\C}(\pP_{n_1}^{t_-}(X, \beta_1))
	\end{align}
	which give right adjoint functors of the functors (\ref{funct:rest}). 
\end{lem}
\begin{proof}
	We only prove the lemma for $t_+$. 
	For an object $(-)$ in the LHS, we need to prove 
	that the object $\ast^R_j(-) \cneq \mathrm{pr}_j \ast^R(-)$, which 
	is a priori an object in the ind-completion of the RHS by Lemma~\ref{lem:adjoint:ind}, 
	is indeed an object in the RHS. 
	It is enough to prove this formally locally on each point of $P_n^t(S, \beta)$
	in the diagram (\ref{dia:MSdagP2}). 
	Let us take a closed point $p\in P_n^t(S, \beta)$ corresponding 
	to the polystable object (\ref{pstable}).
	Then formally locally around $p$, the functor $\ast_j^R$ is a descendant of 
	the functor (\ref{formal:ev0}) composed with 
	the projection 
	\begin{align*}
		\mathrm{pr}_j \colon 
		\Ind\Dbc([\widehat{\fN}_0^{(1)}/G^{(1)}] \times [\widehat{\fN}_0^{(2)}/G^{(2)}]) \to
		\Ind\Dbc([\widehat{\fN}_0^{(1)}/G^{(1)}] \times [\widehat{\fN}_0^{(2)}/G^{(2)}])_{\lambda \mathchar`- \wt= j}.
		\end{align*}
		Here $\lambda$ is the one parameter subgroup (\ref{lambda}). 
	Therefore it is enough to show that the composition of (\ref{formal:ev0}) with the above 
	projection restricts to the functor
	\begin{align}\notag
		\mathrm{pr}_j(\ev_1^{\dag}, \ev_2^{\dag})_{\ast}^{\ind}(\ev_3^{\dag})^! \colon 	
		\Dbc([\widehat{\fN}_0/G]) \to 
		\Dbc([\widehat{\fN}_0^{(1)}/G^{(1)}] \times [\widehat{\fN}_0^{(2)}/G^{(2)}])_{\lambda \mathchar`- \wt= j}. 
	\end{align}

	Below we use the notation of the diagram (\ref{dia:wext:simple}). 
	By Lemma~\ref{lem:commute1} and Lemma~\ref{lem:commute2}, 
	we are reduced to showing that
	the composition of the functors
	\begin{align*}
	&\MF_{\qcoh}^{\C}(\xX_1, w) \stackrel{f_2^!}{\to} \MF_{\qcoh}^{\C}(\xX_2, w_2) \stackrel{g_{2\ast}}{\to}	
	\MF_{\qcoh}^{\C}(\xX_3, w^{\ext}) 
	\stackrel{g_1^{\ast}}{\to} \MF_{\qcoh}^{\C}(\xX_4, w_4) \\
	&\stackrel{f_{1\ast}}{\to} 
		\MF_{\qcoh}^{\C}(\xX_5^{(1)} \times \xX_5^{(2)}, 
	w^{(1)}+w^{(2)}  ) \stackrel{\mathrm{pr}_j}{\to}
		\MF_{\qcoh}^{\C}(\xX_5^{(1)} \times \xX_5^{(2)}, 
	w^{(1)}+w^{(2)})_{\lambda \mathchar`- \wt= j}	
		\end{align*}
	restricts to the functor
	\begin{align*}
		\MF_{\coh}^{\C}(\xX_1, w) \to \MF_{\coh}^{\C}(\xX_5^{(1)} \times \xX_5^{(2)}, 
		w^{(1)}+w^{(2)})_{\lambda \mathchar`- \wt= j}.
		\end{align*}
	By the derived base change, the above composition functor is equivalent to the 
	following composition 
	\begin{align*}
		&\MF_{\qcoh}^{\C}(\xX_1, w) \stackrel{f_2^!}{\to} \MF_{\qcoh}^{\C}(\xX_2, w_2) \stackrel{r_1^{\ast}}{\to}	
		\MF_{\qcoh}^{\C}(\xX_6, w_6) \\
		&\stackrel{q_{1\ast}}{\to} 
		\MF_{\qcoh}^{\C}(\xX_5^{(1)} \times \xX_5^{(2)}, 
		w^{(1)}+w^{(2)}  ) \stackrel{\mathrm{pr}_j}{\to}
		\MF_{\qcoh}^{\C}(\xX_5^{(1)} \times \xX_5^{(2)}, 
		w^{(1)}+w^{(2)})_{\lambda \mathchar`- \wt= j}. 	
	\end{align*}
	Since $f_2$ is a representable morphism of smooth stacks, 
	it is quasi-smooth and $f_2^!$ is given by 
	$f_2^!(-)=f_2^{\ast}(-) \otimes \omega_{f_2}$. 
	Therefore $r_1^{\ast}f_2^!$
	gives the functor 
	\begin{align*}
		r_1^{\ast} f_2^! \colon 
		\MF_{\coh}^{\C}(\xX_1, w) \to \MF_{\coh}^{\C}(\xX_6, w_6). 
		\end{align*}
		It is enough to show that the functor $\mathrm{pr}_j q_{1\ast}$
		gives the functor 
		\begin{align}\label{funct:pr}
		\mathrm{pr}_j q_{1\ast} \colon 
		\MF_{\coh}^{\C}(\xX_6, w_6) \to 
		\MF_{\coh}^{\C}(\xX_5^{(1)} \times \xX_5^{(2)}, w^{(1)}+w^{(2)})_{\lambda \mathchar`- \wt= j}.
	\end{align}
		The morphism $q_1$ factors as 
	\begin{align*}
		q_1 \colon \xX_6 \stackrel{q_1'}{\to} \xX_7 \cneq 
		\left[\left(\widehat{T}^{(1)}_0 \oplus
		\widehat{T}^{(2)}_0 \oplus 
		O^{(1)\vee} \oplus O^{(2)\vee}\right)/G^{\ext}
		\right] 
		\stackrel{q_1''}{\to} \xX_5^{(1)} \times \xX_5^{(2)}
	\end{align*}
	where $G^{\ext}$ acts on $(\widehat{T}^{(1)}_0 \oplus
	\widehat{T}^{(2)}_0 \oplus 
	O^{(1)\vee} \oplus O^{(2)\vee})$
	through the projection $G^{\ext} \twoheadrightarrow G^{(1)} \times G^{(2)}$.
	Since $T^{(21)} \oplus O^{(12)\vee}$ is of weight $-1$
	with respect to 
	$\lambda$, 
	by Lemma~\ref{lem:stack:push}
	the push-forward $q_{1\ast}'$ restricts to the functor 
	\begin{align*}
		q_{1\ast}' \colon 
		\MF_{\coh}^{\C}(\xX_6, w_6) \to 
		\MF_{\coh}^{\C}(\xX_7, w_7)_{\lambda \mathchar`- \rm{below}},
	\end{align*}
where $w_7=q_1''^{\ast}(w^{(1)}+w^{(2)})$. 
	Since 
	$q_{1\ast}''$ is a gerbe over a finite dimensional 
	linear space 
	$\oplus_{i}\Hom(V_i^{(2)}, V_i^{(1)})$, 
	the functor $q_{1\ast}''$ gives 
	\begin{align*}
		q_{1\ast}'' \colon 
		\MF_{\coh}^{\C}(\xX_7, w_7) \to 
		\MF_{\coh}^{\C}(\xX_5^{(1)} \times \xX_5^{(2)}, w^{(1)}+w^{(2)}). 
	\end{align*}
	Therefore $q_{1\ast}=q_{1\ast}'' \circ q_{1\ast}'$ 
	restricts to the functor 
	\begin{align*}
		q_{1\ast} \colon 
		\MF_{\coh}^{\C}(\xX_6, w_6) \to \MF_{\coh}^{\C}(\xX_5^{(1)} \times \xX_5^{(2)}, 
		w^{(1)}+w^{(2)})_{\lambda \mathchar`- \rm{below}},
	\end{align*}
	which concludes that $\mathrm{pr}_j q_{1\ast}$ gives
	the functor 
	 (\ref{funct:pr}).
\end{proof}

\begin{lem}\label{lem:radjoint2}
	Under Assumption~\ref{assum}, 
	the functors
	(\ref{funct:rest})
	are fully-faithful. 
\end{lem}
\begin{proof}
	We only prove the lemma for $t_+$. 
	By Lemma~\ref{lem:radjoint}, 
	it is enough to show that the natural transform
	\begin{align*}
		\id \to \ast_j^R \circ \ast_j
	\end{align*}
	is an isomorphism. 
	It is enough to this formally locally on $P_n^t(S, \beta)$. 
	So we are reduced to showing that the functor 
	\begin{align*}
		(\ev_3^{\dag})_{\ast} 
		(\ev_1^{\dag}, \ev_2^{\dag})^{\ast} \colon 
		\Dbc([\widehat{\fN}_0^{(1)} \times \widehat{\fN}_0^{(2)}/G^{(1)}
		\times G^{(2)}])_{\lambda \mathchar`- \wt= j}&/\cC_{\widehat{\zZ}_1^{t_+\us (1)} \times [\Crit(w^{(2)})/G^{(2)}]} \\
		&\to \Dbc([\widehat{\fN}_0/G])/\cC_{\widehat{\zZ}^{t_+\us}_{l-1}}
	\end{align*}
	from the diagram (\ref{dia:ext:N}) is fully-faithful.
	Let us consider the following composition functor
	\begin{align*}
		\MF_{\coh}^{\C}(\xX_5^{(1)} \times \xX_5^{(2)}, w^{(1)}+w^{(2)})_{\lambda \mathchar`- \wt= j} &\stackrel{f_1^{\ast}}{\to}
		\MF_{\coh}^{\C}(\xX_4, w_4) \stackrel{g_{1!}}{\to} \MF_{\coh}^{\C}(\xX_3, w^{\ext}) \\
		 &\stackrel{g_2^{\ast}}{\to}
		\MF_{\coh}^{\C}(\xX_2, w_2) \stackrel{f_{2\ast}}{\to}\MF_{\coh}^{\C}(\xX_1, w).
		\end{align*}
		By Lemma~\ref{lem:commute1} and Lemma~\ref{lem:commute3}, 
		it is enough to show that the descendant of the 
		above composition functor 
	\begin{align}\label{ff:mf:loc}
		f_{2\ast}g_2^{\ast}	g_{1 !} f_1^{\ast} \colon 
		\MF_{\coh}^{\C}((\xX_5^{(1)} \setminus \widehat{\zZ}^{t_+\us (1)}) \times 
		\xX_5^{(2)}, w^{(1)}+w^{(2)})_{\lambda \mathchar`- \wt= j}
		\to \MF_{\coh}^{\C}(\xX_1 \setminus \widehat{\zZ}^{t_+\us}_{l-1}, w)
	\end{align}
is fully-faithful. 
	Note that we have 
	\begin{align*}
		g_{1!}(-)=g_{1\ast}(- \otimes \det O^{(21)\vee}[-\dim O^{(21)}]). 
	\end{align*}
	Since the $G^{\ext}$-action on $O^{(21)}$ factors 
	through $G^{\ext} \twoheadrightarrow G^{(1)} \times G^{(2)}$, 
	the $G^{\ext}$-character 
	$\det O^{(21)\vee}$
	is written as $f_1^{\ast}\det O^{(21)\vee}$ where 
	we regard 
	$\det O^{(21)\vee}$ as a $(G^{(1)} \times G^{(2)})$-character. 
	Therefore we have 
	\begin{align*}
		f_{2\ast}g_2^{\ast}	g_{1 !} f_1^{\ast}(-)
		\cong q_{2\ast}q_1^{\ast}(- \otimes 
		\det O^{(21)\vee}[-\dim O^{(21)}]). 	
	\end{align*}
	Here the first functor is an equivalence 
	\begin{align}\label{otimes:O}
		\otimes \det O^{(21)\vee}[-\dim O^{(21)}] \colon 
		&\MF_{\coh}^{\C}((\xX_5^{(1)} \setminus \widehat{\zZ}^{t_+\us (1)}) \times 
		\xX_5^{(2)}, w^{(1)}+w^{(2)})_{\lambda \mathchar`- \wt= j} \\
		\notag
		&\stackrel{\sim}{\to}
		\MF_{\coh}^{\C}((\xX_5^{(1)} \setminus \widehat{\zZ}^{t_+\us (1)}) \times 
		\xX_5^{(2)}, w^{(1)}+w^{(2)})_{\lambda \mathchar`- \wt= j+\dim O^{(21)}}
	\end{align}
	By Lemma~\ref{lem:KNS}, the substack 
	$\widehat{\sS}_l^{t_{+}} \subset [\Crit(w)/G]\setminus 
	\widehat{\zZ}_{l-1}^{t_{+}\us}$ is a KN strata
	for the determinant character $\chi_0 \colon G \to \C$, 
	with 
	associated one parameter subgroup 
	$\lambda^{-1}$. 
	The center is given by 
	\begin{align*}
		\left[ \Crit(w)^{\lambda}/G^{\lambda} \right]
		\setminus \widehat{\zZ}_{l-1}^{t_+\us}=
		\left(\left[ \Crit(w^{(1)})/G^{(1)} \right] \setminus 
		\widehat{\zZ}^{t_+\us (1)}
		\right) \times [\Crit(w^{(2)})/G^{(2)}]
	\end{align*}
	by Lemma~\ref{lem:exP}. Moreover as we already mentioned, 
	the left diagram in (\ref{dia:X}) is identified with the diagram 
	(\ref{dia:closure})
	for $\lambda_{\alpha}=\lambda^{-1}$. 
	Therefore from the equivalence (\ref{equiv:bar}), 
	we conclude that 
	the functor 
	\begin{align}\label{funct:formal:loc}
		q_{2\ast}q^{\ast}_1 \colon 
		\MF_{\coh}^{\C}((\xX_5^{(1)} \setminus 
		\widehat{\zZ}^{t_+\us (1)}) \times \xX_5^{(2)}, 
		w^{(1)}+w^{(2)})_{\lambda \mathchar`- \wt= j+\dim O^{(21)}}
		\to \MF_{\coh}^{\C}(\xX_1 \setminus \widehat{\zZ}^{t_+\us}_{l-1})
	\end{align}
	is fully-faithful. Therefore the functor (\ref{ff:mf:loc}) is fully-faithful. 
\end{proof}
We denote by 
\begin{align*}
	\Upsilon_j^{\pm, \mathbf{d}} \subset \dDT^{\C}(\pP_n^t(X, \beta) \setminus \zZ_{l-1}^{t_{\pm}\us})
\end{align*}
the essential images of the fully-faithful functors (\ref{funct:rest}). 
\begin{lem}\label{lem:orthognoal}
	We have 
	\begin{align*}
		\Hom(\Upsilon_j^{+, \mathbf{d}}, \Upsilon_{j'}^{+, \mathbf{d}})=0 \ (j<j'), \ 
		\Hom(\Upsilon_j^{-, \mathbf{d}}, \Upsilon_{j'}^{-, \mathbf{d}})=0 \ (j>j'). 
	\end{align*}
\end{lem}
\begin{proof}
	We only prove the first case. 
	It is enough to show that 
	$\ast_j^R \circ \ast_{j'} \cong 0$ for $j<j'$. 
	As in the proof of Lemma~\ref{lem:radjoint}, it is enough to prove
	this formally locally on $P_n^t(S, \beta)$. 
	Note that the LHS of (\ref{funct:formal:loc})
	have weight $-j-\dim O^{(21)}$ with respect to $\lambda^{-1}$. 
	Therefore by Theorem~\ref{thm:window}, 
	the essential images of the functors (\ref{funct:formal:loc}) are 
	semiorthogonal for $j<j'$, so the lemma follows. 
\end{proof}

\begin{lem}\label{lem:ladjoint}
	The functors (\ref{funct:rest}) admit left adjoints 
	\begin{align*}
		&\ast_j^L \colon 
		 \dDT^{\C}\left(\pP^t_{n}(X, \beta) \setminus \zZ^{t_{+}\us}_{l-1} \right) \to 
		\left(\dDT^{\C}(\pP_{n_1}^{t_+}(X, \beta_1)) \otimes 
		\dDT^{\C}(\mM_{n_2}^H(X, \beta_2))\right), \\
		\notag	&\ast_j^L \colon
		\dDT^{\C}\left(\pP^t_{n}(X, \beta) \setminus \zZ^{t_{-}\us}_{l-1} \right) \to 
		\left(\dDT^{\C}(\mM_{n_2}^H(X, \beta_2))
		\otimes \dDT^{\C}(\pP_{n_1}^{t_-}(X, \beta_1)) \right). 
	\end{align*}	
\end{lem}
\begin{proof}
	We only prove the lemma for $t_+$. 
	In order to simplify the notation, we write the left diagram of (\ref{dia:MSdagP})
	as
	\begin{align*}
		\xymatrix{
			\fM_3 & \ar[l]_-{(\ev_1^{\dag}, \ev_2^{\dag})} 
			\fM_2 \ar[r]^-{\ev_3^{\dag}} & \fM_1. 
		}
	\end{align*}
	We denote by $\mathbb{D}_i$ the Serre duality 
	equivalence for $\Dbc(\fM_i)$ in (\ref{Serre}). 
	We have the following 
	adjoint pairs for the functors between 
	$\Ind \Dbc(\fM_1)$ and $\Ind \Dbc(\fM_3)_j$,
	\begin{align}\notag
		\mathrm{pr}_j 
		\mathbb{D}_{3}
		(\ev_1^{\dag}, \ev_2^{\dag})_{\ast}(\ev_3^{\dag})^{!}
		\mathbb{D}_{1} \dashv 
		\mathbb{D}_1 \ev_{3\ast}^{\dag}(\ev_1^{\dag}, \ev_2^{\dag})^{\ast} \mathbb{D}_3 i_j. 
	\end{align}
Here the above functors are given by 
\begin{align*}
	\xymatrix{
\Ind \Dbc(\fM_1)	
\ar@<0.5ex>[r]^-{\mathbb{D}_1} &
\ar@<0.5ex>[l]^-{\mathbb{D}_1}
\Ind \Dbc(\fM_1)^{\rm{op}} \ar@<0.5ex>[r]^-{(\ev_3^{\dag})^{!}} & 
\ar@<0.5ex>[l]^-{\ev_{3\ast}^{\dag}}
\Ind \Dbc(\fM_2)^{\rm{op}}  \\ 
\ar@<0.5ex>[r]^-{(\ev_1^{\dag}, \ev_2^{\dag})_{\ast}}&\ar@<0.5ex>[l]^-{(\ev_1, \ev_2)^{\ast}}
\Ind \Dbc(\fM_3)^{\rm{op}} \ar@<0.5ex>[r]^-{\mathbb{D}_3} & 
\ar@<0.5ex>[l]^-{\mathbb{D}_3}
\Ind\Dbc(\fM_3) \ar@<0.5ex>[r]^-{\mathrm{pr}_j} & 
\ar@<0.5ex>[l]^-{i_j}
\Ind\Dbc(\fM_3)_j. 
}
	\end{align*}
	Since $\ev_3^{\dag}$ is proper, 
	using~\cite[Corollary~9.5.9]{MR3136100}
	we have 
	\begin{align*}
		\mathbb{D}_1 \ev_{3\ast}^{\dag}(\ev_1^{\dag}, \ev_2^{\dag})^{\ast} \mathbb{D}_3 i_j
		\cong 	
		\ev_{3\ast}^{\dag}\mathbb{D}_2(\ev_1^{\dag}, \ev_2^{\dag})^{\ast} \mathbb{D}_3 i_j 
		\cong \ev_{3\ast}^{\dag}(\ev_1^{\dag}, \ev_2^{\dag})^{!}i_j.
	\end{align*}
	As $(\ev_1^{\dag}, \ev_2^{\dag})$ is quasi-smooth 
	and the relative 
	dualizing complex 
	$\omega_{(\ev_1^{\dag}, \ev_2^{\dag})}$ is of the form 
	$(\ev_1^{\dag}, \ev_2^{\dag})^{\ast}\lL[k]$ for a line bundle 
	$\lL$ on $\fM_3$ and $k\in \mathbb{Z}$, 
	we conclude that 
	\begin{align*}
		\mathbb{D}_1 \ev_{3\ast}^{\dag}(\ev_1^{\dag}, \ev_2^{\dag})^{\ast} \mathbb{D}_3 i_j
		\cong 	\ev_{3\ast}^{\dag}(\ev_1^{\dag}, \ev_2^{\dag})^{\ast}
		\otimes \lL \circ i_j[k]. 
	\end{align*}
	The composition $\otimes \lL \circ i_j$ is isomorphic to 
	$i_{j+\wt(\lL)} \circ \otimes \lL$ for 
	the equivalence 
	\begin{align*}
		\otimes \lL \colon 
		\Ind\Dbc(\fM_3)_j
		\stackrel{\sim}{\to}
		\Ind\Dbc(\fM_3)_{j+\wt(\lL)}.
	\end{align*}
	Therefore we have the following adjoint pairs 
	for the functors between $\Ind\Dbc(\fM_1)$ and $\Ind\Dbc(\fM_3)_{j}$
	\begin{align}\label{ladjoint:ast}
		\otimes \lL \circ \mathrm{pr}_{j-\wt(\lL)}
		\mathbb{D}_{3}
		(\ev_1^{\dag}, \ev_2^{\dag})_{\ast}(\ev_3^{\dag})^{!}
		\mathbb{D}_{1} \dashv 
		\ev_{3\ast}^{\dag}(\ev_1^{\dag}, \ev_2^{\dag})^{\ast} i_j. 
	\end{align}
	Since the dualizing functors and tensor products with 
	line bundle preserve 
	singular supports and compact objects, 
	as in the proof of Lemma~\ref{lem:adjoint:ind}
	the LHS of (\ref{ladjoint:ast}) induces the functor
	\begin{align*}
		\ast^L_j \colon \dDT^{\C}\left(\pP^t_{n}(X, \beta) \setminus \zZ^{t_{+}\us}_{l-1} \right) \to 
		\dDT^{\C}(\pP_{n_1}^{t_+}(X, \beta_1)) \otimes 
		\dDT^{\C}(\mM_{n_2}^H(X, \beta_2))_j
	\end{align*}
	giving a left adjoint of the functor (\ref{funct:rest}). 
\end{proof}
\subsubsection{Proofs of Conjectures}
Let $\ast_j^R$, $\ast_j^L$ be the 
right and left adjoints functors of $\ast_j$
given in Lemma~\ref{lem:radjoint}, Lemma~\ref{lem:ladjoint}. 
We define the following subcategories
\begin{align*}
	&\wW_{k^+(\mathbf{d})}^+
	\cneq \bigcap_{j\le k^+(\mathbf{d})} \Ker(\ast_j^R) 
	\cap \bigcap_{j>k^+(\mathbf{d})} \Ker(\ast_j^L)
	\subset
	\dDT^{\C}\left(\pP_n^t(X, \beta)\setminus \zZ_{l-1}^{t_+\us}\right), \\
	&\wW_{k^-(\mathbf{d})}^-
	\cneq \bigcap_{j\ge k^-(\mathbf{d})} \Ker(\ast_j^R) 
	\cap \bigcap_{j<k^-(\mathbf{d})} \Ker(\ast_j^L)
	\subset
	\dDT^{\C}\left(\pP_n^t(X, \beta)\setminus \zZ_{l-1}^{t_-\us}\right). 
\end{align*}

\begin{lem}\label{lem:sod:P}
	There exist semiorthogonal decompositions of the form 
	\begin{align}\notag
		\dDT^{\C}\left(\pP_n^t(X, \beta)\setminus \zZ_{l-1}^{t_+\us}\right)
		&=\left\langle \ldots, 
		\Upsilon_{\lfloor k^+(\mathbf{d}) \rfloor +2}^{+, \mathbf{d}}, \Upsilon_{\lfloor k^+(\mathbf{d}) \rfloor +1}^{+, \mathbf{d}}, 
		\wW_{k^+(\mathbf{d})}^+, \Upsilon_{\lfloor k^+(\mathbf{d}) \rfloor}^{+, \mathbf{d}}, 
		\Upsilon_{\lfloor k^+(\mathbf{d}) \rfloor -1}^{+, \mathbf{d}}, \ldots  \right\rangle, \\
		\notag
		\dDT^{\C}\left(\pP_n^t(X, \beta)\setminus \zZ_{l-1}^{t_-\us}\right)
		&=\left\langle \ldots, 
		\Upsilon_{\lceil k^-(\mathbf{d}) \rceil -2}^{-, \mathbf{d}}, \Upsilon_{\lceil k^-(\mathbf{d}) \rceil -1}^{-, \mathbf{d}}, 
		\wW_{k^-(\mathbf{d})}^-, \Upsilon_{\lceil k^-(\mathbf{d}) \rceil}^{-, \mathbf{d}}, 
		\Upsilon_{\lceil k^-(\mathbf{d}) \rceil+1}^{-, \mathbf{d}}, \ldots  \right\rangle. 	
	\end{align}
\end{lem}
\begin{proof}
	We only prove the lemma for $t_+$.  
	Note that the RHS is semiorthogonal by Lemma~\ref{lem:orthognoal} and the definition of 
	$\wW_{k^+(\mathbf{d})}^+$. 
	It is enough to show that any object in the LHS is 
	obtained as successive extensions of objects 
	in the RHS. 
	
		For an object $\eE$ in the LHS, 
	the proof of Lemma~\ref{lem:radjoint} shows that 
	$\ast_j^R(\eE)=0$ for $j \ll 0$ formally locally 
	at $p \in P_n^t(S, \beta)$. 	
	Since $\pP_n^t(S, \beta) \to P_n^t(S, \beta)$ is universally closed, 
	there exists a Zariski open neighborhood of 
	$p \in P_n^t(S, \beta)$ on which 
	$\ast_j^R(\eE)=0$ for $j\ll 0$. 
	As $P_n^t(S, \beta)$ is of finite type, 
	we conclude that $\ast_j^R(\eE)=0$ for $j\ll 0$. 
	A similar argument shows that 
	$\ast_j^L(\eE)=0$ for $j \gg 0$. 
	
	Suppose that $\eE$ is not an object in 
	$\wW_{k^+(\mathbf{d})}^+$. 
	Then there is 
	$j_1 \le k^+(\mathbf{d})$ such that 
	$\ast_j^R(\eE)=0$ for $j<j_1$ and 
	$\ast_{j_1}^R(\eE) \neq 0$, 
	or 
	there is $j_1' >k^+(\mathbf{d})$
	such that $\ast_j^L(\eE)=0$ for $j>j_1'$ and 
	$\ast_{j_1'}^L(\eE)\neq 0$. 
	Below we assume the former case. 
	The latter case is similarly discussed. 
	We have the distinguished triangle 
	$	\ast_{j_1}\ast_{j_1}^R(\eE)
	\to \eE \to \eE_1$, 
	where 
	$\ast_{j}^R(\eE_1)=0$ for $j\le j_1$. 
	Repeating the above constructions for $\eE_1$, 
	we have a distinguished triangle
	\begin{align*}
		\eE_2  \to \eE \to \eE_3, \ 
		\eE_2 \in \langle \Upsilon_{\lfloor k^+(\mathbf{d}) \rfloor}^{+, \mathbf{d}}, \ldots, \Upsilon_{j_1}^{+, \mathbf{d}}\rangle, \ 
		\eE_3 \in \bigcap_{j\le k^+(\mathbf{d})} \Ker(\ast_j^R). 
	\end{align*}
	If $\ast_j^L(\eE_3)=0$ for all $j>k^+(\mathbf{d})$, then 
	$\eE_3 \in \wW_{k^+(\mathbf{d})}^+$. 
	Otherwise there is $j_1'>k^+(\mathbf{d})$ such that 
	$\ast_j^L(\eE_3) \neq 0$ for $j>j_1'$ and 
	$\ast_{j_1'}^L(\eE_3) \neq 0$. 
	Similarly to above, 
	we have the distinguished triangle 
	$\eE_4 \to \eE_3 \to \ast_{j_1'} \ast_{j_1'}^L(\eE)$
	such that  
	$\ast_j^L(\eE_4)=0$ for $j\ge j_1'$. 
	We also have $\ast_j^R(\eE_4)=0$ for $j\le k^+(\mathbf{d})$
	by applying $\ast_j^R$ to the above triangle
	and using Lemma~\ref{lem:orthognoal}. 
	By repeating the above construction for $\eE_4$, 
	we obtain the distinguished triangle 
	\begin{align*}
		\eE_5 \to \eE_3 \to \eE_6, \ 
		\eE_6 \in \langle \Upsilon_{j_1'}^{+, \mathbf{d}}, \ldots, \Upsilon_{\lfloor k^+(\mathbf{d}) \rfloor +1}^{+, \mathbf{d}}\rangle, \ 
		\eE_5 \in \wW_{k^+(\mathbf{d})}^+. 
	\end{align*}
	Therefore we obtain the desired semiorthogonal decomposition. 
\end{proof}

Using the above preparations, we prove 
Conjectures~\ref{conj:SOD:l}, \ref{conj:SOD}. 

\begin{thm}\label{thm:pfconj}
	Under Assumption~\ref{assum}, 
	Conjecture~\ref{conj:SOD:l} and 
	Conjecture~\ref{conj:SOD} hold. 
\end{thm}
\begin{proof}
	It is enough to prove Conjecture~\ref{conj:SOD:l}. 
	Below we only prove Conjecture~\ref{conj:SOD:l} for $t_+$
	under Assumption~\ref{assum}.  
	For a decomposition $\mathbf{d} \in \mathbf{D}_{(\beta, n)}^{t, l}$
	denoted as $(\beta, n)=(\beta_1, n_1)+(\beta_2, n_2)$,  
	we already proved that the functors (\ref{funct:rest}) 
	are fully-faithful 
	in Lemma~\ref{lem:radjoint2}, whose essential images fit 
	into the semiorthogonal decomposition in Lemma~\ref{lem:sod:P}. 
	Let $\mathbf{d}' \in \mathbf{D}_{(\beta, n)}^{t, l}$ be another decomposition
	$(\beta, n)=(\beta_1', n_1')+(\beta_2', n_2')$. 
	Note that objects in $\Upsilon_j^{+, \mathbf{d}}$, $\Upsilon_j^{+, \mathbf{d}'}$ have singular supports 
	contained in 
	\begin{align*}
		(\pP_{n_1}^{t_+}(X, \beta_1) \ast \mM_{n_2}^H(X, \beta_2)) \cup \zZ_{l-1}^{t_+\us}, \ 
		(\pP_{n_1'}^{t_+}(X, \beta_1') \ast \mM_{n_2'}^H(X, \beta_2'))\cup \zZ_{l-1}^{t_+\us}
		\end{align*}
	respectively. Since they are disjoint outside 
	$\zZ_{l-1}^{t_+\us}$, by Lemma~\ref{lem:orth:N}
	$\Upsilon_j^{\pm, \mathbf{d}}$ is orthogonal to 
	$\Upsilon_{j'}^{\pm, \mathbf{d}'}$.
	Therefore from Lemma~\ref{lem:sod:P}, we 
	obtain the semiorthogonal decomposition (\ref{conj:sodv}). 
	
	Below we show that the composition functor (\ref{conj:compose})
	is an equivalence. 
	In order to simplify the argument, 
	we assume that $\mathbf{D}_{(\beta, n)}^{t, l}$ consists 
	of only $\mathbf{d}$.  
	So we show that the composition 
	\begin{align}\label{compose:Wd}
		\wW_{k^+(\mathbf{d})}^+
	\hookrightarrow 
	\dDT^{\C}\left(\pP_n^t(X, \beta)\setminus \zZ_{l-1}^{t_+\us}\right)
	\twoheadrightarrow 
	\dDT^{\C}\left(\pP_n^t(X, \beta)\setminus \zZ_{l}^{t_+\us}\right)	
		\end{align}
	is an equivalence. 	
	As objects in $\Upsilon_j^{+, \mathbf{d}}$ have singular supports 
	contained in $\zZ_{l}^{t_+\us}$, 
	the composition 
	\begin{align*}
		\Upsilon_j^{+, \mathbf{d}} \hookrightarrow 
		\dDT^{\C}\left(\pP_n^t(X, \beta)\setminus \zZ_{l-1}^{t_+\us}\right)
		\twoheadrightarrow 
		\dDT^{\C}\left(\pP_n^t(X, \beta)\setminus \zZ_{l}^{t_+\us}\right)
	\end{align*}
	is zero. Therefore the functor (\ref{compose:Wd}) is essentially surjective 
	from the semiorthogonal decomposition 
	in Lemma~\ref{lem:sod:P}. 
	
	It remains to show that the functor (\ref{compose:Wd}) is fully-faithful. 
	Let $\iota \colon U \to P_n^t(S, \beta)$ be an etale morphism as in Theorem~\ref{thm:AHR}, and 
	take the following Cartesian diagrams as in (\ref{Cartesian:U})
	\begin{align*}
		\xymatrix{
			\mathfrak{P}_U \ar[d]_-{\iota_{\fP}} \diasquare & \ar@<0.3ex>@{_{(}->}[l] \pP_U \ar[r]\ar[d] \diasquare & U \ar[d]_-{\iota} \\
			\fP_n^t(S, \beta) & \ar@<0.3ex>@{_{(}->}[l]  \pP_n^t(S, \beta) \ar[r] & P_n^t(S, \beta).  
		}
	\end{align*}
	By Lemma~\ref{lem:inter}, we 
	have the fully-faithful functor 
	\begin{align*}
		\dDT^{\C}\left(\pP_n^t(X, \beta)\setminus \zZ_{l-1}^{t_+\us}\right)
		\hookrightarrow 
		\lim_{U \stackrel{\iota}{\to} P_n^t(S, \beta)}
		\left(\Dbc(\fP_U)/\cC_{\iota_{\fP}^{\ast}\zZ_{l-1}^{t_+\us}}\right). 
	\end{align*}
	Let $\eE_1, \eE_2$ be an object in $\wW_{k^+(\mathbf{d})}^+$. 
	By the above fully-faithful functor, it is enough to show 
	that the natural morphism 
	\begin{align}\label{nat:func}
		\Hom_{\Dbc(\fP_U)/\cC_{\iota_{\fP}^{\ast}\zZ_{l-1}^{t_+\us}}}(\eE_1|_{\fP_{U}}, 
		\eE_2|_{\fP_{U}})
		\to \Hom_{\Dbc(\fP_U)/\cC_{\iota_{\fP}^{\ast}\zZ_{l}^{t_+\us}}}(\eE_1|_{\fP_{U}}, 
		\eE_2|_{\fP_{U}})	
	\end{align}
	is an isomorphism. 
	The above morphism is regarded as a morphism in $D_{\qcoh}(U)$, 
	so it is enough to show the above isomorphism 
	formally locally at any point in $U$. 
	Similarly to Lemma~\ref{lem:radjoint}, Lemma~\ref{lem:radjoint2}, we prove 
	the corresponding claim for derived categories of factorizations via 
	Koszul duality. 
	
	For each $p\in P_n^t(S, \beta)$
	corresponding to the object (\ref{pstable}) and a decomposition (\ref{decom:I}), 
	we use the notation of the diagram (\ref{dia:wext:simple}). 
	Let $\widehat{\Upsilon}^+_j$ be
	the essential image of the functor (\ref{ff:mf:loc}). 
	Since the functor (\ref{ff:mf:loc}) is the composition of (\ref{otimes:O}) and (\ref{funct:formal:loc}), 
	by Theorem~\ref{thm:window} we have the semiorthogonal decomposition 
	\begin{align}\label{SOD:l}
		\MF^{\C}_{\coh}(\xX_1 \setminus \widehat{\zZ}_{l-1}^{t_+\us}, w)
		=\langle \ldots, \widehat{\Upsilon}^+_{\lfloor k^+(\mathbf{d}) \rfloor+2}, 
		\widehat{\Upsilon}^+_{\lfloor k^+(\mathbf{d}) \rfloor+1}, 
		\widehat{\wW}_{k^{+}(\mathbf{d})}^{+}, 	\widehat{\Upsilon}^+_{\lfloor k^+(\mathbf{d}) \rfloor}, 
		\widehat{\Upsilon}^+_{\lfloor k^+(\mathbf{d}) \rfloor-1}, \ldots 
		\rangle. 
	\end{align}
	such that the composition functor 
	\begin{align*}
		\widehat{\wW}_{k^{+}(\mathbf{d})}^{+} \hookrightarrow 
		\MF^{\C}_{\coh}(\xX_1 \setminus \widehat{\zZ}_{l-1}^{t_+\us}, w)
		\to 
		\MF^{\C}_{\coh}(\xX_1 \setminus \widehat{\zZ}_{l}^{t_+\us}, w) 
	\end{align*}
	is an equivalence. Since $\ast_j^R$ and $\ast_j^L$ commute with 
	base change, 
	each object $\eE_i$ restricted to the formal fiber at $p \in P_n^t(S, \beta)$
	corresponds to an object in $\widehat{\wW}_{k^+(\mathbf{d})}^+$
	under the equivalence (\ref{equiv:MF})
	\begin{align*}
		\Dbc([\widehat{\fN}_0/G])/\cC_{ \widehat{\zZ}_{l-1}^{t_+\us}} \stackrel{\sim}{\to}
		\MF^{\C}_{\coh}(\xX_1 \setminus \widehat{\zZ}_{l-1}^{t_+\us}, w).
		\end{align*}
	Therefore we conclude 
	the formal local isomorphism of (\ref{nat:func}). 
\end{proof}

We next prove Conjecture~\ref{conj:FF}
\begin{thm}\label{thm:conj2}
	Under Assumption~\ref{assum}, 
	Conjecture~\ref{conj:FF} holds. 
\end{thm}
\begin{proof}
	For each point $p \in P_n^t(S, \beta)$
	corresponding to a polystable object $I$ as in (\ref{pstable})
	and a decomposition (\ref{decom:I}), 
	we consider the diagram (\ref{dia:wext:simple})
	as before. 
	Let $\lambda$ be the one parameter subgroup (\ref{lambda}). 
	We set 
	\begin{align*}
		\eta^+ &\cneq \wt_{\lambda^{-1}}(\det \mathbb{L}_{q_2}^{\vee}|_{0})
		=\hom(I^{(1)}, F^{(2)})+\ext^2(F^{(2)}, F^{(1)})-\hom^{-1}(I^{(1)}, F^{(2)}), \\ 
		\eta^- &\cneq \wt_{\lambda}(\det \mathbb{L}_{q_2'}^{\vee}|_{0})
		=\ext^1(F^{(2)}, F^{(1)})+\hom^1(I^{(1)}, F^{(2)})-\hom(F^{(2)}, F^{(1)}). 
	\end{align*}
	We also denote by 
	\begin{align*}
		i^+ \colon (\xX_5^{(1)} \setminus \widehat{\zZ}^{t_+\us (1)}) \times \xX_5^{(2)} \to 
		\xX_1, \ 
		i^- \colon \xX_5^{(2)} \times (\xX_5^{(1)} \setminus  \widehat{\zZ}^{t_-\us (1)})
		\to \xX_1	
	\end{align*}
	the inclusions into the 
	$\lambda$-fixed parts. 
	We denote by 
	\begin{align}\label{window:hat}
		\widehat{\wW}_{k^{\pm}}^{\pm} \subset \MF^{\C}_{\coh}(\xX_1, w)
	\end{align}
	the subcategory of objects $\pP$ such that 
	for any decomposition (\ref{decom:I}) of $I$ we have 
	\begin{align}\label{wt:cond}
		&\wt_{\lambda^{-1}}(i^{+\ast}\pP) \subset [-k^+(\mathbf{d})-\dim O^{(21)}, -k^+(\mathbf{d})-\dim O^{(21)}+\eta^+), \\
		\notag
		&\wt_{\lambda}(i^{-\ast}\pP) \subset [k^-(\mathbf{d})-\dim O^{(12)}, k^-(\mathbf{d})-\dim O^{(12)}+\eta^-)
	\end{align}
	respectively. 
	Here $\mathbf{d}$ denotes the decomposition 
	$(\beta, n)=(\beta_1, n_1)+(\beta_2, n_2)$ 
	with $\cl(I^{(1)})=(1, \beta_1, n_1)$
	and $[F^{(2)}]=(\beta_2, n_2)$. 
	We set $k^{\pm}(\mathbf{d})$ to be 
	\begin{align*}
		k^+(\mathbf{d})&=\frac{1}{2}\chi(I^{(1)}, F^{(2)})=\frac{1}{2}(n_2+\beta_1 \beta_2), \\
		k^-(\mathbf{d}) &=\frac{1}{2}\chi(F^{(2)}, F^{(1)})=-\frac{1}{2}\beta_1 \beta_2.
	\end{align*}
	Then we have 
	\begin{align*}
		k^+(\mathbf{d})+\dim O^{(21)}&=
		\frac{1}{2}\hom(I^{(1)}, F^{(2)})-\frac{1}{2}\hom^{-1}(I^{(1)}, F^{(2)})
		-\frac{1}{2}\hom^{1}(I^{(1)}, F^{(2)})+\ext^2(F^{(2)}, F^{(1)}) \\
		&=\frac{1}{2}\eta^{+}
		+\frac{1}{2}(\ext^2(F^{(2)}, F^{(1)})-\hom^1(I^{(1)}, F^{(2)})) \\
		&=\frac{1}{2}\eta^+ +\frac{1}{2}\langle \delta_0, \lambda^{-1} \rangle. 
	\end{align*}
	Here $\delta_0$ is the following $G$-character 
	\begin{align*}
		\delta_0 \cneq \det \Hom^1(I, F)=\bigotimes_{i, j}
		\det O^{(ij)}
	\end{align*}
	which satisfies that 
	\begin{align*}
		\langle \delta_0, \lambda \rangle =\hom^1(I^{(1)}, F^{(2)})-\ext^2(F^{(2)}, F^{(1)}). 
	\end{align*}
	Therefore the first condition 
	in (\ref{wt:cond}) is equivalent to 
	\begin{align}\label{wt:cond2}
		\wt_{\lambda^{-1}}(i^{+\ast}\pP) \subset
		\left[-\frac{1}{2}\eta^+ -\frac{1}{2}\langle 
		\delta_0, \lambda^{-1} \rangle, 
		\frac{1}{2}\eta^+ -\frac{1}{2}\langle
		\delta_0, \lambda^{-1} \rangle\right).
	\end{align}
	A similar computation shows that the second condition in (\ref{wt:cond}) 
	is equivalent to 
	\begin{align}\label{wt:cond2.5}
		\wt_{\lambda}(i^{-\ast}\pP) \subset
		\left[-\frac{1}{2}\eta^- -\frac{1}{2}\langle 
		\delta_0, \lambda \rangle, \frac{1}{2}\eta^- -\frac{1}{2}\langle 
		\delta_0, \lambda \rangle\right). 
	\end{align}

	By Theorem~\ref{thm:window} and an argument of Theorem~\ref{thm:pfconj}, 
	an object $\eE \in \dDT^{\C}(\pP_n^t(X, \beta))$ lies in 
	$\wW_{k^{\pm}}^{\pm}$ if and only if 
	for any point $p \in P_n^t(S, \beta)$ 
	the restriction of $\eE$ to the formal fiber at $p$ 
	corresponds to an object in $\widehat{\wW}^{\pm}_{k^{\pm}}$ under the 
	Koszul duality in Theorem~\ref{thm:knoer},
		\begin{align*}
		\Dbc([\widehat{\fN}_0/G]) \stackrel{\sim}{\to}
		\MF^{\C}_{\coh}(\xX_1, w).
	\end{align*}	
	Therefore it is enough to show the inclusion 
	\begin{align}\label{inc:W}
		\widehat{\wW}_{k^-}^- \subset \widehat{\wW}_{k^+}^+. 
	\end{align}
	Note that we have 
	\begin{align*}
		\hom(I_0, F_i)+\ext^2(F_i, F_0)-\hom^1(I_0, F_i) -\ext^1(F_i, F_0)
		&=\chi(I_0, F_i)+\chi(F_i, F_0) \\
		&=\chi(F_i)>0.
	\end{align*}
Here the last inequality follows from 
$\chi(F_i)=t \cdot (H \cdot [F_i]) >0$.  
The above inequality implies that, by Remark~\ref{rmk:IF}, the multiplicity of 
$V_i$ in the $G$-representation $\Hom(I, F)$ is strictly bigger 
that of $V_i^{\vee}$. 
So $\Hom(I, F)$ is the space of quiver representations
satisfying the condition in (\ref{cond:Eij}). 
Moreover by Lemma~\ref{lem:KNS}, 
the subcategories in (\ref{inc:W})
are window subcategories with respect to $\chi_0^{\pm}$. 
Therefore we can try to apply Proposition~\ref{prop:inclu:W} 
to conclude the inclusion (\ref{wt:cond2}). 

However there is a slight issue, as $\delta_0$ is not 
necessary $\chi_{0}$-generic. 
Instead we discuss by perturbing $\delta_0$ as follows. 
Let $0<\varepsilon \ll 1$ be an irrational number, 
and set $\delta_0'=\delta_0+\varepsilon \chi_{0}$. 
Then $\delta_0'$ is $\chi_{0}$-generic. 
Since $\langle \chi_0, \lambda \rangle >0$, 
the condition (\ref{wt:cond2.5}) is equivalent to 
\begin{align*}
		\wt_{\lambda}(i^{-\ast}\pP) \subset
\left[-\frac{1}{2}\eta^- -\frac{1}{2}\langle 
\delta_0', \lambda \rangle, \frac{1}{2}\eta^- -\frac{1}{2}\langle 
\delta_0', \lambda \rangle\right]. 
\end{align*}
Moreover since 
we have 
\begin{align}\label{eta:diff}
	\eta^+-\eta^- &= \chi(I^{(1)}, F^{(2)})+\chi(F^{(2)}, F^{(1)}) \\
\notag	&=\chi(F^{(2)}) =n_2>0, 
\end{align}
we have 
\begin{align*}
\left[-\frac{1}{2}\eta^- -\frac{1}{2}\langle 
\delta_0', \lambda^{-1} \rangle, \frac{1}{2}\eta^- -\frac{1}{2}\langle 
\delta_0', \lambda^{-1} \rangle\right]
\subset 
\left[-\frac{1}{2}\eta^+ -\frac{1}{2}\langle 
\delta_0, \lambda^{-1} \rangle, \frac{1}{2}\eta^+ -\frac{1}{2}\langle 
\delta_0, \lambda^{-1} \rangle\right). 
\end{align*}
Therefore we can directly apply the argument of Proposition~\ref{cor:magic}
(also see Corollary~\ref{corcor:magic})
	to show 
	that the LHS of (\ref{inc:W}) 
	is the magic window over $\bS$ with respect to $\delta_0'$, where 
	$\bS$ coincides with the symmetric structure of $\Hom(I, F)$ defined as in 
	(\ref{sym:SU}), and the RHS of (\ref{inc:W})
	contains this magic window so 
	that 
	the inclusion (\ref{inc:W}) holds. 
\end{proof}

Finally we prove Conjecture~\ref{conj:simple}. 
\begin{thm}\label{thm:simple}
	Under Assumption~\ref{assum}, 
	Conjecture~\ref{conj:simple} holds. 	
\end{thm}
\begin{proof}
	For simplicity, we assume that $\mathbf{D}_{(\beta, n)}^{t, 1}$ consists of one element $\mathbf{d}$, 
	denoted as $(\beta, n)=(\beta_1, n_1)+(\beta_2, n_2)$. 
	We take $k^{\pm}(\mathbf{d})$ as in (\ref{kpm}). 
	It is enough to show the following identity in 
	$\dDT^{\C}(\pP_n^t(X, \beta))$
	\begin{align}\label{SOD:simples}
		\wW_{k^+}^+=
		\left \langle  
		\Upsilon^{-, \mathbf{d}}_{[-\frac{1}{2}\beta_1 \beta_2-\frac{1}{2}n_2, 
			-\frac{1}{2}\beta_1 \beta_2 )}, \wW_{k^-}^-, 
		\Upsilon^{-, \mathbf{d}}_{[-\frac{1}{2}\beta_1 \beta_2, 
			-\frac{1}{2}\beta_1 \beta_2+\frac{1}{2}n_2 )}
		\right \rangle. 	
	\end{align}	
	Let us take a point $p\in P_n^t(S, \beta)$ corresponding to
	a polystable object (\ref{pstable}). If it is not stable, 
	then the simple wall condition implies that $k=1$ and $\dim V_1=1$, 
	so that $G=\C$. 
	The semiorthogonal decomposition (\ref{SOD:l}) in this case is
	\begin{align*}
		\MF^{\C}_{\coh}(\xX_1, w)
		=\langle \ldots, \widehat{\Upsilon}^+_{\lfloor k^+(\mathbf{d}) \rfloor+2}, 
		\widehat{\Upsilon}^+_{\lfloor k^+(\mathbf{d}) \rfloor+1}, 
		\widehat{\wW}_{k^{+}}^{+}, 	\widehat{\Upsilon}^+_{\lfloor k^+(\mathbf{d}) \rfloor}, 
		\widehat{\Upsilon}^+_{\lfloor k^+(\mathbf{d}) \rfloor-1}, \ldots 
		\rangle
	\end{align*}
	where $\widehat{\wW}_{k^+}^+$ is the same as (\ref{window:hat}). 
	We have the similar semiorthogonal decomposition for $t_-$-part
	\begin{align*}
		\MF^{\C}_{\coh}(\xX_1, w)
		=\langle \ldots, \widehat{\Upsilon}^-_{\lceil k^-(\mathbf{d}) \rceil-2}, 
		\widehat{\Upsilon}^-_{\lceil k^-(\mathbf{d}) \rceil-1}, 
		\widehat{\wW}_{k^{-}}^{-}, 	\widehat{\Upsilon}^-_{\lceil k^-(\mathbf{d}) \rceil}, 
		\widehat{\Upsilon}^-_{\lceil k^-(\mathbf{d}) \rceil+1}, \ldots 
		\rangle. 
	\end{align*}
	Here $\widehat{\Upsilon}^-_{j}$ is the essential image of the fully-faithful functor 
	from the right diagram of (\ref{dia:wext:simple})
	\begin{align}\notag
		f'_{2\ast}{g_2'}^{\ast}g'_{1 !} {f_1'}^{\ast} \colon 
		\MF_{\coh}^{\C}(\xX_5^{(2)}\times 
		\xX_5^{(1)}, w^{(2)}+w^{(1)})_{\lambda \mathchar`- \wt= j}
		\to \MF_{\coh}^{\C}(\xX_1, w),
	\end{align}
	where similarly to the proof Lemma~\ref{lem:radjoint} the above functor is written as 
	\begin{align*}
		f'_{2\ast}{g_2'}^{\ast}	g'_{1 !} {f_1'}^{\ast}(-)
		\cong q'_{2\ast}{q_1'}^{\ast}(- \otimes 
		\det O^{(12)\vee}[-\dim O^{(12)}]). 	
	\end{align*}
	Note that the condition (\ref{wt:cond2})
	is equivalent to 
	\begin{align*}
		\wt_{\lambda}(i^{+\ast}\pP) \subset
		\left(-\frac{1}{2}\eta^+ -\frac{1}{2}\langle 
		\delta_0, \lambda \rangle, 
		\frac{1}{2}\eta^+ -\frac{1}{2}\langle 
		\delta_0, \lambda \rangle\right].
	\end{align*}
By (\ref{eta:diff}), 
	comparing with the
	condition (\ref{wt:cond2.5})
	the following identity is well-known in~\cite{MR3327537, MR3895631}
	(also see~\cite[Theorem~5.3]{KoTo}) for variation of GIT quotients for $\C$-actions
	\begin{align*}
		\widehat{\wW}_{k^+}^+=
		\langle 	\widehat{\Upsilon}^-_{\lceil k^-(\mathbf{d})-n_2/2 \rceil}, \ldots
		\widehat{\Upsilon}^-_{\lceil k^-(\mathbf{d}) \rceil-1}, 
		\widehat{\wW}_{k^{-}}^{-}, 	\widehat{\Upsilon}^-_{\lfloor k^-(\mathbf{d}) \rfloor}, \ldots, 
		\widehat{\Upsilon}^-_{\lfloor k^-(\mathbf{d})+n_2/2 \rfloor}
		\rangle.
	\end{align*}
	This means that the identity 
	(\ref{SOD:simples}) holds formally locally at $p\in P_n^t(S, \beta)$. 
	Since this holds for any $p$, the identity
	 (\ref{SOD:simples}) holds. 
\end{proof}

\section{Some auxiliary results}\label{sec:technical}
In this section, we prove several technical results which 
were postponed in previous sections. 
\subsection{Comparisons of DT categories}
\label{sec:compare}
Here we give comparisons of DT categories in Definition~\ref{defi:catDT}
and Definition~\ref{def:DThat}. 
Below we use the following: 
\begin{thm}\label{thm:quotient}\emph{(\cite[Theorem~7.2.1]{MR2681709})}
	Let $\dD$ be a triangulated category and 
	$\dD' \subset \dD$ a thick triangulated subcategory 
	closed under taking small coproducts. 
	Suppose that $\dD$ and $\dD'$ are compactly generated. 
	Then $\dD/\dD'$ is also compactly generated, 
	and we have the fully-faithful functor 
	$\dD^{\rm{cp}}/{\dD'}^{\rm{cp}} \hookrightarrow (\dD/\dD')^{\rm{cp}}$
	with dense image. 
	Here $\dD^{\rm{cp}} \subset \dD$ is the subcategory of 
	compact objects. 
\end{thm}
\subsubsection{Proof of Proposition~\ref{lem:indDT}}\label{subsec:eqDT}
\begin{proof} 
	Applying Theorem~\ref{thm:quotient}
	for the subcategory 
	$\Ind \cC_{\zZ}\subset \Ind \Dbc(\fM)$
	and using the assumption on the compact generation of $\Ind \cC_{\zZ}$
	and Theorem~\ref{thm:QCA}, 
	we have the fully-faithful functor
	\begin{align*}
		\Dbc(\fM)/\cC_{\zZ}\hookrightarrow \left(
		\Ind \Dbc(\fM)/\Ind \cC_{\zZ}
		\right)^{\rm{cp}}
	\end{align*}
	with dense image. Moreover 
	$\Ind \Dbc(\fM)/\Ind \cC_{\zZ}$ is compactly 
	generated, so we have the equivalence
	\begin{align}\label{equiv:Ind}
			\Ind \left( \Dbc(\fM)/\cC_{\zZ} \right)
			\stackrel{\sim}{\to}
		\Ind \Dbc(\fM)/\Ind \cC_{\zZ}.
		\end{align}
	
	Applying the equivalence (\ref{equiv:Ind})
	for each smooth morphism $\alpha \colon \fU \to \fM$, 
	we have the localization sequence of triangulated categories
	\begin{align}\label{ex:onU}
		\Ind(\cC_{\alpha^{\ast}\zZ}) \stackrel{i_{\fU}}{\to} 
		\Ind(\Dbc(\fU)) \stackrel{j_{\fU}}{\to}
		\Ind\left(\Dbc(\fU)/\cC_{\alpha^{\ast}\zZ}\right). 
	\end{align}
	By~\cite[Section~4.3]{MR3300415}, 
	the functor $i_{\fU}$ admits a right adjoint 
	$i_{\fU}^R$.
	Therefore there also exists a right adjoint $j_{\fU}^R$ of $j_{\fU}$
	such that for each $\eE \in \Ind(\Dbc(\fU))$ we have the 
	exact triangle
	\begin{align*}
		i_{\fU} \circ i_{\fU}^R(\eE) \to \eE \to j_{\fU}^R \circ j_{\fU}(\eE).
	\end{align*}
	By the above triangle, we have $j_{\fU} \circ j_{\fU}^R \cong \id$. 
	On the other hand by taking the limit of the sequence (\ref{ex:onU}), 
	we obtain the sequence
	\begin{align}\label{exact:ad}
		\Ind \cC_{\zZ} \to \Ind \Dbc(\fM) \stackrel{j}{\to}
		\lim_{\fU\stackrel{\alpha}{\to}\fM}
		\Ind\left(\Dbc(\fU)/\cC_{\alpha^{\ast}\zZ}\right).
	\end{align}
	Since $j_{\fU}^R$ is functorial for $\fU$, 
	the functor $j$ also admits a right adjoint $j^R$
	by taking the limit of $j_{\fU}^R$ 
	such that $j\circ j^R \cong \id$. 
	Also by the construction of $j$, 
	we have that $\Ker(j)=\Ind \cC_{\zZ}$. 
	By~\cite[Lemma~3.4]{Hperf}, 
	the above properties of the sequence (\ref{exact:ad})
	implies that 
	the sequence (\ref{exact:ad}) is a localization 
	sequence, i.e. 
	we have the equivalence
	\begin{align*}
		\Ind \Dbc(\fM)/\Ind \cC_{\zZ} \stackrel{\sim}{\to}
		\lim_{\fU\stackrel{\alpha}{\to}\fM}
		\Ind\left(\Dbc(\fU)/\cC_{\alpha^{\ast}\zZ}\right).
	\end{align*}
	We have the commutative diagram
	\begin{align*}
		\xymatrix{
			\Dbc(\fM)/\cC_{\zZ} \ar[r]\ar[d] & 
			\lim_{\fU \stackrel{\alpha}{\to}\fM}
			\left(\Dbc(\fU)/\cC_{\alpha^{\ast}\zZ} \right) \ar[d] \\
			\left(\Ind \Dbc(\fM)/\Ind \cC_{\zZ} \right)^{\rm{cp}}
			\ar[r]^-{\sim} &
			\left(\lim_{\fU \stackrel{\alpha}{\to}\fM}
			\Ind\left(
			\Dbc(\fU)/\cC_{\alpha^{\ast}\zZ}\right) \right)^{\rm{cp}}.
		}
	\end{align*}
	Here the left vertical arrow is fully-faithful with dense image, 
	the right vertical arrow 
	exists by Lemma~\ref{lem:c:inc}
	and it is fully-faithful.
	Therefore the top horizontal arrow is also full-faithful with dense image. 
\end{proof}

We 
have used the following lemma: 
\begin{lem}\label{lem:c:inc}
	Any object in the subcategory
	\begin{align}\label{compact:lim}
		\lim_{\fU \stackrel{\alpha}{\to}\fM}
		\left(\Dbc(\fU)/\cC_{\alpha^{\ast}\zZ} \right)
		\subset 
		\lim_{\fU \stackrel{\alpha}{\to}\fM}
		\Ind\left(
		\Dbc(\fU)/\cC_{\alpha^{\ast}\zZ}\right) 
	\end{align}
	is a compact object. 
\end{lem}
\begin{proof}
	The lemma is proved for $\zZ=\emptyset$ 
	in~\cite[Proposition~3.4.2 (b)]{MR3037900}. 
	Since the action of $D_{\rm{qcoh}}(\fU)$ on 
	$\Ind (\Dbc(\fU))$ by taking tensor products 
	preserves
	$\Ind (\cC_{\alpha^{\ast}\zZ})$
	(see~\cite[Lemma~4.2.2]{MR3300415}), the
	quotient category 
	$\Ind(\Dbc(\fU))/\Ind(\cC_{\alpha^{\ast}\zZ})=
	\Ind(\Dbc(\fU)/\cC_{\alpha^{\ast}\zZ})$ is a module over 
	$D_{\rm{qcoh}}(\fU)$. 
	Therefore the proof of~\cite[Proposition~3.4.2 (b)]{MR3037900} 
	applies verbatim. 
	Here we give the argument for the reader's convenience. 
	
	Let $\eE_1$ be an object in the LHS of (\ref{compact:lim}), and $\eE_2$ an object 
	in the RHS of (\ref{compact:lim}). 
	For each smooth morphism $\alpha \colon \fU \to \fM$, 
	we have the objects 
	$\alpha^{\ast}\eE_1 \in \Dbc(\fU)/\cC_{\alpha^{\ast}\zZ}$
	and $\alpha^{\ast}\eE_2 \in \Ind \left( \Dbc(\fU)/\cC_{\alpha^{\ast}\zZ} \right)$. 
	We have its inner Hom
	\begin{align*}
		\hH om(\alpha^{\ast}\eE_1, \alpha^{\ast}\eE_2) \in D_{\qcoh}(\fU),
		\end{align*}
	determined by the property that for any $\fF \in D_{\qcoh}(\fU)$ we have 
	\begin{align*}
		\Hom_{D_{\qcoh}(\fU)}(\fF, 	\hH om(\alpha^{\ast}\eE_1, \alpha^{\ast}\eE_2))
		\cong \Hom_{\Ind \left( \Dbc(\fU)/\cC_{\alpha^{\ast}\zZ} \right)}
		(\fF \otimes \alpha^{\ast}\eE_1, \alpha^{\ast}\eE_2). 
		\end{align*}
	Since $\alpha^{\ast}\eE_1$ is a compact object in 
	$\Ind(\Dbc(\fU)/\cC_{\alpha^{\ast}\zZ})$, the assignment 
	$\eE_2 \mapsto \hH om(\alpha^{\ast}\eE_1, \alpha^{\ast}\eE_2)$ 
	preserves colimit. Moreover for 
	each smooth morphism $\rho \colon \fU' \to \fU$ as in 
	the diagram (\ref{dia:smooth}), 
	we have the natural morphism 
	\begin{align*}
		\rho^{\ast} \hH om(\alpha^{\ast}\eE_1, \alpha^{\ast}\eE_2) 
		\to \hH om({\alpha'}^{\ast}\eE_1, {\alpha'}^{\ast}\eE_2)
		\end{align*}
	which is an isomorphism by the argument of~\cite[Lemma~3.4.4]{MR3037900}. 
	Therefore we have the object
	\begin{align*}
	\hH om(\eE_1, \eE_2) =\{\hH om(\alpha^{\ast}\eE_1, \alpha^{\ast}\eE_2)\}_{\fU \stackrel{\alpha}{\to} \fM}
	 \in D_{\qcoh}(\fM). 
		\end{align*}
	Moreover the assignment 
	$\eE_2 \mapsto \hH om(\eE_1, \eE_2)$ 
	preserves colimit. 
	We have 
	\begin{align*}
		\Hom(\eE_1, \eE_2)=
		\Gamma(\fM, \hH om(\eE_1, \eE_2)). 
		\end{align*}
	As $\fM$ is QCA, 
	the functor $\Gamma(\fM, -)$ preserves colimit by~\cite[Theorem~1.4.2]{MR3037900}. 
	Therefore the assignment 
	$\eE_2 \mapsto \Hom(\eE_1, \eE_2)$ preserves colimit, 
	which implies that $\eE_1$ is a compact object
	in the RHS of (\ref{compact:lim}). 
\end{proof}
In general, it is not known whether $\Ind \cC_{\zZ}$ is compactly 
generated or not (see~\cite[Remark~8.12]{MR3300415}). 
This is the case for a derived stack in Definition~\ref{def:tuple} by 
Lemma~\ref{lem:gcomplete}, so we obtain the following: 

\begin{cor}\label{cor:UGlimit}
	Let $\fM=[\fU/G]$ be a derived stack as in Definition~\ref{def:tuple}, 
	and $\zZ \subset t_0(\Omega_{\fM}[-1])$ a conical 
	closed substack. 
	Then the natural functor
	\begin{align*}
		\Dbc([\fU/G])/\cC_{\zZ} \to
		\lim_{\fU'\stackrel{\alpha'}{\to} [\fU/G]}
		\left(\Dbc(\fU')/\cC_{{\alpha'}^{\ast}\zZ}\right)
	\end{align*}
	is fully-faithful with dense image. 
	Here $\fU'$ is an affine derived scheme of the form (\ref{frak:U}),
	and $\alpha'$ is a smooth morphism. 
\end{cor}
\begin{proof}
	The corollary follows from Lemma~\ref{lem:gcomplete} and 
	Proposition~\ref{lem:indDT}. 
\end{proof}

\subsubsection{Restrictions to open substacks}
Let $\fM$ be a quasi-smooth and QCA derived stack 
with classical truncation $\mM=t_0(\fM)$. 
Let $\wW \subset \mM$ be a closed 
substack, and take the  
derived open substack
$\mathfrak{M}_{\circ} \subset \mathfrak{M}
$
whose truncation is $\mM \setminus \wW$. 
We define 
$\Ind D^{b}_{\rm{coh}}(\mathfrak{M})_{\wW}$ to be the kernel 
of the restriction functor
\begin{align}\notag
	\Ind D^{b}_{\rm{coh}}(\mathfrak{M})_{\wW} \cneq 
	\Ker \left( \Ind D^{b}_{\rm{coh}}(\mathfrak{M}) \to 
	\Ind D^{b}_{\rm{coh}}(\mathfrak{M}_{\circ})   \right),
\end{align}
and set $\Dbc(\fM)_{\wW} \cneq \Ind \Dbc(\fM)_{\wW} \cap \Dbc(\fM)$. 
Let $p_0 \colon \nN=t_0(\Omega_{\fM}[-1]) \to \mM$ be the projection 
and set $\zZ=p_0^{-1}(\wW)$. 
By~\cite[Corollary~4.1.2]{MR3300415}, 
we have 
\begin{align}\label{eq:IndCW}
	\Ind \cC_{\zZ}=\Ind \Dbc(\fM)_{\wW}, \ 
	\cC_{\zZ}=\Dbc(\fM)_{\wW},
	\end{align}
	as subcategories in $\Ind \Dbc(\fM)$. 
	We have the following generalization of Theorem~\ref{thm:QCA}: 
\begin{lem}\label{lem:compactW}
	The category $\Ind \Dbc(\fM)_{\wW}$ is compactly generated 
	with compact objects $\Dbc(\fM)_{\wW}$. 
	\end{lem}
\begin{proof}
Using the fact that $\Coh(\mM)_{\wW} \cneq \Ind \Dbc(\fM)_{\wW} \cap \Coh(\mM)$
is the heart of a t-structure on $\Dbc(\fM)_{\wW}$, 
the same argument of~\cite[Proposition~3.5.1]{MR3037900}
shows that $\Coh(\mM)_{\wW}$ generates 
$\Ind \Dbc(\fM)_{\wW}$. 
Then the arguments of~\cite[Theorem~3.3.5]{MR3037900}
apply 
verbatim to show that $\Ind \Dbc(\fM)_{\wW}$ is 
compactly generated with compact objects
$\Dbc(\fM)_{\wW}$. 	
We give the argument for the completeness. 

We show that $\Coh(\mM)_{\wW}$ generates
$\Ind \Dbc(\fM)_{\wW}$. 
Since $\Ind \Dbc(\fM)_{\wW}$ is generated by the image of 
$\iota_{\ast} \colon 
\Ind \Dbc(\mM)_{\wW} \to \Ind \Dbc(\fM)_{\wW}$ where 
$\iota \colon \mM \hookrightarrow \fM$ is the
natural closed immersion,
we can assume that $\fM=\mM$. 
By the similar argument, we can also assume that $\mM$ is reduced. 
Then as $\mM$ is QCA, it admits a finite stratification $\mM_i$ 
such that each strata $\mM_i$ is smooth. 
Let $\wW_i=\wW \cap \mM_i$. 
For each strata $\mM_i$, 
the category $\Dbc(\mM_i)_{\wW_i}$
is compactly generated with compact objects 
$\Dbc(\mM_i)_{\wW_i}$ by~\cite[Corollary~9.2.7, 9.2.8]{MR3300415}. 
Since any object in the latter category is a successive extension of objects in 
$\Coh(\mM_i)_{\wW_i}$ up to shift, 
the assertion holds on $\mM_i$. 
Therefore by induction, it is enough to show the following:
for a closed substack $\mM_1 \subset \mM$ and 
$\mM_2 \cneq \mM \setminus \mM_1$, 
if $\Ind \Dbc(\mM_i)_{\wW_i}$
is generated by $\Coh(\mM_i)_{\wW_i}$, 
then $\Ind \Dbc(\mM)_{\wW}$ is generated by $\Coh(\mM)_{\wW}$. 

For $\fF \in \Ind \Dbc(\mM)_{\wW}$, suppose that 
$\Hom(\eE, \fF)=0$ for all $\eE \in \Coh(\mM)_{\wW}$. 
We need to show that $\fF=0$. 
Let $j_i \colon \mM_i \to \mM$ be the inclusion. 
We have the distinguished triangle
\begin{align*}
	j_{1\ast}^{\ind}j_1^!(\fF) \to \fF \to j_{2\ast}^{\ind}j_2^{\ind \ast}\fF. 
	\end{align*}
For any $\fF' \in \Coh(\mM_1)_{\wW_1}$, we have 
\begin{align*}
	\Hom_{\mM_1}(\fF', j_1^!(\fF))=\Hom_{\mM}(j_{1\ast}^{\ind}\fF', \fF)=0
	\end{align*}
since $j_{1\ast}^{\ind}\fF' \in \Coh(\mM)_{\wW}$. 
Therefore $j_1^!(\fF)=0$, so the natural map 
$\fF \to j_{2\ast}^{\ind} j_2^{\ind \ast}\fF$
is an isomorphism. 
In particular for every $\eE \in \Ind \Dbc(\fM)_{\wW}$, we have 
the isomorphism
\begin{align*}
	\Hom_{\mM}(\eE, \fF) \stackrel{\cong}{\to} \Hom_{\mM_2}(j_2^{\ind\ast}\eE, j_2^{\ind\ast}\fF). 
	\end{align*}
Since $j_{2}^{\ast} \colon \Coh(\mM)_{\wW} \to \Coh(\mM_2)_{\wW_2}$ is essentially 
surjective (see the argument 
of~\cite[Corollary~15.5]{MR1771927}), we also conclude that $j_{2}^{\ind\ast}\fF=0$, 
therefore $\fF=0$. 

It remains to show that compact objects of $\Ind \Dbc(\fM)_{\wW}$ 
coincide with $\Dbc(\fM)_{\wW}$. Since 
the inclusion $\Ind \Dbc(\fM)_{\wW} \hookrightarrow \Ind \Dbc(\fM)$
admits a continuous right adjoint (see (\ref{singular:adjoint})), we have 
\begin{align*}
	(\Ind \Dbc(\fM)_{\wW})^{\rm{cp}}=(\Ind \Dbc(\fM))^{\rm{cp}} \cap \Ind \Dbc(\fM)_{\wW}. 
	\end{align*}
Therefore the claim holds from Theorem~\ref{thm:QCA}. 
	\end{proof}

\begin{rmk}
	For a conical closed substack $\zZ \subset t_0(\Omega_{\fM}[-1])$, 
	the subcategory $\Ind \cC_{\zZ} \cap \Coh(\mM) \subset \cC_{\zZ}$ is not
	necessary
	the heart of a t-structure on $\cC_{\zZ}$ in general, so the above 
	argument does not apply 
	to the compact generation of $\Ind \cC_{\zZ}$. 
	In Subsection~\ref{subsub:cpt} we will show the compact 
	generation of $\Ind \cC_{\zZ}$ when $\mM$ admits a good moduli space. 
\end{rmk}

\begin{lem}\label{lem:localization}
	The sequence of triangulated categories
	\begin{align}\notag
		D^{b}_{\rm{coh}}(\mathfrak{M})_{\wW} \to
		D^{b}_{\rm{coh}}(\mathfrak{M}) \to 
		D^{b}_{\rm{coh}}(\mathfrak{M}_{\circ})
	\end{align}
	is a localization sequence. 
	\end{lem}
\begin{proof}
	By~\cite[Section~4.1]{MR3136100}, 
	the sequence of triangulated categories 
	\begin{align}\notag
		\Ind \Dbc(\fM)_{\wW} \to \Ind \Dbc(\fM) \to 
		\Ind \Dbc(\fM_{\circ})
	\end{align}	
	is a localization sequence. 
	Therefore by~Theorem~\ref{thm:quotient}, Theorem~\ref{thm:QCA} 
	and Lemma~\ref{lem:compactW}, we have the induced 
	fully-faithful 
	functor
	\begin{align}\label{quotient:surj}
		D^b_{\rm{coh}}(\mathfrak{M})/D_{\rm{coh}}^b(\mathfrak{M})_{\wW}
		\hookrightarrow D^b_{\rm{coh}}(\mathfrak{M}_{\circ}). 
	\end{align}
	It is enough to show that the above 
	functor is essentially surjective. 
	Note that cohomology sheaves of an object in the RHS of (\ref{quotient:surj})
	is a coherent sheaf on the classical stack $\mM \setminus \wW$. 
Also 
	any coherent sheaf on the classical stack $\mM \setminus \wW$
		extends to a coherent sheaf on $\mM$
	(see~\cite[Corollary~15.5]{MR1771927}). 
	Since the LHS of (\ref{quotient:surj}) is a triangulated subcategory, 
	it follows that 
	the restriction functor in (\ref{quotient:surj}) is essentially surjective. 
\end{proof}

Let 
us take a conical closed substack $\zZ \subset \nN=t_0(\Omega_{\fM}[-1])$, 
and use notation in Subsection~\ref{subsub:replace}. 
We have the following lemma: 
\begin{lem}\label{lem:exact:Z}
	We have the localization sequence
	\begin{align*}
		D^b_{\rm{coh}}(\mathfrak{M})_{\wW}/\cC_{\zZ \cap p_0^{-1}(\wW)} \to 
		\mathcal{DT}^{\mathbb{C}^{\ast}}(\nN^{\rm{ss}}) \to 
		\mathcal{DT}^{\mathbb{C}^{\ast}}(\nN^{\rm{ss}}_{\circ}). 
	\end{align*}
\end{lem}
\begin{proof}
	We have the following commutative diagram
	\begin{align*}
		\xymatrix{
			D_{\rm{coh}}^b(\mathfrak{M})_{\wW} \ar@<-0.3ex>@{^{(}->}[r]  & 
			D_{\rm{coh}}^b(\mathfrak{M}) \ar@{->>}[r]
			& D^b_{\rm{coh}}(\mathfrak{M}_{\circ}) \\
			\cC_{\zZ \cap p_0^{-1}(\wW)} \ar@<-0.3ex>@{^{(}->}[r]
			\ar@<-0.3ex>@{^{(}->}[u] & \cC_{\zZ} \ar@{->>}[r] 
			\ar@<-0.3ex>@{^{(}->}[u] & \cC_{\zZ_{\circ}}.
			\ar@<-0.3ex>@{^{(}->}[u] 
		}
	\end{align*}
	Here each horizontal sequence is a localization sequence by Lemma~\ref{lem:localization}. 
	By taking quotients, we obtain the desired localization sequence. 
\end{proof}

When $\zZ=p_0^{-1}(\wW)$
for the projection $p_0 \colon \nN \to \mM$, 
we have a simpler description of 
the quotient category (\ref{quot:D}) by the following lemma: 

\begin{lem}\label{lem:dstack:support}
	Suppose that $\zZ=p_0^{-1}(\wW)$. 
	Then the 
	restriction functors give equivalences
	\begin{align}\notag
		\mathcal{DT}^{\mathbb{C}^{\ast}}(\nN^{\rm{ss}})
		\stackrel{\sim}{\to} \wdDT(\nN^{\rm{ss}}) \stackrel{\sim}{\to}
		D^b_{\rm{coh}}(\mathfrak{M}_{\circ}). 
	\end{align}
\end{lem}
\begin{proof}
The equivalence $\dDT^{\C}(\nN^{\rm{ss}}) \stackrel{\sim}{\to} \Dbc(\fM_{\circ})$ follows 
	from the identity $\cC_{\zZ}=\Dbc(\fM)_{\wW}$ in (\ref{eq:IndCW})  
	together with 
	Lemma~\ref{lem:localization}. 
	In particular $\dDT^{\C}(\nN^{\rm{ss}})$ is idempotent complete. 
	Then the equivalence $\dDT^{\C}(\nN^{\rm{ss}}) \stackrel{\sim}{\to}
	\wdDT^{\C}(\nN^{\rm{ss}})$
	follows from Proposition~\ref{lem:indDT} and Lemma~\ref{lem:compactW}. 
\end{proof}

We next consider the case
that $p_0^{-1}(\wW) \subset \zZ$. 
In this case, the open immersion (\ref{open:N}) is an isomorphism. 
\begin{lem}\label{lem:replace}
	When $p_0^{-1}(\wW) \subset \zZ$,  
	the restriction functors
	give equivalences
	\begin{align*}
		\mathcal{DT}^{\mathbb{C}^{\ast}}
		(\nN^{\rm{ss}}) \stackrel{\sim}{\to}
		\mathcal{DT}^{\mathbb{C}^{\ast}}
		(\nN_{\circ}^{\rm{ss}}), \ 
		\wdDT^{\C}
		(\nN^{\rm{ss}}) \stackrel{\sim}{\to}
		\wdDT^{\C}
		(\nN_{\circ}^{\rm{ss}})
	\end{align*}
\end{lem}
\begin{proof}
	The assumption together with  
	Lemma~\ref{lem:dstack:support}
	imply that 
	\begin{align*}
		\cC_{\zZ \cap p_0^{-1}(\wW)}=\cC_{p_0^{-1}(\wW)}=D^b_{\rm{coh}}(\mathfrak{M})_{\wW}. 
	\end{align*}
	Therefore the first equivalence holds from Lemma~\ref{lem:exact:Z}. 
	
	For a smooth morphism $\alpha \colon \fU \to \fM$,
	by setting $\fU_{\circ}\cneq \fU \setminus \alpha^{-1}(\wW)$, the first equivalence 
	gives an equivalence 
	\begin{align*}
		\Dbc(\fU)/\cC_{\alpha^{\ast}\zZ} \stackrel{\sim}{\to}
		\Dbc(\fU_{\circ})/\cC_{\alpha^{\ast}\zZ_{\circ}}.
	\end{align*}
	Therefore the second equivalence holds by taking the limit. 
\end{proof}
\subsection{Compact generation of $\Ind \cC_{\zZ}$}
Let $\fM$ be a quasi-smooth QCA derived stack over $\mathbb{C}$ 
whose classical truncation $\mM=t_0(\fM)$ admits a good moduli space
$\pi_{\mM} \colon \mM \to M$. 
Here we prove the compact
generation of $\Ind \cC_{\zZ}$ for a conical 
closed substack $\zZ \subset t_0(\Omega_{\fM}[-1])$. 
We note that when $\zZ=\mM \subset t_0(\Omega_{\fM}[-1])$ is the zero section, 
then 
\begin{align*}
	\Ind \cC_{\mM}=D_{\qcoh}(\fM)
\end{align*}
and its compact generation is proved in~\cite[Theorem~A]{Hperf}, 
as $\fM$ is s-global type by Theorem~\ref{thm:AHR}. 
Indeed we will see that the same argument applies almost 
verbatim to show the compact 
generating of $\Ind \cC_{\zZ}$.  
\subsubsection{Presheaves of triangulated categories}
Let $\mM=t_0(\fM)$ be the classical truncation of $\fM$, and 
assume that it admits a good moduli space
\begin{align*}
	\pi_{\mM} \colon \mM \to M. 
\end{align*}
Let $\dD_{{\rm{\acute{e}t}}/M}$ be the category of 
\'{e}tale morphisms $\iota \colon U \to M$
as in Subsection~\ref{subsubsec:gmoduli}. 
Then we have the associated diagram (\ref{Cartesian:U}), so 
in particular a derived stack $\fM_U$ together with 
a morphism $\iota_{\fM} \colon \fM_U \to \fM$. 
For a conical closed substack $\zZ \subset t_0(\Omega_{\fM}[-1])$
and $(\iota \colon U \to M) \in \dD_{{\rm{\acute{e}t}}/M}$, 
we set 
\begin{align*}
	\tT_{\zZ}(U) \cneq \Ind \cC_{\iota_{\fM}^{\ast}\zZ} \subset \Ind \Dbc(\fM_U). 
\end{align*}
For a morphism $\rho \colon U' \to U$ in $\dD_{{\rm{\acute{e}t}}/M}$, 
from the diagram (\ref{dia:rho})
we have adjoint pairs
\begin{align*}
	\xymatrix{
		\tT_{\zZ}(U) \ar@<0.5ex>[r]^{\rho^{\ast}} & \tT_{\zZ}(U') \ar@<0.5ex>[l]^{\rho_{\ast}}
	}, \quad \rho^{\ast} \cneq \rho_{\fM}^{\ind\ast}, 
\rho_{\ast} \cneq \rho_{\fM\ast}^{\ind}, \quad \rho^{\ast} \dashv \rho_{\ast}. 
\end{align*}
Therefore $U \mapsto \tT_{\zZ}(U)$ is a $\dD_{{\rm{\acute{e}t}}/M}$-presheaf of 
triangulated categories with adjoints
in the sense of~\cite[Section~5]{Hperf}. 
For an open immersion $U_{\circ} \subset U$, 
we set
\begin{align*}
	\tT_{\zZ}(U)_{U\setminus U^{\circ}} &\cneq 
	\Ker(\tT_{\zZ}(U) \to \tT_{\zZ}(U_{\circ})) \\
	&=\Ind \cC_{\iota_{\fM}^{\ast}\zZ \cup p_U^{\ast}(\mM_U \setminus \mM_{U_{\circ}})}
\end{align*}
where $p_U \colon t_0(\Omega_{\fM_U}[-1]) \to \mM_U$ is the projection. 

\begin{prop}\label{prop:sheaf}
	The presheaf of triangulated categories $\tT_{\zZ}$ satisfies the 
	following properties. 
	
	(i) For $(\iota \colon U \to M) \in \dD_{{\rm{\acute{e}t}}/M}$, 
	the category $\tT_{\zZ}(U)$ is closed under small 
	coproduct. 
	
	(ii) For any morphism $\rho \colon U' \to U$ in $\dD_{{\rm{\acute{e}t}}/M}$, 
	the right adjoint $\rho_{\ast}$ preserves small coproducts. 
	
	(iii) For every Cartesian square in $\dD_{{\rm{\acute{e}t}}/M}$
	\begin{align*}
		\xymatrix{
			V' \ar[r]^-{g'} \ar[d]_-{\rho_V} \diasquare & U' \ar[d]^-{\rho} \\
			V \ar[r]^-{g} & U
		}
	\end{align*}
	the natural transformation 
	$g^{\ast}\rho_{\ast} \to \rho_{V\ast}(g')^{\ast}$ is an isomorphism. 
	
	(iv) For any open immersion $U_{\circ} \subset U$
	and an \'{e}tale neighborhood $\rho \colon U' \to U$ of $U \setminus U_{\circ}$, 
	the pull-back $\rho^{\ast}$ induces an equivalence 
	$\rho^{\ast} \colon \tT_{\zZ}(U)_{U \setminus U_{\circ}} \stackrel{\sim}{\to}
	\tT_{\zZ}(U')_{U' \setminus U_{\circ}'}$. 
	
	(v) For every finite faithfully flat morphism 
	$\rho \colon U' \to U$, the functor 
	$\rho_{\ast} \colon \tT_{\zZ}(U') \to \tT_{\zZ}(U)$
	admits a right adjoint $\rho^{!}$ that preserves
	small coproducts, conservative, and commutes with pull-back along 
	open immersions. 
\end{prop}
\begin{proof}
	(i) For any quasi-smooth affine derived scheme $\fU$, 
	the category $\Ind \Dbc(\fU)$ is co-complete, and the 
	subcategory of fixed singular supports are closed under direct 
	sums as it is a left orthogonal of some subcategory in $\Ind \Dbc(\fU)$
	(see~\cite[Section~3.1.4]{MR3300415}). 
	Then (i) follows from the definition of $\Ind \cC_{\zZ}$ as a limit (\ref{lim:ind}). 
	
	(ii) Since $\rho_{\fM}$ is representable, the functor $\rho_{\ast}$ is continuous 
	by~\cite[Proposition~3.1.1]{MR3136100}.
	
	(iii) The desired isomorphism is proved in~\cite[Proposition~3.2.2]{MR3701352}. 
	
	(iv)  Let $\wW=\mM_U \setminus \mM_{U_{\circ}}$
	and $\wW'=\mM_{U'} \setminus \mM_{U'_{\circ}}$. 
	It is enough to show that the adjoint pairs 
	\begin{align}\label{adjoint:U}
		\xymatrix{
			\Ind \Dbc(\fM_U)_{\wW}
			\ar@<0.5ex>[r]^{\rho_{\fM}^{\ind\ast}} & \Ind \Dbc(\fM_{U'})_{\wW'} 
			\ar@<0.5ex>[l]^{\rho_{\fM \ast}^{\ind}}
		}
	\end{align}
	are equivalences, since both 
	functors preserve subcategories 
	$\tT_{\zZ}(U)_{U \setminus U_{\circ}}$ and $\tT_{\zZ}(U')_{U' \setminus U_{\circ}'}$. 
	It is proved in~\cite[Proposition~4.2, (1) $\Rightarrow$ (4)]{HaMay} that 
	the adjoint pairs 
	\begin{align}\notag
		\xymatrix{
			D_{\qcoh}(\fM_U)_{\wW}
			\ar@<0.5ex>[r]^{\rho_{\fM}^{\ast}} & D_{\qcoh}(\fM_{U'})_{\wW'} 
			\ar@<0.5ex>[l]^{\rho_{\fM \ast}}
		}
	\end{align}
	are equivalences. 
	As $\rho$ is an \'{e}tale neighborhood of $U \setminus U_{\circ}$, the 
	above functors restrict to 
	equivalences between 
	$\Dbc(\fM_U)_{\wW}$ and 
	$\Dbc(\fM_{U'})_{\wW'}$. 
	As both sides in (\ref{adjoint:U})
	are obtained as ind-completions of 
	$\Dbc(\fM_U)_{\wW}$ and 
	$\Dbc(\fM_{U'})_{\wW'}$ by Lemma~\ref{lem:compactW},  
	the functors in (\ref{adjoint:U}) are equivalences. 
	
	(v) As $\rho$ is proper, the right adjoint $\rho_{\fM}^!$ is continuous (see~\cite[Section~3.3.7]{MR3136100}). 
	Since $\rho$ is \'{e}tale, 
	we have $\rho_{\fM}^{!}(-)=\rho_{\fM}^{\ast}(-) \otimes \omega_{\rho}$, which is 
	obviously 
	conservative and commutes with pull-back along open immersions. 	
\end{proof}

\subsubsection{Compact generation of $\Ind \cC_{\zZ}$}\label{subsub:cpt}
In the setting of the previous subsection, we have the following: 
\begin{thm}\label{thm:compact}
	Let $\fM$ be a quasi-smooth and QCA derived stack
	such that $\mM=t_0(\fM)$ admits a good moduli space 
	$\pi_{\mM} \colon \mM \to M$. Then 
	the category $\Ind \cC_{\zZ}$ is compactly generated with 
	compact objects $\cC_{\zZ}$. 
\end{thm}
\begin{proof}
	We note that the conditions in Proposition~\ref{prop:sheaf}
	are exactly the assumptions imposed in~\cite[Theorem~6.9]{Hperf}, 
	so the compact generation of $\Ind \cC_{\zZ}$ follows 
	from the argument of \textit{loc.~cit.~} almost 
	verbatim. 
	More precisely, let $\dD_{{\rm{\acute{e}t}}/M}^{\rm{cp}} \subset \dD_{{\rm{\acute{e}t}}/M}$
	be the subcategory of $\iota \colon U \to M$ such that 
	$\tT_{\zZ}(U)$ is compactly generated. 
	If $\iota \colon U \to M$ satisfies the condition in Theorem~\ref{thm:AHR}, 
	then $\fM_U$ is of the form $[\fU/G]$ as in Proposition~\ref{prop:extend}, 
	so $\tT_{\zZ}(U)$ is compactly generated by Lemma~\ref{lem:gcomplete}. 
	Therefore there is a separated, quasi-finite and faithfully-flat 
	morphism $\iota \colon U \to M$
	such that $(\iota \colon U \to M) \in \dD_{{\rm{\acute{e}t}}/M}^{\rm{cp}}$. 
	By applying~\cite[Theorem~6.1]{HaMay},
	we have $(\id \colon M \to M) \in \dD_{{\rm{\acute{e}t}}/M}^{\rm{cp}}$
	if we have the followings: 
	
	(D1) For any morphism $\rho \colon U' \to U$ in $\dD_{{\rm{\acute{e}t}}/M}$, 
	if $(U \to M) \in \dD_{{\rm{\acute{e}t}}/M}^{\rm{cp}}$, 
	then $(U' \to M) \in \dD_{{\rm{\acute{e}t}}/M}^{\rm{cp}}$. 
	
	(D2) For any finite 
	morphism $\rho \colon U' \to U$ in $\dD_{{\rm{\acute{e}t}}/M}$, 
	if $(U' \to M) \in \dD_{{\rm{\acute{e}t}}/M}^{\rm{cp}}$, 
	then $(U \to M) \in \dD_{{\rm{\acute{e}t}}/M}^{\rm{cp}}$. 
	
	(D3) For any open immersion $U_{\circ} \subset U$
	and an \'{e}tale neighborhood $\rho \colon U' \to U$ of $U \setminus U_{\circ}$, if $(U_{\circ} \to M)$, $(U_{\circ}' \to M)$ 
	and $(U' \to M)$ are objects in $\dD_{{\rm{\acute{e}t}}/M}^{\rm{cp}}$, 
	then $(U \to M) \in \dD_{{\rm{\acute{e}t}}/M}^{\rm{cp}}$. 
	
	The claim for (D1) holds
	from~\cite[Example~3.11]{Hperf}, 
	since 
	$\rho_{\ast} \colon \tT_{\zZ}(U') \to \tT_{\zZ}(U)$ is 
	continuous and conservative. 
	Indeed $\rho_{\fM}^{\ast}\tT_{\zZ}(U)^{\rm{cp}}$ are compact 
	objects and generate
	$\tT_{\zZ}(U')$. 
	The claim for (D2) follows from~\cite[Proposition~6.9]{Hperf}, 
	as we have properties (i)-(v) in Proposition~\ref{prop:sheaf}. 
	Indeed $\rho_{\ast}(\tT_{\zZ}(U'))$ are compact objects
	and generate $\tT_{\zZ}(U)$. 
	The claim for (D3) follows from~\cite[Proposition~6.8]{Hperf}
	by setting $V=\emptyset$ in~\textit{loc.~cit.~}. 
	
	The above argument shows that $\Ind \cC_{\zZ}$ is compactly generated. 
	It remains to show that 
	the subcategory of compact objects
	$(\Ind \cC_{\zZ})^{\rm{cp}}$ coincide with $\cC_{\zZ}$. 	
	Since $\Dbc(\fM)=(\Ind \Dbc(\fM))^{\rm{cp}}$
	by Theorem~\ref{thm:QCA}, 
	we have 
	\begin{align*}
		\cC_{\zZ} \subset (\Ind \Dbc(\fM))^{\rm{cp}}
		\cap \Ind \cC_{\zZ} \subset (\Ind \cC_{\zZ})^{\rm{cp}}. 
	\end{align*} 
	As for the converse direction, let us take 
	$\eE \in  (\Ind \cC_{\zZ})^{\rm{cp}}$ 
	and a smooth morphism $\alpha \colon \fU \to \fM$ for 
	a quasi-smooth affine derived scheme $\fU$. 
	It is enough to show that $\alpha^{\ast}\eE \in \cC_{\alpha^{\ast}\zZ}$.
	Since $\alpha^{\ind}_{\ast} \colon \Ind \Dbc(\fU) \to \Ind \Dbc(\fM)$
	is continuous, 
	its restriction 
	$\alpha^{\ind}_{\ast} \colon 
	\Ind \cC_{\alpha^{\ast}\zZ} \to \Ind \cC_{\zZ}$ is 
	also continuous. 
	Therefore $\alpha^{\ast}\eE \in (\Ind \cC_{\alpha^{\ast}\zZ})^{\rm{cp}}
	=\cC_{\alpha^{\ast}\zZ}$, where the latter identity 
	follows from Lemma~\ref{lem:gcomplete}. 
\end{proof}

We keep the setting of Theorem~\ref{thm:compact}.
As in the proof of Theorem~\ref{thm:equivalence2}, 
we denote by $\dD'_{{\rm{\acute{e}t}}/M} \subset \dD_{{\rm{\acute{e}t}}/M}$ the 
subcategory of 
\'{e}tale morphisms $\iota \colon U \to M$ 
satisfying the condition in Theorem~\ref{thm:AHR}. 
For a morphism $\rho \colon U' \to U$ in $\dD'_{{\rm{\acute{e}t}}/M}$, 
we have the Cartesian diagrams (\ref{dia:rho}). 
In particular, we have the pull-back functor 
\begin{align*}
	\rho_{\fM}^{\ast} \colon 
	\Dbc(\fM_U)/\cC_{\iota_{\fM}^{\ast}\zZ} \to \Dbc(\fM_{U'})/\cC_{\iota_{\fM}^{'\ast}\zZ}
\end{align*}
so that we can take their limit. 
\begin{lem}\label{lem:inter}
	We have a fully-faithful functor
	\begin{align}\label{DT:inter}
		\dDT^{\C}(\nN^{\rm{ss}}) \hookrightarrow 
		\lim_{(U \stackrel{\iota}{\to}M)\in \dD'_{{\rm{\acute{e}t}}/M}}
		\left( \Dbc(\fM_U)/\cC_{\iota_{\fM}^{\ast}\zZ} \right). 
	\end{align}
\end{lem}
\begin{proof}
	Using the result of Theorem~\ref{thm:compact}, 
	the proof is pararell to that of Proposition~\ref{lem:indDT}. 
	As in the proof of Proposition~\ref{lem:indDT}, 
	for each $(U \stackrel{\iota}{\to}M) \in \dD'_{{\rm{\acute{e}t}}/M}$
	the sequence 
	\begin{align}\label{localizing:MU}
		\Ind \cC_{\iota_{\fM}^{\ast}\zZ} \to 
		\Ind \Dbc(\fM_U) \to 
		\Ind\left( \Dbc(\fM_U)/\cC_{\iota_{\fM}^{\ast}\zZ}  \right)
		\end{align}
	is a localization sequence 
	such that the right arrow admits a functorial right adjoint. 
As the union of $\iota_{\fM} \colon \fM_U \to \fM$ is an 
\'{e}tale cover of $\fM$, we have the equivalences
\begin{align*}
	\Ind \Dbc(\fM)
	\stackrel{\sim}{\to}	
		\lim_{(U \stackrel{\iota}{\to}M)\in \dD'_{{\rm{\acute{e}t}}/M}}
		\Ind \Dbc(\fM_U), \ 
		\Ind \cC_{\zZ}
		\stackrel{\sim}{\to}	
		\lim_{(U \stackrel{\iota}{\to}M)\in \dD'_{{\rm{\acute{e}t}}/M}}
		\Ind \cC_{\iota_{\fM}^{\ast}\zZ}.
	\end{align*}
Therefore by taking the limit of 
(\ref{localizing:MU}) for all $U \to M$
in $\dD'_{{\rm{\acute{e}t}}/M}$, we obtain the 
localizing sequence
\begin{align*}
	\Ind \cC_{\zZ} \to \Ind \Dbc(\fM) \to 
		\lim_{(U \stackrel{\iota}{\to}M)\in \dD'_{{\rm{\acute{e}t}}/M}}
		\Ind\left( \Dbc(\fM_U)/\cC_{\iota_{\fM}^{\ast}\zZ}  \right). 
	\end{align*}
By Theorem~\ref{thm:quotient} and Theorem~\ref{thm:compact}, 
we have the fully-faithful functor 
\begin{align*}
	\Dbc(\fM)/\cC_{\zZ}\hookrightarrow 
	\lim_{(U \stackrel{\iota}{\to}M)\in \dD'_{{\rm{\acute{e}t}}/M}}
\Ind\left( \Dbc(\fM_U)/\cC_{\iota_{\fM}^{\ast}\zZ}  \right). 
	\end{align*}
We have the natural functors
\begin{align*}
	\Dbc(\fM)/\cC_{\zZ}
	\to \lim_{(U \stackrel{\iota}{\to}M)\in \dD'_{{\rm{\acute{e}t}}/M}}
	(\Dbc(\fM_U)/\cC_{\iota_{\fM}^{\ast}\zZ})
	\to \lim_{(U \stackrel{\iota}{\to}M)\in \dD'_{{\rm{\acute{e}t}}/M}}
	\Ind\left( \Dbc(\fM_U)/\cC_{\iota_{\fM}^{\ast}\zZ}  \right).
	\end{align*}
Since the above composition and the second arrow are fully-faithful, 
the first arrow is also fully-faithful. 
\end{proof}

\subsection{Some technical lemmas in derived algebraic geometry}
\subsubsection{Equivariant affine derived schemes}
Let $G$ be a reductive algebraic group. 
We denote by 
$cdga^G$ 
the $\infty$-category of $G$-equivariant 
cdga's
consisting of non-positive degrees. 
Here a $G$-equivariant cdga is a cdga 
which admits a $G$-action and satisfies obvious 
compatibility with differentials and products. 
The $\infty$-category of 
$G$-equivariant affine derived schemes
is defined by
\begin{align*}
	dAff^G \cneq W^{-1}(cdga^G)^{\rm{op}}.
\end{align*}
Here $W^{-1}(-)$ means $\infty$-categorical localization by 
weak homotopy equivalences (cf.~\cite[Section~2.2]{MR3285853}). 
Let $dSt/BG$ be the $\infty$-category of derived stacks over $BG$. 
We have the natural $\infty$-functor
\begin{align*}
	\Pi \colon dAff^G \to dSt/BG, \ 
	\fU \mapsto [\fU/G]. 
\end{align*}
We use the following two propositions:
\begin{prop}\label{prop:cdga1}
	Let $\fM$ be a quasi-smooth derived stack over $BG$
	such that $t_0(\fM)=[\uU/G]$ for an affine scheme $\uU$ 
	with $G$-action. 
	Then there exists a 
	$G$-equivariant 
	tuple $(Y, V, s)$
	as in Definition~\ref{def:tuple}
	and an equivalence 
	$\fM \sim \Pi(\fU)$ 
	where $\fU=\Spec \rR(V \to Y, s)$ is an object in 
	$dAff^G$. 
\end{prop}
\begin{proof}
	Note that we have a sequence of square-zero extensions
	\begin{align*}
		\mM=t_0(\mathfrak{M}) \hookrightarrow
		t_{\le1}(\mathfrak{M}) \hookrightarrow t_{\le 2}(\mathfrak{M})
		\hookrightarrow
		\cdots
	\end{align*}
	such that $t_{\le n}(\fM) \hookrightarrow \fM$ is an equivalence 
	for $n\gg 0$. 
	Let us take distinguished triangles
	\begin{align*}
		\iI_n \to \oO_{t_{\le n+1}(\fM)} \to \oO_{t_{\le n}(\fM)}. 
	\end{align*}
	Then the square zero extension 
	$t_{\le n}(\fM) \hookrightarrow t_{\le n+1}(\fM)$
	in $dSt/BG$ corresponds 
	to an element in 
	(see~\cite[Section~1]{MR3701353})
	\begin{align}\label{AC:hom}
		\Hom_{t_{\le n}(\fM)}(\mathbb{L}_{t_{\le n}(\fM)/BG}, \iI_n[1]). 
	\end{align}
	Suppose that $t_{\le n}(\fM)$ is equivalent to 
	$\Pi(\fU_{\le n})$ for some $\fU_{\le n} \in dAff^G$. 
	Then as $G$ is reductive, 
	the Hom space in (\ref{AC:hom}) is isomorphic to 
	$\Hom_{\fU_{\le n}}(\mathbb{L}_{\fU_{\le n}}, \iI_n[1])^G$. 
	An element of the above 
	corresponds to a square zero extension 
	$\fU_{\le n} \hookrightarrow \fU_{\le n+1}$
	in $dAff^G$. Therefore $t_{\le n+1}(\fM)$ is equivalent to 
	$\Pi(\fU_{\le n+1})$ for some $\fU_{\le n+1} \in dAff^G$. 
	By the induction, we see that 
	$\fM$ is equivalent to $[\fU/G]$ for some $\fU \in dAff^G$. 
	
	We are left to prove that $\fU$ is equivalent in $dAff^G$ 
	to an affine 
	derived scheme 
	$\Spec \rR(V\to Y, s)$ associated with a $G$-equivariant 
	tuple $(Y, V, s)$ as in Definition~\ref{def:tuple}. 
	We prove this by following an argument of~\cite[Theorem~4.1]{BBJ}. 
	Similarly to \textit{loc.~cit.~}, below we will not fix a particular model 
	for the $G$-equivariant cdga $\oO_{\fU}$, and regard $\fU$ as an
	object 
	in the $\infty$-category $dAff^G$.
	Also any 
	map $\fU \to \fU'$ is regarded as a morphism in 
	the $\infty$-category $dAff^G$. 
	
	As $G$ is reductive, we can find a $G$-equivariant closed embedding 
	$\uU \hookrightarrow Y$ for some smooth affine scheme $Y$ with a 
	$G$-action, 
	and let $I \subset \oO_Y$ be the ideal sheaf which defines $\uU$. 
	As $\fU$ is quasi-smooth, the natural morphism 
	of cotangent complexes 
	$\phi \colon 
	\mathbb{L}_{\mathfrak{U}}|_{\uU} \to \tau_{\ge -1} \mathbb{L}_{\uU}$
	is a $G$-equivariant 
	perfect obstruction theory on $\uU$ (see~\cite{BF}),
	i.e. 
	$\phi$ is a morphism in 
	the derived category of $G$-equivariant coherent sheaves on $\uU$, 
	$\hH^0(\phi)$ is an isomorphism and $\hH^{-1}(\phi)$ is 
	a surjection. 
	Since $\uU$ and $Y$ are affine, 
	the morphism $\phi$ is 
	represented by 
	a morphism of complexes of $G$-equivariant sheaves on $\uU$
	\begin{align*}
		\xymatrix{
			V^{\vee}|_{\uU} \ar[r] \ar[d] & \Omega_Y|_{\uU} \ar@{=}[d] \\
			I/I^2 \ar[r] & \Omega_Y|_{\uU}. 
		}
	\end{align*}
	Here $V \to Y$ is a $G$-equivariant vector bundle, 
	and the left arrow $V^{\vee}|_{\uU} \to I/I^2$ is a surjection. 
	One can lift the surjection $V^{\vee}|_{\uU} \to I/I^2$
	to a $G$-equivariant map $s \colon V^{\vee} \to I$, which we can assume to be surjective 
	by shrinking $Y$ if necessary. 
	Since $Y$ is smooth, we can lift the closed immersion 
	$\uU \hookrightarrow Y$ to 
	a map $j \colon \mathfrak{U} \to Y$ in $dAff^G$. Then the diagram
	\begin{align}\notag
		\xymatrix{
			\mathfrak{U} \ar[r]^-{j} \ar[d]_-{j}  & Y \ar[d]^-{0} \\
			Y \ar[r]^-{s} & V 
		}
	\end{align}
	is commutative in $dAff^G$ up to equivalence. 
	Therefore the above diagram induces a map 
	$\fU \to \Spec \rR(V \to Y, s)$ in $dAff^G$. 
	The above map induces the isomorphism on the classical 
	truncation, and the induced map on cotangent complexes 
	is a quasi-isomorphism by the construction. 
	Therefore $\fU$ is equivalent to
	$\Spec \rR(V \to Y, s)$
	in $dAff^G$. 
\end{proof}

\begin{prop}\label{prop:cdga2}
	For $\fU, \fU' \in dAff^G$, let 
	$f \colon \Pi(\fU) \to \Pi(\fU')$ be a morphism in 
	$dSt/BG$. 
	Then there exists a morphism 
	$\widetilde{f} \colon \fU \to \fU'$ in $dAff^G$
	such that $f \sim \Pi(\widetilde{f})$. 
\end{prop}
\begin{proof}
	The proof is similar to Proposition~\ref{prop:cdga1} using deformation 
	argument. 
	We set
	\begin{align*}
		f_{\le n} \cneq f|_{t_{\le n}(\Pi(\fU))} \colon 
		t_{\le n}(\Pi(\fU)) \to \Pi(\fU'). 
	\end{align*}
	For $n=0$, 
	$f_{\le 0}$ factors through 
	a morphism of classical Artin stacks 
	$[t_0(\fU)/G] \to [t_0(\fU')/G]$ over $BG$. 
	By pulling back via $\Spec \mathbb{C} \to BG$, 
	it lifts to a $G$-equivariant morphisms of 
	affine schemes $t_0(\fU) \to t_0(\fU')$. 
	By setting $\widetilde{f}_{\le 0}$ to be the composition of 
	the above morphism with $t_0(\fU') \hookrightarrow \fU'$, 
	we have $f_{\le 0} \sim \Pi(\widetilde{f}_{\le 0})$. 
	
	Suppose that 
	there exists $\widetilde{f}_{\le n} \colon t_{\le n}(\fU) \to \fU'$ in 
	$dAff^G$ such that 
	$f_{\le n} \sim \Pi(\widetilde{f}_{\le n})$. 
	We take the distinguished triangle 
	\begin{align*}
		\jJ_n \to \oO_{t_{\le n+1}(\Pi(\fU))} \to \oO_{t_{\le n}(\Pi(\fU))}. 
	\end{align*}
	Note that $\jJ_n$ is an object in $\Dbc(\Pi(\uU))$ for $\uU=t_0(\fU)$
	push-forward along with the closed immersion 
	$\Pi(\uU) \hookrightarrow \Pi(\fU_{\le n+1})$. 
	We set $S_{n}^i$, $T_{n}^i$ to be 
	\begin{align*}
		S_n^i \cneq \Hom_{\uU}(\widetilde{f}_{\le 0}^{\ast}(\mathbb{L}_{\fU'}|_{\uU'}),  
		p^{\ast}\jJ_n[i])^G, \ 
		T_n^i \cneq \Hom_{[\uU/G]}(f_{\le 0}^{\ast}(\mathbb{L}_{\Pi(\fU')/BG}
		|_{\Pi(\uU')}), \jJ_n[i]). 
	\end{align*}
	Here $\uU'=t_0(\fU')$ and $p \colon \uU \to [\uU/G]$.
	Since $G$ is reductive, we have 
	natural isomorphisms 
	\begin{align}\label{isom:ST}
		S_{n}^i \stackrel{\cong}{\to} T_n^i.
	\end{align}
	The obstruction of extending 
	$\widetilde{f}_{\le n}$ to $\widetilde{f}_{\le n+1}$ in $dAff^G$
	lies in 
	$S_{n}^1$, which corresponds to 
	the obstruction in $T_n^1$ 
	of extending $f_{\le n}$ to $f_{\le n+1}$
	in $dSt/BG$
	under the isomorphism (\ref{isom:ST}). 
	As $f_{\le n}$ is extended to $f_{\le n+1}$ by the assumption, 
	the above obstruction vanishes so that 
	there exists an extension $\widetilde{f}_{\le n+1}$ of $\widetilde{f}_{\le n}$ in $dAff^G$. 
	An extension of $\widetilde{f}_{\le n}$ to $\widetilde{f}_{\le n+1}$
	in $dAff^G$
	is classified by $S_n^0$, 
	while an extension of $f_{\le n}$ to $f_{\le n+1}$
	in $dSt/BG$
	is classified by $T_n^0$. 
	Therefore by the isomorphism (\ref{isom:ST}) there exists
	an extension $\widetilde{f}_{\le n+1}$ such that $\Pi(\widetilde{f}_{\le n+1}) \sim f_{\le n+1}$. 
	By the induction, we obtain 
	$\widetilde{f} \colon \fU \to \fU'$ in $dAff^G$ such that $\Pi(\widetilde{f}) \sim f$. 
\end{proof}

\subsubsection{Proof of Proposition~\ref{prop:extend}}\label{subsub:extend}
\begin{proof}
	Since the category of \'{e}tale morphisms to 
	$\mM$ is equivalent to that of 
	\'{e}tale morphisms to 
	$\fM$, there is a unique derived stack $\fM_U$ which fits into 
	the Cartesian diagram (\ref{Cartesian:U}). 
	Then by Proposition~\ref{prop:cdga1}, 
	the derived stack $\fM_U$ is equivalent to 
	$[\fU/G]$ 
	associated with a $G$-equivariant 
	tuple $(Y, V, s)$. 
	
	The final statement follows from Lemma~\ref{lem:Gchar}. 
	Indeed by shrinking $U$ if necessary, the $\mathbb{R}$-line bundle 
	$\iota_{\fM}^{\ast}(l)$ on $[\uU/G]$ is 
	pulled back  
	via $[\uU/G] \to BG$, 
	so can be 
	extended to $[Y/G]$ by taking the pull-back via 
	$[Y/G] \to BG$. 
\end{proof}

\subsubsection{Proof of Lemma~\ref{lem:closed}}\label{subsec:lem4.9}
\begin{proof}
	We take a $G$-equivariant
	closed immersion $\iota \colon Y \hookrightarrow W$
	for a smooth affine $\mathbb{C}$-scheme
	with a $G$-action. 
	We choose such $W$ which admits 
	a $G$-equivariant vector bundle $W' \to W$ with a $G$-invariant 
	regular 
	section $t \colon W \to W'$ such that its zero locus is a
	single reduced point $\{w_0\}$. 
	For example as 
	$Y$ is of finite presentation, we 
	can take $W$ to be a finite dimensional $G$-representation 
	and $W'=W \times W \to W$ with section given by the diagonal. 
	We set 
	\begin{align*}
		Y''=W \times Y', \ 
		V''=W' \times V'
	\end{align*} 
	and regard $V''$ as a $G$-equivariant vector bundle 
	on $Y''$ by the projection
	$V'' \to Y''$. 
	We have the following $G$-equivariant 
	commutative diagram
	\begin{align}\label{diagram:VYY2}
		\xymatrix{
			V \ar@{.>}[r]_-{g''} \ar[d]  \ar@/^10pt/[rr]^-{g}
			&V'' \ar[r] \ar[d] & V' \ar[d]  \inclusion^-{g'} & V'' \ar[d]  \\
			Y \ar[r]^-{f''} \ar@/^10pt/[u]^-{s} \ar@/_10pt/[rr]_-{f}
			& Y'' \ar[r]^-{p} \ar@/_10pt/[u]_-{s''}& Y' \ar@/_10pt/[u]_-{s'} 
			\inclusion_-{f'} & Y''\ar@/_10pt/[u]_-{s''}. 
		}
	\end{align}
	Here the middle horizontal arrows are projections, 
	and 
	\begin{align*}
		f''(y)=(\iota(y), f(y)), \ 
		s''(w, y')=(t(w), s'(y')), \ f'(y')=(w_0, y'), \ 
		g'(v')=(t(w_0), v').
	\end{align*}
	Note that $f'$, $f''$ are closed immersions. 
	
	Let $\fU''$ be the derived zero locus of $s''$. 
	By the constructions of $(Y'', V'', s'')$, 
	the projection $p$ 
	and the closed immersion $f'$ 
	in the diagram (\ref{diagram:VYY2})
	induce
	equivalences 
	\begin{align}\label{equiv:p}
		\bff''' \colon [\fU''/G]
		\stackrel{\sim}{\to} [\fU'/G], \ 
		\bff' \colon [\fU'/G] \stackrel{\sim}{\to} [\fU''/G]
	\end{align}
	which are 
	inverse each other in the $\infty$-category $dSt/BG$. 
	Since we assumed that the diagram (\ref{dia:VY0}) induces an equivalence 
	$[\fU/G] \stackrel{\sim}{\to} [\fU'/G]$, 
	by composing with (\ref{equiv:p}) we obtain 
	an equivalence 
	$[\fU/G] \sim [\fU''/G]$
	in $dSt/BG$. 
	By Proposition~\ref{prop:cdga2}, 
	the equivalence $[\fU/G] \sim [\fU''/G]$
	lifts to an equivalence 
	$\fU \sim \fU''$ in $dAff^G$. 
	Now we have the morphisms 
	$\fU \hookrightarrow Y \stackrel{f''}{\to} Y''$
	in $dAff^G$, 
	and as $\oO_{\fU''}$ is cofibrant over $\oO_{Y''}$ 
	in $cdga^G$, 
	the equivalence $\fU \sim \fU''$ is 
	given by an actual morphism of $G$-equivariant cdga's, i.e. 
	we have a $G$-equivariant 
	vector bundle morphism $g''$ in the dotted arrow in 
	(\ref{diagram:VYY2}) 
	which induces an equivalence 
	$\fU \stackrel{\sim}{\to} \fU''$. 
	Then it induces an equivalence 
	$\bff'' \colon [\fU/G] \stackrel{\sim}{\to}[\fU''/G]$
	as desired. 
\end{proof}

\subsubsection{Proof of Lemma~\ref{lem:4.11}}\label{subsec:proof:4.11}
\begin{proof}
	By Proposition~\ref{prop:cdga2}, 
	the equivalence $\bff$ lifts to an equivalence 
	$\bff^{\sharp} \colon \fU \stackrel{\sim}{\to} \fU'$
	in $dAff^G$. We then construct the diagram (\ref{dia:equiv}) 
	following the argument of~\cite[Theorem~4.2]{BBJ}. 
	By the equivalence $\bff^{\sharp} \colon \fU \stackrel{\sim}{\to}\fU'$
	together with 
	closed immersions 
	$\fU \hookrightarrow Y$, $\fU' \hookrightarrow Y'$,
	we have the $G$-equivariant 
	closed immersion 
	$\fU \hookrightarrow Y \times Y'$. 
	Then similarly to the proof Proposition~\ref{prop:cdga1}, 
	there exist a $G$-invariant affine open subset 
	$\widetilde{Y} \subset Y \times Y'$ which contains 
	$\uU$, a $G$-equivariant 
	vector bundle $\widetilde{V} \to \widetilde{Y}$
	with a $G$-invariant section $\widetilde{s}$,
	such that $\fU \hookrightarrow \widetilde{Y}$ factors
	through an equivalence 
	$\fU \stackrel{\sim}{\to} \widetilde{\fU}$
	in $dAff^G$, where $\widetilde{\fU}$ is the derived zero 
	locus of $\widetilde{s}$. 
	Let us consider the composition 
	$\widetilde{\fU} \hookrightarrow \widetilde{Y} \to Y$, 
	where the latter map is the projection. 
	Since $\oO_{\fU}$ is cofibrant over $\oO_{Y}$
	in $cdga^G$,  
	the above map factors through 
	an equivalence 
	$\widetilde{\fU} \stackrel{\sim}{\to} \fU$ in $dAff^G$, which is 
	induced by an actual morphism of $G$-equivariant cdga's. 
	So we obtain a left diagram in (\ref{dia:equiv})
	and an equivalence $\widetilde{\bff}$. 
	The existence of a right diagram in (\ref{dia:equiv}) and $\widetilde{\bff}'$ 
	also follow similarly. 
\end{proof}

\subsubsection{The case of formal fibers}\label{subsub:casefib}
Here we discuss 
the formal fiber version of Lemma~\ref{lem:closed}
and Lemma~\ref{lem:4.11}, which is used in 
Lemma~\ref{lem:window1.5}, Lemma~\ref{lem:window}. 
We consider the setting of Lemma~\ref{lem:window1.5}, so that 
we have the equivalence of derived 
stacks $\widehat{\bff}$ in the diagram (\ref{dia:window1.5}). 
We assume that $\widehat{\bff}$ commutes with 
maps to $BG$ given by canonical $G$-torsors. 
We have the formal fiber version of Lemma~\ref{lem:4.11}: 
\begin{lem}\label{lem:tuple:formal}
There exist $G$-equivariant commutative 
	diagram
	\begin{align}\label{dia:ffcase}
	\xymatrix{
		\widehat{V}_y \ar[d] & \widetilde{V} \ar[r] \ar[l] \ar[d] &
		\widehat{V}'_{y'} \ar[d] \\
		\widehat{Y}_y \ar@/^10pt/[u]^-{\widehat{s}_y} & \widetilde{Y} \ar[r]^-{\widetilde{f}'} 
		\ar[l]_-{\widetilde{f}} \ar@/^10pt/[u]^-{\widetilde{s}} 
		& \widehat{Y}'_{y'}\ar@/_10pt/[u]_-{\widehat{s}'_{y'}}
	}
\end{align}
	such that, by setting $\widetilde{\fU}$ to be 
the derived zero locus of $\widetilde{s}$, 
the above diagram induces equivalences
$\widetilde{\bff} \colon [\widetilde{\fU}/G] \stackrel{\sim}{\to} [\widehat{\fU}_y/G]$, 
$\widetilde{\bff}' \colon [\widetilde{\fU}/G] \stackrel{\sim}{\to}
 [\widehat{\fU}_{y'}'/G]$
which commute with $\widehat{\bff}$, i.e. 
$\bff \circ \widetilde{\bff} \sim \widetilde{\bff}'$. 
Here $\widetilde{Y}$ is a $G$-invariant open subset of $\widehat{Y}_y 
\times \widehat{Y}'_{y'}$
which contains $(x, x')$, $\widetilde{V} \to \widetilde{Y}$ is a 
$G$-equivariant vector bundle and $\widetilde{s}$ is a $G$-invariant section. 
	\end{lem}
\begin{proof}
	The proof is obvious from the proof of Lemma~\ref{lem:4.11}. 
	\end{proof}
We also have the formal fiber version Lemma~\ref{lem:closed}: 
\begin{lem}\label{lem:tuple:formal2}
	In the setting of Lemma~\ref{lem:tuple:formal}, 
we have 
$G$-equivariant commutative diagrams
\begin{align}\label{dia:ffcase2}
		\xymatrix{
		\widetilde{V} \ar[r]^-{g''} \ar[d] & V'' \ar[d] \\
		\widetilde{Y} \inclusion^-{f''} \ar@/^10pt/[u]^-{s}
		& Y'', \ar@/_10pt/[u]_-{s''}&
	} \ 
	\xymatrix{
		\widehat{V}_y \ar[r]^-{g'} \ar[d] & V'' \ar[d] \\
		\widehat{Y}_y \inclusion^-{f'} \ar@/^10pt/[u]^-{s'}
		& Y'', \ar@/_10pt/[u]_-{s''}&
	}
	\end{align}
where $Y''=\widetilde{Y} \times \widehat{Y}_y$, 
$V'' \to Y''$ is a $G$-equivariant vector bundle, $s''$ is a $G$-invariant 
section, 
satisfying the followings: 
\begin{enumerate}
	\item The diagrams (\ref{dia:ffcase2})
	induce equivalences of derived stacks
	\begin{align}\notag
		\bff'' \colon [\widetilde{\fU}/G] \stackrel{\sim}{\to}
		[\fU''/G], \
		\bff' \colon [\widehat{\fU}_y/G] \stackrel{\sim}{\to}
		[\fU''/G]
	\end{align}
	such that 
	$\bff'\circ \widehat{\bff} \sim \bff''$. 
	Here $[\fU''/G]$ is the derived 
	zero locus of $s'' \colon [Y''/G] \to [V''/G]$. 
	\item 
	The morphisms $f'$, $f''$ are closed immersions. 
\end{enumerate} 
Moreover the same statement holds for the right square of (\ref{dia:ffcase}). 
	\end{lem}
\begin{proof}
	The same argument of the proof of Lemma~\ref{lem:closed} applies. 
	In this case, we can also take 
	$Y''=\widetilde{Y} \times \widehat{Y}_y$ 
	as we show in the following. 
	From the proof of Lemma~\ref{lem:closed}, 
	we first take a $G$-equivariant closed immersion 
	$\widetilde{Y} \hookrightarrow W$
	with a vector bundle $W' \to W$ and a regular section $t$
	whose zero locus is a reduced single point. 
	In the situation of the lemma, 
	we can take $W=\widetilde{Y}$. 
	Indeed the condition $G=\Aut(x)$
	and \'{e}tale slice theorem
	implies that $\widehat{Y}_y$ is isomorphic to the formal 
	fiber of $T_{x}Y \to T_x Y \ssslash G$
	at $0$ (see the proof of Proposition~\ref{prop:rest:an}), 
	so we can assume that $Y$ is a $G$-representation and $x=0$. 
	The same also applies to $Y'$. 
	Then we have the trivial vector bundle $(Y \times Y') \times (Y \times Y')
	\to (Y \times Y')$ by the second projection, with a regular 
	section given 
	by the diagonal whose zero locus is $(0, 0)$. 
	By restricting the above vector bundle to the formal 
	fiber $\widehat{Y}_y \times \widehat{Y}'_{y'} \to Y \times Y'$
	and its $G$-invariant open subset $\widetilde{Y}$, we obtain 
	the $G$-equivariant vector bundle $W' \to \widetilde{Y}$ with a regular section 
	whose zero locus is a reduced point $(0, 0)$. 
	So we can take $W=\widetilde{Y}$. 
	Then the argument of the proof of Lemma~\ref{lem:closed}
	implies that we can take $Y''=\widetilde{Y} \times \widehat{Y}_y$. 
		\end{proof}

\subsubsection{Proof of Lemma~\ref{lem:rest0}}\label{subsec:lemG}
\begin{proof}
	Let
	$\mathfrak{m} \subset \oO_{\widehat{\uU}_y}$ be the maximal ideal 
	of $x$ and $\uU^{[n]} \subset \widehat{\uU}_y$ the 
	closed subscheme defined by 
	$\mathfrak{m}^n$. 
	We claim that the compatible isomorphisms 
	\begin{align}\label{isom:fgn}
		f|_{[\uU^{[n]}/G]} \cong f'|_{[\uU^{[n]}/G]}
	\end{align}
	exist by the induction on $n$. 
	The $n=1$ case follows from $f \circ \mu \cong f'\circ \mu$. 
	Suppose that 
	the claim holds for $n$. 
	By a standard deformation theory of morphisms of stacks, 
	the set of possible extensions of $f|_{[\fU^{[n]}/G]}$ to 
	$[\fU^{[n+1]}/G]$ is a torsor over 
	\begin{align*}
		\Hom_{[\uU^{[n]}/G]}(f|_{[\uU^{[n]}/G]}^{\ast}\mathbb{L}_{BG}, 
		\mathfrak{m}^n/\mathfrak{m}^{n+1})=0
	\end{align*}
	since $\mathbb{L}_{BG}=\mathfrak{g}^{\vee}[-1]$, $\uU^{[n]}$ is affine and $G$ is reductive. Therefore the isomorphism (\ref{isom:fgn}) lifts to an 
	isomorphism for $n+1$, so the claim holds. 
	
	Since $\widehat{\uU}_y\ssslash G$ is complete local, 
	by~\cite[Corollary~3.6]{AHR}
	a morphism $[\widehat{\uU}_y/G] \to BG$ is determined by its 
	restriction to $\varinjlim \left[\uU^{[n]}/G \right]$. 
	Therefore we conclude that $f|_{[\widehat{\uU}_y/G]} \cong f'|_{[\widehat{\uU}_y/G]}$.
	We then apply the same deformation argument above
	for the square zero extensions 
	\begin{align*}
		[\widehat{\uU}_y/G]=t_0([\widehat{\fU}_y/G]) 
		\hookrightarrow t_{\le 1}([\widehat{\fU}_y/G])
		\hookrightarrow t_{\le 2}(]\widehat{\fU}_y/G]) \hookrightarrow \cdots
	\end{align*}
	and conclude that $f \sim f'$. 
\end{proof}

\subsection{Formal neighborhood theorem}\label{subsec:formal}
In this section, we prove the formal neighborhood theorem for 
moduli stacks of semistable sheaves on smooth projective varieties. 
A similar result is proved in~\cite{MR3811778} analytic locally on the 
good moduli spaces, using gauge theory argument. 
Instead of gauge theory, we use formal GAGA theorem~\cite{FGAGA, AHR2}
to give a purely algebraic proof for the formal neighborhood theorem. 
\subsubsection{Formal GAGA for good moduli spaces}\label{subsec:FGAGA}
Let $\xX$ be a classical Artin stack with affine diagonal, 
and 
\begin{align*}
	\pi \colon \xX \to \Spec R
\end{align*}
its good moduli space.
We assume that $R$ is a complete local Noetherian ring
with maximal ideal $\mathfrak{m} \subset R$, and 
$\pi$ is of finite type.  
Let $\iI \subset \oO_{\xX}$ be the ideal sheaf 
generated by the pull-back of $\mathfrak{m}$. 
We define the closed substack 
$\xX_n \subset \xX$ 
to be defined by the ideal $\iI^n$. 
We have the 
formal stack $\varinjlim \xX_n$. 
The formal GAGA theorem for $\xX$ is stated as follows: 

\begin{thm}\label{thm:FGAGA}\emph{(\cite[Theorem~1.1]{FGAGA}, \cite[Corollary~1.7]{AHR2})}
	The restriction functor
	\begin{align*}
		\Coh(\xX) \to \Coh(\varinjlim \xX_n)
	\end{align*}
	is an equivalence of categories. 
\end{thm}

\subsubsection{Ext-quivers associated with simple collections}\label{subsec:Equiver}
Let $Z$ be a smooth projective variety over $\mathbb{C}$. 
A collection of coherent sheaves $(E_1, \ldots, E_m)$ on $Z$ is 
called a \textit{simple collection} 
if $\Hom(E_i, E_j)=\mathbb{C} \cdot \delta_{ij}$. 
The Ext-quiver $Q_{E_{\bullet}}$ associated with 
the collection $(E_1, \ldots, E_m)$
is defined as follows: 
the vertex set $V(Q_{E_{\bullet}})$ and the 
edge set $E(Q_{E_{\bullet}})$ are given by 
\begin{align*}
	V(Q_{E_{\bullet}})=\{1, \ldots, m\}, \ 
	E(Q_{E_{\bullet}})=\bigcup_{1\le i, j \le m}E_{i, j}. 
\end{align*}
Here 
$E_{i, j}$ the set of edges from $i$ to $j$
with $\sharp E_{i, j}=\ext^1(E_i, E_j)$. 
Let $\mathbf{E}_{i, j}$ be the $\mathbb{C}$-vector space 
spanned by $E_{i, j}$. 
We set
\begin{align}\notag
	E_{i, j}^{\vee} \cneq \{ e^{\vee} : 
	e \in E_{i, j}\} \subset \mathbf{E}_{i, j}^{\vee}. 
\end{align}
Here for $e \in E_{i, j}$, 
the element $e^{\vee} \in \mathbf{E}_{i, j}^{\vee}$
is defined by the condition
$e^{\vee}(e)=1$
and $e^{\vee}(e')=0 $
for any $e\neq e' \in E_{i, j}$, i.e. 
$E_{i, j}^{\vee}$ is the dual basis of $E_{i, j}$.

By setting 
$\overline{E}=\oplus_{i=1}^m E_i$, 
we have linear maps 
\begin{align}\label{mn:Eb}
	m_n \colon \Ext^1(\overline{E}, \overline{E})^{\otimes n} \to \Ext^2(\overline{E}, \overline{E})
\end{align}
given by a minimal $A_{\infty}$-structure of the 
dg-algebra $\RHom(\overline{E}, \overline{E})$. 
The 
map (\ref{mn:Eb})
only consists of the direct sum factors of the form
(see~\cite[Section~5.1]{MR3811778})
\begin{align}\notag
	m_n \colon 
	\Ext^1(E_{\psi(1)}, E_{\psi(2)}) \otimes 
	&\Ext^1(E_{\psi(2)}, E_{\psi(3)}) \otimes
	\cdots \\
	\cdots \otimes 
	&\Ext^1(E_{\psi(n)}, E_{\psi(n+1)})  
	\label{factor}
	\to \Ext^2(E_{\psi(1)}, E_{\psi(n+1)}).
\end{align}
Here $\psi$ is a map 
$\psi \colon \{1, \ldots, n+1\} \to \{1, \ldots, m\}$, 
and the above $\{m_n\}_{n\ge 2}$ 
give a minimal $A_{\infty}$-category structure 
on the dg-category generated by $(E_1, \ldots, E_m)$. 
By taking the dual and the products of (\ref{factor})
for all $n\ge 2$, we obtain 
the linear map
\begin{align*}
	\mathbf{m}^{\vee} \cneq \prod_{n\ge 2} m_n^{\vee} \colon 
	\Ext^2(\overline{E}, \overline{E})^{\vee} \to 
	\prod_{n\ge 2} \bigoplus_{\begin{subarray}{c}
			\{1, \ldots, n+1 \} \\ \stackrel{\psi}{\to}  
			\{1, \ldots, m\}
	\end{subarray}}
	&\Ext^1(E_{\psi(1)}, E_{\psi(2)})^{\vee} \otimes \cdots \\
	&\cdots \otimes 
	\Ext^1(E_{\psi(n)}, E_{\psi(n+1)})^{\vee}. 
\end{align*}
Note that an element of the RHS is 
an element of the completed 
path algebra $\mathbb{C}\lkakko Q_{E_{\bullet}} \rkakko$.  
Let $\mathbf{I} \subset \mathbb{C}\lkakko Q_{E_{\bullet}} \rkakko$
be the topological closure of the ideal generated by the 
image of $\mathbf{m}^{\vee}$. 
By~\cite[Section~6.4]{MR3811778}, 
the quotient algebra
\begin{align*}
	A\cneq \mathbb{C}\lkakko Q_{E_{\bullet}} \rkakko/\mathbf{I}
\end{align*}
is a pro-representable hull for the NC
deformation functor associated with the collection $(E_1, \ldots, E_m)$.

Let $E \in \Coh(X)$ be given by 
\begin{align*}
	E=\bigoplus_{i=1}^m V_i \otimes E_i
\end{align*}
for finite dimensional vector spaces $V_i$. 
Then 
\begin{align*}
	\left[
	\Ext^1(E, E)/\Aut(E)
	\right]=
	\left[ \bigoplus_{(i\to j) \in Q_{E_{\bullet}}} 
	\Hom(V_i, V_j)/\prod_{i=1}^m \GL(V_i)
	\right]
\end{align*}
is the moduli stack of $Q_{E_{\bullet}}$-representations with 
dimension vector $\vec{v}=(\dim V_i)_{1\le i\le m}$. 
Here $\Aut(E)$ acts on $\Ext^1(E, E)$ by the conjugation. 
The fiber of the morphism to the good moduli space
\begin{align}\label{ext:good}
	\left[\Ext^1(E, E)/\Aut(E)\right] \to \Ext^1(E, E)\ssslash \Aut(E)
\end{align}
at the origin consists of 
nilpotent $Q_{E_{\bullet}}$-representations. 
More precisely, let $I \subset \oO_{\Ext^1(E, E)}$
be the ideal sheaf which defines the fiber (\ref{ext:good})
at the origin, and 
$T_{n} \hookrightarrow \Ext^1(E, E)$ the closed 
subscheme defined 
by $I^n$. 
Then the formal stack 
$\varinjlim \left[T_n/\Aut(E)\right]$ 
represents the 2-functor
\begin{align}\label{M:quiver}
	\mM_{Q_{E_{\bullet}}}^{\rm{nil}}(\vec{v})\colon 
	Aff^{op} \to Groupoid
\end{align}
which sends an affine $\mathbb{C}$-scheme 
$T$ to the groupoid of data
\begin{align}\label{quiver:data}
	\{(\vV_i, \phi_e)\}_{i\in V(Q_{E_{\bullet}}), e \in E(Q_{E_{\bullet}})}, \ 
	\phi_e \colon \vV_{s(e)} \to \vV_{t(e)}.
\end{align}
Here $\vV_i$ is a vector bundle 
on $T$ with rank $\dim V_i$, 
$(s(e), t(e))=(i, j)$ for $e \in E_{i, j}$, 
and $\phi_e$ are $\oO_T$-module homomorphisms
whose sufficiently large number of compositions 
are zero. 

For an element 
\begin{align*}
	u=(u_e)_{e \in E(Q_{E_{\bullet}})} \in \Ext^1(E, E), \ 
	u_e \colon V_{s(e)} \to V_{t(e)},
\end{align*}
we consider the following $\Ext^2(E, E)$-valued
$\Aut(E)$-equivariant formal 
function
\begin{align}\notag
	&\kappa(u)=\sum_{\begin{subarray}{c}
			n\ge 2, \\
			\{1, \ldots, n+1\} \stackrel{\psi}{\to}
			\{1, \ldots, m\}
	\end{subarray}}
	\sum_{e_i \in E_{\psi(i), \psi(i+1)}} 
	m_n(e_1^{\vee}, \ldots, e_n^{\vee}) \cdot 
	u_{e_n} \circ \cdots \circ u_{e_2} \circ u_{e_1}.
\end{align}
The above sum is a finite sum if $u$ corresponds to a nilpotent 
$Q_{E_{\bullet}}$-representation. 
Therefore 
the restriction $\kappa|_{T_n}$ determines the 
$\Aut(E)$-equivariant algebraic map
\begin{align*}
	\kappa_n \cneq 
	\kappa|_{T_n} \colon T_n \to \Ext^2(E, E). 
\end{align*}
By the above arguments, we obtain sections of 
vector bundles over $T_n$
\begin{align*}
	\xymatrix{
		\left[(\Ext^2(E, E) \times T_n)/\Aut(E)\right] \ar[r]
		& \ar@/_15pt/[l]_{\kappa_n}
		\left[T_n/\Aut(E)\right].
	} 
\end{align*}
Let $\nN_n \hookrightarrow T_n$ be the closed subscheme defined by 
$\kappa_n=0$. 
Then the formal stack $\varinjlim \left[\nN_n/\Aut(E) \right]$
represents the sub 2-functor of 
(\ref{M:quiver}), 
\begin{align}\notag
	\mM^{\rm{nil}}_{(Q_{E_{\bullet}}, \mathbf{I})}(\vec{v}) \subset 
	\mM^{\rm{nil}}_{Q_{E_{\bullet}}}(\vec{v})
\end{align}
consisting of groupoids 
(\ref{quiver:data})
such that $\{\phi_e\}_{e\in E(Q_{E_{\bullet}})}$ satisfy the relation 
$\mathbf{I} \subset \mathbb{C}\lkakko Q_{E_{\bullet}} \rkakko$. 

By Theorem~\ref{thm:FGAGA}, 
the system of 
sections $\{\kappa_n\}_{n\ge 1}$ above 
uniquely lifts to 
a section of a vector bundle on $\left[\widehat{\Ext^1(E, E)}_0/\Aut(E)\right]$
\begin{align}\notag
	\xymatrix{
		\left[(\Ext^2(E, E) \times \widehat{\Ext^1(E, E)}_0)/\Aut(E)\right] \ar[r] & \ar@/_20pt/[l]_{\kappa}
		\left[\widehat{\Ext^1(E, E)}_0/\Aut(E)\right].
	} 
\end{align}
Here as
in Subsection~\ref{subsec:fib}, 
$\widehat{\Ext^1(E, E)}_0$
is the formal fiber at the origin 
$0 \in \Ext^1(E, E)\ssslash \Aut(E)$. 
Let $\widehat{\nN}_0 \subset \widehat{\Ext^1(E, E)}_0$ be the 
closed subscheme defined by $\kappa=0$. 
We have the quotient stack with good moduli space
\begin{align}\label{good:U0}
	[\widehat{\nN}_0/\Aut(E)] \to \widehat{\nN}_0 \ssslash \Aut(E).
\end{align} 
The good moduli space $\widehat{\nN}_0 \ssslash \Aut(E)$ 
is a closed subscheme of $\widehat{\Ext^1(E, E)}_0 \ssslash \Aut(E)$, 
hence written as $\Spec R$ for a complete local Noetherian ring $R$. 
The formal stack $\varinjlim \left[\nN_n/\Aut(E) \right]$ is obtained 
by taking the colimit of 
thickened fibers of the morphism (\ref{good:U0}) at the origin 
as in Subsection~\ref{subsec:FGAGA}.

\subsubsection{Moduli stacks of semistable sheaves}
For an ample divisor $H$ on $Z$
and $v \in H^{2\ast}(Z, \mathbb{Q})$, 
we denote by $\mM_{Z, H}(v)$ the moduli stack of 
Gieseker $H$-semistable sheaves $E$ on $Z$ with 
$\ch(E)=v$. 
By the GIT construction of the moduli stack $\mM_{Z, H}(v)$
(see~\cite{MR1450870}), 
it admits a good moduli space
(see~\cite[Example~8.7]{MR3237451})
\begin{align*}
	\pi \colon 
	\mM_{Z, H}(v) \to M_{Z, H}(v). 
\end{align*}
The good moduli space $M_{Z, H}(v)$
is a projective scheme which parametrizes 
$H$-polystable sheaves, i.e. a
closed point $y \in M_{Z, H}(v)$ corresponds to 
a direct sum
\begin{align}\label{polystable:append}
	E=\bigoplus_{i=1}^m V_i \otimes E_i
\end{align}
where each $E_i$ is a Gieseker $H$-stable sheaf, 
$E_i$ is not isomorphic to $E_j$ for $i\neq j$, 
and $E_i$ has the same reduced Hilbert polynomial with $E_j$. 
Note that $(E_1, \ldots, E_m)$ is a simple collection. 
Let $R=\widehat{\oO}_{M_{Z, H}(v), y}$, 
which is a  
complete local Noetherian ring, and set 
\begin{align}\label{mstack:complete}
	\widehat{\mM}_{Z, H}(v)_y
	\cneq \mM_{Z, H}(v) \times_{M_{Z, H}(v)} \Spec R
	\to \Spec R. 
\end{align}
Note that the above map is a good 
moduli space morphism of $\widehat{\mM}_{Z, H}(v)_y$, 
since the good moduli morphism is preserved by the base change
(see~\cite[Proposition~4.7]{MR3237451}). 

Let $\left[\widehat{\nN}_0/\Aut(E)\right]$ be the 
quotient stack constructed in (\ref{good:U0})
associated with the above collection 
$(E_1, \ldots, E_m)$
for the polystable sheaf (\ref{polystable:append}). 
The following result gives the formal neighborhood theorem 
for Gieseker moduli spaces. 
\begin{thm}\label{thm:formal}
	We have the commutative isomorphisms
	\begin{align}\label{dia:fthm}
		\xymatrix{
			\left[\widehat{\nN}_0/\Aut(E)\right]
			\ar[r]^-{\cong} \ar[d]  & \widehat{\mM}_{Z, H}(v)_y
			\ar[d] \\
			\widehat{\nN}_0 \ssslash \Aut(E) \ar[r]^-{\cong} & \Spec R. 
		}
	\end{align}
\end{thm}
\begin{proof}
	Note that the closed fiber of the morphism (\ref{mstack:complete}) 
	corresponds to $H$-semistable sheaves
	which are $S$-equivalent to $E$, 
	so in particular they are objects 
	in $\langle E_1, \ldots, E_m \rangle_{\rm{ex}}$. 
	Here $\langle -\rangle_{\rm{ex}}$ is the extension 
	closure of $(-)$. 
	Let $\iI \subset \oO_{\widehat{\mM}_{Z, H}(v)_y}$ 
	be the ideal sheaf which is the pull-back of the 
	maximal ideal $\mathfrak{m} \subset R$ 
	by the morphism (\ref{mstack:complete}). 
	Let $\mM_n \hookrightarrow \widehat{\mM}_{Z, H}(v)_y$
	be the closed substack defined by $\iI^n$. 
	Then the formal stack 
	$\varinjlim \mM_n$ represents the functor
	\begin{align*}
		\mM^{E}_{Z, H}(v) \colon 
		Aff^{op} \to Groupoid
	\end{align*}
	sending an affine $\mathbb{C}$-scheme 
	$T$
	to the groupoid of flat families of 
	coherent sheaves in $\langle E_1, \ldots, E_m\rangle_{\rm{ex}}$. 
	
	By~\cite[Corollary~6.7]{MR3811778},
	we have an equivalence of categories
	\begin{align*}
		\Phi \colon \modu_{\rm{nil}}(A) \stackrel{\sim}{\to}
		\langle E_1, \ldots, E_m \rangle_{\rm{ex}}. 
	\end{align*} 
	Here the left hand side is the abelian category of 
	finitely generated 
	nilpotent right $A$-modules. 
	The above functor is given by
	\begin{align*}
		\Phi(M)=M \otimes_A \eE, \ 
		\eE \in \Coh(\oO_X \widehat{\otimes}A)
	\end{align*}
	where $\eE$ is 
	the universal object 
	of NC deformation theory associated with the 
	collection $(E_1, \ldots, E_m)$. 
	Therefore the functor $\Phi$ gives an isomorphism of 2-functors
	\begin{align*}
		\Phi \colon 
		\mM_{(Q_{E_{\bullet}}, \mathbf{I})}^{\rm{nil}}(\vec{v}) \stackrel{\cong}{\to}
		\mM_{Z, H}^{E}(v). 
	\end{align*}
	Therefore the functor $\Phi$ also induces the isomorphism of 
	formal stacks
	\begin{align}\notag
		\Phi \colon 
		\varinjlim \left[\nN_n/\Aut(E)\right] \stackrel{\cong}{\to}
		\varinjlim \mM_n. 
	\end{align}
	Then applying 
	Theorem~\ref{thm:FGAGA}, 
	we have equivalences
	\begin{align}\label{dia:GAGA}
		\xymatrix{
			\Coh\left(\widehat{\mM}_{Z, H}(v)_y\right) \ar[r]^-{\sim} \ar[d]^-{\sim} 
			& \Coh\left(\left[\widehat{\nN}_0/\Aut(E)\right]\right) \ar[d]^-{\sim} \\
			\Coh\left(\varinjlim \mM_n\right) \ar[r]^-{\sim}_-{\Phi^{\ast}} &
			\Coh\left(\varinjlim \left[\nN_n/\Aut(E)\right]\right). 
		}
	\end{align}
	Here the vertical arrows are restriction functors and 
	the top arrow is defined by the above commutative diagram. 
	The top arrow sends 
	$\oO_{\widehat{\mM}_{Z, H}(v)_y}$ to 
	$\oO_{\left[\widehat{\nN}_0/\Aut(E)\right]}$
	and 
	preserves the tensor products, as 
	these properties are satisfied in the bottom arrow. 
	Therefore by the Tannaka duality for Artin stacks 
	(see~\cite[Theorem~1.1]{AHR3}), 
	there exists an unique isomorphism of stacks 
	\begin{align*}
		\Psi \colon \left[\widehat{\nN}_0/\Aut(E)\right] \stackrel{\cong}{\to}
		\widehat{\mM}_{Z, H}(v)_y
	\end{align*}
	such that the top arrow of (\ref{dia:GAGA})
	is isomorphic to $\Psi^{\ast}$. 
	Therefore we obtain the top isomorphism in the diagram (\ref{dia:fthm}). 
	The bottom isomorphism in the diagram (\ref{dia:fthm}) follows from 
	the uniqueness of good moduli spaces. 
\end{proof}

The case of moduli stacks of pairs is similarly proved as follows. 
\begin{lem}\label{lem:PS:formal}
	The derived stack $\fP_n^t(S, \beta)$ in (\ref{dmoduli:PS}) satisfies 
	formal neighborhood theorem. 
	\end{lem}
\begin{proof}
	The same argument of Theorem~\ref{thm:formal} applies by replacing 
	$\Coh(Z)$ with the abelian category $\aA_S$. 
	The only thing to check is that 
	the cotangent complex as deformations of pairs
	is the same as that as objects in $\aA_S$. 
	Namely 	
	for a pair $J=(\oO_S \to F) \in \aA_S$, let 
	$I=(\oO_S \to F) \in \Dbc(S)$ be the associated 
	object where $\oO_S$ is located in degree zero. 
	We show that there is a distinguished triangle
	\begin{align}\label{dist:pair1}
		\RHom_S(I, F) \to 
		\RHom_{\aA_S}(J, J)[1] \to \mathbb{C}[1]. 
		\end{align}
	So in particular, 
	we have isomorphisms $\Hom_S^i(I, F) \stackrel{\cong}{\to}
	\Ext_{\aA_S}^{i+1}(J, J)$ for $i=0, 1$. 
	From the exact sequence 
	$0 \to (0 \to F) \to J \to (\oO_S \to 0) \to 0$, 
	in $\aA_S$, we have the distinguished triangle
	\begin{align}\label{dist:pair2}
		\RHom_{\aA_S}(J, 0 \to F) \to 
		\RHom_{\aA_S}(J, J) \to \RHom_{\aA_S}(J, \oO_S \to 0). 
		\end{align}
	 We have the  adjoint pair of functors
	 \begin{align*}
	 	\xymatrix{
	 		\Dbc(S)
	 		\ar@<0.5ex>[r]^-{i} &
	 		\ar@<0.5ex>[l]^-{j}
	 		D^b(\aA_S)
	 	}, \
 i(F)=(0 \to F), \ 
 	j(W \otimes \oO_S \stackrel{s}{\to}F)=\mathrm{Cone}(s)
 	\end{align*}
 such that $j \dashv i$. 
 	 We also have the adjoint pair 
 	 \begin{align*}
 \xymatrix{
 	\Dbc(\mathbb{C})
 	\ar@<0.5ex>[r]^-{i'} &
 	\ar@<0.5ex>[l]^-{j'}
 	D^b(\aA_S)
 }, i'(W)=(W \otimes \oO_S \to 0), \ 
j'(W \otimes \oO_S \to F)=W
\end{align*}
such that $j'\dashv i'$. 
 Therefore the distinguished triangle (\ref{dist:pair2}) implies (\ref{dist:pair1}). 
	\end{proof}

\bibliographystyle{amsalpha}
\bibliography{math}

\newcommand{\etalchar}[1]{$^{#1}$}
\providecommand{\bysame}{\leavevmode\hbox to3em{\hrulefill}\thinspace}
\providecommand{\MR}{\relax\ifhmode\unskip\space\fi MR }
\providecommand{\MRhref}[2]{%
  \href{http://www.ams.org/mathscinet-getitem?mr=#1}{#2}
}
\providecommand{\href}[2]{#2}
\begin{thebibliography}{MNOP06}

\bibitem[AG15]{MR3300415}
D.~Arinkin and D.~Gaitsgory, \emph{Singular support of coherent sheaves and the
  geometric {L}anglands conjecture}, Selecta Math. (N.S.) \textbf{21} (2015),
  no.~1, 1--199.

\bibitem[AHD20]{AHR}
J.~Alper, J.~Hall, and D.~David, \emph{A {L}una \'{e}tale slice theorem for
  algebraic stacks}, Ann. of Math. (2) \textbf{191} (2020), no.~3, 675--738.

\bibitem[AHLH]{AHLH}
J.~Alper, D.~Halpern-Leistner, and J.~Heinloth, \emph{Existence of moduli
  spaces for algebraic stacks}, arXiv:1812.01128.

\bibitem[AHR]{AHR2}
J.~Alper, J.~Hall, and D.~Rydh, \emph{The \'etale local structure of algebraic
  stacks}, arXiv:1912.06162.

\bibitem[Alp13]{MR3237451}
J.~Alper, \emph{Good moduli spaces for {A}rtin stacks}, Ann. Inst. Fourier
  (Grenoble) \textbf{63} (2013), no.~6, 2349--2402.

\bibitem[AOV08]{MR2427954}
D.~Abramovich, M.~Olsson, and A.~Vistoli, \emph{Tame stacks in positive
  characteristic}, Ann. Inst. Fourier (Grenoble) \textbf{58} (2008), no.~4,
  1057--1091.

\bibitem[BBBJ15]{MR3352237}
O.~B. Bassat, C.~Brav, V.~Bussi, and D.~Joyce, \emph{A `{D}arboux theorem' for
  shifted symplectic structures on derived {A}rtin stacks, with applications},
  Geom. Topol. \textbf{19} (2015), no.~3, 1287--1359.

\bibitem[BBD{\etalchar{+}}15]{MR3353002}
C.~Brav, V.~Bussi, D.~Dupont, D.~Joyce, and B.~Szendr\H oi, \emph{Symmetries
  and stabilization for sheaves of vanishing cycles}, J. Singul. \textbf{11}
  (2015), 85--151, With an appendix by J\"org Sch\"urmann. \MR{3353002}

\bibitem[BBJ19]{BBJ}
C.~Brav, V.~Bussi, and D.~Joyce, \emph{A {D}arboux theorem for derived schemes
  with shifted symplectic structure}, J. Amer. Math. Soc. \textbf{32} (2019),
  399--443.

\bibitem[BDF{\etalchar{+}}16]{MR3488035}
M.~Ballard, D.~Deliu, D.~Favero, M.~U. Isik, and L.~Katzarkov,
  \emph{Resolutions in factorization categories}, Adv. Math. \textbf{295}
  (2016), 195--249.

\bibitem[Beh09]{MR2600874}
K.~Behrend, \emph{Donaldson-{T}homas type invariants via microlocal geometry},
  Ann. of Math. (2) \textbf{170} (2009), no.~3, 1307--1338.

\bibitem[BF97]{BF}
K.~Behrend and B.~Fantechi, \emph{The intrinsic normal cone}, Invent.~Math.~
  \textbf{128} (1997), 45--88.

\bibitem[BFK14]{MR3270588}
M.~Ballard, D.~Favero, and L.~Katzarkov, \emph{A category of kernels for
  equivariant factorizations and its implications for {H}odge theory}, Publ.
  Math. Inst. Hautes \'{E}tudes Sci. \textbf{120} (2014), 1--111.

\bibitem[BFK19]{MR3895631}
\bysame, \emph{Variation of geometric invariant theory quotients and derived
  categories}, J. Reine Angew. Math. \textbf{746} (2019), 235--303.

\bibitem[BIK08]{MR2489634}
D.~Benson, S.~B. Iyengar, and H.~Krause, \emph{Local cohomology and support for
  triangulated categories}, Ann. Sci. \'{E}c. Norm. Sup\'{e}r. (4) \textbf{41}
  (2008), no.~4, 573--619.

\bibitem[BJM19]{BDM}
V.~Bussi, D.~Joyce, and S.~Meinhardt, \emph{On motivic vanishing cycles of
  critical loci}, J. Algebraic Geom. \textbf{28} (2019), no.~3, 405--438.

\bibitem[BLL04]{MR2051435}
A.~I. Bondal, M.~Larsen, and V.~A. Lunts, \emph{Grothendieck ring of
  pretriangulated categories}, Int. Math. Res. Not. (2004), no.~29, 1461--1495.

\bibitem[BM14]{MR3194493}
A.~Bayer and E.~Macr\`\i, \emph{Projectivity and birational geometry of
  {B}ridgeland moduli spaces}, J. Amer. Math. Soc. \textbf{27} (2014), no.~3,
  707--752. \MR{3194493}

\bibitem[BO]{B-O2}
A.~Bondal and D.~Orlov, \emph{Semiorthogonal decomposition for algebraic
  varieties}, preprint, arXiv:9506012.

\bibitem[Bri07]{MR2373143}
T.~Bridgeland, \emph{Stability conditions on triangulated categories}, Ann. of
  Math. (2) \textbf{166} (2007), no.~2, 317--345.

\bibitem[Bri11]{MR2813335}
\bysame, \emph{Hall algebras and curve-counting invariants}, J. Amer. Math.
  Soc. \textbf{24} (2011), no.~4, 969--998.

\bibitem[BRTV18]{MR3877165}
A.~Blanc, M.~Robalo, B.~To\"{e}n, and G.~Vezzosi, \emph{Motivic realizations of
  singularity categories and vanishing cycles}, J. \'{E}c. polytech. Math.
  \textbf{5} (2018), 651--747.

\bibitem[BS]{DBOMS}
D.~Berth and O.~M. Schnurer, \emph{Decompositions of derived categories of
  gerbes and families of {B}rauer-{S}everi varieties}, arXiv:1901.08945.

\bibitem[C{\u{a}}l]{Cal}
A.~C{\u{a}}ld{\u{a}}raru, \emph{The {M}ukai pairing, {I}:~{T}he {H}ochschild
  structure}, arXiv:0308079.

\bibitem[C{\u{a}}l00]{MR2700538}
\bysame, \emph{Derived categories of twisted sheaves on {C}alabi-{Y}au
  manifolds}, ProQuest LLC, Ann Arbor, MI, 2000, Thesis (Ph.D.)--Cornell
  University.

\bibitem[Cal19]{Calaque}
D.~Calaque, \emph{Shifted cotangent stacks are shifted symplectic}, Annales de
  la faculte des sciences de Toulouse \textbf{28} (2019), 67--90.

\bibitem[Che10]{MR2734686}
X.~W. Chen, \emph{Unifying two results of {O}rlov on singularity categories},
  Abh. Math. Semin. Univ. Hambg. \textbf{80} (2010), no.~2, 207--212.

\bibitem[Dav18]{MR3739230}
B.~Davison, \emph{Positivity for quantum cluster algebras}, Ann. of Math. (2)
  \textbf{187} (2018), no.~1, 157--219.

\bibitem[DG13]{MR3037900}
V.~Drinfeld and D.~Gaitsgory, \emph{On some finiteness questions for algebraic
  stacks}, Geom. Funct. Anal. \textbf{23} (2013), no.~1, 149--294.

\bibitem[Dia12]{MR2888981}
D.~E. Diaconescu, \emph{Moduli of {ADHM} sheaves and the local
  {D}onaldson-{T}homas theory}, J. Geom. Phys. \textbf{62} (2012), no.~4,
  763--799.

\bibitem[Dim04]{MR2050072}
A.~Dimca, \emph{Sheaves in topology}, Universitext, Springer-Verlag, Berlin,
  2004.

\bibitem[Dri04]{MR2028075}
V.~Drinfeld, \emph{D{G} quotients of {DG} categories}, J. Algebra \textbf{272}
  (2004), no.~2, 643--691.

\bibitem[Efi18]{MR3815168}
A.~I. Efimov, \emph{Cyclic homology of categories of matrix factorizations},
  Int. Math. Res. Not. IMRN (2018), no.~12, 3834--3869.

\bibitem[EP15]{MR3366002}
A.~I. Efimov and L.~Positselski, \emph{Coherent analogues of matrix
  factorizations and relative singularity categories}, Algebra Number Theory
  \textbf{9} (2015), no.~5, 1159--1292.

\bibitem[Gai13]{MR3136100}
D.~Gaitsgory, \emph{ind-coherent sheaves}, Mosc. Math. J. \textbf{13} (2013),
  no.~3, 399--528, 553.

\bibitem[GR17a]{MR3701352}
D.~Gaitsgory and N.~Rozenblyum, \emph{A study in derived algebraic geometry.
  {V}ol. {I}. {C}orrespondences and duality}, Mathematical Surveys and
  Monographs, vol. 221, American Mathematical Society, Providence, RI, 2017.

\bibitem[GR17b]{MR3701353}
\bysame, \emph{A study in derived algebraic geometry. {V}ol. {II}.
  {D}eformations, {L}ie theory and formal geometry}, Mathematical Surveys and
  Monographs, vol. 221, American Mathematical Society, Providence, RI, 2017.

\bibitem[GZB15]{FGAGA}
A.~Geraschenko and D.~Zureick-Brown, \emph{Formal {GAGA} for good moduli
  spaces}, Algebraic Geometry \textbf{2} (2015), 214--230.

\bibitem[Hir17]{MR3631231}
Y.~Hirano, \emph{Derived {K}n\"orrer periodicity and {O}rlov's theorem for
  gauged {L}andau-{G}inzburg models}, Compos. Math. \textbf{153} (2017), no.~5,
  973--1007.

\bibitem[HLa]{HalpK3}
D.~Halpern-Leistner, \emph{The {D}-equivalence conjecture for moduli spaces of
  sheaves on a {K}3 surface}, available in
  http://www.math.columbia.edu/~danhl/.

\bibitem[HLb]{HalpK32}
\bysame, \emph{Derived {$\Theta$}-stratifications and the {$D$}-equivalence
  conjecture}, arXiv:2010.01127.

\bibitem[HLc]{HalpRem}
\bysame, \emph{Remarks on {T}heta-stratifications and derived categories},
  arXiv:1502.03083.

\bibitem[HLd]{HalpTheta}
\bysame, \emph{{$\Theta$}-statifications, {$\Theta$}-reductive stacks, and
  applications}, arXiv:1608.0479.

\bibitem[HL97]{MR1450870}
D.~Huybrechts and M.~Lehn, \emph{The geometry of moduli spaces of sheaves},
  Aspects of Mathematics, E31, Friedr. Vieweg \& Sohn, Braunschweig, 1997.

\bibitem[HL15]{MR3327537}
D.~Halpern-Leistner, \emph{The derived category of a {GIT} quotient}, J. Amer.
  Math. Soc. \textbf{28} (2015), no.~3, 871--912.

\bibitem[HLS20]{HLKSAM}
D.~Halpern-Leistner and S.~V. Sam, \emph{Combinatorial constructions of derived
  equivalences}, J. Amer. Math. Soc. \textbf{33} (2020), no.~3, 735--773.

\bibitem[HR]{HaMay}
J.~Hall and D.~Rydh, \emph{Mayer-{V}ietoris squares in algebraic geometry},
  arXiv:1606.08517.

\bibitem[HR17]{Hperf}
\bysame, \emph{Perfect complexes on algebraic stacks}, Compos. Math.
  \textbf{153} (2017), no.~11, 2318--2367.

\bibitem[HR19]{AHR3}
\bysame, \emph{Coherent {T}annaka duality and algebraicity of {H}om-stacks},
  Algebra Number Theory \textbf{13} (2019), no.~7, 1633--1675.

\bibitem[Isi13]{MR3071664}
M.~U. Isik, \emph{Equivalence of the derived category of a variety with a
  singularity category}, Int. Math. Res. Not. IMRN (2013), no.~12, 2787--2808.

\bibitem[Joy]{Jslide}
D.~Joyce, \emph{Shifted symplectic geometry, {C}alabi-{Y}au moduli spaces, and
  generalizations of {D}onaldson-{T}homas theory: our current and future
  research}, Talks given Oxford, October 2013, at a workshop for EPSRC
  Programme Grant research group,
  https://people.maths.ox.ac.uk/joyce/PGhandout.pdf.

\bibitem[Joy15]{MR3399099}
\bysame, \emph{A classical model for derived critical loci}, J. Differential
  Geom. \textbf{101} (2015), no.~2, 289--367.

\bibitem[JS12]{JS}
D.~Joyce and Y.~Song, \emph{A theory of generalized {D}onaldson-{T}homas
  invariants}, Mem. Amer. Math. Soc. \textbf{217} (2012), no.~1020, iv+199.

\bibitem[JT17]{MR3607000}
Y.~Jiang and R.~Thomas, \emph{Virtual signed {E}uler characteristics}, J.
  Algebraic Geom. \textbf{26} (2017), no.~2, 379--397.

\bibitem[Kaw02]{MR1949787}
Y.~Kawamata, \emph{{$D$}-equivalence and {$K$}-equivalence}, J. Differential
  Geom. \textbf{61} (2002), no.~1, 147--171.

\bibitem[Kel98]{MR1492902}
B.~Keller, \emph{Invariance and localization for cyclic homology of {DG}
  algebras}, J. Pure Appl. Algebra \textbf{123} (1998), no.~1-3, 223--273.

\bibitem[Kel99]{MR1667558}
\bysame, \emph{On the cyclic homology of exact categories}, J. Pure Appl.
  Algebra \textbf{136} (1999), no.~1, 1--56.

\bibitem[Kin94]{Kin}
A.~King, \emph{Moduli of representations of finite-dimensional algebras},
  Quart.~J.~Math.~Oxford Ser.(2) \textbf{45} (1994), 515--530.

\bibitem[KM98]{MR1658959}
J.~Koll\'{a}r and S.~Mori, \emph{Birational geometry of algebraic varieties},
  Cambridge Tracts in Mathematics, vol. 134, Cambridge University Press,
  Cambridge, 1998, With the collaboration of C. H. Clemens and A. Corti,
  Translated from the 1998 Japanese original.

\bibitem[Kra10]{MR2681709}
H.~Krause, \emph{Localization theory for triangulated categories}, Triangulated
  categories, London Math. Soc. Lecture Note Ser., vol. 375, Cambridge Univ.
  Press, Cambridge, 2010, pp.~161--235. \MR{2681709}

\bibitem[KS]{K-S}
M.~Kontsevich and Y.~Soibelman, \emph{Stability structures, motivic
  {D}onaldson-{T}homas invariants and cluster transformations}, preprint,
  arXiv:0811.2435.

\bibitem[KS06]{Kabook}
M.~Kashiwara and P.~Schapira, \emph{Categories and sheaves}, Grundlehren der
  Mathematischen Wissenschaften [Fundamental Principles of Mathematical
  Sciences], vol. 332, Springer-Verlag, Berlin, 2006.

\bibitem[KS11]{MR2851153}
M.~Kontsevich and Y.~Soibelman, \emph{Cohomological {H}all algebra, exponential
  {H}odge structures and motivic {D}onaldson-{T}homas invariants}, Commun.
  Number Theory Phys. \textbf{5} (2011), no.~2, 231--352.

\bibitem[KT]{KoTo}
N.~Koseki and Y.~Toda, \emph{Derived categories of {T}haddeus pair moduli
  spaces via d-critical flips}, arXiv:1904.04949.

\bibitem[Kuz11]{MR2801403}
A.~Kuznetsov, \emph{Base change for semiorthogonal decompositions}, Compos.
  Math. \textbf{147} (2011), no.~3, 852--876.

\bibitem[KV]{KaVa2}
M.~Kapranov and E.~Vasserot, \emph{The cohomological {H}all algebra of a
  surface and factorization cohomology}, arXiv:1901.07641.

\bibitem[Lie07]{MR2309155}
M.~Lieblich, \emph{Moduli of twisted sheaves}, Duke Math. J. \textbf{138}
  (2007), no.~1, 23--118.

\bibitem[LMB00]{MR1771927}
G.~Laumon and L.~Moret-Bailly, \emph{Champs alg\'{e}briques}, Ergebnisse der
  Mathematik und ihrer Grenzgebiete. 3. Folge. A Series of Modern Surveys in
  Mathematics [Results in Mathematics and Related Areas. 3rd Series. A Series
  of Modern Surveys in Mathematics], vol.~39, Springer-Verlag, Berlin, 2000.

\bibitem[Lun73]{MR0342523}
D.~Luna, \emph{Slices \'etales}, 81--105. Bull. Soc. Math. France, Paris,
  M\'emoire 33.

\bibitem[MNOP06]{MR2264664}
D.~Maulik, N.~Nekrasov, A.~Okounkov, and R.~Pandharipande,
  \emph{Gromov-{W}itten theory and {D}onaldson-{T}homas theory. {I}}, Compos.
  Math. \textbf{142} (2006), no.~5, 1263--1285.

\bibitem[MR10]{MrRi}
I.~Mirkovi\'{c} and S.~Riche, \emph{Linear {K}oszul duality}, Compos. Math.
  \textbf{146} (2010), no.~1, 233--258.

\bibitem[MR16]{MrRi2}
\bysame, \emph{Linear {K}oszul duality, {II}: coherent sheaves on perfect
  sheaves}, J. Lond. Math. Soc. (2) \textbf{93} (2016), no.~1, 1--24.
  \MR{3455779}

\bibitem[MR19]{MeRe}
S.~Meinhardt and M.~Reineke, \emph{Donaldson-{T}homas invariants versus
  intersection cohomology of quiver moduli}, J. Reine Angew. Math. \textbf{754}
  (2019), 143--178.

\bibitem[MT18]{MR3842061}
D.~Maulik and Y.~Toda, \emph{Gopakumar-{V}afa invariants via vanishing cycles},
  Invent. Math. \textbf{213} (2018), no.~3, 1017--1097.

\bibitem[Muk87]{Mu2}
S.~Mukai, \emph{On the moduli space of bundles on ${K}$3 surfaces ${I}$},
  Vector {B}undles on {A}lgebraic Varieties, M.~F.~Atiyah et al.~,Oxford
  University Press (1987), 341--413.

\bibitem[Nag13]{MR3079250}
K.~Nagao, \emph{Donaldson-{T}homas theory and cluster algebras}, Duke Math. J.
  \textbf{162} (2013), no.~7, 1313--1367.

\bibitem[NAS]{ADS}
W.~Donovan N.~Addington and E.~Segal, \emph{The {P}faffian-{G}rassmannian
  {E}quivalence {R}evisited}, arXiv:1401.3661.

\bibitem[Neg]{Negut2}
A.~Negu\c{t}, \emph{Hecke correspondences for smooth moduli spaces of sheaves},
  arXiv:1804.03645.

\bibitem[Neg19]{Negut}
\bysame, \emph{Shuffle algebras associated to surfaces}, Selecta Math. (N.S.)
  \textbf{25} (2019), no.~3, Art. 36, 57.

\bibitem[NN11]{MR2836398}
K.~Nagao and H.~Nakajima, \emph{Counting invariant of perverse coherent sheaves
  and its wall-crossing}, Int. Math. Res. Not. IMRN (2011), no.~17, 3885--3938.

\bibitem[OR]{ObRoz}
A.~Oblomkov and L.~Rozansky, \emph{Categorical {C}hern character and braid
  groups}, arXiv:1811.03257.

\bibitem[Orl09]{Orsin}
D.~Orlov, \emph{Derived categories of coherent sheaves and triangulated
  categories of singularities}, Algebra, arithmetic, and geometry: in honor of
  {Y}u. {I}. {M}anin. {V}ol. {II}, Progr. Math., vol. 270, Birkh\"{a}user
  Boston, Inc., Boston, MA, 2009, pp.~503--531.

\bibitem[P{\u{a}}d]{Tudor}
T.~P{\u{a}}durairu, \emph{K-theoretic {H}all algebras for quivers with
  potential}, arXiv:1911.05526.

\bibitem[Pot93]{LeP}
J.~Le Potier, \emph{Syst{\`e}mes coh{\'e}rents et structures de niveau},
  Ast{\'e}risque, vol. 214, Soci{\'e}t{\'e} Math{\'e}matique de France, 1993.

\bibitem[PP12]{MR2931331}
A.~Polishchuk and L.~Positselski, \emph{Hochschild (co)homology of the second
  kind {I}}, Trans. Amer. Math. Soc. \textbf{364} (2012), no.~10, 5311--5368.

\bibitem[PP17]{MR3600040}
R.~Pandharipande and A.~Pixton, \emph{Gromov-{W}itten/{P}airs correspondence
  for the quintic 3-fold}, J. Amer. Math. Soc. \textbf{30} (2017), no.~2,
  389--449.

\bibitem[PS]{PoSa}
M.~Porta and F.~Sala, \emph{Two dimensional categorified {H}all algebras},
  arXiv:1903.07253.

\bibitem[PT09]{MR2545686}
R.~Pandharipande and R.~P. Thomas, \emph{Curve counting via stable pairs in the
  derived category}, Invent. Math. \textbf{178} (2009), no.~2, 407--447.

\bibitem[PT10]{MR2552254}
\bysame, \emph{Stable pairs and {BPS} invariants}, J. Amer. Math. Soc.
  \textbf{23} (2010), no.~1, 267--297.

\bibitem[PTVV13]{MR3090262}
T.~Pantev, B.~To\"{e}n, M.~Vaqui\'{e}, and G.~Vezzosi, \emph{Shifted symplectic
  structures}, Publ. Math. Inst. Hautes \'{E}tudes Sci. \textbf{117} (2013),
  271--328.

\bibitem[PV11]{MR3112502}
A.~Polishchuk and A.~Vaintrob, \emph{Matrix factorizations and singularity
  categories for stacks}, Ann. Inst. Fourier (Grenoble) \textbf{61} (2011),
  no.~7, 2609--2642.

\bibitem[Sac19]{Sacca}
G.~Sacca, \emph{Relative compactified {J}acobians of linear systems on
  {E}nriques surfaces}, Trans.~AMS.~ \textbf{371} (2019), 7791--7843.

\bibitem[Shi12]{MR2982435}
I.~Shipman, \emph{A geometric approach to {O}rlov's theorem}, Compos. Math.
  \textbf{148} (2012), no.~5, 1365--1389.

\bibitem[ST11]{MR2869309}
J.~Stoppa and R.~P. Thomas, \emph{Hilbert schemes and stable pairs: {GIT} and
  derived category wall crossings}, Bull. Soc. Math. France \textbf{139}
  (2011), no.~3, 297--339.

\bibitem[Sum74]{MR0337963}
H.~Sumihiro, \emph{Equivariant completion}, J. Math. Kyoto Univ. \textbf{14}
  (1974), 1--28.

\bibitem[Tho00]{HCasson}
R.~P. Thomas, \emph{A holomorphic {C}asson invariant for {C}alabi-{Y}au
  3-folds, and bundles on {$K3$} fibrations}, J. Differential Geom. \textbf{54}
  (2000), no.~2, 367--438.

\bibitem[Toda]{Toddbir}
Y.~Toda, \emph{Birational geometry for d-critical loci and wall-crossing in
  {C}alabi-{Y}au 3-folds}, arXiv:1805.00182.

\bibitem[Todb]{TodDK}
\bysame, \emph{Semiorthogonal decompositions of stable pair moduli spaces via
  d-critical flips}, to appear in JEMS, arXiv:1805.00183.

\bibitem[Tod09]{Tolim}
\bysame, \emph{Limit stable objects on {C}alabi-{Y}au 3-folds}, Duke Math.~J.~
  \textbf{149} (2009), 157--208.

\bibitem[Tod10a]{Tcurve1}
\bysame, \emph{Curve counting theories via stable objects~{I}: {DT/PT}
  correspondence}, J.~Amer.~Math.~Soc.~ \textbf{23} (2010), 1119--1157.

\bibitem[Tod10b]{Tolim2}
\bysame, \emph{Generating functions of stable pair invariants via
  wall-crossings in derived categories}, Adv.~Stud.~Pure Math.~ \textbf{59}
  (2010), 389--434, New developments in algebraic geometry, integrable systems
  and mirror symmetry (RIMS, Kyoto, 2008).

\bibitem[Tod12]{Tsurvey}
\bysame, \emph{Stability conditions and curve counting invariants on
  {C}alabi-{Y}au 3-folds}, Kyoto Journal of Mathematics \textbf{52} (2012),
  1--50.

\bibitem[Tod13]{MR3021446}
\bysame, \emph{Curve counting theories via stable objects {II}: {DT}/nc{DT}
  flop formula}, J. Reine Angew. Math. \textbf{675} (2013), 1--51.

\bibitem[Tod18]{MR3811778}
\bysame, \emph{Moduli stacks of semistable sheaves and representations of
  {E}xt-quivers}, Geom. Topol. \textbf{22} (2018), no.~5, 3083--3144.

\bibitem[Tod20a]{Thall}
\bysame, \emph{Hall algebras in the derived category and higher-rank {DT}
  invariants}, Algebr. Geom. \textbf{7} (2020), no.~3, 240--262.

\bibitem[Tod20b]{THtype}
\bysame, \emph{Hall-type algebras for categorical {D}onaldson-{T}homas theories
  on local surfaces}, Selecta Math. (N.S.) \textbf{26} (2020), no.~4, 64.

\bibitem[To{\"{e}}07]{Todg}
B.~To{\"{e}}n, \emph{The homotopy theory of {$dg$}-categories and derived
  {M}orita theory}, Invent. Math. \textbf{167} (2007), no.~3, 615--667.

\bibitem[To{\"{e}}11]{MR2762557}
\bysame, \emph{Lectures on dg-categories}, Topics in algebraic and topological
  {$K$}-theory, Lecture Notes in Math., vol. 2008, Springer, Berlin, 2011,
  pp.~243--302.

\bibitem[To{\"{e}}14a]{MR3285853}
\bysame, \emph{Derived algebraic geometry}, EMS Surv. Math. Sci. \textbf{1}
  (2014), no.~2, 153--240.

\bibitem[To{\"{e}}14b]{MR3728637}
\bysame, \emph{Derived algebraic geometry and deformation quantization},
  Proceedings of the {I}nternational {C}ongress of {M}athematicians---{S}eoul
  2014. {V}ol. {II}, Kyung Moon Sa, Seoul, 2014, pp.~769--792.

\bibitem[TV07]{MR2493386}
B.~To\"{e}n and M.~Vaqui\'{e}, \emph{Moduli of objects in dg-categories}, Ann.
  Sci. \'{E}cole Norm. Sup. (4) \textbf{40} (2007), no.~3, 387--444.

\bibitem[VdB04]{MR2057015}
M.~Van~den Bergh, \emph{Three-dimensional flops and noncommutative rings}, Duke
  Math. J. \textbf{122} (2004), no.~3, 423--455.

\bibitem[vVdB17]{MR3698338}
\v{S}. \v{S}penko and M.~Van~den Bergh, \emph{Non-commutative resolutions of
  quotient singularities for reductive groups}, Invent. Math. \textbf{210}
  (2017), no.~1, 3--67.

\bibitem[Zhaa]{Zhao2}
Y.~Zhao, \emph{A {C}ategorical {Q}uantum {T}oroidal {A}ction on {H}ilbert
  {S}chemes}, arXiv:2009.11267.

\bibitem[Zhab]{Zhao}
\bysame, \emph{On the {K}-theoretic {H}all algebra of a surface},
  arXiv:1901.00831.

\end{thebibliography}

\vspace{5mm}

Kavli Institute for the Physics and 
Mathematics of the Universe (WPI), University of Tokyo,
5-1-5 Kashiwanoha, Kashiwa, 277-8583, Japan.

\textit{E-mail address}: yukinobu.toda@ipmu.jp

\end{document}